\documentclass[nospthms]{svmono}

\usepackage[latin2]{inputenc}

\usepackage{mathptmx}
\usepackage{helvet}
\usepackage{courier}
\usepackage{graphicx}
\usepackage{multicol}
\usepackage{footmisc}

\usepackage[T1]{fontenc}
\usepackage{euscript}
\usepackage{amsmath}
\usepackage{amsthm}
\usepackage{amssymb}
\usepackage{amscd}
\usepackage{epic}
\usepackage{eufrak}
\usepackage{enumerate}
\usepackage{array}
\usepackage[bookmarksnumbered,bookmarksopen,pdfusetitle]{hyperref}

\usepackage{makeidx}
\usepackage{multicol}        % used for the two-column index

\newtheorem{theorem}[equation]{Theorem}
\newtheorem{lemma}[equation]{Lemma}
\newtheorem{proposition}[equation]{Proposition}
\newtheorem{corollary}[equation]{Corollary}
\newtheorem{definition}[equation]{Definition}

\theoremstyle{definition}
\newtheorem{example}[equation]{Example}
\newtheorem{remark}[equation]{Remark}

\newtheorem{bekezdes}[equation]{}

\newcommand{\Q}{\mathbb Q}
\newcommand{\R}{\mathbb R}
\newcommand{\Z}{\mathbb Z}
\newcommand{\N}{\mathbb N}
\newcommand{\bp}{\mathbb P}

\newcommand{\cV}{\mathcal V}
\newcommand{\cvg}{\mathcal V(\Gamma)}
\newcommand{\evg}{\mathcal E(\Gamma)}

\newcommand{\bfc}{{\Bbb C}}
\newcommand{\bfz}{{\Bbb Z}}

\newcommand{\calc}{{\mathcal C}}
\newcommand{\co}{{\mathcal O}}

\newcommand{\calv}{{\mathcal V}}
\newcommand{\cala}{{\mathcal A}}
\newcommand{\calw}{{\mathcal W}}
\newcommand{\cale}{{\mathcal E}}

\newcommand{\cals}{{\mathcal S}}
\newcommand{\calp}{{\mathcal P}}

\newcommand{\csj}{\C_{\Sigma_j}}
\newcommand{\csi}{\C_{\Sigma_i}}
\newcommand{\gj}{\G_{\C,j}^2}
\newcommand{\bnu}{{\mathbf{\nu}}}
\newcommand{\G}{\Gamma}
\newcommand{\C}{{\calc}}
\newcommand{\noi}{\noindent}

\newcommand{\gc}{\G_\C}
\newcommand{\gce}{\G_{\C}^1}
\newcommand{\gck}{\G_{\C}^2}
\newcommand{\gcce}{G_{\C}^1}

\newcommand{\fcd}{F_{c,d}}
\newcommand{\Wedge}{{\mathsf W_{\eta, M}}}
\newcommand{\de}{\mathbf{D}}    % az elozoek unioja
\newcommand{\dec}{\mathbf{D}_c} % ahol f eltunik
\newcommand{\ded}{\mathbf{D}_d} % ahol g eltunik
\newcommand{\deo}{\mathbf{D}_0} % ahol f is g is eltunik

\newcommand{\wP}{\widetilde{\Phi}}

\newcommand{\ic}{irreducible component}

\newcommand{\bs}{\bigskip}
\newcommand{\sm}{\smallskip}
\newcommand{\bd}{{\bf d}}
\newcommand{\bn}{{\bf n}}
\newcommand{\fedog}{{\mathcal G}(\Gamma,(\bn,\bd))}
\newcommand{\wfp}{\widetilde{F}_\Phi}
\newcommand{\wfv}{\widetilde{F}_v}
\newcommand{\wfw}{\widetilde{F}_w}
\newcommand{\wfW}{\widetilde{F}_{{\calw}}}
\newcommand{\wsk}{\widetilde{\cals}_k}
\newcommand{\nsk}{\cals_k^{norm}}
\newcommand{\gbk}{\Gamma_{van}}
\newcommand{\rank}{{\rm rank}\, }\newcommand{\corank}{{\rm corank}\, }
\newcommand{\im}{{\rm im}\,}
\newcommand{\Gmod}{G^{m}}
\newcommand{\Gemod}{G_1^{m}}
\newcommand{\Gkjmod}{G_{2,j}^{m}}
\newcommand{\coker}{{\rm coker}\,}
\newcommand{\tw}{T(E_w)}\newcommand{\tv}{T(E_v)}
\newcommand{\n}{\mathfrak{n}}
\newcommand{\dg}{\mathfrak{d}}
\newcommand{\labelpar}{\label}
\newcommand{\tib}{{\tiny{\Box}}}

\newcommand{\ix}{\index}
\newcommand{\inc}{{\mathfrak{I}}}

\def\blfootnote{\xdef\@thefnmark{}\@footnotetext}

\makeindex

\begin{document}

\author{Andr\'as N\'emethi, \'Agnes Szil\'ard}

\title{Milnor fiber boundary of a non--isolated surface singularity}

\maketitle

\frontmatter%%%%%%%%%%%%%%%%%%%%%%%%%%%%%%%%%%%%%%%%%%%%%%%%%%%%%%

\tableofcontents

\mainmatter%%%%%%%%%%%%%%%%%%%%%%%%%%%%%%%%%%%%%%%%%%%%%%%%%%%%%%%

%\part{\ Introduction}

\chapter{Introduction}\label{s:INT}
\section{\ Motivations, goals and results}

\bekezdes \label{intro}
The origins of the present work go back to
some milestones marking the birth of singularity theory of complex dimension $\geq 2$.
They include the Thesis of Hirzebruch (1950) \ix{Hirzebruch} containing, among others,
the modern theory of cyclic quotient \ix{singularities!cyclic quotient}
singularities;  Milnor's construction \ix{Milnor} of the exotic 7--spheres as plumbed manifolds
associated with  `plumbing graphs'; Mumford's article \ix{Mumford} about normal surface singularities
\cite{Mu} stressing for the first time the close relationship of the topology with
the algebra; the treatment and  classification of links of singularities
by Hirzebruch and his students in the 60's (especially Brieskorn \ix{Brieskorn} and J\"anich, \ix{J\"anich}
and later  their students)
based on famous results on classification of manifolds by Smale, Thom, Pontrjagin, Adams, Kervaire
and Milnor, and the signature theorem of Hirzebruch. Since then, and since
the appearance of the very influential book   \cite{MBook} of Milnor in 1968,
the theory of normal surface singularities
and isolated hypersurface singularities produced an
enormous amount of significant results. In all of them, the {\it link} \ix{link} of an isolated
singularity plays a central role.
\ix{Smale} \ix{Thom} \ix{Pontrjagin} \ix{Adams} \ix{Kervaire}
\ix{singularities!normal surface} \ix{singularities!hypersurface}

In the presence of a smoothing, like in the case of
hypersurfaces, the link appears also as the boundary of the smoothing, the Milnor fiber.
This interplay has enormous consequences.

First, the link of a singularity is the boundary of an
arbitrary small neighbourhood of the singular point, hence one can {\em localize}
 the link in {\em any arbitrarily small} representative.
Hence,  by resolving the singularity, the link
 appears as the boundary of a small tubular neighbourhood
of the exceptional locus, that is, as a plumbed 3--manifold.
In this way, e.g. for isolated surface singularities, a
bridge is created between the link and resolution:  the resolution  at topological level is
codified in the resolution graph,
which  also serves as a plumbing graph for the link.

On the other hand, the same manifold
is the boundary of the Milnor fiber, that `nearby fiber' whose degeneration and monodromy
measures the complexity of the singularity. This interplay between the two holomorphic fillings
of the link, the resolution and the Milnor fiber, produces (perhaps) the nicest index--theoretical
relations of hypersurface singularity theory: relations of Durfee \cite{Du} and Laufer \cite{Lam}
valid for hypersurfaces, and
generalizations of Looijenga, Wahl \cite{LjW,Wa}, Seade \cite{Se} for more general
smoothings. It culminates  in deformation theory describing \ix{deformation!theory} \ix{smoothings}
the miniversal deformation spaces in some cases. \ix{monodromy}
\ix{Looijenga} \ix{Wahl} \ix{Seade} \ix{Durfee} \ix{Laufer}

The links of normal surface singularities, as special oriented 3--manifolds represented by
negative definite plumbing graphs, started to have  recently a significant role in low--dimensional topology as well:
they not only provide crucial testing manifolds for the Seiberg--Witten, \ix{Seiberg--Witten invariants}
\ix{graph!negative definite}
or Heegaard Floer theory
of 3--manifolds, and provide ground for surprising connections between these topological invariants and
the algebraic/analytic invariants of the singularity germs
(see for example  \cite{Ninv} and \cite{Nlattice} and the references therein),
but they appear also as the outputs
of natural topological classification results,  as solutions of some universal topological properties
--- for example,  the rational singularities appear as $L$--spaces \cite{Nlattice},
that is,  3--manifolds with vanishing reduced
Heegaard Floer homology.
\ix{Heegaard Floer homology}
Similarly,  the classification of symplectic fillings --- or more particularly,
the classification of rational ball fillings --- of some 3--manifolds find their natural
foreground in some singularity links, see \cite{NPPP} and  \cite{SSzW} and references therein.

\bekezdes \ Although the literature of isolated singularities is huge, surprisingly, the literature of
non--isolated singularities, even of the non--isolated hypersurface singularities, is rather restricted.
One of the main difficulties is generated by the fact that in this case the link is not smooth.

On the other hand, the boundary of the Milnor fiber is smooth, but (till the present work)
there was no construction which
would guarantee that it also appears as the boundary of any arbitrary small representative
of an isolated singular germ. Lacking such a result, it is hard to prove in a conceptual way
that this manifold is a plumbed manifold, even in the case of surfaces.

The present work aims to fill in this gap: we provide a general procedure that may be used
to attack and treat non--isolated hypersurface singularities.
First, the   localization property
 guarantees the existence of a plumbing representation. \ix{plumbing!representation}
But the strategy and the presentation is not limited to the plumbing representation of the boundary
(the proof of this fact is just the  short Proposition  \ref{lem:sksmooth}),    we target
a uniform conceptual treatment of the involved invariants of the singular germs including its connection with
the normalization, transversal type singularities and different monodromy operators. \ix{monodromy}
Although the whole presentation is for
surface singularities, some results can definitely be extended to  arbitrary dimensions.

\vspace{2mm}

More precisely, some of the conceptual results of the
present work are the following. Below,
the germ of the holomorphic function
$f:(\bfc^3,0)\to (\bfc,0)$ defines  a complex analytic hypersurface singularity with
1--dimensional singular locus $\Sigma$. Its zero set $\{f=0\}$ is denoted by $V_f$, and
its  Milnor fiber by $F$. Its oriented
boundary $\partial F$ is a connected oriented 3--manifold.

\begin{enumerate}
\item The oriented 3--manifold has a plumbing representation. In fact, we not only prove this result,
but provide a concrete algorithm for the construction of the plumbing graph: given {\it any} germ $f$,
anyone, with some experience in blowing ups and handling  equations of resolutions,  is able to determine
the graph after some work. (For the algorithm, see Chapter \ref{s:ALG}.)
The output graphs, in general, are not irreducible, are not negative definite, \ix{graph!negative definite}
 hence in their discussion the usual calculus --- blowing up and down
$(-1)$--rational curves --- is not sufficient.
In our graph--manipulations we will use the {\it calculus of oriented 3--manifolds} as it is described in
Neumann's foundational article \cite{neumann} (in fact, we will
restrict ourself to a `reduced class of graph--operations').
The corresponding background material and preliminary
discussions are presented in Chapter \ref{s:PLU}.   For some interesting examples
and  peculiarities see section \ref{LE}.

In this direction some particular results were known in the literature.
The fact that the boundary of the Milnor fiber is plumbed was announced by
F. Michel  and A. Pichon   in \cite{MP,MPerrata}, cf. \ref{rem:THutan}(c), and simple
examples  were provided in \cite{MP,MPW,CC} obtaining certain  lens spaces and Seifert manifolds.
Randell \cite{Ra} and Siersma  \cite{Si,Si3} determined the homology
of the boundary $\partial F$  for several cases. Moreover,
they  characterized via different criteria  those situations when $\partial F$ is
homology sphere; see also \ref{re:SiRa}.
\ix{Michel} \ix{Pichon} \ix{graph decoration!edge} \ix{Neumann} \ix{plumbing!graph calculus}
\ix{plumbing!graph calculus!reduced}\ix{lens space}\ix{Seifert manifold}

However,
the present work  uses a different and novel strategy compared with the existing  literature
of non--isolated singularities. Moreover, it provides the plumbing graph for arbitrary germs
(even for which ad hoc methods are not available), and
%The general picture provided by the present book
it points out for the first time in the literature
that for these boundary manifolds one needs to use
`extended', general plumbing graphs --- where the edges might have negative weights as well.
Such graphs were not used at all  in complex algebraic geometry before.
The sign decorations of edges
are  irrelevant for trees, but are crucial in the presence of cycles  in the plumbing graph, as it
is shown in many examples in this work.

The reason for the appearance of the negative edges is the following:
certain parts of the graph behave as `usual graphs of complex algebraic geometry' but with opposite
orientation. This fact is the outcome of the {\em real analytic} origin of the plumbing representation,
see the next subsection \ref{1.3}.\\

\item Recall that  for any germ $g:(\bfc^3,0)\to (\bfc,0)$ such that the pair $(f,g)$ forms an ICIS
(isolated complete intersection singularity), by a result of Caubel \cite{CC}, $g$ determines an
open book decomposition on $\partial F$ (similarly as the classical Milnor fibration is cut out by the
argument of $g$). For any such $g$, our method
determines, as an additional decoration in the plumbing graph,  the `multiplicity system' of
this open book decomposition too (for definitions, see \ref{Openbooks}).\\
\ix{Caubel} \ix{open book decomposition}

\item The boundary $\partial F$ consists of two parts (a fact already proved  by D. Siersma
\cite{Si,Si3} and used by
F. Michel, A. Pichon and C. Weber in \cite{MP,MPerrata,MPW} too): the first, $\partial _1F$,
is the complement of an open tubular neighbourhood of the strict transforms of $\Sigma$
in the link of  the normalization of $V_f$.
The second one, $\partial _2F$, can be
 recovered from the  transversal plane curve singularity types
 of $\Sigma$ together with the corresponding vertical monodromy actions.
 Already determining these two  independent pieces can be a non--trivial task,
but the identification of their  gluing can be incomparably  harder.
 The present work clarifies this gluing
 completely (in fact, resolves it so automatically, that if one does not
 look for the  phenomenon  deliberately, one will not even see it).\\
\ix{Siersma} \ix{Weber}  \ix{normalization of $V_f$}
\ix{monodromy} \ix{monodromy!vertical}

\item The monodromy action on $\partial F$, in general, is not trivial. In fact, its restriction to
 the piece
$\partial _1F$  is trivial, but the monodromy  on $\partial_2F$  can be rather complicated. In order to understand
the monodromy action on  $\partial_2F$, we need the {\em vertical} monodromies of the transversal singularity types,
and, in fact, via the Wang exact sequence, we need their
Jordan block structure corresponding to eigenvalue one; on the other hand, the monodromy action on
$\partial_2F$ is induced by the {\em horizontal} monodromy of the transversal singularity types.
In the body of the paper we determine this commuting pair of actions, and as a by--product,
the homology of $\partial F$ and the
 characteristic polynomial of the algebraic monodromy action on $H_1(\partial F)$ too
 (under certain assumptions).
\ix{monodromy!vertical!transversal} \ix{monodromy!horizontal!transversal}

 The discussion includes the study of some monodromy operators of the
 ICIS $(f,g)$ as well, as detailed below.
\end{enumerate}

\bekezdes\labelpar{1.3} \
The `Main Algorithm' is based on a special construction: we  take an arbitrary germ
$g:(\bfc^3,0)\to (\bfc,0)$
such that $\Phi=(f,g)$ forms an ICIS. The topology of such a map is described in Looijenga's book
 \cite{Lj}, and it will be used intensively. We recall the necessary  material in Chapter
\ref{s:ICIS}.\ix{Main Algorithm}

In general, $\Phi$ provides a powerful tool to analyse the germ $f$ or its
$g$--polar properties. Usually, for an arbitrary germ $f$ with 1--dimensional singular locus, one
takes a plane curve singularity $P:(\bfc^2,0)\to (\bfc,0)$ and considers the composed function
$P\circ \Phi$. For certain germs $P$, this  can be thought of as an approximation of $f$ by  isolated singularities.
For example, if $P(c,d)=c+d^k$, then $P\circ \Phi=f+g^k$ is one of the most studied test
series, the Iomdin series of $f$ associated with $g$ \cite{Io,Le2}.
\ix{polar} \ix{singularities!composed}  \ix{ICIS!discriminant}

If we wish to understand  the geometry of $P\circ \Phi$, for example its Milnor fiber,
then we need to analyse all the intersections of $\{P=\delta \}$ with
the discriminant $\Delta$  of $\Phi$, and we have to understand the whole
monodromy representation of
$\Phi$ over the complement of $\Delta$ --- a very difficult task, in general.
On the other hand,  if we `only' wish to determine
 some `correction terms' --- for example,    $i(f)-i(f+g^k)$  for an invariant $i$ ---,
 it is enough to study $\Phi$
only above a  neighbourhood of the link of the distinguished discriminant component
 $\Delta_1:=\Phi(V_f)$.

This fact has been exploited at  many different levels, and for several invariants, see e.g.
the articles of L\^e and  Teissier initiating and developing the `theory of discriminants'
\cite{Le1,Le2,Te}; the article of Siersma \cite{Si2} about the zeta function of the Iomdin series,
or its generalizations by the  first author, cf.  \cite{NPhg} and the references therein.
\ix{L\^e} \ix{Teissier} \ix{Siersma} \ix{Iomdin series}

Using this principle,
in \cite{eredeti} we  determined the {\em links} of members of the Iomdin series, that is,
the links  of isolated singularities $f+g^k$ for  $k\gg0$.
In that work the key new ingredient was the  construction of  a special graph $\gc$,
dual to a curve configuration  ${\mathcal C}$ in an embedded
 resolution of $V_{fg}\subset \bfc^3$ localized above  a  `wedge neighbourhood' of $\Delta_1$
(for the terminology and more comments see \ref{why}).
\ix{curve configuration $\C$} \ix{wedge neighbourhood}

The point is that  the very same  graph $\gc$
not only contains all the information necessary to determine the links of  the Iomdin series (and
the correction terms  $i(f)-i(f+g^k)$ for several invariants $i$),
but it is the right object to determine $\partial F$ as well.

The bridge that connects $\partial F$ with the previous discussion (about series and discriminant
of the ICIS $(f,g)$)
is realized by the following fact.
Let $k$ be a sufficiently large even integer. Consider the local {\em real} analytic germ
$$\{f=|g|^k\}\subset (\bfc^3,0).$$
Then, for large $k$, its link (intersection with an arbitrary small sphere)
is a smooth oriented 3--manifold,
independent of the choice of $k$, which, most importantly,  {\it is
diffeomorphic with $\partial F$}, cf. Proposition  \ref{lem:sksmooth}.
In particular, it means that $\partial F$ can be `localized':
it appears as the boundary of an arbitrary small
neighbourhood of an analytic germ! But, in this (non--isolated)
case,  the corresponding space--germ is not complex,
but real analytic.
\ix{localization} \ix{graph!covering}

As a consequence, after resolving this real analytic
singularity, the tubular neighbourhood of the exceptional set
provides a plumbing representation $G$ of $\partial F$.
The point is that the graph $\gc$ %(introduced in \cite{eredeti} and mentioned above)
codifies all the necessary  information
to recover the topology of the resolution and of the plumbing:
the plumbing graph $G$ of $\partial F$ appears as a `graph
covering' of $\gc$. The Main Algorithm (cf. Chapter  \ref{s:ALG}) \ix{Main Algorithm}
provides a pure combinatorial description of
$G$ derived from $\gc$.  (The necessary  abstract theory of `coverings
of graphs', developed in \cite{cyclic}, is reviewed in Chapter \ref{ss:2.3}.)

The method emphasizes the importance of real analytic germs, and their  necessity even in the
study of complex geometry. %For different aspects regarding real analytic germs, see e.g.
%\cite{PS1,PS2,SeBook}.

In Chapters \ref{s:vh}--\ref{s:MHS} we determine several related homological invariants from
$\gc$ (characteristic polynomials of horizontal/vertical monodromies, Jordan block structure, cf.
below).
In fact, we believe that $\gc$ contains even more information than what was exploited in \cite{eredeti}
or here, which can be the subject of future research.

On the other hand, finding the
graphs $\gc$ can sometimes be a serious job. Therefore, we decided to provide examples
for $\gc$ in abundance in order to help the reader  understand the present work better, and support  possible
future research as well.

As different embedded resolutions might produce different graphs $\gc$, readers with more experience
in resolutions might find even simpler graphs in some cases. Each $\gc$ can equally be used for the theory
worked out.

\bekezdes \  The above presentation already suggests that the geometry of $\Phi$ near the discriminant
component  $\Delta_1$
is reflected in the topology of $\partial F$ as well. Technically, this is described in the commuting
actions of the horizontal and vertical monodromies on the fiber of $\Phi$.
Similarly, one can consider the horizontal/vertical monodromies of the local transversal types of $\Sigma$,
associated with $\Delta_1$.
The determination of these two pairs of representations
is an important task (independently of the identification of $\partial F$), and it is crucial in many constructions
about non-isolated singularities.

We wish to stress that our method provides (besides the result
regarding $\partial F$) a uniform discussion of these monodromy representations,
and gives  a clear procedure to determine  the corresponding
characters of the $\Z^2$--representations and the  characteristic polynomials
with precise closed formulae.
This can be considered as the {\it generalization} to isolated complete intersection
singularities {\it of A'Campo's formula} \cite{A'CMon,AC,AC2}, valid for hypersurface singularities.
Moreover, we determine even the Jordan block structures
in those characters  which are needed for the homology of $\partial F$ and its algebraic monodromy.
This is a generalization of results of Eisenbud--Neumann \cite{EN} regarding Jordan blocks of
algebraic monodromies associated with 3--dimensional graph manifolds.
\ix{monodromy!vertical} \ix{monodromy!horizontal}\ix{A'Campo's formula}\ix{Eisenbud--Neumann book}

\bekezdes \  The material of the present work can be grouped into several parts.

The first part  contains
results regarding the germ $f$, the ICIS $\Phi$ and the graph $\gc$ read from a resolution.
This is an introductory part, where we list all the background  material needed from the literature.

The second part
  contains the description of $\partial F$, its invariants, and establishes the main
connections with other geometrical objects.

 Finally, the third part (but basically everywhere in the chapters),
  we present many examples, among them treatments of specific classes of singularities
 as homogeneous, suspensions, cylinders, etc.

Any example is given in two steps: first, the graph $\gc$ has to be determined, a
more or less independent task.  This can be done in many different ways using one's
preferred resolution tricks. Nevertheless, in most of the cases, it is not  a trivial procedure,
except for special cases such as cylinders or homogeneous germs.
Then, in the second step, we run the Main Algorithm to get the plumbing graph of $\partial F$ or its
monodromy, or the open book decomposition of $g$ living on $\partial F$.

Our examples  test and illustrate  the theory, emphasizing the new aspects:
the obtained  graphs are not the usual negative \ix{graph!negative definite}
definite graphs provided by resolutions of normal surface singularities, not even the dual
graphs of complex curve configurations of a complex surface (where the intersections of curves are
always positive). They contain pieces with `opposite' orientation,  hence  vertices may have to be  connected by
negative--edges too. Indeed, in the resolution of the real singularity
$\{f=|g|^k\}$, some singularities are {\em orientation reversing} equivalent with
complex Hirzebruch--Jung singularities; hence the corresponding Hirzebruch--Jung strings should be sewed
in the final graph with opposite orientation.
\ix{graph!resolution!string}

In Chapters \ref{s:cyl}--\ref{s:fAB} for certain  families (basically, for `composed singularities')
we also provide alternative, topological constructions of $\partial F$.

\bekezdes \ The titles of the chapters and sections  already listed in the Contents
 were chosen  so that they would guide the reader easily  through the sections.

\bekezdes \label{MPcomment} \ Most of the theoretical results of the present work were obtained in
2004--2005; the Main Algorithm was presented at the Singularity Conference at Leuven, 2005.
Since that time we added several examples and completed the theoretical part.
A completed version  \cite{ARXIV} was posted on the Algebraic Geometry preprint server  on 2 September 2009.
\ix{Main Algorithm}

\section{\ List of examples with special properties}\label{LE}
\setcounter{equation}{0}

In order to arouse  the curiosity of the reader, and to exemplify the variety
of 3--manifolds obtained as $\partial F$, we list some peculiar examples.
They are extracted from the body of the work where a lot more examples are found.

In the sequel we use the symbol $\approx$ for  {\em orientation preserving diffeomorphism}.
If $M$ is a 3--manifold with fixed orientation, then
$-M$ is $M$ with opposite orientation.\\

For certain  choices of $f$, the boundary $\partial F$ might have one of the following
peculiar properties:

\begin{bekezdes} $\partial F$ cannot be represented by a negative definite plumbing, but
$-\partial F$ admits such a representation, even with all edge decorations positive, see  example \ref{ex:347b}.
\end{bekezdes}
\ix{graph!negative definite}

\begin{bekezdes}
Neither $\partial F$, nor $-\partial F$  can be represented by a negative definite plumbing, see
\ref{ss:dsmall}(4a).
\end{bekezdes}

\begin{bekezdes}
$\partial F$ can be represented by a negative definite graph, but it is impossible to
arrange all the  edge decorations
positive. Hence, such a graph cannot be the graph of a normal surface singularity.
See \ref{ex:dminusz2}.
\end{bekezdes}
\ix{graph decoration!edge}

\begin{bekezdes}
$\partial F$ can be represented by a negative definite graph with all edge decorations
positive.  If $\partial F$ is a lens space then this property is automatically true.
On the other hand, examples with this property and which are not lens spaces are rather
rare.   The examples \ref{ss:unicusp}(e-f)  are Seifert manifolds with three special orbits.
\end{bekezdes}
\ix{lens space} \ix{Seifert manifold}

\begin{bekezdes} In all the examples, $\partial F$
is not orientation preserving diffeomorphic with the link of the normalization of $V_f$
(evidently, provided that the singular locus of $V_f$ is 1--dimensional); a fact already noticed in \cite{MP}.
\end{bekezdes}
\ix{normalization of $V_f$}

\begin{bekezdes} There are  examples when $\partial F$
is orientation {\it reversing} diffeomorphic with the link of the normalization of $V_f$,
 see \ref{ss:unicusp}(d).
\end{bekezdes}

\begin{bekezdes} There are examples when $\partial F\approx -\partial F$. In the world of negative definite
plumbed manifolds,
 if both $M$ and $-M$ can be represented by negative definite graphs then
 the graph is either a string or it is cyclic \cite{neumann}. Here in  \ref{bek:plumf} we provide examples
 for $\partial F\approx -\partial F$ with plumbing graphs containing an arbitrary number of  cycles.
 See also  \ref{ss:dsmall}(3b). \ix{graph!resolution!string}
\end{bekezdes}

\begin{bekezdes}
It may happen that $\partial F$ fibers over $S^1$ with $\rank\, H_1(\partial F)$ arbitrary large,
  cf. \ref{ss:GEOM}.
\end{bekezdes}

\begin{bekezdes}
$\partial F$ might be non--irreducible, cf. \ref{bek:classifi2}(c) and \ref{ss:partF}.
\end{bekezdes}

\begin{bekezdes} Any $S^1$--bundles with non--negative Euler number
over any oriented surface can be realized as $\partial F$, cf. \ref{ss:XY} and
\ref{ex:almostfan}.
\end{bekezdes}

\begin{bekezdes}
The monodromy on $\partial F$, in general, is not trivial, the algebraic monodromy might even have
Jordan blocks of size two, cf. \ref{ss:comgeo}.
\end{bekezdes}
\ix{monodromy!Jordan block} %\ix{monodromy!Jordan block!of size two}
\ix{monodromy!characteristic polynomial}\ix{characteristic polynomial|see{monodromy/characteristic polynomial}}
\ix{mixed Hodge structure!weight filtration}

\begin{bekezdes}
There exist pairs of singularity germs with diffeomorphic
 $\partial F$ but different characteristic polynomials of $H_1(\partial F)$,
and/or different mixed Hodge weight filtrations on $H_1(\partial F)$,
and/or different multiplicities,  cf. \ref{ex:MHSnot}.
\end{bekezdes}

\part{Preliminaries}

\chapter{The topology of a hypersurface germ $f$ in three variables}\label{s:egy}
\section{\ The link and the Milnor fiber $F$ of hypersurface singularities}\labelpar{ss:2.0a}

Let $f:({\bfc}^n,0)\to ({\bfc},0)$ be the germ of a complex  analytic function and set
$(V_f,0)=(f^{-1}(0),0)$. Its singular locus  $(Sing(V_f),0)$
consists of points $\Sigma:=\{x\,:\, \partial f(x)=0\}$.
\ix{singular locus}

Our primary interest is the {\it local structure of $f$}, namely a collection of
invariants and properties containing information about local ambient topological
type of $(V_f,0)$ --- sometimes called the `local Milnor package' of the germ.
To start with, we fix some notations: $B_\epsilon$ is the closed ball in $\bfc^n$ of radius
$\epsilon$ and centered at the origin; $S_\epsilon= S^{2n-1}_\epsilon $ is its boundary $\partial B_\epsilon$;
$D_r$ denotes a complex disc of radius $r$, while $D^2_r$ is a bidisc.  Usually,
$S^k$ denotes the $k$--sphere with its natural orientation,
and  $T^\circ$ the interior of the closed ball or tubular neighbourhood $T$.

The next theorem characterizes the homeomorphism type of the triple $(B_\epsilon, B_\epsilon\cap V_f,0)$
showing its local {\it conic structure}. In the general case of semi--analytic sets it was proved by
Lojasiewicz \cite{Loj}, the case of germs of complex algebraic/analytic
hypersurfaces with isolated singularities  was established by Milnor \cite{MBook}, while the generalization to
non--isolated hypersurface singularities was done  by Burghelea and Verona
\cite{BV}:
\ix{Lojasiewicz} \ix{Milnor} \ix{Burghelea} \ix{Verona} \ix{conic structure}
\ix{link!of a surface singularity|textbf} \ix{singularities!isolated|textbf}

\begin{theorem}\cite{Loj,MBook,BV} There exists $\epsilon_0>0$ with the property that
for any $0<\epsilon\leq \epsilon_0$ the homeomorphism type of $(B_\epsilon,B_\epsilon\cap V_f)$
is independent of $\epsilon$, and is the same as the homeomorphism type  of the real cone on the pair
$(S_\epsilon,S_\epsilon\cap V_f)$, where 0 corresponds to the vertex of the cone.
\end{theorem}
The intersection  $K:=V_f\cap S_\epsilon$ is called  the \textit{link}
of $(V_f,0)$ ($0<\epsilon\leq \epsilon_0$). By Milnor \cite{MBook},
\begin{equation}\label{eq:Kcon1}
\mbox{$K$  is $(n-3)$--connected.}
\end{equation}
Moreover,  $K$ is an oriented manifold provided that $f$ has an isolated singularity.
By the above theorem, the local structure is completely determined by
$K$ and its embedding into the $(2n-1)$--sphere. A partial information about this embedding is provided by the
complement $S_\epsilon \setminus K$. The fundamental result of Milnor in \cite{MBook}
states that it is a locally trivial fiber bundle over the circle.

\begin{theorem}\cite{MBook} There exists $\epsilon_0$ with the property that for any
$0<\epsilon\leq \epsilon_0$ the map $f/|f|:S_\epsilon\setminus K\to S^1=\{z\in\bfc\,:\, |z|=1\}$
is a smooth locally trivial fibration. Moreover, for any such $\epsilon$, there exists $\delta_\epsilon$
with the property that for any $0<\delta\leq \delta_\epsilon$, the restriction $f:B_\epsilon^\circ\cap f^{-1}(
\partial D_\delta)\to \partial D_\delta$ is a smooth locally trivial fibration. Its diffeomorphism type
is independent of the choices of $\epsilon$ and $\delta$. Furthermore, these two fibrations are diffeomorphic.
\end{theorem}

\ix{Milnor!fibration|textbf} \ix{Milnor!fiber|textbf} \ix{Milnor!geometric monodromy|textbf}
\ix{monodromy!geometric|textbf}
Either of the fibrations above is referred to as the {\it local Milnor fibration}
of  the germ $f$ at the origin; its fiber is called the {\it Milnor fiber}.
In this book we will deal mainly with the second fibration.
Let $F_{\epsilon,\delta}:=B_\epsilon\cap f^{-1}(\delta)$ be the
Milnor fiber of $f$ in a small {\em closed} Milnor ball $B_\epsilon$
(where $0<\delta\ll \epsilon$). Sometimes, when $\epsilon$ and
$\delta$ are irrelevant, it will be simply denoted by $F$.
It is a smooth oriented $(2n-2)$--manifold with boundary $\partial F_{\epsilon,\delta}$.
If $f$ has an isolated singularity then $\partial F_{\epsilon,\delta}$ and $K$ are diffeomorphic.

The geometric  monodromy (well defined up to an isotopy) of the Milnor fibration
$\{F_{\epsilon,e^{i\alpha}\delta}\}_{\alpha\in[0,2\pi]}$ is
called the {\it Milnor geometric monodromy} of $F_{\epsilon,\delta}$. It
induces a Milnor geometric monodromy action on the boundary
$\partial F_{\epsilon, \delta}$.  This restriction can be chosen the identity  if
$\Sigma$ is empty or a point, otherwise this action can  be non--trivial.

If $f$ has an isolated singularity then the fibration and its fiber $F$ have some
very pleasant properties.
\begin{theorem}\label{th:ISO} If $ \Sigma=\{0\}$  then the following facts hold.

\vspace{2mm}

(a) \cite{MBook} \ The homotopy type of $F$ is a bouquet (wedge) of
$(n-1)$--spheres; their number, $\mu(f)$, is called the `Milnor number' of $f$.

\vspace{2mm}

(b) \cite{MBook} \ The Milnor fibration on the sphere provides an open book decomposition
of $S_\epsilon$ with binding $K$. In particular, the closure of any fiber
$(f/|f|)^{-1}(e^{i\alpha})$ \, ($\alpha\in[0,2\pi]$) \, is the link $K$ .

\vspace{2mm}

(c) (Monodromy Theorem, see  \cite{BrieskornMon,Clemens,Landman,Lj} for different versions, or
\cite{GR,Kulikov} for a comprehensive discussions) Let $M:H_{n-1}(F,\bfz)\to H_{n-1}(F,\bfz)$
be the algebraic monodormy operator induced by the geometric monodormy, and let $P(t)$ be its
characteristic polynomial. Then all the roots of $P$ are roots of unity. Moreover, the size
of the Jordan blocks of $M$ for eigenvalue $\lambda\not=1$ (respectively $\lambda=1$)
is bounded by $n$ (respectively by $n-1$).
\end{theorem}
\ix{Milnor!number|textbf} \ix{monodromy!Theorem}

\begin{example}\label{ex:planecurves}
If $n=2$, then $(V_f,0)\subset (\bfc^2,0)$ is called plane curve singularity.
Even if $f$ is isolated, it might have several local irreducible components, let their number be
$\#(f)$. Then the Milnor number $\mu(f)$ satisfies Milnor's identity \cite{MBook}:
\begin{equation}\label{eq:MILNORDELTA}
\mu(f)=2\delta(f)-\#(f)+1,
\end{equation}
where $\delta(f)$ is the {\it Serre--invariant}, or {\it delta--invariant},  or the number of
double points concentrated at the singularity \cite{Serre}.
\ix{Serre--invariant} \ix{singularities!plane curve|textbf}

$\delta(f)=0$ if and only if $\mu(f)=0$,
if and only if $f$ is smooth.
\end{example}

For more details about invariants of plane curve singularities, see \cite{BrKn,Wall}; about
 the topology of isolated hypersurface singularities in general,
see \cite{AGV,MBook,SeBook}.

For the convenience of the reader we recall the definition of the
{\it open book decomposition}  as well.

\begin{definition}\label{def:OBD}
An open book decomposition of a smooth manifold $M$ consists of a codimension 2 submanifold $L$,
embedded in $M$ with trivial normal bundle, together with a smooth fiber bundle decomposition of its complement
$p:M\setminus L\to S^1$. One also requires a trivialization of the a tubular neighborhood of $L$ into the form
$L\times D$ such that the restriction of $p$ to $L\times (D\setminus 0)$ is the map $(x,y)\mapsto y/|y|$.

The submanifold $L$ is called the  {\it binding} of the open book, while the fibers of $p$ are the {\it pages}.
\end{definition}
\ix{open book decomposition|textbf}
\ix{open book decomposition!binding|textbf} \ix{open book decomposition!pages|textbf}

\bekezdes
The above Theorem  \ref{th:ISO} about isolated hypersurface
singularities became a model  for the investigation of non--isolated hypersurface germs as well.
For these germs similar statements are still valid in some weakened forms.  For example,
$F$ is a parallelizable manifold of real dimension $2n-2$,  it has the homotopy type of
a finite CW--complex of dimension $n-1$ \cite{MBook}. Moreover,  by a result of Kato and Matsumoto \cite{KM},
\begin{equation}\label{eq:KM}
\mbox{$F$ is $(n-2-\dim\Sigma)$--connected. }
\end{equation}
\ix{Kato} \ix{Matsumoto} \ix{singularities!non--isolated|textbf}
For example, if $n=3$ and $\dim\Sigma=1$ then $F$ is a connected finite CW--complex of dimension 2,
and $K$ is connected of real dimension 3. This is the best one can say using the general theory.
Since all the spaces $K$, $F$, and $\partial F$ might have non--trivial fundamental groups, these
spaces  are
extremely good  sources for codifying important information, but their study and complete characterization
is much harder than the study of simply connected spaces.

\section{\ Germs  with 1--dimensional singular locus. Transversal type}\labelpar{ss:2.0b}
In this section  we restrict ourself to the case
of a complex  analytic  germ $f:({\bfc}^3,0)\to ({\bfc},0)$ whose singular locus $(\Sigma,0)$
 is 1--dimensional.  Denote by  $\cup_{j=1}^s \Sigma_j$ the
decomposition of $\Sigma$ into irreducible components. As we have already mentioned, in this case
$K$ is singular too: its singular part is
$L=\cup_jL_j$, where $L_j:=K\cap \Sigma_j$.
\ix{transversal!equisingulaity type|textbf}

An important ingredient of the topological description of the germ $(V_f,0)$ and of the
local Milnor fibration is the
collection of {\it transversal type singularities},  $T\Sigma_j$, associated with
the components  $\Sigma_j$ ($j\in \{1,\ldots, s\}$) \cite{Io,Le2};  their  definition follows.

$T\Sigma_j$ is the equisingularity type (that is, the embedded topological type)
of the local plane curve singularity
$f|_{Sl_q}:(Sl_q,q)\to (\bfc,0)$, where $q\in \Sigma_j\setminus
\{0\}$, and $(Sl_q,q)$ is a transversal smooth complex 2--dimensional
slice-germ of $\Sigma_j$ at $q$.  The topological type of $f|_{Sl_q}$
is independent of the choice of $q$ and
$Sl_q$. Similarly, its Milnor fiber
$(f|_{Sl_q})^{-1}(\delta)\subset Sl_q$ is independent of $q$ and
$\delta$ (for $\delta$ small), and it will be denoted by $F_j'$.
We write $\mu'_j$ for the  Milnor number, $\#T\Sigma_j$ for
the number of irreducible components, and $\delta'_j$ for the Serre-invariant.
\ix{Milnor!number!of transversal type|textbf}

The monodromy diffeomorphism of $F_j'$, induced by the family
$[0,2\pi]\ni\alpha\mapsto (f|_{Sl_q})^{-1}(\delta e^{i\alpha})$, is
called the \ix{transversal!monodromy!horizontal|textbf} \emph{horizontal monodromy of $F'_j$}, and is denoted
by $m'_{j,hor}$. The diffeomorphism  induced by the family
$s\mapsto (f|_{Sl_{q(s)}})^{-1}(\delta)$,  above an oriented
simple loop $s\mapsto q(s)\in \Sigma_j\setminus  \{0\}$, which
generates $\pi_1(\Sigma_j\setminus \{0\})=\Z$ (and with $\delta$
small and fixed), is called the \ix{transversal!monodromy!vertical|textbf} \emph{vertical monodromy}. It is
denoted by $m'_{j,ver}$. Both monodromies are well-defined up to
isotopy, and they commute up to an isotopy.
\ix{monodromy!transversal|see{transversal/monodromy}}
\ix{transversal!monodromy|textbf}
\ix{monodromy!horizontal!transversal|textbf}\ix{monodromy!vertical!transversal|textbf}

For a possible explanation of the names 'horizontal/vertical', see \ref{HORVER}.

The primary goal of the present work is the study of the {\it boundary $\partial F$ of the
Milnor fiber $F$}, although sometimes we will provide  results regarding the Milnor
fiber itself or about the ambient topological type as well (but these will be  mostly
immediate consequences of results regarding the boundary of $F$).

\section{\ The decomposition of the boundary of the Milnor fiber}
\labelpar{ss:2.1}
\setcounter{equation}{0}
As in \ref{ss:2.0b}, let $\partial F$ denote the boundary of the
Milnor fiber $F$  of $f$, that is  $\partial F=\partial F_{\epsilon,\delta}=
S_\epsilon \cap f^{-1}(\delta)$.
By the above discussion it is a smooth oriented 3--manifold.
\ix{Milnor!fiber!boundary|textbf} \ix{Siersma} %\ix{Milnor!fiber!boundary!decomposition}

Siersma in \cite{Si} provides  a natural
decomposition $$\partial F=\partial_1 F\cup\partial_2 F,$$ which will be described next.

Let $T(L_j)$ be a small closed tubular neighbourhood of $L_j$ in
$S_\epsilon$, and denote by $T^\circ(L_j)$ its interior. Then
$\partial_2F=\cup_j\partial_{2,j}F$ with $\partial_{2,j}F=\partial
F\cap T(L_j)$, and  $\partial_1F=\partial F\setminus
\cup_jT^\circ(L_j)$. The parts $\partial_1F$ and $\partial_2F$ are glued
together along their boundaries, which is a union of tori.
%(For more details, see \cite{Si}.)

\begin{theorem}\labelpar{prop:sier}\cite{Si}
\begin{enumerate}
\item\labelpar{prop:sier1}
For each $j$, the natural projection $T(L_j)\to L_j$
induces a locally trivial fibration of $\partial_{2,j}F$
over $L_j$ with fiber $F_j'$  (the Milnor fiber of $T\Sigma_j$)
 and monodromy $m'_{j,ver}$ of $F_j'$. This induces a  fibration of
$\partial(\partial_{2,j}F)$ over $L_j$ with fiber $\partial F_j'$.

\vspace{2mm}

\item\labelpar{mon}
The Milnor monodromy of $\partial F$ can be chosen in such a way that it preserves
both $\partial_1F$ and $\partial_2F$. Moreover, its restriction on
$\partial_1F$ is trivial, and also on the gluing tori \ix{gluing tori|textbf}
$\partial(\partial_1F)=-\cup_j\partial(\partial_{2,j}F)$.

\vspace{2mm}

\item\labelpar{prop:sier2}
The Milnor monodromy on $\partial_2F$ might be nontrivial.
This monodromy action  on each $\partial_{2,j}F$ is induced by the horizontal
monodromy $m'_{j,hor}$ acting on $F_j'$.
(Since it commutes with $m'_{j,ver}$,
it induces an action on the total space $\partial_{2,j}F$ of the bundle described in part 1.)
\end{enumerate}%\ix{transversal!monodromy}
\end{theorem}

Notice that by Theorem \ref{prop:sier}(\ref{prop:sier1}), since $F'_j$
is connected, we  also obtain that
\begin{equation}\label{eq:CONj}
\mbox{$\partial_{2,j}F$ is connected.}\end{equation}

\begin{remark}\labelpar{rem:link}
One has the following relationship connecting the boundary
$\partial F$ and the link $K$. Definitely, $K$ has a similar
decomposition $K=K_1\cup K_2$, where $K_2=\cup_jK_{2,j}$,
$K_{2,j}:=T(L_j)\cap K$, and $K_1:=K\setminus \cup_j T^\circ
(L_j)$. Then $K_1\approx \partial _1F$, hence $\partial
K_1\approx \partial(\partial _1F)$ as well. On the other hand, each $K_{2,j}$
has the homotopy type of $L_j$. More precisely, the homeomorphism type of
$K$ can be obtained from $\partial F$ by the following `surgery':
one replaces  each $\partial_{2,j}F$ ---
considered as the total space of a fibration with base space $L_j$ and fiber $F'_j$  ---,  by
 a total space of a fibration with base space $L_j$ and whose fiber is the real cone over $\partial F'_j$.
\ix{link!of a surface singularity!decomposition|textbf}\ix{link!of a surface singularity}

In particular, if each transversal type singularity $T\Sigma_j$ is locally irreducible, then $K$ is an
oriented topological 3--manifold (since   the real cone over $\partial F'_j$ is a topological disc).

For another construction/characterization  of $K$ see \ref{prop:KK}.
\end{remark}
\ix{Siersma} \ix{Randell} \ix{sphere!homology}

\ix{link!of a surface singularity}

\begin{corollary}\label{cor:Kcon}
$\partial F$ is connected.
\end{corollary}
\begin{proof} Use Remark~\ref{rem:link} and the fact that $K$ is connected, cf. (\ref{eq:Kcon1}).
\end{proof}

\begin{remark}\label{re:SiRa} Consider a germ $f$ as above with 1--dimensional
singular locus and let $M_q:H_q(F,\Z)\to H_q(F,\Z)$ be the monodromy operators acting on the
homology of the Milnor fiber. Furthermore, let $M'_{j,ver}$ be the algebraic vertical transversal
monodromy induced by $m'_{j,ver}$.\ix{transversal!monodromy}\ix{monodromy}

Randell and Siersma  in the articles \cite{Ra,Si,Si3} determined the homology of the link $K$
and of the boundary $\partial F$  for several cases. Moreover,
they  characterized via different criteria  those situations when the link  $K$ and $\partial F$ are
homology spheres:

\begin{enumerate}
\item \cite[(3.6)]{Ra} \ $K$ is  a homology sphere if and only if $\det(M_q-I)=\pm 1$ for $q=1,2$.
\vspace{2mm}

\item  \cite{Si,Si3} \ $\partial F$ is a homology sphere if and only if
$\det(M_2-I)=\pm 1$ and $\det(M'_{j,ver}-I)=\pm 1$ for any $j$.\vspace{2mm}

\item  \cite{Si} (in  \cite{Si} attributed to Randell too) \
$\partial F$ is a homology sphere if and only if $K$ is homology sphere and
$\det(M'_{j,ver}-I)=\pm 1$ for any $j$.
\end{enumerate}

In fact,  these statements were proved for arbitrary dimensions. For more comments on these properties see
\ref{O9}.
\end{remark}

\chapter{The topology of a pair $(f,g)$}\labelpar{s:ICIS}

\section{\ Basics of ICIS. \ Good representatives}\labelpar{ss:ICIS}
\setcounter{equation}{0} \ix{ICIS|textbf}
In many cases it is convenient to add to the germ $f$ another germ,
say $g$, such that the pair $(f,g)$ forms an \emph{isolated
complete intersection singularity} (ICIS in short). Traditionally,
one studies  the $g$-\textit{polar geometry} of $f$ in this way,
generalizing the classical polar geometry, when $g$ is a
generic linear form. This method, suggested by Thom and developed by
L\^e D\~ung Tr\'ang \cite{Le1,Le2,LeMon} and Teissier \cite{Te0,Te},
computes certain  invariants of $f$ by {\it induction on the dimension}.
This lead to the polar invariants of Teissier,  the carrousel description of the
monodromy by L\^e, and later to the study   of certain invariants of series and `composed singularities'
by Siersma \cite{Si2} and the first author \cite{NPhg,Ne126,Neta}, or of L\^e cycles and numbers   by
Massey \cite{MI1,MI2,M}. These techniques  have  generalizations
in the theory of one-parameter and equisingular deformations, initiated by
Zariski and Teissier, producing great results such as the L\^e Ramanujam Theorem \cite{LR} and recent work
of Fern\'andez  de Bobadilla, see \cite{deB1,deB2} and references therein.
\ix{singularities!ICIS|see{ICIS}}\ix{Thom}\ix{L\^e}\ix{Teissier}\ix{Siersma}\ix{Massey}
\ix{Zariski!Main Theorem}\ix{Bob@de Bobadilla}\ix{monodromy}
\ix{Ramanujam}\ix{equisingularity}\ix{polar}\ix{singularities!series}\ix{singularities!composed}

However, as an independent strategy,
the germ $g$ might also serve as an auxiliary object
 to determine abstract {\it $g$-independent invariants} of $f$.
For example, this book contains the  description of   $\partial F$ in terms of a
pair $(f,g)$.

In all the above  methods the key ingredient is the fiber structure of the ICIS $(f,g)$.

\vspace{2mm}

First,  we provide some basic definitions and properties of isolated complete intersection singularities.
Although, they are defined generally in the context of germs
 $(\bfc^n,0)\to (\bfc^k,0)$, we will keep
our specific dimensions $n=3$ and $k=2$; in this way we also fix the basic notations we will need.

For more details regarding this section see the book of Looijenga \cite{Lj}.
\ix{Looijenga}

\bekezdes
 Consider an analytic germ $\Phi=(f,g):(\bfc^3,0)\to (\bfc^2,0)$.
The map $\Phi$ defines an ICIS if the following property holds:  if
 $I\subset\co_{\bfc^3,0}$ denotes  the ideal generated by $f,g$ and
the $2\times 2$ minors of the Jacobian matrix $(d\Phi)$ in the local algebra $\co_{\bfc^3,0}$ of
convergent analytic germs $(\bfc^3,0)\to (\bfc,0)$, then
 $\dim \co_{\bfc^3,0}/I<\infty$. In other words,
the scheme-theoretical intersection $\Phi^{-1}(0)=
\{ f=0\}\cap\{ g=0\}$ has only an isolated singularity at the origin.
In particular,   $\ (Sing(V_f))\cap V_g= \{ 0\}$ and
 $ V_f$ intersects $V_g$ in the complement of the
origin transversally in a smooth punctured curve.

\begin{example}\label{ex:ICIS}
Let $f:(\bfc^3,0)\to (\bfc,0)$ be an analytic germ with 1--dimensional critical locus.
Then any generic {\it linear} form  $g:(\bfc^3,0)\to (\bfc,0)$ has the property that
the pair $\Phi=(f,g)$ forms an ICIS.

In particular, any such $f$ can be completed to an ICIS $\Phi=(f,g)$.
\end{example}

 The {\em critical locus} $(C_\Phi,0)$  of $\Phi$ is the set of points of
$(\bfc^3,0)$, where $\Phi$ is not a local
submersion. Its image
$(\Delta_{\Phi},0):=\Phi(C_\Phi,0)\subset(\bfc^2,0)$ is called the
{\em discriminant locus}  of $\Phi$.
\ix{ICIS!critical locus|textbf} \ix{ICIS!discriminant|textbf}

\begin{lemma}\label{icisfib} \cite[2.B]{Lj} \
 Fix a germ  $\Phi$ as above.
 Then  there exist a sufficiently small closed ball
$B_\epsilon\subset\bfc^3$  of radius $\epsilon$, and a bidisc  $D^2_\eta\subset\bfc^2$ with radius
 $0<\eta\ll\epsilon$ such that:

\begin{enumerate}
\item   the set $(\Phi^{-1}(0)\setminus\{0\})\cap B_\epsilon$
is non-singular;\vspace{2mm}

\item  $\partial B_{\epsilon'}$ intersects $\Phi^{-1}(0)$
      transversally for all ~$0<\epsilon'\leq\epsilon$;\vspace{2mm}

\item $C_\Phi\cap\Phi^{-1}(D^2_\eta)\cap\partial B_\epsilon=\emptyset$;
and   the restriction of $\Phi$ to
      $\Phi^{-1}(D^2_\eta)\cap\partial B_\epsilon$ is a submersion.
\end{enumerate}\end{lemma}
\begin{definition}\label{goodrepr}
The map $\Phi:\,\Phi^{-1}(D^2_\eta)\cap B_\epsilon\rightarrow
D^2_\eta$ with the above properties is called a {\it `good' representative} of the
ICIS $\Phi$.  In the sequel, in the presence of such a good representative,
$\Sigma_\Phi$ will denote the intersection of the critical locus with $\Phi^{-1}(D^2_\eta)$ and
  $\Delta_\Phi$  its image in $D^2_\eta$.
\end{definition}
\ix{ICIS!good representative|textbf}

Also, we prefer to denote  the local coordinates of $(\bfc^2,0)$  by $(c,d)$.

\vspace{2mm}

With these notations one has the following fibration theorem,
cf. 2.8 in \cite{Lj}:

\begin{theorem}\label{fifibration} \

(i)\
$\Phi:B_\epsilon\cap\Phi^{-1}(D^2_\eta)\longrightarrow D^2_\eta$
is proper. The analytic sets
 $\Sigma_\Phi$ and $\Delta_\Phi$ are 1--dimensional,
and the restriction $\Phi|_{\Sigma_\Phi}:\,\Sigma_\Phi\rightarrow\Delta_\Phi$
is proper with finite fibers.

(ii)\  $\Phi:
(\Phi^{-1}(D^2_\eta-\Delta_\Phi)\cap B_\epsilon,
\Phi^{-1}(D^2_\eta-\Delta_\Phi)\cap \partial B_\epsilon)
\rightarrow D^2_\eta-\Delta_\Phi$ is a smooth
locally trivial fibration of a pair of spaces.
\end{theorem}
\ix{ICIS!Milnor fibration|textbf} \ix{ICIS!monodromy representation|textbf}
\ix{monodromy!representtaion of ICIS|textbf}

\begin{definition}\label{milnorfibre}
A fiber $\fcd=\Phi^{-1}(c,d)\cap B_\epsilon$,
 for $(c,d)\in D^2_\eta-\Delta_\Phi$,  is
called a {\it Milnor fiber of $\Phi$},  while the fibration itself
is referred to as {\it the Milnor fibration of $\Phi$}. The fiber sometimes is denoted by
$F_\Phi$ too.

For any fixed base point $b_0=(c_0,d_0)\subset D^2_\eta-\Delta_\Phi$,
one has the natural {\bf geometric monodromy representation}:
$$m_{geom,\Phi}:\, \pi_1(D^2_\eta-\Delta_\Phi, b_0)\ \longrightarrow\
\mbox{\it Diff}\,^\infty (F_{b_0})/isotopy.$$
It induces the  {\bf algebraic  monodromy representation}
$$M_{\Phi}:\, \pi_1(D^2_\eta-\Delta_\Phi, b_0)\ \to
 Aut H_*(F_{b_0},{\bf Z}).$$
\end{definition}

%The Milnor fiber $\fcd$ of $\phi$ is connected
\begin{proposition}\labelpar{prop:ICIScon}\cite{Lj} \
The Milnor fiber $\fcd$ of ~$\Phi$~ is connected.
\end{proposition}
\bekezdes\label{bek:dj}
%Let $\Delta_\Phi=\Delta_1\cup...\cup\Delta_t$ denote the decomposition
%of  $\Delta_\Phi$  into irreducible components.
If $f$ has a 1-dimensional singular locus, then the singular locus $\Sigma=\Sigma_f$
 is a subset of $\Sigma_\Phi$ and
$\Phi(\Sigma)=\{c=0\}$ is an irreducible component of the discriminant
$\Delta_\Phi$. By convention, we denote this component by $\Delta_1$.
Then $\Phi^{-1}(\Delta_1)\cap \Sigma_\Phi$ is exactly $\Sigma$.
Recall that the irreducible components of $\Sigma$ are denoted by $\{\Sigma_j\}_{j=1}^s$.
Part (i) of Theorem \ref{fifibration}  guarantees that
the restriction $\Phi: \Sigma_{j}\rightarrow\Delta_1$ is a branched covering
for any $1\leq j\leq s$. Let $d_{j}$ denote its degree.
Note that this agrees with the
degree of the restriction of the map $g$ to $\Sigma_j$.

In general,  it is extremely difficult to determine either the geometric monodromy
$m_{geom,\Phi}$, or the algebraic monodromy representation
$M_{\Phi}$, hence it is hard to recover information about the
global Milnor fibration. This is  mainly due to the fact that the fundamental group
$\pi_1 (D^2_\eta-\Delta_\Phi, b_0)=
\pi_1(\partial D^2_\eta\setminus \Delta_\Phi, b_0) $
is non--abelian, in general. Nevertheless, the fundamental group of  a small tubular neighbourhood of
$\partial D^2_\eta\cap \Delta_1$ in  $\partial D^2_\eta$ is abelian, hence the fiber structure above it can be
understood more easily.  The  representation restricted to the fundamental group of this  tubular neighbourhood
still contains key information about the geometry of the fibration `near $\Delta_1$', hence about the
singular locus of $f$.

The next definition targets this restriction of the  representation.

For any fixed $c_0$, set $D_{c_0}:=\{c=c_0\}\cap D^2_\eta$. Then,
if $|c_0|\ll \eta$, the circle $\partial D_{c_0}$ is disjoint from
$\Delta_\Phi$. Consider the torus $T_\delta:=\cup_{c_0} \partial
D_{c_0}$, where the union is over $c_0$ with $|c_0|=\delta>0$.
Hence, for $0<\delta\ll\eta$, the restriction of $\Phi$ on
$\Phi^{-1}(T_\delta)$ is a fiber bundle with fiber $F_\Phi$.

\begin{definition}\label{hvmonodromy}
The monodromy above a circle in $T_\delta$, consisting of points with fixed $d$--coordinates,
 is called the \textit{horizontal monodromy} of $\Phi$
near $\Delta_1$, and it is denoted by $m_{\Phi,hor}$. Similarly,  the
monodromy above a circle in $T_\delta$, consisting of points with fixed  $c$--coordinates,
(e.g., above $\partial D_{\delta}$) is the \textit{vertical
monodromy} of $\Phi$ near $\Delta_1$; it is  denoted by $m_{\Phi,ver}$.
\end{definition}
They are defined up to an isotopy, and they commute up to an isotopy.
\ix{ICIS!horizontal monodromy|textbf}\ix{ICIS!vertical monodromy|textbf}
\ix{monodromy!horizontal!of an ICIS|textbf}\ix{monodromy!vertical!of an ICIS|textbf}

\begin{remark}\label{HORVER}
Usually, in our figures, in  $D^2_\eta$ we take the $c$--coordinate as the horizontal, while the
$d$--coordinate  as the vertical axis. Hence, the circle in $T_\delta$ with $d$ constant is a
`horizontal' circle, while $c$  a  circle with $c$ constant is  `vertical'.
\end{remark}

\begin{remark}\labelpar{rem:phif}
Set $(V_g,0):=g^{-1}(0) $ in $(\bfc^3,0)$.
Clearly, $F_{\epsilon,\delta}$ and $\Phi^{-1}(D_\delta)$ can be identified, where the second space is considered in
$B_\epsilon\cap \Phi^{-1}(D^2_\eta)$, and its `corners' are smoothed.
Under this identification, $F_{\epsilon,\delta}\cap V_g$ corresponds to
$\Phi^{-1}(\delta,0)$. Hence $\partial F_{\epsilon,\delta}$ and $\partial \Phi^{-1}(D_\delta)$ can also be identified
in such a way that $\partial F_{\epsilon,\delta}\cap V_g$ corresponds to
$\partial\Phi^{-1}(\delta,0)$. Notice that
$\partial \Phi^{-1}(D_\delta)$ consists of two parts, one of them being
$\Phi^{-1}(\partial D_\delta)$, the other the complement of the  interior of
$\Phi^{-1}(\partial D_\delta)$ defined as
$$\partial' \Phi^{-1}(D_\delta):=\cup_{(\delta,d)\in D_{\delta}}\partial
(\Phi^{-1}(\delta,d)).$$
By   triviality over $D_\delta$ of the family $\cup_{(\delta,d)\in D_{\delta}}\partial
(\Phi^{-1}(\delta,d))$, the part
 $\partial' \Phi^{-1}(D_\delta)$ is diffeomorphic to the product
$D_\delta \times \partial \Phi^{-1}(\delta,0)$.
\end{remark}

\section{\ The Milnor open book decompositions of  $\partial F$}
\labelpar{be:fibr}\setcounter{equation}{0}
%\begin{bekezdes}
Besides the geometry of the ICIS, the germ $g$  provides a
different package as well. These are  invariants determined by
the \emph{generalized Milnor fibration}, or open book decomposition, induced by $\arg(g)=g/|g|$
on $\partial F$.

Before we describe it, we recall an  immediate natural generalization of
Milnor's result \ref{th:ISO}\,(b) valid for isolated complete intersections, which was established by Hamm.
Although, again, the result is valid for any map $(\bfc^n,0)\to (\bfc^k,0)$; we state it only for
$n=3$ and $k=2$.\ix{Hamm}
\begin{theorem}\label{th:Hamm}\cite{Hamm}
Assume that $\Phi=(f,g):(\bfc^3,0)\to (\bfc^2,0)$ is an ICIS and $f$ has an isolated singularity at the origin.
Let $K_f$ be the link of $f$ in a sufficiently small sphere $S_\epsilon$. Then $g$ defines an open book decomposition
in $K_f$ with binding $K_f\cap V_g$ and fibration $g/|g|:K_f\setminus V_g\to S^1$. The pages are diffeomorphic to the
fibers of $\Phi$.
\end{theorem}
\ix{Milnor!fibration}

At first glance it is not immediate what  the right generalization of Hamm's
result would be for the case when $f$ has 1--dimensional singular locus, since the link is singular.

The generalization was established by Caubel. The next results
are either proved or follow from the statements proved in \cite{CC}:

\begin{theorem}\labelpar{CC} \
\begin{enumerate}
\item The argument of the restriction of $g$ on $\partial
F_{\epsilon,\delta}\setminus V_g$  defines an open book
decomposition on $\partial F_{\epsilon, \delta}$ with binding $\partial F_{\epsilon, \delta}\cap V_g$,
and fibration
$g/|g|: \partial F_{\epsilon,\delta}\setminus V_g \to S^1$.

\vspace{2mm}

\item The fibration $g/|g|: \partial F_{\epsilon,\delta}\setminus V_g \to S^1$
is equivalent to the fibration $\Phi:\Phi^{-1}(\partial
D_{\delta})\to \partial D_\delta$ with monodromy $m_{\Phi,ver}$ ($0<\delta\ll\eta$).

\vspace{2mm}

\item Moreover, this structure is compatible with the action of the Milnor monodromy
on $\partial F_{\epsilon,\delta}$ in the following sense.  The restriction of the Milnor
monodromy of $\partial F_{\epsilon,\delta}$ on
a tubular neighbourhood of $\partial F_{\epsilon,\delta}\cap V_g$ is trivial, and its restriction on
$\partial F_{\epsilon,\delta}\setminus V_g$
is equivalent to the horizontal monodromy of
$\Phi^{-1}(\partial D_\delta)$ over the oriented circle $\{|c|=\delta\}$
(induced by the local trivial family
$\{\Phi^{-1}(\partial D_c)\}_{|c|=\delta}$).
\end{enumerate}\end{theorem}
\ix{Caubel}

\begin{proof}
The first part is stated and proved in Proposition 3.4 of
\cite{CC}. Although the second part is not stated in
[loc.cit.], it follows from the proof of Proposition 3.4. and by
similar arguments as the proof of Theorem 5.11  in Milnor's book \cite{MBook}.
The last monodromy statement can be  proved  the same way.
\end{proof}

\section{\ The decomposition of $\partial F$ revisited}\labelpar{felo}
\setcounter{equation}{0}
 Next, we present how one can
recover the decomposition $
\partial F_{\epsilon,\delta}=\partial_1F\cup \partial_2F$,
cf. \ref{ss:2.1}, from the structure of $\Phi$
via the identification of
$ F_{\epsilon,\delta}$ with $\Phi^{-1}(D_\delta)$, cf. \ref{rem:phif}.
\ix{Milnor!fiber!boundary}

For any $j\in\{1,\ldots, s\}$, let $T^\phi_j$ be a small closed
tubular neighbourhood in $\bfc^3$ of $\Phi^{-1}(\partial
D_0)\cap \Sigma_j$. Then, for $\delta$ sufficiently small, and for any
$(\delta,d)\in \partial D_\delta$, the fiber $\Phi^{-1}(\delta,d)$
intersects $\partial T^\phi_j$ transversally. In particular,
$\Phi^{-1}(\partial D_\delta)\cap T^\phi_j$ and
 $\Phi^{-1}(\partial D_\delta)\setminus (\cup_j T^\phi_j)$
are fiber bundles over $\partial D_\delta$.
%\end{bekezdes}

\begin{proposition}\labelpar{felbo} One has the following facts:
\begin{enumerate}
\item
There is an orientation preserving homeomorphism
$$\partial F_{\epsilon,\delta}\longrightarrow  \Phi^{-1}(\partial D_\delta)\cup \partial' \Phi^{-1}(D_\delta)$$
which sends a tubular neighbourhood $T(V_g)$ of $V_g$  onto
$ \partial' \Phi^{-1}(D_\delta)$ and $\partial F_{\epsilon,\delta}
\setminus T^\circ (V_g)$ onto $\Phi^{-1}(\partial D_\delta)$
(identifying even their fiber structures, cf. \ref{CC}). Under this
identification, $T(V_g)\subset \partial_1F_{\epsilon,\delta}$,
and the fibration $$g/|g|:\partial_1F_{\epsilon,\delta}\setminus
T^\circ(V_g)\to S^1 $$  corresponds to
 $$\Phi:\Phi^{-1}(\partial D_\delta)\setminus (\cup_j T^{\phi,\circ}_j)\to \partial D_\delta,$$
while $$g/|g|:\partial_{2,j}F_{\epsilon,\delta}\to S^1 \ \ \ \ \mbox{to} \ \ \ \
\Phi:\Phi^{-1}(\partial D_\delta)\cap T^\phi_j\to \partial D_\delta.$$

The identifications are compatible with the action of the Milnor/horizontal
monodromies (over the circle $|c|=\delta$).\vspace{2mm}

\item For each $j\in\{1,\ldots, s\}$,
the fibration $g/|g|:\partial_{2,j}F_{\epsilon,\delta}\to S^1$
can be identified with the pullback of the
fibration $\partial_{2,j}F_{\epsilon,\delta}\to L_j$
%with fiber $F'_j$ and monodromy $m'_{j,ver}$
( cf. \ref{prop:sier}) under the map
$\arg(g|L_j):L_j\to S^1$, which is a regular cyclic covering of $S^1$
of degree $d_j$. Therefore, the fiber of
$g/|g|:\partial_{2,j}F_{\epsilon,\delta}\to S^1$
is a disjoint union of $d_j$ copies of $F'_j$ and the monodromy of this fibration is
$$m^\Phi_{j,ver}
(x_1,\ldots, x_{d_j})= (m'_{j,ver}(x_{d_j}), x_1,\ldots, x_{d_j-1}).$$
The action of the Milnor monodromy on $\partial_{2,j}F_{\epsilon,\delta}$
restricted to the fiber of $g/|g|$ is the `diagonal' action:
$$m^\Phi_{j,hor}
(x_1,\ldots, x_{d_j})= (m'_{j,hor}(x_1),\ldots, m'_{j,hor}(x_{d_j})).$$
\end{enumerate}
\end{proposition}
\begin{proof}
The first part follows by combining arguments of \cite{MBook} and \cite{CC}, as in
the second part of
\ref{CC}. The point is that when we `push out' $\Phi^{-1}(\partial D_\delta)$
along the level sets of $\arg(g)$ into $\partial F_{\epsilon,\delta}$, this
can be done by a vector field which preserves a tubular neighbourhood of each $\Sigma_j$.
The second part is standard, it reflects the fiber structure of $\Phi$,
see e.g. \cite{NPhg} or \cite{Si2}.
\end{proof}
\ix{monodromy!vertical}\ix{monodromy!horizontal}

Above, by Theorem \ref{CC} and by the structure of open books, the fibrations
$\partial F_{\epsilon,\delta}\setminus V_g$ and
$\partial F_{\epsilon,\delta}\setminus T^\circ(V_g)$ over $S^1$, both induced by
$g/|g|$,  are equivalent.  A similar fact is true for the fibrations
$\partial_1 F_{\epsilon,\delta}\setminus V_g$ and
$\partial_1 F_{\epsilon,\delta}\setminus T^\circ(V_g)$.

\section{\ Relation with the normalization of the zero locus of $f$}\labelpar{d1}\setcounter{equation}{0}
The space $\partial _1F$ and the fibration $g/|g|:\partial_1F_{\epsilon,\delta}\setminus
V_g\to S^1 $ have another `incarnation' as well.
\ix{normalization of $V_f$}

In order to see this, let $n:(V_f^{norm},n^{-1}(0))\to (V_f,0)$ be
the normalization of $(V_f,0)$.  For the definition and general properties of the normalization of
2--dimensional analytic spaces, see the book of Laufer \cite{La} or the monograph of
L. Kaup and B. Kaup \cite{KK} . Note that
each local irreducible
component of  $(V_f,0)$ lifts to a connected component of the
normalization, hence $(V_f^{norm},n^{-1}(0))$ stands  here for a
multi--germ of normal surface singularities.
Moreover, any normal surface singularity has at most an isolated singularity, but usually this germ
is not a hypersurface germ, its embedded dimension can be arbitrarily large.

If $(X,0)$ is  an irreducible normal surface singularity, represented in some affine space, say
$(X,0)\subset (\bfc^N,0)$, then similarly as for hypersurface singularities one defines its link
 $K_X$ as $X\cap S_\epsilon\subset S_\epsilon\subset \bfc^N$,
 for $\epsilon$ sufficiently small \cite{La,Lj}. Moreover,
$K_X$  is connected  (see, for example, \cite[4.1]{La}).

 Furthermore,  if  $g:(X,0)\to (\bfc,0)$ is an
analytic non--constant germ on $(X,0)$, then similarly as in the cases of Milnor \ref{th:ISO}(b)
 and Hamm \ref{th:Hamm} one gets an open book decomposition of $K_X$ with binding
 $K_X\cap V_g$ and projection $g/|g|:K_X\setminus V_g\to S^1$ \cite{Hamm,LeMon,CSS}.

\vspace{2mm}

Let us return  to our situation.
We denote the link of the multi-germ $(V_f^{norm},n^{-1}(0))$ by $K^{norm}$. It is the
disjoint union of all the links of the components of $(V_f^{norm},n^{-1}(0))$.
Consider as well the lifting
$g\circ n:(V_f^{norm},n^{-1}(0))\to (\bfc,0)$ of $g$, which
determines  an open book decomposition on $K^{norm}$ with binding  $K^{norm}\cap V_{g\circ n}$
and Milnor fibration $$\arg(g\circ n):K^{norm}\setminus V_{g\circ n}\to S^1.$$
Furthermore,  for any $j\in\{1,\ldots,s\}$
let us denote by $St(\Sigma_j)\subset V_f^{norm}$ the strict inverse image of $\Sigma_j$,
that is the closure of $n^{-1}(\Sigma_j\setminus 0)$.
Set $St(\Sigma):=\cup_j St(\Sigma_j)$.
Then,  $K^{norm}\cap V_{g\circ n}\cap St(\Sigma)
=\emptyset$, and
$$\arg(g\circ n):(K^{norm}\setminus V_{g\circ n}, St(\Sigma))\to
S^1$$
is a locally trivial fibration of a pair of spaces.
\ix{link!of a surface singularity!of normalization of $V_f$}

Usually $St(\Sigma_j)$ is {\it not} irreducible. An  upper bound for the number of its irreducible
components is the number of components of $\partial F'_j$, or equivalently, the number of irreducible
branches $\#T\Sigma_j$ of the local transversal type $T\Sigma_j$. Nevertheless,  $|St(\Sigma_j)|$  can
sometimes be strictly smaller, see the discussion and examples of  sections \ref{2edg}
or \ref{1es2}. Compare also with \ref{gl}.

\begin{proposition}\labelpar{d2}
Let $T_{St(\Sigma)}$ be a small closed tubular neighbourhood
of $St(\Sigma)$ in $\partial V_f^{norm}$. Then the following facts hold:

\vspace{2mm}

(a) $\partial _1F$ is orientation preserving diffeomorphic to $ K^{norm}\setminus
T^\circ _{St(\Sigma)}$. In particular, the number of connected components of
$\partial_1F$ is the number of irreducible components of $f$.

\vspace{2mm}

(b)
The fibrations of the  pairs of
spaces
$$\arg(g\circ n):(K^{norm}\setminus
(V_{g\circ n}\cup T^\circ _{St(\Sigma)}), \partial T_{St(\Sigma)})\to
S^1$$
and
$$\arg(g):(\partial_1F_{\epsilon,\delta}\setminus
V_g, \partial(\partial_1F_{\epsilon,\delta})) \to S^1 $$
are equivalent.

In particular, for any $j$, the number of tori along which the
connected space $\partial_{2,j}F$ is glued to
$\partial_1F$ agrees with the number of irreducible components of $St(\Sigma_j)$.
\end{proposition}
\begin{proof} The normalization map is an isomorphism above the regular
part of $V_f$.
\end{proof}

The above facts show clearly that $\partial_1F$ is guided by the
link of the normalization, while $\partial_2F$ by the local behaviour
near  $\Sigma$.

\vspace{2mm}

\bekezdes By the results of the above subsections, the fiber $F_{g,\partial F}$
of the fibration  $\arg(g):\partial F\setminus V_g\to S^1$
provided by Theorem  \ref{CC} can be compared with the fiber $F_{g,K^{norm}}$ of the
fibration $\arg(g\circ n):K^{norm}\setminus V_{g\circ n}\to S^1$.

Indeed, by  \ref{felbo} and the above discussion one obtains that  the fiber
$F_{g,K^{norm}}$ intersects $St(\Sigma)$ in $N:=\sum_j
\#T\Sigma_j\cdot d_j$ points.

\begin{corollary}\label{cor:FF} \

\begin{enumerate}
\item The fiber $F_{g,\partial F}$ can be obtained as follows: take the fiber
$F_{g,K^{norm}}$ and the $N$ intersection points of it with $St(\Sigma)$, delete some small disc
neighbourhoods of these points, and then,  for each $j\in\{1,\ldots, s\}$,  glue to  the resulting
surface with boundary  $d_j$ copies of $F'_j$ along their boundaries.\vspace{2mm}

\item In particular, at the level of Euler characteristics, one has
$$\chi(F_{g,\partial F})=\chi(F_{g,K^{norm}})+\sum_jd_j(1-\mu'_j-\#T\Sigma_j)=
\chi(F_{g,K^{norm}})-2\cdot \sum_jd_j\delta_j'.$$

\item Assume that $\dim \Sigma_f=1$.
Then, for any germ $g$ such that $(f,g)$ is an ICIS,
 the Euler characteristics of the pages of the two
open book decompositions induced by the argument of $g$ on
$\partial F$ and $K^{norm}$ are not equal. More precisely, one has the strict inequality:
$$\chi(F_{g,\partial F})<\chi(F_{g,K^{norm}}).$$ \ix{Milnor!fiber!boundary}

\end{enumerate}
\end{corollary}
\begin{proof} (1) follows from the above discussions, (2) rewrites (1) at
the Euler characteristic level and uses the
Milnor identity (\ref{eq:MILNORDELTA}), while (3) follows from the fact that the Serre invariant is strictly
positive for a non-smooth plane curve singularity. \end{proof}
\ix{Serre--invariant}

%\setcounter{temp}{\value{section}}
%\part{Plumbing graphs and their coverings (preliminaries)}%\labelpar{sec:2}
%\setcounter{section}{\value{temp}}

\chapter{Plumbing graphs and oriented plumbed 3--manifolds}\labelpar{s:PLU}
\section{\ Oriented plumbed manifolds}\labelpar{ss:2.2}
The first  goal of the present work is to provide
a plumbing representation of the 3-manifold $\partial F$, where $F$ is the Milnor fiber
of a hypersurface singularity $f:(\bfc^3,0)\to (\bfc,0)$
with 1--dimensional singular locus. The construction will  be compatible with the
decomposition of $\partial F$ into $\partial_1F$ and $\partial_2F$, hence it also provides
plumbing
representations for these oriented 3--manifolds with boundary.\ix{Milnor!fiber!boundary}

Even more, for any  $g$ such that the pair $(f,g)$ forms an ICIS,  as in \ref{ss:ICIS}, we will also
provide a  plumbing representation of the pair $(\partial F,\partial F\cap V_g)$
and of the multiplicity system of the generalized Milnor
fibrations $\partial F\setminus V_g$, $\partial_1 F\setminus V_g$
and $\partial_2 F$ over $S^1$ induced by $g/|g|$.

In this Chapter we recall the necessary  definitions and relevant constructions.
Regarding   plumbed 3-manifolds and plumbing
calculus we follow Neumann's seminal article \cite{neumann}
with small modifications, which will be explained  below.
\ix{Neumann}

\bekezdes\label{bek:pl} {\bf The plumbing graph.} \
For any graph $\Gamma$, we denote the set of
vertices by $\cvg$ and the set of edges by $\evg$. If there is no danger of
confusion, we denote them simply by ${\mathcal V}$ and ${\mathcal
E}$.

In the case of plumbing graphs of {\it closed} 3-manifolds, any
vertex has two decorations, both integers: one of them is the  Euler
obstruction, or \textit{`Euler number'}, while
 the other one is the
\textit{`genus'}, written as  $[g_v]$ and omitted if it is zero.
Furthermore, the edges also have two possible decorations:
$+$ or $-$. In most of the cases we omit the decoration $+$,
nevertheless we prefer to emphasize the sign $-$ with the symbol
$\circleddash$.
\ix{plumbing|textbf}\ix{plumbing!graph|textbf}\ix{plumbing!graph!orientable|textbf}
\ix{graph decoration!Euler number|textbf}
\ix{graph decoration!genus|textbf}\ix{plumbing!representation|textbf}\ix{graph decoration!edge|textbf}

Although, for plumbing representations of links of normal surface
singularities we need only the sign $+$, and for such graphs the intersection
matrix associated with the graph is always negative definite, in
the present situation both restrictions  should be relieved.
Nevertheless, all our 3-manifolds are {\it oriented}, hence we will restrict
ourselves to \textit{`orientable plumbing graphs'}, cf.
\cite[(3.2)(i)]{neumann}. These are characterized by  $g_v\geq 0$ for any vertex $v$.
\ix{graph!negative definite}\ix{matrix!intersection}

\bekezdes\label{bek:PlCon} {\bf The plumbing construction.} \
Fix a {\it connected} plumbing graph $\G$.
\ix{plumbing!construction|textbf}

The oriented plumbed 3--manifold $M(\G)$ associated with $\G$
is  constructed using a set of
$S^1$--bundles $\{\pi_v:B_v\to S_v\}_{v\in {\mathcal V}}$, whose total space $B_v$ has a fixed
orientation. They are indexed by the set of
vertices $\calv$ of the plumbing graph, so  that the base--space
$S_v$ of $B_v$ is a closed orientable real surface of genus $g_v$,
and the Euler number of the bundle $B_v$ is the Euler number
decoration of the vertex $v$ on the graph. Then one glues these
bundles corresponding  to the edges of $\G$ as follows. First, one chooses an orientation of
$S_v$ and of the fibers compatible with the orientation of $B_v$. Then,
 for each edge adjacent to $v$ one fixes a point
$p\in S_v$, an orientation preserving  trivialization $D_p\times S^1\to \pi^{-1}_v(D_p)$
above a small closed disc $D_p\ni p$, and one deletes its
interior $D^\circ _p\times S^1$. Here, similarly as above,
 $S^1$ is the unit circle with its natural orientation.
Then,  any edge  $(v,w)$ of
$\G$ determines  $\partial D_p\times S^1$ in $B_v$ and
$\partial D_q \times S^1$ in $B_w$,  both diffeomorphic to
$S^1\times S^1$.  They are glued by an identification map $\epsilon \bigl(
\begin{smallmatrix}  0 & 1\\1&0
\end{smallmatrix} \bigr)$, where $\epsilon=\pm$ is the decoration
of the edge.

If we allow disconnected plumbing graphs, for example $\G$ is the
disjoint union of the plumbing graphs $\G_1$ and $\G_2$, written as  $\G=\G_1+\G_2$, then
by convention $M(\G)$ is the {\it oriented connected sum} $M(\G)=M(\G_1)\#M(\G_2)$.
Furthermore, sometimes it is convenient to allow the empty graph too. It corresponds to
$M(\emptyset)=S^3$. \ix{connected sum|textbf}

For more details, see \cite{neumann}, page 303.

\vspace{2mm}

In order to codify certain  additional geometric information
 (e.g., information on links in the 3--manifold, or boundary
components),  we will be working with
plumbing graphs that have  {\bf extra
decorations}, as made  precise in the next subsections.

\begin{bekezdes}\labelpar{links}
{\bf Oriented links in oriented closed plumbed 3--manifolds} are represented on the
graph by {\it arrowhead vertices}; the other  `usual' vertices will be
called {\it non--arrowheads}. Arrowhead vertices  have no `Euler
number' or `genus' decorations. Each arrowhead is connected by an
edge to some non--arrowhead vertex $v$, and this edge has a sign--decoration
$+$ or $\circleddash$, similarly to the edges connecting two non--arrowhead
 vertices (whose significance was explained in the preceding subsection).
 Any arrowhead supported by a non--arrowhead $v$ codifies a generic
$S^1$--fiber of $B_v$, while the sign--decoration of the supporting edge
determines an orientation on it. The correspondence is realized as follows.
For each non--arrowhead vertex $v$ choose an orientation of
$S_v$ and of the fibers as in the previous subsection. This is used in the gluing of the
bundles too, but it also identifies the orientation of link--components:
if an arrowhead is supported by a $+$--edge then the link component inherits the
fiber--orientation, otherwise the orientation is reversed.
The disjoint  union  of these oriented $S^1$--fibers, indexed by the set of  all arrowheads,
constitutes an oriented link $K$ in the oriented plumbed 3--manifold $M=M(\G)$.
\ix{graph decoration!arrowheads|textbf}\ix{link!of curve singularities|textbf}

There is an exception to the above description  when the
graph consists of a double arrow. In this case the 3--manifold is
$S^3$, and the arrows represent two Hopf link--components. If the sign--decoration of the
double--arrow is $+$ then both Hoph link--components inherit the orientation of the
oriented Hopf $S^1$--fibration, otherwise the orientation of one of them is reversed.
\ix{link!Hopf}

Usually we write $\cala$ for the set of arrowheads, and $\calw$
for the set of non--arrowheads, that is $\calv=\cala \sqcup \calw$.

\bekezdes\label{bek:MULT} {\bf Multiplicity systems.} \
Plumbing graphs with arrowhead vertices, in general, might carry
 an extra set of decorations: each arrowhead and non-arrowhead
vertex has an additional  \textit{`multiplicity weight'}, denoted
by $(m_v)$.
\ix{graph decoration!multiplicity system|textbf}\ix{multiplicity system|textbf}

The typical example of a graph with arrowheads and
multiplicity system is provided by an embedded resolution graph of
an analytic function defined on a normal surface singularity,
where the arrowheads correspond to the strict transforms of the zero set
of the analytic function, the non--arrowheads to irreducible
exceptional divisors, and the multiplicities are the vanishing
orders of the pull--back of the function along the exceptional
divisors and strict transforms;  see  \ref{ss:NSS}.

More generally, the set of multiplicities represent a relative
2--cycle in the corresponding oriented plumbed 4-manifold, which in the
homology group relative to the boundary  represents zero. The {\it oriented plumbed 4--manifold}
$P=P(\G)$ is constructed in a similar way as the plumbed
3--manifold $M=M(\G)$: one replaces the $S^1$--bundles with
the corresponding disc--bundles ${\mathcal D}_v$ and one
glues them by a similar procedure. Then $P$ is a 4--manifold with
boundary such that $\partial P=M$. Each vertex $v$ determines a
2--cycle $C_v$ in $P$: If $v$ is an arrowhead then $C_v$ is an
oriented generic disc--fiber of ${\mathcal D}_v$ --- hence it is a
relative cycle. If $v$ is a non--arrowhead vertex then $C_v$ is
the oriented `core' (i.e. the zero section) of ${\mathcal D}_v$, ---
hence this is an absolute cycle. Their simultaneous
orientations can be arranged compatibly with the graph.
\ix{plumbing!4-manifold $P(\Gamma)$|textbf}

In the present work we consider only those multiplicity systems
$\{m_v\}_{v\in {\mathcal V}}$,  which
satisfy a set of compatibility relations. These relations are equivalent to
the fact that the class of $C(\underline{m}):=\sum_{v\in
\calv}m_vC_v$ in $H_2(P,\partial P,\Z)$ is zero. This  can  be
rewritten as follows. Let $w$ be a fixed non--arrowhead vertex
with Euler number $e_w$. Let $\cale_w $ be the set of all adjacent
edges, excluding  loops supported by $w$. For each $e\in
\cale_w$, connecting $w$ to the vertex $v(e)$ (where $v(e)$ may be an arrowhead or not),
let $\epsilon_{e}\in \{+,-\}$ be its sign--decoration.
Then:
\begin{equation}\label{eq:2.2.1}
e_wm_w+\sum_{e\in\cale_w}\epsilon_{e} m_{v(e)}=0.
\end{equation}
\end{bekezdes}
Indeed, this follows from the fact that $H_2(P,\Z)$ is freely
generated by the absolute classes of $C_w$ ($w\in\calw$), hence
the intersection numbers $(C(\underline{m}),C_w)$ vanish for all
non-arrowhead vertices $w$ if and only if $[C(\underline{m})]=0$
in  $H_2(P,\partial P,\Z)$.

\bekezdes\label{bek:AI}{\bf The intersection matrix and the multiplicity system.} \
The combinatorics and a part of the decorations of the graph $\Gamma$ can be codified
into the {\it intersection matrix} of $\Gamma$. Furthermore, in the presence of
arrowheads, the position of the arrowheads
can be codified in an {\it incidence matrix}. Their definition is the following.

\begin{definition}\label{def:incidence}
The `intersection matrix' $A$ of $\G$ is the symmetric matrix of
size $|\calw|\times |\calw|$ whose entry $a_{wv}$ is the Euler
number of $w$ if $w=v$, while for $w\not=v$ it is $\sum_e
\epsilon_e$, where the sum is over all edges $e$ connecting $w$
and $v$, and $\epsilon_e\in\{+,-\}$ is the edge--decoration of
$e$.

The  graph $\G$ is called negative definite if the intersection matrix $A$ is negative
definite. \ix{graph!negative definite|textbf}

The `incidence matrix' $\inc$ of the arrows of
$\G$ is a matrix of size $|\calw|\times
|\cala|$. For any  $a\in\cala$ let $w_a$ be that non--arrowhead
vertex which supports the arrowhead  $a$. Then, for each
$a\in\cala$, the entry $(a,w_a)$ is 1, all the other entries are
zero.

The matrix $(A,\inc)$ of size $|\calw|\times (|\calw|+|\cala|)$
consists of the blocks $A$ and $\inc$.
\ix{matrix!intersection|textbf}\ix{matrix!incidence|textbf}
\end{definition}

The matrix of the  linear system of equations (\ref{eq:2.2.1}) in
variables $\{m_w\}_{w\in\calv}$ is the matrix $(A,\inc)$. Hence, if $A$ is non--degenerate then from
the position of the arrowheads, or from the incidence matrix, one recovers uniquely all the
multiplicities.
%In other words, the pair $(M,K)$ determines the multiplicity system
%provided that $\det(A)\not=0$.

Note also that if $P$ is the plumbed 4--manifold associated with the  graph $\G$, cf. \ref{bek:MULT},
then $A$ can also be
interpreted as the intersection form on $H_2(P,\Z)$ associated with the basis $\{[C_w]\}_{w\in\calw}$.

\begin{bekezdes}\labelpar{Openbooks}{\bf The multiplicity system associated
with an open book decomposition.} Consider a pair $(M,K)$ as in \ref{links}.
$K$ is called a {\it fibred link} if it is the binding of an open book decomposition
of $M$.
\ix{link!of curve singularities!fibred|textbf}

The case of a fibred link $K$ in a 3--manifold
$M$ has a special interest in purely topological discussions too.  Links provided by singularity
theory are usually fibred. In such  cases the pair $(M,K)$ has a
plumbed representation provided by a plumbing graph (decorated with Euler
numbers and genera) and arrows (representing $K$).
Additionally,  $p:M\setminus K\to S^1$ is a locally trivial
fibration with a trivialization in a neighbourhood of $K$, cf. \ref{def:OBD}. In
particular,  $p$ sends any oriented meridian of any oriented component of
$K$ to the positive generator of $H_1(S^1,\bfz)$.

In such a situation,
we define a multiplicity system associated with the open book decomposition as follows.

First, we fix the
link--components as distinguished fibers of the corresponding building blocks
$B_w$. Then, for each non--arrowhead
vertex $w$, let $\gamma_w$ be an oriented generic $S^1$--fiber
of $B_w$, different from  any fixed fiber corresponding to
components of $K$. Here we use the same  orientation of the $S^1$--bundle
which was used in the plumbing construction. For $a\in\cala$ we define $\gamma_a$
as the oriented meridian of the corresponding oriented component of
$K$. For any loop $\gamma$ let $[\gamma]$ be its homology class.
\end{bekezdes}
\begin{definition}\labelpar{bek:mult}
The multiplicity system associated with the fibration $p$
is the collection of integers $m_v:=p_*([\gamma_w])\in H_1(S^1,\Z)=\Z$\,,
 $v\in\calv$. (Clearly, $m_a=1$ for  $a\in\cala$.)
\end{definition}
\ix{graph decoration!multiplicity system!of a fibration|textbf}

The fact that this is indeed a multiplicity system can be seen as
follows. Let $F$ be the oriented page of $p$ in $M$, with
$\partial F=K$. By a homotopy one can push $F\setminus K$ in the interior of
$P$ keeping $K=\partial F$ fixed. Then  its relative homology
class can  be represented by the
 relative cycle  $C(\underline{m})$.  On the other hand,
the corresponding  relative homology class is zero, since $F$
 sits in $\partial P$.

\vspace{2mm}

We wish to emphasize that if $M$ is a rational homology sphere, and $K$ is the
binding of an open book decomposition of $M$,  then the
open book decomposition can be recovered from the pair $(M,K)$ by a theorem of Stallings.
This means that  there is a {\it unique} open book decomposition for any fixed binding whenever
$H_1(M,\Q)=0$. The book of Eisenbud and Neumann in
\cite[page 34]{EN} also provides two different arguments for this fact, one of them
based on \cite{BL}, the other on \cite{W}.
\ix{Stallings} \ix{Eisenbud--Neumann book}

On the other hand, in general,
the information codified in the plumbing
data of the pair $(M,K)$ together with the multiplicity system
(that is  the graph with arrows decorated with Euler numbers, genera
and multiplicities), contains less information than the open book decomposition
itself;  see  \cite{cyclic} for different
examples, or \ref{link} here.

\bekezdes\label{bek:zeroarrowhead} {\bf Arrowheads with multiplicity zero.} \
Assume that  $K_1\sqcup K_2$ is an  oriented link in $M$,
and the pair $(M,K_1\sqcup K_2)$ has a plumbing representation as in \ref{links}.
Here $K_1$ and $K_2$ consist of two disjoint sets of link components such that
$(M,K_1)$ is a fibred link. In particular, its open book decomposition $p$ defines
 a multiplicity system on all the non-arrowheads and on all the arrows of $K_1$
 as in \ref{Openbooks}.
 Then, we can put zero multiplicities on all the arrows of $K_2$. (In fact, $p_*([\gamma_a])=0$
 for all meridians of $K_2$, hence this definition also works.)
 In this way, using zero--multiplicity arrowheads, we can identify link components of $M$ which are not
components in the binding of the open book decomposition $(M,K_1)$.

\begin{bekezdes}\labelpar{boundarycomp} {\bf Manifolds with
boundary.} Similarly, one can codify plumbed  oriented 3-manifolds
with boundary, where each boundary component is a torus. In
general, one starts with an oriented closed 3--manifold  $M$ with a link $K$ in it,
cf. \ref{links}. Let $L$ be the collection of some of the components of $K$.
Then after a small closed tubular neighbourhood $T(L)$ of
$L$ is fixed,  one deletes its interior $T^\circ (L)$
obtaining  a manifold with boundary $M\setminus T^\circ (L)$.  The other
components of $K$, which are not in $L$,  are kept as link components in $M\setminus T^\circ (L)$.

At the level of plumbing graphs, in the present article,  this will be
codified as follows. Assume that the arrowhead representing a connected component of $L$ is
connected by an edge to the non--arrowhead vertex $v$. Then replace this
supporting edge by a {\it dash--edge} (and delete its sign--decoration, or consider
it irrelevant).  An arrowhead that is supported by a dash--edge will be called {\it dash--arrow}.
Therefore, the dash--arrows
represent deleted solid tori containing as their core the components of the corresponding link.
\ix{graph decoration!dash--edge|textbf}\ix{graph decoration!dash--arrow|textbf}

Equivalently, if $r_w$ is the number of dash--arrows supported by
the non--arrowhead vertex $w$, then one can also get the
plumbed manifold with boundary using the plumbing construction by
deleting $r_w$ solid tori, the inverse image of $r_w$ small open
discs of the base space of $B_w$, from $B_w$. (This is codified in
\cite{neumann} by the decoration $[r_w]$ of $w$, instead of the
$r_w$ dash-arrows of $w$ used here.)

Notice that in this way, (i.e. by replacing an arrow supported by
a usual edge by an arrow supported by a dash--edge) one loses
some information:  for example,  the Euler--number of the supporting non--arrowhead $w$ becomes
irrelevant.
\end{bekezdes}

\begin{bekezdes}\labelpar{bek:multbou}{\bf Fibrations and multiplicities.}
Additionally, if $M\setminus T^\circ (K)$ is a locally trivial
fibration $p$ over $S^1$, one can define again a multiplicity
system: $m_w:=p_*([\gamma_w])$ for each non--arrowhead $w$,
as above. Nevertheless, in this case,  the arrowheads supported by
dash--arrows will carry no multiplicity decorations (or,
equivalently, they will be disregarded). This system satisfies the
compatibility relations (\ref{eq:2.2.1}) in the following modified
way: if $w$ is a non--arrowhead which supports no dash--arrow,
then (\ref{eq:2.2.1}) is valid for that $w$. But, in general, no
other relation holds. (That is,  (\ref{eq:2.2.1}) is valid for all
non--arrowhead vertices $w$ with the convention that the
multiplicity of the dash--arrows `can be anything'.)
\ix{graph decoration!multiplicity system!of a fibration}

Notice that if $(M,K)$ has an open book decomposition, then the
fibration of the complement of $K$ contains less information than
the original open book decomposition: in general,  a fibration
$K\setminus T^\circ (K)$  cannot be extended  canonically to an open
book decomposition. Similarly, the multiplicity system associated
with a fibration $p:M\setminus T^\circ (K)\to S^1$ contains less
information than the multiplicity system associated with the
original open book decomposition.
\end{bekezdes}

\section{\ The plumbing calculus}\labelpar{CALC}\setcounter{equation}{0}
\bekezdes\label{bek:calculus} \
The  {\bf plumbing calculus of oriented plumbed 3--manifolds
and the corresponding plumbing graphs} targets the following
classification result, cf.  \cite[(3.2)(i)]{neumann}. According to this, there are 8
permitted operations  of plumbing graphs.  In Neumann's notation,
seven of them are: R0(a), R1,  R3, R5, R6, R7, R2/4.
In Neumann's list an eighth operation appears as well,  R0(b)'.  Since it
can be replaced by three consecutive
applications of R0(a), we will omit it. (Note also that  Neumann's list contains
an additional operation R8; this one applies for graphs with dash--arrows,
and it will appear  below in \ref{RED-CALC2}.)
\ix{Neumann}\ix{plumbing!graph calculus|textbf}\ix{plumbing!graph calculus!oriented|textbf}
\ix{plumbing!graph calculus!operations|textbf}

These operations satisfy the following two
key properties:

\begin{enumerate}
\item ({\bf Stability of the calculus})
Applying any of the above seven  operations, or their inverses, to a plumbing graph
$\G$ does not change the oriented  diffeomorphism type of $M(\G)$.
\item ({\bf Sufficiency of the calculus}) If $\G_1$ and $\G_2$ are two plumbing graphs
and $M(\G_1)$ and $M(\G_2)$ are diffeomorphic by an orientation preserving diffeomorphism, then
$\G_1$ and $\G_2$ are related by  a sequence of the above operations or their inverses.
\end{enumerate}

The `oriented calculus' is part of a larger set of operations, for which a similar
statement is valid as above;
it connects non--necessarily orientable  plumbed 3--manifolds and their  plumbed graphs.
The larger class  additionally
contains  those operations which reverse orientation, or which are valid for non--orientable
manifolds too. For the complete list, from R0 to R8,  see \cite{neumann}.
The oriented calculus selects exactly those operations which preserve
the orientation of the orientable 3--manifold.
As we are interested only in the oriented special class, we discuss only these ones.

\vspace{2mm}

For the completeness of the presentation %(and fidelity to the classical presentation \cite{neumann})
we provide these operations, at least those which will be used in the present work.
The operations below are applied for one of the connected components of the graph $\G$.

\vspace{2mm}

\noindent{\bf [R0](a)} \  Reverse the signs on all edges other than loops
adjacent to any fixed  vertex.

\vspace{2mm}

\noindent{\bf [R1] (blowing down)} \
$\epsilon=\pm 1$ and the edge signs $\epsilon_0,\ \epsilon_1,\
\epsilon_2$ are related by
$\epsilon_0=-\epsilon\epsilon_1\epsilon_2$.

\vspace{2mm}

\noindent a.)
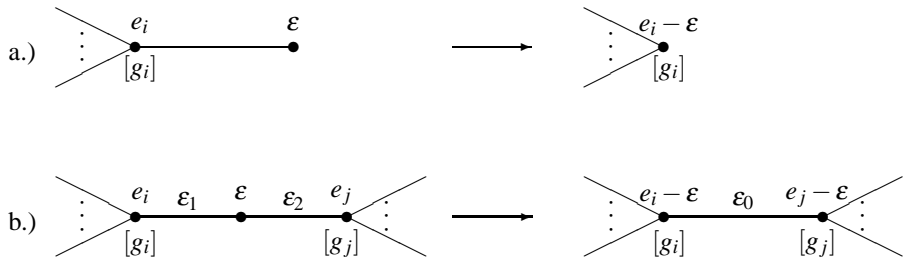
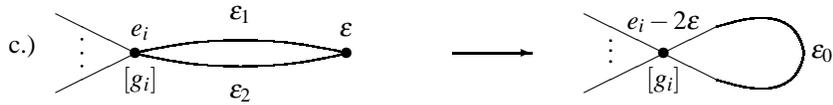
\begin{picture}(100,35)(5,15)

% bal oldali graf:
\put(40,20){\circle*{4}} \put(100,20){\circle*{4}} % ket csucs
\put(40,20){\line(1,0){60}}                        % kozottuk el
\put(40,20){\line(-2,-1){30}}                      % elso csucstol balra elek
\put(40,20){\line(-2, 1){30}}

\put(20,25){\makebox(0,0){$\cdot$}}                % koztuk dots
\put(20,20){\makebox(0,0){$\cdot$}}
\put(20,15){\makebox(0,0){$\cdot$}}

\put(42,28){\makebox(0,0){$e_i$}}  % a csucsokon sulyok
\put(42,12){\makebox(0,0){$[g_i]$}}
\put(100,29){\makebox(0,0){$\epsilon$}}

\put(160,20){\vector(1,0){30}}                     % nyil a ket graf kozott

% jobboldali graf:
\put(240,20){\circle*{4}}                    % egy csucs
\put(242,28){\makebox(0,0){$e_i-\epsilon$}}  % a csucson sulyok
\put(242,12){\makebox(0,0){$[g_i]$}}

\put(240,20){\line(-2,-1){30}}                      % elso csucstol balra elek
\put(240,20){\line(-2, 1){30}}

\put(220,25){\makebox(0,0){$\cdot$}}                % koztuk dots
\put(220,20){\makebox(0,0){$\cdot$}}
\put(220,15){\makebox(0,0){$\cdot$}}
\end{picture}

\vspace{4mm}

\noindent b.)
\begin{picture}(100,50)(5,15)

% bal oldali graf:
\put(40,20){\circle*{4}} \put(80,20){\circle*{4}}
\put(120,20){\circle*{4}}                           % harom csucs

\put(40,20){\line(1,0){80}}                        % kozottuk el

\put(40,20){\line(-2,-1){30}}                      % elso csucstol balra elek
\put(40,20){\line(-2, 1){30}}
\put(20,25){\makebox(0,0){$\cdot$}}                % koztuk dots
\put(20,20){\makebox(0,0){$\cdot$}}
\put(20,15){\makebox(0,0){$\cdot$}}

\put(120,20){\line(2,-1){30}}                      % harmadik csucstol
                                                    % jobbra elek
\put(120,20){\line(2, 1){30}}
\put(135,25){\makebox(0,0){$\cdot$}}                % koztuk dots
\put(135,20){\makebox(0,0){$\cdot$}}
\put(135,15){\makebox(0,0){$\cdot$}}

\put(42,28){\makebox(0,0){$e_i$}}                  % a csucsokon sulyok
\put(42,10){\makebox(0,0){$[g_i]$}}
\put(80,29){\makebox(0,0){$\epsilon$}}
\put(118,28){\makebox(0,0){$e_j$}}
\put(118,10){\makebox(0,0){$[g_j]$}}

\put(60,26){\makebox(0,0){$\epsilon_1$}}      % az eleken sulyok
\put(100,26){\makebox(0,0){$\epsilon_2$}}

\put(160,20){\vector(1,0){30}}                     % nyil a ket graf kozott

% jobboldali graf:
\put(240,20){\circle*{4}}                    % baloldali csucs
\put(242,28){\makebox(0,0){$e_i-\epsilon$}}  % a csucson sulyok
\put(242,10){\makebox(0,0){$[g_i]$}}

\put(240,20){\line(-2,-1){30}}                      % elso csucstol balra elek
\put(240,20){\line(-2, 1){30}}

\put(220,25){\makebox(0,0){$\cdot$}}                % koztuk dots
\put(220,20){\makebox(0,0){$\cdot$}}
\put(220,15){\makebox(0,0){$\cdot$}}

\put(300,20){\circle*{4}} % jobboldali csucs
\put(298,28){\makebox(0,0){$e_j-\epsilon$}}  % a csucson sulyok
\put(298,10){\makebox(0,0){$[g_j]$}}

\put(300,20){\line(2,-1){30}}                      % harmadik csucstol
                                                    % jobbra elek
\put(300,20){\line(2, 1){30}}
\put(315,25){\makebox(0,0){$\cdot$}}                % koztuk dots
\put(315,20){\makebox(0,0){$\cdot$}}
\put(315,15){\makebox(0,0){$\cdot$}}

\put(240,20){\line(1,0){60}} % a ket csucs kozott el
\put(270,26){\makebox(0,0){$\epsilon_0$}}      % az elen sulyok

\end{picture}

\vspace{3mm}

\vskip 1mm\noindent c.)
\begin{picture}(100,50)(5,20)
% eredeti graf:
\put(40,20){\circle*{4}} \put(120,20){\circle*{4}}   % ket csucs

\put(40,20){\line(-2,-1){30}}                      % elso csucstol balra elek
\put(40,20){\line(-2, 1){30}}
\put(20,25){\makebox(0,0){$\cdot$}}                % koztuk dots
\put(20,20){\makebox(0,0){$\cdot$}}
\put(20,15){\makebox(0,0){$\cdot$}}

\put(42,28){\makebox(0,0){$e_i$}}                  % a csucsokon sulyok
\put(42,10){\makebox(0,0){$[g_i]$}}
\put(120,29){\makebox(0,0){$\epsilon$}}

%a ket csucs kozott ket el
\qbezier(40,20)(80,30)(120,20)
\qbezier(40,20)(80,10)(120,20)

% a ket elen felirat
\put(80,35){\makebox(0,0){$\epsilon_1$}}
\put(80,5){\makebox(0,0){$\epsilon_2$}}

%%%%%%%%%%%%%%%%%%%%%%%%%%%%%%%%%%%%%%%%%%%%%%%%%%%%%

\put(160,20){\vector(1,0){30}}                     % nyil a ket graf kozott

%%%%%%%%%%%%%%%%%%%%%%%%%%%%%%%%%%%%%%%%%%%%%%%%%%%%%

% eredmeny graf
\put(240,20){\line(-2,1){30}}       % egy csucs es sulyai
\put(240,20){\line(-2,-1){30}}
\put(240,20){\circle*{4}}
\put(220,24){\makebox(0,0){$\vdots$}}
\put(240,32){\makebox(0,0){$e_i-2\epsilon$}}
\put(240,10){\makebox(0,0){$[g_i]$}}

% maga a hurok:
\put(240,20){\line(2,1){20}} \put(240,20){\line(2,-1){20}}
\qbezier(260,30)(290,40)(293,20) \qbezier(260,10)(290,0)(293,20)

% suly a hurkon:
\put(300,20){\makebox(0,0){$\epsilon_0$}}

\end{picture}

\vspace{13mm}

\noindent{\bf [R3] ($0$-chain absorption)} \ The edge signs
$\epsilon_i'$ ($i=1,...,s$) are related by
$\epsilon_i'=-\epsilon\overline{\epsilon}\epsilon_i$ provided that
the edge sign in question is not on a loop, and
$\epsilon_i'=\epsilon_i$, if it is on a loop.

\noindent
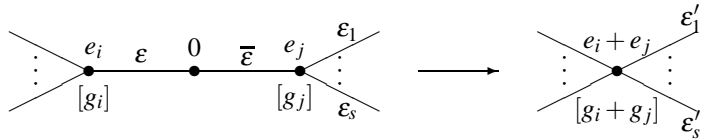
\begin{picture}(100,70)(-5,-10)

% bal oldali graf:
\put(40,20){\circle*{4}} \put(80,20){\circle*{4}}
\put(120,20){\circle*{4}}                           % harom csucs

\put(40,20){\line(1,0){80}}                        % kozottuk el

\put(40,20){\line(-2,-1){30}}                      % elso csucstol balra elek
\put(40,20){\line(-2, 1){30}}
\put(20,25){\makebox(0,0){$\cdot$}}                % koztuk dots
\put(20,20){\makebox(0,0){$\cdot$}}
\put(20,15){\makebox(0,0){$\cdot$}}

\put(120,20){\line(2,-1){30}}                      % harmadik csucstol
                                                    % jobbra elek
\put(120,20){\line(2, 1){30}}
\put(135,25){\makebox(0,0){$\cdot$}}                % koztuk dots
\put(135,20){\makebox(0,0){$\cdot$}}
\put(135,15){\makebox(0,0){$\cdot$}}

\put(42,28){\makebox(0,0){$e_i$}}                  % a csucsokon sulyok
\put(42,10){\makebox(0,0){$[g_i]$}}
\put(80,29){\makebox(0,0){$0$}}
\put(118,28){\makebox(0,0){$e_j$}}
\put(118,10){\makebox(0,0){$[g_j]$}}

\put(60,26){\makebox(0,0){$\epsilon$}}      % az eleken sulyok
\put(100,26){\makebox(0,0){$\overline{\epsilon}$}}

% sulyok a jobboldali "kivezeto" eleken
\put(138,35) {\makebox(0,0){$\epsilon_1$}}
\put(138,5){\makebox(0,0){$\epsilon_s$}}

\put(165,20){\vector(1,0){30}}                     % nyil a ket graf kozott

% jobboldali graf:
\put(240,20){\circle*{4}}                    % baloldali csucs
\put(240,30){\makebox(0,0){$e_i+e_j$}}  % a csucson sulyok
\put(240,5){\makebox(0,0){$[g_i+g_j]$}}

\put(240,20){\line(-2,-1){30}}                      % elso csucstol balra elek
\put(240,20){\line(-2, 1){30}}

\put(220,25){\makebox(0,0){$\cdot$}}                % koztuk dots
\put(220,20){\makebox(0,0){$\cdot$}}
\put(220,15){\makebox(0,0){$\cdot$}}

\put(240,20){\line(2,-1){30}}                      % csucstol
                                                    % jobbra is elek
\put(240,20){\line(2, 1){30}}
\put(260,25){\makebox(0,0){$\cdot$}}                % koztuk dots
\put(260,20){\makebox(0,0){$\cdot$}}
\put(260,15){\makebox(0,0){$\cdot$}}

% sulyok a jobboldali "kivezeto" eleken
\put(268,40) {\makebox(0,0){$\epsilon_1'$}}
\put(268,0){\makebox(0,0){$\epsilon_s'$}}

\end{picture}

\vspace{4mm}

\noindent{\bf [R5] (oriented handle absorption)}

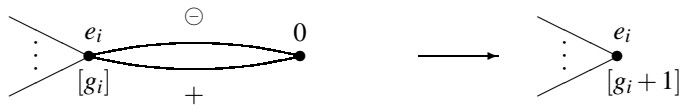
\vspace{3mm}

\noindent
\begin{picture}(100,60)(-5,-10)
% eredeti graf:
\put(40,20){\circle*{4}} \put(120,20){\circle*{4}}   % ket csucs

\put(40,20){\line(-2,-1){30}}                      % elso csucstol balra elek
\put(40,20){\line(-2, 1){30}}
\put(20,25){\makebox(0,0){$\cdot$}}                % koztuk dots
\put(20,20){\makebox(0,0){$\cdot$}}
\put(20,15){\makebox(0,0){$\cdot$}}

\put(42,28){\makebox(0,0){$e_i$}}                  % a csucsokon sulyok
\put(42,10){\makebox(0,0){$[g_i]$}}
\put(120,29){\makebox(0,0){$0$}}

%a ket csucs kozott ket el
\qbezier(40,20)(80,30)(120,20)
\qbezier(40,20)(80,10)(120,20)

% a ket elen felirat
\put(80,35){\makebox(0,0){$\circleddash$}}
\put(80,5){\makebox(0,0){$+$}}

%%%%%%%%%%%%%%%%%%%%%%%%%%%%%%%%%%%%%%%%%%%%%%%%%%%%%

\put(165,20){\vector(1,0){30}}                     % nyil a ket graf kozott

%%%%%%%%%%%%%%%%%%%%%%%%%%%%%%%%%%%%%%%%%%%%%%%%%%%%%

% jobboldali graf:
\put(240,20){\circle*{4}}                    % egy csucs
\put(242,28){\makebox(0,0){$e_i$}}  % a csucson sulyok
\put(250,10){\makebox(0,0){$[g_i+1]$}}

\put(240,20){\line(-2,-1){30}}                      % elso csucstol balra elek
\put(240,20){\line(-2, 1){30}}

\put(220,25){\makebox(0,0){$\cdot$}}                % koztuk dots
\put(220,20){\makebox(0,0){$\cdot$}}
\put(220,15){\makebox(0,0){$\cdot$}}
\end{picture}

\vspace{3mm}

\noindent{\bf [R6] (splitting)} \ If $\G$ has the form

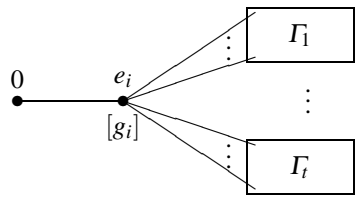
\begin{picture}(100,110)(-30,-35)

\put(40,20){\circle*{4}} \put(80,20){\circle*{4}}
\put(40,20){\line(1,0){40}}
\put(40,28){\makebox(0,0){$0$}}
\put(120,43){\makebox(0,0){$\vdots$}}
\put(120,2){\makebox(0,0){$\vdots$}}
\put(80,20){\line(3,2){50}}  \put(80,20){\line(3,1){50}}
\put(80,20){\line(3,-2){50}}  \put(80,20){\line(3,-1){50}}
 \put(80,28){\makebox(0,0){$e_i$}}
 \put(80,10){\makebox(0,0){$[g_i]$}}
\put(150,23){\makebox(0,0){$\vdots$}}
\put(127,35){\framebox(40,20)}
 \put(127,-15){\framebox(40,20)}
 \put(147,45){\makebox(0,0){$\G_1$}}
\put(147,-5){\makebox(0,0){$\G_t$}}
\end{picture}

\noindent where each $\G_j$ is connected and, for each $j\in\{1,\ldots,t\}$,
$\G_j$ is connected to the vertex ~$i$
by $k_j$ edges, then replace $\G$ by the disjoint union of $\G_1, \ldots, \G_t$, and
$2g_i+\sum_j(k_j-1)$  copies of
\begin{picture}(20,15)(0,0)
\put(10,0){\circle*{4}}
\put(10,10){\makebox(0,0){$0$}}
\end{picture}
.

\vspace{5mm}

\noindent{\bf [R7] (Seifert graph exchange)} \ Replace

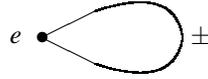
\begin{picture}(100,55)(0,-5)
\put(140,20){\circle*{4}}
\put(140,20){\line(2,1){20}} \put(140,20){\line(2,-1){20}}
\qbezier(160,30)(190,40)(193,20) \qbezier(160,10)(190,0)(193,20)
\put(130,20){\makebox(0,0){$e$}}
\put(200,20){\makebox(0,0){$\pm $}}
\end{picture}

\noindent where $e\in\{-1,0,+1\}$, by a star--shaped graph with all genera zero.

There are six cases, the pair $(e,\pm)$ can be $(-1,+)$, $(0,+)$, $(1,+)$, $(-1,\circleddash)$,
$(0,\circleddash)$ and  $(1,\circleddash)$.
The corresponding star--shaped graphs are (in Kodaira's notation, used in elliptic fibrations, or
using the notations of extended {\it A-D-E} graphs):
$II,\ III(\tilde{A}_1),\ IV(\tilde{A}_2),\ II^*(\tilde{E}_6),\ III^*(\tilde{E}_7),\ IV^*(\tilde{E}_8)$.

Since they will not be used in the sequel, we omit the picture of the
six graphs.

\vspace{2mm}

\noindent{\bf [R2/4] \ (Unoriented  handle absorption followed by two $\R P^2$--extrusions)}
This operation will not be used in the sequel, hence again we will not give more details about it.

The interested reader can find details on  both [R7] and [R2/4] in  \cite{neumann}.

\bekezdes In the literature there are several special classes of graphs
codifying special families of 3--manifolds, for which the graph calculus, that is the set of
allowed operations,  is more restrictive. Such special classes  are for example,
`spherical plumbing graphs', `orientable plumbing graphs with  no cycles', or `star--shaped plumbing graphs'.
For more examples and for their calculus, see e.g. Theorem 3.2 in \cite{neumann}.

Besides the study of special families of 3--manifolds, there is another motivation to
consider reduced sets of operations. If the class of plumbing graphs considered is the result of
a special geometric construction, then they might carry some  information in their shape or decorations
which might be lost in the diffeomorphism type of $M(\G)$. In such a situation, if we wish to
preserve that extra information, then we must use only those operations which preserve it. Of course, in such a case,
we cannot always expect the validity of the `sufficiency of the calculus', but we gain a stronger `stability'.

For example, in  the case of the
plumbing calculus of {\it negative definite resolution graph} of a normal \ix{graph!negative definite}
surface singularity, or the {\it dual graph} of any kind of complex
curve configuration on a smooth complex surface, we prefer to use
only $(-1)$--blow ups and its inverse
instead of all possible  operations of the smooth (oriented or non--oriented) calculus.
In this way, we can make  sure that the graph modified by the operation can again be realized
in the corresponding complex analytical context. Moreover, the blow ups preserve the
number of independent cycles and the genera of the graph (for their definition see below),
which carry some analytic Hodge theoretical  information, cf. \ref{re:nss}.

\bekezdes\label{bek:RESCAL} {\bf The `reduced plumbing calculus'.} \
In the present article,
guided by results of the present work,
 we  also select a special set of operations from those used in
the calculus of oriented plumbed 3--manifolds and their graphs.  The collection  of these operations is called
{\it reduced set of operations}, and they generate the {\it reduced plumbing calculus}.
\ix{plumbing!graph calculus!reduced|textbf}

Our graphs are constructed from a singularity
theoretical viewpoint. Using only the
reduced  set of operations allows for preserving features
of these  graphs inherited from their algebro--geometric/analytic
construction, which might be lost if we run all the operations of
the smooth calculus.

The principle by which  we select the `{\it reduced set of operations}' is
the following. For any decorated plumbing graph $\G$ let $c(\G)$ be
the number of independent cycles in $\G$ (i.e. the rank of
$H_1(|\G|)$, where $|\G|$ is the topological realization of $\G$).
Furthermore, let $g(\G)$ be the sum of the genus decorations of
$\G$, i.e $g(\G)=\sum_{w\in\calw(\G)}g_w$. The point is that all the
graphs provided by our main construction, associated with a fixed
geometrical object (singularity) --- but  depending essentially on
the choice of an embedded resolution ---, share two properties: all of
them are connected, and $c(\G)+g(\G)$ is the same  for all of them
(describing a geometric entity independent of the construction, cf. \ref{GEO}).
Our {\it reduced set}
contains exactly those operations of the oriented calculus which
\begin{equation}\label{eq:speclist}\begin{split}
&\mbox{
{\it preserve connectedness and }}\\ & \mbox{ {\it keep the integer $c(\G)+g(\G)$ of the graphs fixed}.}
\end{split}
\end{equation}
{\it In Neumann's notation this list %(in Neumann's notation)  (valid for graphs without any kind of arrows)
 is the following} : {\bf R0(a)} (reverse
the sign--decoration on all edges other than loops adjacent to a
vertex), {\bf R1} (blowing down $\pm 1$ vertices), {\bf R3} (0--chain
absorption), and {\bf R5} (oriented handle absorption).
The inverses of R1, R3 and R5 are called: blowing up, 0--chain extrusion, oriented handle extrusion.
\ix{number of independent cycles|textbf}

\begin{remark}\labelpar{re:split}
There is one particular case of  the splitting operation  R6 which
still satisfies the requirements (\ref{eq:speclist}).
This operation has the following form,  where the two left--edges
might have any sign--decorations:\\

\noindent{\bf [R6$^{naive}$] (`Naive' splitting)}

\noindent
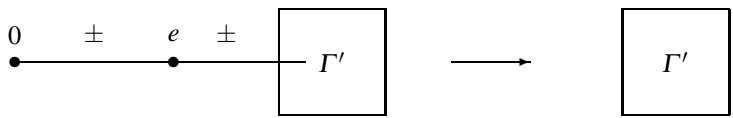
\begin{picture}(100,70)(-30,-15)
% eredeti graf:
\put(0,20){\circle*{4}} \put(60,20){\circle*{4}}
\put(0,20){\line(1,0){110}}
\put(100,0){\framebox(40,40){}}
\put(165,20){\vector(1,0){30}}
\put(120,20){\makebox(0,0){$\G'$}}
\put(0,30){\makebox(0,0){$0$}}
\put(60,30){\makebox(0,0){$e$}}
\put(30,30){\makebox(0,0){$\pm$}}\put(80,30){\makebox(0,0){$\pm$}}
\put(230,0){\framebox(40,40){}}
\put(250,20){\makebox(0,0){$\G'$}}
\end{picture}

\noindent Hence,
this operation can also be inserted in the list of reduced calculus; nevertheless, one can prove that
it is a consequence of those already listed there. Indeed, by R0(a) we can assume that the signs of the left
edges are $+$. By blowing
up the left--edge, and blowing down the strict
transform of the 0--vertex, we realize that $e$ can be replaced by $e\pm 1$. Hence by repeating this pair
of operations, we can reduce  $e$ to 0. Then a 0--chain absorption of this newly created
0--vertex finishes the argument.
\end{remark}

% Additionally, in the presence of
%dash--arrows,  one adds R8 (annulus
%absorption) as well.

In this book we will manipulate only  the operations listed above.
Nevertheless, if the reader wishes to use some other operations of
the (oriented) plumbing calculus, this is perfectly fine if she or he wishes to
focus only on $\partial F$.\ix{Milnor!fiber!boundary}
In fact, sometimes it is helpful to have in mind the `splitting
operation' R6  too, since it helps to represent some of the
manifolds as connected sums. \ix{connected sum}

\begin{definition}\label{def:sim}
We write $\G_1\sim \G_2$ if $\G_1$ can be obtained from $\G_2$ by the
above reduced plumbing calculus.
\end{definition}

 \begin{bekezdes}\labelpar{strictly} {\bf The strictly reduced calculus.}
We can go further, and consider an even more restricted
  set of operations. It is based on the  conjecture that under our construction of graphs,
all the possible graphs associated with the
  same geometric object, a non--isolated singular germ $f$,
share the same integers $c(\G)$ and $g(\G)$ (independently of the different choices in the
construction).  In fact, we conjecture that
these integers  are related to  the weight filtration of the mixed Hodge structure
 on $H^1(\partial F,{\mathbb C})$, see Chapter \ref{s:MHS}.
\ix{plumbing!graph calculus!strictly reduced|textbf}
\ix{plumbing!graph calculus!with arrows|textbf}\ix{mixed Hodge structure!weight filtration}

 Since these numbers are modified under the usual calculus, in fact even under the
 reduced calculus, we get
 that the weight filtration of the mixed Hodge structure is not a topological/smooth invariant of
 $\partial F$ (provided that the above mentioned conjecture is true).
 For concrete examples see \ref{re:MHSnot} and \ref{ex:MHSnot}.

 This suggests, that if we would like to preserve  this analytic information as well, we  have to exclude the
 oriented handle absorption R5 from the list of operations of the `reduced calculus', and use an
 even  more restrictive set,
 which is  called {\em strictly reduced oriented calculus}.
 Hence, it only includes the operations R0(a), R1 and R3
 and their inverses.
 \end{bekezdes}

\begin{bekezdes}\labelpar{RED-CALC2} {\bf Reduced oriented plumbing calculus of graphs with  arrows.}
If the graph has some arrows and/or dash--arrows, then all the
above operations  R0(a), R1, R3 and R5 of the reduced calculus
have their natural analogs, complemented with some additional rules:
\begin{enumerate}
\item
the vertex involved in R0(a), the $(\pm 1)$--vertex in R1,  and the 0--vertex in
R3 and R5 should be a non--arrowhead;\vspace{2mm}

\item the $(\pm 1)$--vertex
blown down in R1 can have at most two edges (including also those ones which support arrowheads);
if the vertex has exactly one edge supporting an arrowhead, then we do not modify it by blow down;
\vspace{2mm}

\item
the 0--vertex absorbed in R3 and R5,
cannot support any kind of arrow;\vspace{2mm}

\item
by the operations,
the arrows of the other vertices are naturally kept,  and in the case
of R3 the arrows supported by vertices $i$ and $j$  are
summed;\vspace{2mm}

\item
if a vertex supports a dash--arrow then its Euler number is irrelevant.
\end{enumerate}

\vspace{2mm}

Additionally, one has the following operation as well:

\vspace{2mm}

\noi {\bf [R8] (Annulus absorption for dash--arrows)}

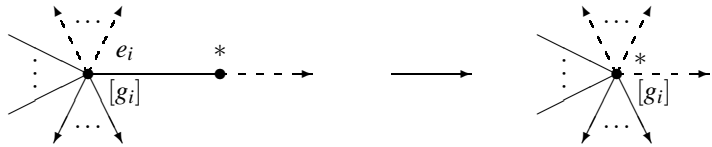
\begin{picture}(100,70)(0,-10)

% bal oldali graf:
\put(40,20){\circle*{4}} \put(90,20){\circle*{4}} % ket csucs
\put(40,20){\line(1,0){50}}                        % kozottuk el
\put(40,20){\line(-2,-1){30}}                      % elso csucstol balra elek
\put(40,20){\line(-2, 1){30}}

\dashline[3]{3}(90,20)(120,20)\put(120,20){\vector(1,0){5}}
\dashline[3]{3}(40,20)(50,40)
\dashline[3]{3}(40,20)(30,40) \put(30,40){\vector(-1,2){3}}
\put(50,40){\vector(1,2){3}}
\put(40,40){\makebox(0,0){$\ldots$}}

 \put(40,20){\vector(-1,-2){13}}
\put(40,20){\vector(1,-2){13}}
\put(40,0){\makebox(0,0){$\ldots$}}

\put(20,25){\makebox(0,0){$\cdot$}}                % koztuk dots
\put(20,20){\makebox(0,0){$\cdot$}}
\put(20,15){\makebox(0,0){$\cdot$}}

\put(54,28){\makebox(0,0){$e_i$}}  % a csucsokon sulyok
\put(54,13){\makebox(0,0){$[g_i]$}}
\put(90,29){\makebox(0,0){$*$}}

\put(155,20){\vector(1,0){30}}                     % nyil a ket graf kozott

% jobboldali graf:
\put(240,20){\circle*{4}}                    % egy csucs
\put(249,26){\makebox(0,0){$*$}}  % a csucson sulyok
\put(254,13){\makebox(0,0){$[g_i]$}}

\put(240,20){\line(-2,-1){30}}                      % elso csucstol balra elek
\put(240,20){\line(-2, 1){30}}

\put(220,25){\makebox(0,0){$\cdot$}}                % koztuk dots
\put(220,20){\makebox(0,0){$\cdot$}}
\put(220,15){\makebox(0,0){$\cdot$}}

%\dashline[3]{3}(100,20)(120,20)\put(120,20){\vector(1,0){3}}
\dashline[3]{3}(240,20)(250,40)
\dashline[3]{3}(240,20)(230,40) \put(230,40){\vector(-1,2){3}}
\put(250,40){\vector(1,2){3}}
\put(240,40){\makebox(0,0){$\ldots$}}

\put(240,20){\vector(-1,-2){13}}
\put(240,20){\vector(1,-2){13}}
\put(240,0){\makebox(0,0){$\ldots$}}

\dashline[3]{3}(240,20)(270,20)\put(270,20){\vector(1,0){5}}
\end{picture}

\vspace{3mm}

\noi Here, if on the left hand side the vertex supports $s$ arrows and $t$ dash--arrows,
then on the right hand side it supports $s$ arrows and $t+1$ dash--arrows. The Euler number
$*$ can be any integer.

%\vspace{1mm}

A special, `degenerate' version of this is the operation

\begin{picture}(100,40)(0,5)
\put(90,20){\circle*{4}} % ket csucs
\dashline[3]{3}(90,20)(120,20)\put(120,20){\vector(1,0){5}}
\dashline[3]{3}(90,20)(60,20)\put(60,20){\vector(-1,0){5}}
\put(90,29){\makebox(0,0){$*$}}

\dashline[3]{3}(220,20)(260,20)\put(260,20){\vector(1,0){5}}
\put(220,20){\vector(-1,0){5}}

\put(155,20){\vector(1,0){30}}                     % nyil a ket graf kozott
\end{picture}

\noindent where both graphs represent a manifold with boundary obtained from $S^3$ by removing
the tubular neighbourhoods of two Hopf link--components.

\vspace{3mm}

In the presence of a  multiplicity system, all the above operations can be extended with
taking the corresponding multiplicities into account in a natural and unique way such that the formulae
(\ref{eq:2.2.1}) stay stable under the operations.
We emphasize again, that the multiplicity of the dash--arrows  is not
well--defined, hence if a non--arrowhead vertex supports a dash--arrow, then its
Euler number   is not well--defined either.
\end{bekezdes}

\bekezdes\label{bek:-G} {\bf Changing the orientation.} If $\G$ is an orientable plumbing graph,
that is a graph with all $g_v\geq 0$, then let $-\G$ be the same graph with the signs of all Euler
and edge decorations reversed. Then $M(-\G)=-M(\G)$, that is, $-\G$ provides the same manifold as
$\G$ but with opposite orientation.

\section{\ Examples. \ Resolution graphs of surface singularities}\label{ss:NSS}
\setcounter{equation}{0}

Let $(X,x)$ be a normal surface singularity and fix the germ
$f:(X,x)\to (\bfc,0)$ of an analytic function. In this section we
review the definition of the embedded resolution graph $\G(X,f)$ of $f$.
More details can be found in the
books of Laufer \cite{La} and Eisenbud--Neumann \cite{EN}, and also in
the survey article of  Lipman \cite{Li}.
In section \ref{link} we also recall the basic topological  properties of the
link $K_X$ of $(X,x)$ and of the pair $(X,f^{-1}(0))$ including
the representation $\arg_*(f)$ provided by the Milnor fibration associated with
$f$.
\ix{Laufer}\ix{Eisenbud--Neumann book} \ix{Lipman}\ix{singularities!normal surface|textbf}

We use the notation $(V_f,x)=(f^{-1}(0),x)$.

\bekezdes\label{1.1} {\bf The embedded resolution.}
Let $(X,x)$ be a normal surface singularity and
let $f:(X,x)\to (\bfc,0)$ be the germ of an analytic function.
An embedded resolution  $\phi:({\cal Y},D)\to (U,V_f)$ of
$(V_f,x)\subset (X,x)$  is characterized
by the following properties.  There is a sufficiently small neighborhood
$U$ of $x$ in $X$, smooth analytic manifold ${\cal Y}$, and analytic proper
map $\phi:{\cal Y}\to U$ such that:\vspace{2mm}

1)\ if $E=\phi^{-1}(x)$, then the restriction $\phi|_{{\cal Y}\setminus E}:
{\cal Y}\setminus E\to U\setminus\{x\}$ is biholomorphic, and ${\cal Y}\setminus E$ is dense in ${\cal Y}$;
\vspace{2mm}

2)\ $D=\phi^{-1}(V_f)$
is a divisor  with only normal crossing singularities, i.e.
at any point $P$ of $E$, there are local coordinates $(u,v)$ in a small
neighbourhood of $P$, such that in these coordinates
$f\circ\phi=u^av^b$ for some non--negative integers $a$ and $b$.\vspace{2mm}

If such an embedded resolution $\phi$ is fixed, then
$E=\phi^{-1}(x)$ is called the  {\it exceptional curve} associated  with $\phi$.
Let $E=\cup_{w\in \calw}E_w$ be its decomposition  in irreducible components.
The closure $S$ of $\phi^{-1}(V_f\setminus\{0\})$ is called the
strict transform of $V_f$.
Let $\cup_{a\in \cala}S_a$ be its  decomposition into irreducible components.
Obviously,  $D=E\cup S$.

For simplicity,  we will assume that $\calw\not=\emptyset$,
any two irreducible components of $E$ have  at most one intersection point,
and  no irreducible exceptional component   has a self--intersection.
This can always be realized by some additional blow ups.

\bekezdes\label{resgraph}{\bf The embedded resolution graph $\G(X,f)$. }
We construct the {\it dual embedded resolution graph}  $\G(X,f)$ of the pair $(X,f)$,
associated with a fixed  resolution $\phi$,  as follows. Its vertices
$\calv=\calw\sqcup   \cala$  consist of the nonarrowhead vertices $\calw$  corresponding to the
irreducible exceptional components, and arrowhead vertices $\cala$
correponding to the irreducible components of the strict transform $S$.
If two irreducible divisors corresponding to $v_1,v_2\in \calv$
have an intersection point then we connect $v_1$ and $v_2$ by an edge in
$\G(X,f)$. %The set of edges is denoted by $\cale$.

The graph $\G(X,f)$ is decorated as follows. The edges are decorated by $+$.
Any  vertex $w\in \calw$ is
decorated with the
self--intersection   $e_w:=E_w\cdot E_w$, which equals to the Euler number of the normal bundle
of $E_w$ in ${\cal Y}$, and with
the genus $g_w$ of $E_w$. Furthermore,
the third decoration is the {\it multiplicity} (of $f$), defined for any  $v\in\calv$, which  is
the vanishing order of $f\circ \phi$ along the
irreducible component corresponding to $v$. For example, if $f$ defines
an isolated singularity, then  for any
$a\in\cala$ one has $m_a=1$.
\ix{graph!resolution|textbf}
\ix{graph!resolution!embedded|textbf}

\bekezdes\label{1.4}{\bf  The resolution graph $\G(X)$ of $(X,x)$.} \
We say that $\phi:{\cal Y}\to U$ is a resolution of $(X,x)$ if
${\cal Y}$ is a smooth analytic manifold,  $U$ a
sufficiently small neighbourhood of $x$ in $X$,
$\phi$ is a proper analytic map, such that
${\cal Y}\setminus E$ (where $E=\phi^{-1}(x)$) is dense in ${\cal Y}$ and
the restriction $\phi|_{{\cal Y}\setminus E}:{\cal Y}\setminus E\to U\setminus
\{x\}$ is a biholomorphism.

If $E$ is a normal crossing curve, then
the topology  of the resolution and the combinatorics of the irreducible
exceptional components  $\cup_wE_w$
are codified in the {\it dual resolution graph} $\G(X)$,  associated with $\phi$.
It is defined similarly as $\G(X,f)$ in \ref{resgraph}, but without arrowheads and
multiplicities.

\bekezdes\label{1.3b}{\bf Some properties of the graphs  $\G(X,f)$ and $\G(X)$.} \

(1)\ $\G=\G(X)$ can serve as a plumbing graph:\ix{plumbing!graph}
the associated oriented plumbed 3--manifold $M(\G)$ is diffeomorphic to the link
$K_X$ of $X$, and the  space ${\cal Y}$ of the resolution can be identified with
the plumbed 4--manifold $P(\G)$  considered in \ref{bek:MULT}.

Moreover, the multiplicity system of $\G(X,f)$
satisfies the system of equations (\ref{eq:2.2.1}).\vspace{2mm}

(2) The graphs $\G(X,f)$ and $\G(X)$ depend on the choice of  $\phi$. Nevertheless,
different dual graphs associated with different resolutions are connected by a sequence
of blow ups and blow downs of $(-1)$--rational curves (operation R1 with $\epsilon =-1$).

By  \cite{neumann}, from $K_X$ one can recover $\G(X)$ up to this blow up ambiguity.\vspace{2mm}

(3) The intersection matrix $A$ is \ix{graph!negative definite}
negative definite; see  \cite{Mu}, \cite{La}, or \cite{GRa}.
In particular, $A$ is non--degenerate, hence the multiplicities $\{m_w\}_{w\in\calw}$
can be recovered from the Euler numbers
and the multiplicities $\{m_a\}_{a\in\cala}$, cf. \ref{bek:AI}.\ix{matrix!intersection}\vspace{2mm}

(4)\ $m_v>0$ for any $v\in{\cal V}$,
hence the set of multiplicities determine the
Euler numbers completely via the equations (\ref{eq:2.2.1}).
This `naive' property has a rather important technical  advantage:
a multiplicity can always be determined by a local computation, on the other hand the Euler number is a
global characteristic class.

This principle will be used frequently in the present book.\vspace{2mm}

(5) The graphs  $\G(X,f)$ and $\G(X)$
are  connected  as follows from  Zariski's Main  Theorem, see  \cite{La}
or \cite{Hartshorne}.
\ix{graph!resolution}
\ix{graph!resolution!properties}
\ix{graph!negative definite}
\ix{Zariski}

\bekezdes\label{1.6}{\bf Examples.} \

\vspace{2mm}

\noindent   {\bf Plane curve singularities.}
If  $(X,x)$ is smooth,  then $(V_f,0)\subset (X,x)$ can be resolved using only quadratic modifications.
In this case, the graph $\G(X,f)$ is a tree, and $g_w=0$ for any $w\in\calw$.
See e.g. \cite{BrKn}.
\ix{singularities!plane curve}

\vspace{2mm}

\noindent   {\bf Cyclic coverings.}
Start with a normal surface singularity $(X,x)$ and a germ $f:(X,x)\to
(\bfc,0)$. Consider the covering $b:(\bfc,0)\to (\bfc,0)$ given by
$z\mapsto z^N$, and construct the fiber product:
$$(X,x)\prod_{f,b}(\bfc,0)=\{(x',z)\in(X\times \bfc,x\times 0)\ :\ f(x')=
z^N\}.$$
By definition, $X_{f,N}$ is the normalization of
$(X,x)\prod_{f,b}(\bfc,0)$.
The first projection induces  a ramified covering $X_{f,N}\to X$ branched along
$V_f$, with covering  transformation group $\Z_N$. For more details see \ref{ss:b}.
\ix{singularities!cyclic covers|textbf}

\vspace{2mm}

\noindent  {\bf Hirzebruch--Jung singularities \cite{BPV,Hir1,La,R74,R81}.}
For a normal surface singularity, the following conditions are equivalent.
If $(X,x)$ satisfies either one of them, then it is called Hirzebruch--Jung
singularity.
\ix{singularities!Hirzebruch--Jung|textbf}

\vspace{2mm}

(a) \ The resolution graph $\G(X)$ is a string, and  $g_w=0$ for any
$w\in\calw$. (In the terminology of low-dimensional topology, this is equivalent to the fact that
the link $K_X$ is a {\it lens space}.) %\marginpar{REFRENCE}
%(If the graph is minimal, that is, there is no vertex $w\in\calw$ with Euler number $-1$,
%$g_w=0$ and number of adjacent edges $\leq 2$,  then
% $e_w\leq -2$ for any $w$.)
\ix{graph!resolution!string}\ix{lens space|textbf}

\vspace{2mm}

(b) \ There is a finite proper map
$\pi:(X,x)\to (\bfc^2,0)$ such that the reduced discriminant locus
of $\pi$, in some local coordinates $(u,v)$ of $(\bfc^2,0)$, is $\{uv=0\}$.

\vspace{2mm}

(c) \ $(X,x)$ is isomorphic with exactly one of the `model spaces'
$\{A_{n,q}\}_{n,q}$, where $A_{n,q}$ is the normalization of
 $(\{xy^{n-q}+z^n=0\},0)$, where $0<q<n,\ (n,q)=1$.

 \vspace{2mm}

Usually, Hirzebruch--Jung singularities appear as in (b).
If there is a map $\pi$ as in  (b) with smooth reduced discriminant locus, then
$(X,x)$ is automatically smooth. The following local situation is a  prototype.

\vspace{2mm}

 For any three {\em strictly positive} integers $a,b$
and $c$, with gcd$(a,b,c)=1$, let $(X,x)$ be the normalization of
$(\{x^a y^b+z^c=0\},0)\subset (\bfc^3,0)$.
Then the projection to the $(x,y)$--plane induces a map which satisfies (b).
Let $z:(X,x)\to
(\bfc,0)$ be induced by $(x,y,z)\mapsto z$. Then the minimal
embedded resolution graph of the pair $(X,z)$ is the following.
(In the sequel sometimes we write $(a,c)$ for $\mbox{gcd}(a,c)$.)

First,  consider  the unique $0\leq \lambda <c/(a,c)$  and
$m_1\in\N$ with: \begin{equation}\label{eq:2.2.2} b+\lambda \cdot
\frac{a}{(a,c)}=m_1\cdot \frac{c}{(a,c)}.\end{equation}
If $\lambda\not=0$, consider the
continued fraction:

\begin{equation}\label{eq:HCF}
\frac{c/(a,c)}{\lambda}=k_1-{1\over\displaystyle
k_2-{\strut 1\over\displaystyle\ddots -{\strut 1\over k_s}}}, \ \ \ \
k_1,\ldots, k_s\geq 2.\end{equation} Then the next string is the embedded
resolution graph of $z$:
\ix{graph!resolution!string|textbf}
\ix{Hirzebruch--Jung string|see{graph/resolution/string}}
\ix{decorations|see{graph decorations}}

%\vspace{2mm}

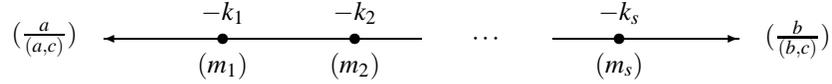
\begin{picture}(370,50)(70,0)
\put(95,25){\makebox(0,0)[r]{$(\frac{a}{(a,c)})$}}
%\put(40,25){\makebox(0,0)[r]{$Str(a,b;N):$}}
\put(355,25){\makebox(0,0)[l]{$(\frac{b}{(b,c)})$}}
\put(150,35){\makebox(0,0){$-k_1$}}
\put(200,35){\makebox(0,0){$-k_2$}}
\put(300,35){\makebox(0,0){$-k_s$}}
\put(150,15){\makebox(0,0){$(m_1)$}}
\put(200,15){\makebox(0,0){$(m_2)$}}
\put(300,15){\makebox(0,0){$(m_s)$}} \put(150,25){\circle*{4}}
\put(200,25){\circle*{4}} \put(300,25){\circle*{4}}
\put(225,25){\vector(-1,0){120}} \put(275,25){\vector(1,0){70}}
\put(250,25){\makebox(0,0){$\cdots$}}
\end{picture}

%\vspace{2mm}

\noindent The arrow at the left (resp. right) hand side codifies
the strict transform of $\{x=0\}$ (resp. of $\{y=0\}$).
%All vertices have genus $g_w=0$, i.e. they represent rational
%irreducible exceptional divisors
The first vertex has multiplicity $m_1$ given by (\ref{eq:2.2.2});
while $m_2,\ldots, m_s$ can be computed by  (\ref{eq:2.2.1})
with all edge--signs $\epsilon=+$, namely:
$$-k_1m_1+\frac{a}{(a,c)}+m_2=0,\ \mbox{and}\
-k_im_i+m_{i-1}+m_{i+1}=0\ \ \mbox{for \ $i\geq 2$}.$$

The same graph might also serve as the resolution graph of the germ
$x$,  induced by the projection $(x,y,z)\mapsto x$.
The multiplicities of $x$ are given in the
next graph, where the arrows codify {\it the  same} strict transforms:

%\vspace{2mm}

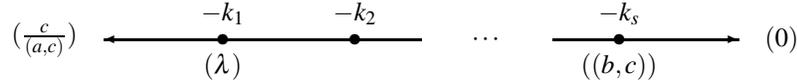
\begin{picture}(400,50)(70,0)
\put(355,25){\makebox(0,0)[l]{$(0)$}}
\put(150,35){\makebox(0,0){$-k_1$}}
\put(200,35){\makebox(0,0){$-k_2$}}
\put(300,35){\makebox(0,0){$-k_s$}}
\put(150,15){\makebox(0,0){$(\lambda)$}}
\put(95,25){\makebox(0,0)[r]{$(\frac{c}{(a,c)})$}}
\put(150,25){\circle*{4}} \put(200,25){\circle*{4}}
\put(300,25){\circle*{4}} \put(300,15){\makebox(0,0){$((b,c))$}}
\put(275,25){\vector(1,0){70}} \put(225,25){\vector(-1,0){120}}
\put(250,25){\makebox(0,0){$\cdots$}}
\end{picture}

%\vspace{2mm}

\noindent The other multiplicities can  again be  computed by
(\ref{eq:2.2.1}), with all edge--signs $+$.
Symmetrically,  the graph of $y$ is:

%\vspace{2mm}

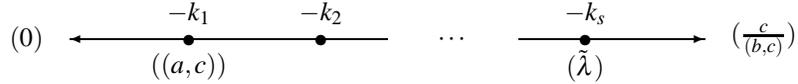
\begin{picture}(400,50)(70,0)
\put(95,25){\makebox(0,0)[r]{$(0)$}}
\put(355,25){\makebox(0,0)[l]{$(\frac{c}{(b,c)})$}}
\put(150,35){\makebox(0,0){$-k_1$}}
\put(200,35){\makebox(0,0){$-k_2$}}
\put(300,35){\makebox(0,0){$-k_s$}}
\put(150,15){\makebox(0,0){$((a,c))$}} \put(150,25){\circle*{4}}
\put(200,25){\circle*{4}} \put(300,25){\circle*{4}}
\put(300,15){\makebox(0,0){$(\tilde{\lambda})$}}
\put(225,25){\vector(-1,0){120}} \put(275,25){\vector(1,0){70}}
\put(250,25){\makebox(0,0){$\cdots$}}
\end{picture}

%\vspace{2mm}

\noi where $0\leq \tilde{\lambda}<c/(c,b)$ and
$$a+\tilde{\lambda}\cdot \frac{b}{(b,c)}=m_s\cdot\frac{c}{(b,c)}.$$
Hence, the embedded resolution graph
$\G(X,x^i y^j z^k)$ of the germ
$x^iy^jz^k$ defined on $X$ is the graph with the
same shape, same Euler numbers and genera, and the multiplicity
$m_v$ (for any vertex $v$) satisfying
$$m_v(x^iy^jz^k)=i\cdot m_v(x)+j\cdot m_v(y)+ k\cdot m_v(z).$$
\ix{graph!resolution!string}

\begin{remark}\label{re:HJCF}
Note the negative sign in the (unusual) continued fraction expansion (\ref{eq:HCF}).
Such an expression  in the sequel will be called {\it Hirzebruch--Jung  continued fraction},
and it will be denoted by $[k_1,k_2,\ldots,k_s]$.
In the present work all continued fraction expansions are of this type.
\end{remark}
\ix{Hirzebruch--Jung continued fraction|textbf}

\begin{bekezdes}\label{def:STRING} {\bf The strings $Str$.}
Next, we wish to define a string
which will be used systematically in the main result. Strangely
enough, its Euler numbers will be deleted. Before providing the
precise definition, let us give a reason for it; see also  \ref{1.3b}(4).

As we already mentioned, once we know all the  multiplicities and
edge decorations, all the Euler numbers can be recovered by
(\ref{eq:2.2.1}). In the main construction, we will glue together
graphs whose multiplicity systems on common parts agree, but under
the gluings the Euler numbers might change. In particular, as
`elementary blocks' of the gluing construction we consider graphs
with multiplicity decorations, but no Euler numbers: their `correct'
Euler numbers will be determined last, after the gluing has been done and after
deciding the edge--decorations.
\end{bekezdes}

\begin{definition}\labelpar{def:2.2.1} \ix{graph!resolution!string}
In the sequel, for positive integers $a$, $b$ and $c$,
$$Str^+(a,b;c\,|\,i,j;k) \ \ \mbox{or} \ \ \ Str(a,b;c\,|\,i,j;k)$$
denotes the string $\G(X,x^iy^jz^k)$,
together with its two arrowheads,  all  vertices (arrow--heads or not)
weighted with  multiplicities as above,
and with all edge decorations $+$, but with all Euler numbers
deleted. Moreover, $$Str^\circleddash(a,b;c\,|\,i,j;k)$$ denotes the
same string but with all edges (connecting arrowheads and
non-arrowheads) decorated with $\circleddash$.

In particular, if $\lambda=0$, then the corresponding string
$Str^\pm(a,b;c\,|\,i,j;k)$ is a double arrow (decorated with $+$
or $\circleddash$) having no non--arrowhead vertex.

The same $\pm$--double-arrow will be used  for the string
$Str^\pm(a,b;c\,|\,i,j;k)$ even if $a=0$ or $b=0$ (with the
convention $a/(a,b)=0$ whenever $a=0$).
\end{definition}

\section{\ Examples. \ Multiplicity systems and Milnor fibrations}\label{link}
\setcounter{equation}{0}

\bekezdes\label{3.1}{\bf The homology of the link $K_X$ of $(X,x)$.}
Let $(X,x)$ be a normal surface singularity, and let $K_X$ be its link, cf. \ref{d1}.
We fix a resolution ${\cal Y}$ with exceptional set $E=\{E_w\}_{w\in\calw}$
and dual resolution graph $\G=\G(X)$. The  intersection matrix $A$\ix{matrix!intersection}
can be identified with a
$\bfz$--linear map $A:\bfz^{|\calw|}\to (\bfz^{|\calw|})^*$,
where for a $\bfz$--module $M$, $M^*$ denotes its dual $Hom_{\bfz}(M,\bfz)$.
Since $A$ is non--degenerate, $\coker (A)$ is a torsion group with
$|\coker (A)|=|\det (A)|$. Then from the homological long exact sequence
of the pair $({\cal Y},K_X)=(P(\G),M(\G))$ one has

\begin{proposition}\label{3.4} \cite{HNK,Mu,Sch} \
$$H_1(K_X,\bfz)= \coker(A)\oplus H_1(E,\Z)=\coker (A)\oplus \bfz^{2g(\G)+c(\G)}.$$
\end{proposition}
\ix{link!of a surface singularity!homology of}

\bekezdes\label{3.6}{\bf The topology of the pair $(K_X,V_f)$.} \
We consider an analytic germ $f:(X,x)\to (\bfc,0)$; which is not necessarily an isolated singularity.
Set $K_f:=K_X\cap V_f$.
In this section we wish to compare the
multiplicity system of $\G(X,f)$ and the generalized Milnor fibration
$\arg:=f/|f|:K_X\setminus K_f\to S^1$. If $f$ defines an isolated singularity, then this is an open book
of $K_X$ with binding $K_f$.

The next Proposition is a general fact for compact 3--manifolds, see for example
\cite[page 34]{EN}: \ix{Eisenbud--Neumann book}
\begin{proposition}\label{3.12}  The fibration $\arg:K_X\setminus K_f\to
S^1$ is completely  determined, up to an isotopy,
by  the induced representation $\arg_*:H_1(K_X\setminus K_f,\bfz)\to \bfz$.
Moreover, if  ~$\bfz_n:=
\coker(\arg_*)$, then the page
of $\arg$ has $n$ connected components which are cyclically permuted by the monodromy.
\end{proposition}

The map $\arg_*$ can be compared with the multiplicity system as follows.

Let $\bfz^{\calv}$ be the free abelian group generated by $\{\langle v\rangle\}_{v\in\calv}$.
Define the group $H_{\G}$ as the quotient of
$\bfz^{\calv}$ factorized by the subgroup generated by:
$$e_w\langle w\rangle +\sum_{v\in\calv_w}\langle v\rangle \ \ \ (\mbox{for all $w\in\calw$}).$$

\begin{proposition}\label{3.9} \cite{HNK,neumann,Sch} \ One has the following exact sequences
$$0\to\bfz^{\cala}\stackrel{i}{\to}H_{\G}\to \coker (A)\to 0,$$
$$0\to H_{\G}\stackrel{j}{\to} H_1(K_X\setminus K_f,\bfz)\stackrel{q}{\to}
H_1(E,\bfz)\to 0,$$
where $i$ is the composed map $\bfz^{\cala}\hookrightarrow \bfz^{\calv}\to
H_{\G}$, and $j([\langle v\rangle] )=[\gamma_v]$.
(For the definition of $\gamma_v$ see \ref{Openbooks}.)
\end{proposition}
Define  ${\bf m}:H_{\G}\to\bfz$ by
 ${\bf m}([\langle v\rangle ])=m_v$. Then $\arg_*\circ j={\bf m}$ inserts into the diagram:
\ix{link!of curve singularities!homology of}

\begin{picture}(300,70)(0,0)
\put(150,50){\makebox(0,0){$
0\to H_{\G}\stackrel{j}{\longrightarrow} H_1(K_X\setminus K_f,\bfz)
\stackrel{q}{\longrightarrow}H_1(E,\bfz)\to 0.$}}
\put(140,35){\vector(0,-1){15}}
\put(155,30){\makebox(0,0){$\arg_*$}}
\put(90,40){\vector(2,-1){40}}
\put(94,30){\makebox(0,0){${\bf m}$}}
\put(140,15){\makebox(0,0){$\bfz$}}
\end{picture}

%\vspace{2mm}

If $K_X$ is a rational homology sphere, that is $H_1(E,\bfz)=0$, then $j$ is an isomorphism,
and the set $\{m_a\}_{a\in\cala}$
determines completely the Milnor fibration up to an isotopy.

In general one has the  divisibilities,  where $n$ is the order of  ${\rm coker}(\arg_*)$ as above,
$$n\,\big|\,{\rm gcd}\{m_v:v\in\calv\}\,\big|\,{\rm gcd}\{m_a:a\in\cala\}.$$
Nevertheless,  it might happen
that $n\not= {\rm gcd}\{m_v:v\in\calv\}$, hence  $n$ cannot be determined from ${\bf m}$.
Or, even if $n={\rm gcd}\{m_v:v\in\calv\}$, in general,  from the multiplicities
one cannot recover $\arg_*$.
We illustrate  this by some examples borrowed from \cite{cyclic}.

In these  examples $(X,x)$ will be a Brieskorn singularity, or a cyclic covering.
The reader might determine the corresponding graphs by his/her preferred method applicable for such cases;
nevertheless,  for such  singularities, and for any
of the coordinate functions, one can deduce the embedded resolution
by using the algorithm of cyclic coverings presented in the section \ref{ss:b}.

\begin{example}\label{3.14}
Set $(X,x)=(\{x^2+y^7-z^{14}=0\},0)\subset
(\bfc^3,0)$ and take $f_1(x,y,z)=z^2$ and $f_2(x,y,z)=z^2-y$. Then
$\G(X,f_1)=\G(X,f_2)$ is the graph

%\vspace{2mm}

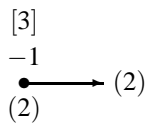
\begin{picture}(400,70)(0,-5)
\put(110,20){\circle*{4}}
\put(110,30){\makebox(0,0){$-1$}}
\put(110,43){\makebox(0,0){$[3]$}}
\put(110,20){\vector(1,0){30}}
\put(110,10){\makebox(0,0){$(2)$}}
\put(150,20){\makebox(0,0){$(2)$}}
\end{picture}

%\vspace{2mm}

In both cases $\coker(\arg_*)$ is a factor
of $\coker({\bf m})=\bfz_2$. In fact,
 $\coker(\arg_*(f_1))=\bfz_2$ and  $\arg_*(f_2)$ is
onto. Indeed, the Milnor fibration of $z^2$ is the pullback by $z\mapsto z^2$
of the Milnor fibration of $z$, hence $\coker(\arg_*)=\bfz_2$.
For the second statement it is enough to verify that
the double covering
$\{x^2+y^7-z^{14}=w^2+y-z^2=0\}\subset \bfc^4$ is irreducible
(notice that our equations are quasi--homogeneous, hence we can replace a small
ball centered at the origin with the whole affine space). But this is
true if its intersection with $y=1$, that is $C:=\{x^2=z^{14}-1;w^2=z^2-1\}
\subset \bfc^3$,  is irreducible. The covering $C\to \bfc$,  $(x,w,z)\mapsto z$,
is a $\bfz_2\times\bfz_2$ covering. The monodromy around $\pm 1$ is
$(-1,-1)$, and around any $\alpha$ with $\alpha^{14}=1$ and $\alpha^2\not=1$
is $(-1,+1)$. Hence the global monodromy group is
the whole group $\bfz_2\times\bfz_2$, in particular $C$ is irreducible.

Notice also that all the multiplicities of $f_2$ are even, nevertheless there
is no germ $g:(X,x)\to (\bfc,0)$ with $f_2=g^2$.
\end{example}

\begin{example}\label{3.16}\
Set $(X,x)=(\{z^2+(x^2-y^3)(x^3-y^2)=0\},0)$
and $f_1=x^2$ and $f_2=x^2-y^3$.
Then $\G(X,f_1)=\G(X,f_2)$ is the following graph

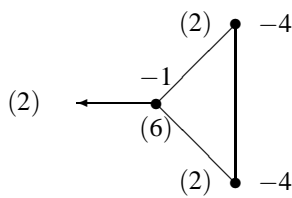
\begin{picture}(400,100)(0,-30)
\put(110,20){\circle*{4}}
\put(140,50){\circle*{4}}
\put(140,-10){\circle*{4}}
\put(110,20){\line(1,1){30}}
\put(110,20){\line(1,-1){30}}
\put(140,50){\line(0,-1){60}}
\put(110,30){\makebox(0,0){$-1$}}
\put(155,50){\makebox(0,0){$-4$}}
\put(125,50){\makebox(0,0){$(2)$}}
\put(155,-10){\makebox(0,0){$-4$}}
\put(125,-10){\makebox(0,0){$(2)$}}
\put(110,20){\vector(-1,0){30}}
\put(110,10){\makebox(0,0){$(6)$}}
\put(60,20){\makebox(0,0){$(2)$}}
\end{picture}

\noindent
Then again,   $\arg_*(f_1)$  has cokernel $\bfz_2$, and $\arg_*(f_2)$ is onto.
\end{example}
Notice that in the above examples, $(X,x)=(\{z^2+h(x,y)=0\},0)$, $h$ is reduced, $f_2$
divides $h$ but it is not equal to $h$.
For all such cases the monodromy argument given in \ref{3.14} is valid.
But all these examples  define non--isolated singularities.

In order to construct examples of germs which define isolated singularities,
we will use the well--known construction of series of singularities.
Namely, assume that $f_1$ and $f_2$ have the same graph but have different
representations $\arg_*$, and  their zero sets have non--isolated singularities.
Next, we find a germ $g$ such that the zero set of $f_i$ and $g$
have no common components ($i=1,2$).
Then, for sufficiently large $k$,
the germs $f_1+g^k$ and $f_2+g^k$, will
define isolated singularities with the same embedded
resolution graphs, but different representations $\arg_*$.

\begin{example}\label{3.17}\
Set $(X,x)=(\{x^2+y^7-z^{14}=0\},0)\subset
(\bfc^3,0)$ and take $f_1(x,y,z)=z^2$ and $f_2(x,y,z)=z^2-y$ as in \ref{3.14}.
Let $P$ be the intersection point of the strict transform $S_a$ of
$\{f_i=0\}$ with the exceptional divisor $E$. Then, in some local  coordinate
system $(u,v)$ of $P$, $\{u=0\}$ represents $E$,
$\{v=0\}$ represents $S_a$, and $f_i=u^2v^2$. Consider $g=y$. Since $y$ in the
neighborhood of $P$ can be represented as $y=u^2$ (modulo a local invertible
germ), $f_i+g^k$ near $P$ has the form $u^2v^2+u^{2k}$. For example, if $k=2$,
then one needs one more blowing up in order to resolve $f_i+g^k$.

Therefore, $\G(X,z^2+y^k)=\G(X,z^2-y+y^k)$ for any $k\geq 2$; and for
$k=2$,  the graphs have the following form:

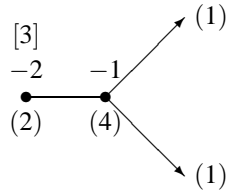
\begin{picture}(400,100)(0,-30)
\put(110,20){\circle*{4}}
\put(140,20){\circle*{4}}
\put(110,20){\line(1,0){30}}
\put(140,20){\vector(1,1){30}}
\put(140,20){\vector(1,-1){30}}

\put(110,30){\makebox(0,0){$-2$}}
\put(140,30){\makebox(0,0){$-1$}}
\put(110,43){\makebox(0,0){$[3]$}}
\put(110,20){\line(1,0){30}}

\put(110,10){\makebox(0,0){$(2)$}}
\put(140,10){\makebox(0,0){$(4)$}}
\put(180,50){\makebox(0,0){$(1)$}}
\put(180,-10){\makebox(0,0){$(1)$}}
\end{picture}

Notice that now ${\bf m}$ is onto, hence both $\arg_*(f_i)$
are onto. Nevertheless, $\arg_*(f_1)\not=arg_*(f_2)$ since their restrictions
to a subgroup of $H_1(K_X\setminus K_f)$ (localized near the exceptional curve of genus 3)
are different. This follows similarly as in \ref{3.14}.
\end{example}
\begin{example}\label{3.18}\
Set $(X,x)=(\{z^2+(x^2-y^3)(x^3-y^2)=0\},0)$
and $f_1=x^2+y^k$ and $f_2=x^2-y^3+y^k$, where $k\geq 4$.
Then by a similar argument as above,
 $\G(X,f_1)=\G(X,f_2)$, but $\arg_*(f_1)\not= \arg_*(f_2)$. The graph for $k=4$ is

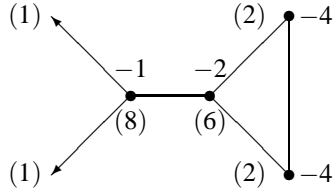
\begin{picture}(400,100)(-20,-30)
\put(80,20){\circle*{4}}
\put(110,20){\circle*{4}}
\put(140,50){\circle*{4}}
\put(140,-10){\circle*{4}}
\put(80,20){\line(1,0){30}}

\put(110,20){\line(1,1){30}}
\put(110,20){\line(1,-1){30}}
\put(140,50){\line(0,-1){60}}
\put(110,30){\makebox(0,0){$-2$}}
\put(80,30){\makebox(0,0){$-1$}}
\put(150,50){\makebox(0,0){$-4$}}
\put(125,50){\makebox(0,0){$(2)$}}
\put(150,-10){\makebox(0,0){$-4$}}
\put(125,-10){\makebox(0,0){$(2)$}}
\put(80,20){\vector(-1,-1){30}}
\put(80,20){\vector(-1,1){30}}
\put(110,10){\makebox(0,0){$(6)$}}
\put(80,10){\makebox(0,0){$(8)$}}
\put(40,50){\makebox(0,0){$(1)$}}
\put(40,-10){\makebox(0,0){$(1)$}}
\end{picture}

\end{example}

\chapter{Cyclic coverings of graphs}\labelpar{ss:2.3}
\section{\ The general theory of cyclic coverings}
\labelpar{ss:2.3b}\setcounter{equation}{0}
In this section  we review a graph--theoretical construction from  \cite{cyclic}.

%In
%order to avoid some extra-notations,  we will assume that our
%graphs have no loops. The reader is invited to extend the
%presentation for graphs with loops.

% If $v\in{\mathcal V}$ is a fixed vertex, let
%${\mathcal E}_v$ be the set of edges which have $v$ as one of
%their end--points.
%If $|\Gamma|$ is the topological realization of the graph $\Gamma$,
%then we denote  the rank of $H_1(|\Gamma|,\bfz)$ by $c_{\Gamma}$,
%i.e. $c_{\Gamma}$ is the number of independent $p_{\mathcal V}(v_1)$ and  $p_{\mathcal V}(v_2)$cycles of the graph
%$\Gamma$.

\begin{definition}\labelpar{def:2.3.1}
A morphism of graphs $p:\Gamma_1\rightarrow\Gamma_2$ consists of
two maps $p_{\mathcal V}: {\mathcal V}(\Gamma_1)\rightarrow
{\mathcal V}(\Gamma_2)$ and $p_{\mathcal E}: {\mathcal
E}(\Gamma_1)\rightarrow {\mathcal E}(\Gamma_2)$, such that if
$e\in{\mathcal E}(\Gamma_1)$ has end--vertices $v_1$ and $v_2$, then
$p_{\mathcal V}(v_1)$ and  $p_{\mathcal V}(v_2)$ are the
end--vertices of $p_{\mathcal E}(e)$. If $p_{\mathcal V}$ and
$p_{\mathcal E}$ are isomorphisms of sets, then we say that $p$ is
an isomorphism of graphs.
\ix{graph!covering|textbf}\ix{graph!covering!data|textbf}

If ~$\Gamma$ is a graph, we say that $\Z$ acts on $\Gamma$, if
there are group--actions $a_{\calv}:
\Z\times\calv\rightarrow\calv$ and
 $a_{\cale}: \Z\times\cale\rightarrow\cale$ of \ $\Z$
with the following  compatibility property: if $e\in {\mathcal E}$
has end--vertices  $v_1$ and $v_2$, then $a_\cale (h,e)$ has
end--vertices $a_\calv (h,v_1)$ and $a_\calv (h,v_2)$. The action is
trivial if $a_\calv$ and $a_\cale$ are trivial actions.

If \ $\Z$ acts on both $\Gamma_1$ and $\Gamma_2$, then  a morphism
$p:\Gamma_1\to \Gamma_2$ is equivariant if the maps $p_{\mathcal
V}$ and $p_{\mathcal E}$ are equivariant with respect to the
actions of $\Z$. If additionally $p$ is an isomorphism then it is
called an equivariant isomorphism of graphs.
\end{definition}

Fix a finite graph $\Gamma$, and assume that $\bfz$ acts on it in
a trivial way.

\begin{definition}\labelpar{def:2.3.2}  A $\bfz$--covering, or cyclic covering of
\,$\Gamma$ consists of a finite graph $G$, that carries a
$\bfz$--action, together with an equivariant morphism
$p:G\rightarrow\Gamma$ such that the restriction of the
$\bfz$--action on any set of type $p^{-1}(v)$ ($v\in \cvg$),
respectively $p^{-1}(e)$ ($e\in\evg$), is transitive.
\end{definition}

Fix a cyclic covering $p:G\to\G$.
For any $v\in\cvg$,  let $\n_v\bfz$  be  the maximal subgroup of $\bfz$
which acts trivially on $p^{-1}(v)$.
Similarly,  for any $e\in\evg$ with end--vertices
$\{v_1,v_2\}$, let $\dg_e\cdot
{\rm lcm}(\n_{v_1},\n_{v_2})\bfz$ be  the maximal subgroup of $\bfz$
which acts trivially on  $p^{-1}(e)$.
These numbers  define a system of strictly positive integers
$$(\bn, \bd)=\left\{ \{\n_v\}_{v\in\calv(\Gamma)};
\{\dg_e\}_{e\in\cale(\Gamma)} \right\}.$$
\begin{definition}\labelpar{def:2.3.2b}
$(\bn, \bd)$  is called the
covering data of the covering $p:G\to\G$.
\end{definition}
%\vspace{2mm}
\noindent
Sometimes we write
$$\n_e:=\dg_e\cdot {\rm lcm}(\n_{v_1},\n_{v_2}).$$
\begin{definition}\labelpar{eq:2.3.3} Two cyclic coverings $p_i :
G_i\rightarrow\Gamma\ \ (i=1,2)$ are equivalent, denoted by  $G_1 \sim G_2$,
if there is an equivariant isomorphism $q:G_1\rightarrow G_2$ such
that $p_2 \circ q=p_1$.

The set of equivalence classes of cyclic coverings of \,$\Gamma$,
all  associated with a fixed covering data $(\bn,\bd)$, is
denoted by $\fedog$.
\end{definition}

\begin{theorem}\label{th:COVERINGS} \cite{cyclic} \ %One has the following facts:
$\fedog$ has an abelian group structure and it is independent of $\bd$.
%\begin{enumerate}
%\item

% where $T$ plays the role of the neutral element.
%\item  $\fedog$
%\item $\fedog$  contains a distinguished element  $T$, the
%`trivial covering'.  It  is characterized
% by the existence of a non-equivariant morphism
%of graphs  $s:\Gamma\rightarrow T$ with $p\circ s=id_\Gamma$.
%\end{enumerate}
\end{theorem}
In general, $\fedog$ is non-trivial.
Here are some examples.

\begin{example}\label{ex:COVERING} \cite{cyclic} \ For a covering data $(\bn ,\bd )$ one has:
\begin{enumerate}
\item
Assume that $\Gamma$ is a tree. Then  $\fedog =0$ for any $(\bn ,\bd )$.
Therefore, up to an isomorphism,  there is only one covering $G$ of $\G$. It
has exactly ${\rm gcd}\{\n_v\}_{v\in{\mathcal
V}(\Gamma)}$  connected components. \ix{graph!covering}

\vspace{2mm}

\item
Assume that  $\G$ is  a cyclic graph, that is  $\calv(\G)=\{v_1,v_2,\ldots,v_k\}$ and
\ $\cale(\G)=\{(v_1,v_2),(v_2,v_3),$ $\ldots,(v_k,v_1)\}$, where $k\geq 3$.
Then ${\cal G}(\G,({\bf n},{\bf d}))=\bfz_n$, where
$n={\rm gcd}\{n_v: v\in\calv(\G)\}$.

\vspace{2mm}

\item
For any subgraph $\Gamma '\subset\Gamma$ there is a natural surjection
$pr:\fedog\rightarrow {\cal G}(\Gamma ', (\bn,\bd))$.

\vspace{2mm}

\item
Let $\G$ be a graph with $c(\G)=1$, and
let $\G'$ be the  unique minimal cyclic subgraph of $\G$. If
$n:={\rm gcd}\{n_v:v\in\calv(\G')\}$,  then $\fedog=\bfz_n$.
\end{enumerate}
\end{example}

\begin{remark}\label{cgcg} If $p:G\to \Gamma$ is a cyclic covering  then
$c(G)\geq c(\Gamma)$.
Indeed,  if $|G|$ and $|\Gamma|$ denote the topological realizations of
the corresponding graphs, then the invariant subspace $H_1(|G|,\Q)^{\Z}$ is isomorphic to
$H_1(|\Gamma|, \Q)$.  \end{remark}\ix{graph!topological realization}

 On the other hand, we have the  following result which will be a key
 ingredient of the main construction in \ref{covdata}:

\begin{theorem}\labelpar{th:2.3.1} \cite[(1.20)]{cyclic}
Fix $\G$ and ${\bf (n,d)}$ as above. Set ${\mathcal V}^1:=\{v\in
{\mathcal V}(\G):\n_v=1\}$. Let $\G_{\not=1}$ be the subgraph of
\,$\G$ obtained from $\G$ by deleting the vertices from ${\mathcal
V}^1$ and their adjacent edges. If each connected component of
$\G_{\not=1}$ is a tree, then ${\mathcal G}(\G,{\bf (n,d)})=0$.
\end{theorem}

\bekezdes\labelpar{re:2.3.1} {\bf Variations.} \ One extends
the set of coverings as follows (cf. \cite[1.25]{cyclic}).

\begin{enumerate}
\item  Assume that we have two types of vertices: arrowheads $\cala$
and non--arrowheads $\calw$, i.e. $\calv=\cala\sqcup\calw$. Then
in the definition of a coverings $p:G\to\G$ we add the following
axiom: $\cala(G)=p^{-1}
(\cala(\G))$.

\vspace{2mm}

\item Assume that our graphs have some decorations. Then for a
covering $p:G\to\G$ we also
require that the decorations of $G$ must be equivariant.

\vspace{2mm}

\item `Equivariant string insertion'  means the following modification of $\G$.
One  starts with the following data: \ix{graph!resolution!string}

\vspace{1mm}

(a)\ a graph $\G$ and a system ${\bf (n,d)}$ as above;

\vspace{1mm}

(b)\ a covering $p:G\to \G$, as an element of ${\mathcal
G}(\G,{\bf (n,d)})$;

\vspace{1mm}

(c)\ for any edge $e$ of $\G$, we fix a string $Str^\pm(e)$ (with
decorations):
%\end{enumerate}

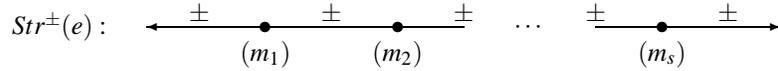
\vspace{3mm}

\begin{picture}(400,40)(40,0)
\put(90,25){\makebox(0,0)[r]{$Str^\pm(e):$}}
\put(150,15){\makebox(0,0){$(m_1)$}}
\put(200,15){\makebox(0,0){$(m_2)$}}
\put(300,15){\makebox(0,0){$(m_s)$}} \put(150,25){\circle*{4}}
\put(200,25){\circle*{4}} \put(300,25){\circle*{4}}
\put(225,25){\vector(-1,0){120}} \put(275,25){\vector(1,0){70}}
\put(250,25){\makebox(0,0){$\cdots$}}
\put(125,30){\makebox(0,0){$\pm$}}
\put(175,30){\makebox(0,0){$\pm$}}
\put(225,30){\makebox(0,0){$\pm$}}
\put(275,30){\makebox(0,0){$\pm$}}
\put(325,30){\makebox(0,0){$\pm$}}
\end{picture}

\vspace{2mm}

Then the new graph $G(\{Str^\pm(e)\}_{e})$ is constructed as
follows: we replace each edge $\tilde{e}\in p_{\cale}^{-1}(e)$
(with end--vertices $\tilde{v}_1$ and $\tilde{v}_2$ and decoration
$\pm$) of $G$ as  shown below:

\end{enumerate}

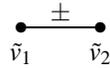
\vspace{2mm}

\begin{picture}(100,50)(-110,0)
\put(0,25){\circle*{4}} \put(30,25){\circle*{4}}\put(0,25){\line(1,0){30}}
\put(15,30){\makebox(0,0){$\pm$}}
\put(0,15){\makebox(0,0){$\tilde{v}_1$}}
\put(30,15){\makebox(0,0){$\tilde{v}_2$}}
\end{picture}

is replaced by

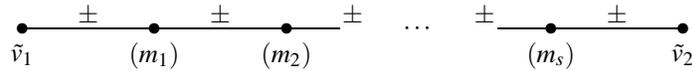
\begin{picture}(300,50)(80,0)
\put(150,15){\makebox(0,0){$(m_1)$}}
\put(200,15){\makebox(0,0){$(m_2)$}}
\put(300,15){\makebox(0,0){$(m_s)$}}
\put(100,15){\makebox(0,0){$\tilde{v}_1$}}
\put(350,15){\makebox(0,0){$\tilde{v}_2$}}

\put(100,25){\circle*{4}} \put(150,25){\circle*{4}}
\put(350,25){\circle*{4}} \put(200,25){\circle*{4}}
\put(300,25){\circle*{4}} \put(100,25){\line(1,0){120}}
\put(350,25){\line(-1,0){70}}
\put(250,25){\makebox(0,0){$\cdots$}}
\put(125,30){\makebox(0,0){$\pm$}}
\put(175,30){\makebox(0,0){$\pm$}}
\put(225,30){\makebox(0,0){$\pm$}}
\put(275,30){\makebox(0,0){$\pm$}}
\put(325,30){\makebox(0,0){$\pm$}}
\end{picture}

\section{\ The universal cyclic covering of $\G(X,f)$}
\labelpar{ss:a}\setcounter{equation}{0}

Fix a normal surface singularity $(X,x)$ and a germ $f:(X,x)\to (\bfc,0)$ of an analytic function.
Fix also a resolution  $\phi:({\cal Y},D)\to (U,V_f)$ as in \ref{1.1},
and consider the associated embedded resolution graph $\G(X,f)$.
In this section we recall the construction of
 a canonical cyclic covering of this graph via the Milnor fibration of
$f$. It is called the {\it universal cyclic covering of $\G(X,f)$}. For more details
see \cite{MDB,cyclic}.
\ix{graph!covering!universal|textbf}

\vspace{2mm}

There are several reasons why we include this construction in the present work:

\vspace{2mm}

$\bullet$ \ The universal cyclic covering is the prototype of all cyclic coverings provided by
geometric constructions. For an analogous construction, which is an important ingredient in some proofs
 of the book, see \ref{felbont}.

\vspace{1mm}

$\bullet$ \ It shows, for a fixed graph $\G(X,f)$, how  one can codify graph-theoretically
the possible differences of $\arg_*$. Thus, it is a necessary complement of section
\ref{link}.

\vspace{1mm}

$\bullet$ \ It guides all the geometry, in particular the resolution graphs, of cyclic coverings
$X_{f,N}$. %; though, we will not include this algorithm, cf. next subsection \ref{ss:b}.

\vspace{1mm}

$\bullet$ \ By examples which show that a graph  can have several  cyclic coverings,
we emphasize  the role and power of
the key Theorem  \ref{th:2.3.1}, which guarantees that the graph provided by the Main Algorithm in
Chapter \ref{s:ALG} is well--defined and unique.

\bekezdes\label{5.1} {\bf The  construction of the covering $p:G(X,f)\to \G(X,f)$.}

Let $\tw$ ($w\in\calw$) be a small
tubular neighborhood of the irreducible curve $E_w$.
By our assumption that any two irreducible components of the exceptional divisor have at most one
intersection point (see \ref{1.1}),
for any edge $e=(v,w)$ connecting two non-arrowheads,
the intersection $\tw\cap\tv$ is
 a bidisc $T_e$. If $T(S_a)$, for $a\in\cala$,  is a small tubular neighborhood of
the irreducible component $S_a$ of the strict transform $S$ (cf. \ref{1.1}),
 and $a$ is adjacent to $w_a\in\calw$, then corresponding to the edge $e=(a,w_a)$
we introduce the bidisc $T_e=T(S_a)\cap T(E_{w_a})$.
Set $T=(\cup_w\tw)\cup(\cup_aT(S_a))$.

Next, we consider the smooth fiber $f^{-1}(\delta)\subset X$ lifted  via
$\phi$. For sufficiently small $\delta>0$, the fiber $F:=(f\circ\phi)^{-1}
(\delta)\subset {\cal Y}$ is in $T$. Set $F_w=F\cap \tw$ for any $w\in\calw$,
$F_a=F\cap T(S_a)$ for any $a\in\cala$, and $F_e=F\cap T_e$ for any $e\in\cale$.

It is possible to chose the geometric
monodromy  acting on $F$ in such a way that it preserves the
subspaces $\{F_v\}_{v\in\calv}$ and $\{F_e\}_{e\in\cale}$. Then the connected
components of $F_v$, respectively  of  $F_e$, are cyclically permuted by this action.
Let $\n_v$ and $\n_e$ be the number of connected components
of $F_v$ and $F_e$ respectively. Then, for any $e=(v_1,v_2)$, we have
$\n_e=\dg_e\cdot {\rm lcm}(\n_{v_1},
\n_{v_2})$ for some $\dg_e\geq 1$.

Now, we are able to construct the covering $p:G(X,f)\to \G(X,f)$  \ix{graph!covering}
associated with the resolution $\phi$.
Above a vertex
$v\in\calv(\G(X,f))$ there are exactly $\n_v$ vertices of $G(X,f)$ which
correspond to the connected components of $F_v$. The $\bfz$--action is induced
by the monodromy.
If $v$ is an arrowhead in $\G$ then by our agreement \ref{re:2.3.1}(1), all the
 vertices in $G$ above $v$  are arrowheads.
Above an edge $e$ of $\G$, there are $\n_e$ edges of
$G$. They correspond to the connected components of $F_e$. The $\bfz$--action is
again generated by the monodromy.

Fix an edge $\tilde{e}$ of $G$ (above the edge $e$ of $\G$) which corresponds
 to the connected component $F_{\tilde{e}}$ of $F_e$. Similarly, take a vertex
$\tilde{v}$ of $G$ (above the vertex $v$ of $\G$) which corresponds to the
connected component $F_{\tilde{v}}$ of $F_v$. Then $\tilde{e}$ has as an
 end  the vertex $\tilde{v}$  if and only if $F_{\tilde{e}}
\subset F_{\tilde{v}}$. In particular,
$\tilde{v_1}$ and $\tilde{v_2}$ are connected in $G$
 if and only if $F_{\tilde{v}_1}\cap F_{\tilde{v}_2}\not=\emptyset$:
if %$F_{\tilde{v}_1}\cap F_{\tilde{v}_2}\not=\emptyset$, then
$F_{\tilde{v}_1}\cap F_{\tilde{v}_2}$ has  $\dg_e$
 connected components, then  $\tilde{v}_1$ and
$\tilde{v}_2$ are connected exactly by $\dg_e$ edges.

\vspace{2mm}

Next, we list some basic properties of the universal covering graph complemented with several
examples. For more details see  \cite{cyclic}.  \ix{graph!covering}

\begin{example}\labelpar{ex:LINKRHS}
Assume that the link $K_X$ is a rational homology sphere.
Then the covering data can be uniquely recovered from $\G(X,f)$, hence, by
 \ref{ex:COVERING}(1), $G(X,f)$ itself is also determined.
Indeed, for any $v\in \calv(\G)$ let $\calv_v$  denote the set of
vertices adjacent to $v$. Then  $n_v= {\rm gcd}\{m_w:\ w\in\calv_v\cup\{v\}\}$
for any $v\in\calv(\G)$, and $n_e:={\rm gcd}(m_{v_1},m_{v_2})$
for any $e=(v_1,v_2)\in\cale(\G)$.  Moreover, the number of connected
components of $G(X,f)$ is exactly ${\rm gcd}\{m_v: v\in\calv(\G)\}$.
\end{example}

In general, one has the following connectedness result.

\begin{lemma}\labelpar{5.3} The number of connected components of the
graph $G(X,f)$ is equal to the number of connected components of the Milnor
fiber $F$ of the germ $f$. This number also agrees with
 $|\coker(\arg_*(f))|$.  In particular,
if $f$ defines an isolated singularity, then $G(X,f)$ is a connected graph.

Since $|\coker(\arg_*(f))|$ divides ${\rm gcd}\{m_v:v\in\calv\}$, the graph
$G(X,f)$ is  connected  whenever   ${\rm gcd}\{m_v:v\in\calv\}=1$.
(Nevertheless,  $G(X,f)$ can be  connected   even if ${\rm gcd}\{m_v:v\in\calv\}\not=1$, see e.g.
\ref{5.9}.)
\end{lemma}

\begin{example}\labelpar{5.9}
Set $(X,x)=(\{x^2+y^7-z^{14}=0\},0)\subset
(\bfc^3,0)$ and take $f_1(x,y,z)=z^2$ and $f_2(x,y,z)=z^2-y$ as in \ref{3.14}.
The next diagrams show  the coverings $p:G(X,f_i)\to \G(X,f_i)$ for  $i=1,2$.  \ix{graph!covering!universal}

Note that the number of connected components  of the graphs $G(X,f_i)$ is different.

\begin{picture}(400,180)(40,-20)
\put(110,20){\circle*{4}}
\put(110,30){\makebox(0,0){$[3]$}}
\put(110,20){\vector(1,0){30}}
\put(110,10){\makebox(0,0){$(2)$}}
\put(150,20){\makebox(0,0){$(2)$}}
\put(125,-10){\makebox(0,0){$(i=1)$}}
\put(125,80){\vector(0,-1){20}}
\put(120,70){\makebox(0,0){$p$}}

\put(110,100){\circle*{4}}
\put(110,130){\circle*{4}}
\put(110,100){\vector(1,0){30}}
\put(110,130){\vector(1,0){30}}

\put(210,20){\circle*{4}}
\put(210,30){\makebox(0,0){$[3]$}}
\put(210,20){\vector(1,0){30}}
\put(210,10){\makebox(0,0){$(2)$}}
\put(250,20){\makebox(0,0){$(2)$}}
\put(225,-10){\makebox(0,0){$(i=2)$}}
\put(225,80){\vector(0,-1){20}}
\put(220,70){\makebox(0,0){$p$}}

\put(210,110){\circle*{4}}
\put(210,110){\vector(1,1){30}}
\put(210,110){\vector(1,-1){30}}
\end{picture}

\end{example}

\begin{example}\labelpar{5.10}
Set $(X,x)=(\{x^2+y^7-z^{14}=0\},0)\subset
(\bfc^3,0)$ and take $f_1(x,y,z)=z^2+y^2$ and $f_2(x,y,z)=z^2-y+y^2$
as in \ref{3.17}.
The next diagrams show
the coverings $p:G(X,f_i)\to \G(X,f_i)$, where $i=1,2$.  \ix{graph!covering!universal}

In this case the number of independent cycles
of the graphs $G(X,f_i)$ is different.

\begin{picture}(400,220)(40,-40)
\put(110,20){\circle*{4}}
\put(140,20){\circle*{4}}
\put(110,20){\line(1,0){30}}
\put(140,20){\vector(1,1){30}}
\put(140,20){\vector(1,-1){30}}
\put(110,30){\makebox(0,0){$[3]$}}
\put(110,10){\makebox(0,0){$(2)$}}
\put(140,10){\makebox(0,0){$(4)$}}
\put(180,50){\makebox(0,0){$(1)$}}
\put(180,-10){\makebox(0,0){$(1)$}}
\put(140,-20){\makebox(0,0){$(i=1)$}}

\put(140,80){\vector(0,-1){20}}
\put(130,70){\makebox(0,0){$p$}}

\put(110,90){\circle*{4}}
\put(110,150){\circle*{4}}
\put(140,120){\circle*{4}}
\put(110,90){\line(1,1){30}}
\put(110,150){\line(1,-1){30}}
\put(140,120){\vector(1,1){30}}
\put(140,120){\vector(1,-1){30}}

\put(210,20){\circle*{4}}
\put(240,20){\circle*{4}}
\put(210,20){\line(1,0){30}}
\put(240,20){\vector(1,1){30}}
\put(240,20){\vector(1,-1){30}}
\put(210,30){\makebox(0,0){$[3]$}}
\put(210,10){\makebox(0,0){$(2)$}}
\put(240,10){\makebox(0,0){$(4)$}}
\put(280,50){\makebox(0,0){$(1)$}}
\put(280,-10){\makebox(0,0){$(1)$}}
\put(240,-20){\makebox(0,0){$(i=2)$}}

\put(240,80){\vector(0,-1){20}}
\put(230,70){\makebox(0,0){$p$}}

\put(210,120){\circle*{4}}
\put(240,120){\circle*{4}}
\put(225,120){\circle{29}}
\put(240,120){\vector(1,1){30}}
\put(240,120){\vector(1,-1){30}}

\end{picture}

\end{example}

\begin{example}\labelpar{5.12}
Set $(X,x)=(\{z^2+(x^2-y^3)(x^3-y^2)=0\},0)$
and $f_1=x^2+y^4$ and $f_2=x^2-y^3+y^4$ as in \ref{3.18}.
Then the coverings $p:G(X,f_i)\to \G(X,f_i)$ (for $i=1,2$) are:  \ix{graph!covering!universal}

\begin{picture}(400,210)(45,-30)
\put(80,20){\circle*{4}}
\put(110,20){\circle*{4}}
\put(140,50){\circle*{4}}
\put(170,-10){\circle*{4}}
\put(80,20){\line(1,0){30}}
\put(110,-20){\makebox(0,0){$(i=1)$}}
\put(110,20){\line(1,1){30}}
\put(110,20){\line(2,-1){60}}
\put(140,50){\line(1,-2){30}}
\put(125,50){\makebox(0,0){$(2)$}}
\put(155,-10){\makebox(0,0){$(2)$}}
\put(80,20){\vector(-1,-1){30}}
\put(80,20){\vector(-1,1){30}}
\put(110,10){\makebox(0,0){$(6)$}}
\put(80,10){\makebox(0,0){$(8)$}}
\put(40,50){\makebox(0,0){$(1)$}}
\put(40,-10){\makebox(0,0){$(1)$}}

\put(120,90){\vector(0,-1){20}}
\put(110,80){\makebox(0,0){$p$}}

\put(80,120){\circle*{4}}
\put(110,110){\circle*{4}}
\put(110,130){\circle*{4}}
\put(140,160){\circle*{4}}
\put(140,140){\circle*{4}}
\put(170,100){\circle*{4}}
\put(170,80){\circle*{4}}
\put(80,120){\line(3,1){30}}
\put(80,120){\line(3,-1){30}}
\put(110,110){\line(1,1){30}}
\put(110,130){\line(1,1){30}}
\put(110,110){\line(2,-1){60}}
\put(110,130){\line(2,-1){60}}
\put(140,160){\line(1,-2){30}}
\put(140,140){\line(1,-2){30}}
\put(80,120){\vector(-1,-1){30}}
\put(80,120){\vector(-1,1){30}}

\put(260,20){\circle*{4}}
\put(290,20){\circle*{4}}
\put(320,50){\circle*{4}}
\put(350,-10){\circle*{4}}
\put(260,20){\line(1,0){30}}
\put(290,-20){\makebox(0,0){$(i=2)$}}
\put(290,20){\line(1,1){30}}
\put(290,20){\line(2,-1){60}}
\put(320,50){\line(1,-2){30}}
\put(305,50){\makebox(0,0){$(2)$}}
\put(335,-10){\makebox(0,0){$(2)$}}
\put(260,20){\vector(-1,-1){30}}
\put(260,20){\vector(-1,1){30}}
\put(290,10){\makebox(0,0){$(6)$}}
\put(260,10){\makebox(0,0){$(8)$}}
\put(220,50){\makebox(0,0){$(1)$}}
\put(220,-10){\makebox(0,0){$(1)$}}

\put(300,90){\vector(0,-1){20}}
\put(290,80){\makebox(0,0){$p$}}

\put(260,120){\circle*{4}}
\put(290,110){\circle*{4}}
\put(290,130){\circle*{4}}
\put(320,160){\circle*{4}}
\put(320,140){\circle*{4}}
\put(350,100){\circle*{4}}
\put(350,80){\circle*{4}}
\put(260,120){\line(3,1){30}}
\put(260,120){\line(3,-1){30}}
\put(290,110){\line(3,5){30}}

\put(290,130){\line(3,1){30}}
\put(290,110){\line(2,-1){60}}
\put(290,130){\line(2,-1){60}}
\put(320,160){\line(1,-2){30}}
\put(320,140){\line(1,-2){30}}
\put(260,120){\vector(-1,-1){30}}
\put(260,120){\vector(-1,1){30}}
\end{picture}

\end{example}

\begin{remark}\label{rem:HOM} {\bf Determining  the complex monodromy.}

Consider a pair $(X,f)$ as above.  Let $\arg:K_X\setminus K_f\to S_1$ be its generalized Milnor fibration
as in \ref{3.12}. Assume that the Milnor fiber $F$ has $n$ connected components, and let $M:H_1(F,\bfc)\to
H_1(F,\bfc)$ be its algebraic monodromy acting on $H_1(F,\bfc)$. Let $P(t)$ denote its characteristic
polynomial $P(t)=\det(tI-M)$. Let $\{m_v\}_{v\in \calv}$ be the multiplicities of the graph $\G=\G(X,f)$.
In addition, for any non-arrowhead vertex $w\in\calw$ of $\G(X,f)$,
denote by $\delta_w$ the number of adjacent vertices (with the notation
of \ref{ex:LINKRHS},   $\delta_w=|\calv_w|$).
Then the following facts hold:
\begin{enumerate}
\item {\bf A'Campo's formula} \cite{A'CMon,AC,AC2}
\begin{equation}\label{eq:ACampo}
\frac{t^n-1}{P(t)}=\prod_{w\in\calw(\G)}(t^{m_w}-1)^{2-2g_w-\delta_w}.
\end{equation}
Hence, for the Euler characteristic of the page $F$ one also has
\begin{equation}\label{eq:ACampo2}
\chi(F)=\sum_{w\in\calw(\G)}m_w\cdot (2-2g_w-\delta_w),
\end{equation}
and the dimension of the generalized 1--eigenspace $(H_1(F,\bfc))_1$ is
\begin{equation}\label{eq:geneig1} \ix{generalized eigenspace}
(\dim H_1(F,\bfc))_1=2g(\G)+2c(\G)+|\cala (\G )|-1.
\end{equation}
Note that if $f$ is an isolated singularity, then $\chi(F)$ and $P(t)$ can be  determined from the
pair $(K_X,K_f)$, that is,  only from the binding of the open book decomposition,
and the additional facts regarding $\arg_*(f)$  are not needed.
\ix{monodromy!characteristic polynomial}\ix{A'Campo's formula|textbf}

\vspace{2mm}

This is in contrast with the Jordan block structure of the operator $M$, see part 3 of this list.\vspace{2mm}

\item By the Monodormy Theorem \cite{Clemens,Landman,BrieskornMon,Lj,LeMon,A'CMon,AC,AC2},
the size of a Jordan block of $M$ can at most be  2. Let $\#^2_\lambda $ be the number of Jordan blocks
with eignevalue $\lambda$ and size 2.  Then from the Wang exact sequence applied
to the Milnor fibration one gets (see e.g. \cite{NSM}):
\begin{equation}\label{eq:1111}
\dim\ker (M-I)=2g(\G)+c(\G)+|\cala (\G )|-1.\end{equation}
Therefore, using  (\ref{eq:geneig1}) too,
$\#^2_1=c(\G(X))$, hence $\#^2_1$  depends only on $K_X$, and it
is independent of  the germ $f$.\vspace{2mm}
\ix{monodromy!Theorem}

\item \cite[3.25]{cyclic} \
Let $G=G(X,f)$ be the universal covering of $\G(X,f)$, let $|G|$ be its topological realization.
The cyclic action on $G$ induces an action on $|G|$.  \ix{graph!covering!universal}
At homological level, this induces a finite morphism
$M_{|G|}$ on $H^1(|G|,\bfc)$. Then there is an isomorphism of pairs
\begin{equation}\label{eq:|G|}
(H^1(|G|,\bfc),M_{|G|})= ({\rm im}(M^{N}-I),M),
\end{equation}
where $N$ is an integer such that $\lambda^N=1$ for any eigenvalue $\lambda$ of $M$. Hence,
$\#^2_\lambda(f)$ equals the multiplicity of the root $\lambda$ in the characteristic polynomial
of $M_{|G|}$.

For example, in the case of Example \ref{5.10}, the monodromy of $f_1$ is finite, while the monodromy
of $f_2$ has a Jordan block of size 2, that is $\#^2_{-1}(f_2)=1$.\vspace{2mm}

\item Nevertheless, if $K_X$ is a rational homology sphere, each $\#^k_\lambda(f)$ can be determined from
$\G=\G(X,f)$ \cite{EN,Ne2,cyclic}.  Assume that $f$ defines an isolated singularity, and
consider the integers $n_v$ ($v\in \calv(\G)$) and $n_e$ ($e\in \cale(\G)$) as in \ref{ex:LINKRHS}.
For any
%\begin{remark}\label{rem:Jordan}
fixed positive integer $N$ set
 $\n_w:={\rm gcd}(n_w,N)$ and $\n_e:={\rm gcd}(n_e,N)$.   Then
\begin{equation}\label{CG}
\sum_{\lambda^N=1}\, \#^2_\lambda (f)=\sum_{e\in \cale(\G)} (\n_e-1)-
\sum_{w\in \calw(\G)} (\n_w-1).\end{equation}
\end{enumerate}
\end{remark}
\ix{monodromy!Jordan block}\ix{graph!topological realization}

\section{\ The resolution graph of $f(x,y)+z^N=0$}
\labelpar{ss:b}\setcounter{equation}{0}

\begin{remark}
Let $(X,x)$ be a normal surface singularity and $f$ the  germ of an analytic function
defined on $(X,x)$. Let us also fix a positive integer $N$.

The dual resolution graph of the cyclic covering $X_{f,N}$
depends essentially on the map $\arg_*(f)$. In
section \ref{link} we  emphasized that $\arg_*(f)$ cannot be determined from
the dual embedded resolution graph $\G(X,f)$. Hence, in general, the resolution graph
of $X_{f,N}$ cannot be determined from $\G(X,f)$ and $N$ either.  \ix{graph!covering!universal}
Usually, from $\G(X,f)$ and $N$ one can read the covering data  \ix{graph!covering!data}
of all the  edges and all the vertices with $g_w=0$, and obviously the graph
$\G=\G(X,f)$ itself, but  still missing are the covering data of vertices with $g_w>0$
and the global `twisting data' codified in $\fedog$, whenever this covering group is non--trivial.
For concrete examples,  and a detailed discussion,  see \cite{cyclic}.

In order to determine the resolution graph of $X_{f,N}$, one needs to consider the {\it
universal cyclic covering graph} of $\G(X,f)$, which codifies the missing information, see \cite{cyclic}.
On the other hand, if the link of $(X,x)$ is a rational homology sphere, then
the covering data can be recovered from the multiplicity system of $\G(X,f)$ and the integer $N$,
furthermore, any cyclic covering
of $\G(X,f)$ can be determined from $\G(X,f)$ and the covering data in a unique way, cf. \ref{ss:2.3b}.
In particular,
in such a case, $\G(X_{f,N})$ can be recovered from $\G(X,f)$  and $N$. This is definitely valid
when $(X,x)$ is a smooth germ, that is when $X_{f,N}=\{f(x,y)+z^N=0\}$ is the cyclic cover
of a plane curve singularity.

Since this special case is what will be needed in the sequel,
 in this section it is all we recall.
This algorithm   was used in several articles, see
\cite{Artal,HKK,La,NSign,Five,cyclic,O,OW}. For its generalization to the Iomdin series, see \cite{eredeti}.
It can also be viewed as the starting  point of our Main Algorithm.
\end{remark}

\bekezdes\labelpar{bek:f+z^N} {\bf The resolution graph of $X_{f,N}$ for a plane curve singularity $f$.}

Assume that $f:(\bfc^2,0)\to (\bfc,0)$ is an isolated plane curve singularity, and let us fix a
positive integer $N$. As above, $X_{f,N}$ denotes the germ of the isolated hypersurface singularity
$\{f(x,y)+z^N=0\}$. Germs of this type are also called
 {\it suspensions}. The projection $(x,y,z)\mapsto z$ induces a map $z:(X_{f,N},0)\to (\bfc,0)$.
\ix{graph!resolution!suspension|textbf}
\ix{singularities!suspensions|textbf}

In fact, the algorithm provides a possible embedded resolution graph $\G(X_{f,N},z)$
from $\G(\bfc^2,f)$ and $N$. By adding  the germ $z$ we exploit fully the power
of the multiplicity system and its `local nature', and
we also exemplify the comment \ref{1.3b}(4). \\

The graph $\G(X_{f,N},z)$ is a cyclic covering $p:\G(X_{f,N},z)\to \G(\bfc^2,f)$
(with arrowheads and decorations, cf. \ref{re:2.3.1})  \ix{graph!covering}
~of  $\G(\bfc^2,f)$.
The covering data is the following. \\  \ix{graph!covering!data}

(a) \ For any $w\in\calw(\G(\bfc^2,f))$, let $\calv_w$ be the set of all the vertices
(arrowheads and non--arrowheads) adjacent to $w$. Furthermore, set
 $$n_w:={\rm gcd}(m_v\,:\, v\in\calv_w\cup\{w\}).$$
Above $w\in{\cal W}(\G(\bfc^2,f))$ there are $\n_w:={\rm gcd}(n_w,N)$ vertices of $\G(X_{f,N},z)$,
each with multiplicity $m_w/{\rm gcd}(m_w,N)$ and genus $\tilde{g}$, where:
$$2-2\tilde{g}=\frac{(2-|\calv_w|)\cdot{\rm gcd}(m_w,N)+\sum_{v\in{\cal V}_w}\ {\rm gcd}(m_w,m_v,N)}
{\n_w}.$$\ix{graph!resolution!string}

(b) \ Consider an  edge $e=(w_1,w_2)$ of $\G(\bfc^2,f)$

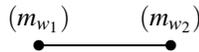
\begin{picture}(300,40)(0,0)
\put(150,10){\circle*{4}}
\put(200,10){\circle*{4}}
\put(150,10){\line(1,0){50}}
\put(150,20){\makebox(0,0){$(m_{w_1})$}}
\put(200,20){\makebox(0,0){$(m_{w_2})$}}
\end{picture}

\vspace{2mm}

\noindent Set $n_e={\rm gcd}(m_{w_1},m_{w_2})$. Then $e$
is covered by $\n_e:={\rm gcd}(n_e,N)$ copies of the following string (cf. \ref{def:STRING})
$$Str(\frac{m_{w_1}}{\n_e},\frac{m_{w_2}}{\n_e};\frac{N}{\n_e}\big|0,0;1).$$

(c) \ An arrowhead of $\G(\bfc^2,f)$

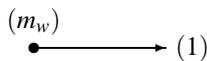
\begin{picture}(300,40)(0,0)
\put(150,10){\circle*{4}}
\put(150,10){\vector(1,0){50}}
\put(150,20){\makebox(0,0){$(m_{w})$}}
\put(210,10){\makebox(0,0){$(1)$}}
\end{picture}

\vspace{2mm}

\noindent
is covered by one string of type $Str(m_w,1;N|0,0;1)$, whose right arrowhead will
remain an arrowhead of $\G(X_{f,N},z)$ with multiplicity 1, and its left arrowhead
is identified with the unique vertex above $w$.

\vspace{2mm}

(d) \ In this way, we obtain all the vertices, edges and arrowheads
of $\G(X_{f,N},z)$, all the
multiplicities of $z$,  and some of the Euler numbers.  The missing
Euler  numbers can be determined by (\ref{eq:2.2.1}).
This ends the construction of $\G(X_{f,N},z)$. \\

If we drop the arrowheads and multiplicities of $\G(X_{f,N},z)$, we obtain
$\G(X_{f,N})$. The graphs $\G(X_{f,N},z)$ and $\G(X_{f,N})$, in general, are not minimal.
They can be simplified by blowing down operation.

\bekezdes\label{8.7} {\bf Brieskorn singularities.} \ Evidently, the above strategy
can be applied for an arbitrary
Brieskorn singularity $f(x,y,z)=x^{a_1}+y^{a_2}+z^{a_3}$ too. Nevertheless, in the next paragraph
we present a much shorter procedure valid for this case, see  \cite{OW}.
\ix{singularities!Brieskorn|textbf}\ix{graph!resolution!Brieskorn}

 For any cyclic permutation
$(i,j,k)$ of $(1,2,3)$ take:
$$d_i:=(a_i,[a_j,a_k]);\ \alpha_i:=a_i/d_i;\ q_i:=[a_j,a_k]/d_i;$$
and $0\leq \omega_i<\alpha_i$ with $1+\omega_iq_i\equiv 0\ ({\rm mod}\ \alpha_i)$.
Here $[\cdot,\cdot]={\rm lcm}$ and $(\cdot,\cdot)={\rm gcd}$.\ix{Hirzebruch--Jung continued fraction}

For any $i\in\{1,2,3\}$, we construct a string $St_i$. If $\omega_i\not=0$, take
\begin{equation}\label{eq:COFR533}
\frac{\alpha_i}{\omega_i}=k_{i1}-{1\over\displaystyle
k_{i2}-{\strut 1\over\displaystyle\ddots
-{\strut 1\over k_{is}}}}, \ \ k_{i1},\ldots, k_{is}\geq 2.\end{equation}
Then $St_i$ is the following graph (with a distinguished vertex ${\bf x}$):

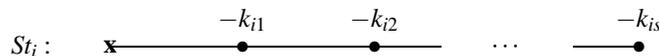
\begin{picture}(400,40)(30,10)
\put(150,35){\makebox(0,0){$-k_{i1}$}}
\put(200,35){\makebox(0,0){$-k_{i2}$}}
\put(300,35){\makebox(0,0){$-k_{is}$}}
\put(70,25){\makebox(0,0){$St_i:$}}
\put(150,25){\circle*{4}}
\put(100,25){\makebox(0,0){${\bf x}$}}
\put(200,25){\circle*{4}}
\put(300,25){\circle*{4}}
\put(225,25){\line(-1,0){125}}
\put(275,25){\line(1,0){25}}
\put(250,25){\makebox(0,0){$\cdots$}}
\end{picture}

If $\omega_i=0$, then $St_i$ contains only
the distinguished vertex ${\bf x}$, and it has no edges.

Then $\G(V_f)$ is the star--shaped graph obtained using
$(a_1,a_2)$ copies of $St_3$,
$(a_2,a_3)$ copies of $St_1$, and
$(a_3,a_1)$ copies of $St_2$, by identifying
their distinguished vertices
${\bf x}$. This vertex in $\G(V_f)$ will have genus $g$ and Euler  number $e$, where:
$$2-2g=\sum(a_i,a_j)-\frac{\prod(a_i,a_j)}{(a_1,a_2,a_3)}$$ and $$
 -e=\frac{(a_1,a_2,a_3)}{\alpha_1\alpha_2\alpha_3}+\sum(a_i,a_j)\frac{\omega_k}{\alpha_k}.$$
Above $\sum$ and $\prod$ denotes the cyclic sum, respectively
cyclic  product.

\bekezdes\label{bek:SEIFERT} {\bf Seifert invariants. Orbifold Euler number.} \cite{OW}
Those star shaped graphs whose Euler numbers on the legs are $\leq -2$
 characterize  the Seifert manifolds.
Their topological invariants are  codified as follows. Assume that $\{1,\ldots,\nu\}$ is an index set for
the legs.
Each  leg,  via a continued fraction expansion as in (\ref{eq:COFR533}),
determines a pair $(\alpha_\ell,\omega_\ell)$,
the corresponding `orbit invariant'.
Furthermore, the central vertex has a genus decoration $[g]$ and Euler number $e$. Then the
collection $\{(\alpha_\ell,\omega_\ell), 1\leq \ell\leq \nu; g,e\}$ is
the {\it Seifert invariant} of the corresponding plumbed 3--manifold.
\ix{Seifert manifold!invariants|textbf}\ix{Seifert manifold|textbf}
\ix{orbifold Euler number|textbf}\ix{Hirzebruch--Jung continued fraction}
\ix{Seifert manifold!orbifold Euler number|see{orbifold Euler number}}

Usually, one also defines the {\it orbifold Euler number} by
\begin{equation}\label{eq:eorb1} e^{orb}:=e+\sum_\ell \omega_\ell/\alpha_\ell.
\end{equation}
The intersection matrix (in the present normal form)
is negative definite if and only if $e^{orb}<0$. \ix{graph!negative definite}
The rational number
 $e^{orb}$ is also called  `virtual degree', see \cite{WA}. \ix{matrix!intersection}

One also has %\marginpar{REFERENCE}
$$-e^{orb}\cdot\prod_\ell \alpha_\ell=\det(-A)$$
and
\begin{equation}\label{eq:eorb2}
e^{orb}(-M)=-e^{orb}(M).
\end{equation}

%\setcounter{temp}{\value{section}}
%\part{The  graph $\gc$ of a pair $(f,g)$
%(preliminaries)}
%\setcounter{section}{\value{temp}}

\chapter{The  graph $\gc$ of a pair $(f,g)$. The definition}\labelpar{ss:1.3}
\section{\ The construction of the curve $\C$ and its dual graph}\labelpar{construct}

\bekezdes{\bf Introductory words.}\label{bek:INWO}
The main tool of the present book is the weighted graph $\gc$
introduced and studied in \cite{eredeti}. It has two types of
vertices, non-arrowheads and arrowheads. The non-arrowhead
vertices  have two types of decorations: the first one is an
ordered triple $(m;n,\nu)$ for some integers $m,\nu>0$ and $n\geq
0$, while  the second one is the `genus' decoration  $[g]$, where $g$ is
a non--negative integer. If $g=0$ then we might omit this decoration. Any
arrowhead has only one decoration, namely the ordered triple
$(1;0,1)$. The edges are not directed and loops are accepted.
Each edge has a decoration $\in\{1,2\}$, which in some special situations
can  be omitted, since it can be recovered from the other decorations,  cf. \ref{remarks}.

The graph $\gc$ was introduced  to study
hypersurface singularities in three variables with 1--dimensional singular locus, and it was the
main tool used  getting resolution graphs of the members of the
generalized Iomdin series.
 \ix{Iomdin series}

More precisely, in that article we started with a hypersurface germ $f$ as above,
and  chose  an additional germ
$g:(\bfc^3,0)\to (\bfc,0)$ such that the pair $(f,g)$ formed  an
ICIS. The final output was the resolution graphs of the series of hypersurface
singularities  $f+g^k$, for $k$ large, determined in terms of $\gc$ and $k$.
Motivated by the fact that in addition  $\gc$ contains all the information
needed to treat `almost all' the correction terms of the invariants of the series (see
\cite{eredeti}, or  \ref{1.3} and \ref{why}  here),
we called $\gc$  {\it `universal'}. Its power is
reinforced by the present work as well.

Nevertheless, perhaps, the name {\em bi--colored relative graph associated with
the pair $(f,g)$} tells more about the geometry encoded in the graph.
Here, the first attribute points out that the edges can be  decorated by two `colors'
(1 and 2),
a key fact which has enormous geometrical effects and the source of pathological behaviors.
By `relative'
we wish to stress, that the graph codifies the $g$--polar geometry
of $f$ concentrated near the singular locus of $V_f$. In particular,
in $\gc$ the functions $f$ and $g$ do not have a symmetric role. For
more motivation and supporting intuitive arguments regarding
$\gc$, see \ref{why} too.

Geometrically, the graph $\gc$ is
the decorated dual graph of a special curve configuration $\C$ in the
embedded resolution of the pair $(\bfc^3,V_f\cup V_g)$.
Its presentation  is the subject of the next paragraphs.

\bekezdes\labelpar{gc} {\bf The definition  of $\C$ and $\gc$ \cite{eredeti}.} \
Consider an ICIS   $(f,g):({\bfc}^3,0)\to ({\bfc}^2,0)$  and the
local divisor $(D,0):=(V_f\cup V_g,0)\subset(\bfc^3,0)$. Let
$$\Phi:\Phi^{-1}(D_{\eta}^2)\cap B_{\epsilon}\to D_{\eta}^2$$ denote  a  good representative of $(f,g)$
as in \ref{ss:ICIS}. Denote by   $\Delta_\Phi\subset D^2_\eta$ its discriminant, as before.
\ix{curve configuration $\C$|textbf}
\ix{curve configuration $\C$!its dual graph $\G_\C$|textbf}\ix{graph!$\G_\C$|textbf}

Take an embedded resolution  $$r:V^{emb}\rightarrow U$$ of the pair
$(D,0)\subset(\bfc^3,0)$. This means the following. The space $V^{emb}$ is smooth,  $r$ is
proper, $U$ is a small representative of $(\bfc^3,0)$ of type
$U=\Phi^{-1}(D_{\eta}^2)\cap B_{\epsilon}$, and the total
transform $\de:=r^{-1}(D)$ is a normal crossing divisor. Moreover,
we assume that the restriction of $r$ on $V^{emb}\setminus
r^{-1}(Sing(V_f)\cup Sing(V_g))$ is a biholomorphic isomorphism.
(Note that $Sing(V_f)\cup Sing(V_g)$ is a smaller set than the
`usual' singular locus $Sing(D)$ of $D$; nevertheless,  since the intersection $V_f\cap V_g$ is
already transversal off the origin, the above assumption can
always be satisfied for a convenient resolution.)

In particular,
$$\wP=\Phi\circ r:
\big(\,r^{-1}(\Phi^{-1}(D^2_\eta\setminus\Delta_\Phi)\cap B_\epsilon),
r^{-1}(\Phi^{-1}(D^2_\eta\setminus\Delta_\Phi)\cap  \partial B_\epsilon\,)\big )
\rightarrow D^2_\eta\setminus\Delta_\Phi$$ is a smooth locally trivial
fibration of a pair of spaces.

Note that the topology of $r$ is rather complicated, more complicated than
the topology of a germ defined on a normal surface singularity. While in that case the exceptional
locus is a curve, here the exceptional locus is a surface. The description and characterization
of the embedding and intersection properties  of the components of
$\de$ can be a rather difficult task. The point is that in our next construction we will not need
all these data, but  only a special curve configuration. %\ix{curve configuration $\C$}
This curve is identified by the
vanishing behaviour of the pullbacks of $f$ and $g$ on $\de$.

Denote by $\dec$ those irreducible components of the total
transform $\de$ along which  {\it only} $f\circ r$ vanishes, that is  $$\dec
=\overline{{ r^{-1}(V_f\setminus V_g)}}.$$
Here \ $\overline{\cdot}$ \  denotes the closure.  Similarly, consider
$$\ded =\overline{{ r^{-1}(V_g\setminus V_f)}} \ \ \mbox{and} \ \ \ \deo ={
r^{-1}(V_f\cap V_g)}.$$

\begin{definition}\labelpar{cdef} The curve configuration $\C$ is
defined by  %\ix{curve configuration $\C$}
\begin{equation*}
\C=(\dec\cap\deo)\cup(\dec\cap\ded).
\end{equation*}
\end{definition}
Thus, for each irreducible component $C$ of $\C$,  there are
exactly two irreducible components $B_1$ and $B_2$ of $\de$ for
which $C$ is a component of $B_1\cap B_2$. By the definition of
$\C$, we can assume that $B_1$ is such that \emph{only} $f\circ r$
vanishes on it, and  $B_2$ is such that either only $g\circ r$
vanishes on it, or both $f\circ r$ and $g\circ r$.

Let $m_{f,B_i}$ (respectively $m_{g,B_i}$) be the vanishing order
 of $f\circ r$ (respectively of $g\circ r$) along $B_i$
($i=1,2$). Then $m_{f,B_1}>0$, $m_{g,B_1}=0$, $m_{f,B_2}\geq 0$,
and  $m_{g,B_2}>0$.
To the component $C$ we assign the triple
$(m_{f,B_1};m_{f,B_2},m_{g,B_2})$.

% Notice that from the  ordered triple
%$(m_{f,B_1};m_{f,B_2},m_{g,B_2})$ one can recover the local
%equations for $f\circ r|_{U_p}$ and $g\circ r|_{U_p}$ in a small
%coordinate neighbourhood $U_p$ of any point $p\in C_i$, see
%(\ref{summary}).

A component $C$ of $\C$ is  either a compact (projective) curve or
it is non--compact,  isomorphic to a complex disc.   The union of
the non--compact components is the strict transform of $V_f\cap V_g$. Therefore,
$(m_{f,B_1};m_{f,B_2},m_{g,B_2})=(1;0,1)$ for them. The compact
components are exactly those which are contained in $r^{-1}(0)$.

\vspace{2mm}

The graph $\gc$  is
the dual graph of the curve configuration $\C$. \ix{curve configuration $\C$}

\vspace{2mm}

The set of vertices $\calv$ consists of  non-arrowheads $\calw$
and arrowheads $\cala$. The non-arrowhead vertices correspond to
the compact irreducible curves of $\C$ while the arrowhead
vertices correspond to the non-compact ones.

In $\gc$ one connects the vertices $v_i$ and $v_j$ by $\ell$ edges
 if the corresponding curves $C_i$ and $
C_j\subset\C $ intersect in $\ell$ points. Moreover, if a compact
component $C_i\subset\C$, corresponding  to a vertex
$v_i\in\calw$, intersects itself, then each self--intersection point
determines a loop supported by $v_i$  in the graph $\G_{\C}$. The edges
are not directed.

One decorates the graph $\gc$ as follows:

\begin{enumerate}
\item Each non-arrowhead vertex $v\in\calw$ has two
weights: the ordered triple of integers
$(m_{f,B_1};m_{f,B_2},m_{g,B_2})$
 assigned to the {\ic} $C$ corresponding to $v$,
and the genus $g$ of the normalization of $C$.\\
\item Each arrowhead vertex has a single weight: the ordered triple
$(1;0,1)$.\\
\item Each  edge has a weight $\in\{1,2\}$  determined as follows. By
construction, any edge corresponds to
 an  intersection point of three
local irreducible components of $\de$. Among them either one or
two local components belong to $\de_c$.  Correspondingly, in the
first case let the weight of the edge be $1$, while in the second case
$2$.
\end{enumerate}

\section{\ Summary of notation for
$\gc$ and local equations}\labelpar{summary} The next table and
local coordinate realizations will be helpful in the further
discussions and proofs.\\

\noindent {\bf Vertices:}

The vertex
\begin{picture}(30,20)(0,10)
\put(15,15){\circle*{4}} \put(15,0){\makebox(0,0)[b]{$[g]$}}
\put(15,18){\makebox(0,0)[b]{$(m;n,\nu)$}}
\end{picture}
codifies a compact irreducible component $C$ of $\C$ of genus $g$.

\vspace{4mm}

\noindent
There is a local neighbourhood $U_p$ of any generic point $p\in C$
with local coordinates $(u,v,w)$ such that $U_p\cap\de=B_1^l\cup
B_2^l$ , $B_1^l=\{ u=0\}$, $ B_2^l=\{ v=0\}$, and
$$\mbox{
$f\circ
r|_{U_p}=u^{m}v^{n}$, and $g\circ r|_{U_p}=v^{\nu}$ with
$m,\nu>0;\ n\ge 0$.}$$
Here $B_i^l$ are \emph{local} components of $\de$ at $p$. A  missing  $[g]$ means $g=0$.

\vspace{2mm}

The arrowhead vertex
\begin{picture}(60,10)(0,0)
\put(10,3){\vector(1,0){5}}%$\blacktriangleright$}}
\put(37,-3){\makebox(0,0)[b]{$(1;0,1)$}}
\end{picture}
codifies a non-compact component. The local description goes by the
same principle as above.

\vspace{3mm}

\noindent {\bf Edges:} An edge corresponds either to an intersection
point $p\in C_{i}\cap C_{j}$,  or to a self-intersection of $C_{i}$ if
 $i=j$. There is a local neighbourhood $U_p$ of $p$ with
 local coordinates $(u,v,w)$ such that $U_p\cap\de= B_1^l\cup B_2^l\cup B_3^l$,
 and $B_1^l=\{u=0 \},\ B_2^l=\{v=0\},\ B_3^l=\{ w=0\}$. Moreover,
 the local equations of $f$ and $g$ are as follows:

\vspace{4mm}

 \noi {\it An edge with decoration 1:}

\vspace{2mm}

 \bs\noi \begin{minipage}[c]{2.5in}

\begin{picture}(140,55)(-20,-30)
\put(20,30){\circle*{4}} \put(100,30){\circle*{4}}
\put(20,30){\line(1,0){80}}
\put(20,35){\makebox(0,0)[b]{$(m;n,\nu)$}}
\put(100,35){\makebox(0,0)[b]{$(m;l,\lambda)$}}
\put(20,15){\makebox(0,0)[b]{$[g]$}}
\put(100,15){\makebox(0,0)[b]{$[g']$}}
\put(20,0){\makebox(0,0)[b]{$v_1$}}
\put(100,0){\makebox(0,0)[b]{$v_2$}}
\put(60,35){\makebox(0,0)[b]{$1$}}
\end{picture}
\end{minipage} \begin{minipage}[b]{2in}
corresponds to local  equations:
$$ f\circ r|_{U_p}= u^{m}v^{n}w^{l}, \ g\circ
r|_{U_p}=v^{\nu}w^{\lambda},$$ where $m,\nu,\lambda>0$ and
$n,l\geq 0$.
\end{minipage}

\vspace{3mm}

\noi \begin{minipage}[b]{2.5in} One has similar equations with\\
$m=\lambda=1$, $l=0$, $\nu>0$,  $n\geq 0$ \\
 if $v_2$ is an arrowhead:\\
\
\end{minipage}  \begin{minipage}[b]{2.5in}

\begin{picture}(140,50)(-20,0)
\put(20,30){\circle*{4}}
\put(20,30){\vector(1,0){70}}
\put(20,35){\makebox(0,0)[b]{$(1;n,\nu)$}}
\put(120,30){\makebox(0,0){$(1;0,1)$}}
\put(20,15){\makebox(0,0)[b]{$[g]$}}
\put(60,35){\makebox(0,0)[b]{$1$}}
\end{picture}
\end{minipage}

\vspace{3mm}

 \noi {\it An edge with decoration 2:}

\vspace{5mm}

\sm\noi
\begin{minipage}[c]{2.5in}

\begin{picture}(140,50)(-20,-30)
\put(20,30){\circle*{4}} \put(100,30){\circle*{4}}
\put(20,30){\line(1,0){80}}
\put(20,35){\makebox(0,0)[b]{$(m;n,\nu)$}}
\put(100,35){\makebox(0,0)[b]{$(m';n,\nu)$}}
\put(20,15){\makebox(0,0)[b]{$[g]$}}
\put(100,15){\makebox(0,0)[b]{$[g']$}}
\put(20,0){\makebox(0,0)[b]{$v_1$}}
\put(100,0){\makebox(0,0)[b]{$v_2$}}
\put(60,35){\makebox(0,0)[b]{$2$}}
\end{picture}
\end{minipage}
\begin{minipage}[b]{2in}
provides local equations:
$$f\circ r|_{U_p}= u^{m}v^{m'}w^{n},\ \  g\circ
r|_{U_p}=w^{\nu}$$ with $m,m',\nu>0$ and $n\geq 0$.
\end{minipage}

\vspace{2mm}

\noi \begin{minipage}[b]{2.5in} One has similar equations with
\\ $m'=\nu=1$, $n=0$ and $m>0$\\if $v_2$ is an arrowhead:
\\ \ \\ \ \end{minipage} \begin{minipage}[b]{2.5in}

\begin{picture}(140,40)(-20,-10)
\put(20,30){\circle*{4}}
\put(20,30){\vector(1,0){70}}
\put(20,35){\makebox(0,0)[b]{$(m;0,1)$}}
\put(120,30){\makebox(0,0){$(1;0,1)$}}
\put(20,15){\makebox(0,0)[b]{$[g]$}}
\put(60,35){\makebox(0,0)[b]{$2$}}
\end{picture}
\end{minipage}

\begin{remark}\labelpar{remarks}
There is a compatibility between the weights that sometimes simplifies
the decorations. Indeed, consider the following edge:

%\vspace{10mm}

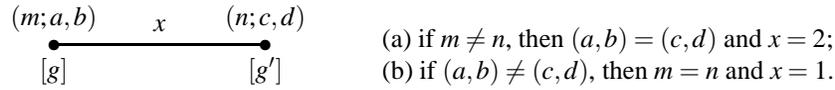
\noi
\begin{minipage}[b]{2.5in}

\begin{picture}(140,60)(-10,5)
\put(20,30){\circle*{4}} \put(100,30){\circle*{4}}
\put(20,30){\line(1,0){80}}
\put(20,35){\makebox(0,0)[b]{$(m;a,b)$}}
\put(100,35){\makebox(0,0)[b]{$(n;c,d)$}}
\put(20,15){\makebox(0,0)[b]{$[g]$}}
\put(100,15){\makebox(0,0)[b]{$[g']$}}
%\put(20,0){\makebox(0,0)[b]{$v_1$}}
%\put(100,0){\makebox(0,0)[b]{$v_2$}}
\put(60,35){\makebox(0,0)[b]{$x$}}
\end{picture}

\end{minipage}\hspace{-10mm}
\begin{minipage}[b]{2.6in}
 (a) if $m\neq n$, then $(a,
b)=(c,d)$ and $x=2$;\\
 (b) if $(a,b)\neq (c,d)$, then $m=n$ and $x=1$.
\\ \
\end{minipage}

\sm\noi In particular, in the cases (a)--(b) above, the weight of
the  edge  is determined by the weights of the vertices, hence, in these  cases
it can  be omitted.
\end{remark}

\begin{remark}\labelpar{re:CONJ}
Clearly, different resolutions $r$ provide different curve configurations  $\C$. \ix{curve configuration $\C$}
In particular, the graph $\gc$ is not unique.
We believe that
there is a `calculus' of such graphs connecting different graphs
$\gc$ coming from different embedded resolutions of $(D,0)\subset
(\bfc^3,0)$, %but this was not yet developed with all the details;
see the open problem \ref{Oa1} at the end of the book.
%we will return to it in a forthcoming work.
\end{remark}

\begin{remark}
It is rather long and difficult
to find a resolution $r$. In the literature there are very
explicit resolution algorithms, but they are rather
involved, and in general very
slow, and usually  with many irreducible exceptional components.
Nevertheless, in the next Chapters we  list many examples.

Finding the resolution of those examples, which do not have some specific form
(which would help to find a canonical sequence of modifications or a direct
resolution) we use a sequence of
ad hoc blow ups following the naive principle: `blow up the worst
singular locus', with the hope to obtain a more or less small configuration.
Some of the  computations are long, and are not given here.
(Several of them were, in fact, done with the help of {\em Mathematica}.)
Hence, we admit that for the reader the verification of some of the
examples of $\G_\C$ listed in the body of the book
can be a really difficult job. Also, since the resolution procedure is not unique,
an independent computation might lead to a different $\G_\C$.
%trying to get a resolution might lead to a different one.)

On the other hand, we will also list several families  where we can find in a conceptual
way resolutions which reflect  the geometry and the structure of the
singularities. The next chapters, starting from Chapter \ref{hom}, contain an abundance of them.

Here we list  some preliminary (specially chosen) examples   in order to help the reader follow the
first properties of $\G_\C$ discussed in Chapter \ref{s:GC}.
  %Nevertheless,
%in the next discussions of the first properties of the graph $\gc$, checking on concrete
%examples the theory would help considerably the reader. Here are some  preliminary
%examples serving merely this reason.
\end{remark}

\begin{example}\labelpar{nu2}
Assume that $f=x^2y^2+z^2(x+y)$ and take $g=x+y+z$. Then a
possible $\gc$ is:

\begin{picture}(200,100)(-50,0)
\put(55,40){\circle*{4}} \put(55,60){\circle*{4}}
\put(90,30){\circle*{4}} \put(90,70){\circle*{4}}
\put(55,20){\circle*{4}} \put(55,80){\circle*{4}}
\put(130,50){\circle*{4}} \put(170,50){\circle*{4}}
\put(210,50){\circle*{4}} \qbezier(55,40)(55,40)(90,30)
\qbezier(55,60)(55,60)(90,70) \qbezier(55,20)(55,20)(90,30)
\qbezier(55,80)(55,80)(90,70) \put(130,50){\line(-2,1){40}}
\put(130,50){\line(-2,-1){40}} \put(130,50){\line(1,0){80}}
\put(170,50){\vector(0,-1){20}}
\put(30,20){\makebox(0,0){$(2;3,1)$}}
\put(30,40){\makebox(0,0){$(2;4,1)$}}
\put(30,60){\makebox(0,0){$(2;4,1)$}}
\put(30,80){\makebox(0,0){$(2;3,1)$}}
\put(100,20){\makebox(0,0){$(2;8,2)$}}
\put(100,80){\makebox(0,0){$(2;8,2)$}}
\put(175,60){\makebox(0,0){$(1;12,4)$}}
\put(135,60){\makebox(0,0){$(1;8,2)$}}
\put(220,60){\makebox(0,0){$(1;3,1)$}}
\put(170,20){\makebox(0,0){$(1;0,1)$}}
\end{picture}

\end{example}

\begin{example}\labelpar{ex:ketA2}  Assume that $f=y^3+(x^2-z^4)^2$ and $g=z$.
Then a possible  $\gc$ is:

\vspace{5mm}

%\marginpar{KELL a legalso csomo???}

\begin{picture}(300,160)(-5,-50)
\put(150,30){\circle*{4}} \put(50,30){\circle*{4}}
\put(100,30){\circle*{4}} \put(200,30){\circle*{4}}
\put(250,30){\circle*{4}}
\put(50,50){\circle*{4}}\put(250,50){\circle*{4}}

\put(30,70){\circle*{4}} \put(70,70){\circle*{4}}
\put(230,70){\circle*{4}} \put(270,70){\circle*{4}}

\put(50,30){\line(1,0){200}}\put(50,50){\line(1,1){20}}
\put(50,50){\line(-1,1){20}} \put(250,50){\line(1,1){20}}
\put(250,50){\line(-1,1){20}}
\put(50,30){\line(0,1){20}}\put(250,30){\line(0,1){20}}

\put(30,90){\circle*{4}}
\put(70,90){\circle*{4}}
\put(270,90){\circle*{4}}
\put(230,90){\circle*{4}}
\put(150,10){\circle*{4}}\put(150,-10){\circle*{4}}

\put(30,70){\line(0,1){20}}\put(70,70){\line(0,1){20}}
\put(230,70){\line(0,1){20}}\put(270,70){\line(0,1){20}}
\put(150,30){\vector(0,-1){60}}

\put(3,90){\makebox(0,0){$(3;9,1)$}}
\put(3,70){\makebox(0,0){$(6;9,1)$}}
\put(3,50){\makebox(0,0){$(6;12,1)$}}
\put(95,90){\makebox(0,0){$(2;8,1)$}}
\put(95,70){\makebox(0,0){$(2;12,1)$}}

\put(205,90){\makebox(0,0){$(2;8,1)$}}
\put(205,70){\makebox(0,0){$(2;12,1)$}}

\put(295,90){\makebox(0,0){$(3;9,1)$}}
\put(295,70){\makebox(0,0){$(6;9,1)$}}
\put(295,50){\makebox(0,0){$(6;12,1)$}}

\put(50,20){\makebox(0,0){$(1;12,1)$}}
\put(100,20){\makebox(0,0){$(1;18,2)$}}
\put(150,40){\makebox(0,0){$(1;24,3)$}}
\put(200,20){\makebox(0,0){$(1;18,2)$}}
\put(250,20){\makebox(0,0){$(1;12,1)$}}

\put(128,10){\makebox(0,0){$(1;12,2)$}}
\put(128,-10){\makebox(0,0){$(1;12,3)$}}
\put(150,-40){\makebox(0,0){$(1;0,1)$}}
\end{picture}

\end{example}

\begin{example}\labelpar{ex:347}  Assume that
 $f=x^3y^7-z^4$ and $g=x+y+z$. A possible graph $\gc$ is:

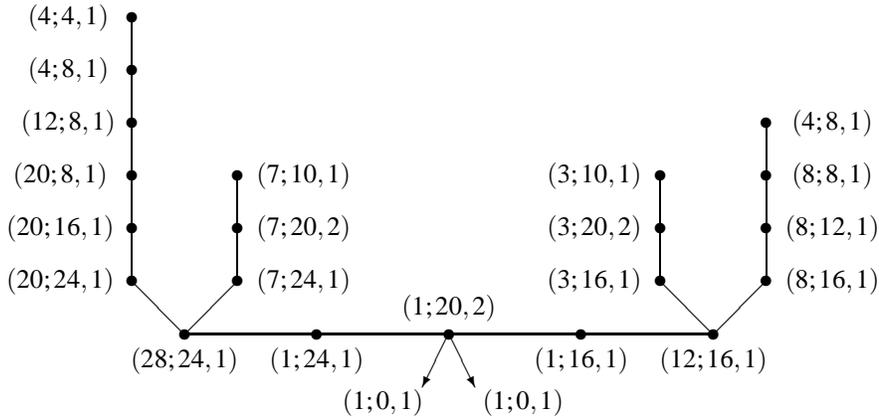
\vspace{5mm}

\begin{picture}(300,160)(0,0)
\put(150,30){\circle*{4}} \put(50,30){\circle*{4}}
\put(100,30){\circle*{4}} \put(200,30){\circle*{4}}
\put(250,30){\circle*{4}} \put(30,50){\circle*{4}}
\put(70,50){\circle*{4}} \put(230,50){\circle*{4}}
\put(270,50){\circle*{4}}
\put(150,30){\vector(-1,-2){10}}\put(150,30){\vector(1,-2){10}}
\put(50,30){\line(1,0){200}}\put(50,30){\line(1,1){20}}
\put(50,30){\line(-1,1){20}} \put(250,30){\line(1,1){20}}
\put(250,30){\line(-1,1){20}}

\put(30,70){\circle*{4}}
\put(30,90){\circle*{4}}\put(30,110){\circle*{4}}
\put(30,130){\circle*{4}}\put(30,150){\circle*{4}}

\put(70,70){\circle*{4}}\put(70,90){\circle*{4}}

\put(270,70){\circle*{4}}
\put(270,90){\circle*{4}}\put(270,110){\circle*{4}}

\put(230,70){\circle*{4}}\put(230,90){\circle*{4}}

\put(30,50){\line(0,1){100}}\put(70,50){\line(0,1){40}}
\put(230,50){\line(0,1){40}}\put(270,50){\line(0,1){60}}

\put(6,150){\makebox(0,0){$(4;4,1)$}}
\put(6,130){\makebox(0,0){$(4;8,1)$}}
\put(6,110){\makebox(0,0){$(12;8,1)$}}
\put(3,90){\makebox(0,0){$(20;8,1)$}}
\put(3,70){\makebox(0,0){$(20;16,1)$}}
\put(3,50){\makebox(0,0){$(20;24,1)$}}

\put(95,90){\makebox(0,0){$(7;10,1)$}}
\put(95,70){\makebox(0,0){$(7;20,2)$}}
\put(95,50){\makebox(0,0){$(7;24,1)$}}

\put(205,90){\makebox(0,0){$(3;10,1)$}}
\put(205,70){\makebox(0,0){$(3;20,2)$}}
\put(205,50){\makebox(0,0){$(3;16,1)$}}

\put(295,110){\makebox(0,0){$(4;8,1)$}}
\put(295,90){\makebox(0,0){$(8;8,1)$}}
\put(295,70){\makebox(0,0){$(8;12,1)$}}
\put(295,50){\makebox(0,0){$(8;16,1)$}}

\put(50,20){\makebox(0,0){$(28;24,1)$}}
\put(100,20){\makebox(0,0){$(1;24,1)$}}
\put(150,40){\makebox(0,0){$(1;20,2)$}}
\put(200,20){\makebox(0,0){$(1;16,1)$}}
\put(250,20){\makebox(0,0){$(12;16,1)$}}

\put(125,4){\makebox(0,0){$(1;0,1)$}}
\put(178,4){\makebox(0,0){$(1;0,1)$}}
\end{picture}

\end{example}

\section{\ Assumption A}\labelpar{re:w2} In the  \ix{Assumption A|textbf}
graph $\gc$ a special attention is needed for edges of weight 2
with both end--vertices having first multiplicity $m=1$ (including the case of loops too,
when the two end-vertices coincide). Such an edge
corresponds to a point $p$ which lies at the intersection of three
locally irreducible components on exactly two of which only
$f\circ r$ vanishes and that happens with multiplicity one. Thus
the intersection of these two locally irreducible components is
the strict transform of a component of $Sing(V_f)$ with
transversal type $A_1$, which has {\em not } been blown up during
the resolution procedure $r$.

\ix{curve configuration $\C$}
Performing an additional blow up along this intersection, we
obtain a new embedded resolution $r'$, whose special curve
configuration $\C'$ will have an additional rational curve. The relevant edge of the
dual
graph $\gc$ is changed to the dual graph $\G_{\C'}$ with a new part via the following
transformation:

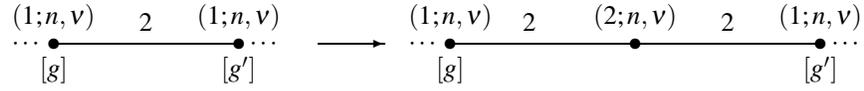
\vspace{1cm}

\begin{picture}(400,30)(10,0)
\put(20,30){\circle*{4}} \put(90,30){\circle*{4}}
\put(20,30){\line(1,0){70}}
\put(10,30){\makebox(0,0){$\cdots$}}
\put(100,30){\makebox(0,0){$\cdots$}}
\put(20,35){\makebox(0,0)[b]{$(1;n,\nu)$}}
\put(90,35){\makebox(0,0)[b]{$(1;n,\nu)$}}
\put(20,15){\makebox(0,0)[b]{$[g]$}}
\put(90,15){\makebox(0,0)[b]{$[g']$}}
\put(55,35){\makebox(0,0)[b]{$2$}}

\put(120,30){\vector(1,0){25}}

\put(170,30){\circle*{4}}
\put(240,30){\circle*{4}} \put(310,30){\circle*{4}}
\put(170,30){\line(1,0){140}}
\put(160,30){\makebox(0,0){$\cdots$}}
\put(320,30){\makebox(0,0){$\cdots$}}
\put(170,35){\makebox(0,0)[b]{$(1;n,\nu)$}}
\put(240,35){\makebox(0,0)[b]{$(2;n,\nu)$}}
\put(310,35){\makebox(0,0)[b]{$(1;n,\nu)$}}
\put(170,15){\makebox(0,0)[b]{$[g]$}}
\put(310,15){\makebox(0,0)[b]{$[g']$}}
\put(275,35){\makebox(0,0)[b]{$2$}}
\put(200,35){\makebox(0,0)[b]{$2$}}
\end{picture}

In order to avoid some pathological cases in the
discussion, and to have a uniform treatment of the properties of
$\gc$ (e.g. of the `cutting edges' and subgraphs  $\gce$ and
$\gck$, see the next chapter), we will assume that $\gc$ has no such
edges, that is,  we have already performed the extra blow-ups, if it was
necessary. The same discussion/assumption is valid for loops and for the
situation when one of the end--vertices above is replaced by an arrowhead.
\ix{cutting edge}

This assumption is not crucial at all, the interested reader might
eliminate it at the price of having to slightly  modify  the
statements of the  forthcoming sections \ref{gamma1} and
\ref{ss:GCK}. In fact, in the algorithm which provides
$\partial F$, Assumption A is irrelevant.\ix{Milnor!fiber!boundary}

\chapter{The  graph $\gc$. Properties}\label{s:GC}

\section{\ Why one should  work with $\C$?}\labelpar{why}\setcounter{equation}{0}

From the definition \ref{cdef} of the curve $\C$ it is not so transparent why exactly this
configuration should play the crucial role in several results
regarding  non-isolated hypersurface singularities.
In this section we wish to stress a universal property of $\C$,
which  motivates the definition, and will imply some immediate properties as well.
We keep  the notations of  Chapter \ref{ss:1.3}. In particular,
 $\Phi:\Phi^{-1}(D^2_\eta)\cap B_\epsilon\to D^2_\eta$
 is a good representative, $(c,d)$ are the coordinates of the target, and
 $\Phi(\Sigma_f)=\Delta_1=\{c=0\}$.
\ix{graph!$\G_\C$}\ix{wedge neighbourhood|textbf}

We start with the definition of a special set `near' the discriminant component $\Delta_1$.
For any integer $M>0$, define the {\it wedge neighbourhood} of $\Delta_1$  by
$$\Wedge=\{ (c,d)\in D^2_\eta\  |\  0<|c|<|d|^M \}.$$
It is easy to verify that
\begin{equation}\labelpar{lem:wedge1}  \mbox{
$\Wedge \subset\ D^2_\eta\setminus\Delta_\Phi$~ provided that   ~$M\gg 0$.}
\end{equation}
Hence, $\Wedge$ is a small  tubular neighbourhood of $\Delta_1\setminus \{0\}$, not intersecting the other
components of the discriminant. In particular, the restriction of $\Phi$ over $\Wedge$  is a smooth
locally trivial fiber bundle, equivalent with the restriction of the fibration over the torus
$T_\delta$,  containing all the information about the commuting monodromies
$m_{\Phi,hor}$ and $m_{\Phi,ver}$ near $\Delta_1$, cf. \ref{ss:ICIS}.

The geometry of the fibration over
$\Wedge$, or of these two monodromies,  were  frequently used in the literature
in connection with the  `correction terms' of several singularity invariants. More precisely, if $i$ denotes
a numerical invariant, then, usually  $i(f+g^k)$,
associated with the series $f+g^k$ approximating $f$  with $k$ large, are not easy to determine. Nevertheless,
the correction term $i(f+g^k)-i(f)$, in many cases, depends only on the behaviour of $\Phi$ above $\Wedge$.
This is the source of several formulas: in the classical case of   Iomdin  $i$ is  the
Euler characteristic of the Milnor fiber
\cite{Io}, in  Siersma's article \cite{Si2}  $i$ is the zeta function,
in M. Saito's paper  \cite{Saito}   $i$ is the spectrum, while in \cite{Ne126,Neta}
 $i$ is the signature of the Milnor fiber.

\vspace{2mm}

Now, the restriction of $\Phi$ above $\Wedge$ provides an alternative definition/charac\-terization of
the curve configuration $\C$, associated with the resolution $r$.
\ix{curve configuration $\C$}

\begin{lemma}\labelpar{lem:wedge2} {\bf First characterization of $\C$ }\cite[(5.5)]{eredeti}.

 Fix a resolution $r$ and  set $\wP=\Phi\circ r$.
Then, for  ~$M\gg 0$, one has
 \begin{equation}\label{zart}
\overline{\wP^{-1}(\Wedge)}\cap\wP^{-1}(0)=\C.
\end{equation}
\end{lemma}
\begin{proof} The proof follows by case-by-case local verification
in the neighbourhood of  different type of points of the resolution $r$,
and from the properness of $\wP$.

For example, let us take a point  $p\in\deo\setminus \C$. We have to show that
$p\not\in\overline{\wP^{-1}(\Wedge)}$.  Assume the contrary, and take
a local coordinate neighbourhood $U_p$ of $p$ with local coordinates  $(u,v,w)$ such that
$\wP|_{U_p}=(u^m,u^{\mu})$ for some integers  ${\mu},m>0$ (up to an invertible element).
Then there exists a sequence $\{p_j=(u_j,v_j,w_j)\}_{j=1}^\infty$ in $\wP^{-1}(\Wedge)\cap U_p$
such that $\lim_{j\to\infty}p_j=p$,  hence $u_j\to 0$.
Since $p_j\in{\wP^{-1}(\Wedge)}$ it follows that  $|u_j|^m<|u_j|^{M{\mu}}$.
 This leads to a contradiction for $M$  sufficiently large.

The other local verifications are similar and are left to the reader.
\end{proof}

Therefore, one expects, that from the dual graph $\gc$ of the special curve
configuration $\C$, endowed with all the necessary multiplicity
data (codifying the local behavior of $f$ and $g$ near $\C$),
one is able to extract topological information about
any  set $\cals\subset\Phi^{-1}(D^2_\eta)\cap B_\epsilon  $ with
$\Phi(\cals)\subset\Wedge$.  \ix{curve configuration $\C$}

Note that for $\delta>0$ sufficiently small, $\partial D_\delta\in \Wedge$,
 hence  $\Phi^{-1}(\partial D_\delta)$, the non-trivial part of the boundary of the Milnor fiber
 (by the discussion of \ref{rem:phif}) is such a space.
Hence, $\Phi$ over $\Wedge$, or its lifting via $r$ codified in $\gc$,
should contain crucial information regarding
$\partial F$ and its Milnor monodromy.\ix{Milnor!fiber!boundary}

This fact  is exploited in the present work.

\vspace{2mm}

Now we continue with some immediate consequences. From (\ref{zart}) one obtains the following.
\begin{corollary}\labelpar{connectedhez}

a) \  For any open (tubular) neighbourhood $T(\C )\subset V^{emb}$,
there exist a sufficiently large $M$  and a sufficiently small
$\eta$ such that for any $(c,d)\in\Wedge$ the `lifted Milnor
fiber'  ~$\wP^{-1}(c,d)$ is in $T(\C)$.

b) \ For any $p\in\C$ and local neighbourhood $U_p$ of $p$, there
exist a sufficiently large $M$  and a sufficiently small $\eta$
such that for any $(c,d)\in\Wedge$ one has
$U_p\cap\wP^{-1}(c,d)\neq\emptyset$.
\end{corollary}

Therefore, the curve $\C$ can be regarded as the `limit' of the lifted Milnor fiber.
In particular, from \ref{connectedhez} and from the connectedness of the Milnor fiber, see
\ref{prop:ICIScon}, we obtain that

\begin{corollary}\labelpar{cor:CONMF}
$\C$ is connected.
\end{corollary}
\begin{remark}\labelpar{felbont} We point out a natural decomposition of
$\widetilde{F}_\Phi$, as an immediate consequence of
the fact that the curve $\C$ is  a `limit' of this fiber. Here
$\widetilde{F}_\Phi=\wP^{-1}(c,d)$, with  $(c,d)\in\Wedge$ as in \ref{connectedhez}.
The construction resonate with the construction of the universal cyclic covering graph presented in
\ref{5.1}. \ix{graph!covering!universal}

\vspace{2mm}

For each edge $e$ of $\gc$, let $P_e$ be the
corresponding intersection point of two components of
$\C$. Let $U_{P_e}$ be a small neighbourhood of $P_e$. Then
 $\widetilde{F}_\Phi$ intersects $U_{P_e}$ in a
tubular neighbourhood of some embedded circles of
$\widetilde{F}_\Phi$. Consider the collection $B$ of all these
circles. Then $\widetilde{F}_\Phi$ is a union
\begin{equation}
\bigcup_{v\in\cV(\gc)}\widetilde{F}_v, \end{equation}where each
$\widetilde{F}_v$ is the closure of those components of $\widetilde{F}_\Phi\setminus B$
which are sitting  in a tubular neighbourhood of the
corresponding component $C_v$ of $\C$. Moreover, both horizontal
and vertical monodromy actions over the torus $T_\delta$ (cf. \ref{ss:ICIS})
can  be chosen in
such a way that they preserve the cutting circles $B$ and each subset $\widetilde{F}_v$.
\ix{cutting circles|textbf}
Furthermore, their
restrictions on each $\widetilde{F}_v$ are isotopic to a pair of commuting
finite actions on $\widetilde{F}_v$.
(This last fact follows from the
fact that any $S^1$--bundle over a non-closed Riemann surface is
trivial, and from the particular form of the local equations of $f$ and $g$
near $C_v$, cf. \ref{summary}.)
\end{remark}

The above limit procedure can be pushed further. Recall that $\Delta_1=\{c=0\}\cap D^2_\eta$, hence
$\Delta_1$ is in the closure of $\Wedge$. Hence, the limit procedure  $(c,d)\mapsto (0,0)$
with $(c,d)\in\Wedge$
can be done in two steps: first $(c,d)$ tends to some point of $\Delta_1\setminus \{0\}$, then along
this discriminant component we approach the origin. The analogues of \ref{lem:wedge2} and \ref{connectedhez}
are  the following:

\begin{lemma}\labelpar{lem:wedge3} {\bf Second characterization of $\C$ }\cite[(5.8)]{eredeti}.

With the same notations as in \ref{lem:wedge2}, one has
 \begin{equation}\label{zart2}
\overline{\wP^{-1}(\Delta_1\setminus \{0\})}\cap\wP^{-1}(0)=\C.
\end{equation}
In particular,

a)\  For any  open tubular neighbourhood $T(\C )\subset V^{emb}$,
if $|d|$ is sufficiently small then    \ $\wP^{-1}((0,d))\subset T(\C)$.

b)\ For  any $p\in\C$ and neighbourhood $U_p$ of $p$, if $|d|$ is
sufficiently small then $U_p\cap\wP^{-1}((0,d))\neq\emptyset$.
\end{lemma}

This lemma will also have connectivity consequences, see  \ref{m2tetel}.

\section{\ A partition of $\G_\C$ and cutting edges}\labelpar{ss:parcut}
\setcounter{equation}{0}

We start to review from \cite{eredeti} those properties of $\gc$ that are needed in
the sequel.

\begin{proposition} \ $\gc$ is
connected.\end{proposition}
\begin{proof} Use \ref{cor:CONMF} and the construction of $\gc$. \end{proof}

\begin{definition}\labelpar{stricttra} \   The vertices of the graph $\gc$ can be divided
into two disjoint sets
${\calv}(\gc)={\calv}^1(\gc)\cup{\calv}^2(\gc)$, where
${\calv}^1(\gc)$ (respectively ${\calv}^2(\gc)$)
 consists  of  the vertices decorated by $(m;n,\nu)$ with $m=1$
(respectively $m\geq 2$).

We will use similar notations for
$\calw(\gc)$ and $\cala(\gc)$.

The set of edges $(v_1,v_2)$ with ends $v_1\in \calv^1(\gc)$ and
$v_2\in\calw^2(\gc)$ will be called \emph{cutting edges}. Their
edge--decoration is always 2. We denote their index set by
$\cale_{cut}$.
\end{definition}
\ix{cutting edge|textbf}

Note that $\cala(\gc)=\cala^1(\gc)$, hence $\calv^2(\gc)=\calw^2(\gc)$.

According to the decomposition $\calv=\calv^1\cup\calv^2$, we
partition $\gc$ into two graphs $\gce$ and $\gck$.

The description of the subgraphs $\gce$ and $\gck$ is the subject of the next sections.

\section{\ The graph $\gce$}\labelpar{gamma1}\setcounter{equation}{0}

\bekezdes {\bf The construction of  $\gce$.} \
The graph $\gce$ is constructed in two steps.

First, consider the maximal subgraph  of $\gc$,  which is spanned by the vertices $v\in
{\calv}^1(\gc)$ and has no edges of weight $2$. Next,
corresponding to each cutting edge --- whose end-vertices $v_1$ and \ix{cutting edge}
 $v_2$ carry weights $(1;n,\nu)$ and $(m;n,\nu)$,
$m>1$ respectively --- add an arrowhead decorated with the weight $(m;n,\nu)$ connected  to
$v_1$ by an edge. We will keep the decoration 2 of these `inherited cutting edges', although
their type can be recognized by the principle \ref{remarks}.
\ix{graph!$\G_\C^1$|textbf}\ix{graph!of normalization of $V_f$}

In particular, $\gce$ has two types of arrowheads: first, all the
arrowheads of $\gc$ remain  arrowheads of $\gce$, all of them with weight
$(1;0,1)$; then,  each cutting edge provides an  arrowhead with weights of type $(m;n,\nu)$,
with first entry  $m>1$.

\vspace{2mm}

We wish to provide more details regarding  a special situation.

Assume that an edge of $\gc$,  decorated by 2,  supports an arrowhead.
Then, by  Assumption A, cf. \ref{re:w2}, the other vertex of the edge \ix{Assumption A}
should automatically have  weight $(m;0,1)$ with $m>1$.
In such a case,   this edge
becomes a double arrow of $\gce$: an edge supporting two
arrowheads, one with weight $(1;0,1)$, the other with $(m;0,1)$.
This double arrow forms a connected component of $\gce$.

\begin{bekezdes}\label{bek:G1C} {\bf The simplified graph $G^1_{\C}$.} \
Deleting some of the information of $\gce$, we obtain another
graph $G^1_{\C}$, which looks like the weighted embedded
resolution graph of a germ of an analytic function defined on a
normal surface singularity; that is,
$G^1_{\C}$ will be a plumbing graph as in \ref{ss:NSS}.\ix{plumbing!graph}

The construction runs as follows.

First, keep the genus-decorations of all
non-arrowheads.  Next,  for any non-arrowhead vertex, and  for
 any arrowhead with decoration $(1;0,1)$,    replace the
weight $(1;n,\nu)$ by $(\nu)$.  The weight  $(m;n,\nu)$ of an arrowhead vertex
with $m>1$ is replaced by weight $(0)$. Furthermore, delete all old
edge-decorations, and insert everywhere the new edge-decoration +
(hence, they can be even omitted  by our convention from \ref{ss:NSS}). Finally, we
determine the Euler numbers via (\ref{eq:2.2.1}) using
edge-decorations  $\epsilon_e=+$.
\end{bekezdes}
%By convention, a dual resolution graph, which consists of only one
%non--arrowhead vertex with weight (1)
% (i.e. it has no other vertices, and no edges),  codifies the dual embedded
% resolution graph  of a smooth germ defined on a smooth surface, which has
%not been blown up during the resolution. }

\begin{proposition}\labelpar{m1}\cite[(5.27)]{eredeti} \ $G^1_{\C}$ is a possible
embedded resolution graph of
$$V_f^{norm}\stackrel{g\circ n}{\longrightarrow} (\bfc,0),$$
where $n:(V_f^{norm},n^{-1}(0)) \to (V_f,0)$ is the normalization.

In particular, the number of connected components of $\gce$ coincides
with the number of irreducible components of the germ
$(V_f,0)$.  A connected component of $\gce$ which
consists of a double arrow, corresponds to a smooth component
of $V_f^{norm}$ on which $g\circ n$ is smooth, and which has
not been modified during the resolution.

The arrowheads with multiplicity $(0)$  represent the strict
transforms of the singular locus $\Sigma_f$.
\end{proposition}

If we do not wish to preserve the information about the position of the
strict transforms of the singular locus $\Sigma_f$, we delete
 the arrowheads with weight $(0)$  from $G^1_{\C}$. What remains is exactly
the disjoint union of the  embedded resolution graphs of the connected
components of  $(V_f^{norm},g\circ n)$.

In a similar way, if we do not wish to keep any information about the
germ $g$, we delete all  multiplicity decorations and arrowheads with
positive multiplicities. What remains is the collection of
resolution graphs of the components of $V_f^{norm}$, where the $(0)$--arrowheads
mark the strict transforms of $\Sigma_f$.

\begin{example}\labelpar{ex:347b'}  Assume that
 $f=x^3y^7-z^4$ and $g=x+y+z$ as in  \ref{ex:347}.

 There are two cutting edges: the extreme edges of the horizontal string. \ix{cutting edge}

 The graph $\gce$ is the following:

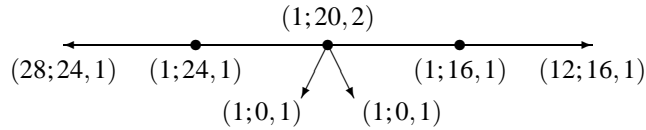
\vspace{5mm}

\begin{picture}(300,60)(0,-15)
\put(150,30){\circle*{4}}
\put(100,30){\circle*{4}} \put(200,30){\circle*{4}}
\put(150,30){\vector(-1,-2){10}}\put(150,30){\vector(1,-2){10}}
\put(100,30){\vector(1,0){150}}\put(100,30){\vector(-1,0){50}}
\put(50,20){\makebox(0,0){$(28;24,1)$}}
\put(100,20){\makebox(0,0){$(1;24,1)$}}
\put(150,40){\makebox(0,0){$(1;20,2)$}}
\put(200,20){\makebox(0,0){$(1;16,1)$}}
\put(250,20){\makebox(0,0){$(12;16,1)$}}
\put(125,4){\makebox(0,0){$(1;0,1)$}}
\put(178,4){\makebox(0,0){$(1;0,1)$}}
\end{picture}

$G^1_\C$ is the graph:

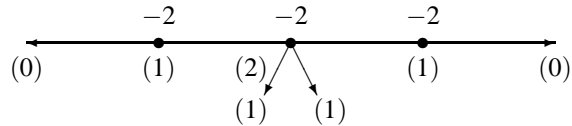
\vspace{5mm}

\begin{picture}(300,60)(0,-15)
\put(150,30){\circle*{4}}
\put(100,30){\circle*{4}} \put(200,30){\circle*{4}}
\put(150,30){\vector(-1,-2){10}}\put(150,30){\vector(1,-2){10}}
\put(100,30){\vector(1,0){150}}\put(100,30){\vector(-1,0){50}}
\put(50,20){\makebox(0,0){$(0)$}}
\put(100,20){\makebox(0,0){$(1)$}}
\put(135,20){\makebox(0,0){$(2)$}}
\put(200,20){\makebox(0,0){$(1)$}}
\put(250,20){\makebox(0,0){$(0)$}}
\put(135,4){\makebox(0,0){$(1)$}}
\put(165,4){\makebox(0,0){$(1)$}}
\put(100,40){\makebox(0,0){$-2$}}
\put(150,40){\makebox(0,0){$-2$}}
\put(200,40){\makebox(0,0){$-2$}}

\end{picture}

The two $(0)$--arrowheads correspond to the two components of the strict transform
of the singular locus of $V_f$. The normalization of $V_f$ is an $A_3$ singularity.

\end{example}

\section{\ The graph $\gck$}\labelpar{ss:GCK}\setcounter{equation}{0}

\bekezdes {\bf The construction of  $\gck$.} \

The `complementary' subgraph $\gck$ is  constructed in two steps as well.

First, consider the maximal subgraph of $\gc$ spanned by the vertices $v\in
{\calv}^2(\gc)$. Then, corresponding to each cutting edge (cf.
 \ref{stricttra}), whose end-vertices $v_1$ and
$v_2$ carry weights $(1;n,\nu)$ and   $(m;n,\nu)$ ($m>1$) respectively,
regardless of $v_1$ being an arrowhead or not,
 glue an  edge decorated by  2 to $v_2$, and make its other end--vertex
an arrowhead  weighted   $(1;n,\nu)$ . \ix{cutting edge}
 \ix{graph!$\G_\C^2$|textbf}

\vspace{2mm}

This ends the definition of $\gck$.
In order to understand its connected components,
we need the following notation.

\begin{definition}\label{def:csj}
For any $j\in\{1, \ldots, s\}$,
we denote by $\de_{c,j}$ the collection of those
components $B\subset \de_c$ which are  projected
via $r$ {\em onto} $\Sigma_j$.  Furthermore, define
$\csj\subset\C$ as the union of those irreducible components $C$
of ~$\C$ weighted  $(m;n,\nu)$ with $m>1$,  for which $C\subset
B$ for some component $B\subset \de_{c,j}$.
\end{definition}
Now we are ready to start the list of properties of $\gck$.
\begin{proposition}\labelpar{m2tetel} \cite[(5.32)]{eredeti} \
There is a one-to-one
correspondence between the connected components of $\gck$ and the
irreducible components of $(Sing(V_f),0)=(\Sigma,0)=\cup_j\Sigma_j$.

More precisely,  $\csj$ is connected and its irreducible components correspond
to the non--arrowhead vertices of one of the connected components
$\Gamma^2_{\C,j}$ of \,$\gck$.
\end{proposition}

\begin{proof} Fix $j$. The connectedness of  $\csj$ follows from the next claim,
similar to the limit properties \ref{connectedhez} and \ref{lem:wedge2}.

\vspace{2mm}

a) \  For any open tubular neighbourhood $T(\csj )\subset V^{emb}$ of $\csj$,
there exists a sufficiently small $\gamma>0$ such that for any
point $q\in \Sigma_j-\{0\}\subset \bfc^3$ with $|q|<\gamma$,
$r^{-1}(q) \subset T(\csj)$.\vspace{2mm}

b) \ For any $p\in\csj$ and local neighbourhood $U_p$ of $p$,
there exists a sufficiently small $\gamma>0$ such that for any
point $q\in \Sigma_j-\{0\}$ with $|q|<\gamma$, \
$U_p\cap r^{-1}(q)\neq\emptyset$.

\vspace{2mm}

These two statements can be verified by similar local computations as  in \ref{connectedhez}.
In fact, they  can be deduced from \ref{lem:wedge2} too.

\vspace{2mm}

This means that $\csj$ is the `limit' of $r^{-1}(q)\cap \de_{c,j}$, where $q\in \Sigma_j\setminus \{0\}$
tends to 0. But the modification $r$, above  any transversal slice $Sl_q$ at $q$ of $\Sigma_j$ (cf. \ref{ss:2.0b}),
realizes an embedded resolution of $(Sl_q, Sl_q\cap V_f,q)$, hence $r^{-1}(q)\cap \de_{c,j}$
constitutes a collection of exceptional curves of an embedded resolution of this plane curve singularity.
Hence, by Zariski's Main Theorem, cf. \ref{1.3b}, it is connected. This implies that its limit
$\csj$ is connected as well.
\end{proof}

\begin{example}\label{ex:347bal}
Consider the example given in \ref{ex:347}. In this case $\gck$ has two components. One of them
(situated at the left of the diagram) is $\G^2_{{\cal C},1}$:

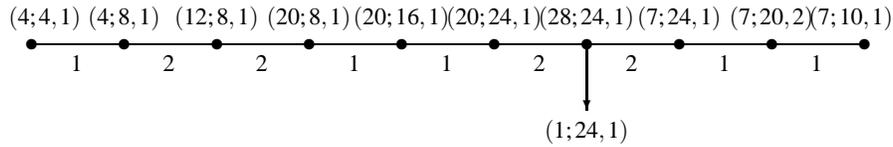
\begin{picture}(300,100)(4,-25)
\put(0,30){\circle*{4}}
\put(35,30){\circle*{4}} \put(70,30){\circle*{4}}
\put(105,30){\circle*{4}}
\put(140,30){\circle*{4}} \put(175,30){\circle*{4}}
\put(210,30){\circle*{4}}
\put(245,30){\circle*{4}} \put(280,30){\circle*{4}}\put(315,30){\circle*{4}}

\put(210,30){\vector(0,-1){25}}
\put(0,30){\line(1,0){315}}

\put(5,40){\makebox(0,0){\small{$(4;4,1)$}}}
\put(35,40){\makebox(0,0){\small{$(4;8,1)$}}}
\put(70,40){\makebox(0,0){\small{$(12;8,1)$}}}
\put(105,40){\makebox(0,0){\small{$(20;8,1)$}}}
\put(140,40){\makebox(0,0){\small{$(20;16,1)$}}}
\put(175,40){\makebox(0,0){\small{$(20;24,1)$}}}
\put(210,40){\makebox(0,0){\small{$(28;24,1)$}}}
\put(245,40){\makebox(0,0){\small{$(7;24,1)$}}}
\put(280,40){\makebox(0,0){\small{$(7;20,2)$}}}
\put(310,40){\makebox(0,0){\small{$(7;10,1)$}}}
\put(210,-3){\makebox(0,0){\small{$(1;24,1)$}}}

\put(17,23){\makebox(0,0){\small{$1$}}}
\put(52,23){\makebox(0,0){\small{$2$}}}
\put(87,23){\makebox(0,0){\small{$2$}}}
\put(122,23){\makebox(0,0){\small{$1$}}}
\put(157,23){\makebox(0,0){\small{$1$}}}
\put(192,23){\makebox(0,0){\small{$2$}}}
\put(227,23){\makebox(0,0){\small{$2$}}}
\put(262,23){\makebox(0,0){\small{$1$}}}
\put(297,23){\makebox(0,0){\small{$1$}}}
\end{picture}

\end{example}

In fact, besides other data, $\Gamma^2_{\C,j}$ contains all the
information about the equisingularity type of the transversal
singularity $T\Sigma_j$ of $\Sigma_j$. In order to make this statement more precise,
we need some preparation.

\bekezdes {\bf A partition of  $\gj$.} \ We fix again the  index $j$.

We introduce  on $\calw(\gj)$  the following equivalence relation.
First, we say that $w_1\sim w_2$ if $w_1$ and $w_2$ are connected
by an edge of weight  1, then we extend $\sim$ to an equivalence
relation.  If $K=\{w_{i_1},\ldots, w_{i_t}\}$ is an equivalence
class, with decorations $(m;n_{i_l},\nu_{i_l})$,  then set
$\nu(K):=\mbox{gcd}(\nu_{i_1},\ldots,\nu_{i_t})$ and $m(K):=m$.

Each class $K$ defines a connected subgraph $G(K)$ of $\gj$ with
vertices from $K$ and all the  1-edges connecting them. For the moment we keep
all the decorations of the corresponding vertices of $G(K)$.

The equivalence classes $\{K_\ell\}_\ell$ determine a partition
of $\calw(\gj)$. The subgraphs $G(K_\ell)$ are connected by 2-edges.

\bekezdes\label{bek:G(K)} {\bf Properties of the subgraphs $G(K)$. } \
 For a fixed equivalence class
 $K=\{w_{i_1},w_{i_2},\ldots, w_{i_t}\}$, consider
the corresponding irreducible curves  $C_{i_1},C_{i_2},$ $\ldots,
 C_{i_t}$ of $\C$. By construction, there exists an irreducible component $B(K)\in
\de_{c,j}$ which contains all of them. Moreover,
the union $\C(K):=C_{i_1}\cup\cdots\cup C_{i_t}$ is a connected curve.
Let $T(K)$ denote  a small tubular neighbourhood of $\C(K)$ in $B(K)$.

Note that the  integer $m(K)$  is exactly  the vanishing order of
$f\circ r$ along $B(K)$. On the other hand, the local
equations show that the restriction $g\circ r|_{T(K)}:T(K)\to \bfc$
provides the principal divisor $(g\circ r)^{-1}(0)=
\sum _k\, \nu_{i_k}\, C_{i_k}$ in $T(K)$.
Here,  $T(K)$ can be changed to the inverse image of a small disc $D$ under
$g\circ r|_{B(K)}$.
Therefore, the divisor  $\sum _k\, \nu_{i_k}\, C_{i_k}$ in $T(K)$
 can be interpreted as a central fiber
 of the proper analytic map $g\circ r|_{T(K)}:T(K)\to D$.

\begin{lemma}\label{genfib} The generic fiber of
$(g\circ r)|_{T(K)}$ is a disjoint union of rational curves.
\end{lemma}

\begin{proof}  Fix $d\not=0$ with $|d|$ sufficiently small. Then
$g^{-1}(d)$ intersects $\Sigma_j$ in $d_j$ points, say  $\{q_i\}_i$, cf. \ref{bek:dj}.
Then $((g\circ r)|T(K))^{-1}(d)=\cup_i\,( r^{-1}(q_i)\cap T(K))$.
But $r^{-1}(q_i)$ is the singular  fiber of an embedded resolution of the
transversal singularity associated with $\Sigma_j$, and $ r^{-1}(q_i)\cap T(K)$
is an irreducible curve of this exceptional locus, hence each
 $r^{-1}(q_i)\cap T(K)$ is a  rational curve.
\end{proof}
Next, we recall a
general property of a morphism whose generic fiber is rational, see  \cite{GH}, page 554:

\vspace{2mm}

\noi {\bf Fact.} \ \ {\em  If $S$ is a minimal smooth surface, $D$ a disc in $\bfc$,
and $\pi:S\to D$  any proper holomorphic map  whose generic fiber
$\pi^{-1}(d)$ (\,$0\not=d\in D$)  is irreducible and rational, then
$\pi $ is a trivial ${\bf P}^1$--bundle over $D$.}

\vspace{2mm}

If  $S$ is not minimal, and the generic fiber of
$\pi:S\to D$  is a disjoint union of (say, $N$) rational curves,  then
by the Stein Factorization theorem (see e.g. \cite{Hartshorne}, page 280),
and if necessary after shrinking $D$, there exists a
map  $b:D'\to D$ given by  $z\mapsto z^N$, and
$\pi':S\to D'$ such that $\pi=b\circ \pi'$, and the generic fiber of
$\pi'$ is irreducible and rational.
Since the central
fibers of $\pi $ and $\pi'$ are the same, it follows  from the above fact that
the central fiber of $\pi$ can be blown down
successively until an irreducible  rational curve is obtained.
Being a principal divisor, its self-intersection is zero.

\vspace{2mm}

This  discussion has the following consequences:

\begin{proposition}\label{gkl} {\bf Properties of the graph $G(K)$.} \

a) \ The graph $G(K)$ is a tree with all genus decorations $g_{w_k}=0$. In particular,
all the irreducible components of \, $\C_{\Sigma_j}$ are rational curves.

b) \ From the integers $\{\nu_{i_k}\}_{k=1}^t$ one can deduce the
self-intersections $C_{i_k}^2$ of the curves $C_{i_k}$ in $B(K)$ as follows.
First notice that the intersection matrix $(C_{i_k}\cdot C_{i_{k'}})_{k,k'}$,
where the intersections are considered in $B(K)$,  is a negative
semi-definite matrix with rank $t-1$, and the central divisor\ix{matrix!intersection}
$\sum_k\nu_{i_k}C_{i_k}$ is one
element of its kernel (cf. \cite{BPV}, page 90). The intersections
$C_{i_k}\cdot C_{i_{k'}}$ for $i_k\not=i_{k'}$ can be read from the graph
$G(K)$ considered as a dual graph. Then the  self-intersections can be
determined from the relations $(\sum_k\nu_{i_k}C_{i_k})\cdot C_{i_{k'}}=0$.
In particular, if $t=1$, then  $C_{i_1}^2=0$. If $t\geq 2$, then the graph
is not minimal; if we blow down successively all the $(-1)$--curves we obtain
a rational curve with self-intersection zero.

c) \ The number of irreducible (equivalently, connected) components of the
generic fiber of $(g\circ r)|_{T(K)}$ is $\nu(K)$.

The fact that each irreducible component of the generic fiber is rational
can be translated into the relation:
$$2\cdot \nu(K)=\sum_{k=1}^t\,\nu_{i_k}(2-\delta^K_{w_{i_k}}),$$
where $\delta^K_{w_{i_k}}$ is the number of  vertices adjacent to $w_{i_k}$
in $G(K)$.

\end{proposition}

\begin{example}\label{ex:347bal2} Consider the example \ref{ex:347} and its  graph
 $\G^2_{{\cal C},1}$ from  \ref{ex:347bal}. The graphs $G(K_l)$ are:

\begin{picture}(300,70)(4,5)
\put(0,30){\circle*{4}}
\put(35,30){\circle*{4}} \put(70,30){\circle*{4}}
\put(105,30){\circle*{4}}
\put(140,30){\circle*{4}} \put(175,30){\circle*{4}}
\put(210,30){\circle*{4}}
\put(245,30){\circle*{4}} \put(280,30){\circle*{4}}\put(315,30){\circle*{4}}

\put(0,30){\line(1,0){35}}
\put(105,30){\line(1,0){70}}
\put(245,30){\line(1,0){70}}

\put(5,40){\makebox(0,0){\small{$(4;4,1)$}}}
\put(35,40){\makebox(0,0){\small{$(4;8,1)$}}}
\put(70,40){\makebox(0,0){\small{$(12;8,1)$}}}
\put(105,40){\makebox(0,0){\small{$(20;8,1)$}}}
\put(140,40){\makebox(0,0){\small{$(20;16,1)$}}}
\put(175,40){\makebox(0,0){\small{$(20;24,1)$}}}
\put(210,40){\makebox(0,0){\small{$(28;24,1)$}}}
\put(245,40){\makebox(0,0){\small{$(7;24,1)$}}}
\put(280,40){\makebox(0,0){\small{$(7;20,2)$}}}
\put(310,40){\makebox(0,0){\small{$(7;10,1)$}}}
\end{picture}

\noindent  The corresponding central divisors and self-intersection numbers are the following:

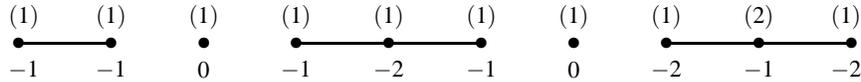
\begin{picture}(300,60)(4,5)
\put(0,30){\circle*{4}}
\put(35,30){\circle*{4}} \put(70,30){\circle*{4}}
\put(105,30){\circle*{4}}
\put(140,30){\circle*{4}} \put(175,30){\circle*{4}}
\put(210,30){\circle*{4}}
\put(245,30){\circle*{4}} \put(280,30){\circle*{4}}\put(315,30){\circle*{4}}

\put(0,30){\line(1,0){35}}
\put(105,30){\line(1,0){70}}
\put(245,30){\line(1,0){70}}

\put(2,40){\makebox(0,0){\small{$(1)$}}}
\put(35,40){\makebox(0,0){\small{$(1)$}}}
\put(70,40){\makebox(0,0){\small{$(1)$}}}
\put(105,40){\makebox(0,0){\small{$(1)$}}}
\put(140,40){\makebox(0,0){\small{$(1)$}}}
\put(175,40){\makebox(0,0){\small{$(1)$}}}
\put(210,40){\makebox(0,0){\small{$(1)$}}}
\put(245,40){\makebox(0,0){\small{$(1)$}}}
\put(280,40){\makebox(0,0){\small{$(2)$}}}
\put(313,40){\makebox(0,0){\small{$(1)$}}}

\put(2,20){\makebox(0,0){\small{$-1$}}}
\put(35,20){\makebox(0,0){\small{$-1$}}}
\put(70,20){\makebox(0,0){\small{$0$}}}
\put(105,20){\makebox(0,0){\small{$-1$}}}
\put(140,20){\makebox(0,0){\small{$-2$}}}
\put(175,20){\makebox(0,0){\small{$-1$}}}
\put(210,20){\makebox(0,0){\small{$0$}}}
\put(245,20){\makebox(0,0){\small{$-2$}}}
\put(280,20){\makebox(0,0){\small{$-1$}}}
\put(313,20){\makebox(0,0){\small{$-2$}}}
\end{picture}

\end{example}

\bekezdes\label{back} Now,  we return to the connected component
$\gj$ of $\gck$ corresponding to $\Sigma_j$ ($1\leq j\leq s$).
We consider the partition  $\{G(K_\ell)\}_\ell$ of  $\gj$; they are connected by
2--edges.
The geometry behind the next discussion is the following.

Recall that
 $T\Sigma_j$ denotes  the equisingular type of the transversal singularity
associated with $\Sigma_j$, cf. \ref{ss:2.0b},   and  ${\rm deg}(g|\Sigma_j)=d_j$, cf. \ref{bek:dj}.
If $(Sl,q)$ is a
transversal slice as in \ref{ss:2.0b}, then $r$ above $(Sl,q)$
determines a resolution of the transversal plane curve singularity
$(Sl, Sl\cap V_f,q)$. We denote its weighted dual embedded
resolution graph  by $G(T\Sigma_j)$. Since in local coordinates it is easier to work with
the pullback of $g$, it is convenient to replace the single point
$q\in \Sigma_j-\{0\}$ by the collection of $d_j$  points
$g^{-1}(d)\cap \Sigma_j$, where $|d|$ is small and non-zero.
The dual weighted graph associated with the curves situated
above these points consists of  exactly $d_j$ identical
copies of $G(T\Sigma_j)$, and it is  denoted by  $d_j\cdot G(T\Sigma_j)$.
\ix{transversal!type}
\ix{graph!covering}

Comparing the curves  $r^{-1}(g^{-1}(d)\cap \Sigma_j)$ and
$\C_{\Sigma_j}$ via the corresponding local equations, and using the
results of Proposition \ref{gkl}, we obtain
a cyclic covering of graphs
$$p:d_j\cdot G(T\Sigma_j)\to \mbox{\{a base graph\}}_j,$$
where the base graph and the covering data can be determined from $\gj$.
This is given in the next paragraphs. \ix{graph!covering!data}

\vspace{2mm}

\noindent {\bf The base graph} will be denoted by  $\gj/\sim$.
It  is obtained from $\gj$ by collapsing
it along edges of weight 1.
More precisely, each subgraph $G(K_\ell)$ is  replaced by a non--arrowhead
vertex. If two subgraphs $G(K_\ell)$ and $G(K_{\ell'})$ are connected by
$k$ 2--edges  in $\gj$, then  the corresponding
vertices of $\gj/\sim$ are connected by $k$ edges.
(In fact, in  \ref{tree} we will see that each  $k\leq 1$.)
If the non--arrowhead vertices of $G(K_\ell)$ support $k$ arrowheads altogether,
then on the corresponding non--arrowhead  vertex of
$\gj/\sim$ one has exactly $k$ arrowheads.

Since $\gj$ is connected, it is obvious that $\gj/\sim$ is connected as well.

\vspace{2mm}

\noi {\bf The covering data of}  $p:d_j\cdot G(T\Sigma_j)\to \gj/\sim$. \ix{graph!covering!data}

Recall from \ref{ss:2.3b} that the {\em covering data } of a projection $p:G\to \G$    is a collection
 of positive integers  $\{\n_v\}_{v\in \calv(\G)}$ and
$\{\n_e\}_{e\in \cale(\G)}$, such that for each edge $e=(v_1,v_2)\in \cale(\G)$
one has  $\n_e=\dg_e\cdot {\rm lcm}(\n_{v_1},\n_{v_2})$
for some integer $\dg_e$.

Now, we define a covering data for $\gj/\sim$. It is provided by the
third entries $\nu$  of the weights $(m;n,\nu)$ of the vertices of
$\gj$ and will be denoted   by $\bnu$.
%\marginpar{hogyen lehet bf}

For any non--arrowhead vertex  $w$ of $\gj/\sim$, which corresponds to $K_\ell$
in  the above construction, set $\n_w:=\nu(K_\ell)$.
For any arrowhead vertex $v$ of $\gj/\sim$, which corresponds to
an arrowhead of $\gj$ with weight $(1;n,\nu)$, set $\n_v:=\nu$.
For any edge of $\gj/\sim$, which comes from a 2--edge $e$ of $\gj$
with endpoints with weight $(*;n,\nu)$, set $\n_e:=\nu $.

\vspace{2mm}

The degeneration of $r^{-1}(g^{-1}(d)\cap \Sigma_j)$ into $\C_{\Sigma_j}$
provides the next result:

\begin{theorem}\label{cts1} {\bf Characterization of the transversal singularities.}

(a) \  For any $j$ there exists a cyclic covering of graphs \ix{graph!covering}
$$p:d_j\cdot G(T\Sigma_j)\to \gj/\sim$$
with covering data $\bnu$  and with the compatibility of the
arrowheads: $\cala(d_j\cdot G(T\Sigma_j))=p^{-1}(\cala(\gj/\sim))$,
cf. \ref{re:2.3.1}.

(b) The decorations of  $G(T\Sigma_j)$ can be recovered
from the decorations of $\gj$ as follows:
$m_w=m(K_\ell)$ for any $w\in \calw(G(T\Sigma_j))$ sitting above a vertex  corresponding to $K_\ell$;
$m_v=1$ for any arrowhead.
The  Euler numbers are  determined via (\ref{eq:2.2.1}).\ix{plumbing!graph}

\vspace{2mm}

In particular, the weighted dual embedded resolution graph $G(T\Sigma_j)$
  can be completely determined from the weighted graph $\gj$.
\end{theorem}
\begin{corollary}\label{tree}
\begin{enumerate}
\item  \ $\gj$ is a connected tree.\vspace{2mm}

\item \ With covering data $\bnu$, \ix{graph!covering}
there is only one cyclic graph covering
$p:G\to \gj/\sim$.

\item \ $d_j={\rm gcd}\{\nu_w|w\in \calw(\gj)\}$,  where $(m_w;n_w,\nu_w)$ is the weight of
$C_w$
 \end{enumerate}
\end{corollary}
\begin{proof} The first part  follows from the connectedness statement from \ref{m2tetel}, form the fact that
$G(T\Sigma_j)$ is a tree, and from \ref{cgcg}.
The second and third parts follow from \ref{ex:COVERING}(1) and the connectedness of $G(T\Sigma_j)$.
\end{proof}

\begin{remark}\labelpar{rm:dj}
\begin{enumerate}
\item
For an example when the covering \ix{graph!covering}
$p:d_j\cdot G(T\Sigma_j)\to \gj/\sim$ is not a bijection, see  \ref{ex:nu2b}.\vspace{2mm}

\item
 As we emphasized in \ref{gkl}(b--c),
the collection of integers  $\{\nu_w\} \, (w\in \calw(\gj))$ satisfies
serious compatibility restrictions.
Moreover, since in the cyclic covering $d_j\cdot G(T\Sigma_j)\to
\gj/\sim$ \ the covering graph  $d_j\cdot G(T\Sigma_j)$ has no cycles,
this imposes  some  additional restrictions on the
integers $\{\nu_w\} \, (w\in \calw(\gj))$.\vspace{2mm}

\item Corollary \ref{tree}(3) implies  that, in fact,
there is a graph covering of
connected graphs $$G(T\Sigma_j)\to \gj/\sim$$ \noindent whose covering data are those
from $\bnu$, where all integers are  divided by $d_j$.
\end{enumerate}
\end{remark}

\begin{example}\label{ex:347bal3} Let us continue the example \ref{ex:347bal2}
(as the continuation of  \ref{ex:347} and  \ref{ex:347bal}). In this case $d_1=1$ and the
$G(T\Sigma_1)\to \Gamma^2_{\C,1}/\sim$ is a bijection.
Hence the algorithm gives for $G(T\Sigma_1)$ the embedded resolution graph

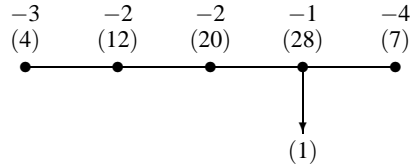
\begin{picture}(300,70)(-70,-5)
\put(0,30){\circle*{4}}
\put(35,30){\circle*{4}} \put(70,30){\circle*{4}}
\put(105,30){\circle*{4}}
\put(140,30){\circle*{4}}

\put(0,30){\line(1,0){140}}
\put(105,30){\vector (0,-1){25}}

\put(0,40){\makebox(0,0){\small{$(4)$}}}
\put(35,40){\makebox(0,0){\small{$(12)$}}}
\put(70,40){\makebox(0,0){\small{$(20)$}}}
\put(105,40){\makebox(0,0){\small{$(28)$}}}\put(105,-2){\makebox(0,0){\small{$(1)$}}}
\put(140,40){\makebox(0,0){\small{$(7)$}}}

\put(0,50){\makebox(0,0){\small{$-3$}}}
\put(35,50){\makebox(0,0){\small{$-2$}}}
\put(70,50){\makebox(0,0){\small{$-2$}}}
\put(105,50){\makebox(0,0){\small{$-1$}}}
\put(140,50){\makebox(0,0){\small{$-4$}}}
\end{picture}

\vspace{3mm}

This is, of course, the minimal embedded resolution graph of $T\Sigma_1$ with local equation
$u^7-v^4=0$, as it is expected from the equations of $f=x^3y^7-z^4$.

\end{example}

\begin{example}\labelpar{ex:ketA2b} Consider the graph in \ref{ex:ketA2} for
 $f=y^3+(x^2-z^4)^2$ and $g=z$. $\Sigma=\{y=x^2-z^4=0\}$ has two components, the transversal
 type of which  are cusps of type $(2,3)$. In the next diagram we put in dash--boxes
  the equivalence classes $K$ and
 the supports of the two components $\{\gj\}_{j=1,2}$:
%\marginpar{LEgalso csomo???}

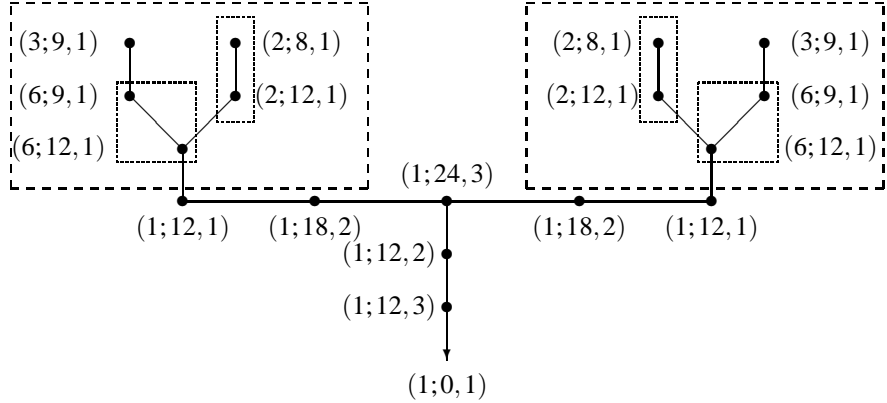
\vspace{5mm}

\begin{picture}(300,160)(-5,-50)
\put(25,45){\dashbox{1}(30,30){}}
\put(63,60){\dashbox{1}(14,40){}}
\put(245,45){\dashbox{1}(30,30){}}
\put(223,60){\dashbox{1}(14,40){}}

\put(-15,35){\dashbox{3}(135,70){}}
\put(180,35){\dashbox{3}(135,70){}}

\put(150,30){\circle*{4}} \put(50,30){\circle*{4}}
\put(100,30){\circle*{4}} \put(200,30){\circle*{4}}
\put(250,30){\circle*{4}}
\put(50,50){\circle*{4}}\put(250,50){\circle*{4}}

\put(30,70){\circle*{4}} \put(70,70){\circle*{4}}
\put(230,70){\circle*{4}} \put(270,70){\circle*{4}}

\put(50,30){\line(1,0){200}}\put(50,50){\line(1,1){20}}
\put(50,50){\line(-1,1){20}} \put(250,50){\line(1,1){20}}
\put(250,50){\line(-1,1){20}}
\put(50,30){\line(0,1){20}}\put(250,30){\line(0,1){20}}

\put(30,90){\circle*{4}}
\put(70,90){\circle*{4}}
\put(270,90){\circle*{4}}
\put(230,90){\circle*{4}}
\put(150,10){\circle*{4}}\put(150,-10){\circle*{4}}

\put(30,70){\line(0,1){20}}\put(70,70){\line(0,1){20}}
\put(230,70){\line(0,1){20}}\put(270,70){\line(0,1){20}}
\put(150,30){\vector(0,-1){60}}

\put(3,90){\makebox(0,0){$(3;9,1)$}}
\put(3,70){\makebox(0,0){$(6;9,1)$}}
\put(3,50){\makebox(0,0){$(6;12,1)$}}
\put(95,90){\makebox(0,0){$(2;8,1)$}}
\put(95,70){\makebox(0,0){$(2;12,1)$}}

\put(205,90){\makebox(0,0){$(2;8,1)$}}
\put(205,70){\makebox(0,0){$(2;12,1)$}}

\put(295,90){\makebox(0,0){$(3;9,1)$}}
\put(295,70){\makebox(0,0){$(6;9,1)$}}
\put(295,50){\makebox(0,0){$(6;12,1)$}}

\put(50,20){\makebox(0,0){$(1;12,1)$}}
\put(100,20){\makebox(0,0){$(1;18,2)$}}
\put(150,40){\makebox(0,0){$(1;24,3)$}}
\put(200,20){\makebox(0,0){$(1;18,2)$}}
\put(250,20){\makebox(0,0){$(1;12,1)$}}

\put(128,10){\makebox(0,0){$(1;12,2)$}}
\put(128,-10){\makebox(0,0){$(1;12,3)$}}
\put(150,-40){\makebox(0,0){$(1;0,1)$}}
\end{picture}

By the above algorithm, one can easily recover the transversal types,
 and the graph $E_6$ of the normalization of $V_f$.

\end{example}

The results of this chapter  together with  Theorem \ref{th:2.3.1}
culminate in the following corollary
which is  crucial for the algorithm presented in Chapter \ref{s:ALG}.

\begin{theorem}\label{unique}  Up to isomorphism of cyclic coverings of
graphs (with a fixed  covering data), there is only one cyclic covering of the
graph $\gc$ provided that the covering data satisfies $\n_v=1$ for any
$v\in \calv^1(\gc)$.
\end{theorem} \ix{graph!covering}

\section{\ Cutting edges revisited}\labelpar{2edg} \setcounter{equation}{0} \

\bekezdes\label{bek:ce1}
 In this section we will analyze in  detail the properties of cutting edges and
we list some consequences.
\ix{cutting edge}

Consider a cutting  edge $e$ of $\gc$, cf.  \ref{stricttra}. Recall that it always has
edge decoration 2.
Assume that the weights of the end-vertices have the form  $(*\,;n,\nu)$.
In order to indicate the dependence on $e$, we write $\nu=\nu(e)$.

Fix local coordinates as in \ref{summary}, and set
$C^*_e:=B_1^l\cap B_2^l$. Here  $B_1^l$ and $B_2^l$ are those two  local components along which
the restriction of the function $f\circ r$ vanishes, but $g\circ r$ does not vanish.
Then $r$ projects $C^*_e$ onto a certain
component $\Sigma_j$ of $\Sigma$. Moreover, $r|_{C^*_e}:C^*_e\to \Sigma_j$ is
finite. Denote its degree by $d(e)$. Obviously
$${\rm deg}(r|C^*_e)\cdot
{\rm deg}(g|\Sigma_j)= {\rm deg}(g\circ r|C^*_e)=\nu(e),$$ hence
\begin{equation}\label{eq:DEDNU}
d(e)\cdot  d_j=\nu(e).\end{equation}

Since $\nu(e) $ and $d_j$ can be obtained from $\gj$ (cf.
\ref{tree}), the degree $d(e)$ can also be recovered from
$\gj$.

\bekezdes \ For every fixed $j\in\{1,\ldots,s\}$, let $\cale_{cut,j}$ be the set of cutting
edges  connecting $\gj$ with $\gce$. Also, write $\#T\Sigma_j$ for
the number of local irreducible components of the transversal singularity
$T\Sigma_j$, which coincides with the number of connected components of $\partial F'_j$.
\ix{cutting edge}

The point is that $\#(T\Sigma_j)\geq |\cale_{cut,j}|$, and, in
general, equality does not hold. Indeed, from the local equations and from the covering
$r|_{C^*_e}:C^*_e\to \Sigma_j$   one deduces that
each $e\in \cale_{cut,j}$ is `responsible' for $d(e)$ local irreducible
components of $T\Sigma_j$. In other words,
\begin{equation}\label{eq:tr}
\#T\Sigma_j=\sum_{e\in\cale_{cut,j}}\, d(e).
\end{equation}

\begin{example}\label{ex:nu2b} Consider the example from \ref{nu2}.
In this case $f=x^2y^2+z^2(x+y)$, hence  it has two singular components
$\Sigma=\Sigma_1\cup\Sigma_2=\{xy=z=0\}$,  whose  transversal type singularities  are $A_1$,
hence $\# T\Sigma_j=2$ for $j\in\{1,2\}$. The linear
function $g$ induces on both $d_j=1$. Furthermore, from the graph \ref{nu2}
we get that in both cases $\cale_{cut,j}=1$. This is compatible with the above discussion, since
for the  cutting edges $\nu(e)=2$ in either case.

The graph $\gj$ has three vertices, all of them are in the same class, hence
$\gj/\sim$ has only one vertex. The covering $G(T\Sigma_j)\to \gj/\sim $ \ is the following:

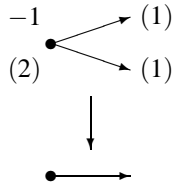
\begin{picture}(200,100)(80,40)
\put(210,60){\circle*{4}}
\put(210,60){\vector(1,0){30}}
\put(200,100){\makebox(0,0){$(2)$}}
\put(200,120){\makebox(0,0){$-1$}}
\put(225,90){\vector(0,-1){20}}
%\put(220,80){\makebox(0,0){$p$}}
\put(250,100){\makebox(0,0){$(1)$}}\put(250,120){\makebox(0,0){$(1)$}}
\put(210,110){\circle*{4}}
\put(210,110){\vector(3,1){30}}
\put(210,110){\vector(3,-1){30}}
\end{picture}

Hence, in general, the covering $G(T\Sigma_j)\to \gj/\sim $ \ is not a bijection. \ix{graph!covering}
\end{example}

\begin{example}\label{ex:Uu} If $\#T\Sigma_j\not=1$, then even if both graphs $\gce$ and $\gck$ are trees, it
might happen that $\gc$ has cycles. For example, in the case presented in \ref{ex:d1d2},
$\Sigma=\Sigma_1$ is irreducible with $\#T\Sigma_j=2$, both $\gce$ and $\gck$ are trees, and $\gc$ has one cycle.
\end{example}

\bekezdes {\bf Relationship with the resolution graph of the normalization $V_f^{norm}$.}\label{bek:ste}

By Theorem \ref{m1} there is a one-to-one correspondence between the cutting edges and
the irreducible components of the strict transform of $\Sigma_f$ in the normalization $V_f^{norm}$ of $V_f$.
By the construction of the graph $\gc$, and by the discussion  \ref{bek:ce1}, these are in bijection with
components of type $C^*_e$. Moreover, via $r$ and Theorem \ref{m1}, each $C^*_e$ can be identified with
the corresponding strict transform component $St_e\subset V_f^{norm}$. In particular,
the restriction of the normalization map $n$ satisfies:
$${\rm deg}(n|_{St_e}:St_e\to \Sigma_j)= {\rm deg}(r|_{C^*_e}:C^*_e\to \Sigma_j)=d(e).$$
If $St(\Sigma_j)$ denotes the strict transform of $\Sigma_j$ in $V^{norm}_f$, then
$$St(\Sigma_j)=\bigcup_{e\in \cale_{cut,j}}\, St_e,  $$
and each $St_e$ contributes  $d(e)$ components in $\#T\Sigma_j$, compatibly with (\ref{eq:tr}).
Looking at the local equations, specifically at the last multiplicities $\nu$, we obtain
the next reinterpretation of the identity (\ref{eq:tr}).
\ix{graph!of normalization of $V_f$}

\begin{corollary}\label{cor:VERTMON} \
\begin{enumerate}
\item
The vertical  monodromy $m'_{j,ver}$ permutes the
connected components of $\partial F'_j$; each orbit corresponds to
a cutting edge $e\in \cale_{cut,j}$, and the  cardinality
of the corresponding orbit is $d(e)$.\vspace{2mm}

\item
The vertical  monodromy $m^\Phi_{j,ver}$ permutes the
connected components of $d_j\cdot \partial F'_j$; each orbit corresponds to
a cutting edge $e\in \cale_{cut,j}$, and the  cardinality
of the corresponding orbit is $d_jd(e)=\nu(e)$.
\end{enumerate}
\end{corollary}\ix{cutting edge}

\bekezdes {\bf The construction of the link $K$ of $f$ from $\gc$.}
Consider the link $K^{norm}:=K_{V_f^{norm}}$ of $V_f^{norm}$. It
 is the disjoint union of the (connected) links
of the irreducible components of  $V_f^{norm}$. In it consider the 1--dimensional sub-manifold
$$\bigcup_j(St_j\cap K^{norm})=\bigcup_j\bigcup_{e\in \cale_{cut,j}}( St_e\cap K^{norm})\subset K^{norm}.$$
Assume that  each component $St_e\cap K^{norm}$, denoted by $S^1_e$ (and which is diffeomorphic to $S^1$)
is marked by two data, one of them is an element $j\in\{1,\ldots,s\}$, the index $j=j(e)$ of $\Sigma_j$
 onto which $St_e$ is mapped, the other is the degree $d(e)$ of $St_e\to \Sigma_j$.
\ix{link!of a surface singularity} \ix{link!of a surface singularity!of normalization of $V_f$}

We claim that from the data $(K^{norm},\cup_e (S^1_e;j(e),d(e)))$ one can recover the link $K$ of $f$.
Indeed, for each $j\in\{1,\ldots, s\}$ fix a circle $S^1=S^1_j$. Moreover, for each $e$ with $j(e)=j$
fix a cyclic covering $\phi_e:S^1_e\to S^1_j$ of degree $d(e)$.  Then $K$ is obtained from $K^{norm}$ by
gluing its points  via the maps $\phi_e$.
\begin{proposition}\label{prop:KK}
Introduce an equivalence relation on $K^{norm}$ as follows: $x\sim x'$ if and only if there exist $e$ and $e'$,
with $x\in S^1_e$ and $x'\in S^1_{e'}$  (where $e=e'$ is allowed) such that $\phi_e(x)=\phi_{e'}(x')$
(and any other equivalence has the form $y\sim y$).  Then
$$K=K^{norm}/\sim.$$
\end{proposition}

Now, clearly, the above data $(K^{norm},\cup_e (S^1_e;j(e),d(e)))$ can be deduced from $\gc$.
Recall that in section \ref{gamma1} we provide the plumbing graph $G_\C^1$ for $K^{norm}$
 from $\gce$. One has only  to modify this construction as follows. In the construction of
 $\gce$ one has to decorate the 2--edges $e$ (or their arrowheads) by the extra decoration $(j(e),d(e))$,
 and keep this extra decoration for the $(0)$--arrowheads of  $G^1_\C$ too. Then this enhanced $G^1_\C$
 contains all the information needed to apply  \ref{prop:KK}.

\bekezdes Here is a  picture summarizing in a schematic form
the essential features of the decomposition of the graph $\gc$
into $\gce$ and $\gck$, the classes $K_\ell$, the types of the intersection points corresponding to
1-- and 2-- and cutting edges, \ix{cutting edge}
and the degenerations (\ref{zart}), \ref{connectedhez} and (\ref{zart2}).

\begin{picture}(300,600)(-10,-100)

\put(40,40){\framebox(100,160)}
\put(40,200){\line(1,1){100}} \put(140,200){\line(1,1){100}}
\put(140,300){\framebox(100,160)}

\put(40,10){\line(1,0){100}}  \put(50,15){\makebox(0,0){$\Sigma_j$}}
 \put(90,0){\makebox(0,0){$0$}}
\put(120,0){\makebox(0,0){$q$}}
 \put(90,10){\circle*{3}}
\put(120,10){\circle*{3}}

\put(0,43){\vector(0,1){154}}\put(0,197){\vector(0,-1){154}}
 \put(-10,120){\makebox(0,0){$\Gamma^1_{{\cal C}}$}}

\put(0,203){\vector(0,1){254}}\put(0,457){\vector(0,-1){254}}
 \put(-10,330){\makebox(0,0){$\Gamma^2_{{\cal C},j}$}}

\put(10,203){\vector(0,1){94}}\put(10,297){\vector(0,-1){94}}
 \put(20,250){\makebox(0,0){$K_{\ell_1}$}} \put(70,207){\makebox(0,0){$C^*_e$}}

\put(10,303){\vector(0,1){154}}\put(10,457){\vector(0,-1){154}}
 \put(20,380){\makebox(0,0){$K_{\ell_2}$}}
\qbezier(90,40)(100,75)(80,110) \qbezier(90,200)(100,165)(80,130)
\qbezier(80,80)(100,120)(80,160)
\qbezier(190,300)(200,335)(180,370) \qbezier(190,460)(200,425)(180,390)
\qbezier(180,340)(200,380)(180,420)
\qbezier(90,200)(140,225)(125,250) \qbezier(135,260)(180,275)(190,300)
\qbezier(115,230)(150,250)(145,270)

\qbezier(120,40)(130,120)(120,200)
\qbezier(220,300)(230,380)(220,460)
\qbezier(120,200)(170,260)(220,300)

\put(175,170){\vector(-1,0){23}}\put(180,170){\makebox(0,0)[l]{generic fiber of  $\Phi$ over $W_{\eta,M}$}}
\put(175,90){\vector(-1,0){70}}\put(180,90){\makebox(0,0)[l]{resolution component of $V^{norm}_f$}}
\put(175,130){\vector(-1,0){50}}\put(180,130){\makebox(0,0)[l]{generic fiber of  $g\big|V^{norm}_f$}}
\put(175,210){\vector(-1,0){45}}\put(180,210){\makebox(0,0)[l]{exceptional component of the}}
\put(225,198){\makebox(0,0)[l]{resolution of $T\Sigma_j$}}
\put(110,400){\vector(1,0){40}}\put(45,400){\makebox(0,0)[l]{$B(K_{\ell_2})\subset \de_{c,j}$}}
\put(240,250){\vector(-1,0){100}}\put(245,250){\makebox(0,0)[l]{$\C$}}

\qbezier(160,40)(120,210)(210,250)
\qbezier(210,250)(290,290)(260,460)
\qbezier(120,200)(170,260)(220,300)

\put(59,182){\circle{20}} \put(59,182){\makebox(0,0){ce}}\put(68,188){\vector(2,1){22}}
\put(156,340){\circle{20}} \put(156,340){\makebox(0,0){e1}}\put(165,346){\vector(2,1){22}}
\put(56,80){\circle{20}} \put(56,80){\makebox(0,0){e1}}\put(65,86){\vector(2,1){22}}
\put(158,283){\circle{20}} \put(158,283){\makebox(0,0){e2}}\put(168,288){\vector(2,1){22}}

\put(220,54){\makebox(0,0)[l]{ce \ \ =  cutting edge}}
\put(220,32){\makebox(0,0)[l]{e1 \ \ = 1--edge}}\put(220,10){\makebox(0,0)[l]{e2 \ \ = 2--edge}}
\put(224,54){\circle{20}} \put(224,32){\circle{20}} \put(224,10){\circle{20}}

\end{picture}

%\setcounter{temp}{\value{section}}
%\part{Examples of the graph  $\gc$}\label{CHV}
%\setcounter{section}{\value{temp}}

\chapter{Examples. Homogeneous singularities}\labelpar{hom}
\section{\ The general case}\labelpar{homgen} \setcounter{equation}{0}
Assume that $f:(\bfc^3,0)\to(\bfc,0)$ is the germ of a homogeneous polynomial of degree $d$, and we
choose $g$ to be a generic linear function with respect to $f$.
\ix{singularities!homogeneous|textbf}
\ix{graph!$\G_\C$}

Let $C\subset \bfc\bp^2$ be the projective plane curve $\{f=0\}$.

We show that a possible graph $\gc$ can easily be determined from
the combinatorics of the components  and the topological
types of the  local singularities of $C$.

In this projective setting we use the following notations.

Let $C=\cup_{\lambda\in\Lambda} C_\lambda$ be the irreducible
decomposition of
  $C$, and set $d_\lambda:=\deg(C_\lambda)$.  Hence $\sum_\lambda d_\lambda=d$.
Furthermore, let $g_\lambda$ be the genus of the normalization of $C_\lambda$.

Let $\{p_j\}_{j\in\Pi}$ be the set of singular points of $C$.
Assume that the local analytic irreducible components  of $(C,p_j)$  are
$(C_{j,i},p_j)_{i\in I_j}$. Clearly, there is an `identification
map' of global/local components
 $c:\cup_j I_j \to \Lambda$ which sends the index of a local component
$C_{j,i}$ into the index $\lambda$ whenever $C_{j,i}\subset
C_\lambda$.

Let $\Gamma_j$ be an embedded resolution graph of the
local plane curve singularity $(C,p_j)\subset (\bfc^2,p_j)$.  It has $|I_j|$
arrowheads, each with multiplicity $(1)$.

These notations agree with some of the notations already considered for germs in three
variables in the previous sections, for example,  in \ref{ss:2.0b}. Indeed,
the number of singular points of $C$ is the same as the number of irreducible components
of $\Sigma_f$, hence $\Pi$ corresponds to $\{1,\ldots,s\}$. Moreover,
the local topological type of the plane curve singularity
$(C,p_j)\subset (\bfc\bp^2,p_j)$ at a singular point $p_j$ of $C$ agrees with the corresponding
$T\Sigma_j$, hence  $|I_j|=\#T\Sigma_j$.
% and $\Gamma_j$.  is the same as $G(T\Sigma_j)$ of (\ref{prop:cov}).

\begin{proposition}\labelpar{prop:hom}
A possible $\gc$ is constructed from the dual graphs
$\{\Gamma_j\}_{j\in\Pi}$ and $|\Lambda|$ additional non--arrowhead
vertices as follows:

First,   for each $\lambda\in\Lambda$ put a non--arrowhead vertex
$v_\lambda$ in $\gc$ and decorate it with $(1;d,1)$ and
$[g_\lambda]$.  Moreover, put $d_\lambda$ edges supported by
$v_\lambda$, each of them decorated by 1 and supporting an arrowhead
weighted by  $(1;0,1)$.

Then, consider each graph $\Gamma_j$, keep its shape, but replace
the decoration of each non--arrowhead with multiplicity $(m)$ by
the new decoration $(m;d,1)$, and decorate all edges by 2.
Furthermore, each arrowhead of $\Gamma_j$, corresponding to the
local component $C_{j,i}$, is identified with $v_\lambda$, where
$\lambda$ corresponds to $C_{j,i}$ via the local/global identification $c$.
\end{proposition}

\begin{proof}
The following sequence of blow ups is performed. First
 the origin of $\bfc^3$ is blown up. This creates an exceptional divisor
$E=\bfc\bp^2$ which intersects the strict  transform $St(V_f)$ of
$V_f$ along $C$. Moreover, the strict transform of $V_g$ (where
$g$ is the chosen linear function)
  intersects the strict transform of each
irreducible component $V(f)_\lambda$ in $d_\lambda$  discs.

The singular part of $St(V_f)$ consists of discs meeting $E$ in
the singular points $p_j$ of $C$.
The plane curve singularity $(C,p_j)\subset (\bfc\bp^2,p_j)$ can be resolved
by a sequence of blow ups infinitely near points of $p$; this sequence is replaced
in the present local product situation by blowing up
infinitely  near  discs following the blowing up procedure of
the corresponding plane curve singularity.
Then the corresponding decorations follow easily.
\end{proof}

\begin{remark}\label{re:HOM}
Notice that if a local graph $\Gamma_j$ is a double--arrow
(representing a local singularity of type $A_1$ with local equation $xy=0$) and both local
irreducible components sit on the same global component
$C_\lambda$, then by the above procedure the double arrow transforms into a loop
supported on $v_\lambda$ decorated by 2. If the two local
irreducible components sit on two different global components,
then it becomes a 2--edge. In both cases, the corresponding edge
will not satisfy Assumption A \ref{re:w2}.\ix{Assumption A}

Nevertheless, if we consider embedded resolution graphs $\Gamma_j$ with at least one
non-arrowhead vertex (e.g. the graphs of $A_1$ singularities will  have
one $(-1)$-vertex), then the graph $\gc$ obtained in this way will satisfy
Assumption A (and will be related with the previous graph  by the moves of \ref{re:w2}).

Usually, it is preferable to take for $\Gamma_j$ the minimal embedded resolution. Nevertheless,
if we want to get $\gc$ satisfying Assumption A, then we follow the convention
that for an $A_1$ singularity $\Gamma_j$ contains  one non-arrowhead.
\end{remark}

In the next examples we ask the reader to determine for each case the graphs $\Gamma_j$.
Several procedures are described in \cite{BrKn,Wall}. We will provide only the output $\gc$.
\begin{example}\labelpar{ex:dd-1} If $f=z^d-xy^{d-1}$ with $d\geq 3$,
then $\gc$ is the following:

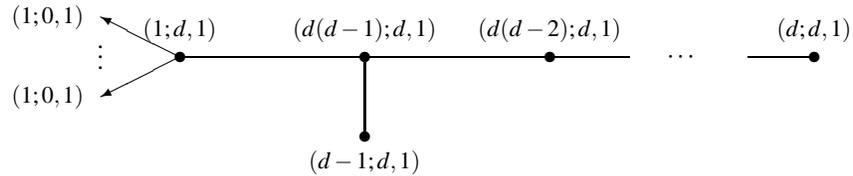
\vspace{2mm}

\begin{picture}(350,95)(-40,-30)
\put(20,25){\circle*{4}} \put(90,25){\circle*{4}}
\put(160,25){\circle*{4}} \put(90,-5){\circle*{4}}
\put(260,25){\circle*{4}} \put(20,25){\line(1,0){170}}
\put(260,25){\line(-1,0){25}}
\put(210,25){\makebox(0,0){$\cdots$}}
\put(-10,28){\makebox(0,0){$\vdots$}} \put(90,25){\line(0,-1){30}}

\put(20,35){\makebox(0,0){\small{$(1;d,1)$}}}
\put(90,35){\makebox(0,0){\small{$(d(d-1);d,1)$}}}
\put(160,35){\makebox(0,0){\small{$(d(d-2);d,1)$}}}
\put(90,-15){\makebox(0,0){\small{$(d-1;d,1)$}}}
\put(260,35){\makebox(0,0){\small{$(d;d,1)$}}}

\put(-30,40){\makebox(0,0){\small{$(1;0,1)$}}}
\put(-30,10){\makebox(0,0){\small{$(1;0,1)$}}}

\put(20,25){\vector(-2,1){30}} \put(20,25){\vector(-2,-1){30}}
\end{picture}

\vspace{2mm}

\noi where there are $d$ arrowheads, and all the edges connecting
non--arrowheads have decoration  2.

\end{example}

\begin{example}\labelpar{ex:3A2}
Assume that $f=x^2y^2+y^2z^2+z^2x^2-2xyz(x+y+z)$. Then $C$ is an
irreducible rational curve with three $A_2$ (ordinary cusp) singularities.
Therefore, $\gc$ is:

\vspace{2mm}

\begin{picture}(350,95)(-35,-25)
\put(20,25){\vector(-2,1){30}} \put(20,25){\vector(-2,-1){30}}
\put(20,25){\vector(-4,1){30}} \put(20,25){\vector(-4,-1){30}}
\put(20,25){\circle*{4}}
%\put(-30,28){\makebox(0,0){$\vdots$}}
\put(20,37){\makebox(0,0){\small{$(1;4,1)$}}}

\put(20,25){\line(2,1){40}} \put(20,25){\line(1,0){40}}
\put(20,25){\line(2,-1){40}}

\put(55,20){\framebox(30,10){}} \put(55,40){\framebox(30,10){}}
\put(55,0){\framebox(30,10){}}

\put(110,27){\makebox(0,0){$\mbox{where}$}}
\put(175,27){\makebox(0,0){$\mbox{is}$}}

\put(132,20){\framebox(30,10){}} \put(127,25){\line(1,0){10}}
\put(195,-5){\framebox(90,55){}}
\put(270,25){\line(-1,0){85}}
\put(210,25){\circle*{4}} \put(270,25){\circle*{4}}
\put(270,35){\makebox(0,0){\small{$(3;4,1)$}}}
\put(210,35){\makebox(0,0){\small{$(6;4,1)$}}}
\put(210,25){\line(0,-1){20}} \put(210,5){\circle*{4}}
\put(230,5){\makebox(0,0){\small{$(2;4,1)$}}}
\put(-30,40){\makebox(0,0){\small{$(1;0,1)$}}}
\put(-30,10){\makebox(0,0){\small{$(1;0,1)$}}}
\put(-30,20){\makebox(0,0){\small{$(1;0,1)$}}}
\put(-30,30){\makebox(0,0){\small{$(1;0,1)$}}}
\end{picture}
\end{example}

\begin{example}\labelpar{ex:loop}
Let $f=x^d+y^d+xyz^{d-2}$, where $d\geq 3$. Then a possible
$\gc$ which does not satisfy Assumption A is:

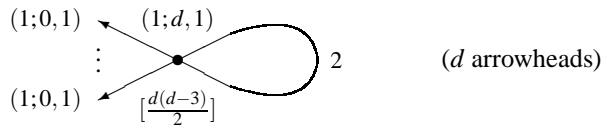
\vspace{2mm}

\begin{picture}(200,70)(-60,-10)
\put(20,25){\vector(-2,1){30}}
\put(20,25){\vector(-2,-1){30}}
\put(20,25){\circle*{4}} \put(-10,28){\makebox(0,0){$\vdots$}}
\put(20,40){\makebox(0,0){\small{$(1;d,1)$}}}
\put(20,7){\makebox(0,0){\small{$[\frac{d(d-3)}{2}]$}}}

\put(20,25){\line(2,1){20}} \put(20,25){\line(2,-1){20}}

\qbezier(40,35)(70,45)(73,25) \qbezier(40,15)(70,5)(73,25)

\put(80,25){\small{\makebox(0,0){$2$}}}
\put(150,25){\makebox(0,0){$\mbox{($d$ arrowheads}$)}}
\put(-30,40){\makebox(0,0){\small{$(1;0,1)$}}}
\put(-30,10){\makebox(0,0){\small{$(1;0,1)$}}}
\end{picture}

Its modification as in  \ref{re:w2}, or as in  (\ref{re:HOM}),  which satisfies Assumption A is:

\vspace{2mm}

\begin{picture}(200,70)(-60,-10)
\put(20,25){\vector(-2,1){30}}
\put(20,25){\vector(-2,-1){30}}
\put(20,25){\circle*{4}}
\put(-10,28){\makebox(0,0){$\vdots$}}
\put(20,40){\makebox(0,0){\small{$(1;d,1)$}}}
\put(20,7){\makebox(0,0){\small{$[\frac{d(d-3)}{2}]$}}}

\put(20,25){\line(2,1){20}}
\put(20,25){\line(2,-1){20}}

\qbezier(40,35)(70,45)(73,25)
\qbezier(40,15)(70,5)(73,25)

\put(73,25){\circle*{4}}
\put(95,25){\makebox(0,0){\small{$(2;d,1)$}}}
\put(75,40){\makebox(0,0){\small{$2$}}}
\put(75,10){\makebox(0,0){\small{$2$}}}
\put(180,25){\makebox(0,0){$\mbox{($d$ arrowheads}$)}}
\put(-30,40){\makebox(0,0){\small{$(1;0,1)$}}}
\put(-30,10){\makebox(0,0){\small{$(1;0,1)$}}}
\end{picture}
\end{example}

\begin{remark}
Consider a {\bf Zariski pair} $(C_1,C_2)$. This means that $C_1$ and $C_2$ are
two irreducible projective curves that
 have the same degree and the topological type of their {\it local}
singularities are the same, while their embeddings in the projective plane are topologically different.
Then the two graphs $\gc(C_1)$ and $\gc(C_2)$ provided by the above algorithm will be the same.
In particular, any invariant derived from $\gc$ (e.g. $\partial F$)
will not differentiate Zariski pairs. \ix{Zariski!pair}\ix{Milnor!fiber!boundary}
\end{remark}

\begin{remark} Since $d_j=1$ for any $j$, and $\nu(e)=d(e)=1$ for any cutting edge $e$, one also
has $\#T\Sigma_j=|\cale_{cut,j}|$. \ix{cutting edge}
\end{remark}

\section{\ Line arrangements}\labelpar{arrang} \setcounter{equation}{0} \
A special case of \ref{homgen} is the case of line arrangements
in $\bfc\bp^2$, that is,  each connected component of $C$ is a line.
\ix{singularities!line arrangements}\ix{line arrangements|see{singularities/line arrangements}}

Having  an arrangement, let $\{L_\lambda\}_{\lambda\in\Lambda}$
be the set of lines, and $\{p_j\}_{j\in\Pi}$ the set of
intersection points. Write  $|\Lambda|=d$, and for each $j$ set
$m_j$ for the cardinality of $I_j:=\{L_\lambda\,:\, L_\lambda\ni p_j\}$.
Then $\gc$ can be constructed as follows:

For each $\lambda\in\Lambda$ put a non--arrowhead vertex
$v_\lambda$ with weight $(1;d,1)$. For each $j\in \Pi$ put a
non--arrowhead vertex $v_j$ with weight $(m_j;d,1)$. Join the
vertices $v_\lambda$ and $v_j$ with a 2--edge whenever $p_j\in
L_\lambda$. Finally, put on each vertex $v_\lambda$ an edge with
decoration 1, which  supports an arrowhead with weight $(1;0,1)$.

Notice that $v_j$ is connected with $m_j$ vertices of type
$v_\lambda$. Clearly, $v_j$ corresponds to the exceptional divisor
obtained by blowing up an intersection point of $m_j$ lines.
Notice that if in the special case of $m_j=2$ --- i.e. when $p_j$
sits only on $L_{\lambda_1}$ and $L_{\lambda_2}$ ---, this blow up
is imposed by Assumption A, cf. \ref{re:w2}. Nevertheless, if we
wish to neglect Assumption A, then this vertex $v_j$ can be
deleted together with the two adjacent edges, and one can simply
put a  2--edge connecting $v_{\lambda_1}$ with $v_{\lambda_2}$.

\begin{example}\labelpar{ex:a3}
In the case of the $A_3$ arrangement $f=xyz(x-y)(y-z)(z-x)$,
the two graphs $\gc$ (satisfying  Assumption A or not) are:
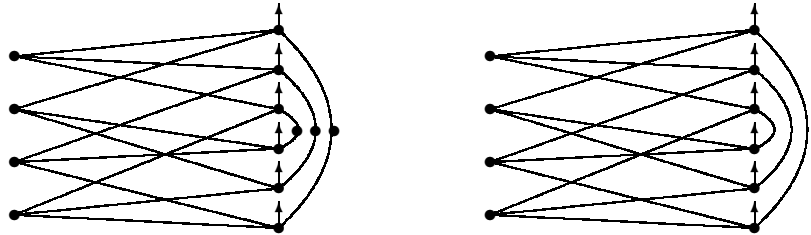
\ix{singularities!line arrangements!$A_3$}

\vspace{3mm}

\begin{picture}(200,120)(50,-10)
\put(50,20){\circle*{4}} \put(50,40){\circle*{4}}
\put(50,60){\circle*{4}} \put(50,80){\circle*{4}}
\put(150,15){\circle*{4}} \put(150,30){\circle*{4}}
\put(150,45){\circle*{4}} \put(150,60){\circle*{4}}
\put(150,75){\circle*{4}} \put(150,90){\circle*{4}}
\put(150,15){\vector(0,1){10}} \put(150,30){\vector(0,1){10}}
\put(150,45){\vector(0,1){10}} \put(150,60){\vector(0,1){10}}
\put(150,75){\vector(0,1){10}} \put(150,90){\vector(0,1){10}}
\qbezier(50,80)(150,90)(150,90) \qbezier(50,80)(150,75)(150,75)
\qbezier(50,80)(150,60)(150,60) \qbezier(50,60)(150,90)(150,90)
\qbezier(50,60)(150,45)(150,45) \qbezier(50,60)(150,30)(150,30)
\qbezier(50,40)(150,75)(150,75) \qbezier(50,40)(150,45)(150,45)
\qbezier(50,40)(150,15)(150,15) \qbezier(50,20)(150,60)(150,60)
\qbezier(50,20)(150,30)(150,30) \qbezier(50,20)(150,15)(150,15)
\qbezier(150,90)(190,52)(150,15) \qbezier(150,75)(178,52)(150,30)
\qbezier(150,60)(165,52)(150,45)
\put(171,52){\circle*{4}}\put(157,52){\circle*{4}}\put(164,52){\circle*{4}}

\put(230,20){\circle*{4}} \put(230,40){\circle*{4}}
\put(230,60){\circle*{4}} \put(230,80){\circle*{4}}
\put(330,15){\circle*{4}} \put(330,30){\circle*{4}}
\put(330,45){\circle*{4}} \put(330,60){\circle*{4}}
\put(330,75){\circle*{4}} \put(330,90){\circle*{4}}
\put(330,15){\vector(0,1){10}} \put(330,30){\vector(0,1){10}}
\put(330,45){\vector(0,1){10}} \put(330,60){\vector(0,1){10}}
\put(330,75){\vector(0,1){10}} \put(330,90){\vector(0,1){10}}
\qbezier(230,80)(330,90)(330,90) \qbezier(230,80)(330,75)(330,75)
\qbezier(230,80)(330,60)(330,60) \qbezier(230,60)(330,90)(330,90)
\qbezier(230,60)(330,45)(330,45) \qbezier(230,60)(330,30)(330,30)
\qbezier(230,40)(330,75)(330,75) \qbezier(230,40)(330,45)(330,45)
\qbezier(230,40)(330,15)(330,15) \qbezier(230,20)(330,60)(330,60)
\qbezier(230,20)(330,30)(330,30) \qbezier(230,20)(330,15)(330,15)
\qbezier(330,90)(370,52)(330,15) \qbezier(330,75)(358,52)(330,30)
\qbezier(330,60)(345,52)(330,45)
%\put(371,52){\circle*{4}}\put(157,52){\circle*{4}}\put(164,52){\circle*{4}}
\end{picture}

\vspace{2mm}

\noindent where on the left hand graph the four left-vertices are weighted by  $(3;6,1)$,
the next  six by $(1;6,1)$, the remaining three by $(2;6,1)$, and the arrowheads by $(1;0,1)$.
The edges supporting arrowheads are decorated by 1, the others by 2.
\end{example}

\begin{example}\labelpar{ex:genarr} The `simplified' graph
$\gc$, which does not satisfy Assumption A,
for the generic arrangement with $d$ lines consists of $d$
vertices $v_\lambda$, each decorated with $(1;d,0)$, each
supporting an arrow $(1;0,1)$, and any pair of non--arrowheads is
connected by a 2--edge.
\end{example}
\ix{singularities!line arrangements!generic}

\chapter{Examples. Families associated with plane curve singularities}\labelpar{s:EXFA}
\section{\ Cylinders of plane curve singularities}\labelpar{cyl} \setcounter{equation}{0}

Consider $f(x,y,z)=f'(x,y)$ and $g(x,y,z)=z$, where
$f':(\bfc^2,0)\to(\bfc,0)$ is an isolated  plane curve
singularity.
It is well--known (see e.g. \cite{BrKn,Hartshorne,Wall}) that  the embedded resolution
of $(\bfc^2,V_{f'})$ can be obtained by a sequence of quadratic transformations.
Replacing the quadratic transformations of the
infinitely  near  points of $0\in \bfc^2$ by blow ups along
the infinitely near  1--dimensional axis  of the  $z$--axis,
one obtains the following picture.
\ix{singularities!cylinders|textbf}

Let $\G(\bfc^2,f')$ denote the minimal
embedded resolution graph of the plane curve singularity
$f':(\bfc^2,0)\to(\bfc,0)$. Recall that, besides the
Euler numbers and genera of the non--arrowheads,  each vertex has
a multiplicity decoration $(m)$, the vanishing order of the
pull-back of $f'$ along the corresponding irreducible curve.

We say that $\{f=0\}$ is the {\it cylinder } of the plane curve $\{f'=0\}$.

\vspace{2mm}

In this situation, one can get  a possible dual graph $\gc$ from
$\G(\bfc^2,f')$ via the following conversion.

The shapes of the two graphs agree, only the decorations are modified:
the Euler numbers  are deleted, while for each vertex the
multiplicity $(m)$ is replaced by $(m;0,1)$. The genus decorations
in $\gc$ --- similarly as in $\G(\bfc^2,f')$ --- of all non-arrowheads
are zero.  Moreover, all edges in $\G_\C$ have weight 2.

\begin{example}\labelpar{ex:cyln}

\noindent Let $f(x,y,z)=f'(x,y)=(x^2-y^3)(y^2-x^3)$. Then
$\G(\bfc^2,f')$ is:

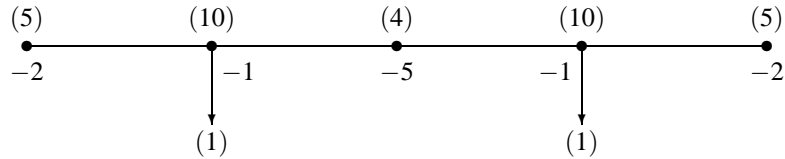
\vspace{3mm}

\begin{picture}(300,75)(10,-20)
\put(20,25){\circle*{4}} \put(90,25){\circle*{4}}
\put(160,25){\circle*{4}} \put(230,25){\circle*{4}}
\put(300,25){\circle*{4}} \put(20,25){\line(1,0){280}}
\put(90,25){\vector(0,-1){30}} \put(230,25){\vector(0,-1){30}}
\put(20,15){\makebox(0,0){$-2$}} \put(100,15){\makebox(0,0){$-1$}}
\put(160,15){\makebox(0,0){$-5$}}
\put(220,15){\makebox(0,0){$-1$}}
\put(300,15){\makebox(0,0){$-2$}}

\put(20,35){\makebox(0,0){$(5)$}}
\put(90,35){\makebox(0,0){$(10)$}}
\put(160,35){\makebox(0,0){$(4)$}}
\put(230,35){\makebox(0,0){$(10)$}}
\put(300,35){\makebox(0,0){$(5)$}}

\put(90,-12){\makebox(0,0){$(1)$}}
\put(230,-12){\makebox(0,0){$(1)$}}

\end{picture}

\vspace{7mm}

\noindent which is transformed into $\gc$ as:

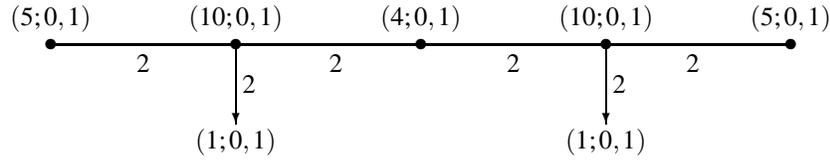
\vspace{3mm}

\begin{picture}(300,80)(10,-25)
\put(20,25){\circle*{4}} \put(90,25){\circle*{4}}
\put(160,25){\circle*{4}} \put(230,25){\circle*{4}}
\put(300,25){\circle*{4}} \put(20,25){\line(1,0){280}}
\put(90,25){\vector(0,-1){30}} \put(230,25){\vector(0,-1){30}}
\put(55,18){\makebox(0,0){$2$}} \put(128,18){\makebox(0,0){$2$}}
\put(195,18){\makebox(0,0){$2$}} \put(263,18){\makebox(0,0){$2$}}

\put(95,10){\makebox(0,0){$2$}} \put(235,10){\makebox(0,0){$2$}}

\put(20,35){\makebox(0,0){$(5;0,1)$}}
\put(90,35){\makebox(0,0){$(10;0,1)$}}
\put(160,35){\makebox(0,0){$(4;0,1)$}}
\put(230,35){\makebox(0,0){$(10;0,1)$}}
\put(300,35){\makebox(0,0){$(5;0,1)$}}

\put(90,-12){\makebox(0,0){$(1;0,1)$}}
\put(230,-12){\makebox(0,0){$(1;0,1)$}}
\end{picture}

\vspace{2mm}

\end{example}

\begin{bekezdes}
It is easy to verify that $\gck=\gc$, and $\gce$ consists of $|\cala|$
double arrows, where  $|\cala|$ is the number of irreducible components of $f'$. The statements of
\ref{gamma1} and \ref{ss:GCK} can easily be verified.
\end{bekezdes}

\section{\ Germs of type $f=zf'(x,y)$}\labelpar{ss:PROD} \setcounter{equation}{0}

 Here $f':(\bfc^2,0)\to(\bfc,0)$ is an isolated  plane curve
singularity as above, $f(x,y,z):=zf'(x,y)$ and $g$ is a generic linear form in variables
$(x,y,z)$.

For this family  we found no nice uniform presentation of $\gc$
with similar simplicity  and conceptual conciseness as in \ref{cyl}, or in the homogeneous case.
(We face the same obstruction as in the case of suspensions,
explained in the second paragraph of \ref{221}).
Since  the  3--manifold $\partial F$ can be determined completely and rather easily\ix{Milnor!fiber!boundary}
for any $f=zf'(x,y)$ by another method, which  will be presented in Chapter
 \ref{s:zf'}, we decided to omit general technical
graph--presentations here. Nevertheless, particular testing examples can be determined without difficulty.
For example,  consider    $f'=x^{d-1}+y^{d-1}$ when $f$ becomes homogeneous
 and $\gc$ can be determined as in Chapter  \ref{hom}. Or, consider  $f'=x^2+y^3$, whose $\gc$ is  below.
For more comments (and mysteries) regarding the possible graphs $\gc$,
see \ref{comments}.

 \begin{example}\labelpar{ex:zf'}
Assume that $f=z(x^2+y^3)$ and take $g$ to be a  generic linear form.
The  `ad hoc blowing up procedure', using the
naive principle to blow up the `worst singular locus',
provides the following  $\gc$,
where we only marked the 2--edges, and all unmarked edges are 1--edges:

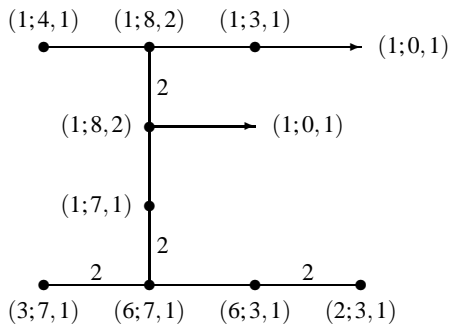
\begin{picture}(200,130)(-80,0)
\put(0,10){\circle*{4}}\put(40,10){\circle*{4}}
\put(80,10){\circle*{4}}\put(120,10){\circle*{4}}
\put(0,100){\circle*{4}}\put(40,100){\circle*{4}}
\put(80,100){\circle*{4}}
\put(40,40){\circle*{4}}\put(40,70){\circle*{4}}
\put(40,10){\line(0,1){90}}\put(40,70){\vector(1,0){40}}
\put(0,10){\line(1,0){120}}\put(0,100){\vector(1,0){120}}
\put(0,0){\makebox(0,0){\small{$(3;7,1)$}}}\put(40,0){\makebox(0,0){\small{$(6;7,1)$}}}
\put(80,0){\makebox(0,0){\small{$(6;3,1)$}}}\put(120,0){\makebox(0,0){\small{$(2;3,1)$}}}
\put(0,110){\makebox(0,0){\small{$(1;4,1)$}}}\put(40,110){\makebox(0,0){\small{$(1;8,2)$}}}
\put(80,110){\makebox(0,0){\small{$(1;3,1)$}}}\put(140,100){\makebox(0,0){\small{$(1;0,1)$}}}
\put(20,40){\makebox(0,0){\small{$(1;7,1)$}}}\put(20,70){\makebox(0,0){\small{$(1;8,2)$}}}
\put(100,70){\makebox(0,0){\small{$(1;0,1)$}}}
\put(20,15){\makebox(0,0){\small{$2$}}}
\put(100,15){\makebox(0,0){\small{$2$}}}
\put(45,25){\makebox(0,0){\small{$2$}}}
\put(45,85){\makebox(0,0){\small{$2$}}}
\end{picture}

\vspace{3mm}

\end{example}

\section{\ Double suspensions}\labelpar{ss:ds}\setcounter{equation}{0}

Suspension, or cyclic covering  singularities are defined by functions of
 the form  $f(x,y,z)=f'(x,y)+z^d$, where
$f':(\bfc^2,0)\to(\bfc,0)$ is  plane curve
singularity. If we wish to get $f$ non--isolated, we have to start
with $f'$ non--isolated. When $d=2$ the germ is called double suspension of $f'$.
When $f'$ is not very complicated,  one might find a convenient resolution
by  `ad hoc' blow ups, such as in the following case:
\ix{singularities!suspensions}

\begin{example}\labelpar{221}
Assume that $f=x^2y+z^2$ and $g=x+y$. Then a possible $\gc$ is:

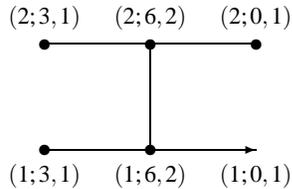
\begin{picture}(200,90)(-50,5)
\put(50,25){\circle*{4}} \put(50,65){\circle*{4}}
\put(90,25){\circle*{4}} \put(90,65){\circle*{4}}
\put(130,65){\circle*{4}}

\put(50,65){\line(1,0){80}} \put(50,25){\vector(1,0){80}}
\put(90,25){\line(0,1){40}}

\put(50,75){\makebox(0,0){\small{$(2;3,1)$}}}
\put(50,15){\makebox(0,0){\small{$(1;3,1)$}}}
\put(90,75){\makebox(0,0){\small{$(2;6,2)$}}}
\put(90,15){\makebox(0,0){\small{$(1;6,2)$}}}
\put(130,75){\makebox(0,0){\small{$(2;0,1)$}}}
\put(130,15){\makebox(0,0){\small{$(1;0,1)$}}}
\end{picture}

\end{example}

Of course, for the general family, we need a more conceptual and uniform procedure.
In general, when determining  $\gc$, the construction of an
embedded resolution $r$, as in \ref{construct},  is not always simple, and it
depends essentially on the choice of the germ $g$. Ideally, for any $f$, it would be nice to
find a germ  $g$ such that the pair $(f,g)$ would admit a resolution $r$ which reflects
 {\em only} the geometry of $f$, e.g. it is a `canonical', or `minimal' embedded resolution of
$V_f$. For example, in the homogeneous case, resolving $f$ we  automatically get
a resolution which is good for the pair $(f,g)$ as well, provided that
$g$ is a generic linear form.  But, in general,   `canonical' resolutions
 attached to $f$ by some geometric constructions used to resolve hypersurfaces do not have the
  extra property that they resolve a well--chosen $g$ as well (or, at least,
  the authors do not know such a general statement). Usually, the strict transform
of $g$ may still have `bad contacts' with the created exceptional divisors even if we take
for $g$ the generic linear form.

Nevertheless, for double
 suspensions $f=f'+z^2$, if one constructs a `canonical'
 resolution using the classical Jung construction fitting with the
shape of $f$ (that is,  based on the projection onto the $(x,y)$--plane, similarly
 as the methods described in \ref{ss:b}),
the obtained embedded resolution will be compatible with
 $g$ too, provided that we take for $g$ a generic linear form.
We expect that a similar phenomenon is valid for arbitrary suspensions as well.

 Since the embedded resolution of double suspensions is already present in the literature  \cite{BMN},
this case can be exemplified without too much extra work. Nevertheless, the computations are not trivial,
  and their verification will require some effort from the reader, and familiarity with
  \cite{BMN}. In the sequel we present the main steps
needed to understand the procedure, we provide some examples, and we let the reader  explore his/her favorite example.  %To the general case ($d>2$) we will return back later in a forthcoming work.

\vspace{2mm}

We prefer to write $g$ as $g'(x,y)+z$, where $g'$ is a generic linear form (with respect to $f'$)
in variables $(x,y)$.

The embedded resolution of $V_{fg}\subset (\bfc^3,0)$
is constructed in several steps as in \cite{BMN}. Although in that  article
 $f'$ is isolated, the same procedure works in our case as well.
 We summarize the steps  in the following diagram:

\begin{picture}(100,90)(-40,20)
\put(20,80){\makebox(0,0){$\tilde{X}$}}
\put(100,80){\makebox(0,0){$X$}}
\put(180,80){\makebox(0,0){$U^3$}}
\put(100,30){\makebox(0,0){$Z$}}
\put(180,30){\makebox(0,0){$U^2$}}
\put(40,80){\vector(1,0){40}}
\put(60,88){\makebox(0,0){$\psi$}}
\put(120,80){\vector(1,0){40}}
\put(140,88){\makebox(0,0){$\phi'$}}
\put(100,70){\vector(0,-1){30}}
\put(90,55){\makebox(0,0){$p'$}}
\put(180,70){\vector(0,-1){30}}
\put(190,55){\makebox(0,0){$p$}}
\put(120,30){\vector(1,0){40}}
\put(140,38){\makebox(0,0){$\phi $}}
\put(200,79){\makebox(0,0){$\supset$}}
\put(215,78){\makebox(0,0){$V_{fg}$}}
\put(200,29){\makebox(0,0){$\supset$}}
\put(215,28){\makebox(0,0){$V_{f'g'}$}}
\end{picture}

\noindent where

\begin{enumerate}
\item  $U^3$ is a small representative of $(\bfc^3,0)$ and $p:U^3\to U^2$
is induced by the projection $(x,y,z)\to (x,y)$.\vspace{2mm}

\item  $\phi:Z\to U^2$~ is an  embedded resolution of
$(V_{f'g'},0)\subset (\bfc^2,0)$. We attach to
each irreducible  component $D$ of the exceptional divisor and to each strict transform
component two nonnegative integers: the vanishing order
$m(f')$, respectively $m(g')$ of $f'$, respectively $g'$, along that component.

%Note that this means: for a generic point $P$ of the exceptional divisor
%(in particular, it is a point of only one irreducible component) there are
%local coordinates $(u,v)$ and a coordinate neighbourhood $V\subset Z$ of $P$
%such that
%$f\circ\Phi |_V=u^{m(f)}$ and $g'\circ\Phi |_V=u^{m(g')}$. Similarly, for
%$P$, a point at the intersection of two irreducible components with
%multiplicities $m_i(f), m_i(g')$ and $m_j(f), m_j(g')$, in an appropriate
%coordinate neighbourhood in $Z$ one has
%$f\circ\Phi |_V=u^{m_i(f)}v^{m_j(f)}$ and $g'\circ\Phi |_V=u^{m_i(g')}v^{m_j(g')}$.

We take the minimal embedded resolution modified as in  \cite[(3.1)]{BMN}:
 we assume that there are no pairs
of irreducible components $D_{v_1}, D_{v_2}$ with
$(D_{v_1}, D_{v_2})\not=0$ having  both multiplicities $m_{v_1}(f'), m_{v_2}(f')$ odd. This can
always be achieved from the minimal embedded resolution by an  additional blow up
at those intersection points where the condition is not satisfied.\vspace{2mm}

\item  $p':X\to Z$ is the pull-back of $p:U^3\to U^2$ via $\phi$, that is, $X$ is the product
 of $Z$ with the $z$--disc.  %Let $T=(\phi')^{-1}(V_{fg})$ be the total transform of $V_{fg}$ in $X$.
By construction, in some local coordinates $(u,v,z)$ with $p'(u,v,z)=(u,v)$,
any  strict transform component  of
$V_f$ in $X$ has equation  $u^{m_w(f')}+z^2$ above the generic point of an exceptional curve of $Z$, and  $u^{m_w(f')}v^{m_v(f')}+z^2$ above an intersection point. The strict transform
of  $V_g$ is  smooth; its local equations have similar form with the exponent
of $z$ being one. Note that  the contact of these two
 spaces along $z=0$ is rather non--trivial.\vspace{2mm}

%As $p\circ (hg)^{-1}=(fg')^{-1}$, above an appropriate coordinate
%neighbourhood $V$ of a generic point $P$ of the exceptional
%divisor in $Z$, $(p')^{-1}(V)$ admits local coordinates $(u,v,z)$ with
%$p'(u,v,z)=(u,v)$ such that $T\subset (p')^{-1}(V)$ is given by
%$(u^{m(f)}+z^2)(u^{m(g)}+z^2)=0$. Similarly, for $P$, a point at
%the intersection of two irreducible components with
%multiplicities $m_i(f), m_i(g')$ and $m_j(f), m_j(g')$, above an appropriate
%coordinate neighbourhood in $Z$ one has that $T\cap p'^{-1}(V)\subset X$
%is given by
%$(u^{m_i(f)}v^{m_j(f)}+z^2)(u^{m_i(g')}v^{m_j(g')}+z^2)=0$.

\item  $\psi$ is an embedded resolution of
$(\phi')^{-1}(V_f)\subset X$, determined similarly as in \cite[(3.4)]{BMN}. This procedure
constructs  a `tower' of exceptional ruled surfaces over each exceptional divisor of $Z$.
The algorithm of \cite{BMN} constructs over each divisor of $Z$ a `minimal' tower, and the
towers above divisors with even multiplicities are constructed first. Both these two
conventions will be released now in order to get a resolution for the {\it pair} $V_f\cup V_g$.
\end{enumerate}

The composed map $\phi'\circ \psi:\tilde{X}\to U^3$ serves for the modification $r$.
Nevertheless, we wish to say here a word of warning. Usually we require that $r$ is an isomorphism
above the complement of the singular locus of $V_f$.
As it is explained in the Introduction of
\cite{BMN}, or can be verified using the  definitions, $\phi'\circ \psi$ fails to be an isomorphism
above the union of $Sing(V_f)$ with the $z$--axis (because of the blow ups of the infinite near
$z$--axis during the modification $\phi'$, as pull back of $\phi$).
 For example, the
Milnor fiber of $f$ is not lifted diffeomorphically under this modification:
it is blown up at its intersection points with the $z$--axis.
Nevertheless, as the boundary $\partial F_{\epsilon,\delta}$ has no intersection points with
the $z$--axis provided that $\delta\ll\epsilon$, this modification serves in this procedure as a
genuine embedded resolution.

\vspace{2mm}

%The above geometric construction is made  precise in \cite{BMN}, constructing above each
%exceptional divisor of $\phi$ a tower of non--minimal ruled surfaces. In this construction one can
%follow the strict transform of $V_g$ too.
% The point is that if we follow the convention of \cite[(3.4)]{BMN}, namely that
% one constructs {\it first} the  towers above irreducible curves of the exceptional divisor
%in $Z$ with {\it even} multiplicity $m(f)$, then $r$  automatically provides
%an embedded resolution of the total transform of
%$V_{fg}$ as well.

%\smallskip The above strategy leads to a combinatorial algorithm when in
%Step 1. one determines the embedded resolution graph $\Gamma_{fg'}$
%of the plane curve
%singularity $(fg')^{-1}(0)\subset(\bfc^2,0)$ and in Step 2. (as a result of
%analysing $p\circ r$) the
%universal graph $\gc$ of $(h,g)$ appears ``above''
%$\Gamma_{fg'}$.

%\smallskip

The above strategy leads to a combinatorial algorithm in two steps.
In  {\it Step 1} one determines the embedded resolution graph $\G(f',g'):=\Gamma(\bfc^2,f'g')$ (with the
additional property mentioned in (2) above), but now
weighted with both multiplicities $(m_v(f'),m_v(g'))$ of $f'$ and $g'$. In {\it Step 2}
we determine  the `towers' similarly as in   \cite{BMN},
eventually constructed in a different order, or with  extra
blow ups.   In the concrete examples below we will indicate the differences with \cite{BMN}. Then,
 one reads from the `towers'   the graph $\gc$ of $(f,g)$. This
appears as a `modified cover' of  $\Gamma(f',g')$.

\vspace{2mm}

The following facts  might be helpful in the construction of the
above `modified cover' of graphs.

The non--arrowhead vertices of $\gce$ cover the non-arrowhead
vertices of $\G(f',g')$ as follows. Fix $w\in \calw(\G(\bfc^2,f'g'))$. If $m_w(f')$
is even, {\it and} all the $f'$--multiplicities of the adjacent vertices are even, then $w$
is covered by 2 non-arrowhead vertices. In all other cases it is covered by only one
vertex.  This structure follows closely the structure and the position of the strict transform of $f$
in the resolution towers as it is described in section 3.5 of \cite{BMN}.

The $(0,1)$ arrowheads of $\G(f',g')$ are covered by $(1;0,1)$--arrowheads of $\gc$. If
a non-arrowhead $w$  supports such an arrowhead in $\G(f',g')$, and $m_w(f')$ is even, then
it is covered by two arrowheads, otherwise only by one. Geometrically this is the only place where
the strict transform of $g$ plays a role.

The   two properties above describe  behaviors common with  $\Z_2$ graph coverings. \ix{graph!covering}

Next, above the arrowheads of $\G(f',g')$ of type $(1,0)$ we put nothing.
The graph $\gck$ appears above the arrowheads of $\G(f',g')$ of type $(m,0)$, $m>1$.
Fix such an arrowhead and the corresponding strict transform $St_a(f')$ of $f'$, which is supported
by the exceptional component $E_w$. Then the entire tower above $St_a(f')$ is in  $\de_c$
and all the intersection curves with the tower above $E_w$ enter in $\C$. Therefore, above the arrowheads of
 $\G(f',g')$ of type $(m,0)$ with $m>1$  a lot of curves of $\C$ may appear,
(and this part does not behave like a cyclic covering).

\vspace{2mm}
In certain cases, the genera of the projective irreducible components $C$ of
the special curve configuration $\C$ may be difficult to determine from the {\it local }
equations of $C$.  However,
if the link of the normalization of $V_f$ is a rational homology sphere, then we can be sure that
in $\gc$ all the genus decorations are zero, cf. \ref{m1} and \ref{gkl}. This is the case
in all the examples worked out in this section. \ix{curve configuration $\C$}

\vspace{2mm}

Although in all the cases considered below the graph  $\Gamma(f',g')$  is easy to determine,
we provide them in order to emphasize the covering nature of the procedure.

\begin{example}\labelpar{ex:AB}
 Assume that  $f'(x,y)=x^ay^b$, where $a>0$ and $b>0$.
 The normalization is a Hirzebruch--Jung singularity, hence its link is a rational homology sphere.

We will distinguish three cases depending on the parity of $a$ and $b$.

\vspace{3mm}

\noindent {\bf Case 1.} \  If both $a$ and $b$ are even, that is
$f'(x,y)=x^{2n}y^{2m}$, and $g'=x+y$, then by {\it Step 1} the dual graph of the
 minimal good embedded
resolution of  $V_{f'g'}$, weighted by the vanishing orders of both $f'$ and $g'$,  is

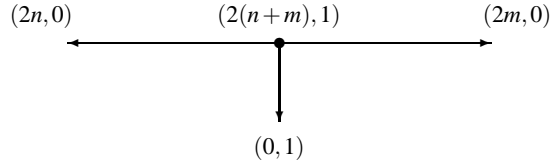
\vspace{2mm}

\begin{picture}(300,86)(-40,30)
\put(100,80){\circle*{4}}
\put(100,90){\makebox(0,0){\small{$(2(n+m),1)$}}}

\put(100,80){\vector(1,0){80}}
\put(190,90){\makebox(0,0){\small{$(2m,0)$}}}

\put(100,80){\vector(-1,0){80}}
\put(10,90){\makebox(0,0){\small{$(2n,0)$}}}

\put(100,80){\vector(0,-1){30}}
\put(100,40){\makebox(0,0){\small{$(0,1)$}}}
\end{picture}

\vspace{2mm}

\noindent Here {\it Step 2} follows closely \cite{BMN} with the following additional information:
the tower above the exceptional divisor weighted
$(2(n+m),1)$ was constructed first, then the towers above the strict transforms
of $f'$.
% were constructed next. The exceptional ruled surfaces
%in these towers are all type $c$.
 A possible graph $\gc$ of $(f,g)$ is:

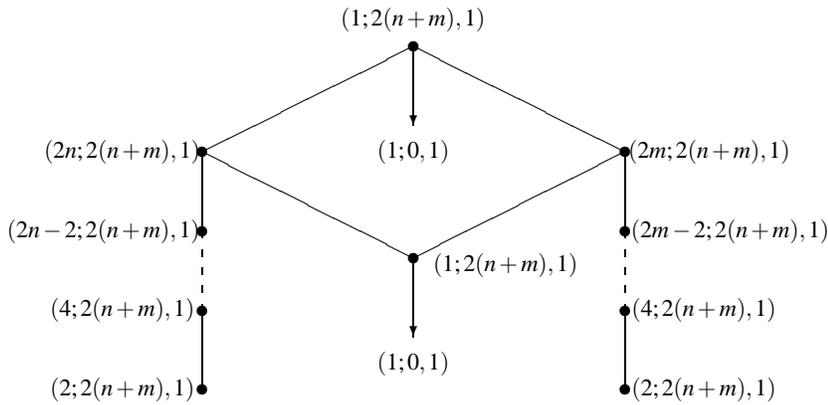
\vspace{2mm}

\begin{picture}(300,180)(-40,50)
% bal oszlop csucsai + sulyok
\put(20,70){\circle*{4}}
\put(20,100){\circle*{4}}
\put(20,130){\circle*{4}}
\put(20,160){\circle*{4}}

\put(-10,70){\makebox(0,0){\small{$(2;2(n+m),1)$}}}
\put(-10,100){\makebox(0,0){\small{$(4;2(n+m),1)$}}}
\put(-17,130){\makebox(0,0){\small{$(2n-2;2(n+m),1)$}}}
\put(-10,160){\makebox(0,0){\small{$(2n;2(n+m),1)$}}}

% jobb oszlop csucsai + sulyok
\put(180,70){\circle*{4}}
\put(180,100){\circle*{4}}
\put(180,130){\circle*{4}}
\put(180,160){\circle*{4}}

\put(210,70){\makebox(0,0){\small{$(2;2(n+m),1)$}}}
\put(210,100){\makebox(0,0){\small{$(4;2(n+m),1)$}}}
\put(220,130){\makebox(0,0){\small{$(2m-2;2(n+m),1)$}}}
\put(212,160){\makebox(0,0){\small{$(2m;2(n+m),1)$}}}

% kozepso sor csucsai +sulyok
\put(100,120){\circle*{4}}
\put(100,200){\circle*{4}}

%\put(130,120){\makebox(0,0){\tiny{$(1;2(n+m),1)$}}}
\put(135,117){\makebox(0,0){\small{$(1;2(n+m),1)$}}}
\put(100,210){\makebox(0,0){\small{$(1;2(n+m),1)$}}}

%%%%%%%%%%%%%%%%%%%%%%%%%%%%%%%%%%%%%%%%%%%%
%osszekoto elek+ sulyok
\put(20,70){\line(0,1){30}}
\put(180,70){\line(0,1){30}}

\put(20,130){\line(0,1){30}}
\put(180,130){\line(0,1){30}}

\dashline[3]{3}(20,100)(20,130)
\dashline[3]{3}(180,100)(180,130)

\put(100,200){\line(2,-1){80}}
\put(100,200){\line(-2,-1){80}}
\put(100,120){\line(2,1){80}}
\put(100,120){\line(-2,1){80}}

%nyilak + sulyaik
\put(100,120){\vector(0,-1){30}}
\put(100,200){\vector(0,-1){30}}
\put(100,80){\makebox(0,0){\small{$(1;0,1)$}}}
\put(100,160){\makebox(0,0){\small{$(1;0,1)$}}}

\end{picture}

\vspace{2mm}
\noindent {\bf Case 2.}
 \ If $a$ is even and $b$ is odd, that is
$f'(x,y)=x^{2n}y^{2m+1}$,  and $g'=x+y$, then by {\it Step 1}  one gets the graph:

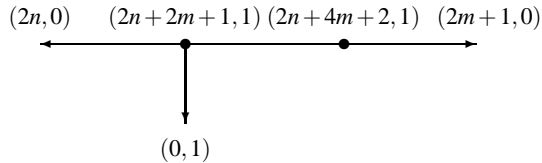
\vspace{2mm}

\begin{picture}(300,86)(-35,25)
\put(85,80){\circle*{4}}
\put(145,80){\circle*{4}}
\put(85,90){\makebox(0,0){\small{$(2n+2m+1,1)$}}}
\put(145,90){\makebox(0,0){\small{$(2n+4m+2,1)$}}}

\put(85,80){\line(1,0){60}}

\put(85,80){\vector(-1,0){55}}
\put(30,90){\makebox(0,0){\small{$(2n,0)$}}}

\put(145,80){\vector(1,0){50}}
\put(200,90){\makebox(0,0){\small{$(2m+1,0)$}}}

\put(85,80){\vector(0,-1){30}}
\put(85,40){\makebox(0,0){\small{$(0,1)$}}}

\end{picture}

\vspace{2mm}

\noindent Here the assumption of not having adjacent irreducible components with
both multiplicities $m_i(f), m_j(f)$ odd  is taken into account.

By {\it Step 2},   a possible universal graph $\gc$
of $(h,g)$ is:

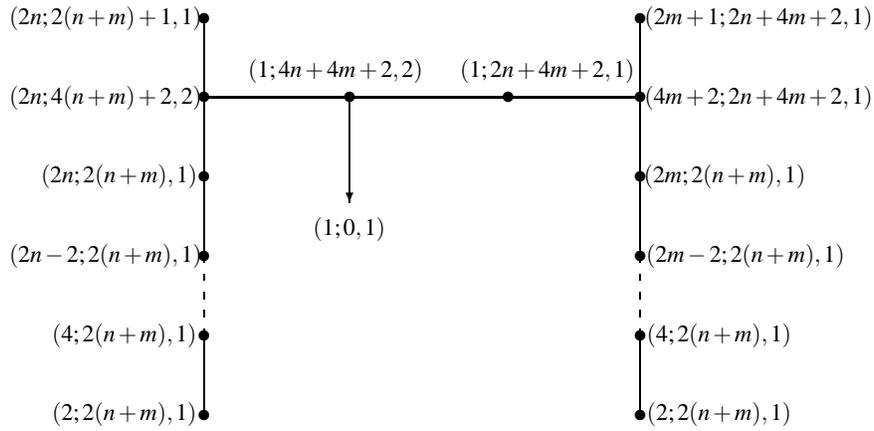
\begin{picture}(300,200)(-35,50) %EZT MAR LEKICSINYITETTEM!
% bal oszlop csucsai + sulyok
\put(30,80){\circle*{4}} \put(30,110){\circle*{4}}
\put(30,140){\circle*{4}} \put(30,170){\circle*{4}}
\put(30,200){\circle*{4}} \put(30,230){\circle*{4}}

\put(0,80){\makebox(0,0){\small{$(2;2(n+m),1)$}}}
\put(0,110){\makebox(0,0){\small{$(4;2(n+m),1)$}}}
\put(-7,140){\makebox(0,0){\small{$\tiny(2n-2;2(n+m),1)$}}}
\put(-2,170){\makebox(0,0){\small{$(2n;2(n+m),1)$}}}
\put(-7,200){\makebox(0,0){\small{$(2n;4(n+m)+2,2)$}}}
\put(-7,230){\makebox(0,0){\small{$(2n;2(n+m)+1,1)$}}}

% jobb oszlop csucsai + sulyok
\put(195,80){\circle*{4}} \put(195,110){\circle*{4}}
\put(195,140){\circle*{4}} \put(195,170){\circle*{4}}
\put(195,200){\circle*{4}} \put(195,230){\circle*{4}}

\put(225,80){\makebox(0,0){\small{$(2;2(n+m),1)$}}}
\put(225,110){\makebox(0,0){\small{$(4;2(n+m),1)$}}}
\put(235,140){\makebox(0,0){\small{$\tiny(2m-2;2(n+m),1)$}}}
\put(227,170){\makebox(0,0){\small{$(2m;2(n+m),1)$}}}
\put(240,200){\makebox(0,0){\small{$(4m+2;2n+4m+2,1)$}}}
\put(240,230){\makebox(0,0){\small{$(2m+1;2n+4m+2,1)$}}}

% kozepso sor csucsai + sulyok
\put(85,200){\circle*{4}} \put(145,200){\circle*{4}}

\put(80,210){\makebox(0,0){\small{$(1;4n+4m+2,2)$}}}
\put(160,210){\makebox(0,0){\small{$(1;2n+4m+2,1)$}}}

%osszekoto elek+ sulyok
\put(30,80){\line(0,1){30}}
 \put(195,80){\line(0,1){30}}

\dashline[3]{3}(30,110)(30,140) \dashline[3]{3}(195,110)(195,140)

\put(30,140){\line(0,1){90}} \put(195,140){\line(0,1){90}}

\put(30,200){\line(1,0){165}}

%nyilak + sulyaik
\put(85,200){\vector(0,-1){40}}
\put(85,150){\makebox(0,0){\small{$(1;0,1)$}}}

\end{picture}

Here single  blow-ups above the exceptional curves weighted $(2n+4m+2,1)$
first, then
$(2n+2m+1,1)$ were used to ensure the strict transform of $g$ to be in
transverse position with respect to the exceptional divisors appearing later.
Then towers were constructed in the following order:
first above the exceptional curve weighted $(2n+4m+2,1)$
then $(2n+2m+1,1)$, finally the strict transforms of $f'$.\\

\noindent {\bf Case 3.} \  Finally, if both $a$ and $b$ are odd, that is
$f'(x,y)=x^{2n+1}y^{2m+1}$,  and $g'=x+y$, then {\it Step 1} produces the graph:

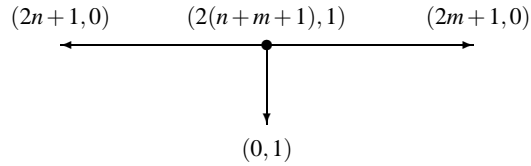
\vspace{2mm}

\begin{picture}(300,90)(-55,20)
\put(100,80){\circle*{4}}
\put(100,90){\makebox(0,0){\small{$(2(n+m+1),1)$}}}

\put(100,80){\vector(1,0){78}}
\put(180,90){\makebox(0,0){\small{$(2m+1,0)$}}}

\put(100,80){\vector(-1,0){78}}
\put(22,90){\makebox(0,0){\small{$(2n+1,0)$}}}

\put(100,80){\vector(0,-1){30}}
\put(100,40){\makebox(0,0){\small{$(0,1)$}}}

\end{picture}

\noindent While a  possible  graph $\gc$ is:

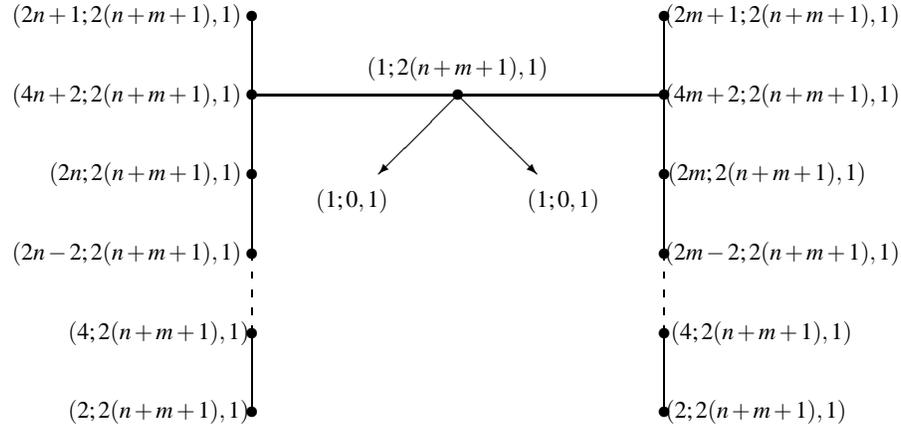
\begin{picture}(300,200)(-55,65)
% bal oszlop csucsai + sulyok
\put(22,80){\circle*{4}}
\put(22,110){\circle*{4}}
\put(22,140){\circle*{4}}
\put(22,170){\circle*{4}}
\put(22,200){\circle*{4}}
\put(22,230){\circle*{4}}

\put(-13,80){\makebox(0,0){\small{$(2;2(n+m+1),1)$}}}
\put(-13,110){\makebox(0,0){\small{$(4;2(n+m+1),1)$}}}
\put(-25,140){\makebox(0,0){\small{$(2n-2;2(n+m+1),1)$}}}
\put(-18,170){\makebox(0,0){\small{$(2n;2(n+m+1),1)$}}}
\put(-25,200){\makebox(0,0){\small{$(4n+2;2(n+m+1),1)$}}}
\put(-25,230){\makebox(0,0){\small{$(2n+1;2(n+m+1),1)$}}}

% jobb oszlop csucsai + sulyok
\put(178,80){\circle*{4}}
\put(178,110){\circle*{4}}
\put(178,140){\circle*{4}}
\put(178,170){\circle*{4}}
\put(178,200){\circle*{4}}
\put(178,230){\circle*{4}}

\put(213,80){\makebox(0,0){\small{$(2;2(n+m+1),1)$}}}
\put(215,110){\makebox(0,0){\small{$(4;2(n+m+1),1)$}}}
\put(223,140){\makebox(0,0){\small{$(2m-2;2(n+m+1),1)$}}}
\put(217,170){\makebox(0,0){\small{$(2m;2(n+m+1),1)$}}}
\put(223,200){\makebox(0,0){\small{$(4m+2;2(n+m+1),1)$}}}
\put(223,230){\makebox(0,0){\small{$(2m+1;2(n+m+1),1)$}}}

% kozepso sor csucsai + sulyok
\put(100,200){\circle*{4}}

\put(100,210){\makebox(0,0){\small{$(1;2(n+m+1),1)$}}}

%osszekoto elek+ sulyok
\put(22,80){\line(0,1){30}}
\put(178,80){\line(0,1){30}}

\dashline[3]{3}(22,110)(22,140)
\dashline[3]{3}(178,110)(178,140)

\put(22,140){\line(0,1){90}}
\put(178,140){\line(0,1){90}}

\put(22,200){\line(1,0){156}}

%nyilak + sulyaik
\put(100,200){\vector(1,-1){30}}
\put(100,200){\vector(-1,-1){30}}

\put(140,160){\makebox(0,0){\small{$(1;0,1)$}}}
\put(60,160){\makebox(0,0){\small{$(1;0,1)$}}}

\end{picture}

Here, a tower above the exceptional curve was constructed first, then
towers above the strict transforms of $f'$.

\end{example}

\begin{example}\labelpar{a2double} Consider the infinite
family $T_{a,2,\infty}$ given by the local equation
$f(x,y,z)=x^a+y^2+xyz$.

If $a=2$ then by a change of coordinates, $f$ can be rewritten as
$f(x,y,z)=x^2+y^2$, a case already treated in \ref{cyl}.
Therefore, in this  subsection  we assume that  $a\geq 3$.
In this case, again, by completing the square and renaming variables,
$f$ can be brought to the form
$f(x,y,z)=x^2(x^{a-2}+y^2)+z^2$. In particular, the previous method can be used with
$g(x,y,z)=x+y+z$.

The singular locus is $\Sigma_f=\{x=y=0\}$
with transversal type $A_1$. Dividing the equation by $x^2$ and taking $t:=y/x$,
we get that $t^2+zt+x^{a-2}=0$, hence $t$ is in the normalization of the local ring, and the
normalization is a hypersurface singularity of type $A_{a-3}$. In particular, the link of the
normalization is a rational homology sphere.

\vspace{2mm}

First, we assume that $a$ is odd, that is $a=2k+3$.

\vspace{2mm}

\noindent {\it Step 1:}  For  $k\geq 1$ the  graph $\G(f',g')$ for
 $f'=x^2(x^{a-2}+y^2)$ and $g'=x+y$  is

%\noindent
%jjjjjjjjj
%jjjjjj jjjjjjj jjjjjjjj jjjjjjjjjj jjj jj jjjjjj jjjjjjj jjjjjj jjjjjjj jjjjjj jjj
%jjjj jjjj jjjjjjjjj jjjjjj jjjjjjj jjjjjjjjjjj jjjjj jjjj jjj jjjjjjjjjjjjjj

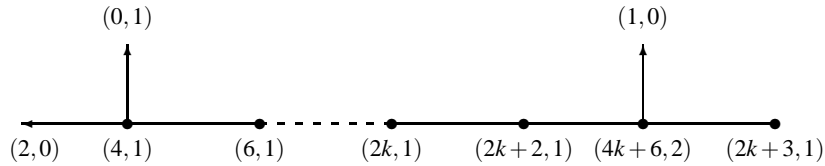
\begin{picture}(300,90)(-10,-10)
\put(40,20){\circle*{4}} \put(90,20){\circle*{4}}
\put(140,20){\circle*{4}} \put(190,20){\circle*{4}}
\put(235,20){\circle*{4}} \put(285,20){\circle*{4}}

\put(40,10){\makebox(0,0){\small{$(4,1)$}}}
\put(90,10){\makebox(0,0){\small{$(6,1)$}}}
\put(140,10){\makebox(0,0){\small{$(2k,1)$}}}
\put(190,10){\makebox(0,0){\small{$(2k+2,1)$}}}
\put(235,10){\makebox(0,0){\small{$(4k+6,2)$}}}
\put(285,10){\makebox(0,0){\small{$(2k+3,1)$}}}

\put(40,20){\vector(-1,0){40}} \put(5,10){\makebox(0,0){\small{$(2,0)$}}}

\put(40,20){\vector(0,1){30}}
\put(40,60){\makebox(0,0){\small{$(0,1)$}}}

\put(235,20){\vector(0,1){30}}
\put(235,60){\makebox(0,0){\small{$(1,0)$}}}

\put(40,20){\line(1,0){50}} \dashline[3]{3}(90,20)(140,20)
\put(140,20){\line(1,0){145}}

\end{picture}

\noindent {\it Step 2:}  A possible  $\gc$ for  $(f,g)$ is:

%\newpage sok szoveg szoveg szoveg szoveg szoveg
%sok szoveg szoveg szoveg szoveg szoveg
%sok szoveg szoveg szoveg szoveg szoveg
%sok szoveg szoveg szoveg szoveg szoveg
%sok szoveg szoveg szoveg szoveg szoveg
%sok szoveg szoveg szoveg szoveg szoveg

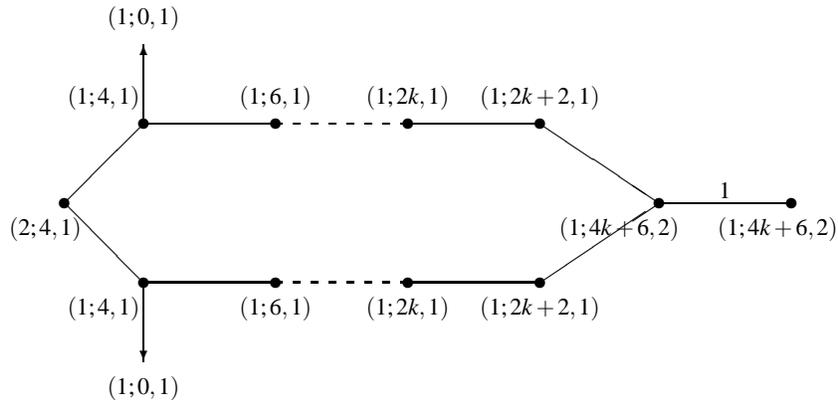
\vspace{2mm}
\begin{picture}(300,170)(-10,-30)
\put(10,50){\circle*{4}}

\put(40,20){\circle*{4}}\put(40,80){\circle*{4}}
\put(90,20){\circle*{4}} \put(90,80){\circle*{4}}
\put(140,20){\circle*{4}} \put(140,80){\circle*{4}}
\put(190,20){\circle*{4}} \put(190,80){\circle*{4}}
\put(235,50){\circle*{4}} \put(285,50){\circle*{4}}

\put(3,40){\makebox(0,0){\small{$(2;4,1)$}}}

\put(25,10){\makebox(0,0){\small{$(1;4,1)$}}}
\put(25,90){\makebox(0,0){\small{$(1;4,1)$}}}
\put(90,10){\makebox(0,0){\small{$(1;6,1)$}}}
\put(90,90){\makebox(0,0){\small{$(1;6,1)$}}}
\put(140,10){\makebox(0,0){\small{$(1;2k,1)$}}}
\put(140,90){\makebox(0,0){\small{$(1;2k,1)$}}}
\put(190,10){\makebox(0,0){\small{$(1;2k+2,1)$}}}
\put(190,90){\makebox(0,0){\small{$(1;2k+2,1)$}}}

\put(220,40){\makebox(0,0){\small{$(1;4k+6,2)$}}}
\put(280,40){\makebox(0,0){\small{$(1;4k+6,2)$}}}
\put(260,55){\makebox(0,0){{\small$1$}}}

\put(40,20){\vector(0,-1){30}}
\put(40,-20){\makebox(0,0){\small{$(1;0,1)$}}}

\put(40,80){\vector(0,1){30}}
\put(40,120){\makebox(0,0){\small{$(1;0,1)$}}}

\put(40,20){\line(1,0){50}} \put(40,80){\line(1,0){50}}
\dashline[3]{3}(90,20)(140,20) \dashline[3]{3}(90,80)(140,80)
\put(140,20){\line(1,0){50}} \put(140,80){\line(1,0){50}}

\put(10,50){\line(1,1){30}} \put(10,50){\line(1,-1){30}}
\put(190,20){\line(3,2){45}} \put(190,80){\line(3,-2){45}}

\put(235,50){\line(1,0){50}}

\end{picture}

\noindent For $k=0$ (i.e., for $a=3$), the first graph is

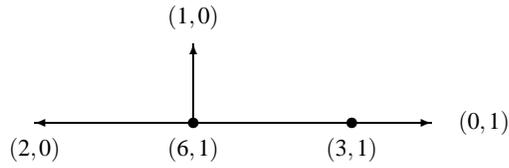
\begin{picture}(300,90)(-100,-10)
% ket csucs
\put(40,20){\circle*{4}}
\put(100,20){\circle*{4}}
\put(40,10){\makebox(0,0){\small{$(6,1)$}}}
\put(100,10){\makebox(0,0){\small{$(3,1)$}}}

\put(40,20){\vector(-1,0){60}}
\put(-20,10){\makebox(0,0){\small{$(2,0)$}}}

\put(40,20){\vector(0,1){30}}
\put(40,60){\makebox(0,0){\small{$(1,0)$}}}

\put(100,20){\vector(1,0){30}}
\put(150,20){\makebox(0,0){\small{$(0,1)$}}}

\put(40,20){\line(1,0){60}}

\end{picture}

\noindent which is `covered' by the graph $\gc$:

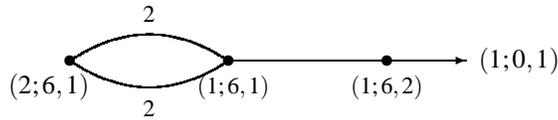
\begin{picture}(300,90)(-100,0)

\put(-20,50){\circle*{4}}
\put(-28,40){\makebox(0,0){$(2;6,1)$}}

\put(40,50){\circle*{4}}
\put(42,40){\makebox(0,0){\small$(1;6,1)$}}

\put(100,50){\circle*{4}}
\put(100,40){\makebox(0,0){\small$(1;6,2)$}}

\put(100,50){\vector(1,0){30}}
\put(150,50){\makebox(0,0){$(1;0,1)$}}

\put(40,50){\line(1,0){60}}

% osszekoto ivek + sulyok:
\qbezier(-20,50)(10,70)(40,50)
\qbezier(-20,50)(10,30)(40,50)
\put(10,68){\makebox(0,0){\small 2}}
\put(10,32){\makebox(0,0){\small 2}}

\end{picture}

\noindent (For an alternative universal graph with different $g$, see \ref{xyz3}.)

\vspace{2mm}

\noindent In case $a=2k$, $k\geq 3$, {\it Step 1} provides

%\newpage szoveg szoveg szoveg szoveg szoveg szoveg
%szoveg szoveg szoveg szoveg szoveg szoveg szoveg szoveg szoveg szoveg szoveg
%szoveg szoveg szoveg szoveg szoveg szoveg szoveg szoveg szoveg szoveg szoveg
%szoveg szoveg szoveg szoveg szoveg szoveg szoveg szoveg

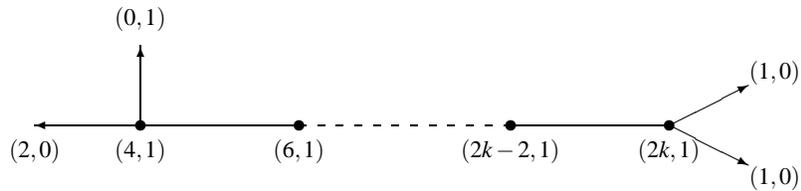
\begin{picture}(300,80)(-10,0)
\put(40,20){\circle*{4}}
\put(100,20){\circle*{4}}
\put(180,20){\circle*{4}}
\put(240,20){\circle*{4}}

\put(40,10){\makebox(0,0){\small{$(4,1)$}}}
\put(100,10){\makebox(0,0){\small{$(6,1)$}}}
\put(180,10){\makebox(0,0){\small{$(2k-2,1)$}}}
\put(240,10){\makebox(0,0){\small{$(2k,1)$}}}

\put(40,20){\vector(-1,0){40}}
\put(0,10){\makebox(0,0){\small{$(2,0)$}}}

\put(40,20){\vector(0,1){30}}
\put(40,60){\makebox(0,0){\small{$(0,1)$}}}

\put(240,20){\vector(2,1){30}}
\put(280,40){\makebox(0,0){\small{$(1,0)$}}}
\put(240,20){\vector(2,-1){30}}
\put(280,0){\makebox(0,0){\small{$(1,0)$}}}

\put(40,20){\line(1,0){60}}
\dashline[3]{3}(100,20)(180,20)
\put(180,20){\line(1,0){60}}

\end{picture}

\vspace{6mm}

\noindent
and $\gc$ is

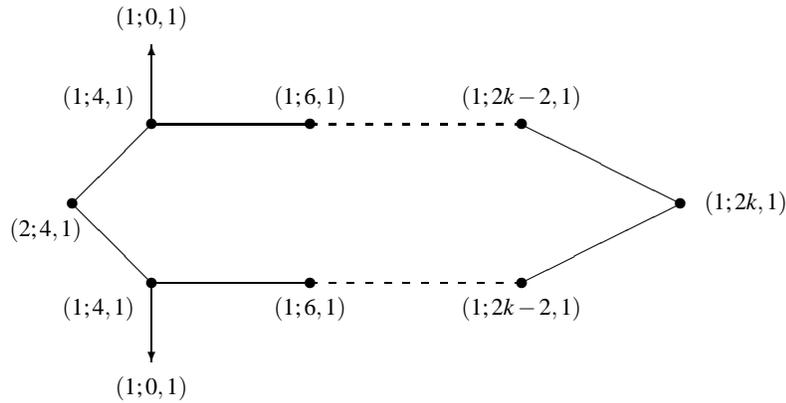
\vspace{2mm}

\begin{picture}(300,130)(-10,0)

\put(10,50){\circle*{4}}
\put(0,40){\makebox(0,0){\small{$(2;4,1)$}}}

\put(40,20){\circle*{4}}
\put(40,80){\circle*{4}}
\put(20,10){\makebox(0,0){\small{$(1;4,1)$}}}
\put(20,90){\makebox(0,0){\small{$(1;4,1)$}}}

\put(100,20){\circle*{4}}
\put(100,80){\circle*{4}}
\put(100,10){\makebox(0,0){\small{$(1;6,1)$}}}
\put(100,90){\makebox(0,0){\small{$(1;6,1)$}}}

\put(180,20){\circle*{4}}
\put(180,80){\circle*{4}}
\put(180,10){\makebox(0,0){\small{$(1;2k-2,1)$}}}
\put(180,90){\makebox(0,0){\small{$(1;2k-2,1)$}}}

\put(240,50){\circle*{4}}
\put(265,50){\makebox(0,0){\small{$(1;2k,1)$}}}

\put(40,20){\vector(0,-1){30}}
\put(40,-20){\makebox(0,0){\small{$(1;0,1)$}}}

\put(40,80){\vector(0,1){30}}
\put(40,120){\makebox(0,0){\small{$(1;0,1)$}}}

\put(40,20){\line(1,0){60}}
\put(40,80){\line(1,0){60}}
\dashline[3]{3}(100,20)(180,20)
\dashline[3]{3}(100,80)(180,80)

\put(10,50){\line(1,1){30}}
\put(10,50){\line(1,-1){30}}
\put(180,20){\line(2,1){60}}
\put(180,80){\line(2,-1){60}}

\end{picture}

\vspace{1.5cm}

\noindent
For $k=2$ (that is $a=4$) {\it Step 1} gives

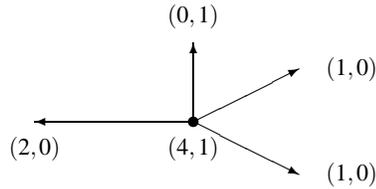
\begin{picture}(300,80)(-40,0)
% ket csucs

\put(100,20){\circle*{4}}
\put(100,10){\makebox(0,0){\small{$(4,1)$}}}

\put(100,20){\vector(-1,0){60}}
\put(40,10){\makebox(0,0){\small{$(2,0)$}}}

\put(100,20){\vector(0,1){30}}
\put(100,60){\makebox(0,0){\small{$(0,1)$}}}

\put(100,20){\vector(2,1){40}}
\put(160,40){\makebox(0,0){\small{$(1,0)$}}}
\put(100,20){\vector(2,-1){40}}
\put(160,0){\makebox(0,0){\small{$(1,0)$}}}

\end{picture}

\vspace{5mm}

\noindent while {\it Step 2} provides

\vspace{5mm}

\begin{picture}(300,80)(-40,0)

\put(40,50){\circle*{4}}
\put(35,40){\makebox(0,0){\small{$(2;4,1)$}}}

\put(100,50){\circle*{4}}
\put(120,50){\makebox(0,0){\small{$(1;4,1)$}}}

\put(100,50){\vector(0,1){30}}
\put(120,80){\makebox(0,0){\small{$(1;0,1)$}}}
\put(100,50){\vector(0,-1){30}}
\put(120,20){\makebox(0,0){\small{$(1;0,1)$}}}

% osszekoto ivek + sulyok:
\qbezier(40,50)(70,70)(100,50)
\qbezier(40,50)(70,30)(100,50)
\put(70,68){\makebox(0,0){\small 2}}
\put(70,32){\makebox(0,0){\small 2}}

\end{picture}

\end{example}

\begin{example}\labelpar{AnPi} Consider the family
 $f(x,y,z)=x^ay(x^2+y^3)+z^2$ with $a\geq 2$ and $g=x+y+z$.
Again, we need to consider two cases depending on parity of $a$.
The equation of the normalization is
$xy(x^2+y^3)+z^2=0$ if $a$ is odd, and it is
$y(x^2+y^3)+z^2=0$  if $a$ is even. In both cases one can determine the
plumbing graph of the link of the normalization using the algorithm \ref{ss:b}.
In particular, one gets that the link of the normalization is a rational homology sphere.

\vspace{2mm}

When $a$ is even, and taking into account, that no adjacent
vertices can have odd $f'$--multiplicities, {\it Step 1}
 gives for $f'g'$:

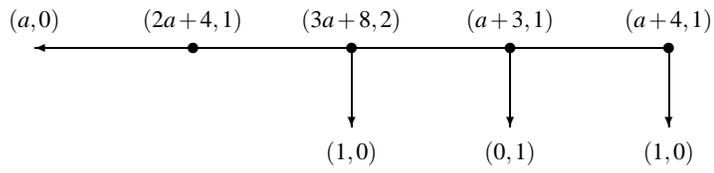
\begin{picture}(300,60)(-70,0)
\put(40,20){\circle*{4}}
\put(100,20){\circle*{4}}
\put(160,20){\circle*{4}}
\put(220,20){\circle*{4}}

\put(40,30){\makebox(0,0){\small{$(2a+4,1)$}}}
\put(100,30){\makebox(0,0){\small{$(3a+8,2)$}}}
\put(160,30){\makebox(0,0){\small{$(a+3,1)$}}}
\put(220,30){\makebox(0,0){\small{$(a+4,1)$}}}

\put(40,20){\vector(-1,0){60}}
\put(-20,30){\makebox(0,0){\small{$(a,0)$}}}

\put(100,20){\vector(0,-1){30}}
\put(100,-20){\makebox(0,0){\small{$(1,0)$}}}

\put(160,20){\vector(0,-1){30}}
\put(160,-20){\makebox(0,0){\small{$(0,1)$}}}

\put(220,20){\vector(0,-1){30}}
\put(220,-20){\makebox(0,0){\small{$(1,0)$}}}

\put(40,20){\line(1,0){180}}
\end{picture}

\vspace{1.5cm}\noindent {\it Step 2} provides

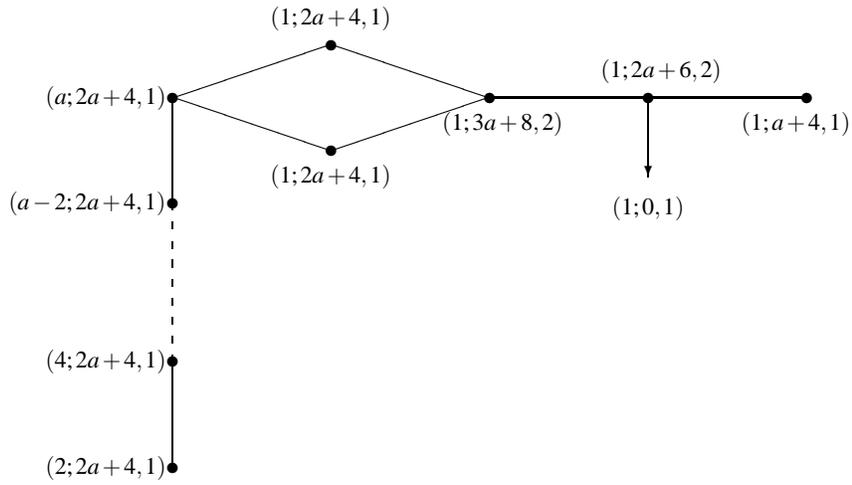
\begin{picture}(300,210)(-70,0)
% bal oszlop csucsai + sulyok
\put(-20,20){\circle*{4}}
\put(-20,60){\circle*{4}}
\put(-20,120){\circle*{4}}
\put(-20,160){\circle*{4}}

\put(-45,20){\makebox(0,0){\small{$(2;2a+4,1)$}}}
\put(-45,60){\makebox(0,0){\small{$(4;2a+4,1)$}}}
\put(-52,120){\makebox(0,0){\small{$(a-2;2a+4,1)$}}}
\put(-45,160){\makebox(0,0){\small{$(a;2a+4,1)$}}}

% hurok "oszlopa"
\put(40,130){\makebox(0,0){\small{$(1;2a+4,1)$}}}
\put(40,140){\circle*{4}}
\put(40,180){\circle*{4}}
\put(40,190){\makebox(0,0){\small{$(1;2a+4,1)$}}}

% kozepso oszlop csucsai + sulyok
\put(100,160){\circle*{4}}
\put(105,150){\makebox(0,0){\small{$(1;3a+8,2)$}}}

% kovetkezo oszlop csucsai + sulyok
\put(160,160){\circle*{4}}
\put(165,170){\makebox(0,0){\small{$(1;2a+6,2)$}}}

% jobbszelso oszlop csucsai + sulyok
\put(220,160){\circle*{4}}
\put(216,150){\makebox(0,0){\small{$(1;a+4,1)$}}}

% elek
\put(-20,20){\line(0,1){40}}
\dashline[3]{3}(-20,60)(-20,120)
\put(-20,120){\line(0,1){40}}

\put(100,160){\line(1,0){120}}

\put(40,140){\line(3,1){60}}
\put(40,140){\line(-3,1){60}}
\put(40,180){\line(3,-1){60}}
\put(40,180){\line(-3,-1){60}}

%  nyilak
%\put(180,160){\vector(1,-3){12}}
\put(160,160){\vector(0,-1){30}}
\put(160,118){\makebox(0,0){\small$(1;0,1)$}}
%\put(198,118){\makebox(0,0){\small$(1;0,1)$}}

\end{picture}

\vspace{2mm}

\noindent
When $a$ is odd, the first graph is:

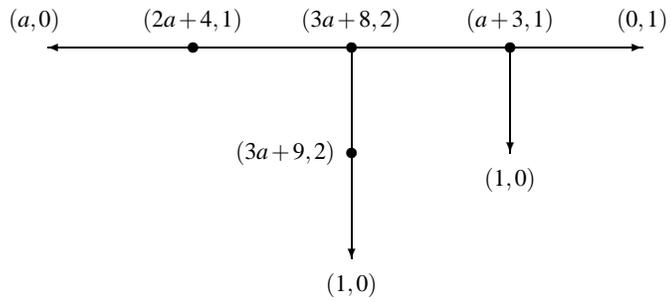
\begin{picture}(300,70)(-70,0)
\put(40,20){\circle*{4}} \put(100,20){\circle*{4}}
\put(160,20){\circle*{4}} \put(100,-20){\circle*{4}}
\put(40,30){\makebox(0,0){\small{$(2a+4,1)$}}}
\put(100,30){\makebox(0,0){\small{$(3a+8,2)$}}}
\put(160,30){\makebox(0,0){\small{$(a+3,1)$}}}
\put(75,-20){\makebox(0,0){\small{$(3a+9,2)$}}}

\put(40,20){\vector(-1,0){55}}
\put(-20,30){\makebox(0,0){\small{$(a,0)$}}}

\put(100,-20){\vector(0,-1){40}}
\put(100,-70){\makebox(0,0){\small{$(1,0)$}}}

\put(160,20){\vector(0,-1){40}}
\put(160,-30){\makebox(0,0){\small{$(1,0)$}}}

\put(160,20){\vector(1,0){50}}
\put(210,30){\makebox(0,0){\small{$(0,1)$}}}

\put(40,20){\line(1,0){140}}
\put(100,20){\line(0,-1){60}}
\end{picture}

\vspace{3cm}\noindent while {\it Step 2} provides

\begin{picture}(300,220)(-70,0)
% bal oszlop csucsai + sulyok
\put(-15,20){\circle*{4}} \put(-15,60){\circle*{4}}
\put(-15,120){\circle*{4}} \put(-15,160){\circle*{4}}
\put(-15,200){\circle*{4}}

\put(-43,20){\makebox(0,0){\small{$(2;2a+4,1)$}}}
\put(-43,60){\makebox(0,0){\small{$(4;2a+4,1)$}}}
\put(-48,120){\makebox(0,0){\small{$(a-1;2a+4,1)$}}}
\put(-43,160){\makebox(0,0){\small{$(2a;2a+4,1)$}}}
\put(-43,200){\makebox(0,0){\small{$(a;2a+4,1)$}}}

% hurok "oszlopa"
 \put(40,160){\circle*{4}}
\put(40,170){\makebox(0,0){\small{$(1;2a+4,1)$}}}

% kozepso oszlop csucsai + sulyok
 \put(100,160){\circle*{4}}
\put(100,170){\makebox(0,0){\small{$(1;6a+16,4)$}}}
 \put(100,120){\circle*{4}}
\put(100,110){\makebox(0,0){\small{$(1;3a+9,2)$}}}

% kovetkezo oszlop csucsai + sulyok
\put(160,160){\circle*{4}}
\put(160,170){\makebox(0,0){\small{$(1;a+3,1)$}}}

% jobbszelso nyilak
\put(160,160){\vector(3,1){50}} \put(160,160){\vector(3,-1){50}}
\put(212,185){\makebox(0,0){\small$(1;0,1)$}}
\put(212,135){\makebox(0,0){\small$(1;0,1)$}}

% elek, nyilak
\put(-15,20){\line(0,1){40}} \dashline[3]{3}(-15,60)(-15,120)
\put(-15,120){\line(0,1){80}}

\put(100,120){\line(0,1){40}}

\put(-15,160){\line(1,0){175}}
\end{picture}

\end{example}

\begin{example}\labelpar{mdndd} The procedure \ref{ex:AB} has a natural generalization for certain
other  suspensions too. For example, consider
$f(x,y,z)=x^{md}y^{nd}+z^d$ with ${\rm gcd}(m,n)=1$, and $g(x,y,z)=x+y+z$, hence
$f'(x,y)=x^{dn}y^{dm}$,  and $g'=x+y$.

Note that $f$ has $d$ local irreducible components, and each component is smooth. Hence, again
we know that all the genus decorations of $\gc$ are zero.

\vspace{2mm}

By {\it Step 1}  the  graph $\G(f',g')$  is

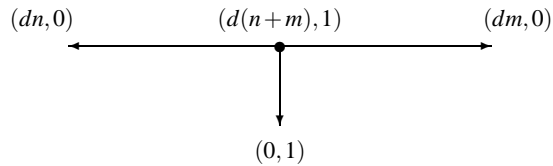
\begin{picture}(300,90)(-55,30)
\put(100,80){\circle*{4}}
\put(100,90){\makebox(0,0){\small{$(d(n+m),1)$}}}

\put(100,80){\vector(1,0){80}}
\put(190,90){\makebox(0,0){\small{$(dm,0)$}}}

\put(100,80){\vector(-1,0){80}}
\put(10,90){\makebox(0,0){\small{$(dn,0)$}}}

\put(100,80){\vector(0,-1){30}}
\put(100,40){\makebox(0,0){\small{$(0,1)$}}}

\end{picture}

\vspace{5mm}

\noindent Finally, by {\it Step 2},  a possible graph $\gc$ of $(f,g)$ is:

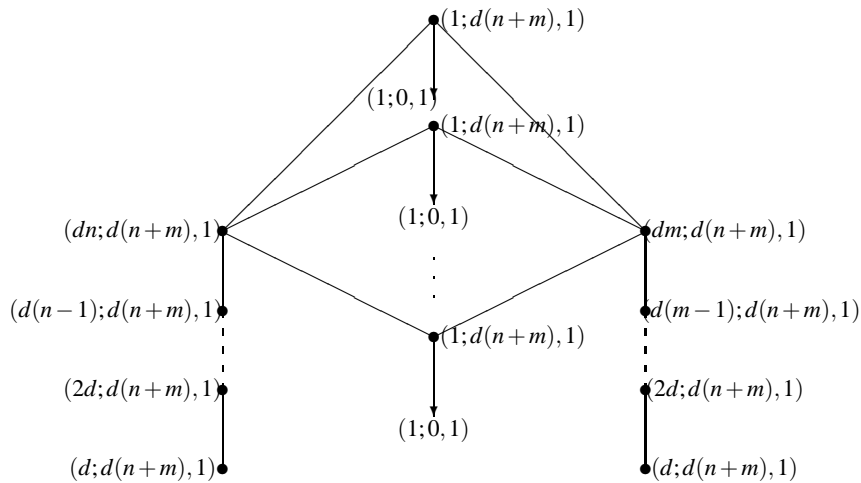
\vspace{2mm}

%\newpage szoveg szoveg szoveg szoveg szoveg szoveg szoveg szoveg
%szoveg szoveg szoveg szoveg szoveg szoveg szoveg szoveg
%szoveg szoveg szoveg szoveg szoveg szoveg szoveg szoveg
%szoveg szoveg szoveg szoveg szoveg szoveg szoveg szoveg

\begin{picture}(300,220)(-55,40)
% bal oszlop csucsai + sulyok
\put(20,70){\circle*{4}}
\put(20,100){\circle*{4}}
\put(20,130){\circle*{4}}
\put(20,160){\circle*{4}}

\put(-10,70){\makebox(0,0){\small{$(d;d(n+m),1)$}}}
\put(-10,100){\makebox(0,0){\small{$(2d;d(n+m),1)$}}}
\put(-20,130){\makebox(0,0){\small{$(d(n-1);d(n+m),1)$}}}
\put(-10,160){\makebox(0,0){\small{$(dn;d(n+m),1)$}}}

% jobb oszlop csucsai + sulyok
\put(180,70){\circle*{4}}
\put(180,100){\circle*{4}}
\put(180,130){\circle*{4}}
\put(180,160){\circle*{4}}

\put(210,70){\makebox(0,0){\small{$(d;d(n+m),1)$}}}
\put(210,100){\makebox(0,0){\small{$(2d;d(n+m),1)$}}}
\put(220,130){\makebox(0,0){\small{$(d(m-1);d(n+m),1)$}}}
\put(210,160){\makebox(0,0){\small{$(dm;d(n+m),1)$}}}

% kozepso sor csucsai +sulyok
\put(100,120){\circle*{4}}
\put(100,200){\circle*{4}}
\put(100,240){\circle*{4}}

\put(130,120){\makebox(0,0){\small{$(1;d(n+m),1)$}}}
\put(130,200){\makebox(0,0){\small{$(1;d(n+m),1)$}}}
\put(130,240){\makebox(0,0){\small{$(1;d(n+m),1)$}}}

%osszekoto elek+ sulyok
\put(20,70){\line(0,1){30}}
\put(180,70){\line(0,1){30}}

\put(20,130){\line(0,1){30}}
\put(180,130){\line(0,1){30}}

\dashline[3]{3}(20,100)(20,130)
\dashline[3]{3}(180,100)(180,130)

\put(100,200){\line(2,-1){80}} \put(100,200){\line(-2,-1){80}}

\put(100,240){\line(1,-1){80}} \put(100,240){\line(-1,-1){80}}

\put(100,120){\line(2,1){80}}
\put(100,120){\line(-2,1){80}}

%nyilak + sulyaik
\put(100,120){\vector(0,-1){30}} \put(100,200){\vector(0,-1){30}}
\put(100,240){\vector(0,-1){30}}
\put(100,85){\makebox(0,0){\small{$(1;0,1)$}}}
\put(100,165){\makebox(0,0){\small{$(1;0,1)$}}}
\put(88,210){\makebox(0,0){\small{$(1;0,1)$}}}

% pont-pont-pont a hurkok kozott
\dashline[4]{1}(100,135)(100,150)

\end{picture}

\noindent where in the middle column there are $d$ vertices.

\vspace{3mm}

For another suspension case, see \ref{ex:347}.
\end{example}

\section{\ The $T_{a,*,*}$--family}\labelpar{ZFG}

\begin{bekezdes}\label{bek:ZFG}  {\bf The $T_{a,\infty,\infty}$--family.}
% x^a+xyz ketszinu grafjai
Let $f=x^a+xyz$ and set $g=x+y+z$.

\vspace{2mm}

If $a=2k+1$, $k> 1$, then a possible universal graph is:

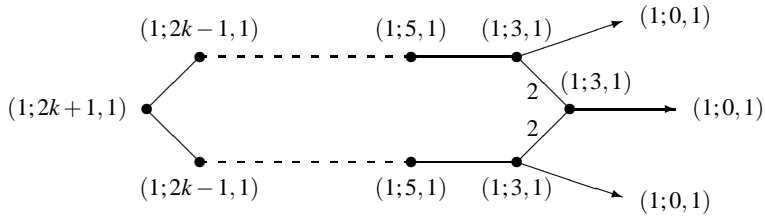
\vspace{2mm}

\begin{picture}(100,100)(-20,-10)
% also sor csucsai + sulyok
\put(60,20){\circle*{4}} \put(140,20){\circle*{4}} \put(180,20){\circle*{4}}
\put(60,10){\makebox(0,0){\small{$(1;2k-1,1)$}}}
\put(140,10){\makebox(0,0){\small{$(1;5,1)$}}}
\put(180,10){\makebox(0,0){\small{$(1;3,1)$}}}

% felso sor csucsai + sulyok
\put(60,60){\circle*{4}} \put(140,60){\circle*{4}} \put(180,60){\circle*{4}}
\put(60,70){\makebox(0,0){\small{$(1;2k-1,1)$}}}
\put(140,70){\makebox(0,0){\small{$(1;5,1)$}}}
\put(180,70){\makebox(0,0){\small{$(1;3,1)$}}}

% kozepso sor csucsai +sulyok
\put(40,40){\circle*{4}} \put(200,40){\circle*{4}}
\put(10,40){\makebox(0,0){\small{$(1;2k+1,1)$}}}
\put(210,50){\makebox(0,0){\small{$(1;3,1)$}}}

%osszekoto elek+ sulyok
\put(40,40){\line(1,1){20}}
\put(40,40){\line(1,-1){20}}

\put(200,40){\line(-1,1){20}}
\put(200,40){\line(-1,-1){20}}
\put(186,47){\makebox(0,0){\small{2}}}
\put(186,33){\makebox(0,0){\small{2}}}

\dashline[3]{3}(60,20)(140,20)
\dashline[3]{3}(60,60)(140,60)

\put(140,20){\line(1,0){40}}
\put(140,60){\line(1,0){40}}

%nyilak + sulyaik
\put(180,20){\vector(3,-1){40}}
\put(200,40){\vector(1,0){40}}
\put(180,60){\vector(3,1){40}}
\put(240,5){\makebox(0,0){\small{$(1;0,1)$}}}
\put(240,75){\makebox(0,0){\small{$(1;0,1)$}}}
\put(260,40){\makebox(0,0){\small{$(1;0,1)$}}}

\end{picture}

\vspace{2mm}

\noindent Above, all unmarked edges have weight 1.
%\vspace{0.5cm}

The resolution process was started by a blow-up at the origin.
Since $f=x^{2k+1}+xyz=x(x^{2k}+yz)$, there is only one singularity remaining,
of the form $\{t^{2k-2}+sr=0\}$. It is resolved
by a series of blow-ups at infinitely near points, resulting in the above graph.

\vspace{2mm}

If  $a=3$ then $f$ is homogeneous. In particular,  a single blow-up at the origin suffices and we
get

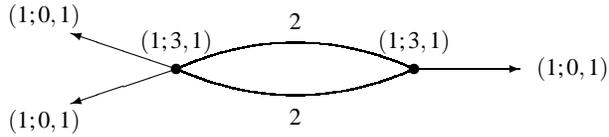
\begin{picture}(100,80)(-20,-0)
% kozepso sor csucsai +sulyok
\put(50,40){\circle*{4}} \put(140,40){\circle*{4}}
\put(50,50){\makebox(0,0){\small{$(1;3,1)$}}}
\put(140,50){\makebox(0,0){\small{$(1;3,1)$}}}

%nyilak + sulyaik
\put(140,40){\vector(1,0){40}}
\put(200,40){\makebox(0,0){\small{$(1;0,1)$}}}

\put(50,40){\vector(-3,1){40}}
\put(0,60){\makebox(0,0){\small{$(1;0,1)$}}}
\put(50,40){\vector(-3,-1){40}}
\put(0,20){\makebox(0,0){\small{$(1;0,1)$}}}

% osszekoto ivek + sulyok:
\qbezier(50,40)(95,60)(140,40)
\qbezier(50,40)(95,20)(140,40)
\put(95,58){\makebox(0,0){\small 2}}
\put(95,22){\makebox(0,0){\small 2}}

\end{picture}

\smallskip
Set now $a=2k$, $k> 2$. The same strategy for the resolution as above
can be followed. However, when the strict transform of $f$ becomes smooth,
it is not in normal crossing with the exceptional divisors.
Two additional blow-ups along singular axes lead to the following universal
graph:

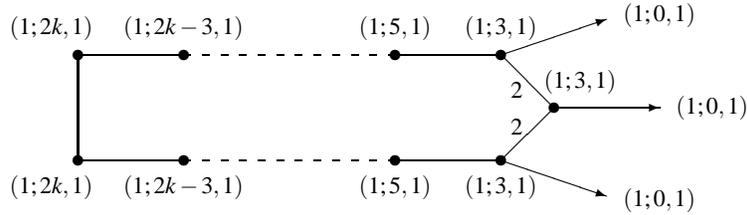
\vspace{2mm}

\begin{picture}(100,100)(-20,-10)
% also sor csucsai + sulyok

\put(20,20){\circle*{4}}
\put(60,20){\circle*{4}} \put(140,20){\circle*{4}} \put(180,20){\circle*{4}}
\put(10,10){\makebox(0,0){\small{$(1;2k,1)$}}}
\put(60,10){\makebox(0,0){\small{$(1;2k-3,1)$}}}
\put(140,10){\makebox(0,0){\small{$(1;5,1)$}}}
\put(180,10){\makebox(0,0){\small{$(1;3,1)$}}}

% felso sor csucsai + sulyok
\put(20,60){\circle*{4}}
\put(60,60){\circle*{4}} \put(140,60){\circle*{4}} \put(180,60){\circle*{4}}
\put(10,70){\makebox(0,0){\small{$(1;2k,1)$}}}
\put(60,70){\makebox(0,0){\small{$(1;2k-3,1)$}}}
\put(140,70){\makebox(0,0){\small{$(1;5,1)$}}}
\put(180,70){\makebox(0,0){\small{$(1;3,1)$}}}

% kozepso sor csucsai +sulyok
\put(200,40){\circle*{4}}
\put(210,50){\makebox(0,0){\small{$(1;3,1)$}}}

%osszekoto elek+ sulyok
\put(20,20){\line(1,0){40}}
\put(20,60){\line(1,0){40}}

\put(20,20){\line(0,1){40}}

\put(200,40){\line(-1,1){20}}
\put(200,40){\line(-1,-1){20}}
\put(186,47){\makebox(0,0){\small{2}}}
\put(186,33){\makebox(0,0){\small{2}}}

\dashline[3]{3}(60,20)(140,20)
\dashline[3]{3}(60,60)(140,60)

\put(140,20){\line(1,0){40}}
\put(140,60){\line(1,0){40}}

%nyilak + sulyaik
\put(180,20){\vector(3,-1){40}}
\put(200,40){\vector(1,0){40}}
\put(180,60){\vector(3,1){40}}
\put(240,5){\makebox(0,0){\small{$(1;0,1)$}}}
\put(240,75){\makebox(0,0){\small{$(1;0,1)$}}}
\put(260,40){\makebox(0,0){\small{$(1;0,1)$}}}

\end{picture}

\vspace{2mm}

(All unmarked edges have weight 1, as before.)

\smallskip Finally, in case
$a=4$ the previous resolution `strategy' leads to

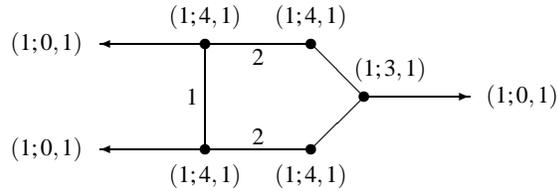
\begin{picture}(100,100)(-20,-10)
% also sor csucsai + sulyok
\put(140,20){\circle*{4}} \put(180,20){\circle*{4}}
\put(140,10){\makebox(0,0){\small{$(1;4,1)$}}}
\put(180,10){\makebox(0,0){\small{$(1;4,1)$}}}

% felso sor csucsai + sulyok
\put(140,60){\circle*{4}} \put(180,60){\circle*{4}}
\put(140,70){\makebox(0,0){\small{$(1;4,1)$}}}
\put(180,70){\makebox(0,0){\small{$(1;4,1)$}}}

% kozepso sor csucsai +sulyok
\put(200,40){\circle*{4}}
\put(210,50){\makebox(0,0){\small{$(1;3,1)$}}}

%osszekoto elek+ sulyok
\put(200,40){\line(-1,1){20}}
\put(200,40){\line(-1,-1){20}}
%\put(186,47){\makebox(0,0){\small{1}}}
%\put(186,33){\makebox(0,0){\small{1}}}

\put(140,20){\line(1,0){40}}
\put(140,60){\line(1,0){40}}
\put(140,20){\line(0,1){40}}
\put(160,25){\makebox(0,0){\small{2}}}
\put(160,55){\makebox(0,0){\small{2}}}
\put(135,40){\makebox(0,0){\small{1}}}

%nyilak + sulyaik
\put(140,20){\vector(-1,0){40}}
\put(200,40){\vector(1,0){40}}
\put(140,60){\vector(-1,0){40}}
\put(80,20){\makebox(0,0){\small{$(1;0,1)$}}}
\put(260,40){\makebox(0,0){\small{$(1;0,1)$}}}
\put(80,60){\makebox(0,0){\small{$(1;0,1)$}}}
\end{picture}

\end{bekezdes}

\begin{bekezdes}\labelpar{a2infty} {\bf The $T_{a,2,\infty}$--family (again).}
Set  $f(x,y,z)=x^a+y^2+xyz$.

The cases $a=3$ and $a=5$ (with $g=z$) were already  considered in \cite{eredeti},
where $\gc$ was obtained using an alternative/ad hoc resolution. The
graphs $\gc$ thus obtained will  serve as clarifying examples for
several geometric discussions in this work as well. The graphs are the following:

\begin{example}\labelpar{xyz3}
If $f=x^3+y^2+xyz$ and $g=z$ then a possible $\gc$
is:

\begin{picture}(200,85)(-50,20)
\put(50,25){\circle*{4}} \put(50,75){\circle*{4}}
\put(100,25){\circle*{4}} \put(100,75){\circle*{4}}
\put(150,50){\circle*{4}}

\put(50,25){\line(1,0){50}} \put(50,75){\line(1,0){50}}
\put(100,25){\line(0,1){50}} \put(150,50){\line(-2,-1){50}}
\put(150,50){\line(-2,1){50}} \put(150,50){\vector(1,0){40}}

\put(50,85){\makebox(0,0){\small{$(3;0,1)$}}}
\put(50,15){\makebox(0,0){\small{$(2;0,1)$}}}
\put(100,85){\makebox(0,0){\small{$(3;6,1)$}}}
\put(100,15){\makebox(0,0){\small{$(2;6,1)$}}}
\put(160,60){\makebox(0,0){\small{$(1;6,1)$}}}
\put(210,50){\makebox(0,0){\small{$(1;0,1)$}}}
\end{picture}

\end{example}

\vspace{3mm}

\begin{example}\labelpar{xyz5}
If $f=x^5+y^2+xyz$ and  $g=z$ then a possible $\gc$
is:

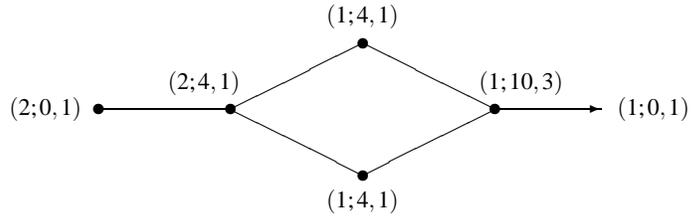
\begin{picture}(200,100)(-50,0)
\put(0,50){\circle*{4}} \put(50,50){\circle*{4}}
\put(100,25){\circle*{4}} \put(100,75){\circle*{4}}
\put(150,50){\circle*{4}}

\put(0,50){\line(1,0){50}} \put(50,50){\line(2,1){50}}
\put(50,50){\line(2,-1){50}} \put(150,50){\line(-2,-1){50}}
\put(150,50){\line(-2,1){50}} \put(150,50){\vector(1,0){40}}

\put(-20,50){\makebox(0,0){\small{$(2;0,1)$}}}
\put(40,60){\makebox(0,0){\small{$(2;4,1)$}}}
\put(100,85){\makebox(0,0){\small{$(1;4,1)$}}}
\put(100,15){\makebox(0,0){\small{$(1;4,1)$}}}
\put(160,60){\makebox(0,0){\small{$(1;10,3)$}}}
\put(210,50){\makebox(0,0){\small{$(1;0,1)$}}}
\end{picture}

\end{example}

The general case (i.e. arbitrary $a\geq 3$),  with $g=x+y+z$, is  clarified in \ref{a2double}.

Notice that in the two examples above and in \ref{a2double}
we used different germs  $g$ and different sequences of blow ups. Thus,
the output graphs are also different.

\end{bekezdes}

\part{Plumbing graphs derived  from $\gc$}

\chapter{The Main Algorithm}\labelpar{s:ALG}\setcounter{equation}{0}

\section{\ Preparations for the Main Algorithm}\labelpar{prep}\setcounter{equation}{0}

\bekezdes {\bf The goal of the chapter.}\ix{plumbing!graph}
The algorithm presented in this chapter provides the plumbing
representations of the 3--manifolds $\partial F$, $\partial _1F$\ix{Milnor!fiber!boundary}
and $\partial_2F$, and the multiplicity systems  of
the open book decomposition of
$(\partial F,V_g)$ as well as the generalized Milnor fibrations
$\partial F\setminus V_g$, $\partial_1 F\setminus V_g$ and
$\partial_2 F$  over $S^1$ induced by $\arg(g)=g/|g|$.
\ix{Main Algorithm|textbf}\ix{Assumption A}
\ix{Assumption B|textbf}

\begin{bekezdes}{\bf Assumptions.}
In \ref{re:w2} we imposed  Assumption A on $\gc$, which can always
be realized by an additional blow up. Although in the Main
Algorithm  this restriction  is irrelevant, in
the geometric interpretations
\ref{gamma1}--\ref{ss:GCK}--\ref{2edg} it simplified  and unified
 the presentation substantially.

In the next paragraph we introduce another restriction,
Assumption B. In contrast with Assumption A,
this new restriction plays a relevant role in the formulation of the
algorithm and  its proof.  Nevertheless, in Chapter \ref{ss:ELI}
we will formulate a new version of the algorithm from which Assumption B will be
removed (but the proof of the new algorithm
will still rely on the proof of the present original version).
\end{bekezdes}

\bekezdes\labelpar{re:w3} {\bf Assumption B.}
In Chapters \ref{s:ALG} and \ref{sec:proof} we will assume that
$\gc$ has no such edge decorated by  $2$  whose end-vertices would have  the middle weights zero.
This requirement is regardless of whether those end-vertices are arrowheads or not.

In the sequel, we call such an edge {\em
vanishing  2--edge}.

 Their absence can be assumed because of the following reason.
Assume that  $\gc$ is associated with some resolution $r$ as in \ref{construct}, and it
has such a 2--edge $e$
\ix{vanishing 2--edge|textbf}

\vspace{3mm}

\begin{picture}(140,50)(-100,10)
%\put(20,30){\circle*{4}} \put(100,30){\circle*{4}}
\put(20,30){\line(1,0){80}}
\put(20,35){\makebox(0,0)[b]{$(m;0,\nu)$}}
%\put(60,40){\makebox(0,0)[b]{$2$}}
\put(100,35){\makebox(0,0)[b]{$(m';0,\nu)$}}
%\put(20,49){\makebox(0,0)[b]{$[g]$}}
%\put(100,49){\makebox(0,0)[b]{$[g']$}}
\put(20,15){\makebox(0,0)[b]{$v$}}
\put(100,15){\makebox(0,0)[b]{$v'$}}
\put(60,20){\makebox(0,0)[b]{$2$}}
\end{picture}

\vspace{3mm}

 \noi where the vertices $v$ and $v'$ correspond to the curves $C_{v}$ and
 $C_{v'}$ of $\C$.
(If $v$ and $v'$ are non--arrowheads,  they might have genus decorations as well.)
 Then the embedded resolution $r$
modified by  an additional blow up with center $p\in C_{v}\cap
C_{v'}$ provides a new graph $\gc'$,  where $e$ is replaced by

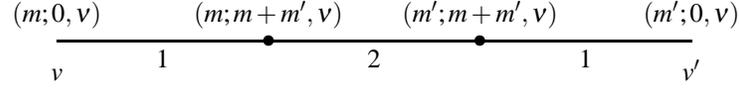
\begin{picture}(140,70)(-100,0)
\put(20,30){\circle*{4}} \put(100,30){\circle*{4}}
\put(-60,30){\line(1,0){240}}
\put(-60,35){\makebox(0,0)[b]{$(m;0,\nu)$}}
%\put(60,40){\makebox(0,0)[b]{$2$}}
\put(180,35){\makebox(0,0)[b]{$(m';0,\nu)$}}
%\put(20,15){\makebox(0,0)[b]{$[g]$}}
%\put(100,15){\makebox(0,0)[b]{$[g']$}}
\put(-60,15){\makebox(0,0)[b]{$v$}}
\put(180,15){\makebox(0,0)[b]{$v'$}}
\put(60,20){\makebox(0,0)[b]{$2$}}
\put(-20,20){\makebox(0,0)[b]{$1$}}
\put(140,20){\makebox(0,0)[b]{$1$}}

\put(20,35){\makebox(0,0)[b]{$(m;m+m',\nu)$}}
\put(100,35){\makebox(0,0)[b]{$(m';m+m',\nu)$}}
\end{picture}

\bekezdes {\bf Terminology --- {\it legs} and  {\it stars}.}\labelpar{ls} \  In the
description of the algorithm we use the following expressions.
\ix{legs}\ix{stars|textbf}

Fix a  non--arrowhead vertex $v$ of
$\gc$ with weights $(m;n,\nu)$ and $[g]$.
Then $v$ determines a  {\it star} in $\gc$, which
 keeps track of all the edges adjacent to $v$ along with their decorations
and the weights $(k;l,\nu)$ of the vertices at the other end of the edges,
but disregards the  type of these vertices.  The aim
is to unify the  different cases represented by loops and
 edges connecting non--arrowheads or arrowheads.

\begin{definition}
A {\bf leg supported by the vertex $v$} has the form:

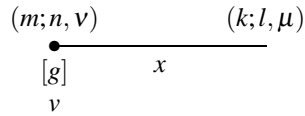
\vspace{3mm}

\begin{picture}(140,50)(-100,0)
\put(20,30){\circle*{4}} %\put(100,30){\circle*{4}}
\put(20,30){\line(1,0){80}}
\put(20,35){\makebox(0,0)[b]{$(m;n,\nu)$}}
%\put(60,40){\makebox(0,0)[b]{$2$}}
\put(100,35){\makebox(0,0)[b]{$(k;l,\mu)$}}
\put(20,15){\makebox(0,0)[b]{$[g]$}}
%\put(100,15){\makebox(0,0)[b]{$[g']$}}
\put(20,5){\makebox(0,0)[b]{$v$}}
%\put(100,15){\makebox(0,0)[b]{$v_2$}}
\put(60,20){\makebox(0,0)[b]{$x$}}
\end{picture}

\noi where $x\in \{1,2\}$ and the decorations satisfy the same  compatibility
conditions as the edges in \ref{remarks}.
Then,  {\bf a star}, by definition,  consists of a
vertex $v$ (together with its decorations)
and a collection of legs supported by $v$.

Once $\gc$ and $v$ are fixed, {\em the star of $v$ in $\gc$}
is constructed as follows. Its `center' has the decorations $(m;n,\nu)$  and $[g]$
of the vertex $v$.
Furthermore, any edge with decoration $x$, with end-vertices $v$ and $v'$ (where $v'\not=v$,
and $v'$ is either an arrowhead or not) provides a leg with decorations $x$ and the
 weight (the ordered triple) of $v'$. In particular, if $v'$ is a non-arrowhead, and it is
 connected to $v$ by more than one edge, then each edge contributes a leg.
Moreover, any loop supported by $v$
and weighted by $x$, provides  {\em two}
legs supported by $v$, both decorated by the same $x$ and their `free ends' by
$(k;l,\mu)=(m;n,\nu)$.
\end{definition}

One has the following geometrical interpretation: regard $\gc$ as
the dual graph of the curve configuration $\C$, and let $v$
correspond to the component $C$. Then the legs of the star of $v$
correspond to the inverse images of the double points of $\C$
sitting on $C$ by the normalization map $C^{norm}\to C$. \ix{curve configuration $\C$}

\section{\ The Main Algorithm: the plumbing graph
of $\partial F$}\labelpar{algo}\setcounter{equation}{0}\ix{Milnor!fiber!boundary}
\ix{Main Algorithm}
\ix{graph!covering}
\ix{graph!of $\partial F$|textbf}

Recall that in this section we assume  that the graph $\gc$
satisfies Assumption B.\ix{plumbing!graph}
\ix{Assumption B}

First, we  construct the plumbing graph of
the open book decomposition of $\partial F$ with binding $V_g$ and fibration
$\arg(g):\partial F\setminus V_g\to S^1$, cf. \ref{CC} and \ref{bek:mult}. Hence,
we have to determine the shape of the graph together with its
arrows, and endow it with the genus and Euler number decorations and
a multiplicity system. The graph  will be
determined  as a  covering graph $G$ of $\gc$, modified with \ix{graph!resolution!string}
strings as in \ref{re:2.3.1}. In order to  do
this, we have to provide the {\em covering data of the graph--covering}, see
\ref{def:2.3.2}.\ix{graph!covering!data}

\begin{bekezdes}{\bf Step 1. --- The covering data of the vertices
$\{\n_v\}_{v\in \calv(\gc)}$.}\labelpar{s1}

\vspace{2mm}

\noi {\bf Case 1.}  \ Consider a non-arrowhead vertex $w$ of \,
$\G_\C$ decorated by $(m;n,\nu)$ and $[g]$. (In fact, by
\ref{gkl}, $g=0$ whenever $m>1$.) Consider its star

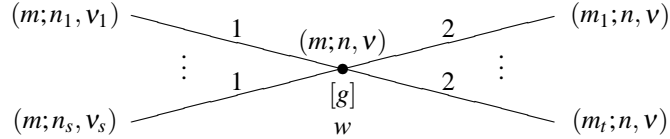
\vspace{2mm}

\begin{picture}(140,80)(-130,-10)
\put(20,30){\circle*{4}}
\put(20,30){\line(4,1){80}}
\put(20,30){\line(4,-1){80}}
\put(20,30){\line(-4,1){80}}
\put(20,30){\line(-4,-1){80}}
\put(20,35){\makebox(0,0)[b]{$(m;n,\nu)$}}

\put(125,45){\makebox(0,0)[b]{$(m_1;n,\nu)$}}
\put(125,5){\makebox(0,0)[b]{$(m_t;n,\nu)$}}
\put(-85,45){\makebox(0,0)[b]{$(m;n_1,\nu_1)$}}
\put(-85,5){\makebox(0,0)[b]{$(m;n_s,\nu_s)$}}

\put(20,15){\makebox(0,0)[b]{$[g]$}}
\put(80,26){\makebox(0,0)[b]{$\vdots$}}
\put(-40,26){\makebox(0,0)[b]{$\vdots$}}
\put(20,5){\makebox(0,0)[b]{$w$}}
%\put(100,15){\makebox(0,0)[b]{$v_2$}}
\put(60,22){\makebox(0,0)[b]{$2$}}
\put(60,42){\makebox(0,0)[b]{$2$}}
\put(-20,22){\makebox(0,0)[b]{$1$}}
\put(-20,42){\makebox(0,0)[b]{$1$}}
\end{picture}

\vspace{2mm}

\noi Let $s$ and $t$ be the number of legs weighted by
$x=1$ and  $x=2$ respectively.
Then, in the covering procedure, above the vertex $w$ of
$\gc$ put $\n_w$ non-arrowhead vertices, where
\begin{equation}\label{NW}\n_w=\gcd(m, n, n_1,...,n_s,
m_1,...,m_t).\end{equation} Furthermore, put on each of these
non--arrowhead vertices the same multiplicity decoration
$(\tilde{m})$, where \begin{equation}\label{NM}\tilde{m}={m
{\nu}\over \gcd (m, n)}\, ,\end{equation}
 and the genus  decoration $[\tilde{g}_w]$ determined by the formula:

\begin{align}\label{NG}
\n_w\cdot (2-2\tilde{g}_w) &
=(2-2g - s-t) \cdot {\rm gcd}(m, n)\\
   & + \sum\limits_{i=1}^s {\rm gcd} (m, n,  n_i) +
    \sum\limits_{j=1}^t {\rm gcd}(m, n, m_j)\nonumber.
\end{align}

\noi (In Step 3, the
Euler number   of each vertex will also  be provided.)

\vspace{2mm}

\noi {\bf Case 2.} \ Consider an arrowhead vertex $v$ of $\G_\C$, that is,
\begin{picture}(50,10)(0,0)
\put(5,3){\vector(1,0){5}}
%\put(5,3){\makebox(0,0){$\blacktriangleright$}}
\put(32,-3){\makebox(0,0)[b]{$(1;0,1)$}}
\end{picture}.
Above the vertex $v$, in the covering graph $G$, put
exactly one arrowhead vertex; i.e. set
 $\n_v=1$.  Let the multiplicity  of this
arrowhead be $1$. In particular, all  arrowheads of $G$ are:
\begin{picture}(40,10)(0,0)
\put(10,3){\vector(1,0){5}} \put(27,-3){\makebox(0,0)[b]{$(1)$}}
\end{picture}.
\end{bekezdes}

\begin{bekezdes}{\bf Step 2. --- The covering data of edges  $\{\n_e\}_{e\in
      \cale(\gc)}$
and the types of inserted strings.}\labelpar{s2}

\vspace{2mm}

\noi {\bf Case 1.}\ Consider an edge $e$ in $\G_\C$,
 with decoration 1:

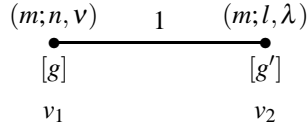
\begin{picture}(160,75)(-90,-10)
\put(20,30){\circle*{4}} \put(100,30){\circle*{4}}
\put(20,30){\line(1,0){80}}
\put(20,35){\makebox(0,0)[b]{$(m;n,\nu)$}}
\put(100,35){\makebox(0,0)[b]{$(m;l,\lambda)$}}
\put(20,15){\makebox(0,0)[b]{$[g]$}}
\put(100,15){\makebox(0,0)[b]{$[g']$}}
\put(20,0){\makebox(0,0)[b]{$v_1$}}
\put(100,0){\makebox(0,0)[b]{$v_2$}}
\put(60,35){\makebox(0,0)[b]{$1$}}
\end{picture}

\noi Define:  $$\n_e=\gcd(m, n, l).$$
Notice that Step 1 guarantees that both $\n_{v_1}$ and $\n_{v_2}$
divide $\n_e$.

Then, above the edge $e$   insert cyclically in $G$ exactly
$\n_e$ strings of type
$$Str\left(\, {n\over \n_e}, {l\over \n_e};{m\over \n_e}\
\Big|\ \nu, \lambda;0 \, \right).$$

If the edge $e$ is a loop (that is,  if $v_1=v_2$), then the procedure is the same with the
only modification that the end--vertices of the $\n_e$ strings are identified cyclically
with the $\n_{v_1}$ vertices above $v_1$, hence they will form 1--cycles in the graph.
In other words, on each vertex above $v_1$ one puts $\n_e/\n_{v_1}$ `closed' strings, that form loops.

\vspace{2mm}

If the right vertex $v_2$ is  an arrowhead, that is, the
edge $e$ is

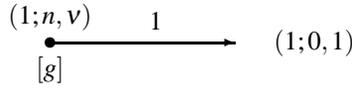
\vspace{1mm}

\begin{picture}(140,60)(-90,0)
\put(20,30){\circle*{4}}
\put(20,30){\vector(1,0){70}}
\put(20,35){\makebox(0,0)[b]{$(1;n,\nu)$}}
\put(120,30){\makebox(0,0){$(1;0,1)$}}
\put(20,15){\makebox(0,0)[b]{$[g]$}}
\put(60,35){\makebox(0,0)[b]{$1$}}
\end{picture}

\vspace{1mm}

\noi then complete the same procedure as above with $m=1$ and
$\n_e=1$:  above such an edge $e$ put a single edge
decorated by $+$, which supports that arrowhead of $G$ which
covers the corresponding arrowhead of $\gc$.

\vspace{2mm}

\noi {\bf Case 2.} Consider an  edge $e$  in
$\G_\C$, with decoration $2$:

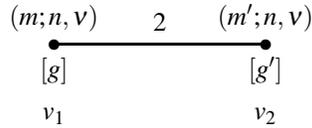
\vspace{2mm}

\begin{picture}(140,70)(-90,-10)
\put(20,30){\circle*{4}} \put(100,30){\circle*{4}}
\put(20,30){\line(1,0){80}}
\put(20,35){\makebox(0,0)[b]{$(m;n,\nu)$}}
\put(100,35){\makebox(0,0)[b]{$(m';n,\nu)$}}
\put(20,15){\makebox(0,0)[b]{$[g]$}}
\put(100,15){\makebox(0,0)[b]{$[g']$}}
\put(20,0){\makebox(0,0)[b]{$v_1$}}
\put(100,0){\makebox(0,0)[b]{$v_2$}}
\put(60,35){\makebox(0,0)[b]{$2$}}
\end{picture}

\vspace{1mm}

 \noi Notice that, by Assumption B (cf. \ref{re:w3}),
$n\not=0$. For such an edge, define: \begin{equation}\label{NE}
\n_e=\gcd(m, m',n).\end{equation}
\ix{Assumption B}

 Notice again that both $\n_{v_1}$ and $\n_{v_2}$  divide $\n_e$.

Then, above the edge $e$   insert cyclically in $G$
exactly $\n_e$ strings of type
$$Str^\circleddash\left(\,{ m\over \n_e},{m'\over \n_e};{n \over \n_e}\
\Big| \ 0,0;\nu\, \right).$$ If the edge is a loop, then we modify
the procedure as in the case of 1--loops. Notice that  by
Assumption B, there are no 2-edges supporting arrowheads.

\vspace{2mm}

Note that above $\gce$ and above  the cutting edges the
`covering degree' is always one. \ix{cutting edge}

\end{bekezdes}

\begin{bekezdes}{\bf Step 3. --- Determination of the
missing Euler numbers.}\labelpar{s3} The decorations provided by
the first two steps are the following: the multiplicities of all the vertices, all
the genera, {\em some} of the Euler numbers, and all the
sign--decorations of the edges (those without $\circleddash$ have
decoration $+$). Then, finally, the missing Euler numbers
are determined by formula (\ref{eq:2.2.1}).
\end{bekezdes}

\bekezdes {\bf The output of the algorithm.}\labelpar{covdata}
Notice that the  set of integers $\{\n_v\}_{v\in\calv(\gc)}$ and
$\{\n_e\}_{e\in\cale(\gc)}$ satisfy the axioms of a covering data.
Furthermore,   if $v\in \calv^1(\gc)$ then $m=1$ hence $\n_v=1$.
Moreover, by Corollary \ref{tree}, each $\Gamma^2_{{\mathcal C},j}$ is
a tree. Therefore, by Proposition \ref{th:2.3.1} and Theorem \ref{unique}
 we get that

\vspace{2mm}

\begin{verse}
{\it there is only one cyclic covering of \, $\gc$ with
this covering  data \\ (up to a graph--isomorphism).}
\end{verse}

\vspace{2mm}

\noi  The graphs obtained by the above algorithm can, in
general, be simplified by the operations of the oriented plumbing calculus
(or their inverses), or by the reduced plumbing calculus.
 If we are interested only in the output oriented
3--manifold, we can apply this freely without any restriction. Nevertheless, if we wish to
keep some information from the (analytic) construction which provides the graph
(for example, if we wish to apply the results of Chapters \ref{s:vh}--\ref{s:MHS}  regarding different
horizontal and vertical monodromies), then
it is better to apply only
the {\em reduced plumbing calculus} of oriented 3--manifolds
(with arrows), cf. \ref{CALC}. This is what we prefer to do in this book.\ix{plumbing!graph}

Moreover, even if we rely only on the reduced calculus,
during the plumbing calculus,  some invariants  might still
change.  For example, the operation R5 modifies  the sum $g(Gr)$ of the genus decorations and the
number $c(Gr)$ of independent 1--cycles of a graph $Gr$.
Since, in the sequel in some discussions these
numbers  will also be involved, by our graph notations we wish to
emphasize that a certain graph is in the  unmodified stage, or it was modified by the calculus.
Hence, we will adopt the following notation:

\begin{definition}\label{def:GGm}
 We write $G$,  $G_1$ and $G_{2,j}$ for
the graphs obtained by the original algorithm (associated with $\gc$, $\gce$ and $\gck$, see below),
while  the general notation for the  {\em modified
graph}  under the reduced plumbing calculus
is $\Gmod$, $\Gemod$,  $\Gkjmod$. \ix{graph!covering}
\end{definition}

Using these notations, one of the main results of the present work is the following.

\begin{theorem}\labelpar{MTh.1}
The oriented 3--manifold $\partial F$ and the link $V_g\cap \partial F$ in it
 can be represented by an orientable  plumbing graph (see \ref{ss:2.2} for the terminology).
\ix{Milnor!fiber!boundary}\ix{plumbing!graph}

More precisely, let $(f,g)$ be as in \ref{ss:ICIS}. Then, the decorated graph $G$
constructed above  is a possible plumbing graph of the
pair $(\partial F,\partial F\cap V_g)$, which carries the
multiplicity system of the open book decomposition
$\arg(g):\partial F\setminus V_g\to S^1$. If one deletes the
arrowheads and the multiplicities, one obtains a possible  plumbing graph
of the boundary of the Milnor fiber $\partial F$ of $f$.
\end{theorem}

The proof of Theorem \ref{MTh.1} will be given in Chapter
\ref{sec:proof}. Nevertheless, in the next paragraphs we wish to stress
the main geometric idea of the proof.

\vspace{2mm}

If $f$ is an {\em isolated} hypersurface singularity, then the link
$K=V_f\cap S^5_\epsilon$ is smooth, and there exists an orientation preserving
diffeomorphism $\partial F\approx K$. Hence, $\partial F$ can be `localized', i.e. can be
represented as a boundary of an arbitrary small representative of a (singular) germ.
If that germ is resolved by a modification --- whose existence is guaranteed by the
existence of resolution of singularities ---, then $K$ appears as the boundary of the
exceptional locus, hence one automatically  gets a plumbing representation for $K$.
Its plumbing data can be read from the combinatorics of the exceptional set
and the multiplicity system from the corresponding vanishing orders.\ix{Milnor!fiber!boundary}

If $f$ is {\it not isolated} then $K$ is not smooth, and the above argument does not work. Even
the fact that $\partial F$ has any kind of plumbing representation is not automatic at all.
Nevertheless, the  case of isolated singularities suggests that, if we were able to
`localize' $\partial F$, as a link of a singular germ, then we would be able to extend
 the above procedure valid for isolated singularities to the non--isolated case as well.
{\it This realization is the main point  in the proof} of Theorem
\ref{MTh.1}, but with the difference that the germ whose local link is
$\partial F$ is {\em not holomorphic} (complex analytic), but it is {\em real analytic}. One has the following surprising
result.

\begin{proposition}\labelpar{prop:LINK} (See \ref{lem:sksmooth}.)
Let $f$ be a hypersurface singularity with a 1--dimensional singular locus.
Take another germ  $g$ such that $(f,g)$ forms an ICIS as in  \ref{ss:ICIS}.
For a sufficiently large even integer $k$ consider the  real analytic germ
$$\cals_k:=\{z\in(\bfc^3,0)\,:\, f(z)=|g(z)|^k\,\}.$$
Then $\cals_k\setminus \{0\}$ is a smooth 4--manifold with a
natural orientation whose link is independent of the choice of $g$ and $k$, and which, in fact,
is orientation preserving  diffeomorphic to $\partial F$.
\end{proposition}
\ix{localization}\ix{Milnor!fiber!boundary}

The proof of  Theorem \ref{MTh.1}, in fact, describes  the topology of a
resolution of $\cals_k$, and it shows that it is  `guided' exactly by  $\gc$. Furthermore,  the algorithm
which provides  $G$ from $\gc$  extracts
 the combinatorics of the exceptional locus and
its tubular neighbourhood from this resolution.

Notice that the above proposition is true for $k$ odd too, nevertheless,
$\cals_k$ for $k$ even has nicer analytic
properties: for example,  if $f$ and $g$ are polynomials, then $\cals_k$ is a real algebraic variety.

\begin{remark}\label{rem:THutan} \

(a) \
We would also like to stress  that
even though  Proposition  \ref{prop:LINK} is formulated and proved for germs in three
 variables, it  is true for any germ
$f:(\bfc^{n},0)\to (\bfc,0)$ with 1--dimensional singular locus --- the only
modification in the statement and its proof is the replacement of
$\bfc^3$ with $\bfc^n$.

\vspace{1mm}

(b) \ The  power of Theorem \ref{MTh.1} is not just the fact that it proves that
$\partial F$ has a plumbing representation; for that already Proposition \ref{prop:LINK}
is enough. Theorem \ref{MTh.1} provides a very clear algorithm for the determination of the
plumbing representation, which can be performed for any concrete example.  Moreover,
from the algorithm one can subtract essential theoretical information as well, as will
be done in the next chapters.\ix{Milnor!fiber!boundary}\ix{plumbing!graph}

\vspace{1mm}

(c) Theorem \ref{MTh.1}  was obtained in
2004--2005; the Main Algorithm was presented at the Singularity Conference at Leuven, 2005.
The material of the present book  was posted on the Algebraic Geometry preprint server in 2009
  \cite{ARXIV}.\ix{Main Algorithm}

The fact that the boundary of the Milnor fiber is plumbed was announced by
F. Michel  and A. Pichon  in 2003 \cite{MP,MPerrata}.
%, and examples  were provided in \cite{MP,MPW} obtaining certain  lens spaces and Seifert manifolds.
Their proof appeared on the preprint server in 2010 \cite{MPNew}.
\ix{Michel}\ix{Pichon}\ix{lens space}\ix{Seifert manifold}

The techniques prior to the present book were not sufficiently powerful to produce examples
with cycles, and even to predict the necessity  of edges with negative decorations.
These are novelties of the present work.
\end{remark}

\section{\ Plumbing graphs of  $\partial_1F$ and
$\partial_2F$}\labelpar{1es2}\setcounter{equation}{0}

The above algorithm, which provides $\partial F$, is compatible with
the decomposition of this space into its parts $\partial_1F$ and
$\partial_2F$. In this section we make this statement precise.
\ix{graph!$\G_\C^1$}\ix{graph!of a surface singularity!of normalization of $V_f$}
\ix{plumbing!graph}

\begin{bekezdes}\labelpar{norm1}  {\bf The graphs of
$(V_f^{norm},g\circ n)$ and $\partial_1F$.} Consider the graph $\gce$.
Repeat  Steps 1 and
2 from the Main Algorithm \ref{algo}, but only for the vertices
and edges contained  in $\Gamma^1_\C$, {\it excluding} any edge inherited from a cutting edge. \ix{cutting edge}
Replace any edge inherited from a cutting edge by an edge
supporting an arrowhead with multiplicity $(0)$. In this way we
get a graph with all the multiplicities determined and
with all edge--decorations  +. Calculate  the Euler numbers by
(\ref{eq:2.2.1}). This graph will be denoted by $G_1$.
(Note that it coincides with the graph $G^1_\C$ considered in \ref{bek:G1C} and
Proposition \ref{m1}.) \ix{Main Algorithm}

\begin{remark}
The vertices of $G_1$ can be identified with
some of the vertices of $G$, hence if we disregard  the decorations of the graphs, then
$G_1$ is a subgraph of $G$. However,  as decorated graph,  $G_1$ is not a subgraph of $G$:
that   end--vertex  of any cutting edge which is situated  in $\gce$ will have  different
Euler numbers in the two graphs $G$ and $G_1$.  All  the
other  Euler numbers evidently coincide. \ix{cutting edge}
\end{remark}

The next theorem is essentially  the same as
Proposition \ref{m1}; we consider it again to have a complete picture
of the algorithm and its consequences.

\end{bekezdes}
\begin{theorem}\labelpar{g1}
$G_1$ is a possible embedded resolution graph of
$(V_f^{norm},g\circ n)$, where the arrows with multiplicity $(0)$
represent the link--components determined by the strict transforms
of $Sing(V_f)$. If we delete these 0--multiplicity arrows, we get
a possible  embedded resolution graph of $(V_f^{norm},g\circ n)$. Furthermore,
if we delete all the arrowheads and all the multiplicities, we get
a possible resolution graph of the normalization of $V_f^{norm}$.

If in $G_1$ we replace the 0--multiplicity arrows by dash-arrows
we get the plumbing representation of the pair $(\partial_1F,V_g\cap \partial F)$,
where the remaining arrows represent the link $V_g\cap \partial F$, and the multiplicities
are the multiplicities of the local trivial fibration  $\arg(g):\partial_1 F\setminus V_g\to S^1$.
In particular, if  all  remaining non--dash--arrows  and all multiplicities are deleted as well,
we get the plumbing graph of the 3--manifold with boundary  $\partial_1F$.

If  in $G_1$ all  arrows are replaced
by dash--arrows and  the multiplicities are deleted,  we get the
plumbing graph of the manifold with boundary $\partial_1F\setminus T^\circ(V_g)$.
\end{theorem}\ix{plumbing!graph}

\begin{bekezdes}\labelpar{2}{\bf The graph of $\partial_2F$.} \ix{graph!resolution!string}
Let $G_2$ be the graph obtained from $G$ as follows. Delete all
vertices and edges of $G$ that are above the vertices and edges of
$\Gamma^1_\C$, and replace the unique string above any cutting
edge by a dash--arrow (putting no multiplicity decoration  on it).
Obviously, $G_2$ has $s$ connected components $\{G_{2,j}\}_{1\leq
j\leq s}$; where $G_{2,j}$ is related with $\Sigma_j$ as in
\ref{ss:GCK}. Clearly, $G_{2,j}$ can be determined from
$\G^2_{\C,j}$ by a similar procedure as $G$ is obtained from $\gc$, and by
adding dash--arrows above the  arrowheads. \ix{cutting edge}
\end{bekezdes}
\ix{graph!$\G_\C^2$}

\begin{theorem}\labelpar{g2} For each $j=1,\ldots, s$,
$G_{2,j}$  is a possible plumbing graph for the 3--manifold with
boundary $\partial_{2,j}F$, where the set of multiplicities
consists of the multiplicity system associated with the fibration
$\arg(g):\partial_{2,j}F\to S^1$. If all
multiplicities are deleted then obviously we
get a plumbing graph of the 3--manifold with boundary
$\partial_{2,j}F$.
\end{theorem}

\begin{bekezdes}\labelpar{gl} {\bf The gluing tori.} Each connected
component $\partial_{2,j}F$ ($j=1, \ldots, s$) is glued to
$\partial_1F$ along $\partial\partial_{2,j}F$, which is a union of
tori. Since for each cutting edge $e$ one has $\n_e=1$ (i.e. in
the Main Algorithm  exactly one string is inserted above $e$), the
number of these tori is exactly the cardinality of
$\cale_{cut,j}$, the number of cutting edges adjacent to
$\Gamma^2_{\C,j}$.
\ix{gluing tori} \ix{cutting edge}\ix{Main Algorithm}

Let $e$ be such a cutting edge, and  use the notations of
\ref{2edg} regarding this edge. Let $T_e$ be the torus component
of $\partial (\partial_{2,j}F)$ corresponding to $e$. Furthermore,
consider the fibration $T_e\to L_j$ from \ref{prop:sier}. Then
the fiber of this projection consists of $d(e)$ circles, the
corresponding orbit of the (permutation) action of $m'_{j,ver}$ on
$\partial F'_j$.
In particular, the number of connected components of $\partial F'_j$
is $\sum_{e\in \cale_{cut,j}}d(e)$, as it was already noticed in (\ref{eq:tr}).
Recall also that $m'_{j,hor}$ acts on $T_e$
trivially, cf. \ref{prop:sier}(\ref{mon}).
For more details see  section \ref{2edg}.

These gluing tori appear in the link $K$ of $V_f$ as well. Indeed,
consider  $L_j=K\cap \Sigma_j$ as in \ref{ss:2.0a}. Let $T(L_j)$ be
a tubular neighbourhood of $L_j$ in $S^5_\epsilon$ as in
\ref{ss:2.1}. Then $\partial T(L_j)$ intersects $K$ in
$|\cale_{cut,j}|$ tori, which can be identified with  the gluing tori
of $\partial F$, see also \ref{rem:link}.\ix{Milnor!fiber!boundary}
\end{bekezdes}

\begin{remark} From the plumbing graphs of $\partial_1F$ and
$\partial_2F$ it is impossible to recover the graph of $\partial
F$, since the {\it gluing} information (an automorphism of the
gluing tori) cannot be read from the partial information contained
in the graphs of $\partial_1F$ and $\partial_2F$. (See e.g.
 examples \ref{ex:d1d2} and \ref{221b}.) This gluing
information is exactly one of the main advantages of the graph
$\gc$ and of the Main Algorithm, which provides the full $\partial
F$.\end{remark}\ix{Main Algorithm}

\begin{remark}\label{rem:closure}
In fact, on $G_{2,j}$ one can put even  more information/decoration
inherited from $G$. If one introduces a `canonical' framing
(closed simple curve) of the boundary components of
$\partial_{2,j}F$, then one can define a well--defined
multiplicity of the dash--arrows as well (inherited from $G$).
Usually, such a framing  is needed when one wishes  to `close'
with solid tori a 3--manifold that has  toric boundary components.

Here we will make this completion via the following construction.
Consider the graph $\overline{G_{2,j}}$ obtained as follows.

From the graph $G$ delete all those
vertices which are vertices of $G_1$. All the remaining vertices are non--arrowheads; keep their
genus, Euler number and multiplicity decorations. Keep all the edges which connect
these vertices, and keep their decorations as well. Finally, keep any
edge which connects a vertex $v$ in  $G_1$ with another  vertex $w$ not in $G_1$, keep its
decoration $\circleddash$, and replace $v$ with an arrowhead having the same
 multiplicity as   $v$ has in $G_1$. This graph is
denoted by $\overline{G_2}$.
Its connected components are indexed by $\{1,\ldots,s\}$ and there is a natural bijection
(induced by inclusion) with the graphs $G_{2,j}$.
The connected component  $\overline{G_{2,j}}$ of $\overline{G_2}$ which contains the vertices
of  $G_{2,j}$   is called the {\it canonical closure} of $G_{2,j}$.
$\overline{G_{2,j}}$ contains $|\cale_{cut,j}|$ arrowheads.
\ix{canonical closure of $G_{2,j}$|textbf}

It is clear that $\overline{G_{2,j}}$ can be obtained from $\G^2_{\C,j}$ as well.

If we delete all the multiplicities of $\overline{G_{2,j}}$, but we keep the arrowheads, we get
a plumbing graph of a closed 3--manifold (without boundary)
and a link in it. This 3--manifold,
denoted by $\overline{\partial_{2,j}F}$,
 will be called the {\it canonical closure} of $\partial_{2,j}F$,
since it can be obtained from $\partial_{2,j}F$ by gluing some solid tori to its boundary components
 in a canonical way dictated by the above construction.
 The corresponding link in it  is denoted by $L_{cut,j}$.
The manifold $\partial_{2,j}F$ is obtained from
$\overline{\partial_{2,j}F}$ by deleting the interior of a tubular neighbourhood $T_j$
of $L_{cut,j}$. \ix{gluing tori}

For each component of $L_{cur,j}$ consider the oriented meridian $\gamma_e$  in $T_j$.
The collection $\{\gamma_e\}_{e\in \cale_{cut,j}}$ serves as a framing in $\partial (\partial_{2,j}F)$.
Using this framing $\partial_{2,j}F$ can be closed in a canonical way to get
  $\overline{\partial_{2,j}F}$.

%The multiplicity system of $\overline{G_{2,j}}$ determines a multilink $(\overline{\partial_{2,j}F},
%L_{cur,j})$ (in the sense of \cite{EN}). Note that  the multiplicities  of the arrowheads
%can be even zero, hence this is not a usual open book decomposition, but can be even a fibration
%without binding, see example ?????????????????.
\end{remark}

\begin{remark}
Using the graph $G$ one can decorate  both the graphs $G_1$ and $G_2$  even more so that all  boundary components
of $\partial_1F$ and $\partial_2F$ will be canonically identified with $S^1\times S^1$ in such a way that gluing them
provides  $\partial F$.\ix{Milnor!fiber!boundary}

Since a complete description of $\partial F$ is already provided by the Main Algorithm,
we  omit the description of these decorations.
But, definitely, the interested reader might consider and add this data to the picture
as well.
\end{remark}

\section{\ First examples of graphs of  $\partial F$,   $\partial_1F$ and
$\partial_2F$}\labelpar{FIRSTEX}\setcounter{equation}{0}

\begin{example}\labelpar{ex:d1d2}\ix{Milnor!fiber!boundary}
Assume that $f=x^3+y^2+xyz$ and $g=z$ as in \ref{xyz3}. Then the
output of the Main Algorithm is the following graph $G$:\ix{Main Algorithm}\ix{plumbing!graph}

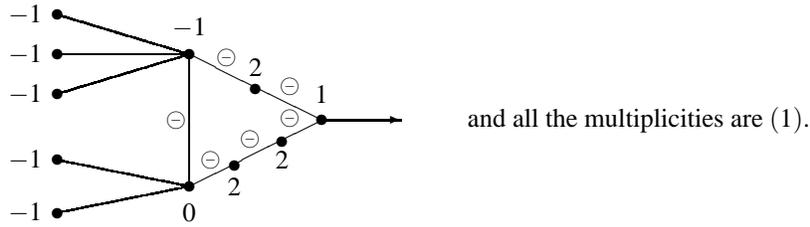
\vspace{2mm}

\begin{picture}(200,90)(30,10)
\put(50,15){\circle*{4}} \put(50,35){\circle*{4}}
\put(50,60){\circle*{4}} \put(50,90){\circle*{4}}
\put(50,75){\circle*{4}}

\qbezier(50,15)(50,15)(100,25) \qbezier(50,35)(50,35)(100,25)
\qbezier(50,60)(50,60)(100,75) \qbezier(50,75)(50,75)(100,75)
\qbezier(50,90)(50,90)(100,75)

\put(100,25){\circle*{4}} \put(100,75){\circle*{4}}
\put(150,50){\circle*{4}} \put(125,62){\circle*{4}}
\put(117,33){\circle*{4}} \put(135,42){\circle*{4}}

%\put(50,25){\line(1,0){50}} \put(50,75){\line(1,0){50}}
\put(100,25){\line(0,1){50}} \put(150,50){\line(-2,-1){50}}
\put(150,50){\line(-2,1){50}} \put(150,50){\vector(1,0){30}}

\put(38,15){\makebox(0,0){$-1$}} \put(38,35){\makebox(0,0){$-1$}}
\put(38,60){\makebox(0,0){$-1$}} \put(38,75){\makebox(0,0){$-1$}}
\put(38,90){\makebox(0,0){$-1$}}

\put(125,70){\makebox(0,0){$2$}} \put(117,25){\makebox(0,0){$2$}}
\put(100,85){\makebox(0,0){$-1$}} \put(100,15){\makebox(0,0){$0$}}
\put(135,35){\makebox(0,0){$2$}} \put(150,60){\makebox(0,0){$1$}}

\put(95,50){\makebox(0,0){$\circleddash$}}
\put(108,35){\makebox(0,0){$\circleddash$}}
\put(123,43){\makebox(0,0){$\circleddash$}}
\put(138,51){\makebox(0,0){$\circleddash$}}
\put(138,63){\makebox(0,0){$\circleddash$}}
\put(114,74){\makebox(0,0){$\circleddash$}}
\put(270,50){\makebox(0,0){$\mbox{and all the multiplicities are
$(1)$}$.}}
\end{picture}

\vspace{2mm}

\noi There is only one non--arrowhead vertex in $G$ with
multiplicity 1. Therefore, one has the following graphs for
$(V_f^{norm},g\circ n)$, $\partial_1F\setminus T(V_g)$, and
$\partial_1F$:

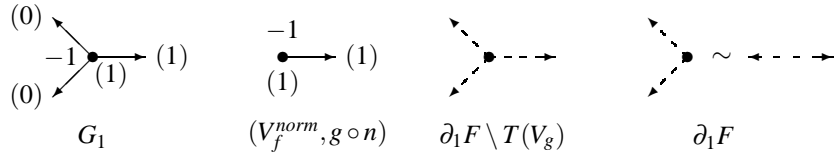
\vspace{2mm}

\begin{picture}(200,65)(20,-25)
\put(50,15){\circle*{4}} \put(50,15){\vector(-1,-1){15}}
\put(50,15){\vector(-1,1){15}} \put(50,15){\vector(1,0){20}}
\put(38,15){\makebox(0,0){$-1$}} \put(25,0){\makebox(0,0){$(0)$}}
\put(25,30){\makebox(0,0){$(0)$}}
\put(80,15){\makebox(0,0){$(1)$}} \put(57,8){\makebox(0,0){$(1)$}}
\put(50,-15){\makebox(0,0){$G_1$}}

\put(122,15){\circle*{4}}  \put(122,15){\vector(1,0){20}}
\put(122,25){\makebox(0,0){$-1$}}
\put(122,5){\makebox(0,0){$(1)$}}
\put(152,15){\makebox(0,0){$(1)$}}
\put(135,-15){\makebox(0,0){$(V_f^{norm},g\circ n)$}}

\put(200,15){\circle*{4}} \put(190,5){\vector(-1,-1){5}}
\put(190,25){\vector(-1,1){5}} \put(220,15){\vector(1,0){5}}
\put(205,-15){\makebox(0,0){$\partial_1F\setminus T(V_g)$}}
\dashline[3]{3}(200,15)(220,15) \dashline[3]{3}(200,15)(190,5)
\dashline[3]{3}(200,15)(190,25)

 \put(275,15){\circle*{4}}
 \dashline[3]{3}(275,15)(265,5)
\dashline[3]{3}(275,15)(265,25) \put(265,5){\vector(-1,-1){5}}
\put(265,25){\vector(-1,1){5}}
  \put(285,-15){\makebox(0,0){$\partial_1F$}}

\dashline[3]{3}(300,15)(325,15) \put(303,15){\vector(-1,0){5}}
\put(325,15){\vector(1,0){5}}
  \put(288,15){\makebox(0,0){$\sim$}}
\end{picture}

\vspace{2mm}

Moreover, by plumbing calculus, the graph of $\partial_2F$ is also
the double dash--arrow
\begin{picture}(30,10)(300,10)
\dashline[3]{3}(300,15)(325,15) \put(303,15){\vector(-1,0){5}}
\put(322,15){\vector(1,0){5}}
\end{picture}
. Notice that both parts $\partial_1F$ and $\partial_2F$ are
extremely  simple 3--manifolds with boundary, namely, both are isomorphic to
$S^1\times S^1\times [0,1]$. The main information in $\partial F$ is exactly how these
parts are glued.\ix{Milnor!fiber!boundary}

By calculus starting from $G$, the boundary $\partial F$ is
represented by the graph:

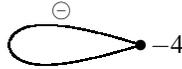
\vspace{2mm}

\begin{picture}(200,40)(50,0)
\put(220,15){\circle*{4}} \qbezier(170,15)(173,30)(220,15)
\qbezier(170,15)(173,0)(220,15)
 \put(230,15){\makebox(0,0){$-4$}}
\put(190,28){\makebox(0,0){$\circleddash$}}
\end{picture}

\vspace{2mm}

\noi The multiplicity system of the
open book decomposition $(\partial F,V_g)$ is given by:

\vspace{2mm}

\begin{picture}(200,40)(-70,0)
\put(50,15){\circle*{4}}
\put(100,15){\circle*{4}}\put(100,15){\vector(1,0){40}}\put(150,15){\makebox(0,0){$(1)$}}
\qbezier(50,15)(75,30)(100,15) \qbezier(50,15)(75,0)(100,15)
 \put(112,22){\makebox(0,0){$-1$}} \put(112,8){\makebox(0,0){$(1)$}}
\put(75,28){\makebox(0,0){$\circleddash$}}
%\put(75,12){\makebox(0,0){$\circleddash$}}
%\put(140,15){\makebox(0,0){$\sim$}}
\put(38,22){\makebox(0,0){$-6$}}\put(40,8){\makebox(0,0){$(0)$}}
\end{picture}

\end{example}

\begin{example}\labelpar{221b}
Assume that $f=x^2y+z^2$ and $g=x+y$, cf. \ref{221}. Then, by
the Main Algorithm and plumbing calculus we get that the (minimal)
plumbing graph of $\partial F$ consists of a unique vertex with
genus zero and Euler number $-4$, i.e. $\partial F$ is the lens
space $L(4,1)$. Moreover,
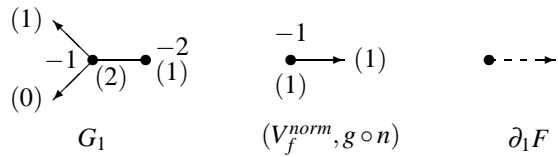
\ix{cutting edge}\ix{gluing tori}\ix{Main Algorithm}\ix{Milnor!fiber!boundary}

\begin{picture}(200,90)(-30,-35)
\put(50,15){\circle*{4}} \put(50,15){\vector(-1,-1){15}}
\put(50,15){\vector(-1,1){15}} \put(50,15){\line(1,0){20}}
\put(38,15){\makebox(0,0){$-1$}} \put(25,0){\makebox(0,0){$(0)$}}
\put(25,30){\makebox(0,0){$(1)$}}\put(70,15){\circle*{4}}
\put(80,10){\makebox(0,0){$(1)$}} \put(57,8){\makebox(0,0){$(2)$}}
\put(50,-15){\makebox(0,0){$G_1$}}\put(80,20){\makebox(0,0){$-2$}}

\put(125,15){\circle*{4}}  \put(125,15){\vector(1,0){20}}
\put(125,25){\makebox(0,0){$-1$}}
\put(125,5){\makebox(0,0){$(1)$}}
\put(155,15){\makebox(0,0){$(1)$}}
\put(140,-15){\makebox(0,0){$(V_f^{norm},g\circ n)$}}

\put(200,15){\circle*{4}}  \put(220,15){\vector(1,0){5}}
\put(215,-15){\makebox(0,0){$\partial_1F$}}
\dashline[3]{3}(200,15)(220,15)
\end{picture}

\noindent and

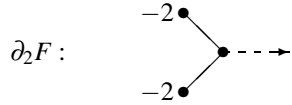
\begin{picture}(200,65)(-110,-25)
\put(50,15){\circle*{4}} \put(50,15){\line(-1,-1){15}}
\put(50,15){\line(-1,1){15}} \put(70,15){\vector(1,0){5}}
\put(25,0){\makebox(0,0){$-2$}} \put(25,30){\makebox(0,0){$-2$}}
\dashline[3]{3}(50,15)(70,15)
\put(-20,15){\makebox(0,0){$\partial_2F:$}}
\put(35,0){\circle*{4}} \put(35,30){\circle*{4}}
\end{picture}

Notice that for the unique cutting 2-edge $e$ in \ref{221} one
has $d_1=1$, hence  $\nu=d(e)=2$. Therefore, although the
transversal singularity has two local irreducible components,
$\partial_1 F$ and $\partial_2F$ are glued by only one torus. The
point is that the transversal type is $A_1$, and the two local
irreducible components of the transversal singularity are permuted by the
vertical monodromy, see \ref{2edg} and \ref{gl} (and compare
also with the next example and Example (3.1) of \cite{Si3}).

The open book decomposition of $(\partial F,V_g)$ is given by:

\vspace{2mm}

\begin{picture}(200,75)(-50,-25)
\put(50,15){\circle*{4}} \put(50,15){\line(-1,-1){15}}
\put(50,15){\line(-1,1){15}}
\put(50,15){\line(1,0){120}}
\put(90,15){\circle*{4}}\put(130,15){\circle*{4}}\put(170,15){\circle*{4}}
\put(130,15){\vector(0,-1){25}}
\put(25,0){\makebox(0,0){$-2$}} \put(25,30){\makebox(0,0){$-2$}}
\put(35,0){\circle*{4}} \put(35,30){\circle*{4}}
\put(50,23){\makebox(0,0){$-1$}}\put(90,23){\makebox(0,0){$-3$}}
\put(130,23){\makebox(0,0){$-1$}}\put(170,23){\makebox(0,0){$-2$}}
\put(53,7){\makebox(0,0){$(2)$}}\put(90,7){\makebox(0,0){$(0)$}}
\put(139,7){\makebox(0,0){$(2)$}}
\put(170,7){\makebox(0,0){$(1)$}}
\put(70,20){\makebox(0,0){$\circleddash$}}
\put(130,-18){\makebox(0,0){$(1)$}}\put(35,-10){\makebox(0,0){$(1)$}}
\put(35,40){\makebox(0,0){$(1)$}}
\end{picture}

\end{example}

\begin{example}\label{UUUU}
In both cases of \ref{xyz3} and \ref{xyz5}, the singular locus
$\Sigma$ of $V_f$  is irreducible and consists of the $z$--coordinate
axis. The transversal type is an $A_1$ singularity, hence
$\#T(\Sigma)=2$. Furthermore, $|\cale_{cut,1}|=2$ and for both
cutting edges $d(e)=1$. Therefore (see \ref{2edg} and
\ref{gl}),  the action of the vertical monodromy does not
permute the two local components, and in both cases \ref{xyz3}
and \ref{xyz5}, there are two gluing tori.
\ix{cutting edge}\ix{gluing tori}

On the other hand, it might happen that the two local components
of a transversal $A_1$ singularity are permuted by the vertical
monodromy, see e.g. \ref{221b}.

Similarly, in the case of
\ref{nu2} (compare  also with \ref{ex:nu2b}),
 $\Sigma=\Sigma_1\cup\Sigma_2$, and for both
$\Sigma_j$ the transversal type is $A_1$, hence $\#T(\Sigma_j)=2$
($j=1,2$). Moreover, for both $j$, $|\cale_{cut,j}|=1$, $d_j=1$, and
$d(e)=2$; hence $\partial_{2,j}F$ is glued to $\partial_1F$ by
exactly one torus.
\end{example}

\begin{example}\labelpar{ex:dminusz2}
Assume that $f=x^d+y^d+xyz^{d-2}$, $d\geq 3$. For $\gc$ see the
second graph of \ref{ex:loop} (which satisfies Assumption A). In
this case $\Sigma$ is irreducible with transversal type $A_1$. The
gluing data are $|\cale_{cut,1}|=2$ and $d_1=1$.  For both cutting
edges $d(e)=1$, hence one has two gluing tori. The graph of
$\partial F$ is:\ix{Milnor!fiber!boundary}

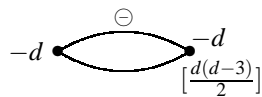
\vspace{2mm}

\begin{picture}(200,60)(-50,-10)
\put(50,15){\circle*{4}} \put(100,15){\circle*{4}}
\qbezier(50,15)(75,30)(100,15) \qbezier(50,15)(75,0)(100,15)

\put(38,15){\makebox(0,0){$-d$}} \put(107,20){\makebox(0,0){$-d$}}
\put(112,5){\makebox(0,0){$[\frac{d(d-3)}{2}]$}}
\put(75,28){\makebox(0,0){$\circleddash$}}
\end{picture}
\end{example}

Recall that if a normal surface singularity is weighted homogeneous, then its link
is a Seifert 3--manifold, and it can be represented by a star--shaped plumbing graph
(or, in the degenerate case, by a string). Note that this is not true in the present situation:
the above equation is homogeneous, nevertheless, the graph has a cycle. The same remark is valid for
the weighted homogeneous equation from \ref{ex:d1d2}. \ix{graph!resolution!string}\ix{Seifert manifold}

Note also that the above graph has a {\it negative definite intersection matrix} $A$, nevertheless
there is no sequence of modifications by plumbing  calculus which would eliminate the
negative edge--decoration. Hence, $\partial F$ cannot be the link of a normal surface singularity.
\ix{graph!negative definite}\ix{matrix!intersection}

\begin{example}\labelpar{ex:347b}
For $f=x^3y^7-z^4$ the graph $\gc$ is given in \ref{ex:347}.
After a computation, we get for  $\Gmod$ the next graph.

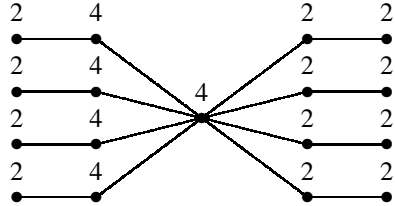
\begin{picture}(200,105)(130,0)

%\put(180,60){\makebox(0,0){$\sim$}}

\put(210,40){\circle*{4}} \put(210,60){\circle*{4}}
\put(210,20){\circle*{4}} \put(210,80){\circle*{4}}

\put(240,40){\circle*{4}} \put(240,60){\circle*{4}}
\put(240,20){\circle*{4}} \put(240,80){\circle*{4}}

\put(320,40){\circle*{4}} \put(320,60){\circle*{4}}
\put(320,20){\circle*{4}} \put(320,80){\circle*{4}}

\put(350,40){\circle*{4}} \put(350,60){\circle*{4}}
\put(350,20){\circle*{4}} \put(350,80){\circle*{4}}

\put(280,50){\circle*{4}}

\put(210,20){\line(1,0){30}}\put(210,40){\line(1,0){30}}
\put(210,60){\line(1,0){30}}\put(210,80){\line(1,0){30}}
\put(320,20){\line(1,0){30}}\put(320,40){\line(1,0){30}}
\put(320,60){\line(1,0){30}}\put(320,80){\line(1,0){30}}

 \qbezier(240,20)(280,50)(320,80) \qbezier(240,40)(280,50)(320,60)
 \qbezier(240,60)(280,50)(320,40) \qbezier(240,80)(280,50)(320,20)

\put(210,30){\makebox(0,0){$2$}}
\put(210,50){\makebox(0,0){$2$}}
\put(210,70){\makebox(0,0){$2$}}
\put(210,90){\makebox(0,0){$2$}}
\put(240,30){\makebox(0,0){$4$}}
\put(240,50){\makebox(0,0){$4$}}
\put(240,70){\makebox(0,0){$4$}}
\put(240,90){\makebox(0,0){$4$}}
\put(320,30){\makebox(0,0){$2$}}
\put(320,50){\makebox(0,0){$2$}}
\put(320,70){\makebox(0,0){$2$}}
\put(320,90){\makebox(0,0){$2$}}
\put(350,30){\makebox(0,0){$2$}}
\put(350,50){\makebox(0,0){$2$}}
\put(350,70){\makebox(0,0){$2$}}
\put(350,90){\makebox(0,0){$2$}}

\put(280,60){\makebox(0,0){$4$}}

\end{picture}

This graph $G^m$ can be transformed into
its `normal form' in the sense of \cite{EN}, that is, with all the Euler numbers on the
legs $\leq -2$.  It is the following: \ix{Eisenbud--Neumann book}

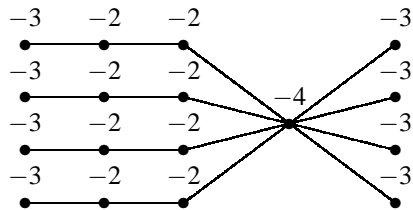
\begin{picture}(200,125)(-55,-10)
\put(10,40){\circle*{4}} \put(10,60){\circle*{4}}
\put(10,20){\circle*{4}} \put(10,80){\circle*{4}}

\put(40,40){\circle*{4}} \put(40,60){\circle*{4}}
\put(40,20){\circle*{4}} \put(40,80){\circle*{4}}

\put(70,40){\circle*{4}} \put(70,60){\circle*{4}}
\put(70,20){\circle*{4}} \put(70,80){\circle*{4}}

\put(150,40){\circle*{4}} \put(150,60){\circle*{4}}
\put(150,20){\circle*{4}} \put(150,80){\circle*{4}}

\put(110,50){\circle*{4}}

\put(10,20){\line(1,0){60}}\put(10,40){\line(1,0){60}}
\put(10,60){\line(1,0){60}}\put(10,80){\line(1,0){60}}

 \qbezier(70,20)(110,50)(150,80) \qbezier(70,40)(110,50)(150,60)
 \qbezier(70,60)(110,50)(150,40) \qbezier(70,80)(110,50)(150,20)

\put(10,30){\makebox(0,0){$-3$}}
\put(10,50){\makebox(0,0){$-3$}}
\put(10,70){\makebox(0,0){$-3$}}
\put(10,90){\makebox(0,0){$-3$}}
\put(150,30){\makebox(0,0){$-3$}}
\put(150,50){\makebox(0,0){$-3$}}
\put(150,70){\makebox(0,0){$-3$}}
\put(150,90){\makebox(0,0){$-3$}}
\put(40,30){\makebox(0,0){$-2$}}
\put(40,50){\makebox(0,0){$-2$}}
\put(40,70){\makebox(0,0){$-2$}}
\put(40,90){\makebox(0,0){$-2$}}
\put(70,30){\makebox(0,0){$-2$}}
\put(70,50){\makebox(0,0){$-2$}}
\put(70,70){\makebox(0,0){$-2$}}
\put(70,90){\makebox(0,0){$-2$}}

\put(110,60){\makebox(0,0){$-4$}}
\end{picture}

Its central
vertex has Euler number  $e= -4$.  The eight pairs of normalized
Seifert invariants $(\alpha_\ell,\omega_\ell)$, $(1\leq \ell\leq 8)$, associated with the eight legs,
are determined as Hirzebruch--Jung  continued fractions associated with the entries of the corresponding legs:
$\alpha_\ell/\omega_\ell$ for $1\leq \ell\leq 4$ is $[2,2,3]=7/5$, while the other four are
$[3]=3/1$; cf. \ref{bek:SEIFERT}.
\ix{Seifert manifold}\ix{Hirzebruch--Jung continued fraction}

Recall that the {\it orbifold Euler number} of the Seifert 3--manifold is defined as $e^{orb}:=
e+\sum_\ell \omega_\ell/\alpha_\ell$, and the  normal form graph is negative definite if and only if $e^{orb}<0$.
In this case  $e^{orb}=4/21>0$, hence the graph \ix{graph!negative definite}
 is {\em not} negative definite.\ix{orbifold Euler number}\ix{Seifert manifold}

In particular, this graph cannot be transformed into a negative definite
graph  by plumbing  calculus.

Note that $G^m$ with opposite  orientation   (that is, $-G^m$) is negative definite.
\ix{graph!negative definite}

\end{example}

\begin{example}\labelpar{nu2b}
For $f=x^2y^2+z^2(x+y)$ a possible graph $\gc$ is given in \ref{nu2}. For a possible $\Gmod$
we get:

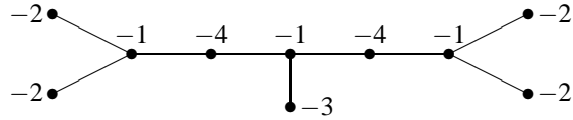
\begin{picture}(200,80)(-50,-10)
\put(30,30){\circle*{4}} \put(60,30){\circle*{4}}
\put(90,30){\circle*{4}} \put(120,30){\circle*{4}}
\put(150,30){\circle*{4}} \put(90,10){\circle*{4}}
\put(0,45){\circle*{4}} \put(0,15){\circle*{4}}
\put(180,45){\circle*{4}} \put(180,15){\circle*{4}}
\put(30,30){\line(1,0){120}}
\put(30,30){\line(-2,1){30}}\put(30,30){\line(-2,-1){30}}
\put(150,30){\line(2,1){30}}\put(150,30){\line(2,-1){30}}
\put(90,30){\line(0,-1){20}}
\put(100,10){\makebox(0,0){$-3$}}
\put(30,37){\makebox(0,0){$-1$}}\put(60,37){\makebox(0,0){$-4$}}
\put(90,37){\makebox(0,0){$-1$}}\put(120,37){\makebox(0,0){$-4$}}
\put(150,37){\makebox(0,0){$-1$}}
\put(-10,45){\makebox(0,0){$-2$}}\put(-10,15){\makebox(0,0){$-2$}}
\put(190,45){\makebox(0,0){$-2$}}\put(190,15){\makebox(0,0){$-2$}}
\end{picture}

\noi (Notice that for this graph it would be possible to use the
$\R\bp^2$--absorption of the {\em non--orientable} calculus, but we do not do that.)

In this case $\partial F$ is a rational homology sphere, since $\det(A)\not=0$.\ix{Milnor!fiber!boundary}
\end{example}

\begin{example}\labelpar{ex:ketA2b'} Assume that $f=y^3+(x^2-z^4)^2$. Then using \ref{ex:ketA2} we get
for $\Gmod$:

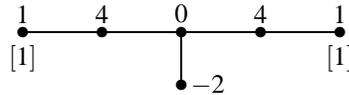
\begin{picture}(200,55)(-50,0)
\put(30,30){\circle*{4}} \put(60,30){\circle*{4}}
\put(90,30){\circle*{4}} \put(120,30){\circle*{4}}
\put(150,30){\circle*{4}} \put(90,10){\circle*{4}}
\put(30,30){\line(1,0){120}}
\put(90,30){\line(0,-1){20}}
\put(100,10){\makebox(0,0){$-2$}}
\put(30,37){\makebox(0,0){$1$}}\put(60,37){\makebox(0,0){$4$}}
\put(90,37){\makebox(0,0){$0$}}\put(120,37){\makebox(0,0){$4$}}
\put(150,37){\makebox(0,0){$1$}}
\put(30,20){\makebox(0,0){$[1]$}}\put(150,20){\makebox(0,0){$[1]$}}

\end{picture}

\end{example}

\begin{example}\labelpar{AnPi2} Finally, the last example is the 1--parameter infinite
family
%. In general, in order to treat such an infinite
%  family uniformly for all the values of the  parameter,
%one needs some ingenuity, or familiarity with the algorithm (and/or familiarity with some other
%algorithms as well, e.g. with that one which determines the resolution graph of cyclic coverings).
%Here we will exemplify the case of
$f=x^ay(x^2+y^3)+z^2$ with $a\geq 2$.
 The reader may consider this as a model for other infinite families.

A graph $\gc$ with  $g=x+y+z$ is given in \ref{AnPi}. \\

\noindent {\bf Case 1.} Assume that  $a$ is even. We  determine $G$ in several steps.

First, the graph $G_1$ can be determined easily (in particular, the normalization of $V_f$
is the $D_5$ singularity):

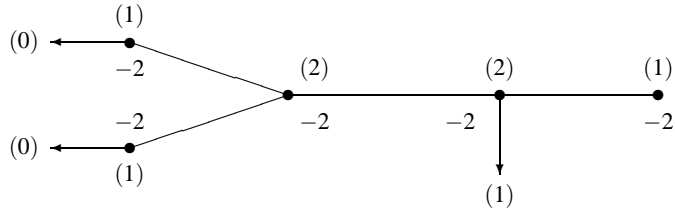
\vspace{7mm}

\begin{picture}(300,80)(-60,110)

% hurok "oszlopa"
\put(40,130){\makebox(0,0){\small{$(1)$}}}\put(40,150){\makebox(0,0){\small{$-2$}}}
\put(40,140){\circle*{4}}
\put(40,180){\circle*{4}}
\put(40,190){\makebox(0,0){\small{$(1)$}}}\put(40,170){\makebox(0,0){\small{$-2$}}}
\put(0,140){\makebox(0,0){\small{$(0)$}}}
\put(0,180){\makebox(0,0){\small{$(0)$}}}
% kozepso oszlop csucsai + sulyok
\put(100,160){\circle*{4}}
\put(110,170){\makebox(0,0){\small{$(2)$}}}
\put(110,150){\makebox(0,0){\small{$-2$}}}

% kovetkezo oszlop csucsai + sulyok
\put(180,160){\circle*{4}}
\put(180,170){\makebox(0,0){\small{$(2)$}}}
\put(165,150){\makebox(0,0){\small{$-2$}}}
% jobbszelso oszlop csucsai + sulyok
\put(240,160){\circle*{4}}
\put(240,170){\makebox(0,0){\small{$(1)$}}}
\put(240,150){\makebox(0,0){\small{$-2$}}}
% elek

\put(100,160){\line(1,0){140}}

\put(40,140){\line(3,1){60}}
\put(40,180){\line(3,-1){60}}
\put(40,140){\vector(-1,0){30}}
\put(40,180){\vector(-1,0){30}}

%  nyilak
\put(180,160){\vector(0,-1){30}}
\put(180,122){\makebox(0,0){\small$(1)$}}

\end{picture}

\noindent \ix{cutting edge} \ix{gluing tori}
Clearly, we have two gluing tori. Let $v_1$ and $v_2$ be the vertices of $G_1$ which support
the (0)--arrows. Next, we wish to determine the multiplicity $m_1$ of the vertex $v_1'$ of $G$, which
is not in $G_1$ and is a neighbour of $v_1$ in $G$. For this we have to analyze the cutting edge with
weights $(a;2a+4,1)$ and $(1;2a+4,1)$. By (\ref{eq:2.2.2}) (pay attention to the left--right
ordering of the ends of the string), $m_1$ satisfies
$a+\lambda= m_1(2a+4)$
for some $\lambda$ with $0\leq \lambda <2a+4$. Hence $m_1=1$. In $G$ the vertex $v_1'$ is glued to $v_1$
by a $\circleddash$--edge, hence the Euler number of $v_1$  in $G$ is $-1$.

Finally, we analyze the graph $\gc^2$. Its shape and the first entries of the weights of the vertices
coincide with the minimal
embedded resolution graph of the (transversal type)
plane curve singularity $u^2+v^a$, provided that we replace $v_1$ and
$v_2$ by arrowheads with multiplicity 1. %(see (\ref{prop:cov}) too).
Comparing the Main Algorithm and the algorithm which provides
the graph of suspension singularities (cf. \ref{ss:b}),
we realize that the part of $G$ above $\gc^2$ is exactly the
resolution graph of $u^2+v^a+w^{2a+4}=0$ with opposite orientation.
More precisely, let
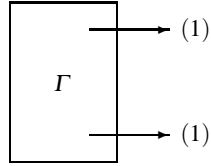
\ix{graph!resolution!Brieskorn}\ix{Main Algorithm}

\vspace{2mm}

\begin{picture}(300,80)(10,0)
\put(35,10){\framebox(40,60){\small{$\Gamma$}}}

\put(65,20){\vector(1,0){30}}
\put(65,60){\vector(1,0){30}}
\put(105,20){\makebox(0,0){\small{$(1)$}}}
\put(105,60){\makebox(0,0){\small{$(1)$}}}

\end{picture}

\vspace{2mm}

\noindent
be the minimal embedded resolution graph of the germ
$w:(\{u^2+v^a+w^{2a+4}=0\},0)\to (\bfc,0)$, induced by the projection
$(u,v,w)\mapsto w$. Let $-\Gamma$  be this graph with opposite orientation (in which one
changes the signs of all Euler numbers and edge--decorations, and keeps the
multiplicities). Then the graph of the open book decomposition
of $(\partial F,V_g)$ is obtained by gluing $-\Gamma$ with $G_1$ such that
the arrows of $-\Gamma$ are identified with $v_1$ and $v_2$
(and the Euler numbers of $v_1$ and $v_1$ are recomputed as
above, or via (\ref{eq:2.2.1}) using  the multiplicities):

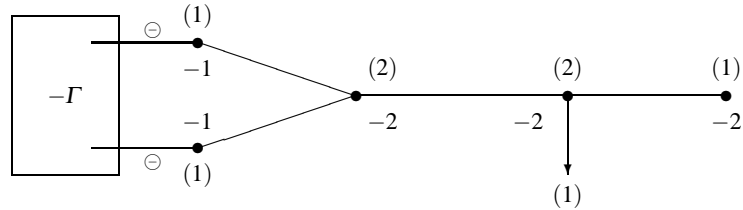
\begin{picture}(300,110)(-60,100)
\put(-30,130){\framebox(40,60){\small{$-\Gamma$}}}

\put(0,140){\line(1,0){40}}
\put(0,180){\line(1,0){40}}
\put(23,135){\makebox(0,0){\small{$\circleddash$}}}
\put(23,185){\makebox(0,0){\small{$\circleddash$}}}

% hurok "oszlopa"
\put(40,130){\makebox(0,0){\small{$(1)$}}}\put(40,150){\makebox(0,0){\small{$-1$}}}
\put(40,140){\circle*{4}}
\put(40,180){\circle*{4}}
\put(40,190){\makebox(0,0){\small{$(1)$}}}\put(40,170){\makebox(0,0){\small{$-1$}}}
%\put(0,140){\makebox(0,0){\small{$(0)$}}}
%\put(0,180){\makebox(0,0){\small{$(0)$}}}
% kozepso oszlop csucsai + sulyok
\put(100,160){\circle*{4}}
\put(110,170){\makebox(0,0){\small{$(2)$}}}
\put(110,150){\makebox(0,0){\small{$-2$}}}
% kovetkezo oszlop csucsai + sulyok
\put(180,160){\circle*{4}}
\put(180,170){\makebox(0,0){\small{$(2)$}}}
\put(165,150){\makebox(0,0){\small{$-2$}}}
% jobbszelso oszlop csucsai + sulyok
\put(240,160){\circle*{4}}
\put(240,170){\makebox(0,0){\small{$(1)$}}}
\put(240,150){\makebox(0,0){\small{$-2$}}}
% elek

\put(100,160){\line(1,0){140}}

\put(40,140){\line(3,1){60}}
\put(40,180){\line(3,-1){60}}
%\put(40,140){\vector(-1,0){30}}
%\put(40,180){\vector(-1,0){30}}

\put(180,160){\vector(0,-1){30}}
\put(180,122){\makebox(0,0){\small$(1)$}}

\end{picture}

This graph has a cycle. Moreover, $\Gamma$ is a star--shaped graph whose central
vertex has genus $\frac{{\rm gcd}(a,4)}{2}-1$. Determining  $\Gamma$ is standard, see
\ref{ss:b}, or \cite{OW}. \\

\noindent {\bf Case 2.} Assume that $a$ is odd, $a\geq 3$. We proceed similarly as above.
The graph $G_1$ is the following:

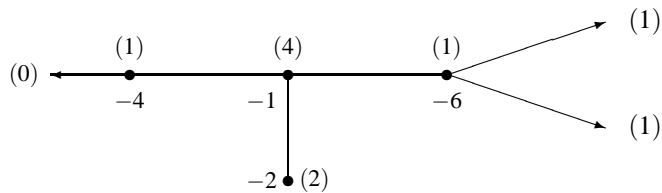
\begin{picture}(300,90)(-60,105)
% hurok "oszlopa"
 \put(40,160){\circle*{4}}
\put(40,170){\makebox(0,0){\small{$(1)$}}}\put(40,150){\makebox(0,0){\small{$-4$}}}
\put(0,160){\makebox(0,0){\small{$(0)$}}}
% kozepso oszlop csucsai + sulyok
 \put(100,160){\circle*{4}}
\put(100,170){\makebox(0,0){\small{$(4)$}}}\put(90,150){\makebox(0,0){\small{$-1$}}}
 \put(100,120){\circle*{4}}
\put(110,120){\makebox(0,0){\small{$(2)$}}}\put(90,120){\makebox(0,0){\small{$-2$}}}

% kovetkezo oszlop csucsai + sulyok
\put(160,160){\circle*{4}}
\put(160,170){\makebox(0,0){\small{$(1)$}}}\put(160,150){\makebox(0,0){\small{$-6$}}}

% jobbszelso nyilak
\put(160,160){\vector(3,1){60}} \put(160,160){\vector(3,-1){60}}
\put(235,180){\makebox(0,0){$(1)$}}
\put(235,140){\makebox(0,0){$(1)$}}

\put(100,120){\line(0,1){40}}

\put(160,160){\vector(-1,0){150}}
\end{picture}

\noindent There is only one gluing torus. Let $v$ be the $(-4)$--vertex, and $v'$
its adjacent vertex of $G$ which is not in $G_1$. Then the multiplicity of $v'$ is again 1.
Hence, the Euler number of $v$ in $G$ is $-3$. Therefore, the graph
of $(\partial F,V_g)$ is  \ix{gluing tori}

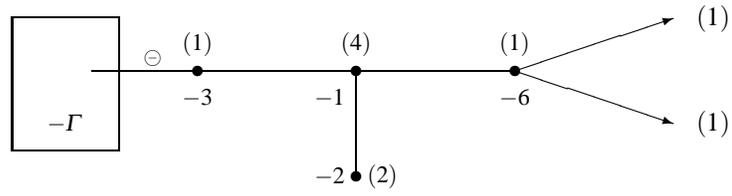
\vspace{2mm}

\begin{picture}(300,90)(-60,110)
\put(-30,130){\framebox(40,50){}}
%\put(0,180){\line(1,0){40}}
\put(23,165){\makebox(0,0){\small{$\circleddash$}}}
\put(-10,140){\makebox(0,0){\small{$-\Gamma$}}}

% hurok "oszlopa"
 \put(40,160){\circle*{4}}
\put(40,170){\makebox(0,0){\small{$(1)$}}}\put(40,150){\makebox(0,0){\small{$-3$}}}
%\put(0,160){\makebox(0,0){\small{$(0)$}}}
% kozepso oszlop csucsai + sulyok
 \put(100,160){\circle*{4}}
\put(100,170){\makebox(0,0){\small{$(4)$}}}\put(90,150){\makebox(0,0){\small{$-1$}}}
 \put(100,120){\circle*{4}}
\put(110,120){\makebox(0,0){\small{$(2)$}}}\put(90,120){\makebox(0,0){\small{$-2$}}}

% kovetkezo oszlop csucsai + sulyok
\put(160,160){\circle*{4}}
\put(160,170){\makebox(0,0){\small{$(1)$}}}\put(160,150){\makebox(0,0){\small{$-6$}}}

% jobbszelso nyilak
\put(160,160){\vector(3,1){60}} \put(160,160){\vector(3,-1){60}}
\put(235,180){\makebox(0,0){$(1)$}}
\put(235,140){\makebox(0,0){$(1)$}}

\put(100,120){\line(0,1){40}}

\put(160,160){\line(-1,0){160}}
\end{picture}

\vspace{4mm}

\noindent where $\Gamma$ is the
minimal embedded resolution graph of the germ $w:(\{u^2+v^a+w^{2a+4}=0\},0)\to (\bfc,0)$,
 and the unique arrow--head of $-\Gamma$ is identified with the
 $(-3)$--vertex. %Note that  $\partial F$ is a rational homology sphere.

\end{example}

\chapter{Proof of the Main Algorithm}\labelpar{sec:proof}

\section{\ Preliminary remarks}\setcounter{equation}{0}
\bekezdes The algorithm and its proof  is a highly generalized version of the algorithm which
determines the resolution graph of cyclic coverings.  Its  origin
goes back to the case of suspensions, when one starts with
an isolated plane curve singularity  $f'$ and a positive integer $n$, and
one determines the resolution graph of the hypersurface singularity
$\{f'(x,y)+z^n=0\}$ from the embedded resolution graph of $f'$ and the
integer $n$; see \ref{ss:b}.
 \ix{graph!covering}\ix{Main Algorithm!proof of}

All the geometrical constructions  behind
the algorithms targeting cyclic coverings are realized
within the framework  of  complex
analytic/algebraic geometry. In particular, all the graphs
involved are negative definite graphs  and the plumbing calculus
reduces to blowing up/down $(-1)$--rational curves. Moreover,
the following  general  principle applies:
for normal surface singularities the resolution graph
is a possible plumbing graph for the link, which is diffeomorphic with
the boundary of the Milnor fiber of any smoothing.

\bekezdes
The first case when  a more complicated
`aid--graph' was used  is in \cite{eredeti}. The starting situation
was the following: having a germ $f$ with 1--dimensional singular locus, and
another germ $g$ such that the pair $(f,g)$ forms an ICIS, one
wished to determine  the resolution graphs of the hypersurface singularities  $f+g^k$,
$k\gg 0$, cf. \ref{bek:INWO}.
In order to find these `usual' --- that is,   negative definite ---
graphs, all the
necessary information about the  ICIS $(f,g)$ was stored in the
`unusual'  decorations of the `unusual' graph $\gc$.

The machinery and construction developed in that article serve
as a model for the present work. We start again with
the very same graph $\gc$, but rather significant differences appear.
Although, in \cite{eredeti},  the entire construction  stayed within the realm  of
complex analytic geometry, similarly as in the case of cyclic coverings,
the present case grows  out of  the complex analytic world.  We must glue
together real analytic spaces with singularities, and sometimes the
gluing maps even  reverse the `canonical' orientations of the
regular parts. This generates additional difficulties we
need to be handle during the proof. The output plumbing graphs
are  `general plumbing graphs', which may not be definite, or not
even non--degenerate). Moreover,  we had to consider a larger set of moves
of the smooth plumbing calculus (not standard in
algebraic geometry) in order to simplify them  or to
reduce them  to their  `normal forms'.

The explanation of the idea why the graph $\gc$ contains all
the information needed to describe $\partial F$ is given in\ix{Milnor!fiber!boundary}
\ref{why}. In fact, that  is the main idea behind the whole
construction. In the next section we outline the main steps of the proof.

\section{\ The guiding principle and the outline of the proof}\setcounter{equation}{0}
\label{outline}

Consider an ICIS $\Phi=(f,g)$ as in \ref{ss:ICIS}, an
embedded resolution $r:V^{emb}\to (\bfc^3,0)$ of the divisor
$V_{f}\cup V_g$ as in \ref{construct},  as well as  a `wedge'
$\Wedge$ of $\Delta_1$ for some $M\gg 0$ as in \ref{why}.

If one has a complex analytic isolated singularity
$(\cals,0)\subset (\bfc^3,0)$ for which $\Phi(\cals)\setminus
\{0\}\subset \Wedge$ then one can construct a resolution of
$\cals$ in three steps.

 First, consider the $r$--strict transform $\widetilde{\cals}\subset V^{emb}$ of $\cals$.
It is contained  in a tubular neighbourhood of $\C$, cf. (\ref{zart}), and
its singular locus is in $\C$. Therefore $\widetilde{\cals}$  can be resolved in two
further steps: first taking the normalization $\cals^{norm}$ of
$\widetilde{\cals}$, then resolving the isolated normal surface
singularities of $\cals^{norm}$. The point is that if $\cals$ is
determined by $f$ and $g$, then $\widetilde{\cals}$ has nice
local equations near any point of $\C$ (which can be recovered
from the decorations of $\gc$). E.g., one can show that $\widetilde{\cals}$ is an
equisingular family of curves along the regular part of $\C$, hence
 the singular locus of ${\cals}^{norm}$ will be situated  above  the
double points of $\C$. Moreover,  all these singular points will
be of Hirzebruch--Jung type. In particular, the last step is the
resolution of these Hirzebruch--Jung  singularities, whose combinatorial data is
again codified in $\gc$.

Summing up we get the diagram:

\vspace{3mm}

\begin{picture}(200,40)(0,0)
\put(60,30){\vector(1,0){20}} \put(115,30){\vector(1,0){20}}
\put(163,30){\vector(1,0){20}}
\put(50,30){\makebox(0,0){$\overline{\cals}$}}
\put(100,30){\makebox(0,0){${\cals}^{norm}$}}
\put(150,30){\makebox(0,0){$\widetilde{\cals}$}}
\put(200,30){\makebox(0,0){${\cals}$}}
\put(150,17){\makebox(0,0){$\cap$}}
\put(200,17){\makebox(0,0){$\cap$}}
\put(150,5){\makebox(0,0){$V^{emb}$}}
\put(200,5){\makebox(0,0){$(\bfc^3,0)$}}
\put(72,35){\makebox(0,0){\tiny{HJ}}}
\put(175,35){\makebox(0,0){$r$}}
\end{picture}

\vspace{5mm}

\noi Corresponding  to the above three horizontal maps,
we have the following steps at the level of graphs:

\begin{itemize}
\item start with the graph $\gc$ (which stores all the local information
about $\widetilde{\cals}$);
\item provide a cyclic covering graph (in the sense of Chapter \ref{ss:2.3})
corresponding  to the normalization step;
\item modify this graph by Hirzebruch--Jung strings (see `variation' \ref{re:2.3.1}).
\end{itemize}\ix{graph!covering}\ix{graph!resolution!string}

A key additional argument is a consequence of Theorem
\ref{th:2.3.1}, which guarantees the uniqueness  of the cyclic
covering graph with the inserted strings.

It is  exactly this guiding principle that was used in
\cite{eredeti} to determine the resolution graph  of any member of the generalized
Iomdin--series $\cals=\{f+g^k=0\}$,  $k\gg 0$.

\vspace{2mm}

Now, we want to obtain the plumbing-graph of the boundary $\partial
F$ of the Milnor fiber of a non--isolated $f$. We show in Proposition  \ref{lem:sksmooth}  that
$\partial F$ is the link of the  {\it real} analytic
germ\ix{Milnor!fiber!boundary}
$${\cals}_k =\{ f=|g|^k\}\subset(\bfc^3,0),$$
and $\Phi(\cals_k)\setminus \{0\}\subset \Wedge$, provided that
$k\gg 0$. Hence, we will run the same procedure as above
within  the world of real analytic geometry, which
forces  some modifications.

A final remark:  the Euler number of an $S^1$-bundle over a curve
is a `global object', its computation in a resolution can be
rather involved (one needs more charts and gluing information connecting them).
Therefore, we will determine the Euler numbers of
our graphs in  an indirect way: we  consider the open book
decomposition induced by $g$, and we  determine the associated
multiplicity system (this can easily be determined from local
data!), then we apply (\ref{eq:2.2.1}).

\section{\ The first step.
The real varieties $\cals_k$}\labelpar{ss:geocon}\setcounter{equation}{0}

 We fix a pair
$\Phi=(f,g)$ as in section \ref{ss:ICIS}, and we use all the
notations and results of that part. In particular, we  fix
a good representative of $\Phi$ whose discriminant is
$\Delta_\Phi$. Similarly as above, we write $(c,d)$ for the
coordinates of $(\bfc^2,0)$.

For any  {\it even} integer $k$ (compare also with \ref{re:k})
we set
$$Z_k:=\{(c,d)\in (\bfc^2,0)\,:\, c=|d|^k\}.$$
The next lemma is elementary and its proof is left to the reader.
\begin{lemma}\labelpar{lem:zk}
$Z_k$ is a smooth real analytic (even algebraic) surface. For
$k$ sufficiently large $Z_k\cap \Delta_\Phi=\{0\}$.  Moreover,
$Z_k\setminus \{0\}\subset \Wedge$ if $k>M$.
% (For the definition of \,$\Wedge $ see (\ref{why}).)
\end{lemma}

\begin{remark}\labelpar{re:k}
%(a) In this article we adopt the following terminology: `smooth
%variety' means a regular complex analytic variety. The  singular
%locus of a complex variety is the set of non-regular points. In
%the real world we will deal with `$C^\infty$ manifolds', and  the
%`singular locus' of a real analytic space is the set of
%non--$C^\infty$ points.
%
%(b)
As mentioned before, all the important facts regarding $Z_k$ (and the space $\cals_k$
which will be defined next) are valid for $k$ odd as well. This is based on the
additional fact that
the  classification of oriented 2-- and 3--dimensional topological
manifolds agrees with the classification of $C^\infty$ manifolds.
Nevertheless,  it is more convenient to use even integers $k$, since for them
$|d|^k$ becomes real algebraic. In fact, later we will impose even more
divisibility assumptions on $k$.

The point is that $k$ has only an auxiliary role and carries no
geometric meaning, e.g., it will not appear in any `final' formula
of $\partial F$. Hence its value, as soon as it is sufficiently
large,  is completely unimportant.\ix{Milnor!fiber!boundary}
\end{remark}

Next, define the real analytic variety $\cals_k$ of real dimension
4 by
$$\cals_k:=\Phi^{-1}(Z_k)=\{z\in(\bfc^3,0)\,:\, f(z)=|g(z)|^k\}.$$
\begin{proposition}\labelpar{lem:sksmooth} For $k$ sufficiently large,
the real variety $\cals_k\setminus \{0\}$ is regular, hence it is
a smooth  oriented 4--dimensional manifold. Moreover, for
 sufficiently small $\epsilon>0$, the sphere $S_\epsilon=S^5_\epsilon$
intersects $\cals_k$ transversally. The intersection $\cals_k\cap
S_\epsilon$ is an oriented 3--manifold, which is  diffeomorphic by an orientation
preserving  diffeomorphism to $\partial F$.

 In
particular, the link of $\cals_k$ (i.e. $\cals_k\cap
S_\epsilon$) is independent of the choice of $k$.
\end{proposition}
Before we start the proof let us indicate how the orientation of
$\cals_k\setminus\{0\}$ is defined. First,  consider $Z_k$.
It is a smooth real manifold. The projection on the $d$--axis  induces a
diffeomorphism; we define the orientation of $Z_k$ by the pullback of the
complex orientation of the $d$--axis via this diffeomorphism. Next, all fibers
of $\Phi$ are complex curves with their natural orientation. On the smooth
part of $\Phi^{-1}(Z_k)$ we define  the product orientation of the base and fibers.

\begin{proof} The first statement  follows from  Lemma \ref{lem:zk} and from the
properties of the ICIS $\Phi$ (or by a direct computation). The second one is
standard, using for example  the `{\it curve selection lemma}' from \cite{MBook}.

\vspace{2mm}

Next, we prove the diffeomorphism  $\cals_k\cap S_\epsilon\simeq \partial F$.

First, recall that in certain   topological arguments regarding
the Milnor fiber of $f$, the sphere $S_\epsilon=\partial
B_\epsilon$ is replaced by the 5--manifold with corners
 $\partial(\Phi^{-1}(D^2_\eta)\cap B_\epsilon)$,
the Milnor fiber $F=\{f=\delta\}\cap B_\epsilon$ by $F^{\tib}
:=\{f=\delta\}\cap \Phi^{-1}(D^2_\eta)\cap B_\epsilon$, and the
boundary $\partial F$ by the boundary with corners  $\partial
F^\tib$. For details, see e.g. \cite{Lj}, or
\ref{rem:phif}. By a similar argument, one shows the equivalence of
$\Phi^{-1}(Z_k)\cap B_\epsilon$ with $\cals^\tib _k:=
\Phi^{-1}(Z_k\cap D^2_\eta)\cap B_\epsilon$, and
$\Phi^{-1}(Z_k)\cap S_\epsilon$ with the 3--manifold with
corners $\partial \cals^\tib _k$. Hence, we need only to show the
equivalence of $\partial F^\tib $ and $\partial \cals^\tib _k$.

Consider the intersection $Z_k\cap \partial D^2_\eta$, i.e., the
solution of the system $\{|c|^2+|d|^2=\eta^2;\ c=|d|^k\}$. It is a
circle along which  $c$ is constant; let this value of $c$
(determined by $\eta$ and $k$) be denoted by $c_0$. Set
 $D_{c_0}=\{c=c_0\}\cap D^2_\eta$ as in  \ref{ss:ICIS}. Then
\begin{equation}\label{eq:par}
\partial D_{c_0}=\partial (Z_k\cap D^2_\eta),
\end{equation}
and
$\partial F^\tib=\partial (\Phi^{-1}(D_{c_0}))$ has a
decomposition:
$$\partial  F^\tib=
\Phi^{-1}(\partial D_{c_0})\cap B_\epsilon \bigcup _{
\Phi^{-1}(\partial D_{c_0})\cap S_\epsilon} \Phi^{-1}(
D_{c_0})\cap S_\epsilon.$$ Via (\ref{eq:par}), $\partial
\cals_k^\tib $ has a decomposition
$$\partial  \cals_k^\tib =
\Phi^{-1}(\partial D_{c_0})\cap B_\epsilon \bigcup _{
\Phi^{-1}(\partial D_{c_0})\cap S_\epsilon} \Phi^{-1}( Z_k\cap
D^2_\eta)\cap S_\epsilon.$$ Notice that there is an isotopy of
$D^2_\eta$, preserving $\partial D^2_\eta$,
 which sends  $D_{c_0}$ into $Z_k\cap D^2_\eta$. Since the restriction of
   $\Phi$ on $\Phi^{-1}(D^2_\eta)\cap S_\epsilon$ is a trivial fibration
over $D^2_\eta$, this isotopy can be lifted. This identifies the pairs
$$( \Phi^{-1}( D_{c_0})\cap S_\epsilon,\Phi^{-1}(\partial D_{c_0})\cap
S_\epsilon)\simeq (\Phi^{-1}( Z_k\cap D^2_\eta)\cap
S_\epsilon, \Phi^{-1}(\partial D_{c_0})\cap S_\epsilon).$$
This ends the proof.
\end{proof}

\section{\ The strict transform $\wsk$ of $\cals_k$ via $r$}\labelpar{ss:sklifted}
\setcounter{equation}{0}

Consider the resolution $r:V^{emb}\to U$ as in \ref{construct}.
Let $\wsk$ be the strict transform of $\cals_k$ by $r$, i.e.
$\wsk$ is the closure  of
$r^{-1}(\cals_k\setminus \{0\})$ (in the eucleidian topology).

\begin{lemma}\labelpar{lem:skc}
$$\wsk\cap r^{-1}(0)=\C.$$\end{lemma}
\begin{proof}
The proof is similar to the proof of (\ref{zart}), and it is left to the
reader.
\end{proof}
Since the restriction of $r$ induces a diffeomorphism
$\wsk\setminus \C\to \cals_k\setminus\{0\}$, we get that the
singular locus of $\wsk$ satisfies
$$Sing(\wsk)\subset \C.$$
Moreover, $r$ induces a diffeomorphism between $\partial \cals_k$
(the subject of our interest) and   $\partial \wsk$. Since $\wsk$
can be replaced by its intersection with an arbitrarily  small
tubular neighbourhood of $\C$, the boundary $\partial \wsk$ can be
localized totally near  $\C$. In fact, this is the main advantage
of the space $\cals_k$: in this way, the wanted 3--manifold appears
as a local link, or, after a resolution, as the boundary of a
tubular neighbourhood of a curve configuration. \ix{curve configuration $\C$}

Next, we analyze the local equations of $\wsk$ in the
neighbourhood of any point of $\C$. For this we use the notations
of \ref{summary}. In all the cases, $U_p$ is a complex 3--ball
around the point $p\in\C$ with three complex local coordinates
$(u,v,w)$.

It is convenient to use the following notation. If $H=\{(u,v,w)\in
U_p\,:\, h(u,v,w)=0\}$ is a real analytic variety in $U_p$, then
we denote by $H^+$ the closure of $H\setminus \{uvw=0\}$. This
way we neglect those components of $H$ which are included  in one of the
coordinate planes. Using this notation, the local equations of
$\wsk$ are as follows.

If $p$ is a generic point of a component $C$ of $\C$ with
decoration $(m;n,\nu)$, then
\begin{equation}\label{eq:loc1}
\wsk\cap U_p=\{(u,v,w)\,:\, u^mv^n=|v|^{\nu k}\}^+= \{(u,v,w)\,:\,
u^m=v^{\frac{\nu k}{2}-n}\bar{v}^{\frac{\nu k}{2}}\}\end{equation}
with $m,\nu>0$.

If $p$ is an intersection (singular) point of $\C$
of type 1 (i.e. if the corresponding edge has decoration 1), then
\begin{equation}\label{eq:loc2}
\wsk\cap U_p=\{(u,v,w)\,:\, u^mv^nw^l=|v|^{\nu k}|w|^{\lambda
k}\}^+\end{equation}
with $m,\nu,\lambda>0$.

 Finally, if $p$ is an intersection (singular) point of $\C$ of
type 2, then
\begin{equation}\label{eq:loc3}
\wsk\cap U_p=\{(u,v,w)\,:\, u^mv^{m'}w^n=|w|^{\nu
k}\}^+\end{equation} with $m,m',\nu>0$.

\section{\ Local complex algebraic models for the points of $\wsk$}\labelpar{ss:BIR}
\setcounter{equation}{0}

Notice that for $k\gg 0$ and for $p$ as in
(\ref{eq:loc1})--(\ref{eq:loc2})--(\ref{eq:loc3}), $\wsk\cap U_p$
is a \textit{real} algebraic variety. We will show that any such
germ is homeomorphic with the germ of a certain  \textit{complex}
algebraic hypersurface. In these computations we will assume that
$k/2$ is a multiple of all the integers appearing in the decorations
of $\gc$. More precisely: whenever in the next discussion a fraction $k/l$
appears for some $l$, then we  assume that $k/l$ is, in fact, an
even integer.

In the next paragraphs $U$ will denote a local neighbourhood of
the origin in $\bfc^3$.

\begin{bekezdes}\labelpar{lh:1} Assume that
{\bf $p$ is a generic point of $\C$} as in (\ref{eq:loc1}).
Consider the map
\begin{equation*} %\label{eq:psi1}
\psi_p:\{(x,y,z)\in U\,:\, x^m=y^n\}\, \longrightarrow \,
\{(u,v,w)\in U_p\,:\, u^mv^n=|v|^{\nu k}\}^+
\end{equation*}
given by the correspondences
\begin{equation}\label{EQ:psi1}
\left\{\begin{array}{lll}
u & = & x^{-1}|y|^{\nu k/m} \\
v & = & y\\ w & = & z
\end{array}\right. \hspace{1cm}
\left\{\begin{array}{lll}
x & = & u^{-1}|v|^{\nu k/m} \\
y & = & v\\ z & = & w.
\end{array}\right.
\end{equation}
Then $\psi_p$ is  regular real algebraic (i.e. it extends over
$x=0$ too), it is birational and a homeomorphism. Moreover, it is a
partial normalization of $\wsk\cap U_p$, i.e. the coordinates
$x,y,z$ of $\{x^m=y^n\}\cap U$ are integral over the ring of regular
functions of $\wsk\cap U_p$. Indeed, birationality follows from
the fact that the second set of equations provides  the inverse of
the first one, and regularity follows from a limit computation, or by
rewriting the first equation into $u=x^{-1}|x|^{\nu k/n}$. This
formula also shows that  $\psi_p$ is  bijective and  a
homeomorphism. Moreover, since $x^m=v^n$, $x$ is integral over the
ring of regular functions of $\wsk\cap U_p$ (a similar statement for
$y$ and $z$ is trivial).

In particular, the normalizations of the source and of the target of
$\psi_p$ canonically coincide.
\end{bekezdes}

\begin{bekezdes}\labelpar{lh:2} Assume that
{\bf $p$ is a singular  point of $\C$ of type 1} as in
(\ref{eq:loc2}). Consider the map
\begin{equation*} %\label{eq:psi2}
\psi_p:\{(\alpha,\beta,\gamma)\in U\,:\,
\alpha^m=\beta^n\gamma^l\}\, \longrightarrow \, \{(u,v,w)\in
U_p\,:\, u^mv^nw^l=|v|^{\nu k}|w|^{\lambda k}\}^+
\end{equation*}
given by
\begin{equation}\label{EQ:psi2}
\left\{\begin{array}{lll}
u & = & \alpha^{-1}|\beta|^{\nu k/m}|\gamma|^{\lambda k/m} \\
v & = & \beta\\ w & = & \gamma
\end{array}\right. \hspace{1cm}
\left\{\begin{array}{lll}
\alpha & = & u^{-1}|v|^{\nu k/m}|w|^{\lambda k/m} \\
\beta & = & v\\ \gamma & = & w.
\end{array}\right.
\end{equation}
Then, again,   $\psi_p$ is  regular real algebraic, birational, and
additionally, it is a homeomorphism. Moreover, it is a partial normalization of
$\wsk\cap U_p$, i.e. the coordinates $\alpha,\beta,\gamma$ are
integral over the ring of regular functions of $\wsk\cap U_p$. Indeed, the
regularity follows from
$$u=\alpha^{m-1}\overline{\alpha}^m |\beta|^{\nu k/m-2n}|\gamma|^{\lambda
k/m-2l},
$$
where $k\gg 0$ and $m>0$. Moreover, $\alpha$ is integral over the
ring of $\wsk\cap U_p$ since $\alpha^m=v^nw^l$.
\ix{singularities!Hirzebruch--Jung}

Hence again, the normalizations of the source and the target of
$\psi_p$ canonically coincide.
\end{bekezdes}

\begin{bekezdes}\labelpar{lh:3} Assume that
{\bf $p$ is a singular  point of $\C$ of type 2} as in
(\ref{eq:loc3}). In this case we can prove considerably less (from the
analytic point of view). We consider  the map
\begin{equation*} %\label{eq:psi2}
\psi_p:\{(\alpha,\beta,\gamma)\in U\,:\,
\alpha^n=\beta^m\gamma^{m'}\}\, \longrightarrow \, \{(u,v,w)\in
U_p\,:\, u^mv^{m'}w^n=|w|^{\nu k}\}^+
\end{equation*}
given by
\begin{equation}\label{EQ:psi3}
\left\{\begin{array}{lll}
u & = & \beta^{-1}|\beta|^{\nu k/n} \\
v & = & \gamma^{-1}|\gamma|^{\nu k/n} \\
w & = & \alpha.
\end{array}\right. \hspace{1cm}
\end{equation}
It is  regular real algebraic and a homeomorphism, but it is {\em
not} birational.

\bekezdes  Notice also that the above maps,
in all three cases,  preserve the coordinate axes.
\end{bekezdes}
\ix{singularities!Hirzebruch--Jung}

\section{\ The normalization $\cals_k^{norm}$ of $\wsk$}\labelpar{ss:NORM}
\setcounter{equation}{0}

\begin{bekezdes}\labelpar{NORM1}
Let $n_\cals:\nsk\to \wsk$ be the normalization of $\wsk$; for its existence, see
\cite{BCR}.
Since
the normalization is compatible with restrictions on smaller open
sets, we get the globally defined $\nsk$ whose restrictions above
an open set of type $\wsk\cap U_p$ are the normalization of that
$\wsk\cap U_p$. In particular, the local behaviour of the
normalization $\nsk$ over the different open neighbourhoods
$\wsk\cap U_p$ can be tested in the charts considered in the
previous section.

In the first case, if $p$ {\bf is a generic point} of $\C$, and $\psi_p$ is the
`partial normalization' from  \ref{lh:1}, then it  induces an
isomorphism of normalizations:
\begin{equation*} %\label{eq:psi1}
\psi_p^{norm}:\{(x,y,z)\in U\,:\, x^m=y^n\}^{norm}\, \longrightarrow \,
\{(u,v,w)\in U_p\,:\, u^mv^n=|v|^{\nu k}\}^{+,norm}.
\end{equation*}
Since the left hand side is smooth, we get that $\nsk$ is smooth
over the regular points of $\C$, hence, after normalization, only
finitely many singular points  survive in $\cals_k^{norm}$, and they are situated
above the double points of $\C$.

If  $p$ {\bf is a  double  point of $\C$ of type 1}, then  $\psi_p$ from
\ref{lh:2} induces again an isomorphism at the level of normalizations:
\begin{equation*} %\label{eq:psi2}
\psi_p^{norm}:\{(\alpha,\beta,\gamma)\,:\,
\alpha^m=\beta^n\gamma^l\}^{norm}\, \longrightarrow \,
\{(u,v,w)\,:\, u^mv^nw^l=|v|^{\nu k}|w|^{\lambda k}\}^{+,norm}.
\end{equation*}
Hence,  the singular points in $\nsk$, situated above the double
points of $\C$ of type 1,  are equivalent with complex analytic
singularities of Hirzebruch-Jung type. Recall that these
singularities  are determined completely combinatorially (e.g. by
the integers $m,n,l$ above), and by the above
chart,  this combinatorial data can  also be recovered from
$\gc$.\ix{singularities!Hirzebruch--Jung}
\end{bekezdes}
\begin{bekezdes}\labelpar{NORM2}
We emphasize that the  two types of charts above in  \ref{NORM1}
are compatible. By this we mean the following: consider a double
point $p$ of $\C$ of type 1, and  a neighbourhood $U_p$ as above.
Then $\C\cap U_p$ is the union of the $v$ and $w$ axis. Let $q$ be
a generic point of $\C\cap U_p$ and consider a sufficiently small
local neighbourhood $U_q\subset U_p$ (where we denote this
inclusion by $j$), and consider also the chart $\psi_p$ over $\wsk\cap U_p$ as in
\ref{lh:1}, respectively $\psi_q$ over $\wsk\cap U_q$ as in
\ref{lh:2}. Then $\psi_p^{-1}\circ j\circ \psi_q$ is a complex
analytic isomorphism onto its image which at the level of
normalization induces an isomorphism of complex analytic smooth
germs.
%In particular, it preserves the `canonical' orientation of the
%involved complex analytic charts.
Indeed, if $q$ is a generic point of the $w$--axis, with non--zero
$w$--coordinate,  and the inclusion
\begin{equation*} %\label{eq:psi1}
 \{(u',v',w')\in U_q\,:\, (u')^m(v')^n=|v'|^{\nu k}\}^{+}
\stackrel{j}{\longrightarrow} \, \{(u,v,w)\,:\, u^mv^nw^l=|v|^{\nu
k}|w|^{\lambda k}\}^{+}
\end{equation*}
is given by $u=u'(w')^{-l/m}|w'|^{\lambda k/m}$, $v=v'$ and
$w=w'$, then  $\psi_p^{-1}\circ j\circ \psi_q$ is given by
$\alpha=xz^{l/m}$, $\beta=y$ and $\gamma=z$. Then the
normalizations tautologically coincide. E.g., assume gcd$(m,n)=1$
and  take the free variables $(t,\gamma)$ normalizing
$\{\alpha^m=\beta^n\gamma^l\}$ by $\alpha=t^n\gamma^{l/m}$,
$\beta=t^m$ and $\gamma=\gamma$. Similarly, consider the free
variables $(s,z)$ normalizing $\{x^m=y^n\}$ by $x=s^n$, $y=s^m$
and $z=z$. Then $(\psi_p^{-1}\circ j\circ \psi_q)^{norm}$ is $t=s$
and $\gamma=z$.

In particular, {\em the two complex charts $\psi_p^{norm}$ and
$\psi_q^{norm}$ of \ref{NORM1} induce the same orientation on
their images, they identify the inverse image  of $\C$ by the same
orientation and induce on a normal slice of $\C$ the same
orientation.}  Note that these are the key gluing--data for a plumbing
construction.
\end{bekezdes}

\begin{bekezdes}\labelpar{eq:psi3} On the other hand, {\bf if $p$ is a singular point of
$\C$ of type 2}, then $\psi_p$ from \ref{lh:3} {\it does  not} induce an
analytic isomorphism, since $\psi_p$ itself is not birational. In this
case, \ref{lh:3} implies  that at the level of normalizations
the induced map
\begin{equation*}
\psi_p^{norm}:\{(\alpha,\beta,\gamma)\,:\,
\alpha^n=\beta^m\gamma^{m'}\}^{norm}\, \longrightarrow \,
\{(u,v,w)\,:\, u^mv^{m'}w^n=|w|^{\nu k}\}^{+,norm}
\end{equation*}
is regular and a  homeomorphism. Nevertheless, one  can prove
slightly more:
\end{bekezdes}
\begin{lemma}\labelpar{Lem:CHART}
$\psi_p^{norm}$  induces a diffeomorphism
over $U_p\setminus \{0\}$.
% (This is contained in the regular/smooth parts of the normalizations.)
% which {\rm reverses} the `canonical'
%orientations of the two complex analytic charts.
\end{lemma}
\begin{proof}
Let $p$ be a double point of $\C$ of type 2 as in \ref{lh:3}.
Then $\C\cap U_p$ is the union of the $u$ and $v$ axes. Let $q$ be
a generic point on the $v$ axis,
%(i.e. with $v$--coordinate non--zero),
--- the other case is completely symmetric. Then $q$ is in the image of the
following  map
\begin{eqnarray*} %\label{eq:psi2}
\lefteqn{\varphi_{p,v}:\{(x,y,z)\in U'\setminus \{z=0\}\,:\,
x^n=y^mz^{m'}\}\, \longrightarrow } \hspace{1cm}\\
& & \hspace{1cm} \{(u,v,w)\in U_p\setminus \{v=0\}\,:\,
u^mv^{m'}w^n=|w|^{\nu k}\}^+
\end{eqnarray*}
given by the correspondences
\begin{equation}\label{EQ:psi4}
\left\{\begin{array}{lll}
u & = & y^{-1}|x|^{\nu k/m} \\
v & = & z^{-1}\\ w & = & x
\end{array}\right. \hspace{1cm}
\left\{\begin{array}{lll}
y & = & u^{-1}|w|^{\nu k/m} \\
z & = & v^{-1}\\ x & = & w.
\end{array}\right.
\end{equation}
Then $\varphi_{p,v}$ is  regular on $U_p\setminus \{z=0\}$, since
$$u=|x|^{\nu k/m-2n} y^{m-1}\overline{y}^m |z|^{2m'},$$
it is birational (its inverse is given by the second set of
equations of (\ref{EQ:psi4})), and it is a partial normalization,
since $y^m=w^nv^{m'}$. Therefore,
\begin{equation*} %\label{eq:psi2}
\varphi_{p,v}^{norm}:\{(x,y,z)\in U'\setminus \{z=0\}\,:\,
x^n=y^mz^{m'}\}^{norm}\, \longrightarrow n_\cals^{-1}(\wsk\cap
U_p\setminus \{v=0\})
\end{equation*}
is an isomorphism. Using this isomorphism, the restriction of
$\psi_p^{norm} $ from \ref{eq:psi3},
\begin{equation*} %\label{eq:psi2}
\psi_p^{norm}:\{(\alpha,\beta,\gamma)\in U\setminus
\{\gamma=0\}\,:\, \alpha^n=\beta^m\gamma^{m'}\}^{norm}\,
\longrightarrow n_\cals^{-1}(\wsk\cap U_p\setminus \{v=0\})
\end{equation*}
can be understood explicitly. Indeed, the map
\begin{eqnarray} \label{eq:psi10}
\lefteqn{\varphi_{p,v}^{-1}\circ\psi_p:\{(\alpha,\beta,\gamma)\in
U\setminus \{\gamma=0\}\,:\, \alpha^n=\beta^m\gamma^{m'}\}\,
\longrightarrow} \hspace{2cm} \\
& & \hspace{1cm} \{(x,y,z)\in U'\setminus \{z=0\}\,:\,
x^n=y^mz^{m'}\}\nonumber \end{eqnarray} is given by
\begin{equation}\label{EQ:psi5a}
\left\{\begin{array}{lll}
x & = & \alpha \\
y & = & \beta |\gamma|^{\nu km'/mn} \\
z & = & \gamma |\gamma|^{-\nu k/n}.
\end{array}\right.
\end{equation}
We claim that this induces a diffeomorphism at the level of
normalization. In order to verify this, we make two reductions.
First, by a cyclic covering argument, we may assume that $m'=1$.
Second, we will also assume that gcd$(m,n)=1$ (otherwise the
normalization will have gcd$(m,n)$ components, and the
normalization maps below must be modified slightly; the details
are left to the reader). We fix two integers $a$ and $b$ such that
$an-bm=1$. Then the left hand side of (\ref{eq:psi10}) is
normalized  by $(t,\gamma)\in (\bfc^2,0)\setminus\{\gamma=0\}$,
$\alpha=t^m\gamma^a$, $\beta=t^n\gamma^b$, $\gamma=\gamma$; while
the right hand side is normalized by $(s,z)\in (\bfc^2,0)\setminus
\{z=0\}$, $x=s^mz^a$, $y=s^nz^b$ and $z=z$. Hence, at the
normalization level
\begin{equation*}
(\varphi_{p,v}^{-1}\circ\psi_p)^{norm}:(\bfc^2,0)\setminus \{\gamma=0\}\to
(\bfc^2,0)\setminus \{z=0\}
\end{equation*}
is given by the diffeomorphism
\begin{equation}\label{EQ:psi5}
\left\{\begin{array}{lll}
s & = & t |\gamma|^{a\nu k/mn} \\
z & = & \gamma |\gamma|^{-\nu k/n}.
\end{array}\right.
\end{equation}
\end{proof}

\begin{bekezdes}\labelpar{NORM3} Consider a singular point
in $\nsk$  above a double point of $\C$ of type 2, say $p$.
By the results of \ref{eq:psi3}, the type of this singularity can again
be identified with a (complex analytic) Hirzebruch--Jung
singularity, identified via the homeomorphism $\psi_p^{norm}$. In
particular, corresponding to that point, in the plumbing graph we
have to insert an appropriate  Hirzebruch--Jung string. In order to do this
we need to clarify orientation--compatibilities at the
intersection points of this string with the inverse image  of
$\C$. More precisely, we have to clarify the compatibility of the
chart $\psi_p$ with `nearby' charts of type \ref{lh:1}. \ix{graph!resolution!string}

Let $p$ be as in the previous paragraph, fix  one of its
neighbourhoods $\wsk\cap U_p$  as in \ref{eq:psi3}, and a generic
point  $q$ on the $v$--axis with small neighbourhood $\wsk\cap
U_q$ and chart $\psi_q:\{(x')^n=(y')^m\}\to
\{(u')^m(w')^n=|w'|^{\nu k}\}=\wsk\cap U_q$ given by
$u'=y'^{-1}|x'|^{\nu k/m}$, $w'=x'$ and $v'=z'$, cf. subsection \ref{lh:1}.
The inclusion $j:\wsk\cap U_q\longrightarrow\wsk\cap U_p$ is given
by the equations $u=u'(v')^{-m'/m}$, $v=v'$ and $w=w'$. Hence
$\varphi_{p,v}^{-1}\circ j\circ \psi_q$ is given by $x=x'$,
$y=y'(z')^{m'/m}$ and $z=(z')^{-1}$. This combined with
(\ref{EQ:psi5a}), the map $\psi_p^{-1}\circ j\circ
\psi_q:\{x'^n=y'^m,\, z'\not=0\}\to
\{\alpha^n=\beta^m\gamma^{m'},\, \gamma\not=0\}$ is given by (the inverse of)
\begin{equation}\label{EQ:psi10}
\left\{\begin{array}{lll}
x' & = & \alpha \\
y' & = & \beta \gamma^{m'/m} \\
z' & = & \gamma^{-1} |\gamma|^{\nu k/n}.
\end{array}\right.
\end{equation}
For simplicity assume again that gcd$(m,n)=1$ (the interested
reader can reproduce the general case). We take integers $a$ and
$b$ with $an-bm=m'$ as above. Then the free coordinates
$(t,\gamma)$ normalize $\{\alpha^n=\beta^m\gamma^{m'}\}$ by
$\alpha=t^m\gamma^a$, $\beta=t^n\gamma^b$ and $\gamma=\gamma$,
while the free coordinates $(s,z')$ normalize $\{x'^n=y'^m\}$ by
$x'=s^m$, $y'=s^n$ and $z'=z'$. In particular, the isomorphism
$(\varphi_{p,v}^{-1}\circ j\circ \psi_q)^{norm}$ is given by the
correspondence $(\bfc^2,0)\setminus \{\gamma=0\}\leftrightarrow
(\bfc^2,0)\setminus \{z'=0\}$:
\begin{equation}\label{EQ:psi11}
\left\{\begin{array}{lll}
s & = & t \gamma^{a/m} \\
z' & = & \gamma^{-1} |\gamma|^{\nu k/n}.
\end{array}\right.
\end{equation}
Notice that the inverse image  of $\C$ (in the two charts) is
given by $t=0$, respectively by $s=0$.
\end{bekezdes}
Now, consider the natural orientations of $n_\cals^{-1}(\wsk\cap
U_p)$ provided via $\psi_p$ by the complex structure of the source
of $\psi_p$, and also the orientation of  $n_\cals^{-1}(\C\cap
U_p)$ via the same procedure. In a similar way, consider the
orientations of  $n_\cals^{-1}(\wsk\cap U_q)$ and
$n_\cals^{-1}(\C\cap U_q)$ induced by $\psi_q$ and the complex
structure of its source. Then (\ref{EQ:psi11}) shows  the
following fact:
\begin{lemma}\labelpar{lemma:type2}
$(\varphi_{p,v}^{-1}\circ j\circ \psi_q)^{norm}$ is a
diffeomorphism which reverses the orientations. Moreover, its
restriction on the inverse images  of $\C$ reverses the
orientation of these Riemann surfaces as well. On the other hand,
the orientation of the transversal slices to the strict transforms
of $\C$ are preserved.
\end{lemma}

\section{\ The `resolution'\, $\overline{\cals}_k$ of $\wsk$}\labelpar{ss:skres}
\setcounter{equation}{0}

The singularities of $\nsk$ are situated  above the double points of
$\C$. Above a double point of type 1 they are isomorphic with
complex analytic Hirzebruch--Jung singularities; their
resolution follows the resolution procedure of these germs, see
\ref{NORM2}.

We did not determine here  the resolution and the real analytic type of the
singularity situating  above a double point of type 2.
Nevertheless, these singularities are also identified by
Lemma  \ref{lemma:type2}, up to an
orientation reversing homeomorphism, with complex analytic
Hirzebruch--Jung singularities. This
is enough to determine the topology of $\nsk$ and
to describe the plumbing representation of its boundary. This will
be done in the next section.

Although, for the purpose of the present work the above topological
representation is sufficient, if we would like to handle real analytic
invariants read from the structure sheaf of the resolution
(like, say, the geometric genus is read from the resolution of a
normal surface singularity), then an explicit description of this
variety would be more than necessary. This type of analytic
questions are beyond the aims of the present work, however we formulate
this problem  as an important goal for further research.

\begin{bekezdes}\labelpar{OP1} {\bf Problem.} Find an explicit
description of the real analytic/algebraic resolution of the singularity
$$\{(u,v,w)\in (\bfc^3,0)\, :\, u^mv^{m'}w^n=|w|^{k}\}^+,$$
where $m,\, m'>0$ and $k$ is a sufficiently large (even) integer.
%(which might satisfy some other divisibility assumptions as well,
%if it is necessary).
\end{bekezdes}

\section{\ The plumbing graph. The end of the proof}\labelpar{ss:PLUM}
\setcounter{equation}{0}

Once the geometry of the tubular neighbourhood of the divisor $\C$
is clarified, it is standard to describe the plumbing
representation of the boundary of this neighbourhood. We  follow
the strategy of \cite{eredeti}, with a modification above the
double points of type 2.

\begin{bekezdes}\labelpar{bek:pl1} Consider a component $C$ of
$\C$ with decoration $(m;n,\nu)$. By \ref{lh:1}, the local
equation of $\wsk$ in a neighbourhood of a generic point of $C$ is
$x^m=y^n$, hence $n_\cals^{-1}(C)\to C$ is a regular  covering of
degree gcd$(m,n)$ over the regular part of $C$. Let $C^{norm}$ be
the normalization of $C$, i.e. the curve obtained by separating
the self--intersection points of $C$ (which are codified by loops
of $\gc$ attached to the vertex $v_C$ which corresponds to $C$).
Then $q:n_\cals^{-1}(C)\to C^{norm}$ is a cyclic branched covering
whose branch points $B$ are situated  above the double points of $\C$. They
correspond bijectively to the legs of the  star of $v_C$, cf.
\ref{s1}. Notice that if $v_C\in\calv^1(\gc)$ (see
\ref{stricttra} for notation) then $m=1$, hence the covering is
trivial. Otherwise $C^{norm}$ is rational by Proposition \ref{gkl}.

Fix a branch point $b\in B$ whose neighbourhood has a local equation
of type $z^c=x^ay^b$. Then $q^{-1}(b)$ has exactly gcd$(a,b,c)$
points; this number automatically divides gcd$(m,n)$ for any choice of $b$. The
number $\n_{v_C}$ of connected components of $n_\cals^{-1}C$ is
the order of coker$(\pi_1(C^{norm}\setminus B)\to \Z_{gcd(m,n)})$,
where a small loop around $b$ is sent to the class of
gcd$(a,b,c)$. Hence the formula (\ref{NW}) for $\n_{v_C}$ follows,
and (\ref{NG}) follows too by an Euler--characteristic argument.

In the local charts of \ref{lh:1},  $g=v^\nu=y^\nu$. Since, by
normalization, $y=t^{m/gcd(m,n)}$, and $t=0$ is the local equation
of the strict transform of  $C$, the vanishing order of $y^\nu$
along the strict transform of $C$ (i.e. the multiplicity of the
open book decomposition of $\arg(g)$)  is $m\nu/\mbox{gcd}(m,n)$,
proving (\ref{NM}). This ends the proof of Step 1  of the Main
Algorithm and of Theorem \ref{MTh.1}.
\end{bekezdes}
\begin{bekezdes}\labelpar{bek:pl2}
Next, one has to insert the  Hirzebruch--Jung strings \ix{graph!resolution!string}
corresponding to the singularities of $\nsk$. Type 1 singular
points behave similarly as those appearing in the case of cyclic
coverings \cite{cyclic}, or in the case of  Iomdin series
\cite{eredeti} (or anywhere in complex analytic geometry). In
particular, the orientation compatibilities of
\ref{NORM2} imply  that these strings  should be glued in with all edges  decorated
$+$, as usual for dual graphs of complex analytic curve
configurations. \ix{curve configuration $\C$}
One the other hand, the way how the Hirzebruch--Jung strings of
type 2 should be  inserted is dictated by Lemma \ref{lemma:type2}.
Assume that the corresponding singularity is above the intersection point $C_1\cap
C_2$ of two components of $\C$. Then, in the plumbing
representation  we have to connect their strict transforms
(denoted by the same symbols) by a string $E_1,
\ldots,E_s$. By \ref{lemma:type2}, when $C_1$ is glued to the
string, its orientation is reversed. In order to keep the
ambient orientation, we have to change the orientation of its
transversal slice too. But this is identified with the first curve
$E_1$ of the string. If the orientation of $E_1$ is changed, then,
similarly as above, we have to change the orientation of its
transversal slice, which is identified with $E_2$. By  iteration, we see, that
all  decorations of all   edges of the string, inserted by Step 2, Case 2 in \ref{s2},
should be $\circleddash$.

The multiplicity decorations are given by the vanishing orders of
$g$, and are computed by the usual procedures, see \ref{1.6}.
This proves Step 2 of the Main Algorithm.
Finally, Step 3 does not require any further explanation, cf.
(\ref{eq:2.2.1}).  This ends the proof of Theorem \ref{MTh.1}.\\

Theorems \ref{g1} and \ref{g2} are particular cases, which are
obtained by forgetting some information from the graph of
$\partial F$.\ix{Milnor!fiber!boundary}
\end{bekezdes}

\section{\ The `extended' monodromy action}\label{ss:EXT}
\setcounter{equation}{0}

Usually, when one has a plumbing graph $G$, besides the 3--manifold constructed by
gluing $S^1$--bundles, one can consider the plumbed 4--manifold constructed by gluing
disc--bundles too. This is the case here as well; in fact, as it is clear from the
constructions of this section,  the plumbed 4--manifold associated with $G$ is exactly
the manifold $\overline{\cals}_k$.
The point we wish to stress in this section is that there is a
natural monodromy action on the pair
$(\overline{\cals}_k,\partial \overline{\cals}_k)$ such that the induced action on
$\partial \overline{\cals}_k$ coincides with the Milnor monodromy action  of $\partial F$.\ix{Milnor!fiber!boundary}

Indeed, instead of only defining the space $Z_k=\{c=|d|^k\}$, as in \ref{ss:geocon}, one can take the
family of spaces  $Z_k(t):=\{c=|d|^ke^{it}\}$  for all values  $t\in[0,2\pi]$, and repeat
the constructions of the present section. In particular, one can define in a natural way
$\cals_k(t)$, $\widetilde{\cals}_k(t)$ and $\overline{\cals}_k(t)$ for all $t$.
This is a locally trivial bundle over the parameter $t$, hence moving $t$ from
0 to $2\pi$, we get the wished action on the pair $(\overline{\cals}_k,\partial \overline{\cals}_k)$.

The monodromy action on the cohomology long exact sequence of
$(\overline{\cals}_k,\partial \overline{\cals}_k)$ will have important  consequences,
see for example  the proof of Corollary \ref{cor:EIG}.

\begin{proposition}\label{prop:extmon} Consider the above monodromy action on the pair
$(\overline{\cals}_k,\partial \overline{\cals}_k)$. Then the following facts hold:\vspace{2mm}

(a) \ The action on $\partial \overline{\cals}_k$ coincides with
 the Milnor monodromy action on $\partial F$.\vspace{2mm}

 (b) \ At homological level, the generalized 1--eigenspace $H_1(\overline{\cals}_k,\bfc)_1$ equals
 $H_1(\widetilde{\cals}_k,\bfc)$.  In particular, \ix{generalized eigenspace}
\begin{equation}\label{eq:EXT} \begin{array}{l}
{\rm rank}\,\, H_1(\overline{\cals}_k)=2g(G)+c(G),\\
{\rm rank}\,\, H_1(\overline{\cals}_k)_1={\rm rank}\,\, H_1(\widetilde{\cals}_k)=2g(\gc)+c(\gc).
\end{array}
\end{equation}\end{proposition}

\begin{proof}
(a) Similarly as in the proof of Proposition  \ref{lem:sksmooth}, when we compared the spaces
$\Phi^{-1}(Z_k)\cap S_\epsilon$ and $\Phi^{-1}(D_{c_0})\cap S_\epsilon$,
one can identify  the spaces
$\Phi^{-1}(Z_k(t))\cap S_\epsilon $ and $\Phi^{-1}(D_{c_0e^{it}})\cap S_\epsilon$
uniformly for any $t$.
The second family enters as a  building block in $\partial F$ via the decomposition
 from  \ref{lem:sksmooth}, and the above action is exactly the Milnor monodromy action.

(b) Let $\G$ be a plumbing graph and $P(\G)$ the associated plumbed 4--manifold, cf.
\ref{bek:MULT}. Then one has the homotopy equivalences of  $P(\G)$ with the core curve configuration of
the plumbing. On the other hand, the first homology of this curve configuration is  $2g(\G)+c(\G)$.
Hence,  ${\rm rank}\, H_1(P(\G))=2g(\G)+c(\G)$. Therefore, in the present situation,
${\rm rank}\,\, H_1(\overline{\cals}_k)=2g(G)+c(G)$. \ix{curve configuration $\C$}

A similar argument shows that
${\rm rank}\,\, H_1(\widetilde{\cals}_k)=2g(\gc)+c(\gc)$.

Clearly, $H_1(\overline{\cals}_k)=
H_1(\cals_k^{norm})$ too.

Next, we wish to  understand the  effect of the monodromy
on $\widetilde{\cals}_k$ and $\cals_k^{norm}$ induced by the $t$--parameter family.

We will consider a generic point of the exceptional curve of $\widetilde{\cals}_k$.
Note that  via the local equations
$c=f=u^mv^n$ and $d=g=v^\nu$ from \ref{summary},
the parameterized equation $c=|d|^ke^{it}$ transforms into
$u^mv^n=|v|^{\nu k}e^{it}$. This equation, via the isomorphism \ref{EQ:psi1},
transforms into $y^n=x^me^{it}$.
 In other words,  $\widetilde{\cals}_k(t)$ locally
is the tubular neighbourhood of the $z$--axis in the variety given by  local equation
$\{(x,y,z)\,:\,  y^n=x^me^{it}\}$. Homotopically,
this set is equivalent with the $z$--axis, the core curve, and the
monodromy action induced on it
is trivial. Analyzing all the other points too, we get that the monodromy action on
 $\widetilde{\cals}_k(0)$  is homotopically trivial.

Let us analyze now the graph covering $\cals^{norm}_k(t)\to \widetilde{\cals}_k(t)$. \ix{graph!covering}

Again, let us take the same local situation as in the above discussion.
For any fixed $t$, the normalization of the variety $y^n=x^me^{it}$ has ${\rm gcd}\,(m,n)$ disjoint
 components; therefore, the $z$--axis in the normalization is covered by   ${\rm gcd}\,(m,n)$
 local discs. These components are cyclically permuted by the monodromy.

Again, analyzing all the points, we get a finite cyclic branched covering of the core curve configuration of
$\widetilde{\cals}_k(t)$ by the core curve configuration of $\cals^{norm}_k(t)$,
and the monodromy action corresponds to the cyclic action of the covering. Therefore,
 $H_1(\cals_k^{norm},\bfc)_1=H_1(\cals_k^{norm}/\mbox{action},\bfc)=H_1(\widetilde{\cals}_k,\bfc)$.
\end{proof}

\chapter{The Collapsing Main Algorithm}\labelpar{ss:ELI}

\section{\ Elimination of Assumption B}\labelpar{ss:elim}\setcounter{equation}{0}

\bekezdes {\bf Preliminary remarks.} \
In the formulation and the proof of the Main Algorithm  \ref{algo}
 the absence  of `vanishing 2--edges' in $\gc$ is  essential.
 If a certain  $\gc$ has such an edge, it can be modified by
a blow up, which replaces the unwanted edge by three
`acceptable' edges, see \ref{re:w3}.
Therefore, in any situation, it is easy to assure the condition
of Assumption B, and the Main Algorithm serves as a complete algorithm
for $\partial F$, $\partial F_1$, $\partial _2F$ and for the different
multiplicity systems.
\ix{Assumption B}\ix{vanishing 2--edge}\ix{Main Algorithm!Collapsing}\ix{Milnor!fiber!boundary}

Nevertheless, if the graph $\gc$ is constructed by a canonical
geometric procedure, and it has vanishing 2--cycles, the
above procedure  of the Main Algorithm,  which starts with  blowing up these edges,
has some inconveniences.\ix{Main Algorithm}

First of all, in the new graph $\gc$ we create several new vertices and edges;
on the other hand, it turns out
that  in the output final graph $G$ all these extra
vertices/edges can be eliminated, collapsed, see \ref{bad:edge}.
This indicates that blowing up $\gc$ might be unnecessary,
and there should be a better procedure
to eliminate the vanishing 2--edges in such a way that the new graph is not
`increasing'.

But, in fact, the main reason to search for another
approach/solution is dictated by a more serious reason:
we will see that `unicolored' graphs/subgraphs (that is, graphs
with uniform edge decorations) have big advantages
in the determination of the geometrical properties (like the structure of  Jordan
blocks, monodromy operators). On the other hand, by the
blow up  step \ref{re:w3}  we might destroy such a
property. Take for example the case of cylinders.  As
constructed in \ref{cyl}, and  before applying the blowing up
procedure,  the graph $\gc$ is unicolored: all the edge--decorations are 2. This
property  is not preserved after the extra blow ups of \ref{re:w3}.
\ix{graph!unicolored}

Moreover, we will see in \ref{lem:twist},  that the `twist' (local
variation map) associated with a {\em vanishing}  2--edge is
{\em vanishing}. Hence, the separating annulus (in the page of
the open book) codified by such an edge is `rigid', and thus it
should be glued rigidly with its neighbourhood. Therefore, in the
language of the graph, such an edge should be rather collapsed
than blown up!

This last statement can also be reinterpreted in the following way. We will prove
in \ref{lem:twist} that the twist of a
1--edge is negative, of a non--vanishing 2--edge is positive, while, as we already
said, of a vanishing 2--edge, is zero. Since by blowing up a vanishing 2--edge
we create two new 1--edges and one new 2--edge, we replace the zero--twist--contribution
by two contributions of different signs. Nevertheless, when handling operators
(see a concrete situation in subsections \ref{diagram}--\ref{lem:unitw}),
sometimes it is more convenient to have a semi--definite matrix rather than a
non--degenerate one, which is not definite.

\bekezdes {\bf The goal of the chapter.} \
In this section we  present an alternative way to modify
$\gc$ and the steps of the Main Algorithm in the presence of vanishing
2--edges. This second method is also  based on the
algorithm just proved: we blow up such a vanishing 2--edge, we run the Main
Algorithm \ref{algo}, then we apply the reduced plumbing calculus for
that part of the graph whose ancestor is that vanishing 2--edge and its adjacent vertices,
and we show that this part collapses into a single orbit of vertices. Moreover, any connected
subgraph whose edges are vanishing 2--edges, by this procedure collapses into a single orbit of vertices.

We keep the output of all these steps as a shortcut, which will be built
in the new  version of the Main Algorithm, called `{\it Collapsing Main Algorithm}'.
\ix{Main Algorithm!Collapsing}

Obviously, if the original graph has no vanishing 2--edges, then
the two algorithms are the same.

\begin{bekezdes}\labelpar{bad:edge} {\bf Discussion.} \
 Consider a vanishing  2--edge $e$ as in \ref{re:w3}:

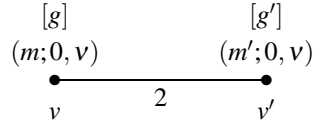
\vspace{3mm}

\begin{picture}(140,50)(-110,10)
\put(20,30){\circle*{4}} \put(100,30){\circle*{4}}
\put(20,30){\line(1,0){80}}
\put(20,35){\makebox(0,0)[b]{$(m;0,\nu)$}}
%\put(60,40){\makebox(0,0)[b]{$2$}}
\put(100,35){\makebox(0,0)[b]{$(m';0,\nu)$}}
%\put(20,15){\makebox(0,0)[b]{$[g]$}}
%\put(100,15){\makebox(0,0)[b]{$[g']$}}
\put(20,15){\makebox(0,0)[b]{$v$}}
\put(100,15){\makebox(0,0)[b]{$v'$}}
\put(20,49){\makebox(0,0)[b]{$[g]$}}
\put(100,49){\makebox(0,0)[b]{$[g']$}}
\put(60,20){\makebox(0,0)[b]{$2$}}
\end{picture}

Assume that  $v,v'\in\calw$ and $v\not=v'$. Assume
that the 1--legs (cf. \ref{ls}) of $v$ have weights
$\{(m;n_i,\nu_i)\}_{i=1}^s$, and the 2--legs of $v$, other than
$e$, are decorated by $\{(m_j;0,\nu)\}_{j=1}^t$. Set
$N:=\mbox{gcd}(m,n_1,\ldots,n_s,m_1,\ldots,m_t)$. We will have
similar notations $s',\ t',\ n_i',$ $ \nu'_i,\ m'_j,\ N'$ for $v'$
too.

In \ref{re:w3} we have replaced $e$ by the string $Str(e)$:

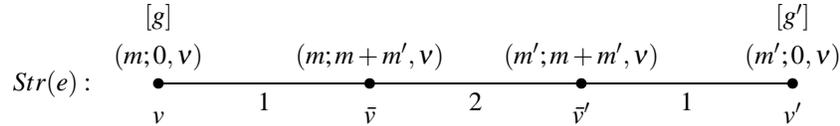
\begin{picture}(140,60)(-110,10)
\put(20,30){\circle*{4}}
\put(100,30){\circle*{4}}\put(-60,30){\circle*{4}}
\put(180,30){\circle*{4}} \put(-60,30){\line(1,0){240}}
\put(-60,35){\makebox(0,0)[b]{$(m;0,\nu)$}}
%\put(60,40){\makebox(0,0)[b]{$2$}}
\put(180,35){\makebox(0,0)[b]{$(m';0,\nu)$}}
\put(-60,50){\makebox(0,0)[b]{$[g]$}}
\put(180,50){\makebox(0,0)[b]{$[g']$}}
\put(-60,15){\makebox(0,0)[b]{$v$}}
\put(180,15){\makebox(0,0)[b]{$v'$}}
\put(20,15){\makebox(0,0)[b]{$\bar{v}$}}
\put(100,15){\makebox(0,0)[b]{$\bar{v}'$}}
\put(60,20){\makebox(0,0)[b]{$2$}}
\put(-20,20){\makebox(0,0)[b]{$1$}}
\put(140,20){\makebox(0,0)[b]{$1$}}

\put(-115,30){\makebox(0,0)[l]{$Str(e):$}}

\put(20,35){\makebox(0,0)[b]{$(m;m+m',\nu)$}}
\put(100,35){\makebox(0,0)[b]{$(m';m+m',\nu)$}}
\end{picture}

Then, let us run the Main Algorithm for this part of the graph. In
the  covering graph $G$ the number of vertices over $v$ is \ix{graph!covering}
$\n_{v}:=\mbox{gcd}(N,m')$,  over $v'$ is
$\n_{v'}:=\mbox{gcd}(N',m)$, and over the new vertices $\bar{v}$
and $\bar{v}'$ the number of vertices is  $\n_e:=\mbox{gcd}(m,m')$.
Moreover, over all the edges we have to put $\n_e$ edges, they
form $\n_e$ strings, containing the $\n_e$ vertices sitting above
the new vertices, and the ends of these strings are cyclically identified with
the vertices sitting over $v$ and $v'$  respectively. The vertices
over $v$ have multiplicity decoration $(\nu)$ and genus decoration
$\tilde{g}$ determined by (\ref{NG}), namely
\begin{align}\label{NG2}
\n_v\cdot (2-2\tilde{g}) &
=(2-2g - s-t-1) \cdot m\\
   & + \sum\limits_{i=1}^s \gcd (m,  n_i) +
    \sum\limits_{j=1}^t \gcd (m, m_j)+\n_e\nonumber.
\end{align}
There is a similar statement for vertices over $v'$ too. The
vertices of $G$ over the new vertices $v$ and $v'$  have zero genera
and multiplicity decorations $(m\nu/\n_e)$ and $(m'\nu/\n_e)$
respectively. Furthermore, if we apply Step 2 of the algorithm from
\ref{s2}, then we realize that above the 1--edges the
corresponding strings are degenerate (hence we insert +--edges
only), while above the 2--edge of $Str(e)$ any inserted string
$Str^\circleddash$ has only one vertex with multiplicity $(\nu)$
and Euler decoration $(m+m')/\n_e$. In particular,  the $\n_e$
strings above $Str(e)$ have the form

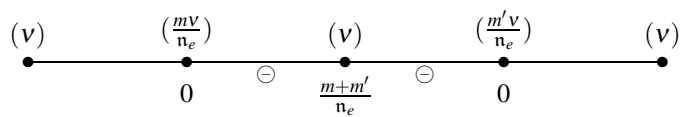
\begin{picture}(140,50)(-100,10)
\put(-60,30){\circle*{4}} \put(0,30){\circle*{4}}
\put(60,30){\circle*{4}} \put(120,30){\circle*{4}}
\put(180,30){\circle*{4}}
\put(-60,30){\line(1,0){240}}
\put(-60,35){\makebox(0,0)[b]{$(\nu)$}}
\put(0,35){\makebox(0,0)[b]{$(\frac{m\nu}{\n_e})$}}
\put(60,35){\makebox(0,0)[b]{$(\nu)$}}
\put(120,35){\makebox(0,0)[b]{$(\frac{m'\nu}{\n_e})$}}
\put(180,35){\makebox(0,0)[b]{$(\nu)$}}
\put(0,15){\makebox(0,0)[b]{$0$}}
\put(120,15){\makebox(0,0)[b]{$0$}}
\put(60,10){\makebox(0,0)[b]{$\frac{m+m'}{\n_e}$}}
\put(30,22){\makebox(0,0)[b]{$\circleddash$}}
\put(90,22){\makebox(0,0)[b]{$\circleddash$}}
%\put(-20,20){\makebox(0,0)[b]{$1$}}
%\put(140,20){\makebox(0,0)[b]{$1$}}
\end{picture}

\vspace{4mm}

The configuration of all the vertices and edges situating  above $e$
and its end--vertices form $\n:=\mbox{gcd}(\n_v,\n_{v'})$
connected components.

 Now, we run the plumbing calculus of oriented plumbed 3--manifolds, cf.
\ref{ss:2.2}. Notice that by two 0--chain absorptions the above
string can be collapsed.
%In particular, by {\em 0--chain
%absorptions} and {\em oriented handle absorptions} all these
%strings can be collapsed identifying their end--vertices as well.
Hence, after {\it 0--chain} and {\it oriented handle absorptions} each
connected component   collapses into a single vertex. Their number
will be  $\n$ and  all of them will carry multiplicity $(\nu)$.
The genus decoration $g_e$ of such a vertex can be computed as
follows. First we have a contribution from 0--chain absorptions,
namely the sum of all the genera of the vertices in the corresponding connected
component, namely $(\n_v\tilde{g}+\n_{v'}\tilde{g}')/\n $. Then,
corresponding to oriented handle absorptions,
 we have to add one for each
1--cycle of that component. This, by an Euler--characteristic
argument is $(\n_e-\n_v-\n_{v'}+\n)/\n$. Hence
$$\n g_e=\n_v\tilde{g}+\n_{v'}\tilde{g}'+\n_e-\n_v-\n_{v'}+\n,$$
that is
$$\n(1-g_e)=\n_v(1-\tilde{g})+n_{v'}(1-\tilde{g}')-\n_e.$$ This
combined with (\ref{NG2}), gives
\begin{align*}\label{NG3}
\n\cdot (2-2g_e) &
= (2-2g - s-t-1) \cdot m
   + \sum\limits_{i} \gcd (m,  n_i) +
    \sum\limits_{j} \gcd (m, m_j)\\
  & +  (2-2g' - s'-t'-1) \cdot m'
    + \sum\limits_{i} \gcd (m',  n'_i) +
    \sum\limits_{j} \gcd (m', m'_j).\nonumber
\end{align*}
\end{bekezdes}

\begin{bekezdes}\labelpar{bad:graph} Clearly,
 a certain  vertex of $\gc$ may be  the
end--vertex of more than one vanishing  2--edge. Hence, if we run
the above procedure \ref{bad:edge} for all the vanishing
2--edges simultaneously, a bigger part of the graph will be
collapsed. We make this fact more precise in the next paragraphs.

\begin{lemma}\label{lem:vancut} If $V_g$ has at most an isolated singularity then
there is no cutting edge of $\gc$ which is simultaneously a vanishing 2--edge  and both its end--vertices are
non--arrowheads. \ix{cutting edge}
\end{lemma}
\begin{proof}
Assume that we have such an edge; consider the corresponding intersection
point of $\C$ and the local equation around it
as in \ref{summary}: $f\circ r|_{U_p}=uv^{m'}$ and
$g\circ r|_{U_p}=w^\nu$, with $m'>1$. Then the local component $u=0$
is  in the strict transform of $V_f$.  Since along  the local component
 $w=0$ only $g$ is vanishing, and  $V_g$ has an isolated singularity, we get that
 $w=0$ is situating in the strict transform of $V_g$ (otherwise $w=0$ would be contained  in an
 exceptional divisor which is  above the origin, but along such a divisor
  $f\circ r$ is  also vanishing). But then we have a compact curve in the intersection
  of the strict transforms of $V_f$ and $V_g$, which is not possible.
\end{proof}

Consider again the graph $\gc$, as it is given by a certain resolution,  and {\em unmodified}
by the blowing up  procedure \ref{re:w3}. Thus it may not even satisfy  Assumption B.  Nevertheless,
for simplicity of the discussion, \ix{Assumption B}
we assume that it has no cutting edge which is simultaneously a vanishing 2--edge and both its end vertices are non--arrowheads. (We believe that this condition is always automatically satisfied.
If $V_g$ has at most an isolated singularity this is guaranteed by    \ref{lem:vancut}).
\ix{cutting edge}
%\marginpar{EZ NEM IGAZ?????}

Consider a maximal connected subgraph $\gbk$ of $\gc$ with only non-arrowhead vertices and
such that all its edges are {\em vanishing 2--edges} connecting these vertices.
In particular, $\gbk$ has no edges supporting arrowheads.
Such a subgraph is supported either by $\gce$ or by $\gck$. If it is a subgraph of $\gce$ then it
has only one vertex. In the other case it might have several vertices. Since  $\Gamma_{\C,j}$  is a tree
by \ref{tree},
\begin{equation}\label{eq:gbk_tree}
\mbox{$\gbk$ is always a tree.}
\end{equation}
We wish to define for each component $\gbk$ the numbers
$\n_{\gbk}$, $m_{\gbk}$ and $g_{\gbk}$.

Assume first that $\gbk$ contains exactly one vertex, say $w$.
Let the decoration of $w$ be
$(m;n,\nu)$. Then using the star of $w$ we define $\n_{\gbk}=\n_w$ as in
(\ref{NW}), $m_{\gbk}=m\nu/\mbox{gcd}(m,n)$ as in (\ref{NM}), and
$g_{\gbk}=\tilde{g}_w $ as in  (\ref{NG}), exactly as in the Main Algorithm.

Next, assume that $\gbk$ contains several vertices. For any
vertex $w$ of $\gbk$,  consider its star in $\gc$:

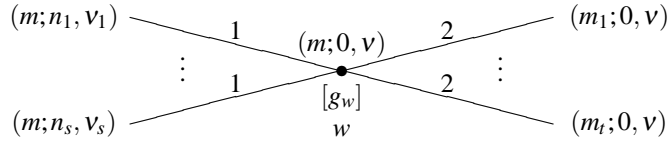
\vspace{6mm}

\begin{picture}(140,60)(-150,0)
\put(20,30){\circle*{4}}
\put(20,30){\line(4,1){80}}
\put(20,30){\line(4,-1){80}}
\put(20,30){\line(-4,1){80}}
\put(20,30){\line(-4,-1){80}}
\put(20,35){\makebox(0,0)[b]{$(m;0,\nu)$}}

\put(125,45){\makebox(0,0)[b]{$(m_1;0,\nu)$}}
\put(125,5){\makebox(0,0)[b]{$(m_t;0,\nu)$}}
\put(-85,45){\makebox(0,0)[b]{$(m;n_1,\nu_1)$}}
\put(-85,5){\makebox(0,0)[b]{$(m;n_s,\nu_s)$}}

\put(20,15){\makebox(0,0)[b]{$[g_w]$}}
\put(80,26){\makebox(0,0)[b]{$\vdots$}}
\put(-40,26){\makebox(0,0)[b]{$\vdots$}}
\put(20,5){\makebox(0,0)[b]{$w$}}
%\put(100,15){\makebox(0,0)[b]{$v_2$}}
\put(60,22){\makebox(0,0)[b]{$2$}}
\put(60,42){\makebox(0,0)[b]{$2$}}
\put(-20,22){\makebox(0,0)[b]{$1$}}
\put(-20,42){\makebox(0,0)[b]{$1$}}
\end{picture}

\vspace{6mm}

Notice that the 2--legs of this star come from two sources: either they are associated with the
edges of $\gbk$, or they are 2--edges supporting arrowheads. Let $\hat{t}_w$ be the number of
this second group. The end decorations of these legs are $(1;0,1)$, hence if
there is a vertex $w$ of $\gbk$ with $\hat{t}_w>1$, then $\nu=1$ automatically.

Associated with the star of $w$ we consider the integers $\n_w$,
%$\tilde{m}_w$,
$\tilde{g}_w$ and $\hat{g}_w$ as follows:

\vspace{2mm}

$\bullet$ \ \  $\n_w:=\mbox{gcd}(m,n_1,\ldots, n_s,m_1,\ldots , m_t)$, as in  (\ref{NW});

\vspace{2mm}

%$\bullet$ \ \ $\tilde{m}_w:=\nu$, cf. (\ref{NM});

$\bullet$ \ \ $\n_w (2-2\tilde{g}_w)  =(2-2g_w - s-t)  m
    + \sum\limits_{i=1}^s \gcd (m,  n_i) +
    \sum\limits_{j=1}^t \gcd (m, m_j)$, as in  (\ref{NG});

 $\bullet$ \ \ $\n_w (2-2\hat{g}_w)  =(2-2g_w - s-t)  m
    + \sum\limits_{i=1}^s \gcd (m,  n_i) +
    \hat{t}_w$.

\vspace{2mm}

Furthermore, for any edge $e\in\cale(\gbk)$, with decorations as in \ref{bad:edge}, define

\vspace{2mm}

$\bullet$ \ \ $\n_e:=\mbox{gcd}(m,m')$,  as in \ref{bad:edge}.

\vspace{2mm}

Then, similarly as in  \ref{bad:edge}, if we eliminate  the
vanishing  2--edges of $\gbk$ by the blow up procedure of \ref{re:w3}, and
run the Main Algorithm \ref{algo}, then above $\gbk$ the graph will have
\begin{equation}\label{NGamma}
\n_{\gbk}:=\mbox{gcd} \{\, \n_w\, : \, w\in \calv(\gbk)\}
\end{equation}
connected components. Indeed, this follows from (\ref{eq:gbk_tree}) and \ref{ex:COVERING}(1).
Then,  after 0--chain and oriented
handle absorptions, the whole subgraph above $\gbk$ will collapse  into
$\n_{\gbk}$ vertices, all with multiplicity
\begin{equation}\label{MGamma}
m_{\gbk}=\nu,
\end{equation} and genus
decoration $g_{\gbk}$, which is determined similarly as in
\ref{bad:edge}. More precisely,  the contribution from the genera of the vertices
is $\sum_w\n_w \tilde{g}_w/\n_{\gbk}$, while the contribution from
the cycles is
\begin{equation}\label{eq:CYCLE}
\big(\, \sum_{e\in \cale(\gbk)}\n_e- \sum _{w\in\calv(\gbk)} \n_w+\n_{\gbk}\,\big) /\n_{\gbk}.
\end{equation}
In particular,  for $g_{\gbk}$ we get:
\begin{equation}\label{eq:GGamma}
\n_{\gbk} (2-2g_{\gbk}) = \sum _{w\in \calv(\gbk)}\, \n_w (2-2\hat{g}_w).
\end{equation}
\end{bekezdes}

\begin{bekezdes}\labelpar{bad:arrow}
If a 2--edge supports an arrowhead then it is automatically a vanishing 2--edge.
Consider such an edge $e$ of $\gc$, whose non--arrowhead vertex $w$ has weight $(m;0,1)$.
Let $\gbk$ be the subgraph as in \ref{bad:graph} which
contains $w$. Since $\n_w=1$,  one obtains  that  $\n_{\gbk}=1$ too, hence $\gbk$ can be
collapsed by the procedure described in  \ref{bad:graph} to a unique vertex.

In $G$, above $e$,  similarly as above we get  exactly one
string of the form

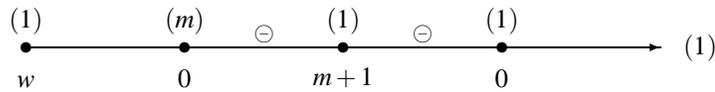
\begin{picture}(140,50)(-100,10)
\put(-60,30){\circle*{4}} \put(0,30){\circle*{4}}
\put(60,30){\circle*{4}} \put(120,30){\circle*{4}}

\put(-60,30){\vector(1,0){240}}
\put(-60,35){\makebox(0,0)[b]{$(1)$}}
\put(0,35){\makebox(0,0)[b]{$(m)$}}
\put(60,35){\makebox(0,0)[b]{$(1)$}}
\put(120,35){\makebox(0,0)[b]{$(1)$}}
\put(195,30){\makebox(0,0){$(1)$}}
\put(0,15){\makebox(0,0)[b]{$0$}}
\put(120,15){\makebox(0,0)[b]{$0$}}
\put(60,15){\makebox(0,0)[b]{$m+1$}}
\put(30,32){\makebox(0,0)[b]{$\circleddash$}}
\put(90,32){\makebox(0,0)[b]{$\circleddash$}}
\put(-60,15){\makebox(0,0)[b]{$w$}}
%\put(140,20){\makebox(0,0)[b]{$1$}}
\end{picture}

\vspace{2mm}

The first 0--vertex can be eliminated by 0--chain
absorption. The obtained  shorter string is glued to the unique
vertex constructed in \ref{bad:graph} corresponding to $\gbk$.
\end{bekezdes}

Now we are able to formulate the new version of the Main Algorithm.\ix{Main Algorithm}

\section{\ The Collapsing  Main Algorithm}\label{algoim}\setcounter{equation}{0}

\begin{bekezdes} Start again with a  graph $\gc$, as it is given by a certain resolution,  and {\em unmodified}
by the blowing up  procedure \ref{re:w3}.  Assume that it has no cutting edge which
is simultaneously a vanishing 2--edge and both its end vertices are non--arrowheads.
 If it has some vanishing
2--edges, we will not blow them up, as in
\ref{re:w3};   instead, we will `collapse' them by the procedure
described in the previous section \ref{ss:elim}.
\ix{Main Algorithm!Collapsing|textbf}\ix{cutting edge}

Denote by $\widehat{\gc}$ the undecorated  graph obtained from
$\gc$ by contracting (independently) each  subgraphs of type
$\gbk$ into a unique vertex.  All 1--edges, non--vanishing
2--edges,  arrowhead vertices and vanishing 2--edges supporting
arrowhead vertices survive inheriting the natural
adjacency  relations.
% One might decorate the vertices by
%$m_{\gbk}$ and $[g_{\gbk}]$, while the edges keep their
%decorations from $\gc$.

Then, we construct a plumbing graph $\widehat{G}$ of the open book
of $\partial F$ with binding $V_g$ and fibration $\arg(g):\partial
F\setminus V_g\to S^1$ as follows. It  will be determined as a
covering graph of $\widehat{\gc}$, modified with strings as in
\ref{re:2.3.1}. In order to identify it, we
have to provide the  covering data of the covering
$\widehat{G}\to \widehat{\gc}$, cf.
\ref{def:2.3.2}.
\end{bekezdes}\ix{graph!covering}\ix{graph!covering!data}\ix{Milnor!fiber!boundary}

\begin{bekezdes}{\bf Step 1. --- The covering data of the vertices
of $\widehat{\gc}$.}\labelpar{s1m}

Over a  vertex of $\widehat{\gc}$, obtained by the  contraction of
the subgraph $\gbk$ of $\gc$, we insert $\n_{\gbk}$ vertices in
$\widehat{G}$, all of them with genus decoration $g_{\gbk}$  and multiplicity $m_{\gbk}$.
These numbers are defined in (\ref{NGamma}),  (\ref{eq:GGamma}) and (\ref{MGamma}) respectively.

Any arrowhead vertex of $\widehat{\gc}$ is covered by  one
arrowhead vertex  of $\widehat{G}$, decorated by multiplicity
decoration  $(1)$, similarly as in the original version
\ref{s1}.
\end{bekezdes}

\begin{bekezdes}{\bf Step 2. --- The covering data of edges
and the types of the inserted strings.}\labelpar{s2m}

\vspace{2mm}

\noi {\bf Case 1.}\ The case of 1--edges is the same as in the
original version \ref{s2}.  Over such an edge $e$,  which in $\gc$
has the form

\vspace{1mm}

\begin{picture}(140,50)(-100,0)
\put(20,30){\circle*{4}} \put(100,30){\circle*{4}}
\put(20,30){\line(1,0){80}}
\put(20,35){\makebox(0,0)[b]{$(m;n,\nu)$}}
\put(100,35){\makebox(0,0)[b]{$(m;l,\lambda)$}}
\put(20,15){\makebox(0,0)[b]{$[g]$}}
\put(100,15){\makebox(0,0)[b]{$[g']$}}
\put(20,0){\makebox(0,0)[b]{$v_1$}}
\put(100,0){\makebox(0,0)[b]{$v_2$}}
\put(60,35){\makebox(0,0)[b]{$1$}}
\end{picture}

\noi insert (cyclically) in $\widehat{G}$ exactly $\n_e=\gcd(m, n,
l)$ strings of type
$$Str\left(\, {n\over \n_e}, {l\over \n_e};{m\over \n_e}\
\Big|\ \nu, \lambda;0 \, \right).$$ If the edge $e$ is a loop,
then the procedure is the same with the only modification that the
end--vertices of the  strings are identified. If the right vertex
$v_2$ is an arrowhead, then complete again the same procedure with
$m=1$ and $\n_e=1$, namely:  above such an edge $e$ put a single
edge decorated by $+$.  This edge  supports that arrowhead of
$\widehat{G}$ which covers the corresponding arrowhead of
$\widehat{\gc}$.

\vspace{2mm}

\noi {\bf Case 2.} The case of {\em non--vanishing }
2--edges is again unmodified. Above such an  edge $e$, which in
$\gc$  has the form  (with $n>0$)

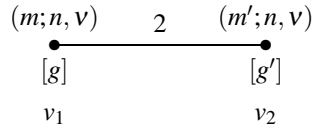
\vspace{2mm}

\begin{picture}(140,50)(-100,0)
\put(20,30){\circle*{4}} \put(100,30){\circle*{4}}
\put(20,30){\line(1,0){80}}
\put(20,35){\makebox(0,0)[b]{$(m;n,\nu)$}}
\put(100,35){\makebox(0,0)[b]{$(m';n,\nu)$}}
\put(20,15){\makebox(0,0)[b]{$[g]$}}
\put(100,15){\makebox(0,0)[b]{$[g']$}}
\put(20,0){\makebox(0,0)[b]{$v_1$}}
\put(100,0){\makebox(0,0)[b]{$v_2$}}
\put(60,35){\makebox(0,0)[b]{$2$}}
\end{picture}

\vspace{1mm}

 \noi insert (cyclically) in $G$ exactly
$\n_e=\gcd(m, m',n)$ strings of type
$$Str^\circleddash\left(\,{ m\over \n_e},{m'\over \n_e};{n \over \n_e}\
\Big| \ 0,0;\nu\, \right).$$ If the edge is a loop, then we modify
the procedure as in the case of 1--loops.

\vspace{2mm}

\noi {\bf Case 3.} Finally, we have to consider the case of those
{\em vanishing}  2-edges which  support arrowheads (the others have been collapsed).

Above such an edge we insert in $\widehat{G}$ one string of type

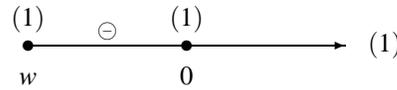
\begin{picture}(140,50)(-50,10)
%\put(-60,30){\circle*{4}} \put(0,30){\circle*{4}}
\put(60,30){\circle*{4}} \put(120,30){\circle*{4}}
\put(60,30){\vector(1,0){120}}
%\put(-60,35){\makebox(0,0)[b]{$(1)$}}
%\put(0,35){\makebox(0,0)[b]{$(m)$}}
\put(60,35){\makebox(0,0)[b]{$(1)$}}
\put(120,35){\makebox(0,0)[b]{$(1)$}}
\put(195,30){\makebox(0,0){$(1)$}}
%\put(0,15){\makebox(0,0)[b]{$0$}}
\put(120,15){\makebox(0,0)[b]{$0$}}
%\put(60,15){\makebox(0,0)[b]{$m+1$}}
%\put(30,22){\makebox(0,0)[b]{$\circleddash$}}
\put(90,32){\makebox(0,0)[b]{$\circleddash$}}
\put(60,15){\makebox(0,0)[b]{$w$}}
%\put(140,20){\makebox(0,0)[b]{$1$}}
\end{picture}

\vspace{2mm}

\noi If the 2--edge is supported in $\gc$ by a vertex which
belongs to $\gbk$, then the end--vertex $w$ of the above string should be
identified with the unique vertex of $\widehat{G}$ corresponding
to that subgraph $\gbk$, keeping its $[g_{\gbk}]$ decoration too.
\end{bekezdes}

\begin{bekezdes}{\bf Step 3. --- Determination of the
missing Euler numbers.}\labelpar{s3m} The first two steps provide
a graph with the next decorations: the multiplicities of all the
vertices, all the genera, some of the Euler numbers and all the
sign--decorations of the edges. The missing Euler numbers  are
determined by formula (\ref{eq:2.2.1}).
\end{bekezdes}

\section{\ The output of the Collapsing Main Algorithm}\labelpar{covdatam}\setcounter{equation}{0}
Similarly as in the first case, Theorem  \ref{th:2.3.1}
guarantees that {\it there is only one cyclic graph--covering of
$\gc$ with this covering  data  (up to a graph--isomorphism).}\ix{Main Algorithm!Collapsing}

In fact, by \ref{bad:graph}, the output graph $\widehat{G}$ of the
`Collapsing Algorithm' and  the output $G$ of the original algorithm
are connected by the reduced oriented plumbing calculus:
$\widehat{G}\sim G$. Hence, Theorem \ref{MTh.1} is valid for
$\widehat{G}$ as well:
 {\it $\widehat{G}$ is a possible plumbing graph
of the pair $(\partial F,\partial F\cap V_g)$, which carries the
multiplicity system of the open book decomposition
$\arg(g):\partial F\setminus V_g\to S^1$.}
The algorithm is again compatible with the decomposition of
$\partial F$ into $\partial_1F$ and $\partial_2F$. The part
regarding $\partial_1F$ is unmodified (since $\gce$ has no
2--edges).  The graph $G_2$ transforms into $\widehat{G_2}$ similarly as $G$ transforms into
$\widehat{G}$. Moreover,
all the statements of \ref{1es2} regarding $G_1$ and $G_2$ are valid for
$\widehat{G_1}=G_1$ and $\widehat{G_2}$ with the natural modifications.
The details are left to the reader.

\begin{bekezdes}{}\labelpar{GHALG} On the other hand, the difference between $G$  and $\widehat{G}$,
the outputs of the original and the new  algorithms, both
{\em  unmodified by plumbing calculus},
can be substantial, sometimes even spectacular.  See e.g. the complete computation
in the case of  cylinders  in Chapter \ref{s:cyl}.

For example, the difference $c(G)-c(\widehat{G})$ is the sum over
all subgraphs $\gbk$  of the  expression appearing in
(\ref{eq:CYCLE}), which can be a rather large number. This will
have  crucial consequences in the discussion of the Jordan blocks
of the vertical monodromies.

Similarly as in \ref{covdata}, we will use the notation $\widehat{G}$
for the output graph obtained by the  Collapsing Main Algorithm described above
{\it unmodified} by any operation of the plumbing calculus. We adopt similar notations for the graphs
of $\partial _1F$ and $\partial _2F$, namely  $\widehat{G_1}$ and
$\widehat{G_2}$. (Recall that $G\sim \widehat{G}$. Hence there is no need to consider
a `modified $\widehat{G}$', since for that
we can use the already introduced  notation $\Gmod$, cf. \ref{covdata}).\ix{Main Algorithm!Collapsing}

\end{bekezdes}

%\setcounter{temp}{\value{section}}
%\part{Homological properties of the monodromies.}
%\setcounter{section}{\value{temp}}

\chapter{Vertical/horizontal monodromies}\labelpar{s:vh}

\section{\ The monodromy operators}\labelpar{ss:MO}\setcounter{equation}{0}
 Let
$f:(\bfc^3,0)\to (\bfc,0)$ be a hypersurface singularity with a 1--dimensional
singular locus. In general, it is rather difficult to determine
the horizontal and vertical monodromies $\{m'_{j,hor}\}_{i=1}^s$
 and $\{m'_{j,ver}\}_{i=1}^s$ of $Sing(V_f)$, especially the vertical one.
 (For the terminology, see \ref{ss:2.0b}.)
It is even more difficult to identify  the  {\it two commuting
actions simultaneously}.
\ix{monodromy}\ix{monodromy!horizontal}\ix{monodromy!vertical}

This difficulty survives at the homological level too: in the
literature there is no general treatment of the corresponding  two commuting
operators. In lack of general theory, the existing
literature is limited to few sporadic examples, which are
obtained by ad hoc methods.

Our goal here is to provide a general procedure to treat these homological objects
and to produce (in principle without any obstruction) examples as complicated as we wish.

Similarly, if we fix another germ $g$ such that $\Phi=(f,g)$ is an ICIS, then
one of the most important tasks is the computation of the  algebraic monodromy
representation of $\Z^2$ induced by $m_{\Phi,hor}$ and
$m_{\Phi,ver}$ (for their definition, see \ref{ss:ICIS}).
Our treatment will include the determination of these objects as well.

In fact, our primary targets are the following  algebraic monodromy operators:

\begin{itemize}
\item
the commuting pair $M'_{j,hor}$ and $M'_{j,ver}$, acting on $H_1(F'_j)$,
induced by $m'_{j,hor}$  and $m'_{j,ver}$ ($1\leq j\leq s$),
 \item
 the commuting pair
$M^\Phi_{j,hor}$ and $M^\Phi_{j,ver}$, acting on $H_1(F_\Phi\cap
T_j)=H_1(F'_j)^{\oplus d_j}$ (cf. \ref{felbo}),
 induced by $m^\Phi_{j,hor}$ and $m^\Phi_{j,ver}$ ($1\leq j\leq s$),
\item
the commuting pair  $M_{\Phi,hor}$ and $M_{\Phi,ver}$, acting on
$H_1(F_\Phi)$, induced by $m_{\Phi,hor}$ and $m_{\phi,ver}$,
cf. \ref{ss:ICIS}.
\end{itemize}
Here, usually, we considered homology with complex coefficients, but obviously, one
might also  consider the integral case. In fact, in some of our examples,
 the additional $\Z$--invariants will also be discussed.

We separate our discussion into two parts. In this chapter we
determine completely (via $\gc$) the character decomposition (i.e.
the semi--simple part) of the relevant $\Z^2$--representations. This includes
the characteristic polynomials of all the  monodromy
operators. Moreover, we  connect the ranks of some generalized \ix{generalized eigenspace}
eigenspaces with the combinatorics of the plumbing graph $G$ as
well. We wish to emphasize that, although we get all our results
rather automatically from the graph $\gc$, all these results are
new, and were out of reach (in this generality) with previous
techniques. This shows once more the power of $\gc$.

The second part treats the structure of the Jordan blocks.
This is considerably  harder. Our main motivation in this part is the computation of
the homology of $\partial
F$ and its algebraic monodromy action.  Since the homology of $\partial F$\ix{Milnor!fiber!boundary}
will be determined via the homology of $\partial F\setminus V_g$,
that is, using the Wang exact sequence in which the operator
$M_{\Phi,ver}-I$ appears, the determination  of the 2--Jordan blocks
of $M_{\Phi,ver}$ with eigenvalue one is
a crucial ingredient.

Therefore, in this work, regarding the vertical monodromies,
 we will concentrate only on the computation of the Jordan blocks with eigenvalue one,
although for some cases we will provide the complete picture.
This second part (including the discussion regarding the Jordan blocks of the vertical monodromies
for eigenvalue one, and the computation of homology and algebraic monodromy of $\partial F$)
constitutes the next Chapters \ref{Start}--\ref{s:th:Jordan}.

\begin{remark} Although $F_\Phi$ is not the local Milnor fiber
of a hypersurface singularity,  a convenient restriction of $\Phi$
provides a map over a sufficiently small disc (a transversal slice
of the $d$--axis at one of its
 generic points) with generic fiber $F_\Phi$ such that
 the horizontal monodromy
$M_{\Phi,hor}$ is the monodromy over the  small punctured disc. In
other words, the horizontal monodromy $M_{\Phi,hor}$ can be
`localized', that is,  it can be represented as the monodromy of a family of
curves over an arbitrarily  small punctured disc.  On the other hand, the
vertical monodromy $M_{\Phi,ver}$ {\it cannot} be localized in this
sense.  In particular, general results about  the monodromy of
families over a punctured disc cannot be applied for
the vertical monodromies. (This is
one of the reasons why their computation is so difficult. This
difficulty will be overcome  here using the  graph $\gc$.)
\end{remark}

\section{\ General facts}\labelpar{char}\setcounter{equation}{0}
The statements of the next lemma are well--known for the
horizontal monodromies by the celebrated Monodromy Theorem (see
e.g. \cite{Clemens,Landman} for the global case,
\cite{BrieskornMon,Lj} for the local case, or \cite{Kulikov} for a
recent monograph).  It  may be known for the vertical
monodromies as well, however we were not able to find a reference
for it:

\begin{lemma}\labelpar{geneig}
The eigenvalues of the operators $M'_{j,hor},\ M'_{j,hor},\
M^\Phi_{j,hor},\ M^\Phi_{j,ver}$ ($1\leq j\leq s$), respectively $M_{\Phi,hor}$ and
$M_{\phi,ver}$, are roots of unity. Moreover, the size of the
Jordan blocks cannot be larger than two.
\end{lemma}
\begin{proof}
For the operators acting on $H_1(F_\Phi)$ use the decomposition
\ref{felbont} of $F_\Phi$, and the fact that the restriction of
the geometric actions on each subset $\widetilde{F}_v$ is isotopic
to a finite action. (As a model for the proof, see e.g. \cite[\S 13]{EN}.) \ix{Eisenbud--Neumann book}
The same is true for the operators acting on
$H_1(F_\Phi\cap T_j)$. In this case only those subsets
$\widetilde{F}_v$ appear which are indexed by the non--arrowhead
vertices of $\Gamma^2_{\C,j}$. Finally, the operators
$M^\Phi_{j,*}$ and $M'_{j,*}$ are connected by a simple algebraic
operation, see \ref{felbo}(2). Compare with the proofs of
\ref{th:char} and \ref{th:Jordan} as well.
\end{proof}

\begin{remark}\labelpar{re:2blocks}
By the Monodromy Theorem valid for isolated hypersurface
singularities,  the Jordan blocks  of
$M'_{j,hor}$ with eigenvalue one must have  size one, see e.g.
\cite[(3.5.9)]{Kulikov}. By the correspondence \ref{felbo}(2),
this fact is true for $M^\Phi_{j,hor}$ as well.
Nevertheless, for the other four operators, such a restriction is
{\it not} true anymore:
In \ref{ex:ACirr2} we provide an example when all $M'_{j,ver}$, $M^\Phi_{j,ver}$
$M_{\phi,hor}$ and $M_{\phi,ver}$ have Jordan blocks of size 2 with
eigenvalue one.
\ix{monodromy!Theorem}

In fact, examples with $M'_{j,ver}$ having such a Jordan block can
be constructed as follows: Fix a topological/equisingularity type
of isolated plane curve singularity $S$ whose monodromy has a
2--Jordan block. Let $o$ be the order of the eigenvalue of this
block. Then one can construct a projective plane curve $C$ of
degree $d$ (sufficiently large), which is a multiple of $o$, and such
that $C$ has a local singularity of type $S$. Let $V_f$ be
the cone over $C$. By a result of Steenbrink \cite{Steenbrink}
$M'_{j,ver}=(M'_{j,hor})^{-d}$, hence $M'_{j,ver}$ has a Jordan
block with eigenvalue one and size two.
\end{remark}

For the number of Jordan blocks we will use the following
notation:
\begin{definition}
For any operator $M$ let $\#^k_\lambda M$ denote  the number of Jordan blocks of $M$
 of size $k$  with eigenvalue $\lambda$.
\end{definition}

%\chapter{Vertical/horizontal monodromies. Characters}\labelpar{chars}

\section{\ Characters. Algebraic preliminaries}\labelpar{ss:chars}\setcounter{equation}{0}
Let $H$ be a finite dimensional $\bfc$--vector space, and assume
that $M\in Aut(H)$. Let $P_M(t)$ (or $P_{H,M}(t)$) be the characteristic polynomial
$\det(tI-M)$ of $M$. For each eigenvalue $\lambda$, let
$H_{M,\lambda}=\ker((\lambda I-M)^N)$ (for $N$ large) be the
generalized $\lambda$--eigenspace of $M$. Obviously, the
multiplicity of $t-\lambda$ in $P_M(t)$ is exactly $\dim
H_{M,\lambda}$.
When $M$ is clear from the context, we simply write $H_\lambda=H_{M,\lambda}$.
\ix{monodromy!characters}\ix{generalized eigenspace|textbf}

Sometimes it is more convenient to replace
$P_M(t)$ by its {\it divisor}
\begin{equation}
Div(H;M):=\sum_\lambda \, \dim H_{M,\lambda}\cdot (\lambda)\in
\Z[\bfc^*].
\end{equation}
More generally, assume that  two commuting automorphisms
$M_1$ and $M_2$ act on $H$. Then, for each pair $(\lambda,\xi)\in
\bfc^*\times \bfc^*$, set $H_{(\lambda,\xi)}:=H_{M_1,\lambda}\cap
H_{M_2,\xi}$ and
\begin{equation}\label{eq:char}
Div(H;M_1,M_2):= \sum_{(\lambda,\xi)}\,
\dim(H_{(\lambda,\xi)})\cdot (\lambda,\xi)\in \Z[\bfc^*\times
\bfc^*].
\end{equation}
Above $\Z[\bfc^*]$ and $\Z[\bfc^*\times \bfc^*]$ are the group rings of
$\bfc^*$ and $\bfc^*\times \bfc^*$ over $\Z$.

\begin{bekezdes}\labelpar{keyex}{\bf Key Example.}
Fix a triple of integers  $(m;n,\nu)$ with $m,\,\nu>0$ and $n\geq
0$.   Let $\calp$ be the set of points
\begin{equation*}
\calp:=\{(u,v)\in \bfc^*\times \bfc^*\,|\,v^\nu=1,\ u^mv^n=1\}.
\end{equation*}
In fact, $\calp$ is a finite subgroup of $\bfc^*\times \bfc^*$ of
order $m\nu$.  On this set of points we define two commuting
permutations: For each pair of real numbers $(t_{hor},t_{ver})$,
consider the set of points
\begin{equation*}
\calp(t_{hor},t_{ver}):=\{(u,v)\in \bfc^*\times
\bfc^*\,|\,v^\nu=e^{it_{ver}},\ u^mv^n=e^{it_{hor}}\}.
\end{equation*}
Fixing $t_{ver}=0$ and moving $t_{hor}$ from $0$ to $2\pi$, we get
a locally trivial family of $m\nu$ points, which defines a
permutation $\sigma_{hor}$ of $\calp$. Similarly, fixing
$t_{hor}=0$ and moving $t_{ver}$ from $0$ to $2\pi$, we get the
permutation  $\sigma_{ver}$ of $\calp$. One can verify that the
two permutations commute, based for example on the fact that
the torus $\{|u|=|v|=1\}$ has an abelian fundamental group.
\ix{Key Example|textbf}

Let $H:=H_0(\calp,\bfc)$ be the vector space with base elements
indexed by the points from $\calp$, i.e. the vector space of
elements of type $\sum_{p\in\calp} c_p\cdot p$, where
$c_p\in\bfc$. For any permutation $\sigma$ of $\calp$, define
$\sigma_*\in Aut(H)$ by $$\textstyle{\sigma_*(\sum_{p}c_p\cdot
p):= \sum_pc_p\cdot \sigma(p)}.$$ Our goal is to determine
$$\Lambda(m;n,\nu):=Div(H;\sigma_{hor,*},\sigma_{ver,*})\in \Z[\bfc^*\times \bfc^*].$$
Consider the following two elements of  $\calp$:
$$\mbox{$\mathfrak{h}:=(e^{2\pi i/m},1)$ \ \ and \ \
$\mathfrak{v}:=(e^{-2\pi in/m\nu}, e^{2\pi i/\nu})$}.$$
By a computation one can verify that $\sigma_{hor}$
and  $\sigma_{ver}$ can be obtained by multiplication in
$\calp$ by $\mathfrak{h}$ and  $\mathfrak{v}$ respectively. Let
$\widetilde{\calp}$ be the subgroup of the permutation group of
$\calp$ generated by $\sigma_{hor}$ and $\sigma_{ver}$. Having the
forms of $\mathfrak{h}$ and $\mathfrak{v}$, one can  easily verify
that $\mathfrak{h}$ and $\mathfrak{v}$ (hence $\sigma_{hor}$ and
$\sigma_{ver}$ in $\widetilde{\calp}$ too) satisfy the relations
$$\mathfrak{h}^m=\mathfrak{v}^{m\nu/(m,n)}=\mathfrak{h}^n\mathfrak{v}^\nu=1,$$
and that $\widetilde{\calp}$ acts transitively on $\calp$.
(Here $(m,n)=$gcd$(m,n)$.)
Define
the group
\begin{equation}\label{gmnnu}
G(m;n,\nu):=\{(\lambda,\xi)\in\bfc^*\times \bfc^* \,:\,
\lambda^m=\lambda^n\xi^\nu=1\}.\end{equation}
 Note that for
$(\lambda,\xi)\in G(m;n,\nu)$, one automatically has
$\xi^{m\nu/(m,n)}=1$.

 Then $G(m;n,\nu)$ is isomorphic to
$\widetilde{\calp}$, and both have order $m\nu$.  Since
$\widetilde{\calp}$ acts transitively on $\calp$ (which has the
same order) both are isomorphic to $\calp$ too. An isomorphism
$\widetilde{\calp}\longrightarrow \calp$ can be
generated by $\sigma_{hor}\mapsto \mathfrak{h}$ and
$\sigma_{ver}\mapsto \mathfrak{v}$.

Therefore, for any
$(\lambda,\xi)\in G(m;n,\nu)$ there is  a
$(\lambda,\xi)$--eigenvector of the action
$(\sigma_{hor},\sigma_{ver})$ which has the form:
$$\sum _{k,l} \lambda^k\xi^l \cdot
\sigma_{hor}^{-k}\sigma_{ver}^{-l}(p_0),$$ where $p_0=(1,1)\in
\calp$, and the index set $k,l$ is taken in such a way that the
set $\{\sigma_{hor}^k\sigma_{ver}^l\}_{k,l}$ is exactly
$\widetilde{\calp}$, and each element is represented once.
In particular,
\begin{equation}\label{divmnn}
\Lambda(m;n,\nu)=\sum_{(\lambda,\xi)\in G(m;n,\nu)}\
(\lambda,\xi).
\end{equation}

As a particular example,  assume that above  one has
$\nu=1$. Then $G(m;n,1)=\{\lambda\,:\, \lambda^m=1\}$ and
$\xi=\lambda^{-n}$. Hence
\begin{equation}\label{be:nu1}
\Lambda(m;n,1)=\sum_{\lambda^m=1}\, (\lambda,\lambda^{-n}).
\end{equation}
From (\ref{divmnn}) it also follows that the characteristic
polynomials of $\sigma_{hor,*}$, respectively of $\sigma_{ver,*}$
acting on $H$ are the following:
\begin{equation}\label{eq:CHAR}
P_{\sigma_{hor,*}}(t)=    (t^{m}-1)^\nu, \ \ \
P_{\sigma_{ver,*}}(t)=(t^{m\nu/(m,n)}-1)^{(m,n)}.
\end{equation}
Moreover, the characteristic polynomial of $\sigma_{hor,*}$
restricted on the generalized 1--eigenspace $H_{\sigma_{ver,*},1}$
of  $\sigma_{ver,*}$ is
\begin{equation}\label{eq:CHAR2}
P_{\sigma_{hor,*}|H_{\sigma_{ver,*},1}}(t)=t^{(m,n)}-1.
\end{equation}
\end{bekezdes}
\ix{generalized eigenspace}

\begin{bekezdes}\labelpar{be:dcov}{\bf The `$d$--covering'.}
Let $H$ be a finite dimensional vector space with two commuting
automorphisms $M_1$ and $M_2$. Furthermore,
fix a positive integer $d$. We define  $H^{(d)}:=H^{\oplus d}$ and
its automorphisms $M_1^{(d)}$ and $M_2^{(d)}$ by
$$M_1^{(d)}(x_1,\ldots,x_d)=(M_1(x_1),\ldots,M_1(x_d)),$$
and
$$M_2^{(d)}(x_1,\ldots,x_d)=(M_2(x_d),x_1,\ldots,x_{d-1}).$$ It is
not hard to see that $Div(H^{(d)}; M_1^{(d)},M_2^{(d)})$ can be
recovered from $Div(H;M_1,M_2)$ and the integer $d$. Indeed,
consider the morphism $$\Xi^{(d)}:\Z[\bfc^*\times\bfc^*]\to
\Z[\bfc^*\times\bfc^*], \ \ \mbox{where} \ \
\Xi^{(d)}((\lambda,\xi)):=\sum_{\alpha^d=\xi}(\lambda,\alpha).$$
Then, one shows that $$\Xi^{(d)}(Div(H;M_1,M_2))=Div(H^{(d)};
M_1^{(d)},M_2^{(d)}).$$
\end{bekezdes}
\ix{d@$d$--covering!of divisors|textbf}

\begin{lemma}\labelpar{le:xi} For any fixed positive integer $d$ one
has:

\vspace{2mm}

(a) $\Xi^{(d)}$ is injective.

\vspace{2mm}

(b) $\Xi^{(d)}(\Lambda(m;n,\nu))=\Lambda(m;n,d\nu)$.
\end{lemma}
\begin{proof} (a)
If one takes $\Omega^{(d)}: \Z[\bfc^*\times\bfc^*]\to
\Q[\bfc^*\times\bfc^*]$, defined by
$\Omega^{(d)}(\lambda,\alpha)=\frac{1}{d}(\lambda,\alpha^d)$, then
$\Omega^{(d)}\circ \Xi^{(d)}$ is the inclusion
$\Z[\bfc^*\times\bfc^*]
\hookrightarrow \Q[\bfc^*\times\bfc^*]$. For (b)
notice that the system $\{\lambda^m=\lambda^n\xi^{\nu}=1,\
\alpha^d=\xi\}$ is equivalent to
$\{\lambda^m=\lambda^n\alpha^{d\nu}=1\}$.
\end{proof}

We will also need the following property of $d$--coverings provided by elementary
linear algebra.
Consider the vector space $H$ and the two commuting automorphisms $M_1$ and $M_2$ as above.
Let $H_{M_2,1}$ be the generalized 1--eigenspace associated with $M_2$, and consider the
restrictions of $M_1$ and $M_2$ (denoted by the same symbols $M_1$ and $M_2$)
to this subspace. In this way we get the triple
$(H_{M_2,1};M_1,M_2)$.
Similarly, for any positive integer $d$, we can consider
the triple $ (H^{(d)}_{M^{(d)}_2,1};M^{(d)}_1,M^{(d)}_2)$.

\begin{lemma}\labelpar{lem:dcovuj}
For any triple $(H;M_1,M_2)$ and for any $d$, one has an isomorphism  of  triples
$$(H_{M_2,1};M_1,M_2)\approx (H^{(d)}_{M^{(d)}_2,1};M^{(d)}_1,M^{(d)}_2).$$
\end{lemma}
\begin{proof}
First note that we may assume that all the eigenvalues of $M_2$ are equal to 1.
Then, it is convenient to write $M_2$ as $\widetilde{M}_2^d$ for some $\widetilde{M}_2$ which commutes with
$M_1$. This can be done as follows: if $M_2=I+N$, where $I$ is the identity and  $N$ is a nilpotent operator, then
$$\widetilde{M}_2=(I+N)^{1/d}:=I+\frac{1}{d}N+\frac{1}{2!d}\Big(\frac{1}{d}-1\Big)N^2+\cdots.$$
Consider  the matrix identities
$$\begin{bmatrix}0&0&\cdots&I\\
\widetilde{M}_2^{d-1}&0& \cdots&0\\
\cdots&\cdots&\cdots&\cdots\\
0&\cdots&\widetilde{M}_2&0\end{bmatrix}
\begin{bmatrix}0&I&0&\cdots\\
0&0&I& \cdots\\
\cdots&\cdots&\cdots&\cdots\\
\widetilde{M}_2^d&0&0&\cdots\end{bmatrix}
\begin{bmatrix}0&\widetilde{M}_2^{-(d-1)}&0&\cdots\\
0&0&\widetilde{M}_2^{-(d-2)}& \cdots\\
\cdots&\cdots&\cdots&\cdots\\
I&0&0&\cdots\end{bmatrix}=
\begin{bmatrix}0&\widetilde{M}_2&0&\cdots\\
0&0&\widetilde{M}_2& \cdots\\
\cdots&\cdots&\cdots&\cdots\\
\widetilde{M}_2&0&0&\cdots\end{bmatrix}
$$
and
$$\frac{1}{d}\begin{bmatrix}1&1&\cdots&1\\
1&\bar{\xi}_2& \cdots&\bar{\xi}_2^{d-1}\\
\cdots&\cdots&\cdots&\cdots\\
1&\bar{\xi}_d&\cdots &\bar{\xi}_d^{d-1}\end{bmatrix}
\begin{bmatrix}0&1&0&\cdots\\
0&0&1& \cdots\\
\cdots&\cdots&\cdots&\cdots\\
1&0&0&\cdots\end{bmatrix}
\begin{bmatrix}1&1&\cdots&1\\
1&\xi_2& \cdots&\xi_d\\
\cdots&\cdots&\cdots&\cdots\\
1&\xi_2^{d-1}&\cdots &\xi_d^{d-1}\end{bmatrix}
=
\begin{bmatrix}1& & & \\
&\xi_2& &\\
& &\cdots&\\
& & &\xi_d\end{bmatrix}
$$
where $1=\xi_1,\xi_2,\cdots, \xi_d$ are the $d$--roots of unity, and
$\bar{\xi_i}$ their conjugates. They show  that $ (H^{(d)};M^{(d)}_1,M^{(d)}_2)$
is isomorphic to $\oplus _{i=1}^d(H; M_1, \xi_i\widetilde{M}_2)$.

On the other hand, $(H; M_1,\widetilde{M}_2)$ and $(H; M_1,M_2)$ are isomorphic.
\end{proof}

\section{\ The divisors $Div_\Phi$, $Div^\Phi_j$ and $Div'_j$ in terms of $\gc$}\labelpar{charth}
\setcounter{equation}{0}

The three pairs of operators listed in section \ref{ss:MO} define three divisors. These are the following.
\begin{definition}
 We set
\begin{equation}\label{mon1char}
Div'_j:=Div(H_1(F'_j);M'_{j,hor},M'_{j,ver}) \ \ \ \ \ \ (1\leq j\leq  s),
\end{equation}
\begin{equation}\label{mon1bchar}
Div^\Phi_j:=Div(H_1(F'_j)^{\oplus
d_j};M^\Phi_{j,hor},M^\Phi_{j,ver}), \ \ (1\leq j\leq s),
\end{equation}
and
\begin{equation}\label{mon2char}
Div_\Phi:= Div(H_1(F_\Phi);M_{\Phi,hor},M_{\Phi,ver}).
\end{equation}
\end{definition}
\ix{monodromy!divisors}

\begin{bekezdes}\labelpar{notnot}{\bf Some old/new  notations.}
Recall that  $\calw(\gc)$ (respectively $\calw(\Gamma^2_{\C,j})$) denote the
set of non--arrowhead vertices of $\gc$ (respectively of
$\Gamma^2_{\C,j}$, for any $j=1,\ldots, s$). For each $w\in
\calw(\gc)$, let $C_w$ be the corresponding irreducible curve in
$\C$, $g_w$ its genus, and $\delta_w$ the number of legs associated with the star of $v$,
i.e. the number of edges in $\gc$
adjacent to $w$, where each loop contributes twice (cf. \ref{ls}).
Moreover, assume that the decoration of $w$ in $\gc$ is
$(m_w;n_w,\nu_w)$.
%$$\Lambda(w):=\Lambda(m_w;n_w,\nu_w).$$
\end{bekezdes}

With these notations, one has the following {\it A'Campo type identities},
generalizations  of the identity (\ref{eq:ACampo}) proved in \cite{AC}:

\begin{theorem}\labelpar{th:char}
\begin{equation}\label{eq:egy}
Div_\Phi-(1,1)=\sum_{w\in\calw(\gc)}\, (2g_w+\delta_w-2)\cdot
\Lambda(m_w;n_w,\nu_w).
\end{equation}
Moreover, for any $j=1,\ldots, s$,
\begin{equation}\label{eq:ketto}
Div^\Phi_j-\sum_{\xi^{d_j}=1}(1,\xi)=\sum_{w\in\calw(\Gamma^2_{\C,j})}\,(\delta_w-2)\cdot
\Lambda(m_w;n_w,\nu_w), \ \ \ \ \ \ \ \ \ \ \ \ \ \ \
%\hspace{20mm}
\end{equation}
\begin{equation}\label{eq:kettob}
Div'_j-(1,1)=\sum_{w\in\calw(\Gamma^2_{\C,j})}\,(\delta_w-2)\cdot
\Lambda(m_w;n_w,\nu_w/d_j),
\end{equation}
where $d_j=\deg(g|\Sigma_j)$, or
$d_j={\rm gcd}(\nu_w) _{w\in\calw(\Gamma^2_{\C,j})}$ by \ref{tree}.

In particular, the above formulae provide the ranks of
$H_1(F_\Phi)$ and $H_1(F'_j)$ as well:
\begin{equation}\label{eq:H1Phi} \begin{split}
{\rm rank}\, H_1(F_\Phi)=&1+
\sum_{w\in\calw(\gc)}\, (2g_w+\delta_w-2)\cdot
m_w\nu_w,\\
\mu_j'={\rm rank}\, H_1(F'_j)=& 1+
\sum_{w\in\calw(\Gamma^2_{\C,j})}\,(\delta_w-2)\cdot
m_w\nu_w/d_j.
\end{split}
\end{equation}

\end{theorem}
\begin{proof}
As in any `A'Campo type' formula (cf. \cite{AC} or (\ref{eq:ACampo}) above), it is more
\ix{A'Campo's formula}
convenient to work with a zeta--function of an action instead of
its characteristic polynomial. In the present case also, we will
determine first the element
$$D(F_\Phi):=Div(H_0(F_\Phi);M_{\Phi,hor}^0,M_{\Phi,ver}^0)-
Div(H_1(F_\Phi);M_{\Phi,hor},M_{\Phi,ver})$$
in $\Z[\bfc^*\times \bfc^*]$.  Above,
$M_{\Phi,hor}^0$ and $M_{\Phi,ver}^0$ are the horizontal and the
vertical monodromies acting on $H_0(F_\Phi)$. Since $F_\Phi$ is
connected, this space is $\bfc$, and
$M_{\Phi,hor}^0=M_{\Phi,ver}^0=Id_\bfc$.  Hence
$Div(H_0(F_\Phi);M_{\Phi,hor}^0,M_{\Phi,ver}^0)=(1,1)$, and thus  the
left hand side of (\ref{eq:egy}) is $-D(F_\Phi)$.
\ix{A'Campo's formula}

The point is that $D(F_\Phi)$ is additive with respect to a
Mayer--Vietoris exact sequence. More precisely, if we consider the
decomposition \ref{felbont}, then
$$D(F_\Phi)=\sum_{w\in \calw(\gc)}D(\widetilde{F}_w)-D(B),$$
where $B$ is the union of  `cutting circles'. Since \ix{cutting circles}
$H_0(B)=H_1(B)$ and the monodromy actions on them can also be identified,
$D(B)=0$. On the other hand, $\widetilde{F}_w$ is a regular
covering over the regular part $C_w^{reg}$ of the curve $C_w$ with
a finite fiber which can be identified with $\calp$ in the Key \ix{Key Example}
Example \ref{keyex}. Moreover, the horizontal and vertical
actions on $\widetilde{F}_w$ are induced by the corresponding
actions on $\calp$. Hence $D(\widetilde{F}_w)=\chi(C_w^{reg})\cdot
D(\calp)=\chi(C_w^{reg})\cdot\Lambda(m_w;n_w,\nu_w)$, where
$\chi(C_w^{reg})$ stands for the Euler--characteristic of
$C_w^{reg}$, and  equals $2-2g_w-\delta_w$.

The proof of (\ref{eq:ketto}) is similar.  Using the results and
notations of \ref{felbo}, one gets that $F_\Phi\cap T_j$ is
cut by `cutting circles' into the pieces
$\{\widetilde{F}_w\}_{w\in \Gamma^2_{\C,j}}$. But $F_\Phi\cap T_j$
consists of $d_j$ copies of $F'_j$. Hence, with the additional
fact that $g_w=0$ for all $w\in\calw(\Gamma^2_{\C,j})$, we get
$$-D(\cup_{d_j}F'_j; m^\Phi_{j,hor},m^\Phi_{j,hor})=
\sum_{w\in\calw(\Gamma^2_{\C,j})}\,(\delta_w-2)\cdot
\Lambda(m_w;n_w,\nu_w).$$ On the other hand, in this case, the
0--homology is different: $H_0(\cup_{d_j}F'_j,\bfc)=\bfc^{d_j}$ on
which the horizontal monodromy acts by identity and the vertical
one by cyclic permutation ---, therefore its contribution is
$\sum_{\xi^{d_j}=1}(1,\xi)$.

Now, using the special forms of $ m^\Phi_{j,hor}$ and
$m^\Phi_{j,hor}$ from \ref{felbo}(2), by \ref{be:dcov} and
\ref{le:xi} one gets that $\Omega^{(d_j)}(Div^\Phi_j)=Div'_j$,
thus  (\ref{eq:kettob}) follows too.
\end{proof}

The results \ref{th:char}, (\ref{eq:CHAR}) and (\ref{eq:CHAR2})  imply:

\begin{corollary}\labelpar{charpols} \

(a) The characteristic polynomial of $M_{\Phi,hor}$ and
$M_{\Phi,ver}$, acting on $H=H_1(F_\Phi,\bfc)$, are
\begin{equation*}\begin{split}
P_{M_{\Phi,hor}}(t)= &(t-1)\cdot \prod_{w\in \calw(\gc)}\
(t^{m_w}-1)^{\nu_w(2g_w+\delta_w-2)}, \\
P_{M_{\Phi,ver}}(t)= &(t-1)\cdot \prod_{w\in \calw(\gc)}\
(t^{m_w\nu_w/(m_w,n_w)}-1)^{(m_w,n_w)(2g_w+\delta_w-2)}.
\end{split}\end{equation*}
 The
characteristic polynomial of the restriction of $M_{\Phi,hor}$ on
the generalized  eigenspace $H_1(F_\Phi,\bfc)_{M_{\Phi,ver},1}$ is
\ix{generalized eigenspace}
$$P_{M_{\Phi,hor}|H_{M_{\Phi,ver},1}}(t)=(t-1)
\cdot \prod_{w\in \calw(\gc)}\
(t^{(m_w,n_w)}-1)^{2g_w+\delta_w-2}.\hspace{2.5cm}$$

(b) There are similar
formulae for the operators acting on $H=H_1(F_\Phi\cap T_j)$:
\begin{equation*}\begin{split}
P_{M^\Phi_{j,hor}}(t)= (t-1)^{d_j}\cdot &\prod_{w\in
\calw(\Gamma^2_{\C,j})}\
(t^{m_w}-1)^{\nu_w(\delta_w-2)},\\
P_{M^\Phi_{j,ver}}(t)= (t^{d_j}-1)\cdot &\prod_{w\in \calw(\Gamma^2_{\C,j})}\
(t^{m_w\nu_w/(m_w,n_w)}-1)^{(m_w,n_w)(\delta_w-2)},\\
P_{M^\Phi_{j,hor}|H_{M^\Phi_{j,ver},1}}(t)=(t-1)
\cdot &\prod_{w\in \calw(\Gamma^2_{\C,j})}\
(t^{(m_w,n_w)}-1)^{\delta_w-2}.
\end{split}\end{equation*}

(c) Finally, the characteristic polynomials of the local
horizontal/vertical monodromies acting on $H=H_1(F'_j,\bfc)$ are
\begin{equation*}\begin{split}
P_{M'_{j,hor}}(t)= (t-1)\cdot &\prod_{w\in \calw(\Gamma^2_{\C,j})}\
(t^{m_w}-1)^{\nu_w(\delta_w-2)/d_j},\\
P_{M'_{j,ver}}(t)= (t-1)\cdot &\prod_{w\in \calw(\Gamma^2_{\C,j})}\
(t^{m_w\nu_w/d_j(m_w,n_w)}-1)^{(m_w,n_w)(\delta_w-2)},\\
P_{M'_{j,hor}|H_{M'_{j,ver},1}}(t)= (t-1)
\cdot &\prod_{w\in \calw(\Gamma^2_{\C,j})}\
(t^{(m_w,n_w)}-1)^{\delta_w-2}.
\end{split}\end{equation*}
 Notice that the output of the
right hand sides of the formulas in (c) (a posteriori) should be independent of
the choice of the germ $g$, since the left hand sides depend only on the germ $f$.
\ix{monodromy!characteristic polynomial}

Notice also that
$P_{M^\Phi_{j,hor}|H_{M^\Phi_{j,ver},1}}(t)=P_{M'_{j,hor}|H_{M'_{j,ver},1}}(t)$.
%
%(d) In particular, from the second equation of (c):
%$$rank\, H_1(F'_j)=1+\sum _{\calw(\Gamma^2_{\C,j})}
%(\delta_w-2)m_w\nu_w/d_j.$$
\end{corollary}

\begin{remark}
The above formulas  from \ref{th:char} and \ref{charpols}
are valid even if $\gc$ does not satisfy Assumption A. \ix{Assumption A}
In this case, it might happen that
$\calw(\Gamma^2_{\C,j})=\emptyset$, see e.g. \ref{ex:loop}.  In
such a situation, by convention, $\sum _{\calw(\Gamma^2_{\C,j})}=0$ and
$\prod _{\calw(\Gamma^2_{\C,j})}=1$.
\end{remark}

\section{\ Examples}\labelpar{ex:charpol}\setcounter{equation}{0}

\begin{bekezdes}\labelpar{ex:charpol1}
Assume that  $f=x^3y^7-z^4$;  see  \ref{ex:347}  for a graph
$\gc$ with $g=x+y+z$.  Then, by \ref{charpols}(c),  the
characteristic polynomials of the two vertical monodromies
$M'_{j,ver}$ ($j=1,2$) are $(t^7-1)^3/(t-1)^3$ corresponding to
the transversal type $y^7-z^4$, and $(t^3-1)^3/(t-1)^3$
corresponding to the transversal type $x^3-z^4$. This can also be
verified  geometrically in an  elementary way: by the
Thom--Sebastiani theorem (see \cite{TS}),  in the first case
$F'_j$ homotopically is the join of 7 points with 4 points.
Analyzing the equation of $f$ we get that the vertical monodromy
is the join of the cyclic permutation of the 7 points with the
trivial permutation of the 4 points. A similar geometric
description is valid for the second case as well.

 In particular, in this case, these vertical operators have no eigenvalue 1.
\end{bekezdes}

\begin{bekezdes}\labelpar{ex:charpol2}
If $f=y^3+(x^2-z^4)^2$ (see \ref{ex:ketA2} for $\gc$), then the
transversal type is $A_2$ and $M'_{1,ver}$ has characteristic
polynomial $(t-1)^2$. Since the eigenvalues of the commuting
operator $M'_{1,hor}$ are distinct, $M'_{1,ver}$ is the
identity.\end{bekezdes}
\begin{bekezdes}\labelpar{ex:charpol3}
If  $f=x^a+y^2+xyz$ ($a=3$ or 5), cf. \ref{xyz3} and
\ref{xyz5}, then $s=1$, the transversal type is $A_1$, and
$Div'_1=(1,1)$.\end{bekezdes}
\begin{bekezdes}\labelpar{ex:charpol4}
If $f=x^2y^2+z^2(x+y)$, or $f=x^2y+z^2$, cf. \ref{nu2} and
\ref{221}, then again each transversal type is $A_1$, but
$Div'_j=(1,-1)$.
\end{bekezdes}

\begin{remark}\labelpar{ex:homch}
Assume that $f$ is homogeneous of degree $d$.  For $g$ a  generic
linear function, $\gc$ was constructed in Chapter \ref{hom}. This
says that any vertex has a decoration of type $(m;d,1)$. Moreover,
by (\ref{be:nu1}),
$$\Lambda(m;d,1)=\sum_{\lambda^m=1}\, (\lambda,\lambda^{-d}).$$
Therefore, the statement of \ref{th:char} is compatible  with
$M_{\Phi,ver}=(M_{\Phi,hor})^{-d}$, or
$M^\Phi_{j,ver}=(M^{\Phi}_{j,hor})^{-d}$  already mentioned.  See also
\ref{feature1} for more details.
\end{remark}

\section{\ Vertical monodromies and the graph $G$}\labelpar{rankH1}\setcounter{equation}{0}

\begin{bekezdes}\label{cale_calw}
As we already explained, we are primarily interested in the generalized eigenspaces
of the vertical monodromies corresponding to eigenvalue one.
Section \ref{charth} provides their ranks  in terms of $\gc$.
In this section we compute them in terms of the {\it combinatorics
of the plumbing graph $G$}.
\ix{monodromy!vertical!of an ICIS}\ix{generalized eigenspace}

The reader is invited to recall the definition of the graphs $G$
and $G_{2,j}$,  the  plumbing graphs of $(\partial F,V_g)$
and  $\partial_{2,j}F$, which were  introduced in  \ref{MTh.1} and
\ref{2}, and are kept  {\em unmodified by plumbing calculus}.

For any graph  $Gr$  with arrowheads $\cala(Gr)$ and
non--arrowheads $\calw(Gr)$, and where the arrowheads are supported by
usual  or  dash--edges, we also define $\cale_\calw(Gr)$
as the set of edges connecting non--arrowhead vertices. Recall
that $c(Gr)$ denotes the number of independent cycles in $Gr$ and
$g(Gr)$ the sum of the genus decorations of $Gr$. These numbers,
clearly, are not independent. E.g., if $Gr$ is connected, then by
an Euler--characteristic argument:
\begin{equation}\label{eq:c}
1-c(Gr)=|\calw(Gr)|-|\cale_\calw(Gr)|.
\end{equation}
\end{bekezdes}
\ix{number of independent cycles}
\ix{generalized eigenspace}

\begin{example}\labelpar{ex:gcG} Let $G$ be one of the output graphs of the
Main Algorithm \ref{algo}. In order to determine $c(G)$ and
$g(G)$, for each  $w\in\calw(\gc)$ we will rewrite the decorations
$m,n,n_1,\ldots, n_s,m_1,\ldots, m_t$ used in the Main Algorithm
 as $m_w,n_w,n_{w,1},$ $\ldots,
n_{w,s},m_{w,1},\ldots, m_{w,t}$. Then, using the formulae of
\ref{algo}, we have the following expressions in terms of $\gc$ for the
cardinalities $|\cala(G)|$, $|\calw(G)|$, and $|\cale_\calw(G)|$
of the corresponding sets associated with  $G$:
\begin{equation*}\begin{split}
|\cala(G)|=&|\cala(\gc)|,\\
|\calw(G)|=& \sum_{w\in\calw(\gc)} \, \mbox{gcd}(
m_w,n_w,n_{w,1},\ldots, n_{w,s},m_{w,1},\ldots, m_{w,t}),\\
2|\cale_\calw(G)|+|\cala(G)|=&\sum_{w\in\calw(\gc)} \, \big( \ \sum_i
\mbox{gcd}(m_w,n_w,n_{w,i})+\sum_{j} \mbox{gcd}(m_w,n_w,m_{w,j})\
\big).\end{split}\end{equation*}
 Moreover,
$$2g(G)=\sum_{w\in\calw(\gc)} \, (2g_w+\delta_w-2)\mbox{gcd}(m_w,n_w)+2|\calw(G)|
-2|\cale_w(G)|-|\cala(G)|;$$ and  $c(G)$ also follows
via (\ref{eq:c}).
\end{example}
\begin{remark}\labelpar{rem:GC}
One can verify that, in general,
$$\mbox{$g(G)\geq g(\gc)$ \ \ and \ \ $c(G)\geq c(\gc)$.}$$
Compare with Remark \ref{cgcg} above or with  \cite[(3.11)]{cyclic}.
\end{remark}

\begin{proposition}\labelpar{prop:gc}
The ranks of the generalized eigenspaces \ix{generalized eigenspace}
$H_1(F_\Phi,\bfc)_{M_{\Phi,ver},1}$ and
$H_1(F'_j,\bfc)_{M'_{j,ver},1}$ satisfy the following identities:
\begin{equation}\label{prop:gc1}
\dim\,H_1(F_\Phi,\bfc)_{M_{\Phi,ver},1}=2g(G)+2c(G)+|\cala(G)|-1.
\end{equation}
\begin{equation}\label{prop:gc2}
\dim\,H_1(F'_j,\bfc)_{M'_{j,ver},1}=2g(G_{2,j})+2c(G_{2,j})+|\cala(G_{2,j})|-1.
\end{equation}
Notice also that by \ref{charpols}, one also has the identity
%of the ranks of the next two generalized eigenspaces too:
\begin{equation}\label{eq:geneig}
\dim H_1(F_\Phi\cap T_j,\bfc)_{M^\Phi_{j,ver},1}=
\dim H_1(F'_j,\bfc)_{M'_{j,ver},1}.
\end{equation}
\end{proposition}
Above, $|\cala(G_{2,j})|$ is the number of arrowheads of
$G_{2,j}$, which, in fact, are all dash--arrows. Their number
is equal to $|\cale_{cut,j}|$, the number of `cutting edges'
adjacent to $\Gamma^2_{\C,j}$. \ix{cutting edge}

\begin{proof}
By the third formula of  \ref{charpols}(a) one has
$$\dim\,H_1(F_\Phi)_{M_{\Phi,ver},1}=1+\sum_{w\in
\calw(\gc)}\mbox{gcd}(m_w,n_w)\cdot (2g_w+\delta_w-2).$$ Then use
the identities of \ref{ex:gcG}
%, this equals to
%$1+2g(G)+2\#\cale_w(G)+\#\cala(G)-2\#\calw(G)$, hence the first
and (\ref{eq:c}). The proof of the second identity is similar.
\end{proof}

\begin{remark}\label{re:UUU} Although, the graph $\Gamma^2_{\C,j}$ is a tree
whose vertices have zero genus--decorations (cf.
\ref{tree}), in general, both $g(G_{2,j})$ and $c(G_{2,j})$
can be non--zero.

Consider for example   a line arrangement with $d$ lines and its graph
$\gc$ as in \ref{arrang}. Fix an intersection point $j\in\Pi$
contained in $m_j$ lines, and consider its corresponding  vertex $v_j$ in
$\gc$. Above $v_j$ there is only one vertex in $G$, whose genus
$\widetilde{g}_j$  via (\ref{NG}) is
$$2-2\widetilde{g}_j=(2-m_j)\cdot \mbox{gcd}(m_j,d)+m_j.$$
Therefore, $\widetilde{g}_j=0$ if and only if $(m_j-2)\cdot
(\mbox{gcd}(m_j,d)-1)=0$; hence $\widetilde{g}_j$ typically is not
zero.

An example when $c(G_{2,j})\not=0$ is provided in  \ref{ex:ACirr} for  $d$
even.

It might happen that   $G_1$ and $G_2$
have no cycles, while $G$ does; for such examples see  \ref{xyz3} or
\ref{xyz5}.
\end{remark}

\begin{bekezdes}\labelpar{GEO}
Proposition \ref{prop:gc} points out an important fact. Although
the graph $\gc$ depends on the  choice of the resolution $r$  in
\ref{construct}, hence the graphs $G$ inherit this
dependency  as well, certain  numerical invariants of $G$, describing
geometrical invariants of the original germ $f$ or of the pair
$(f,g)$, are independent of this ambiguity. Sometimes, even
more surprisingly, the role of $g$  is irrelevant too. Here is the
start of the list of such numerical invariants:

\vspace{3mm}

$\bullet$ \ $|\cala(G)|$= the number of irreducible components of
$V_f\cap V_g$;

\vspace{2mm}

$\bullet$ \ $|\cala(G_{2,j})|$= the number of gluing tori of
$\partial_{2,j}F$, cf.  \ref{gl} (independent of $g$);  \ix{gluing tori}

\vspace{2mm}

$\bullet$ \ $g(G)+c(G)$ is an invariant of $(f,g)$ by
(\ref{prop:gc1});

\vspace{2mm}

$\bullet$ \ $g(G_{2,j})+c(G_{2,j})$ is an invariant of $f$
(independent of $g$) by (\ref{prop:gc2}).

\vspace{3mm}
\noi For the continuation of the list and more comments on
$g(G)$ and $c(G)$, see \ref{re:cgdep} and \ref{ss:Jordan2}.\end{bekezdes}

Note that all these invariants are stable with respect to the {\em
reduced} oriented plumbing calculus (in fact, the definition of
the reduced set of operations relies exactly on this observation).
In particular, we also have:

\begin{corollary}\labelpar{co:stabel}
The statements of Proposition \ref{prop:gc} are valid for any graph $\Gmod$
with $\Gmod\sim G$, and  for any $\Gkjmod$ with $\Gkjmod\sim G_{2,j}$  respectively.
In particular, for $\widehat{G}$ and $\widehat{G}_{2,j}$  too.
\end{corollary}

%\setcounter{temp}{\value{section}}
%\part{Homological invariants of $\partial F$ and $\partial
%F\setminus V_g$}
%\setcounter{section}{\value{temp}}

\chapter{The algebraic monodromy of   $H_1(\partial F)$. Starting
point}\labelpar{Start}

Let us fix again an ICIS $(f,g)$.\ix{Milnor!fiber!boundary}

In order to determine   the characteristic polynomial of the Milnor
monodromy acting on $H_1(\partial F)$ we need to understand two key geometrical objects:
$$\mbox{the pair $(\partial F,\partial F\setminus V_g)$ and
the fibration $\arg(g):\partial F\setminus V_g\to
S^1$.}$$
The first pair compares  the homology of $\partial F$ and $\partial F\setminus V_g$.
Then, from the fibration  we can try to determine the cohomology of $\partial F\setminus V_g$.
This discussion will run through several chapters.
It ends with the complete description in the most significant cases,
however, the program will be obstructed in the general case.

The main reason for the lack of a complete general
description lies in the fact that for the  variation map involved,
the `uniform twist property', usually valid in
{\it complex} geometry,  is not valid  in the present {\em real analytic} situation. Here  for
edges with different decorations we glue together pieces  with
different orientations; for technical details see
\ref{s:Co}.  \ix{twist!uniform}

\section{\ The pair  $(\partial F,\partial F\setminus
V_g)$}\labelpar{COMPAR} \setcounter{equation}{0}

From the long homological exact sequence
of the pair  $(\partial F,\partial F\setminus V_g)$ we get
\begin{equation}\label{eq:LES}
\longrightarrow H_2(\partial F,\partial F\setminus
V_g)\stackrel{\partial }{\longrightarrow} H_1(\partial F\setminus
V_g)\longrightarrow H_1(\partial F)\to 0.
\end{equation}
\ix{link!of curve singularities!homology of}
By excision, $$H_2(\partial F,\partial F\setminus V_g)=
H_0(\partial F\cap V_g)\otimes H_2(D,\partial D),$$ where $D$ is a
real 2--disc, hence it is free of rank $|\cala(G)|=|\cala(\gc)|$,
the number of components of $\partial F\cap V_g$. Since the
monodromy is trivial in a neighbourhood of $\partial F\cap V_g$, cf. Theorem \ref{CC}
and Proposition \ref{felbo},
we get:
\begin{lemma}\labelpar{lem:PAR} The characteristic polynomials of
the restrictions of the Milnor monodromy action of $F$ on $\partial
F$ and on $\partial F\setminus V_g$ satisfy
$$P_{H_1(\partial F),M}(t)=P_{H_1(\partial F\setminus V_g),M}(t)\cdot
(t-1)^{-\rank\, \im \,\partial},$$ where
\begin{equation}\label{eq:EGYEN}
1\leq \rank\,\im\,\partial \leq |\cala(G)|.
\end{equation}
\end{lemma}
Here, the first inequality follows from the fact that
$H_1(\partial F)\not =H_1(\partial F\setminus V_g)$. Indeed,
consider a component of $\partial F\cap V_g$ and a small loop
around it in $\partial F$. Then its homology class  in
$H_1(\partial F)$ is zero, but it is sent into $\pm 1\in H_1(S^1)$
by $\arg(g)_*$, and hence it is non--zero in $H_1(\partial F\setminus
V_g)$.\ix{Milnor!fiber!boundary}

\vspace{2mm}

 At this generality, it is impossible to say more
about $\rank\im\partial$: both bounds in (\ref{eq:EGYEN}) are sharp. For example, one
can prove  (see Chapter \ref{s:HOMOGEN}), that if $f$ is homogeneous
of degree $d$, and the projective curve $C=\{f=0\}$ has
$|\Lambda|$ irreducible components, then $\rank\im
\partial=d-|\Lambda|+1$ and $d=|\cala(G)|$. Therefore, in the
case of arrangements $\rank\im\partial=1$, and if $C$ is
irreducible then $\rank\im\partial =d=|\cala(G)|$.

\section{\ The fibrations  $\arg(g)$}\labelpar{ss:FIBR}\setcounter{equation}{0}
\ix{monodromy!characteristic polynomial}
By Theorem \ref{CC},  the fibration $\arg(g):\partial F\setminus V_g\to
S^1$ is equivalent to the fibration $\Phi^{-1}(\partial
D_\delta)\to \partial D_\delta$ with fiber $F_\Phi$ and monodromy
$m_{\Phi,ver}$. Furthermore,  the Milnor monodromy on $\partial F\setminus V_g$
is  identified with the
induced monodromy by $m_{\Phi,hor}$. Therefore, from the Wang
exact sequence  of this second fibration
\begin{equation}\label{eq:WANg}
H_1(F_\Phi)\,\stackrel{M_{\Phi,ver}-I}{\mbox{--------}\hspace{-2mm}\longrightarrow}\, H_1(F_\Phi)
\longrightarrow H_1(\partial F\setminus V_g)\longrightarrow
H_0(F_\Phi)=\Z\to 0\end{equation}
we get
\begin{equation}\label{eq:divdiv}
P_{H_1(\partial F\setminus V_g),M}(t)= (t-1)\cdot P_{
M_{\Phi,hor}|\coker(M_{\Phi,ver}-I)}(t).
\end{equation}
This formula can be `localized' around the singular locus: the
Wang exact sequence of the fibration $\partial _{2,j}F\to S^1$
provides
\begin{equation}\label{eq:divdiv2}
P_{H_1(\partial_{2,j} F),M}(t)= (t-1)\cdot P_{
M^\Phi_{j,hor}|\coker(M^\Phi_{j,ver}-I)}(t).
\end{equation}
Since the horizontal/Milnor monodromy on $\partial_1F$ is
trivial, the left hand side  of (\ref{eq:divdiv}) differs from the
product (over $j$) of the  left hand side of (\ref{eq:divdiv2})
only by a factor of type $(t-1)^N$. This, of course, is true for
the right hand sides too: for some $N\in \Z$
one has
\begin{equation}\label{eq:MM}
 P_{M_{\Phi,hor}|\coker(M_{\Phi,ver}-I)}(t)=(t-1)^N\cdot
\prod _j  P_{M^\Phi_{j,hor}|\coker(M^\Phi_{j,ver}-I)}(t).\end{equation}

On the other hand, by  Proposition \ref{felbo},
$(H_1(\partial_{2,j}F),M^\Phi_{j,hor},M^\Phi_{j,ver})$ is the
`$d$--covering' of $(H_1(F'_j),M'_{j,hor},M'_{j,ver})$.\ix{d@$d$--covering!of divisors}
Therefore, by Lemma \ref{lem:dcovuj}, the generalized 1--eigenspaces of \ix{generalized eigenspace}
their vertical monodromies can be identified:
\begin{equation}\label{eq:divdiv3}
(H_1(\partial_{2,j}F)_{M^\Phi_{j,ver},1},M^\Phi_{j,hor},M^\Phi_{j,ver})=
(H_1(F'_j)_{M'_{j,ver},1},M'_{j,hor},M'_{j,ver}).
\end{equation}
In particular,
\begin{equation}\label{eq:divdiv4}
P_{ M^\Phi_{j,hor}|\coker(M^\Phi_{j,ver}-I)}(t)= P_{M'_{j,hor}|\coker(M'_{j,ver}-I)}(t),
\end{equation}
the second polynomial being computed at the level of the homology
of the local transversal fiber $H_1(F'_j)$. Summing up, we get
that
\begin{equation}\label{eq:divdiv5}
P_{H_1(\partial F),M}(t)=(t-1)^N\cdot \prod_jP_{
M'_{j,hor}|\coker(M'_{j,ver}-I)}(t), \end{equation} for some
integer $N$. This shows clearly, that in order to determine the
characteristic polynomial $P_{H_1(\partial F),M}$, we need to
clarify the %Jordan blocks of $M'_{j,ver}$ with eigenvalue 1, or, more precisely, the
triplet
$(H_1(F'_j)_{M'_{j,ver},1},M'_{j,hor},M'_{j,ver})$, and
 the rank of $H_1(\partial F)$ which will take care
of the  integer $N$ in (\ref{eq:divdiv5}).

\begin{bekezdes}\labelpar{be:Pszam} It is more convenient to replace
in the above expressions the coker of the operators by their corresponding images.
Moreover,  similarly as above, it is useful   to study in parallel both
 `local' (i.e. the right hand side of (\ref{eq:divdiv5})) and `global'
 (the right hand side of (\ref{eq:divdiv})) expressions.

Accordingly, we introduce the following polynomials:
\end{bekezdes}

\begin{definition}\labelpar{def:P}
Let $P^\#(t)$ be the characteristic polynomial  of $M_{\Phi,hor}$
induced on the image of $(M_{\Phi,ver}-I)$ on the generalized
1--eigenspace $H_1(F_\Phi)_{M_{\Phi,ver},1}$. \ix{generalized eigenspace}

Similarly, for any $1\leq j\leq  s$,  let $P^\#_j(t)$ be the
characteristic polynomial  of  $M'_{j,hor}$ induced on the image
of $(M'_{j,ver}-I)$ on the generalized 1--eigenspace
$H_1(F'_j)_{M'_{j,ver},1}$.
\end{definition}
\noi  Clearly, $P^\#(t)$ has degree $\#^2_1M_{\Phi,ver}$
%(which, by (\ref{eq:egyenl}) is $\leq c(G)$),
while the degree of $P^\#_j(t)$ is $\#^2_1M'_{j,ver}$.
% (which is $\leq c(G_{2,j})$).

\vspace{2mm}

Since the characteristic polynomials of the horizontal monodromies
acting on $H_1(F_\Phi)_{M_{\Phi,ver},1}$ and
$H_1(F'_j)_{M'_{j,ver},1}$ are determined in Corollary \ref{charpols},
the above facts   give
\begin{lemma}\labelpar{eq:CP}
\begin{equation*}\label{eq:CP1}
P_{M_{\Phi,hor}|H_1(\partial F\setminus V_g)}\,(t)=
\frac{(t-1)^2}{P^\#(t)}\cdot \, \prod_{w\in\calw(\gc)}\,
(t^{(m_w,n_w)}-1)^{2g_w+\delta_w-2}.\end{equation*} Moreover, for
any $j$, and for the horizontal monodromy of $\Phi$ induced on
$H_1(\partial_{2,j} F)$
\begin{equation*}\label{eq:CPj}
P_{M_{\Phi,hor}|H_1(\partial_{2,j} F)}\,(t)=\frac{
(t-1)^2}{P^\#_j(t)}\cdot \, \prod_{w\in\calw(\Gamma^2_{\C,j})}\,
(t^{(m_w,n_w)}-1)^{\delta_w-2}.\end{equation*}
\end{lemma}
%Note that the equations (\ref{eq:MM}) and (\ref{eq:divdiv4}) imply
%
% P_{M_{\Phi,hor}|coker(M_{\Phi,ver}-I)}(t)=(t-1)^M\cdot
%\prod _j  P_{M'_{j,hor}|coker(M'_{j,ver}-I)}(t).
%\end{equation}
Again, since the monodormy on $\partial_1F$ is trivial, we have
\begin{equation*}\label{eq:MM2}
P_{M_{\Phi,hor}|H_1(\partial F\setminus V_g)}\,(t)=
\prod_j \,
P_{M_{\Phi,hor}|H_1(\partial_{2,j} F)}\,(t) \ \ \mbox{up to a factor $(t-1)^N$},
\end{equation*}
and  $$\mbox{$m_w=1$ \ for \  $w\not\in \bigcup_j \calw(\Gamma^2_{\C,j})$}.$$
Thus  Lemma \ref{eq:CP} implies
\begin{equation}\label{eq:NN}
\mbox{$P^\#(t)=\prod_j P^\#_j(t)$ up to a factor  $(t-1)^N$.}\end{equation}

\vspace{2mm}

Therefore,  \ref{eq:CP} and \ref{lem:PAR} combined gives that
in order to determine $P_{H_1(\partial F),M}(t)$ one needs to find
$P^\#(t)$ (or, equivalently, all $P^\#_j(t)$) and the rank of $H_1(\partial F)$.

In the next chapters we will treat
these missing terms by different geometric methods.
%the rank of $H_1(\partial F)$ and the polynomials $P^\#$ and $P^\#_j$.

\bekezdes {\bf The size of the Jordan blocks.} By the above discussion we can  now
easily prove an addendum of Lemma \ref{geneig}
regarding the size of the Jordan blocks of different operators.

\begin{proposition}\label{prop:JORBL2}
All the Jordan blocks of the monodromy operators acting on $H_1(\partial F\setminus V_g)$ and
$H_1(\partial F)$ have size at most two.
The number of Jordan blocks of size two of the monodromy acting on $H_1(\partial F\setminus V_g)$
agrees for any fixed eigenvalue with the number of size two Jordan blocks of
$M_{\Phi,hor}$ acting on $\coker (M_{\Phi,ver}-I)$. Moreover, this number is an upper bound for the
number of Jordan blocks of size two of the monodromy acting on $H_1(\partial F)$
for any fixed eigenvalue.
\end{proposition}
\ix{monodromy!Jordan block}

\begin{proof}
By the exact sequence (\ref{eq:LES}) it is enough to prove the statement for $H_1(\partial F\setminus V_g)$.
This homology group can be inserted in the Wang exact sequence (\ref{eq:WANg}). Since
the size of the Jordan blocks of the monodromy acting on $H_1(F_\Phi)$  is at most two by
 \ref{geneig}, it is enough to show that the sequence (\ref{eq:WANg}) has an equivariant splitting.
Note that the last surjection of (\ref{eq:WANg}) is the same as $\arg_*:H_1(\partial F\setminus V_g)\to \Z$,
and the monodromy on $\Z$ acts trivially. Consider a component of $\partial F\cap V_g$, let
$\gamma$ be a small oriented meridian around it in $\partial F$. Then the class of $\gamma$ is preserved by the
Milnor monodromy (as being part of $\partial_1F$) and $\arg_*([\gamma])=1$; hence such a splitting exists.
The last two statements also follow from this discussion.
\end{proof}

\begin{remark}
In order to understand the homology of $\partial F$, one does not
need any information regarding  the generalized
$(\lambda\not=1)$--eigenspaces of
 $M_{\Phi,ver}$ and  $\oplus_j M^\Phi_{j,ver}$,
although they codify important information about the ICIS $\Phi$.
This will be the subject of  forthcoming research. Nevertheless,
for $f$ homogeneous, we will determine the complete Jordan--block
structure via the identities \ref{feature1}, see
\ref{ss:applJo}. \ix{generalized eigenspace}
\end{remark}

\chapter{The ranks of $H_1(\partial F)$ and $H_1(\partial
F\setminus V_g)$ via plumbing}\labelpar{s:Jordan}

\section{\ Plumbing homology and Jordan blocks}\labelpar{ss:Jordan1}\setcounter{equation}{0}
 We start with\ix{Milnor!fiber!boundary}
general facts regarding the rank of the first homology group
of plumbed 3--manifolds. The statements are known, at least for
negative definite graphs;  see Propositions \ref{3.4} and \ref{3.9},
which serve as models for the next discussion.
For simplicity, we will state the results for the
3--manifolds $\partial F$, $\partial F\setminus V_g$ and
$\partial_{2,j}F$, that is  for the graphs $\Gmod\sim G$ and $\Gkjmod\sim
G_{2,j}$ (cf. \ref{MTh.1} and \ref{2}), although they are
valid for any plumbed 3--manifold.

In the next definition, $Gr$ denotes either the graph $\Gmod$ or
$\Gkjmod$ (or any other graph with similar decorations, and  with
two types of vertices: non--arrowheads $\calw$ and arrowheads $\cala$).

Recall from \ref{def:incidence} that $A$ denotes the {\it intersection matrix} of $Gr$,
$\inc$ the {\it incidence matrix} of the arrows of $Gr$, and
$(A,\inc)$ is the block matrix of size $|\calw|\times (|\calw|+|\cala|)$.\ix{matrix!intersection}

\begin{definition}
Set $$\mbox{${\rm corank}\ A_{Gr}:=|\calw|-{\rm rank}\ A$ \ \ and \ \
${\rm corank}\ (A,\inc)_{Gr}:=|\calw|+|\cala|-{\rm rank}\ (A,\inc)$.}$$
\end{definition}
Note that if a graph $Gr$ has some dash--arrows (like $G_{2,j}$),
then the Euler number of the non--arrowhead supporting such dash--arrows
is not well--defined; hence $\rank A_{Gr}$ is not well--defined either.
Nevertheless, $\rank(A,\inc)_{Gr}$ is well--defined even for such graphs.

\begin{lemma}\labelpar{le:h-1} For any $\Gmod\sim G$ and
$\Gkjmod\sim G_{2,j}$ one has
\begin{equation*}\begin{split}
\rank\,H_1(\partial F)=&2g(\Gmod)+c(\Gmod)+\corank A_{\Gmod},\\
\rank\,H_1(\partial F\setminus V_g)=&2g(\Gmod)+c(\Gmod)+\corank(A,\inc)_{\Gmod},\\
%Similarly,   for each $j\in\{1,\ldots, s\}$, one has
\rank\,H_1(\partial_{2,j} F)=&2g(\Gkjmod)+c(\Gkjmod)+
\corank (A,\inc)_{\Gkjmod} \ \ (1\leq j\leq s).\end{split}\end{equation*}
In particular, for
$\rank\im\partial$ from Lemma \ref{lem:PAR} one has
$$\rank\im\partial =\corank(A,\inc)_{\Gmod}-\corank A_{\Gmod}.$$
\end{lemma}
\begin{proof}
Let $P$ be the plumbed 4--manifold associated with a plumbing
graph $Gr$ obtained by plumbing disc--bundles, cf. \ref{bek:MULT}.
Assume that the
arrowheads of $Gr$ represent the link $K\subset \partial P$.
Consider the homology exact sequence of the pair $(P,\partial
P\setminus K)$:
$$H_2(P)\stackrel{i}{\longrightarrow} H_2(P,\partial P\setminus
K)\longrightarrow H_1(\partial P\setminus K) \longrightarrow
H_1(P) \longrightarrow H_1(P,\partial P\setminus K)$$ Notice that
$H_1(P,\partial P\setminus K)=H^3(P,K)=0$, while $H_2(P,\partial
P\setminus K)=H^2(P,K)=H^2(P)\oplus H^1(K)$ (since
$K\hookrightarrow P$ is homotopically trivial). Moreover, the
morphism $i$ can be identified with $(A,\inc)$. Since the rank of
$H_1(P)$ is $2g(Gr)+c(Gr)$, the second identity follows. Taking
 $K=\cala=\emptyset$ in this argument, we get the first identity. The
last identity  follows similarly.
\end{proof}

\begin{remark}\labelpar{re:TOR}
Since $H_1(P,\Z)$ is free of rank $2g+c$, the same argument over $\Z$
shows  that $H_1(\partial P,\Z)=\Z^{2g+c}\oplus \coker (A)$. The
point is that $\coker(A)$ usually has  a $\Z$--torsion summand, as it is shown
by many examples of the present work, see for example
\ref{ex:C4}.
\end{remark}

Lemma \ref{le:h-1}  via  the identities (\ref{eq:divdiv}),  (\ref{eq:divdiv2}) and (\ref{eq:divdiv4})  reads
as
\begin{corollary}\labelpar{cor:h1} \begin{equation*}\begin{split}
\dim\,\coker(M_{\Phi,ver}-I) = &2g(\Gmod)+c(\Gmod)+\corank(A,\inc)_{\Gmod}-1,\\
\dim\,\coker(M'_{j,ver}-I) = &2g(\Gkjmod)+c(\Gkjmod)+\corank(A,\inc)_{\Gkjmod}
-1\ \ \ (1\leq j\leq s).\end{split}\end{equation*}
\end{corollary}
This combined with \ref{prop:gc} and \ref{co:stabel} gives

\begin{corollary}\labelpar{cor:Jordan}
\begin{equation*}\begin{split}
\#^2_1M_{\Phi,ver}=&c(\Gmod)-\corank(A,\inc)_{\Gmod}+|\cala(G)|,\\
\#^2_1M'_{j,ver}=&c(\Gkjmod)-\corank(A,\inc)_{\Gkjmod}+|\cale_{cut,j}| \ \ \ (1\leq j\leq s).
\end{split}\end{equation*}
\end{corollary}

\begin{remark}
Although $c(G)$ and $g(G)$ can be computed easily from the graphs
$\gc$ or $G$, for the ranks of the matrices $A$ and $(A,\inc)$ the
authors found no `easy' formula. Even in case of  concrete
examples their direct computation can be a challenge. These
`global data' of the graphs resonates with the `global
information' codified in the Jordan block structure of the
vertical monodromies.
\end{remark}

\begin{remark}\labelpar{re:cgdep} Clearly, the integers $g(Gr)$,
$c(Gr)$ and $\corank A_{Gr}$ might change under the reduced
calculus (namely, under R4), see e.g. the construction of the `collapsing' algorithm
\ref{algoim}, or the next typical example realized in \ix{Main Algorithm!Collapsing}
\ref{ss:dsmall}(3c):

%\vspace{2mm}

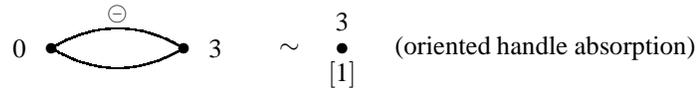
\begin{picture}(200,40)(-20,0)
\put(50,15){\circle*{4}} \put(100,15){\circle*{4}}
\qbezier(50,15)(75,30)(100,15) \qbezier(50,15)(75,0)(100,15)

\put(38,15){\makebox(0,0){$0$}} \put(112,15){\makebox(0,0){$3$}}
\put(75,28){\makebox(0,0){$\circleddash$}}
%\put(75,12){\makebox(0,0){$+$}}
 \put(140,15){\makebox(0,0){$\sim$}}

\put(160,15){\circle*{4}} \put(160,25){\makebox(0,0){$3$}}
\put(160,5){\makebox(0,0){$[1]$}}
\put(180,15){\makebox(0,0)[l]{(oriented handle absorption)}}

\end{picture}

\vspace{2mm}

\noindent  For the first graph
$A=\begin{bmatrix}0&0\\0&3\end{bmatrix}$, hence $\corank A=c=1$
and $g=0$, while for the second one $\corank A=c=0$ and $g=1$.
\end{remark}

On the other hand, the expressions $2g(G)+c(G)+\corank A_G$ and
 $2g(G)+c(G)+\corank(A,\inc)_G$ are stable under the reduced plumbing calculus;
 in fact, they are stable even under all the operations of the
oriented plumbing calculus. This fact together with \ref{GEO} show that
the following graph--expressions are  independent of the
construction of $G$ and  are also {\it stable under the reduced calculus}:

\vspace{3mm}

$\bullet$ \ $|\cala(G)|$,   $g(G)+c(G)$,  cf. \ref{GEO};

\vspace{2mm}

$\bullet$ \ $g(G)+\corank A_G$ \ and \ $g(G)+\corank(A,\inc)_G$ (from \ref{le:h-1}); in particular,
$\corank(A,\inc)_G-\corank A_G$ \ and \ $c(G)-\corank A_G$ as well;

\vspace{2mm}

$\bullet$ \ and all the corresponding  expressions for  $G_{2,j}$:
 $|\cala(G_{2,j})|$,   $g(G_{2,j})+c(G_{2,j})$,
 $g(G_{2,j})+\corank(A,\inc)_{G_{2,j}}$.

\section{\ Bounds for $\corank A$ and $\corank(A,\inc)$}
\labelpar{ss:Jordan2}\setcounter{equation}{0}

\begin{bekezdes}\labelpar{re:c} {\bf Bounds for $\corank(A,\inc)_{\Gmod}$.} From Corollary
\ref{cor:Jordan} and $|\calw|\geq rank(A,\inc)$, we get
\begin{equation}\label{eq:corankAI}
|\cala(G)|\leq \corank(A,\inc)_{\Gmod}\leq c(\Gmod)+|\cala(G)|.
\end{equation}
These inequalities are sharp: the lower bound is realized for example in the case of
homogeneous singularities (see \ref{cor:corank}), while the upper
bound is realized in the case of  cylinders, see (\ref{EX3}).
Decreasing $c(\Gmod)$ by (reduced) calculus, we decrease the difference
between the two bounds as well.

\bekezdes\label{bek:bounD} {\bf Bounds for $\corank A_{\Gmod}$.}
By the last identity of Lemma \ref{le:h-1} and the left inequality of
(\ref{eq:EGYEN}) we get
\begin{equation}\label{eq:corankA2}
\corank A_{\Gmod}\leq \corank(A,\inc)_{\Gmod}-1.
\end{equation}
This and (\ref{eq:corankAI}) imply
\begin{equation}\label{eq:corankA3}
0\leq \corank A_{\Gmod}\leq c(\Gmod)+|\cala(G)|-1.
\end{equation}
Here, again,  both bounds can be realized: the lower bound for $f$
irreducible homogeneous, cf.  \ref{cor:cora}, while the upper
bound for cylinders, see (\ref{EX3}).
\end{bekezdes}

\begin{remark} Finally, since $\corank(A,\inc)_{\Gmod}\geq |\cala(G)|$, \ref{cor:Jordan} implies
\begin{equation}\label{eq:egyenl}
\#^2_1M_{\Phi,ver}\leq c(\Gmod), \end{equation} and, similarly, for any $j$
\begin{equation}\label{eq:egyenj}
\#^2_1M'_{j,ver}\leq c(\Gkjmod).\end{equation}
In particular, if we succeed to decrease $c(\Gmod)$ by reduced calculus,
we get a better estimate for the number of Jordan blocks. In particular,
$c(\widehat{G})$, in general, is a much better estimate than $c(G)$.

In the light of Theorem \ref{CC}, the global inequality
(\ref{eq:egyenl})  reads as follows: Fix a germ $f$ with
1--dimensional singular locus, and choose $g$ such that
$\Phi=(f,g)$ forms an ICIS.  Consider the fibration
$\arg(g):\partial F\setminus V_g\to S^1$. Then (\ref{eq:egyenl})
says that the number of 2--Jordan blocks with eigenvalue 1 of the
algebraic monodromy of the fibration $\arg(g)$, {\it for any germ
$g$}, is dominated by $c(\Gmod)$. The surprising factor here is that
$\Gmod$ is a possible plumbing graph of $\partial F$, and $\partial F$
is definitely independent of the germ $g$.
\end{remark}

It is instructive to compare the above identities and inequalities
 with similar statements valid in the world of complex
geometry, see e.g. Remark \ref{rem:HOM}(2).

\chapter{The characteristic polynomial of $\partial
F$ via $P^\#$ and $P^\#_j$}\labelpar{s:CHPF}

\section{\ The characteristic polynomial of
 $G\to \gc$ and $\widehat{G}\to \widehat{\gc}$}\labelpar{s:digress}\setcounter{equation}{0}

\bekezdes In order to continue our discussion regarding the polynomials
$P^\#$ and $P^\#_j$, $1\leq j\leq s$ (cf. \ref{def:P} and \ref{eq:CP}),
we have to consider some\ix{Milnor!fiber!boundary}
natural `combinatorial' characteristic polynomials associated with the  graph coverings involved.
\ix{monodromy!characteristic polynomial!of $G\to\G_\C$}\ix{graph!covering}

Consider the cyclic graph covering $G\to \gc$, cf. \ref{algo}.
Recall that above a vertex $w\in\calw(\gc)$ there are $\n_w$
vertices, while above an edge $e\in \cale_w(\gc)$ there are exactly $\n_e$ edges of
$G$, where the integers $\n_w$ and $\n_e$ are given in  (\ref{NW})
and (\ref{NE}) of the Main Algorithm.  In particular, they can easily be
read  from the decorations of $\gc$. The cyclic action on
$G$ cyclically permutes  the vertices situating above a fixed $w$ and the edges
situating above a fixed $e$.  Let $|G|$ be the topological realization (as a
topological connected 1--complex) of the graph $G$. Then the
action induces an operator, say $\mathfrak{h}(|G|)$, on
$H_1(|G|,\bfc)$.

Similarly, we consider the covering $\widehat{G}\to \widehat{\gc}$ from \ix{graph!covering}
\ref{algoim}. The vertices of $\widehat{\gc}$ are the contracted subtrees
$[\gbk]$; and above $[\gbk]$,  in $\widehat{G}$ there are exactly $\n_{\gbk}$ vertices.
Those edges of $\widehat{\gc}$ which do not support arrowheads
are inherited from those edges $\cale_\calw^*$ of $\gc$ which connect
non--arrowheads and are not vanishing 2--edges. Each of them is covered by
$\n_e$ edges.  The cyclic action induces the
operator $\mathfrak{h}(|\widehat{G}|)$ on $H_1(|\widehat{G}|,\bfc)$.

\begin{definition}\labelpar{def:Pcov}
We denote the characteristic polynomial of $\mathfrak{h}(|G|)$ and
 of $\mathfrak{h}(|\widehat{G}|)$
%acting on  $H_1(|G|,\bfc)$ (resp. on $H_1(|\widehat{G}|,\bfc)$)
by $P_{\mathfrak{h}(|G|)}(t)$  and $P_{\mathfrak{h}(|\widehat{G}|)}(t)$ respectively.
\end{definition}
By the connectivity of $G$ and $\widehat{G}$, and by the fact that the cyclic action
acts trivially on $H_0(|G|)=H_0(|\widehat{G}|)$, we get
\begin{lemma}\labelpar{bek:comput}
\begin{equation*}
P_{\mathfrak{h}(|G|)}(t)=(t-1)\cdot \frac{\prod_{e\in
\cale_\calw(\gc)}\, (t^{\n_e}-1)}{\prod_{w\in\calw(\gc)}\,
(t^{\n_w}-1)},
\end{equation*}
\begin{equation*}
P_{\mathfrak{h}(|\widehat{G}|)}(t)=(t-1)\cdot \frac{\prod_{e\in
\cale_\calw^*(\gc)}\, (t^{\n_e}-1)}{\prod_{\gbk}\,
(t^{\n_{\gbk}}-1)}.
\end{equation*}
\end{lemma}
\begin{bekezdes}\labelpar{bek:comput2}
We list some additional  properties of these polynomials:

\begin{enumerate}
\item [(a)] Their degrees are
 $\rank\, H_1(|G|)=c(G)$ and  $\rank\, H_1(|\widehat{G}|)=c(\widehat{G})$.\vspace{2mm}

\item[(b)] Analyzing \ref{bad:graph}, one verifies  the divisibility
$P_{\mathfrak{h}(|\widehat{G}|)}\,| \, P_{\mathfrak{h}(|G|)}$.\vspace{2mm}

\item[(c)] By Lemma \ref{bek:comput}, the multiplicity of the factor $(t-1)$ in
$P_{\mathfrak{h}(|G|)}(t)$ is exactly $c(\gc)$. This  fact remains true for
$P_{\mathfrak{h}(|\widehat{G}|)}$ too, since each $\gbk$ is a tree and
$c(\gc)=c(\widehat{\gc})$.
Hence 1 is not a root of the polynomial
$P_{\mathfrak{h}(|G|)}/P_{\mathfrak{h}(|\widehat{G}|)}$.\vspace{2mm}

\item[(d)]  It might happen that  $c(\widehat{G})>c(\gc)$ (see for example \ref{ex:ACirr2}).
Hence  $\mathfrak{h}(|\widehat{G}|) $ might have non--trivial eigenvalues.
\end{enumerate}
\end{bekezdes}

\begin{bekezdes}\labelpar{def:Pcovloc}
Obviously, the above discussion can be `localized'
above the graph $\gc^2$.
With the natural notations, we set
\begin{equation*}
P_{\mathfrak{h}(|G_{2,j}|)}(t)=(t-1)\cdot \frac{\prod_{e\in
\cale_\calw(\Gamma^2_{\C,j})}\,
(t^{\n_e}-1)}{\prod_{w\in\calw(\Gamma^2_{\C,j})}\, (t^{\n_w}-1)},
\end{equation*}
\begin{equation*}
P_{\mathfrak{h}(|\widehat{G_{2,j}}|)}(t)=(t-1)\cdot \frac{\prod_{e\in
\cale_\calw^*(\Gamma_{2,j}^2)}\, (t^{\n_e}-1)}{\prod_{\gbk\subset \Gamma_{2,j}^2}\,
(t^{\n_{\gbk}}-1)}.
\end{equation*}
In fact, the product over $j$ of these localized polynomials contains
{\em all  the non--trivial eigenvalues} of
$P_{\mathfrak{h}(|G|)}$ and $P_{\mathfrak{h}(|\widehat{G}|)}$ respectively.
Indeed, in the formulae of \ref{bek:comput}, if $w$ is a vertex
of $\gc^1$ then $\n_w=1$. Similarly, if $e$ is an edge with at
least one of its  end--vertices in $\gc^1$, then $\n_e=1$ as well.
\end{bekezdes}

\section{\ The characteristic polynomial of $\partial F$}\labelpar{ss:Jordan3}\setcounter{equation}{0}

\bekezdes In this section we compute the
polynomials $P^\#$ and $P^\#_j$ under certain  additional assumptions.
Via the identities of \ref{eq:CP} and \ref{lem:PAR}
 this is sufficient  (together with the results of  Chapter \ref{s:Jordan}
regarding  the $\rank\, H_1(\partial F)$)
to determine the characteristic
polynomial of the Milnor monodromy of $H_1(\partial F)$.
\ix{monodromy!characteristic polynomial!of $\partial F$}\ix{Milnor!fiber!boundary}

Let $\gc$ be the graph read from a resolution as in \ref{gc}, which might have some
vanishing 2--edges that are not yet eliminated by blow ups considered in \ref{re:w3}.

\begin{definition}\labelpar{def:unicolor}
Let $Gr$ be either the graph $\gc$, or $\Gamma^2_{\C,j}$ for some
$j$.  We say that $Gr$ is {\em `unicolored'}, if all its edges connecting non--arrowheads
have the same sign--decoration and there are no  vanishing 2--edges
among them. We say that $Gr$ is {\em almost unicolored},
if those edges which connect non--arrowheads and are not vanishing 2--edges,
have the same sign--decoration.
\end{definition}
\ix{graph!unicolored|textbf}\ix{graph!unicolored!almost|textbf}

Consider the polynomials $P_{\mathfrak{h}(|G|)}(t)$  and
$P_{\mathfrak{h}(|G_{2,j}|)}(t)$ introduced in
\ref{s:digress}. Recall that their degrees are  $c(G)$ and $c(G_{2,j})$
respectively.

\begin{theorem}\labelpar{th:Jordan} \

(I) For any fixed $j$, the polynomial $P^\#_j(t)$ divides the
polynomial $P_{\mathfrak{h}(|G_{2,j}|)}(t)$.

Moreover, the following statements are equivalent:
\begin{equation}\label{eq:echj}\begin{array}{l}
(a) \ \ P^\#_j(t)=P_{\mathfrak{h}(|G_{2,j}|)}(t) \\
(b) \ \ \#^2_1M'_{j,ver}=c(G_{2,j})\\
(c) \ \ \corank(A,\inc)_{G_{2,j}} =|\cale_{cut,j}|.
\end{array}\end{equation}
 These equalities hold
in the following situations: either (i) $\Gamma^2_{\C,j}$ is
unicolored, or after determining  $G_{2,j}$ via the Main
Algorithm \ref{algo}, the graph $G_{2,j}$ satisfies  either (ii) $c(G_{2,j})=0$, or (iii)
$\corank(A,\inc)_{G_{2,j}} =|\cale_{cut,j}|$.
\ix{graph!unicolored}

\vspace{2mm}

(II) The polynomial $P^\#(t)$ divides the polynomial
$P_{\mathfrak{h}(|G|)}(t)$.

Moreover, the following statements are equivalent:
\begin{equation}\label{eq:ech}\begin{array}{l}
(a) \ \ P^\#(t)=P_{\mathfrak{h}(|G|)}(t) \\
(b) \ \ \#^2_1M_{\Phi,ver}=c(G)\\
(c) \ \ corank(A,\inc)_G =|\cala(G)|.
\end{array}\end{equation}
 These equalities hold
in the following situations: either (i) $\gc$ is unicolored, or
after finding  the graph $G$, it   either satisfies (ii) $c(G)=0$ or
(iii) $\corank(A,\inc) =|\cala(G)|$.
\ix{graph!unicolored}

\vspace{2mm}

But, even if (\ref{eq:ech}) does not hold, one has
\begin{equation}\label{POL11}
 P^\#(t)=P_{\mathfrak{h}(|G|)}(t) \ \ \mbox{up to a multiplicative factor of type
$(t-1)^N$}\end{equation}
 whenever (\ref{eq:echj}) holds for all $j$.
\end{theorem}

\begin{remark}\label{re:REMARK}
The equivalent statements (\ref{eq:ech}) are satisfied e.g. by all homogeneous
singularities (see \ref{cor:corank}), as well as  by all cylinders, provided that the algebraic monodromy
of the corresponding plane curve singularity is finite, see (\ref{EX3}). On the other hand, they are not
satisfied by those cylinders, which do not satisfy the above monodromy restriction.
Nevertheless, their case will be covered by the `collapsing'
version \ref{th:Jordanim}.\ix{Main Algorithm!Collapsing}
\end{remark}

The proof of Theorem \ref{th:Jordan} is given in Chapter \ref{s:th:Jordan}. The major application targets
the characteristic polynomials of the
monodromy acting   on  $H_1(\partial F)$:
\begin{theorem}\labelpar{th:charpolG}
 Assume that (\ref{eq:echj}) holds for all $j$, or (\ref{eq:ech}) holds. Then
%\begin{equation*}%\label{eq:hah}
%P_{H_1(\partial F\setminus V_g),M}\,(t)=
%\frac{(t-1)^{2+corank(A,\inc)_G-\#\cala(G)}}{P_{\mathfrak{h}(|G|)}(t)}\cdot
%\, \prod_{w\in\calw(\gc)}\,
%(t^{(m_w,n_w)}-1)^{2g_w+\delta_w-2}\end{equation*}
\begin{equation*}%\label{eq:hah}
P_{H_1(\partial F),M}\,(t)=
\frac{(t-1)^{2+\corank A_G-|\cala(G)|}}{P_{\mathfrak{h}(|G|)}(t)}\cdot
\, \prod_{w\in\calw(\gc)}\,
(t^{(m_w,n_w)}-1)^{2g_w+\delta_w-2}.\end{equation*}
In particular, one has the following formulae for the ranks of eigenspaces:
\begin{equation}\label{rankegy}
\rank\, H_1(\partial F)_1=2g(\gc)+c(\gc)+\corank A_G,
\end{equation}
\begin{equation}\label{ranknemegy}
\rank\, H_1(\partial F)_{\not=1}=2g(G)+c(G)-2g(\gc)-c(\gc).
\end{equation}
%This can also be rewritten as
%\begin{equation*}%\label{eq:vegso}
%\frac{(t-1)^{1+c(G)-{\rm deg}P^\#}\cdot\prod_{w\in\calw(\gc)}\,
%(t^{(m_w,n_w)}-1)^{2g_w+\delta_w-2} \cdot (t^{\n_w}-1)}
%{\prod_{e\in \cale_w(\gc)}\, (t^{\n_e}-1)}.\end{equation*}
%(If (\ref{eq:ech}) holds, then the identities can be simplified
%by  $corank(A,\inc)_G=\#\cala(G)$.)
%
%\vspace{1mm}
More  generally, in any situation (i.e. even if the above assumptions are not satisfied),
there exists a polynomial $Q$ with $Q(1)\not=0$,  which
divides both $P_{\mathfrak{h}(|G|)}$ and
 $\prod_{w\in\calw(\gc)}\, (t^{(m_w,n_w)}-1)^{2g_w+\delta_w-2}$,
such that
\begin{equation*}%\label{eq:hah}
P_{H_1(\partial F),M}\,(t)=
\frac{(t-1)^{N}}{Q(t)}\cdot
\, \prod_{w\in\calw(\gc)}\,
(t^{(m_w,n_w)}-1)^{2g_w+\delta_w-2},\end{equation*}
where $N=2+\corank A_G-|\cala(G)|-c(G)+\deg(Q)$.
\end{theorem}

\begin{proof}
Use \ref{charpols}, \ref{le:h-1}, \ref{cor:h1} and \ref{th:Jordan}.
\end{proof}

\begin{corollary}\labelpar{cor:EIG}
 Assume that (\ref{eq:echj}) holds for all $j$, or (\ref{eq:ech}) holds.
 Then the following facts hold over coefficients in $\bfc$:
\begin{enumerate}
\item[(a)] The intersection matrix
$A_G$ has a generalized  eigen--decomposition $(A_G)_{\lambda=1}\oplus (A_G)_{\lambda\not=1}$
 induced by the Milnor monodromy,
 and $(A_G)_{\lambda\not=1}$ is non--degenerate; \ix{generalized eigenspace}

\item[(b)] There exist subspaces  $K^i\subset H^i(\partial F)_1$ \ for
 $i=1,2$ with
 $$ {\rm codim}\, K^1=\dim K^2=\corank A_G$$ such that
 the cup--product $H^1(\partial F)_\lambda \cup
 H^1(\partial F)_\mu\to H^2(\partial F)_{\lambda \mu}$ has the following properties:
\ix{matrix!intersection@$(A_G)_{\lambda\not=1}$ non--degenerate}

\begin{enumerate}
\item [(1)] $H^1(\partial F)_\lambda\cup H^1(\partial F)_\mu=0$ \ for $1\not=\lambda\not=\overline{\mu}\not= 1$;
\item [(2)]$\oplus _{\lambda\not=1}\ H^1(\partial F)_\lambda\cup
 H^1(\partial F)_{\overline{\lambda}}\subset K^2$;
\item [(3)] $(\oplus _{\lambda\not=1}\ H^1(\partial F)_\lambda) \cup K^1 =0$;
\item [(4)] $K^1\cup K^1\subset K^2$;
\item [(5)] $K^1\cup K^2=0$.
\end{enumerate}
\end{enumerate}
\end{corollary}

\begin{proof} We combine the proof of \ref{le:h-1} with \ref{ss:EXT}.
Set $P_G:=\overline{\cals}_k$, and consider the cohomological
long exact sequence associated with the pair $(P_G,\partial P_G)$. The map
$H^2(P_G,\partial P_G)\to H^2(P_G)$ can be identified with $A_G$. The sequence has a
generalized eigenspace decomposition. Define
 $K^i:=\im[H^i(P_G)_1\to H^i(\partial F)_1]\subset H^i(\partial F)_1$ for
 $i=1,2$.
For  $\lambda\not =1$,
via (\ref{ranknemegy}) and (\ref{eq:EXT}), we get that the inclusion
$H^1(P_G)_{\lambda\not=1}\to H^1(\partial P_G)_{\lambda\not=1}$ is an isomorphism.
Hence $(A_G)_{\lambda\not=1}$ is non--degenerate.

For the second part, %consider $K^1:=\mbox{im}\, [H^1(X)_1\to H^1(\partial X)_1]$,
lift the classes of $H^1(\partial X)$ to $H^1(X)$ before  multiplying them. (5) follows from
$H^3(X)=0$. (Cf. also with \cite{SU}.)
\end{proof}

Such a result, in general, is the by--product of a structure--theorem regarding
the mixed Hodge structure of the cohomology ring
(that is, it is the consequence of the fact that
$H^2$ has no such weight where the product would sit).
Compare also with the comments from \ref{re:nss} and with some of the  open problems  from \ref{OPEN-PROB}.
\bekezdes
The next theorem, based on the `Collapsing Main Algorithm' and the corresponding improved
version of the proof of Theorem \ref{th:Jordan}, provides a better `estimate' in the general case,
and in the special  cases of {\it almost unicolored} graphs
 handles the presence of vanishing 2--edges as well.\ix{Main Algorithm!Collapsing}
\ix{graph!unicolored!almost}\ix{mixed Hodge structure!weight filtration}

Consider the polynomials $P_{\mathfrak{h}(|\widehat{G}|)}(t)$  and
$P_{\mathfrak{h}(|\widehat{G_{2,j}}|)}(t)$ introduced in
\ref{s:digress}. Recall that their degrees are  $c(\widehat{G})$ and $c(\widehat{G_{2,j}})$
respectively.

\begin{theorem}\labelpar{th:Jordanim} \

(I) For any fixed $j$, the polynomial $P^\#_j(t)$ divides the
polynomial $P_{\mathfrak{h}(|\widehat{G_{2,j}}|)}(t)$.

Moreover, the following statements are equivalent:
\begin{equation}\label{eq:echjim}\begin{array}{l}
(a) \ \ P^\#_j(t)=P_{\mathfrak{h}(|\widehat{G_{2,j}}|)}(t) \\
(b) \ \ \#^2_1M'_{j,ver}=c(\widehat{G_{2,j}})\\
(c) \ \ \corank(A,\inc)_{\widehat{G_{2,j}}} =|\cale_{cut,j}|.
\end{array}\end{equation}
 These equalities hold
in the following situations: either (i) $\Gamma^2_{\C,j}$ is almost
unicolored, or after determining   $\widehat{G}_{2,j}$ via the Collapsing  Main
Algorithm  the graph  $\widehat{G}_{2,j}$ satisfies
 either (ii) $c(\widehat{G_{2,j}})=0$, or (iii)
$\corank(A,\inc)_{\widehat{G_{2,j}}} =|\cale_{cut,j}|$.\ix{graph!unicolored!almost}

\vspace{2mm}

(II) The polynomial $P^\#(t)$ divides the polynomial
$P_{\mathfrak{h}(|\widehat{G}|)}(t)$.

Moreover, the following statements are equivalent:
\begin{equation}\label{eq:echim}\begin{array}{l}
(a) \ \ P^\#(t)=P_{\mathfrak{h}(|\widehat{G}|)}(t) \\
(b) \ \ \#^2_1M_{\Phi,ver}=c(\widehat{G})\\
(c) \ \ \corank(A,\inc)_{\widehat{G}} =|\cala(G)|.
\end{array}\end{equation}
 These equalities hold
in the following situations: either (i) $\gc$ is almost unicolored, or
after determining the graph $\widehat{G}$, it  either satisfies  (ii) $c(\widehat{G})=0$ or
(iii) $\corank(A,\inc)_{\widehat{G}} =|\cala(G)|$.\ix{graph!unicolored!almost}

But, even if (\ref{eq:echim}) does not hold, one has
\begin{equation}\label{POL11im}
 P^\#(t)=P_{\mathfrak{h}(|\widehat{G}|)}(t) \ \ \mbox{up to a multiplicative factor of type
$(t-1)^N$}\end{equation}
 whenever (\ref{eq:echjim}) holds for all $j$.
\end{theorem}
This implies:
\begin{theorem}\labelpar{th:charpolGim}
 Assume that (\ref{eq:echjim}) holds for all $j$, or (\ref{eq:echim}) holds. Then
%\begin{equation*}%\label{eq:hah}
%P_{H_1(\partial F\setminus V_g),M}\,(t)=
%\frac{(t-1)^{2+corank(A,\inc)_{\widehat{G}}-\#\cala(G)}}{P_{\mathfrak{h}(|\widehat{G}|)}(t)}\cdot
%\, \prod_{w\in\calw(\gc)}\,
%(t^{(m_w,n_w)}-1)^{2g_w+\delta_w-2}\end{equation*}
\begin{equation*}%\label{eq:hah}
P_{H_1(\partial F),M}\,(t)=
\frac{(t-1)^{2+\corank A_{\widehat{G}}-|\cala(G)|}}{P_{\mathfrak{h}(|\widehat{G}|)}(t)}\cdot
\, \prod_{w\in\calw(\gc)}\,
(t^{(m_w,n_w)}-1)^{2g_w+\delta_w-2}\end{equation*}
In particular,
\begin{equation}\label{rankegyim}
rank\, H_1(\partial F)_1=2g(\gc)+c(\gc)+\corank A_{\widehat{G}},
\end{equation}
\begin{equation}\label{ranknemegyim}
rank\, H_1(\partial F)_{\not=1}=2g(\widehat{G})+c(\widehat{G})-2g(\gc)-c(\gc).
\end{equation}
%This can also be rewritten as
%\begin{equation*}%\label{eq:vegso}
%\frac{(t-1)^{1+c(G)-{\rm deg}P^\#}\cdot\prod_{w\in\calw(\gc)}\,
%(t^{(m_w,n_w)}-1)^{2g_w+\delta_w-2} \cdot (t^{\n_w}-1)}
%{\prod_{e\in \cale_w(\gc)}\, (t^{\n_e}-1)}.\end{equation*}
%(If (\ref{eq:echim}) holds, then the identities can be simplified
%by  $corank(A,\inc)_{\widehat{G}}=\#\cala(G)$.)
%
%\vspace{1mm}
More  generally, in any situation (even if the above assumptions are not satisfied),
there exists a polynomial $Q$ with $Q(1)\not=0$, which
divides both $P_{\mathfrak{h}(|\widehat{G}|)}$ and
 $\prod_{w\in\calw(\gc)}\, (t^{(m_w,n_w)}-1)^{2g_w+\delta_w-2}$,
such that
\begin{equation*}%\label{eq:hah}
P_{H_1(\partial F),M}\,(t)=
\frac{(t-1)^{N}}{Q(t)}\cdot
\, \prod_{w\in\calw(\gc)}\,
(t^{(m_w,n_w)}-1)^{2g_w+\delta_w-2},\end{equation*}
where $N=2+\corank A_{\widehat{G}}-|\cala(G)|-c(\widehat{G})+\deg(Q)$.
\end{theorem}

\begin{corollary}\labelpar{cor:EIGGG}
 Assume that (\ref{eq:echjim}) holds for all $j$, or (\ref{eq:echim}) holds.
Then the statement of Corollary \ref{cor:EIG}
 is valid  provided that we replace $G$ by $\widehat{G}$.
\end{corollary}

\chapter{The proof of the characteristic polynomial formulae}\labelpar{s:th:Jordan}
%(\ref{th:Jordan}) and (\ref{th:Jordanim})

\section{\ Counting Jordan blocks of size 2}\labelpar{s:Co}\setcounter{equation}{0}
Our goal is to prove Theorem \ref{th:Jordan}.
The proof  is based on a specific construction.  The presentation is written for
the graph $G$ (later adapted  to  $\widehat{G}$ as well),
 but it can be reformulated for $G_{2,j}$ as well.
\ix{monodromy!Jordan block}

\begin{bekezdes}\labelpar{qp}{\bf The vertical monodromy $m_{\Phi,ver}$ as a quasi--periodic action.}
First, we wish to understand the geometric monodromy
$m_{\Phi,ver}:F_\Phi\to F_\Phi$. For this
the topology of fibred links, as it is described in
\cite[\S13]{EN}, will be our model. (Although, in [loc.cit.] the machinery is based on
splice diagrams, section 23 of \cite{EN} gives the necessary hints
 for plumbing graphs as well.)\ix{Eisenbud--Neumann book}
\ix{quasi--periodic action}
\ix{variation operator}
\ix{twist}\ix{twist!uniform}

Nevertheless, our situation is more complicated.  First, in the present situation we will
have {\it three} local types/contributions;  two of them do not
appear in the classical complex analytic case of \cite{EN}.
Secondly, the basic property which is satisfied by
analytic germs defined on normal surface
singularities, namely that a `variation operator' has a uniform twist,
in our situation is not true,  it is ruined by the new local contributions.

\vspace{2mm}

Consider the  graph $\gc$ from Chapter \ref{ss:1.3}.  For
simplicity we will write $\cale$, $\cale_\calw$, $\calw$, etc. for $\cale(\gc)$, $\cale_\calw(\gc)$,
$\calw(\gc)$, etc.
It is the dual graph of the curve configuration  $\C\subset V^{emb}$, where
$r:V^{emb}\to U$ is a representative of an embedded resolution of
$(V_f\cup V_g,0)\subset (\bfc^3,0)$. Then, by  \ix{curve configuration $\C$}
\ref{connectedhez}, for any tubular neighbourhood $T(\C)\subset
V^{emb}$ of $\C$, one has an inclusion $(\Phi\circ r)^{-1}(c_0,d_0)$
for $(c_0,d_0)\in\Wedge$. For simplicity, we write $\wfp$  for the
diffeomorphically lifted fiber $(\Phi\circ r)^{-1}(c_0,d_0)$ of $\Phi$.

Next, consider the decomposition \ref{felbont} of $\wfp$.
More precisely, for any intersection point $p\in C_v\cap C_u$
(or, self--intersection point of $C_v$ if $v=u$), which
corresponds to the edge $e\in\cale$ of $\gc$, let $T(e)$
 be a small closed ball centered at
$p$. Let  $T^\circ(e)$ be its interior. Then, for $(c_0,d_0)$ sufficiently
close to the origin, $\wfp\cap T(e)$ is a union of annuli.
Moreover, $\wfp\setminus \cup_eT^\circ (e)$ is a union
$\cup_{v\in\calv} \wfv$, where $\wfv$ is  in a small
tubular neighbourhood of $C_v$ ($v\in \calv$).

If $v$ is an arrowhead supported by an edge $e$,
then the inclusion $\wfv\cap \partial T(e)\subset \wfv$ admits a
strong deformation retract. Hence the pieces $\{\wfv\}_{v\in \cala}$, and the
separating annuli $\{\wfp\cap T(e)\}_{e\in \cale\setminus
\cale_\calw}$  can be neglected. Thus, instead of $\wfp$,
we will consider only
$$\wfp^*:=\wfp\setminus \,(\, \bigcup_{e\in \cale\setminus \cale_\calw}T(e)\cup
\bigcup_{v\in\cala}\wfv\,).$$
In particular, $\wfp^*$ is separated by the annuli $\{\wfp\cap
T(e)\}_{e\in\cale_\calw}$ in surfaces $\{\wfw\}_{w\in\calw}$, and each
$\wfw$ is the total space of a covering, where the base space is
$C_w\setminus \cup_e T(e)$ and the fiber is isomorphic to
$\calp_w:=\calp$ from the Key Example \ref{keyex}.

Moreover, one might choose the horizontal/vertical monodromies of
$\wfp^*$ in such a way that they will preserve this decomposition,
and their action on $\wfw$ will be induced by the permutations
$\sigma_{w,hor}:=\sigma_{hor}$, respectively
$\sigma_{w,ver}:=\sigma_{ver}$ acting on $\calp_w$ (cf. \ref{keyex}).
This shows that $m_{\Phi,ver}$ is isotopic to an action
$\widetilde{m}_{\Phi,ver}$, which preserves the above
decomposition, and its restriction on each $\wfw$ is  finite.
Let $q$ be a common multiple of the orders of
$\{\sigma_{w,ver}\}_{w\in\calw}$. Then
$\widetilde{m}_{\Phi,ver}^q$ is the identity on each $\wfw$, and acts
as a `twist map' on each separating annulus. In the topological
characterization of $\widetilde{m}_{\Phi,ver}$,
this twist is crucial.
\end{bekezdes}

\begin{definition}\labelpar{def:twist}\cite[\S13]{EN} \ix{Eisenbud--Neumann book}
Let $h:A\to A$ be a homeomorphism of the oriented annulus $A=S^1\times [0,1]$
with $h|\partial A=id$. The (algebraic) twist of $h$ is defined as the
intersection number
$${\rm twist}(h):=(x, {\rm var}_h(x)),$$
where $x\in H_1(A,\partial A,\Z)$ is a generator, and the variation map
${\rm var}_h:H_1(A,\partial A,\Z)\to H_1(A,\Z)$ is defined by
${\rm var}_h([c])=[h(c)-c]$, for any relative cycle $c$.

More generally, if $B$ is a disjoint union of annuli and $h:B\to B$
is a homeomorphism with $h^q|\partial B=id$ for some integer  $q>0$,
 for any component $A$ of $B$ define
$${\rm twist}(h;A):=\frac{1}{q}\, {\rm twist}(h^q|A).$$
\end{definition}
Notice that ${\rm
twist}(h;A)$, defined in this way,  is independent of the choice
of $q$.
 \ix{twist|textbf}

\begin{example}\labelpar{ex:twist} In the `classical' situation
one considers an analytic family of curves (over a small disc),
where the  central fiber is a normal crossing divisor and the generic fiber is smooth.
The generic fiber is cut by separating annuli, which are situated   in the
neighbourhood of the normal crossing intersection point of the
central fiber. Around such a point $p$, in convenient local coordinates
$(u,v)\in(\bfc^2,p)$, the family is given by the fibers of
$f(u,v)=u^av^b$ for two positive integers $a$ and $b$. Hence,  the union
of annuli is the fiber $f^{-1}(\epsilon)$ intersected with a small
ball centered at $p$ (and $\epsilon$ is small with
respect to the radius of this ball). It consists of  $\mbox{gcd}(a,b)$
annuli. The Milnor monodromy action $h$ is induced by
$[0,2\pi]\ni t\mapsto f^{-1}(\epsilon e^{i t})$. In order to make
the computation, one must choose $h$ in such a way that its
restriction on the boundary `near'  the $x$--axis is finite of
order $a$, and similarly, on the boundary components `near' the
$y$--axis is finite of order $b$. Then, one can show (see
e.g. \cite[page 164]{EN}) that the twist, for each connected
component $A$, is
$${\rm twist}(h;A)=-\frac{\mbox{gcd}(a,b)}{ab}.$$\end{example} \ix{Eisenbud--Neumann book}
\bekezdes
In fact, in most of the forthcoming   arguments, what is really important is not the
{\it value} of the twist itself, but its {\it sign}.
\begin{definition} In a geometric situation as in \ref{qp},
we say that we have a {\it uniform twist}, if for all the
annuli the signs of all the twists are the same.\ix{twist!uniform|textbf}
\end{definition}
\begin{example}\labelpar{ex:twistuj} In our present situation
of \ref{qp}, we are  interested in the twist of the separating
annuli associated with the edges $e\in\cale_\calw$. Depending on whether
the edge is of type 1 or 2, we have to consider two
different situations. For both cases we will consider the local
equations from \ref{summary}.

If $e$ is a 1--edge, then the fiber (union of annuli) and the vertical
monodromy action are given by
$$u^mv^nw^l=c,\ \ v^\nu w^\lambda=de^{it} \ (\mbox{with $(c,d)$ constant, and} \
t\in [0,2\pi]).$$

If $e$ is a 2--edge, then the fiber (union of annuli) and the vertical
monodromy action are given by
$$u^mv^{m'}w^n=c,\ \  w^\nu=de^{it} \ (\mbox{with $(c,d)$ constant, and} \
t\in [0,2\pi]).$$
A 2--edge $e$ is a vanishing 2--edge if $n=0$.

We invite the  reader to compute
the corresponding twists exactly, in both cases. In the present
proof we need only the following statement.
\end{example}

\begin{lemma}\labelpar{lem:twist} Fix an edge $e\in\cale_\calw$, set $B:=\wfp\cap
T(e)$, and let $A$ be one of the connected components of $B$.\vspace{2mm}

(a) If $e$ is a 1--edge, then  ${\rm twist}(h;A)<0$.\vspace{2mm}

(b) If $e$ is a non--vanishing 2--edge, then  ${\rm twist}(h;A)>0$.\vspace{2mm}

(c) If $e$ is a vanishing 2--edge, then  ${\rm twist}(h;A)=0$.
\end{lemma}
\ix{vanishing 2--edge}

\begin{proof} The first case behaves as a `covering of degree $m$' of the classical
  case $v^\nu w^\lambda=e^{it}$
($t\in [0,2\pi]$), which was exemplified in  \ref{ex:twist}.
The monodromy of the second case behaves as
the inverse of the monodromy of the classical monodromy operator:
$u^mv^{m'}=e^{-nit/\nu}$ ($t\in [0,2\pi]$). Finally, assume that
$n=0$. Then the restriction of the vertical monodromy to $\partial
B$ has order $\nu$, hence one can take $q=\nu$. But $h^\nu$
extends as the identity on the whole $B$. The details are left to
the reader.
\end{proof}

\begin{bekezdes}\labelpar{diagram}
Now we continue   the
proof of Theorem \ref{th:Jordan}. The structure of the 2--Jordan blocks of
$M_{\Phi,ver}$ is codified in the following commutative diagram, as given in
 \cite[(14.2)]{EN}: \ix{Eisenbud--Neumann book}
\end{bekezdes}

\begin{picture}(300,100)(-5,0)
\put(0,80){\makebox(0,0){$H_1(\wfp^*)$}}
\put(100,80,){\makebox(0,0){$H_1(\wfp^*,\wfW)$}}
\put(190,80){\makebox(0,0){$H_0(\wfW)$}}
\put(270,80){\makebox(0,0){$H_0(\wfp^*)$}}
\put(0,20){\makebox(0,0){$H_1(\wfp^*)$}}
\put(100,20){\makebox(0,0){$H_1(B)$}}
%\put(190,20){\makebox(0,0){$H_2(\wfp^*,B)$}}
\put(25,50){\makebox(0,0){\footnotesize{$M^q_{\Phi,ver}-I$}}}
\put(110,50){\makebox(0,0){\footnotesize{$T$}}}
\put(190,50){\makebox(0,0){$H_1(B,\partial B)$}}
\put(155,65){\makebox(0,0){\footnotesize{exc}}}
\put(155,35){\makebox(0,0){\footnotesize{var}}}
\put(50,85){\makebox(0,0){\footnotesize{$i$}}}
\put(50,25){\makebox(0,0){\footnotesize{$j$}}}
\put(320,80){\makebox(0,0){$0$}} %\put(250,20){\makebox(0,0){$0$}}
%\put(150,25){\makebox(0,0){\footnotesize{$\partial$}}}

\put(25,80){\vector(1,0){42}} \put(75,20){\vector(-1,0){50}}
\put(132,80){\vector(1,0){32}} \put(215,80){\vector(1,0){30}}
\put(290,80){\vector(1,0){20}} %\put(160,20){\vector(-1,0){40}}
%\put(240,20){\vector(-1,0){20}}

\put(0,70){\vector(0,-1){40}}
\put(100,70){\vector(0,-1){40}}
\put(130,70){\vector(2,-1){35}}
\put(165,47){\vector(-2,-1){40}}
\end{picture}

Above we use the following notations:
$$\wfW:=\cup_{w\in\calw}\wfw, \ \ \  B:=\cup_{e\in\cale_\calw}(\wfp^*\cap
T(e)), $$ \begin{verse}
\hspace{1cm}$\bullet$ $q$ is a positive integer as in \ref{qp}, hence $M^q_{\Phi,ver}$
is unipotent, \\
\hspace{1cm}$\bullet$ var is the variation map associated with $\tilde{m}^q_{\Phi,ver}$,\\
\hspace{1cm}$\bullet$ exc is the excision {\it isomorphism}, and  $T$ is  the
composite ${\rm var}\circ {\rm exc}$,\\
\hspace{1cm}$\bullet$ the horizontal line is a homological exact sequence.
\end{verse}
Since $B$ is the disjoint union of the separating annuli, ${\rm
var}$ is a diagonal map.  On the diagonal, each entry corresponds to
an annulus $A$, and equals the integer $$q\cdot {\rm
twist}(m_{\Phi,ver};A)$$ determined in  \ref{lem:twist}.

Obviously, the number of {\it all} 2--Jordan blocks of
$M_{\Phi,ver}$ is $\rank {\rm im}(M^q_{\Phi,ver}-I)$. On the other
hand, by the commutativity of the diagram,
\begin{equation}\label{eq:im}
{\rm im}(M^q_{\Phi,ver}-I)\simeq j\circ T({\rm im}(i))\simeq {\rm
im}(i)/{\rm ker} (j\circ T|{\rm im}(i)).
\end{equation}
{\em The point in (\ref{eq:im})  is that ${\rm im}(M^q_{\Phi,ver}-I)$ appears as a
factor space of \, ${\rm im}(i)$.}

\begin{bekezdes}
In some cases ${\rm im}(M^q_{\Phi,ver}-I)$ can be determined
exactly.

Assume that $\gc$ has no vanishing 2--edges,  i.e. the
case \ref{lem:twist}(c) does not occur. Then all  diagonal
entries of var are non--zero, hence both  var and $T$ are
isomorphisms.
%In particular,
%\begin{equation}\label{eq:im2}
%{\rm im}(M^q_{\Phi,ver}-I)\simeq {\rm im}(i)/(T^{-1}({\rm
%im}(\partial))\cap {\rm im}(i)).
%\end{equation}
The next lemma is a direct consequence of  \cite[page 113]{EN}: \ix{Eisenbud--Neumann book}
\end{bekezdes}
\begin{lemma}\labelpar{lem:unitw} Assume that $\gc$ is unicolored \ix{graph!unicolored}
(cf. \ref{def:unicolor}). Then the restriction of $j\circ T$ on
${\rm im}(i)$ is injective.
%$$T^{-1}({\rm im}(\partial))\cap {\rm im}(i)=0.$$
In particular, if $G$ is unicolored,  then  the number of   all
2--Jordan blocks of $M_{\Phi,ver}$  is $\rank \im (i)$.
\end{lemma}
\begin{proof}
For any $y,z\in H_1(\wfp,\wfW)$, we consider the intersection
number $(y,Tz)$, denoted by $\langle y,z\rangle$. Since $T$ is
diagonal, and all  entries on the diagonal have the same sign,
the form $\langle y,z\rangle$ is definite. Assume that
$j\,Ti(x)=0$. Then $$0=(x,j\,Ti(x))_{H_1(\wfp^*)}=( i(x),Ti(x))=\langle
i(x),i(x)\rangle,$$ hence $i(x)=0$.
\end{proof}

\section{\ Characters}\setcounter{equation}{0} The
corresponding $\Z^2$--characters of im$(i)$ can be determined by the  following
exact sequence (as part of diagram \ref{diagram}):
$$0\longrightarrow \ {\rm im}(i)\ \longrightarrow \
 H_1(\wfp^*,\wfW)\ \stackrel{\tilde{\partial}}{\longrightarrow} \
H_0(\wfW)\ \longrightarrow \ H_0(\wfp^*)\ \longrightarrow \ 0,
$$
and the action of the vertical/horizontal monodromies on this
sequence.
%The interested reader is invited to analyze all the
%possible eigenvalues .
In this proof we  only need those
blocks of $M_{\Phi,ver}$ which have eigenvalue 1.
\ix{monodromy!characters}

The action of the algebraic {\em vertical monodromy} on each term of
this sequence is finite: it is induced by a permutation of
the connected components of the spaces $\wfW$ and $(B,\partial B)$. The
corresponding 1--eigenspaces form the following exact sequence:
\begin{equation}\label{eq:exact1}
0\longrightarrow \ {\rm im}(i)_{ver,1}\ \longrightarrow \
 H_1(\wfp^*,\wfW)_{ver,1}\ \stackrel{\tilde{\partial}}{\longrightarrow} \
H_0(\wfW)_{ver,1}\ \longrightarrow \ H_0(\wfp^*)_{ver,1}\ \longrightarrow \ 0.
\end{equation}
This sequence will be compared with another sequence  which computes the
simplicial homology of the connected 1--complex $|G|$. Namely, one
considers the free vector spaces $\bfc^{|\cale_\calw(G)|}$ and
$\bfc^{|\calw(G)|}$,  generated by the edges $\cale_\calw(G)$ and
vertices $\calw(G)$ of $G$, as well as   the boundary
operator $\partial'$. Then one has the exact sequence
\begin{equation}\label{eq:exact2}
0\ \longrightarrow \ H_1(|G|)\ \longrightarrow \
\bfc^{|\cale_\calw(G)|} \ \stackrel{\partial'}{\longrightarrow }\
\bfc^{|\calw(G)|}\ \longrightarrow \ H_0(|G|) \ \longrightarrow \
0.
\end{equation}
\begin{lemma}\labelpar{prop:same}
The two exact sequences (\ref{eq:exact1}) and (\ref{eq:exact2}) are
isomorphic. Moreover, the horizontal monodromy acting on (\ref{eq:exact1})
can be identified with the action on  (\ref{eq:exact2}) induced
by the cyclic action of the covering $G\to\gc$.  (The cyclic action
 is induced by the positive generator of $\Z$, cf. \ref{def:2.3.2}.)
\end{lemma}\ix{graph!covering}
\begin{proof} It suffices to identify the second and the third terms
together with the connecting morphisms, and their compatibility with the monodromy
action. (In fact, the isomorphism  of the last  terms is trivial:
 $ H_0(\wfp^*)_{ver,1}=H_0(\wfp)_{ver,1}=H_0(|G|)=\bfc$ is clear since $\wfp$
and $|G|$ are connected.)

The identification follows from the proof of
the Main Algorithm.
Recall, that in the previous sections, we constructed a decomposition
of $\wfp^*$, the lifted fiber $(\Phi\circ r)^{-1}(c_0,d_0)$. Let us repeat
the very same construction for
$(\Phi\circ r)^{-1}(\partial D_{c_0})$, where $D_{c_0}$ is the disc
$\{(c,d):c=c_0\}$ as in Chapter \ref{s:ICIS}, or in the proof of \ref{lem:sksmooth}.
That is, we decompose the total space of the fibration over $\partial D_{c_0}=S^1$
instead of only one fiber.
Let $Tot(\wfW)$ be the space  we get instead  of $\wfW$, which, in fact,
is the total space of a fibration over $S^1$ with fiber $\wfW$ and monodromy the
vertical monodromy. Since the geometric vertical monodromy permutes the components of $\wfW$,
the algebraic vertical monodromy $M_{H_0(\wfW),ver}$ acts finitely on $H_0(\wfW)$, and
the $coker(M_{H_0(\wfW),ver}-I)$ can be identified with $H_0(\wfW)_{ver,1}$.
Hence, by the Wang exact sequence of the above fibration, we get that
$H_0(\wfW)_{ver,1}=H_0(Tot(\wfW))$. On the other hand, in the proof of
\ref{lem:sksmooth}
(complemented also with the second part of \ref{CC}),
 $\Phi^{-1}(\partial D_{c_0})\cap B_\epsilon$
appears as a part of $\partial \cals_k$. Therefore,
$Tot(\wfW)$ appears as a part of $\partial \wsk$: that part which is situated in the neighbourhood of the
regular part of the compact components of $\C$. Hence, by \ref{NORM1} and \ref{bek:pl1},
$H_0(Tot(\wfW))$ is freely generated by the collection of curves situating
in the normalization  ${\cals}_k^{norm}$ of $\wsk$ located above the compact components of
$\C$. This is codified exactly by the non--arrowhead vertices of  $G$.

For the second terms, first notice that by the isomorphism
exc $:H_1(\wfp^*,\wfW)\to H_1(B,\partial B)$, the  rank of $H_1(\wfp^*,\wfW)$
counts the separating annuli
of $\wfp^*$. Now, we can repeat  the entire construction and argument  above
for $\cale_\calw$ instead of $\calw$. Indeed, $H_1(\wfp^*,\wfW)_{ver,1}$
counts the separating tori (not considering those corresponding to
the binding of the open book) of $(\Phi\circ r)^{-1}(\partial D_{c_0})$,
and this, by the proof of the Main Algorithm, is codified exactly by $\cale_\calw$.
The details are left to the reader.

Since all the maps and identifications are natural and compatible with the
action of the corresponding horizontal monodromies, the morphisms
$\tilde{\partial}$ and $\partial'$  and the actions of the horizontal monodromies are all
identified.
\end{proof}

Now, we are ready to finish the proof of \ref{th:Jordan}. By the
above discussion, $\#^2_1M_{\Phi,ver}$ is equal to the dimension
of $ I:={\rm im}(i)_{ver,1}/{\rm ker} (j\,T|{\rm im}(i))_{ver,1}$, which is
smaller than $\dim (i)_1=\dim H_1(|G|)=c(G)$. Moreover, $P^\#$ is
the characteristic polynomial of the horizontal monodromy acting
on $I$, which clearly divides the characteristic polynomial of the
horizontal monodromy acting on $H_1(|G|)$, which is
$P_{\mathfrak{h}(|G|)}$. If $c(G)=0$, or if
$c(G)=\#^2_1M_{\Phi,ver}$ (i.e., if $\corank(A,\inc)=|\cala|$, cf.
\ref{cor:Jordan}), or if $\gc$ is unicolored (cf.
\ref{lem:unitw}), then $P^\#=P_{\mathfrak{h}(|G|)}$. \ix{graph!unicolored}

\begin{bekezdes} The local case, valid for any $j$, follows by
similar arguments if one replaces $\gc$ by $\Gamma^2_{\C,j}$.  The
last statement  follows from  Lemma \ref{eq:CP}, or from the sentences following it.
\end{bekezdes}

\bekezdes {\bf The proof of Theorem \ref{th:Jordanim}.}
Assume that $\gc$ has a  vanishing 2--edge $e\in\cale_\calw$, that is,   the situation of
 \ref{lem:twist}(c) occurs.
  Let $h:=\tilde{m}_{\phi,ver}$ be as in \ref{qp},
 and fix $q$ such that the restriction of
 $h^q$ on $\wfW$ is the identity. The proof of
 \ref{lem:twist}(c) shows that $h^q$ can be extended  to $\wfp^*\cap T(e)$ by the
 identity map. In particular,
 in such a situation it is better to replace the space
 $\wfW$ from the diagram \ref{diagram} by
 $\wfW'$, defined as the union of $\wfW$ with {\em all} separating
 annuli $\wfp^*\cap T(e)$ corresponding to the vanishing 2--edges from
 $\cale_\calw$.
 Moreover, we define  $B'$ as the union of the other separating annuli.
 Then one gets a new diagram (involving the spaces $\wfp^*$, $\wfW'$, $B'$,
 and morphisms $i'$ and $T'$)
 such that in the new `collapsing' situation
 $T'$ becomes {\em  an isomorphism}. Then all the arguments above,
 complemented with the corresponding facts from Chapter \ref{ss:ELI}
about the `Collapsing Main Algorithm', can be repeated, and the second version
\ref{th:Jordanim} follows as well.

\begin{example}\labelpar{ex:cylver} Assume that $f(x,y,z)=f'(x,y)$ and $g=z$ as in
 \ref{cyl}.  Then we have only vanishing 2--edges,
hence  $\wfW'=\wfp$. Therefore, in the new diagram $H_1(\wfp^*,\wfW')=0$.
In fact, since $\nu=1$ for all vertices, we can even take $q=1$
(use e.g. \ref{keyex}), hence in this case $m_{\phi,ver}$ is
isotopic to the identity. (This can also be proved by a direct
argument, see Chapter \ref{s:cyl}).
\end{example}

\chapter{The mixed Hodge structure of $H_1(\partial F)$}\labelpar{s:MHS}

\section{\ Generalities. Conjectures}\label{ss:MHS}\setcounter{equation}{0}

\bekezdes\labelpar{bek:MHS}
 We believe that a substantial part of the numerical identities and
inequalities obtained in the previous chapters are closely related with general properties  of mixed
Hodge structures (in the sequel abbreviated by MHS)
supported by different (co)homology groups involved in the constructions.
\ix{mixed Hodge structure|textbf}

Although the detailed study  of the mixed Hodge structure on
$H_1(\partial F)$ and related  properties exceeds the aims of the
present work, we decided to dedicate a few paragraphs to this subject too: we wish to
formulate some of the expectations and to shed light on the  results of
the book from this point of view
as well. %Moreover, we wish to stimulate further research in this direction.

For general results, terminology and properties of MHS, see for example the articles
of Deligne \cite{De}.
For the MHS on the cohomology of the Milnor fiber of local singularities see
the articles of Steenbrink \cite{StLimit,StOslo,StHodge},
or consult the monographs of Dimca, Kulikov,  or Peters and Steenbrink   \cite{DimcaBook,Kulikov,PSt}.
On the link of an isolated singularity Durfee defined a MHS \cite{Du2},
for different versions  and generalizations see
 \cite{DuHa,DuSa,Elzein,StOslo,StHodge}.\ix{Durfee}
\ix{Deligne} \ix{Steenbrink} \ix{Kulikov} \ix{Dimca} \ix{Peters}

\vspace{2mm}

For the convenience of the reader we recall the basic definition of MHS.

\begin{definition} \
(a) A pure Hodge structure of weight $m$ is  a pair $(H,F^\bullet)$, where $H$ is a
finite dimensional $\R$--vector space and $F^\bullet$ is a decreasing finite filtration
on $H_\bfc=H\otimes \bfc$ (called the Hodge filtration) such that
$$H_\bfc=F^p\oplus \overline{F^{m-p+1}H_\bfc} $$ for all $p\in\Z$, where the conjugation
\ $\overline{\cdot}$ \ on $H_\bfc$ is induced by the conjugation of ~$\bfc$.

\vspace{2mm}

(b) A mixed Hodge structure is a triple $(H,W_\bullet,F^\bullet)$, where
$H$ is a finite dimensional $\R$--vector space, $W_\bullet$ is a finite increasing filtration
on $H$ (called the weight filtration), and $F^\bullet$ is a finite decreasing filtration
on $H_\bfc$ such that $F^\bullet$ on  $Gr^W_mH$ induces  a Hodge structure of weight $m$ for all $m$.
(Here  $Gr^W_mH:= W_mH/W_{m-1}H$.)
\end{definition}
\ix{mixed Hodge structure!Hodge filtration|textbf}\ix{mixed Hodge structure!weight filtration|textbf}

\bekezdes {\bf  MHS  on the cohomology of $\partial F$.} \ Without entering
in the theory of derived categories and mixed Hodge modules, we outline  a possible way
 to  define a MHS
on the cohomology of $\partial F$.  \ix{Milnor!fiber!boundary}

\vspace{2mm}

If $f:(\bfc^3,0)\to (\bfc,0)$ is a hypersurface singularity, the cohomology of its
Milnor fiber $F$ carries a MHS.
This can be defined via Deligne's nearby cycle functor $\psi_f$ which produces the
mixed Hodge module $\psi_f\R_{\bfc^3}$ supported on $V_f$. If $i$ denotes the inclusion of the
origin into $V_f$, then $H^k(i^*\psi_f\R_{\bfc^3})=H^k(F)$, defining a MHS on $H^k(F)$.

Usually, if $V_f$ has  an isolated singularity, then the MHS of $H^*(\partial F)$ is defined via the
isomorphism of $H^*(\partial F)$ with the  cohomology $H^*(K_f)$
of the link of $V_f$, which is identified with a local cohomology $H^{*+1}_{\{0\}}(V_f)$.
The MHS on the link of  a normal surface singularity can be  defined in a similar way through local
cohomology (see \cite{Du2,DuHa,DuSa,Elzein,StOslo,StHodge}).

If $V_f$ has a non--isolated singularity, this procedure  does not work: the link and $\partial F$ have  different
cohomologies. Therefore, for  such a case we propose the following definition.

Consider  $i^!\psi_f\R_{\bfc^3}$ too; this has  the property that $H^k(i^!\psi_f\R_{\bfc^3})=H^k_c(F)$.
\begin{definition}
Define the MHS on the cohomology of the boundary $\partial F$ via
 \begin{equation}\label{be:MHS}
 {\rm cone}( i^!\psi_f\R_{\bfc^3}\to
i^*\psi_f\R_{\bfc^3}).\end{equation}
\end{definition}
\ix{mixed Hodge structure!nearby cycles}
This definition automatically implies the following fact.
\begin{corollary} One has an  exact sequence of mixed Hodge structures:
\begin{equation}\label{eq:MHSmhs}
0\to H^1(F)\to H^1(\partial F)\to H^2_c(F)\to H^2(F)\to
H^2(\partial F)\to H^3_c(F)\to 0.\end{equation}
\end{corollary}

\begin{remark} The above definition agrees with the `link functor' ${\rm cone}(i^!\to i^*)$ of
Durfee and Saito \cite{DuSa}, which can be considered
in any category where the basic functors $i^!$, $i^*$ and ${\rm cone}$  are defined.
The link functor can be applied to any mixed Hodge module.
If one applies it to the constant sheaf on $V_f$, one  gets a MHS supported by
the cohomology of the link $L$ (this case was discussed in \cite{DuSa}).
When one  applies it to $\psi_f\R_{\bfc^3}$,  one gets the MHS of the
cohomology of $\partial F$. \ix{Durfee}\ix{Saito}
\end{remark}

The following cases are relevant from the point of view of the results of the present work.

\begin{example}\labelpar{re:nss} \

(a) \ Assume that $(X,x)$ is a normal surface singularity. Let $G_X$
be (one of its) dual resolution graphs. It is known that its
intersection matrix $A$ is negative definite. \ix{graph!negative definite}\ix{matrix!intersection}
We wish to give geometric meaning to the integers $g(G_X)$ and $c(G_X)$.
Recall that in this case we only use
the blowing up/down operations of $(-1)$ rational vertices, which keep
the integers $g(G_X)$ and $c(G_X)$ stable. Therefore, they can be recovered from any
negative definite plumbing graph, and thus these numbers  depend only on the link $K_X$ of $(X,x)$.
Recall also that  $\dim H_1(K_X)=2g(G_X)+c(G_X)$, cf. \ref{le:h-1}.

The integers $g(G_X)$ and $c(G_X)$
have  the following Hodge theoretical interpretation. If
$W_\bullet$ denotes the weight filtration of $H_1(\partial X,\R)$, then
$\dim Gr^W_{-1}H_1(K_X)=2g(G_X)$ and $\dim Gr^W_{0}H_1(K_X)=c(G_X)$
(and all the other graded components are zero). Hence,
in this situation, $Gr^W_\bullet H_1(\partial X)$ is topological.
\ix{mixed Hodge structure!of link of surfaces}\ix{mixed Hodge structure!of boundary of tubes}
\ix{mixed Hodge structure!of $\partial F$}
\vspace{2mm}

(b) \  Let us analyze  how the facts from (a) are modified if we consider a more
general situation. Assume that $V$ is a smooth complex surface
and $C\subset V$ a normal crossing curve with all irreducible
components compact. We will denote by the same $V$ a small
tubular neighbourhood of $C$. Then the oriented  3--manifold
$\partial V$ can be represented by a plumbing graph --- the dual
graph of the configuration $C$.  Let this be  denoted  by  $G_C$.
Similarly as above, $A$ denotes the associated intersection matrix
of $G_C$, or, equivalently, the intersection matrix of the
curve--configuration of the irreducible components of $C$. Again (as in \ref{le:h-1}), $\dim
H_1(\partial V)=2g(G_C)+c(G_C)+\corank A$. Moreover, $H_1(\partial
V,\R)$ admits a mixed Hodge structure such that
$\dim Gr^W_{-2}H_1(\partial V)=\corank A$,
 $\dim Gr^W_{-1}H_1(\partial V)=2g(G_C)$
and $\dim Gr^W_{0}H_1(\partial V)=c(G_C)$ (and all the other
graded components are zero), see e.g.  \cite[(6.9)]{ElNe}.

The point is that now this decomposition  depends essentially {\it on
$C$ and  the analytic embedding of $C$ into $V$}, and, in general,
cannot be deduced  from the {\it topology} of $\partial V$ {\it alone}.
(To see this, compare the following two cases:
  the union   of three generic lines in the projective plane
and  a smooth elliptic curve  with self--intersection zero).
This discussion shows that  if we wish  to keep the information
regarding the weight filtration of $\partial V$, then we are only allowed to use
those calculus--operations which preserve $c(G_C)$, $g(G_C)$ and  $\corank A$.
In particular, the oriented handle absorption should not be  allowed.
\end{example}

\vspace{2mm}

We believe that the following properties hold for the MHS defined above:

\bekezdes\labelpar{be:mhs}{\bf Conjecture.} The weight filtration   of the mixed
Hodge structure of $H_1(\partial F)$, defined in (\ref{be:MHS}),  satisfies
$$ 0=W_{-3}\subset W_{-2}\subset W_{-1}\subset W_0=H_1(\partial F).$$
Moreover,   for any graph $G$ provided by the Main Algorithm one has:
\begin{equation}\label{eq:mhs}
\dim \, Gr^W_iH_1(\partial F)=\left\{
\begin{array}{ll}
\corank A_G \ \ & \mbox{if $i=-2$}, \\
2g(G) & \mbox{if $i=-1$}, \\
c(G) & \mbox{if $i=0$}. \end{array}\right.
\end{equation}\ix{mixed Hodge structure!weight filtration}

Obviously, there is a corresponding dual cohomological statement:
\bekezdes\labelpar{be:mhs2}{\bf Conjecture.}  The weight filtration   of the mixed
Hodge structure of $H^1(\partial F)$ satisfies
$$ 0=W_{-1}\subset W_{0}\subset W_{1}\subset W_2=H^1(\partial F),$$
and  for any graph $G$ provided by the Main Algorithm one has:
\begin{equation}\label{eq:mhs2}
\dim \, Gr^W_iH^1(\partial F)=\left\{
\begin{array}{ll}
\corank A_G \ \ & \mbox{if $i=2$}, \\
2g(G) & \mbox{if $i=1$}, \\
c(G) & \mbox{if $i=0$}. \end{array}\right.
\end{equation}
Similarly, the weight filtration of $H^2(\partial F)$ satisfies
$$ 0=W_{1}\subset W_{2}\subset W_{3}\subset W_4=H^2(\partial F),$$
and  for any graph $G$ provided by the Main Algorithm one has:
\begin{equation}\label{eq:mhs3}
\dim \, Gr^W_iH^2(\partial F)=\left\{
\begin{array}{ll}
\corank A_G \ \ & \mbox{if $i=2$}, \\
2g(G) & \mbox{if $i=3$}, \\
c(G) & \mbox{if $i=4$}. \end{array}\right.
\end{equation}

\begin{remark}\label{re:mhs} \
(1) %Geometric maps induce MHS--morphisms. In particular,
The monomorphism $H^1(F)\hookrightarrow H^1(\partial F)$ from (\ref{eq:MHSmhs})
 is {\em strictly compatible} with both
weight and Hodge filtrations; namely,  $$W_i(H^1(F))=H^1(F)\cap W_i(H^1(\partial F)) \ \ \ \mbox{  for any $i$.}$$
There is a similar statement for the Hodge filtration too.

Moreover, we claim that the cup product $H^1(\partial F)\otimes  H^1(\partial F)\mapsto H^2(\partial F$)
is a morphism of mixed Hodge structures, in particular, it preserves the weight filtration
\begin{equation}\label{eq:mhswe}
W_i(H^1(\partial F))\cup W_j(H^1(\partial F))\subset W_{i+j}(H^2(\partial F)).\end{equation}
E.g., if $A_G$ is non--degenerate, then (\ref{eq:mhs2}) and (\ref{eq:mhs3}) would imply that
the cup--product  $H^1(\partial F)\otimes H^1(\partial F)\to H^2(\partial F)$ was trivial.
This is compatible with the fact that  the cup--product is indeed trivial,  whenever $A_G$ is
  non--degenerate: this is a  result of Sullivan \cite{SU},
see also \cite{DuHa}.

Moreover,
(\ref{eq:mhswe}) can be compared with the list of inclusions of \ref{cor:EIG}(b)
where the fact that   $(A_G)_{\lambda\not=1}$ is non--degenerate was also used
(see  also \ref{ss:TRP}).

\vspace{2mm}

(2) If true, the above Conjectures would have the following consequence:
the numerical invariants $\corank A_G$, $g(G)$ and $c(G)$
associated with the graph $G$ obtained from the Main Algorithm are independent of the choice of the
resolution $r$ (or, of the graph $\gc$ used in the algorithm). Therefore, we emphasize again,
if one wishes to keep in $\Gmod$ all the information regarding the MHS of $\partial F$, one has
to use only those operations of the reduced plumbing calculus which preserve these numbers, that is,
one has to exclude the oriented handle absorption R5.

\vspace{2mm}

(3) Conjecture \ref{be:mhs2} is compatible with
\ref{re:nss}(a): the weights of $H^1(K_X)$ are $<2$, while the weighs of
 $H^2(K_X)$ are $>2$.

\vspace{2mm}

(4) In \ref{s:MHS2} (treating homogeneous singularities) and in section \ref{cyl:MHS}  (cylinders)
we provide  evidences for the above Conjecture.
\end{remark}\ix{mixed Hodge structure!weight filtration}

\part{Examples}

\chapter{Homogeneous singularities}\labelpar{s:HOMOGEN}

Assume that $f$ is homogeneous of degree $d\geq 2$. In order to determine a possible
$\gc$, we  take for $g$ a generic linear function.  We  adopt the notations of
Chapter \ref{hom}, where the graph $\gc$ is constructed.

We start our discussion with the list of some specific properties of the geometry of
homogeneous singularities regarding the graphs $\gc$, $G$, or the
ICIS $\Phi$. Then  we  combine this new information with the general
properties established in the previous chapters.
\ix{singularities!homogeneous}
\ix{monodromy}
\ix{monodromy!Jordan block}

\section{\ The first specific feature: $M_{ver}=(M_{hor})^{-d}$}
\labelpar{feature1}\setcounter{equation}{0}
Clearly
$$|\cala(\gc)|=|\cala(G)|=d.$$
Moreover, $d_j=1$ for any $j$, hence
\begin{equation}\label{egyenlo}
M'_{j,hor}=M^\Phi_{j,hor} \ \ \mbox{and} \ \ M'_{j,ver}=M^\Phi_{j,ver}.
\end{equation}
 Since $f$ and $g$ are homogeneous,
by \cite{Steenbrink}, we have
\begin{equation}\label{eq:STEE} \left\{\begin{array}{l}
M'_{j,ver}=(M'_{j,hor})^{-d}\\
M_{\Phi,ver}=(M_{\Phi,hor})^{-d}.\end{array}\right.\end{equation}
In particular, for any of the  pairs  $(M'_{j,ver},M'_{j,hor})$
or $(M_{\Phi,ver},M_{\Phi,hor})$, the  number of 2--Jordan blocks of
the vertical operator coincides with the number of 2--Jordan
blocks of the horizontal operator. Moreover, the horizontal monodromies determine the
corresponding $\bfz^2$--representations completely.

\vspace{2mm}

In fact, the identities (\ref{eq:STEE}) are true at the level of the geometric monodromies as well.
Let us check this for the pair $(m_{\Phi,ver},m_{\Phi,hor})$; the local version is similar.

Consider the homogeneous action on $\bfc^3$ given by
$\lambda*(x,y,z)=(\lambda x,\lambda y,\lambda z)$. This is projected via $\Phi$ to the action
$\lambda*(c,d)=(\lambda^dc,\lambda d)$ of $\bfc^2$. Hence, the monodromy along
the loop $(e^{itd}c,e^{it}d)$, $t\in[0,2\pi]$, is lifted to a trivial action on the
Milnor fiber, that is
\begin{equation}\label{eq:STEEgeo}
m_{\Phi,hor}^d\cdot m_{\Phi,ver}=I \ \ \mbox{and} \ \
(m'_{j,hor})^d\cdot m'_{j,ver}=I.\end{equation}

\section{\ The second specific feature: the graphs $\overline{G_{2,j}}$}
\labelpar{feature2}\setcounter{equation}{0}
Consider the graph $\gc$ of Chapter  \ref{hom}.
In the Main Algorithm \ref{algo} applied for this $\gc$,
one  puts  exactly one vertex in $G$, say
$\widetilde{v}_\lambda$, above a vertex $v_\lambda$
of $\gc$.  If we delete all these vertices
from $G$, we get the union of the graphs $\overline{G_{2,j}}$, cf. \ref{2} and \ref{rem:closure}.
The point is that
\ix{singularities!suspensions}

\vspace{2mm}

 \begin{verse}{\it %if we delete all the arrows and multiplicities of \,
 $\overline{G_{2,j}}$,  with opposite orientation, is a possible embedded resolution graph \\
 associated with the $d$--suspension of the transversal singularity $T\Sigma_j$ \\
 and the germ induced by the projection on the suspension coordinate}.\end{verse}

\vspace{2mm}

\noindent  More precisely, if $f'_j(u,v)=f|_{
(Sl_q,q)}$ ($q\in \Sigma_j\setminus \{0\}$, cf. \ref{ss:2.0b})  is
the local equation of the transversal type plane curve
singularity, then its $d$--suspension is the isolated hypersurface
singularity $X:=(\{w^d=f'_j(u,v)\},0)$. Then $G_{2,j}=-\G(X,w)$.

This follows from a comparison of the  Main Algorithm
\ref{algo} and the algorithm which provides the graph of
suspension singularities \ref{ss:b}. Independently,
it can also be proved by combining  Proposition \ref{felbo}(2) and the local identity of
(\ref{eq:STEEgeo}), or by the covering procedure which will be described in
\ref{feature3}.

In particular, the graph $G$ consists of the disjoint union of
graphs of these suspensions with opposite orientation, and
$|\Lambda|$ new vertices $\widetilde{v}_\lambda$ decorated by
$[g_{\lambda}]$, altogether supporting $d$ arrowheads. The edges connecting these new vertices
to the graphs $\overline{G_{2,j}}$ are $\circleddash$--edges and reflect the
combinatorics of the `identification map' $c$ (i.e. the incidence of the
singular points on the components of $C$, see \ref{homgen}). The arrowheads are distributed as
follows: each $\widetilde{v}_\lambda$ supports $d_{\lambda}$ arrowheads, all of
them connected by  +--edges. The
gluing property of the two pieces (i.e. of $\partial_1F$ and $\partial_2F$)
is codified in the Euler numbers of the
vertices $\widetilde{v}_\lambda$ in $G$.
This is determined by (\ref{eq:2.2.1}), once we carry the multiplicity
system of $g$ in the construction:  in the suspension graphs one has the multiplicity system of
the germ $w:(\{w^d=f'_j(u,v)\},0)\to (\bfc,0)$, and  each $\widetilde{v}_{\lambda}$ and arrowhead has
multiplicity $(1)$. This provides a very convenient `short--cut' to obtain  $G$.
\ix{Main Algorithm!short--cut for homogeneous germs}

\begin{example}\labelpar{ex:C4}
A projective curve of the projective plane is called cuspidal if all  its
singularities are locally analytically irreducible.  \ix{projective curve!cuspidal|textbf}

It is well--known that there exists a rational cuspidal projective
curve $C$ of degree $d=5$ with two local irreducible singularities
with local equations $u^3+v^4=0$ and $u^2+v^7=0$ (see e.g.  \cite{namba}).
Then the above procedure provides  two suspension
singularities, both of Brieskorn type. Their equations are
$u^3+v^4+w^5=0$ and $u^2+v^7+w^5=0$. (We made the choice of $C$
carefully: we wished to get pairwise relative prime exponents in
each singularity, in order to `minimize' $|H_1(\partial F,\Z)|$,
 cf. \ref{ss:ratS}.) The embedded resolution graphs
of the coordinate function $w$ restricted on these suspension singularities
 are the following:

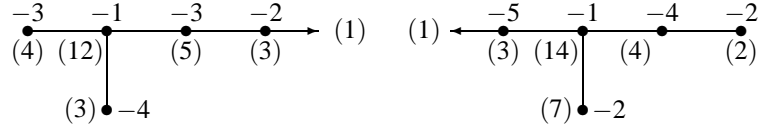
\vspace{1mm}

\begin{picture}(240,65)(-30,-10)
\put(0,30){\circle*{4}} \put(30,30){\circle*{4}}
\put(60,30){\circle*{4}} \put(90,30){\circle*{4}}
\put(30,0){\circle*{4}}

\put(180,30){\circle*{4}} \put(210,30){\circle*{4}}
\put(240,30){\circle*{4}} \put(270,30){\circle*{4}}
\put(210,0){\circle*{4}}

 \put(0,30){\vector(1,0){110}} \put(270,30){\vector(-1,0){110}}
 \put(30,30){\line(0,-1){30}} \put(210,30){\line(0,-1){30}}

\put(0,37){\makebox(0,0){$-3$}} \put(30,37){\makebox(0,0){$-1$}}
\put(60,37){\makebox(0,0){$-3$}} \put(90,37){\makebox(0,0){$-2$}}
\put(40,0){\makebox(0,0){$-4$}} \put(180,37){\makebox(0,0){$-5$}}
\put(210,37){\makebox(0,0){$-1$}}
\put(240,37){\makebox(0,0){$-4$}}
\put(270,37){\makebox(0,0){$-2$}} \put(220,0){\makebox(0,0){$-2$}}

\put(0,22){\makebox(0,0){$(4)$}}
\put(20,22){\makebox(0,0){$(12)$}}
\put(60,22){\makebox(0,0){$(5)$}}
\put(90,22){\makebox(0,0){$(3)$}} \put(20,0){\makebox(0,0){$(3)$}}
\put(180,22){\makebox(0,0){$(3)$}}
\put(200,22){\makebox(0,0){$(14)$}}
\put(230,22){\makebox(0,0){$(4)$}}
\put(270,22){\makebox(0,0){$(2)$}}
\put(200,0){\makebox(0,0){$(7)$}}

\put(122,30){\makebox(0,0){$(1)$}}
\put(150,30){\makebox(0,0){$(1)$}}
\end{picture}

\vspace{1mm}

Reversing the orientations we get
the following graphs, which coincide with the graphs $\overline{G_{2,j}}$
($j=1,2$):

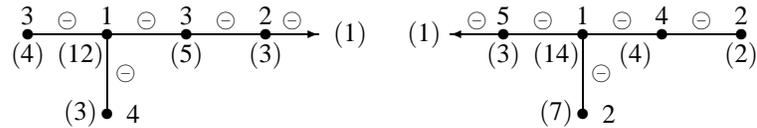
\vspace{1mm}

\begin{picture}(240,65)(-30,-10)
\put(0,30){\circle*{4}} \put(30,30){\circle*{4}}
\put(60,30){\circle*{4}} \put(90,30){\circle*{4}}
\put(30,0){\circle*{4}}

\put(180,30){\circle*{4}} \put(210,30){\circle*{4}}
\put(240,30){\circle*{4}} \put(270,30){\circle*{4}}
\put(210,0){\circle*{4}}

 \put(0,30){\vector(1,0){110}} \put(270,30){\vector(-1,0){110}}
 \put(30,30){\line(0,-1){30}} \put(210,30){\line(0,-1){30}}

\put(0,37){\makebox(0,0){$3$}} \put(30,37){\makebox(0,0){$1$}}
\put(60,37){\makebox(0,0){$3$}} \put(90,37){\makebox(0,0){$2$}}
\put(40,0){\makebox(0,0){$4$}} \put(180,37){\makebox(0,0){$5$}}
\put(210,37){\makebox(0,0){$1$}} \put(240,37){\makebox(0,0){$4$}}
\put(270,37){\makebox(0,0){$2$}} \put(220,0){\makebox(0,0){$2$}}

\put(0,22){\makebox(0,0){$(4)$}}
\put(20,22){\makebox(0,0){$(12)$}}
\put(60,22){\makebox(0,0){$(5)$}}
\put(90,22){\makebox(0,0){$(3)$}} \put(20,0){\makebox(0,0){$(3)$}}
\put(180,22){\makebox(0,0){$(3)$}}
\put(200,22){\makebox(0,0){$(14)$}}
\put(230,22){\makebox(0,0){$(4)$}}
\put(270,22){\makebox(0,0){$(2)$}}
\put(200,0){\makebox(0,0){$(7)$}}

\put(122,30){\makebox(0,0){$(1)$}}
\put(150,30){\makebox(0,0){$(1)$}}

\put(15,35){\makebox(0,0){$\circleddash$}}
\put(45,35){\makebox(0,0){$\circleddash$}}
\put(75,35){\makebox(0,0){$\circleddash$}}
\put(100,35){\makebox(0,0){$\circleddash$}}
\put(37,15){\makebox(0,0){$\circleddash$}}
\put(170,35){\makebox(0,0){$\circleddash$}}
\put(195,35){\makebox(0,0){$\circleddash$}}
\put(225,35){\makebox(0,0){$\circleddash$}}
\put(255,35){\makebox(0,0){$\circleddash$}}
\put(217,15){\makebox(0,0){$\circleddash$}}
\end{picture}

\vspace{1mm}

In order to get $G$,
 we have to insert one more vertex $\widetilde{v}$
(corresponding to $C$)  with multiplicity $(1)$ and genus decoration zero, and $d=5$
arrowheads, all with multiplicity (1), connected with +--edges to
$\widetilde{v}$:

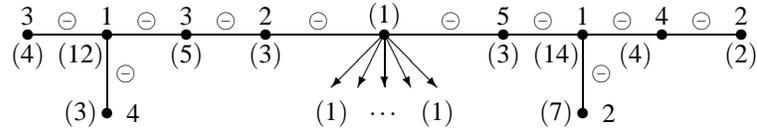
\vspace{1mm}

\begin{picture}(240,65)(-30,-10)
\put(0,30){\circle*{4}} \put(30,30){\circle*{4}}
\put(60,30){\circle*{4}} \put(90,30){\circle*{4}}
\put(30,0){\circle*{4}}

\put(180,30){\circle*{4}} \put(210,30){\circle*{4}}
\put(240,30){\circle*{4}} \put(270,30){\circle*{4}}
\put(210,0){\circle*{4}}\put(135,30){\circle*{4}}

\put(135,30){\vector(-1,-1){20}} \put(135,30){\vector(1,-1){20}}
\put(135,30){\vector(-1,-2){10}}
\put(135,30){\vector(1,-2){10}}\put(135,30){\vector(0,-1){20}}

 \put(30,30){\line(0,-1){30}} \put(210,30){\line(0,-1){30}}
\put(0,30){\line(1,0){270}}

\put(0,37){\makebox(0,0){$3$}} \put(30,37){\makebox(0,0){$1$}}
\put(60,37){\makebox(0,0){$3$}} \put(90,37){\makebox(0,0){$2$}}
\put(40,0){\makebox(0,0){$4$}} \put(180,37){\makebox(0,0){$5$}}
\put(210,37){\makebox(0,0){$1$}} \put(240,37){\makebox(0,0){$4$}}
\put(270,37){\makebox(0,0){$2$}} \put(220,0){\makebox(0,0){$2$}}

\put(0,22){\makebox(0,0){$(4)$}}
\put(20,22){\makebox(0,0){$(12)$}}
\put(60,22){\makebox(0,0){$(5)$}}
\put(90,22){\makebox(0,0){$(3)$}} \put(20,0){\makebox(0,0){$(3)$}}
\put(180,22){\makebox(0,0){$(3)$}}
\put(200,22){\makebox(0,0){$(14)$}}
\put(230,22){\makebox(0,0){$(4)$}}
\put(270,22){\makebox(0,0){$(2)$}}
\put(200,0){\makebox(0,0){$(7)$}}

\put(115,0){\makebox(0,0){$(1)$}}
\put(155,0){\makebox(0,0){$(1)$}}
\put(135,37){\makebox(0,0){$(1)$}}

\put(135,0){\makebox(0,0){$\dots$}}

\put(15,35){\makebox(0,0){$\circleddash$}}
\put(45,35){\makebox(0,0){$\circleddash$}}
\put(75,35){\makebox(0,0){$\circleddash$}}
\put(110,35){\makebox(0,0){$\circleddash$}}
\put(37,15){\makebox(0,0){$\circleddash$}}
\put(160,35){\makebox(0,0){$\circleddash$}}
\put(195,35){\makebox(0,0){$\circleddash$}}
\put(225,35){\makebox(0,0){$\circleddash$}}
\put(255,35){\makebox(0,0){$\circleddash$}}
\put(217,15){\makebox(0,0){$\circleddash$}}
\end{picture}

\vspace{1mm}

The missing Euler number $e$ of $\widetilde{v}$ is computed via
(\ref{eq:2.2.1}), namely $e+5-3-3=0$, hence $e=1$. Deleting
all the multiplicities we obtain  the graph of $\partial F$:

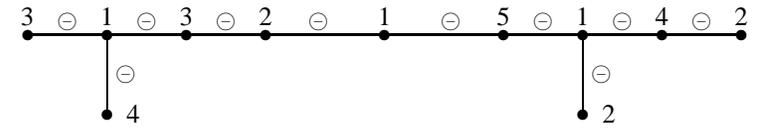
\vspace{1mm}

\begin{picture}(240,65)(-30,-10)
\put(0,30){\circle*{4}} \put(30,30){\circle*{4}}
\put(60,30){\circle*{4}} \put(90,30){\circle*{4}}
\put(30,0){\circle*{4}}

\put(180,30){\circle*{4}} \put(210,30){\circle*{4}}
\put(240,30){\circle*{4}} \put(270,30){\circle*{4}}
\put(210,0){\circle*{4}}\put(135,30){\circle*{4}}

 \put(30,30){\line(0,-1){30}} \put(210,30){\line(0,-1){30}}
\put(0,30){\line(1,0){270}}

\put(0,37){\makebox(0,0){$3$}} \put(30,37){\makebox(0,0){$1$}}
\put(60,37){\makebox(0,0){$3$}} \put(90,37){\makebox(0,0){$2$}}
\put(40,0){\makebox(0,0){$4$}} \put(180,37){\makebox(0,0){$5$}}
\put(210,37){\makebox(0,0){$1$}} \put(240,37){\makebox(0,0){$4$}}
\put(270,37){\makebox(0,0){$2$}} \put(220,0){\makebox(0,0){$2$}}

\put(135,37){\makebox(0,0){$1$}}

\put(15,35){\makebox(0,0){$\circleddash$}}
\put(45,35){\makebox(0,0){$\circleddash$}}
\put(75,35){\makebox(0,0){$\circleddash$}}
\put(110,35){\makebox(0,0){$\circleddash$}}
\put(37,15){\makebox(0,0){$\circleddash$}}
\put(160,35){\makebox(0,0){$\circleddash$}}
\put(195,35){\makebox(0,0){$\circleddash$}}
\put(225,35){\makebox(0,0){$\circleddash$}}
\put(255,35){\makebox(0,0){$\circleddash$}}
\put(217,15){\makebox(0,0){$\circleddash$}}
\end{picture}

\vspace{1mm}

By plumbing calculus, one can blow down twice 1--curves, and also
one may delete the $\circleddash$'s. This possible  graph of
$\partial F$ is not the `normal form' of \cite{neumann},
the interested reader may transform it easily to get a graph with
all Euler numbers negative. But even if we pass to the normal
form, the intersection matrix will not be negative definite, its
index will be $(-12,+1)$.

Notice that $\partial F$ is a rational homology sphere. In fact,
it is easy to verify that $H_1(\partial F,\Z)=\Z_5$.
\end{example}

As a consequence of the above discussion,  we get:

\begin{corollary}\labelpar{cor:Seif}
If $C$ is a rational unicuspidal projective plane curve (i.e. $C$
has only one singular point at which $C$ is locally analytically irreducible), and its
local singularity has only one Puiseux pair, then $\partial F$ can be represented by
a plumbing graph which is either star--shaped or a string.
\end{corollary}
\ix{projective curve!unicuspidal|textbf}

This will be exemplified  more in section \ref{ss:unicusp}.

Another consequence is the following
\begin{corollary}\labelpar{cor:corank}
$$\corank(A,\inc)_G=|\cala(G)|.$$
Therefore, Theorem \ref{th:Jordan} applies and $\#^2_1M_{\Phi,ver}=c(G)$
 and $P^\#(t)=P_{\mathfrak(h)(|G|)}(t)$ too.

In particular,
%the rank of $H_1(\partial F,\Z)$ and
the characteristic polynomial of the
monodromy acting on $H_1(\partial F,\Z)$ is determined by Theorem
\ref{th:charpolG}, and
$$rank\, H_1(\partial F)_{\not=1}=2g(G)+c(G)-2g(\gc)-c(\gc).$$
Moreover, the equivalent statements of (\ref{eq:echj}) are also satisfied, in particular
$\corank(A,\inc)_{G_{2,j}}=|\cale_{cut,j}|$ and
$\#^2_jM'_{j,ver}=c(G_{2,j})$ \ for any $j$.
\end{corollary}
\begin{proof}
Since each $\widetilde{v}_\lambda$ supports at least one
arrowhead, the rank of $(A,\inc)$ is maximal whenever all the ranks of the
intersection matrices associated with $\overline{G_{2,j}}$ are maximal. But
these matrices are definite, hence non--degenerate, cf. \ref{1.3b}(3).  For the other statements see
(\ref{eq:ech}) and (\ref{ranknemegy}).\ix{matrix!intersection}
\end{proof}
After  stating the third specific feature, and determining  $\corank A$,
we will return to
the characteristic polynomial formula.

\section{\ The third specific feature: the
$d$--covering}\labelpar{feature3}\setcounter{equation}{0}\ix{d@$d$--covering!of spaces|textbf}
 Let $C=\{f=0\}\subset \bp^2$ as before.
It is well--known that there is a regular $d$--covering
$F=\{f=1\}\to \bp^2\setminus C$ given by $(x,y,z)\mapsto [x:y:z]$.
Let $T$ be a small tubular neighbourhood of $C$, and let $\partial
T$ be its boundary. In the next discussion it is more convenient
to orient $\partial T$ not as the boundary of $T$, but as the
boundary of $\bp^2\setminus T^\circ$. Then the above covering induces a
regular $d$--covering, which is compatible with the orientations:
$$\partial F\to \partial T.$$
Moreover, the $\Z_d$ covering transformation on $\partial F$ coincides with
a representative of the Milnor monodromy action on $\partial F$.
Using either this, or just the  homogeneity of $f$, we get that
$$\mbox{{\it  the geometric monodromy action is finite of order $d$}.}$$
\begin{example}\labelpar{ex:CUSP}
Let us exemplify this covering procedure on a simple case. Assume that $C$ has degree
3 and has a cusp singularity. Then a possible plumbing representation for
$\partial T$ (oriented as the boundary of $T$) is the graph $(i)$ below.
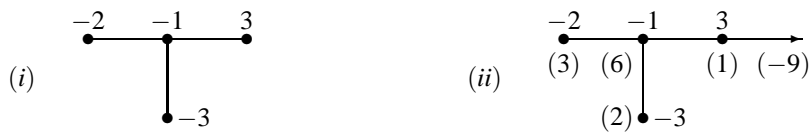
\vspace{1mm}

\begin{picture}(240,55)(-40,-5)
\put(0,30){\circle*{4}} \put(30,30){\circle*{4}}
\put(60,30){\circle*{4}}
\put(30,0){\circle*{4}}

\put(180,30){\circle*{4}} \put(210,30){\circle*{4}}
\put(240,30){\circle*{4}}
\put(210,0){\circle*{4}}

 \put(0,30){\line(1,0){60}} \put(180,30){\vector(1,0){90}}
 \put(30,30){\line(0,-1){30}} \put(210,30){\line(0,-1){30}}

\put(0,37){\makebox(0,0){$-2$}} \put(30,37){\makebox(0,0){$-1$}}
\put(60,37){\makebox(0,0){$3$}}
\put(40,0){\makebox(0,0){$-3$}} \put(180,37){\makebox(0,0){$-2$}}
\put(210,37){\makebox(0,0){$-1$}}
\put(240,37){\makebox(0,0){$3$}}
 \put(220,0){\makebox(0,0){$-3$}}
\put(180,20){\makebox(0,0){$(3)$}}
\put(200,20){\makebox(0,0){$(6)$}}
\put(240,20){\makebox(0,0){$(1)$}}
\put(263,20){\makebox(0,0){$(-9)$}}
\put(200,0){\makebox(0,0){$(2)$}}

\put(-25,15){\makebox(0,0){$(i)$}}\put(150,15){\makebox(0,0){$(ii)$}}

\end{picture}

\vspace{2mm}

Then the $\Z_3$--cyclic covering of $\partial T$
can be computed as the $\Z_3$--covering of the divisor
marked in $(ii)$  by a similar algorithm as described in
\ref{ss:b} for cyclic coverings.
This provides the  graph $(iii)$ below (as the graph covering of the graph $(i)$). It  is
a possible plumbing graph of $\partial F$ with opposite orientation. Changing the orientation and after
reduced calculus, we get the graph $(iv)$. It  is the graph that we get by the Main Algorithm as well
(after modified by reduced calculus).

The monodromy action permutes the three $-2$--curves.

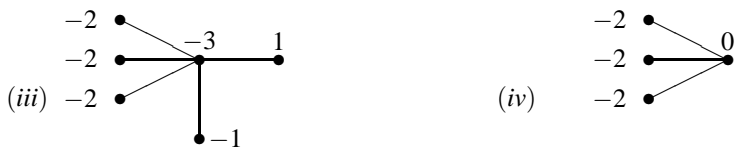
\vspace{1mm}

\begin{picture}(240,65)(-40,-5)
\put(0,15){\circle*{4}} \put(0,45){\circle*{4}}
\put(0,30){\circle*{4}} \put(30,30){\circle*{4}}
\put(60,30){\circle*{4}}
\put(30,0){\circle*{4}}
 \put(0,30){\line(1,0){60}}
 \put(30,30){\line(0,-1){30}}
 \put(30,30){\line(-2,-1){30}} \put(30,30){\line(-2,1){30}}
\put(-15,15){\makebox(0,0){$-2$}}\put(-15,30){\makebox(0,0){$-2$}}
\put(-15,45){\makebox(0,0){$-2$}} \put(30,37){\makebox(0,0){$-3$}}
\put(60,37){\makebox(0,0){$1$}}
\put(40,0){\makebox(0,0){$-1$}}
\put(-35,15){\makebox(0,0){$(iii)$}}\put(150,15){\makebox(0,0){$(iv)$}}

\put(200,15){\circle*{4}} \put(200,45){\circle*{4}}
\put(200,30){\circle*{4}} \put(230,30){\circle*{4}}
 \put(200,30){\line(1,0){30}}
 \put(230,30){\line(-2,-1){30}} \put(230,30){\line(-2,1){30}}
\put(185,15){\makebox(0,0){$-2$}}\put(185,30){\makebox(0,0){$-2$}}
\put(185,45){\makebox(0,0){$-2$}} \put(230,37){\makebox(0,0){$0$}}

\end{picture}
\end{example}

\begin{bekezdes}
The 3--manifold $\partial T$ was studied extensively by several
authors (see for example the articles \cite{S1,S2,H}
of Cohen and Suciu, and  E. Hironaka,  and
the references therein). Hence, once the representation
$\pi_1(\partial T)\to \Z_d$ associated with the covering is
identified, one can try to recover all  data about $\partial F$ from that  of
$\partial T$ (some of them   can be done easily, some of them  by   rather hard work).
Here we will consider only one such  computation, which provides  some numerical data
still missing. This basically targets
the rank of the 1--eigenspace $\dim H_1(\partial F)_1$
associated with the monodromy action.
\end{bekezdes}
\ix{Cohen}\ix{Suciu} \ix{Hironaka, E.}

\begin{bekezdes}\labelpar{bek:dT}{\bf The rank of $ H_1(\partial
F)_1$ and $\corank A_G$.} \

By the above discussion we get $H_1(\partial
F)_1=H_1(\partial T)$. On the other hand, $H_1(\partial T)$ can be computed easily,
via  \ref{re:nss}(b),
since $\partial T$ is also a plumbed 3--manifold (cf. \ref{re:nss}(c)).
One can determine a possible plumbing graph for $\partial T$ as follows.

Start with the curve $C$, blow up the infinitely  near  points of its
 singularities (as in Chapter \ref{hom}), and transform the combinatorics of the resulting
 curve configuration into a graph. In this way we get as a plumbing graph
 the graph $\gc$ with some natural modifications: \ix{curve configuration $\C$}
  we have to keep the genus decorations, delete the arrowheads and the
weight decorations of type $(m;n,\nu)$, and  have to insert  the
Euler numbers. These  can in turn be determined as follows: start with the intersection matrix\ix{matrix!intersection}
of the components of $C$. This, by B\'ezout's theorem,  is the $\Lambda\times \Lambda$ matrix
$A_C$ with entries $(d_\lambda d_\xi)_{\lambda,\xi}$.
This intersection matrix is modified by blow ups to get the intersection matrix
$A_{\gc}$ of the plumbing graph of $\partial T$. Hence, similarly as in \ref{le:h-1}, one has:
 \begin{equation}\label{eq:Tgc} \dim H_1(\partial F)_1=\dim H_1(\partial
T)=2g(\gc)+c(\gc)+\corank A_{\gc}.\end{equation} Using the
notations of Chapter \ref{hom},  we   compute each summand.
$g(\gc)=\sum_{\lambda}g_\lambda$ is clear. Since each
$\Gamma_{j}$ is a tree, one also gets
$$c(\gc)=\sum_j(\, |I_j|-1\,)-(\, |\Lambda|-1\,).$$
Finally, by blowing up, the corank of an intersection matrix stays stable, hence
  $\corank A_{\gc}=\corank A_C=\corank((d_\lambda d_\xi)_{\lambda,\xi})$, i.e.:
\begin{equation}\label{eq:Agc}
\corank A_{\gc}=|\Lambda|-1.\end{equation} Therefore
we get the identities (where $b_1(C)$ denotes the first Betti number of $C$):
\begin{equation}\label{eq:hegy}
\dim H_1(\partial F)_1=2g(\gc)+\sum_j(\, |I_j|-1\,)=b_1(C)+|\Lambda|-1.\end{equation}
This has the following immediate consequence:
\begin{corollary}\labelpar{cor:cora}
$\corank A_G=|\Lambda|-1$. Hence, in general, $A_G$ is
degenerate.
\end{corollary}
\begin{proof}  $\corank A_G=\corank A_{\gc}$ by (\ref{rankegy})  and
(\ref{eq:Tgc}). Then use (\ref{eq:Agc}).
\end{proof}
\end{bekezdes}

Now we start our list of applications.

\section{\ The characteristic polynomial of $\partial F$}\labelpar{ss:charrev}
\setcounter{equation}{0}
Using Theorem \ref{th:charpolG} and Corollary \ref{cor:cora} we get:
\begin{corollary}
\labelpar{th:charpolhom} If $f$ is homogeneous,
then the characteristic polynomial of the Milnor
monodromy acting on $H_1(\partial F)$ is
\begin{equation}\label{eq:vegsohom}
\frac{(t-1)^{|\Lambda|-d}\cdot\prod_{w\in\calw(\gc)}\,
(t^{(m_w,d)}-1)^{2g_w+\delta_w-2} \cdot (t^{\n_w}-1)} {\prod_{e\in
\cale_w(\gc)}\, (t^{\n_e}-1)}.\end{equation} For the integers
$\n_w$ and $\n_e$ see \ref{algo}, with the additional remark
that all the second entries of the vertex--decorations are equal to
$d$.
\end{corollary}
\ix{monodromy!characteristic polynomial!of $\partial F$}

Notice that this is a formula given entirely in terms of $\gc$. If
$c(G)=0$, then  via \ref{bek:comput} it  simplifies to:
\begin{equation}\label{eq:vegsohom2}
(t-1)^{1+|\Lambda|-d}\cdot\prod_{w\in\calw(\gc)}\,
(t^{(m_w,d)}-1)^{2g_w+\delta_w-2}.\end{equation}

Since the monodromy has finite order, the complex algebraic monodromy is
determined completely by its characteristic polynomial. By Corollary \ref{cor:corank},
the complex monodromy is trivial if and only if $c(G)=c(\gc)$ and $g(G)=g(\gc)$, cf. also
with \ref{rem:GC}.

\begin{bekezdes}\labelpar{bek:naivPol} Now, we provide a
description of the characteristic polynomial in terms of the
projective curve $C$.
\ix{monodromy!Jordan block}

Let  $Div=\sum k_\lambda (\lambda)\in\Z[S^1]$ be the divisor of the
characteristic polynomial.

By \ref{feature3}, the multiplicity $k_1$ of the
eigenvalue $\lambda=1$ is $b_1(C)+|\Lambda|-1$.

On the other hand, for $\lambda\not=1$,
section \ref{ss:FIBR} and (\ref{egyenlo}) provide  the following identity in terms
of the local horizontal/Milnor algebraic  monodromies  of the
transversal types (i.e. in terms of the monodromy operators of the local
singularities of $C$):
$$k_\lambda=\sum_j\sum_{\lambda^d=1}\#\{\mbox{Jordan block of $M'_{j,hor}$
with eigenvalue $\lambda$}\} \ \ \ (\lambda\not=1).$$ In order to see this, we apply
\ref{feature1} and the Wang homological exact sequence for
$F'_j\to \partial _{2,j}\to S^1$ with  \ref{cor:h1}.

Indeed, if $B$ is a 1--Jordan block of $M'_{j,hor}$ with eigenvalue
$\lambda$, and $\lambda^d=1$, then this creates a 1--block in
 $M'_{j,ver}$ with eigenvalue 1, and the corresponding 1--dimensional
 eigenspace survives in $\coker(M'_{j,ver}-1)$.

 If $B$ is a 2--Jordan block of $M'_{j,hor}$ with eigenvalue
$\lambda$, and $\lambda^d=1$, then this creates a 2--block in
 $M'_{j,ver}$ with eigenvalue 1, and  in $\coker(M'_{j,ver}-1)$
only a  1--dimensional space survives,  and  $M'_{j,hor}$ acts on it
by multiplication by $\lambda$.
\end{bekezdes}
This and  the last identity of  \ref{cor:corank} imply:
%\begin{equation*}\label{BLOCKS}
$$\sum_j\sum_{\lambda^d=1,\ \lambda\not=1}\ \#\{\mbox{Jordan block of $M'_{j,hor}$
with eigenvalue $\lambda$}\}$$
$$=2g(G)+c(G)-2g(\gc)-c(\gc).$$ %\end{equation*}

\begin{example}\labelpar{ex:ACirr}
Assume that $C$ is irreducible and it has a unique singular
point $(C,p)$ which is locally equisingular to  $\{(f'=u^2+v^3)(u^3+v^2)=0\}$. \ix{equisingularity}
For the embedded
resolution graph of $(C,p)$ see \ref{ex:cyln}. In order to get $\gc$ one
has to add one more vertex, which will have genus decoration
$[g]$, where $2g=(d-1)(d-2)-12$, and it will support $d$ arrows.
The reader is invited to complete the picture of $\gc$.

By (\ref{eq:vegsohom}), we get that the characteristic polynomial
of the monodromy operator acting on  $H_1(\partial F)$ is
$$\frac{(t-1)^{(d-1)(d-2)-10}}{t^{(d,2)}-1}\cdot
\Big(\frac{t^{(d,10)}-1}{t^{(d,5)}-1}\Big)^2.$$ In all the cases,
the multiplicity of the eigenvalue 1 is
$k_1=(d-1)(d-2)-11=(d-1)(d-2)-\mu(f')$. If $d$ is odd, then there
are no other eigenvalues, hence  the complex monodromy on $H_1(\partial
F,\bfc)$ is trivial.

If $d$ is even then other eigenvalues appear too. If $5\nmid d$
then only one appears, namely $-1$, otherwise their divisor is the divisor
of $(t^5+1)^2/(t+1)$.

This result can be compared with \ref{bek:naivPol}:  use
that the algebraic monodromy of the  local transversal singularity
has characteristic polynomial $(t^5+1)^2(t-1)$, and exactly one
2--Jordan block, which has eigenvalue $-1$.

\vspace{2mm}

The graph $\Gmod$ for $d=10$ is the following (for this $d$ all the
Hirzebruch--Jung strings in the Main Algorithm are trivial):

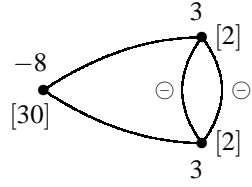
\begin{picture}(200,80)(-50,-8)
%\put(20,25){\vector(-2,1){30}} \put(20,25){\vector(-2,-1){30}}
\put(20,30){\circle*{4}} \put(80,50){\circle*{4}}
\put(80,10){\circle*{4}} \qbezier(20,30)(50,50)(80,50)
\qbezier(20,30)(50,10)(80,10)
\qbezier(80,50)(65,30)(80,10)\qbezier(80,50)(95,30)(80,10)
\put(15,40){\makebox(0,0){$-8$}}
\put(15,20){\makebox(0,0){$[30]$}} \put(78,60){\makebox(0,0){$3$}}
\put(90,50){\makebox(0,0){$[2]$}} \put(78,0){\makebox(0,0){$3$}}
\put(90,10){\makebox(0,0){$[2]$}}
\put(66,30){\makebox(0,0){$\circleddash$}}
\put(95,30){\makebox(0,0){$\circleddash$}}
%\put(50,50){\makebox(0,0){$\circleddash$}}
%\put(55,20){\makebox(0,0){$\circleddash$}}
\end{picture}

In this case $\dim H_1(\partial F)=70$,  $\dim H_1(\partial
F)_1=61$, $A_G$ is non--degenerate with $|\det(A_G)|=50$.

\end{example}

\begin{example}\labelpar{ex:DDD} The characteristic polynomial of the example
\ref{ex:dminusz2} is $(t-1)^{d^2-3d+1}$, hence the complex monodromy on
$H_1(\partial F,\bfc)$ is trivial.
\end{example}

\section{\ $M'_{j,hor},\ M'_{j,hor},$
$M^\Phi_{j,hor},\ M^\Phi_{j,ver},$  $M_{\Phi,hor}$ and
$M_{\phi,ver}$}\labelpar{ss:applJo} \setcounter{equation}{0}
\ix{monodromy}

In this section we show that the local horizontal/Milnor monodromy operators
$\{M'_{j,hor}\}_j$ of the transversal types, together with the integers
$d$ and $c(G)$ completely determine  the
$\Z^2$--representations generated by the commuting pairs
$(M'_{j,hor},M'_{j,ver})$, $(M^\Phi_{j,hor},M^\Phi_{j,ver})$, respectively
$(M_{\Phi,hor},M_{\Phi,ver})$.

\vspace{2mm}

Indeed, the first two pairs agree by  (\ref{egyenlo}).  Hence, by
\ref{feature1},  it is enough to determine $M_{\Phi,hor}$. Moreover,
for eigenvalues $\lambda\not=1$, the following generalized eigenspaces, together with the
corresponding horizontal actions,  coincide:
\begin{equation}\label{eq:JBST}
H_1(F_{\Phi})_{M_{\Phi,hor},\lambda\not=1}=
\oplus_j H_1(F_{\Phi}\cap T_j)_{M^\Phi_{j,hor},\lambda\not=1}=
\oplus_j H_1(F'_j)_{M'_{j,hor},\lambda\not=1}.\end{equation}
Next, we determine the  action of $M_{\Phi,hor}$ on the
generalized 1--eigenspace $H_1(F_{\Phi})_{1}$ of $M_{\Phi,hor}$.
\ix{generalized eigenspace}

The rank of $H_1(F_{\Phi})_{M_{\Phi,hor},1}$ can be determined in several ways.

First, one can consider the rank of the total space $H_1(F_\Phi)$, which is
$(d-1)^2$. Indeed,
$F_\Phi$ is the Milnor number of an ICIS $(f,g)$, where $f$ is homogeneous of degree $d$ and
$g$ is a  generic linear form. Hence,
 their Milnor number is the Milnor number of a homogeneous plane curve singularity of degree $d$. Since we already determined the $(\lambda\not=1)$--generalized eigenspaces,
the rank of $H_1(F_{\Phi})_{1}$ follows too.\ix{Milnor!number!of ICIS|textbf}

There is another formula based on \ref{charpols}(a), which
provides
\begin{equation*}\begin{split}
\rank  H_1(F_{\Phi})_{M_{\Phi,hor},1}=&2g(\gc)+2c(\gc)+d-1
\\ \ &=d-|\Lambda|+b_1(C)+\sum_j(|I_j|-1).\end{split}\end{equation*}
In particular, since we have only one and two--dimensional Jordan blocks:
$$\#^1_1M_{\Phi,hor}+2 \#^2_1M_{\Phi,hor}=\rank  H_1(F_{\Phi})_{M_{\Phi,hor},1}.$$
On the other hand, by \ref{feature1} and (\ref{eq:JBST}) we have
$$\#^2_1M_{\Phi,ver}=\#^2_1M_{\Phi,hor}+\sum_j\sum_{\lambda^d=1\not=\lambda}
\#^2_\lambda M'_{j,hor}.$$
Since by \ref{cor:corank} the left hand side of this identity is $c(G)$, we get
$$\#^2_1M_{\Phi,hor}=c(G)-\sum_j\sum_{\lambda^d=1\not=\lambda}
\#^2_\lambda M'_{j,hor}.$$

\begin{example}\labelpar{ex:ACirr2}
Assume that we are in the situation of Example \ref{ex:ACirr} with
$d=10$. Then both local operators $M'_{1,hor}$ and $M'_{1,ver}$
have exactly one 2--Jordan block, the first with eigenvalue $-1$, the
other with 1. This is compatible with the fact that
$c(G_{2,1})=1$.

Since the order of any eigenvalue of $M_{\Phi,hor}$ divides 10, we
get that $M_{\Phi,ver}$ is unipotent. By Corollary
\ref{cor:corank}, the number of 2--Jordan blocks of $M_{\Phi,ver}$ is $c(G)=2$.
Hence  $M_{\Phi,hor}$  has two 2--Jordan blocks as well.
By (\ref{eq:JBST}), there is only one with eigenvalue $\not=1$,
namely with  eigenvalue $-1$. Hence the remaining block has eigenvalue 1.
\end{example}

\section{\ When is $\partial F$ a rational/integral  homology
sphere?}\labelpar{ss:ratS}\setcounter{equation}{0} By \ref{bek:naivPol}, or by
the statements of \ref{feature2},  we get the
following characterization:
\ix{sphere!rational homology}

\begin{corollary}
$\partial F$ is a rational homology sphere if and only if $C$ is
an irreducible, rational cuspidal curve (i.e. all its
singularities are locally irreducible), and there is no eigenvalue
$\lambda$  of the local singularities $(C,p_j)$ with
$\lambda^d=1$.
\end{corollary}

This situation can indeed  appear, see e.g. \ref{ex:C4}, or many examples
of the present work.

We wish to emphasize  that the classification  of rational
cuspidal projective curves is an  open problem. (For a
partial classification with $d$ small see e.g. \cite{namba}, or
\cite{Spany} and the references therein.) It would be very
interesting and important to continue the program of the present
work: some restrictions on  the structure of $\partial F$
might provide  new data in the classification problem as well.

%\subsection{
The next natural question is:
{\it When is $\partial F$ an integral   homology sphere?}
%\labelpar{ss:intS} %\setcounter{equation}{0}
The answer  is simple:
Never (provided that $d\geq 2$)! This follows from the following
proposition:
\ix{sphere!homology}

\begin{proposition}\labelpar{prop:intS}
Assume that $f$ is homogeneous of degree $d$. Then there exists an
element $h\in H_1(\partial F,\Z)$ whose order is either infinity or a multiple of $d$.
If $ H_1(\partial F,\Z)$ is finite, one can choose such an  $h$  fixed by the monodromy action.
\end{proposition}
\begin{proof}
By \ref{ss:ratS} we may assume that $C$ is irreducible, rational
and cuspidal. Consider the following diagram (for notations, see
\ref{feature3}):

\begin{equation*}
\begin{CD}
\pi_1(\partial F) @ >i>> \pi_1(\partial T)@>p>>\Z_d\\
@VV{q}V @VV{r}V  \\
H_1(\partial F,\Z)@>>>H_1(\partial T,\Z)@ =\Z_{d^2}
\end{CD}
\end{equation*}
\ix{projective curve!cuspidal}

\vspace{2mm}

\noi The first line, provided by the homotopy exact sequence of the covering, is exact. Moreover,
$i$ is injective, and $q,\ r$ and $p$ are surjective.  Let
$\gamma\in\pi_1(\partial T)$ be the class of a small loop around $C$. Then
$p(\gamma^d)=0$, hence $\gamma^d\in\pi_1(\partial F)$.

On the other hand, $H_1(\partial T,\Z)=\Z_{d^2}$, and it is generated
by $[\gamma]$. This follows from the fact (see  \ref{bek:dT} and \ref{re:TOR})
that $H_1(\partial T,\Z)=\coker A_C$, but in this case
$A_C$ has only one entry, namely $d^2$.

Hence $r(\gamma^d)$ has order $d$ in $H_1(\partial
T,\Z)$. In particular, the order of $q(\gamma^d)\in H_1(\partial
F,\Z)$ is a multiple of $d$. Since the monodromy action in
$\pi_1(\partial F)$ is conjugation by $\gamma$, the second
part also follows.
\end{proof}

The above bound is sharp: there are examples when $H_1(\partial
F,\Z)=\Z_d$,  see e.g. \ref{ex:C4}. The above result also shows
that the monodromy  action over $\Z$ is trivial whenever
$H_1(\partial F,\Z)=\Z_d$.

The fact that for $f$ homogeneous $\partial F$ cannot be an integral homology  sphere
was for the first time noticed by Siersma in \cite{Si}, page 466,  where he used a
different argument.\ix{Siersma}

\section{\ Cases with $d$ small}\labelpar{ss:dsmall} \setcounter{equation}{0}\
In this section we treat all possible examples with $d=2$

and $d=3$, and we present some examples with $d=4$. The statements are direct applications of
the  Main Algorithm and the above discussions regarding
the monodromy.
\ix{singularities!homogeneous!small degree}

One of our goals in this `classification list' is to provide  an idea about
 the variety of the possible 3--manifolds obtained in this way, and to check if the same 3--manifold can
be realized from essentially two different situations.

\vspace{2mm}

\begin{bekezdes}($\mathbf{d=2}$) There is only one case: $C$ is a union of two
lines. A possible graph $\Gmod$ is \begin{picture}(30,10)(0,0)
\put(20,4){\circle*{4}} \put(13,4){\makebox(0,0){$0$}}
\end{picture}; hence $\partial F\approx S^1\times S^2$.  The monodromy is trivial.
\end{bekezdes}

\begin{bekezdes}($\mathbf{d=3}$) There are six cases:

\vspace{2mm}

{\bf (3a)} $C$ is irreducible with a cusp. See \ref{ex:CUSP}, or \ref{bek:classifi}(a)
with $d=3$ for the graph $\Gmod$  of $\partial F$. In this case
%\begin{picture}(150,55)(-50,0)
%\put(20,25){\line(2,1){40}} \put(20,25){\line(1,0){40}}
%\put(20,25){\line(2,-1){40}}
%\put(60,25){\circle*{4}}\put(60,45){\circle*{4}}\put(60,5){\circle*{4}}
%\put(20,25){\circle*{4}} \put(20,37){\makebox(0,0){$0$}}
%\put(70,25){\makebox(0,0){$-2$}}\put(70,45){\makebox(0,0){$-2$}}
%\put(70,5){\makebox(0,0){$-2$}}
%\end{picture}
$H_1(\partial F,\Z)=\Z_6\oplus \Z_2$, hence the complex monodromy
is trivial. The integral homology  has the following peculiar  form
(specific to the  homology of any Seifert 3--manifold):\ix{Seifert manifold}
$$H_1(\partial F,\Z)=\langle \ x_1,x_2,x_3,h\ | \ 2x_1=2x_2=2x_3=h,\
x_1+x_2+x_3=0\ \rangle.$$ $h$ has order 3 and it is preserved by
the integral monodromy, while the elements $x_1,x_2,x_3$ are
cyclically permuted (in order to prove this, use e.g. \ref{ex:CUSP}).

\vspace{2mm}

{\bf (3b)} $C$ is irreducible with a node.  A possible graph of
$\partial F$ is

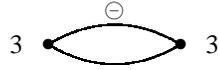
\begin{picture}(200,50)(-20,-5)
\put(50,15){\circle*{4}} \put(100,15){\circle*{4}}
\qbezier(50,15)(75,30)(100,15) \qbezier(50,15)(75,0)(100,15)
\put(38,15){\makebox(0,0){$3$}} \put(112,15){\makebox(0,0){$3$}}
\put(75,28){\makebox(0,0){$\circleddash$}}
\end{picture}

The rank of $H_1(\partial F)$ is 1, and the complex monodromy is
trivial.

By calculus, one can verify that $G\sim -G$.

\vspace{2mm}

{\bf (3c)} $C$ is the union of a line and an irreducible conic,
intersecting each other transversely.  A possible graph for
$\partial F$ is
\begin{picture}(40,10)(0,0)
\put(20,4){\circle*{4}} \put(13,4){\makebox(0,0){$3$}}
\put(30,4){\makebox(0,0){$[1]$}}
\end{picture}
, and the complex monodromy is trivial.

\vspace{2mm}

{\bf (3d)} $C$ is the union of an irreducible conic and one of its tangent
lines. Then $\partial F\approx S^2\times S^1$ and the monodromy is trivial.

\vspace{2mm}

{\bf (3e)} $C$ is the union of three lines in general position; then
$\partial F\approx S^1\times S^1\times S^1$ and the monodromy is trivial.

\vspace{2mm}

{\bf (3f)} $C$ is the union of three lines intersecting each other
in one point: then by {\em reduced} calculus we get for $\partial F$ the graph $\Gmod$:

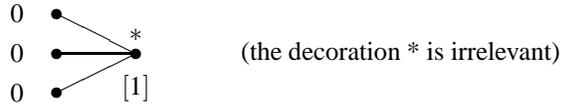
\vspace{1mm}

\begin{picture}(140,65)(100,-5)
\put(200,15){\circle*{4}} \put(200,45){\circle*{4}}
\put(200,30){\circle*{4}} \put(230,30){\circle*{4}}
 \put(200,30){\line(1,0){30}}
 \put(230,30){\line(-2,-1){30}} \put(230,30){\line(-2,1){30}}
\put(185,15){\makebox(0,0){$0$}}\put(185,30){\makebox(0,0){$0$}}
\put(185,45){\makebox(0,0){$0$}} \put(230,37){\makebox(0,0){$*$}}
\put(230,17){\makebox(0,0){$[1]$}}\put(330,30){\makebox(0,0){(the decoration * is irrelevant)}}
\end{picture}

In fact, if we apply the {\em splitting operation} (not permitted by reduced calculus),
we obtain that  $\partial F\approx \#_4 S^2\times S^1$.

The characteristic polynomial of the monodromy is $(t^3-1)(t-1)$.
\end{bekezdes}

\begin{bekezdes}($\mathbf{d=4}$) We will start with the cases  when {\bf $C$ is
irreducible, rational and cuspidal} (i.e. all singularities are locally irreducible).
There are four cases:\ix{projective curve!cuspidal}

\vspace{2mm}

{\bf (4a)} $C$ has three $A_2$ singular points. A possible equation
for $f$ and $\gc$ is given in \ref{ex:3A2}. The boundary $\partial F$
is a rational homology sphere with graph:

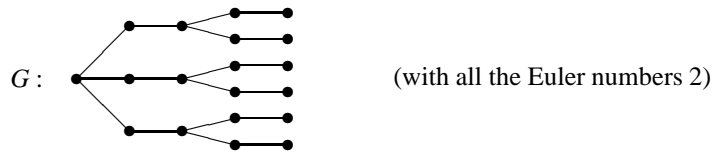
\begin{picture}(250,80)(-50,10)
\put(0,50){\circle*{4}}\put(20,50){\circle*{4}}\put(20,70){\circle*{4}}
\put(20,30){\circle*{4}}
\put(40,50){\circle*{4}}\put(40,70){\circle*{4}}\put(40,30){\circle*{4}}
\put(60,75){\circle*{4}}\put(60,65){\circle*{4}}
\put(60,45){\circle*{4}}\put(60,55){\circle*{4}}
\put(60,25){\circle*{4}}\put(60,35){\circle*{4}}
\put(80,75){\circle*{4}}\put(80,65){\circle*{4}}
\put(80,45){\circle*{4}}\put(80,55){\circle*{4}}
\put(80,25){\circle*{4}}\put(80,35){\circle*{4}}
\put(0,50){\line(1,1){20}} \put(0,50){\line(1,0){20}}
\put(0,50){\line(1,-1){20}} \put(20,50){\line(1,0){20}}
\put(20,70){\line(1,0){20}} \put(20,30){\line(1,0){20}}
\put(40,50){\line(4,1){20}} \put(40,70){\line(4,1){20}}
\put(40,30){\line(4,1){20}} \put(40,50){\line(4,-1){20}}
\put(40,70){\line(4,-1){20}} \put(40,30){\line(4,-1){20}}
\put(60,75){\line(1,0){20}} \put(60,45){\line(1,0){20}}
\put(60,65){\line(1,0){20}} \put(60,35){\line(1,0){20}}
\put(60,55){\line(1,0){20}} \put(60,25){\line(1,0){20}}
\put(180,50){\makebox(0,0){(with all the Euler numbers  $2$)}}
\put(-20,50){\makebox(0,0){ $G:$}}
\end{picture}

\noindent
(In this case, neither $G$, nor $G$ with the opposite
orientation can be represented by a negative definite graph.) \ix{graph!negative definite}

\vspace{2mm}

{\bf (4b)}  $C$ has two singular points with local equations
$u^2+v^3=0$ and $u^2+v^5=0$. This case has an unexpected surprise
in store: having two singular points we expect that the graph will
have two nodes. This is indeed the case for the graph $G$,  the
immediate output  of the Main Algorithm.  Nevertheless, via
reduced plumbing calculus the chain connecting the two nodes
disappears by a final 0--chain absorption. We get that
$\partial F$ is a Seifert manifold with graph:\ix{Seifert manifold}

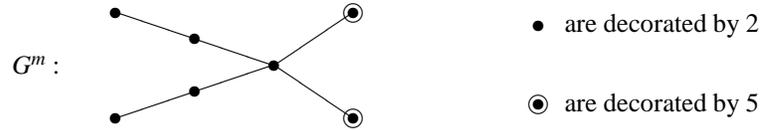
\begin{picture}(250,70)(-30,15)
\put(100,50){\circle*{4}}\put(70,60){\circle*{4}}
\put(40,70){\circle*{4}}
\put(70,40){\circle*{4}}\put(40,30){\circle*{4}}\put(130,70){\circle*{4}}
\put(130,30){\circle*{4}} \put(130,70){\circle{7}}
\put(130,30){\circle{7}}
 \put(100,50){\line(-3,1){60}}
\put(100,50){\line(-3,-1){60}} \put(100,50){\line(3,2){30}}
\put(100,50){\line(3,-2){30}}
\put(10,50){\makebox(0,0){$\Gmod:$}}
 \put(200,65){\circle*{4}}\put(200,35){\circle*{4}} \put(200,35){\circle{7}}
 \put(210,65){\makebox(0,0)[l]{are decorated by $2$}}
\put(210,35){\makebox(0,0)[l]{are decorated by $5$}}
\end{picture}

\noindent
The orbifold Euler number is $4/15>0$, and the order of
$H_1(\partial F,\Z)$ is 60.

\vspace{2mm}

{\bf (4c)} $C$ has one singular point with local equation
$u^3+v^4=0$. The 3--manifold $\partial F$ is a rational homology sphere, its
plumbing graph is given in  \ref{bek:classifi2}(a)
(with $d=4$). It is worth  mentioning that
there are two projectively non--equivalent curves of degree four
with this local data, namely $x^4-x^3y+y^3z=0$ and $x^4-y^3z=0$.

%\begin{picture}(150,55)(-50,0)
%\put(20,25){\line(4,1){40}} \put(20,25){\line(2,1){40}}
%\put(20,25){\line(4,-1){40}}\put(20,25){\line(2,-1){40}}
%\put(60,35){\circle*{4}}\put(60,45){\circle*{4}}\put(60,5){\circle*{4}}
%\put(60,15){\circle*{4}} \put(20,25){\circle*{4}}
%\put(20,37){\makebox(0,0){$0$}}
%\put(70,35){\makebox(0,0){$-3$}}\put(70,45){\makebox(0,0){$-3$}}
%\put(70,5){\makebox(0,0){$-3$}}\put(70,15){\makebox(0,0){$-3$}}
%\end{picture}

\vspace{2mm}

{\bf (4d)} $C$ has one singular point with local equation
$u^2+v^7=0$. In this case, rather
surprisingly, $\partial F$ is the  lens space of type $L(28,15)$
(that is,  the graph is a string with decorations: $-2,\ -8, \ -2$).

\vspace{2mm}

For an example of {\bf irreducible non--rational} $C$, see \ref{ex:dminusz2} or
\ref{ex:DDD}.

For {\bf reducible $C$}  we give three more  examples:

\vspace{2mm}

{\bf (4a')} Let $C$ be  the union of two conics intersecting each other
transversely. The characteristic polynomial is $(t-1)^4$, the complex monodromy
is trivial,   and the graph of $\partial F$ is:

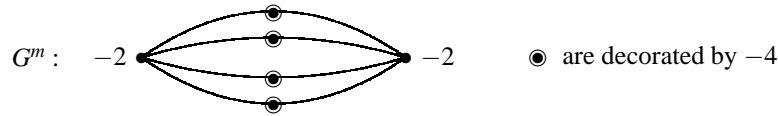
\begin{picture}(200,60)(-20,-15)
\put(50,15){\circle*{4}} \put(150,15){\circle*{4}}
\put(100,22){\circle*{4}} \put(100,7){\circle*{4}}
\put(100,32){\circle*{4}} \put(100,-3){\circle*{4}}
\put(200,15){\circle*{4}}
\put(100,22){\circle{6}} \put(100,7){\circle{6}}
\put(100,32){\circle{6}} \put(100,-3){\circle{6}}
\put(200,15){\circle{6}}
\put(250,15){\makebox(0,0){are decorated by $-4$}}
\qbezier(50,15)(100,30)(150,15) \qbezier(50,15)(100,0)(150,15)
\qbezier(50,15)(100,50)(150,15) \qbezier(50,15)(100,-20)(150,15)
\put(38,15){\makebox(0,0){$-2$}} \put(162,15){\makebox(0,0){$-2$}}
\put(10,15){\makebox(0,0){$\Gmod:$}}
%\put(75,15){\makebox(0,0){$\circleddash$}}
%\put(75,25){\makebox(0,0){$\circleddash$}}
%\put(75,5){\makebox(0,0){$\circleddash$}}
\end{picture}

{\bf (4b')} Let  $C=$ be  the union of a smooth  irreducible conic
and two different tangent lines. The characteristic polynomial is
$(t^2+1)^2(t-1)^3$. The graph $\Gmod$ is
\begin{picture}(30,10)(10,0)
\put(20,4){\circle*{4}} \put(13,4){\makebox(0,0){$0$}}
\put(30,4){\makebox(0,0){$[3]$}}
\end{picture}.

\vspace{2mm}

{\bf (4c')} Let  $C$ be the union of two smooth conics  intersecting  tangentially in two
singular points of type $A_3$ (e.g. the equation of $f$ is $x^2y^2+z^4$).
Then the  graph $\Gmod$ is
\begin{picture}(40,10)(0,0)
\put(20,4){\circle*{4}} \put(13,4){\makebox(0,0){$4$}}
\put(30,4){\makebox(0,0){$[3]$}}
\end{picture}. The characteristic polynomial is
$(t^2+1)^2(t-1)^2$.
\ix{singularities!line arrangements!$A_3$}

\end{bekezdes}

\begin{bekezdes}\labelpar{bek:AC2comp}
It is also instructive to see the case
when  $C$ is given by $(x^3-y^2z)(y^3-x^2z)=0$.  The curve $C$ has two
components, and they intersect each other in 6 points. One of
them, say $(C,p)$, has local equation $(u^3-v^2)(v^3-u^2)=0$, cf.
\ref{ex:cyln}, the others are nodes. Some of the cycles of $G$
are generated by these intersections, and one more by a 2--Jordan
block of $(C,p)$: $c(\gc)=5$ and $c(G)=6$.
\end{bekezdes}

\section{\ Rational unicuspidal
curves with one Puiseux pair}\labelpar{ss:unicusp}\setcounter{equation}{0} \

The present section is motivated by two facts.
\ix{singularities!homogeneous!rational unicuspidal}\ix{projective curve!unicuspidal}

First, by such curves we produce plumbed 3--manifolds with
start--shaped graphs (cf. \ref{cor:Seif}),
some of which are  irreducible 3--manifolds
with Seifert fiber--structure, some of which are not irreducible.\ix{Seifert manifold}
Seifert 3--manifolds play a special  role in the world of
3--manifolds. Therefore, it is an  important task to classify
all Seifert manifolds realized as $\partial F$, where $F$ is a
Milnor fiber as above. Even if we restrict ourselves to the
homogeneous case, the problem looks surprisingly hard: although all unicuspidal rational
curves with one Puiseux pair produce star--shaped graphs,
in \ref{ss:dsmall}(4b) we have found a curve $C$ with two cusps, which
produces a Seifert manifold as well. Since the classification of the
cuspidal rational curves is not finished, the above question looks
hard. On the other hand, notice  that Seifert manifolds can also be produced by
other types of germs as well, e.g. by some weighted homogeneous germs as in \ref{ex:347b}
or \ref{re:SEIFERT},
but not all weighted homogeneous germs provide  Seifert manifolds.\ix{projective curve!cuspidal}

In this section we will list those star--shaped graphs which are
produced by unicuspidal curves with one Puiseux pair.
%(In one of the cases we get a lens space of very special type.)
%The authors confess that they do not see a common characterization of the star--shaped
%graphs obtained in this way.

It is also worth  mentioning, that
the normalization of $V_f$ (for any $f$ where $C$ is irreducible and rational) has a very simple resolution
graph: only one vertex with genus zero and Euler number $-d$. This graph can be compared  with
those obtained for $\partial F$ listed below.

For the second motivation (of more analytic nature), see \ref{bek:Kodaira}.

\begin{bekezdes}\labelpar{bek:classifi} Assume that $C$ is rational
and unicuspidal of degree $d\geq 3$, and the equisingularity type of its
singularity is given by the local equation $u^a+v^b=0$. The
possible triples $(d,a,b)$ are classified in \cite{BLMN}.\ix{projective curve!unicuspidal}

\ix{Fibonacci numbers}
In order to state this result, consider the Fibonacci numbers
$\{\varphi_j\}_{j\geq 0}$ defined by $\varphi_0=0$, $\varphi_1=1$,
and $\varphi_{j+2}=\varphi_{j+1}+\varphi_j$ for $j\geq 0$.

Then $(d,a,b)$ ($1<a<b$) is realized by a curve with the above properties
if and only if it appears in the following list:

\vspace{2mm}

(a) \ $(a,b)=(d-1,d)$;

(b) \ $(a,b)=(d/2,2d-1)$, where $d$ is even;

(c) \ $(a,b)=(\varphi^2_{j-2},\varphi^2_j)$ and
$d=\varphi^2_{j-1}+1=\varphi_{j-2}\varphi_j$, where $j$ is odd and
$\geq 5$;

(d) \ $(a,b)=(\varphi_{j-2},\varphi_{j+2})$ and $d=\varphi_{j}$,
where $j$ is odd and $\geq 5$;

(e) \ $(a,b)=(\varphi_4,\varphi_8+1)=(3,22)$ and $d=\varphi_6=8$;

(f) \ $(a,b)=(2\varphi_4,2\varphi_8+1)=(6,43)$ and
$d=2\varphi_6=16$.

\vspace{2mm}

(a) is realized e.g. by the curve $\{x^d=zy^{d-1}\}$, (b) by
$\{zy-x^2)^{d/2}=xy^{d-1}$. The cases (c) and (d) appear in Kashiwara's list
\cite{Kashiwara}, while (e) and (f) were found by Orevkov in \cite{Orevkov}.
\ix{Kashiwara} \ix{Orevkov}

\begin{bekezdes}\labelpar{bek:classifi2}
Now we list the graphs $\Gmod$ for $\partial F$. The computations are straightforward,
perhaps  except the cases (c) and (d); for these ones we provide more details.
\end{bekezdes}

$\bullet$ {\bf Case (a)}

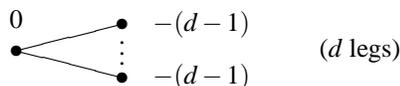
\begin{picture}(150,45)(-50,0)
\put(20,25){\line(4,1){40}} \put(20,25){\line(4,-1){40}}
\put(60,35){\circle*{4}} \put(60,15){\circle*{4}}
\put(20,25){\circle*{4}} \put(20,37){\makebox(0,0){$0$}}
\put(90,35){\makebox(0,0){$-(d-1)$}}\put(60,27){\makebox(0,0){$\vdots$}}
\put(90,15){\makebox(0,0){$-(d-1)$}}\put(150,25){\makebox(0,0){($d$
legs)}}%\put(90,27){\makebox(0,0){$\vdots$}}
\end{picture}

$\bullet$ {\bf Case (b)}

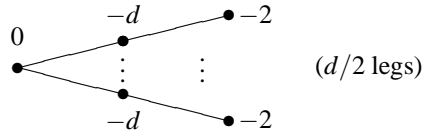
\begin{picture}(150,45)(-50,0)
\put(20,25){\line(4,1){80}} \put(20,25){\line(4,-1){80}}
\put(60,35){\circle*{4}}
\put(60,15){\circle*{4}}\put(100,45){\circle*{4}}\put(100,5){\circle*{4}}
\put(60,35){\circle*{4}} \put(20,25){\circle*{4}}
\put(20,37){\makebox(0,0){$0$}} \put(60,45){\makebox(0,0){$-d$}}
\put(60,27){\makebox(0,0){$\vdots$}}
\put(60,5){\makebox(0,0){$-d$}}\put(153,25){\makebox(0,0){($d/2$
legs)}}\put(90,27){\makebox(0,0){$\vdots$}}
\put(110,45){\makebox(0,0){$-2$}}\put(110,5){\makebox(0,0){$-2$}}
\end{picture}

\vspace{2mm}

In the examples (a) and (b) the above `normal forms' are
{\it not} negative definite, hence the graphs cannot be
transformed by plumbing calculus into a negative definite graph
(without changing the orientation). \ix{graph!negative definite}

\vspace{3mm}

$\bullet$ {\bf Case (c)} Let $\Gamma$ be the minimal resolution graph
of the Brieskorn isolated hypersurface singularity $\{u^a+v^b+w^d=0\}$, where
$(a,b,d)=(\varphi_{j-2}^2,\varphi_j^2,\varphi_{j-2}\varphi_j)$.
This is a star--shaped graph, it can be determined via the procedure \ref{ss:b} or \ref{8.7}.
Here we will indicate the steps of \ref{8.7}.
\ix{graph!resolution!Brieskorn}

Some arithmetical properties of the Fibonacci numbers  enter as
important ingredients in the computations. Namely, we  need the facts  that
 ${\rm gcd}(  \varphi_{j-2},\varphi_j)=1$,  and for $j$ odd one also has:
\begin{equation}\label{eq:FIBON}
\varphi_{j-2}\cdot \varphi_{j+2}=\varphi_j^2+1, \ \  \ \  \
\varphi_{j-2}+\varphi_{j+2}=3\varphi_j.
\end{equation}\ix{Fibonacci numbers}
 Using the notations from \ref{8.7},
we will write $(a_1,a_2,a_3)=(\varphi_{j-2}^2,\varphi_j^2,\varphi_{j-2}\varphi_j)$.
Then $(d_1,d_2,d_3)=(\varphi_{j-2},\varphi_j,\varphi_{j-2}\varphi_j)$,
$(\alpha_1,\alpha_2,\alpha_3)=(\varphi_{j-2},\varphi_j, 1)$, and $(\omega_1,\omega_2,\omega_3)=(1,1,0)$.
Hence, the embedded resolution of the suspension germ  is a star shaped graph, whose each leg  has only one vertex.
There are  $\varphi_j$ legs each with Euler number $-\varphi_{j-2}$,
and  $\varphi_{j-2}$ legs each with Euler number $-\varphi_{j}$. The central vertex has Euler number
$-3$ and genus decoration  $g=(\varphi_j-1)(\varphi_{j-2}-1)/2$.
The arrow of the function germ $w$ has multiplicity 1 and it  is supported by the central vertex which has
multiplicity $\varphi_{j-2}\varphi_j$.

Therefore, the graph $G^m$ is the following:  it is a star shaped graph with all the legs having only one vertex.
It has $\varphi_j$ legs each with Euler number $\varphi_{j-2}$, and
 $\varphi_{j-2}$ legs each with Euler number $\varphi_{j}$, as well  as  a leg  with Euler number 0.
 The central vertex has  genus decoration  $g=(\varphi_j-1)(\varphi_{j-2}-1)/2$ and Euler number
$3$ (although this is not relevant anymore, because of the 0--leg).
\ix{connected sum}

This 3--manifold is not irreducible.
The irreducible components can be
seen if we apply the splitting operation to the unique 0--vertex. Then we get\\
 $2g=(\varphi_j-1)(\varphi_{j-2}-1)$ copies of
\begin{picture}(16,20)(10,0) \put(20,4){\circle*{4}}
\put(20,12){\makebox(0,0){$0$}}
\end{picture}, $\varphi_{j-2}$ copies of  \begin{picture}(16,20)(10,0)
\put(20,4){\circle*{4}} \put(20,12){\makebox(0,0){$\varphi_{j}$}}
\end{picture}, and $\varphi_{j}$ copies of \begin{picture}(16,20)(10,0)
\put(20,4){\circle*{4}}
\put(20,12){\makebox(0,0){$\varphi_{j-2}$}}
\end{picture}.

\vspace{4mm}

$\bullet$ {\bf Case (d)} In this case the graph $\Gmod$ of $\partial F$ is
\begin{picture}(20,10)(10,0)
\put(20,4){\circle*{4}} \put(20,12){\makebox(0,0){$d$}}
\end{picture}. This follows again from the very special arithmetical properties (\ref{eq:FIBON}) of the
Fibonacci numbers.

Similarly as in the previous case,  we have  first to  determine
the minimal resolution graph $\Gamma$ of the Brieskorn singularity
with coefficients $(a,b,d)$. Using (\ref{eq:FIBON}) it is easy to verify that
these integers are pairwise relatively prime. Moreover, $\varphi_{j\pm 2}$ is not divisible by 3,
and  one also has   the following congruences
\begin{equation}\label{eq:FIBON2}
\varphi_{j-2} \varphi_{j+2}\equiv 1\ (mod\ \varphi_j)  \ \  \ \  \
\varphi_{j\pm 2}\varphi_{j}\equiv -3 \ (mod\ \varphi_{j\mp 2}).
\end{equation}\ix{Fibonacci numbers}

Hence, by a computation, $\Gamma$ consists of a central vertex with three legs.
The central vertex  has decoration $-2$ (and genus 0), one  leg
consists of  $\varphi_j-1$ vertices each decorated by $-2$, the second  has three vertices
with decorations $-2,\ -2,\ -\lceil\varphi_{j\pm 2}/3\rceil$ (where the first
$-2$ vertex is connected to the central vertex). The  third leg has two
vertices with decorations $-3$ and $-\lceil\varphi_{j\mp 2}/3\rceil$, where the $-3$
vertex is connected to the central vertex. (Note that among the integers
$\varphi_{j-2}$ and $\varphi_{j+2}$ exactly one has the form $3k-1$. The leg whose
$\alpha$--invariant has the form $3k-1$  will be the leg with two vertices, and the other leg
will have three vertices.)

By the procedure \ref{feature2} we have to change the orientation (hence the $(-2)$--string
transforms into a 2--string) and add
one more vertex to the end-curve of the $2$--string, which will be decorated by 1.
Hence this new extended string can be contracted completely. After this
contraction the central curve becomes a 1--curve. This generates 3 more
 blow--downs and a 0--chain absorption.

As an interesting phenomenon:
the graph of the normalization of  $V_f$ is
\begin{picture}(20,10)(10,0)
\put(20,4){\circle*{4}} \put(20,12){\makebox(0,0){$-d$}}
\end{picture}.

\noindent Hence the boundary of the Milnor fiber coincides with the link of
the normalization with opposite orientation.

\vspace{3mm}

 $\bullet$ {\bf Case (e)}

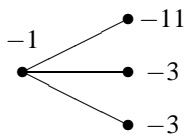
\vspace{1mm}

\begin{picture}(150,55)(-50,0)
\put(20,25){\line(2,1){40}} \put(20,25){\line(1,0){40}}
\put(20,25){\line(2,-1){40}}
\put(60,25){\circle*{4}}\put(60,45){\circle*{4}}\put(60,5){\circle*{4}}
\put(20,25){\circle*{4}} \put(20,37){\makebox(0,0){$-1$}}
\put(73,25){\makebox(0,0){$-3$}}\put(73,45){\makebox(0,0){$-11$}}
\put(73,5){\makebox(0,0){$-3$}}
\end{picture}

\vspace{2mm}

$\bullet$ {\bf Case (f)}\ix{graph!negative definite}

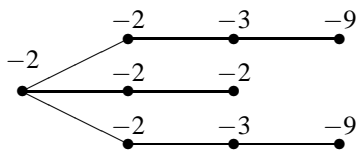
\vspace{2mm}

\begin{picture}(150,55)(-50,0)
\put(20,25){\line(2,1){40}} \put(20,25){\line(1,0){40}}
\put(20,25){\line(2,-1){40}}
 \put(60,25){\line(1,0){40}}
\put(60,45){\line(1,0){80}} \put(60,5){\line(1,0){80}}
\put(60,25){\circle*{4}}\put(60,45){\circle*{4}}\put(60,5){\circle*{4}}
\put(20,25){\circle*{4}}
\put(100,25){\circle*{4}}\put(100,45){\circle*{4}}\put(100,5){\circle*{4}}
\put(140,45){\circle*{4}}\put(140,5){\circle*{4}}
\put(20,37){\makebox(0,0){$-2$}}
\put(60,32){\makebox(0,0){$-2$}}\put(60,52){\makebox(0,0){$-2$}}
\put(60,12){\makebox(0,0){$-2$}}
\put(100,32){\makebox(0,0){$-2$}}\put(100,52){\makebox(0,0){$-3$}}
\put(100,12){\makebox(0,0){$-3$}}
\put(140,52){\makebox(0,0){$-9$}}
\put(140,12){\makebox(0,0){$-9$}}
\end{picture}

\vspace{4mm}

The  normal forms in (e) and (f) are {\it  negative
definite}, a property which rather rarely happens for non--isolated singularities. In particular,
in both cases,  $\partial F$ is diffeomorphic (under a
diffeomorphism which preserves orientation) with the link of a(n)
(elliptic, not complete intersection) normal surface singularity.
\end{bekezdes} \ix{graph!negative definite}

\begin{bekezdes}\labelpar{bek:Kodaira}
Looking at the shapes and decorations
 of the above graphs, one can group them in four
categories (a)--(b), (c), (d) and (e)--(f). This is compatible with  other
groupings based on certain analytic/combinatorial invariants of these curves,
see \cite{BLMN,Spany}.

For example, if $\pi:X\to \bp^2$ is the minimal embedded resolution of $C\subset \bp^2$, and
$\bar{C}$ denotes the strict transform of $C$, then $\bar{C}^2=d^2-ab=3d-a-b-1$.
In the above six cases this value is the following: it is positive for (a) and (b),
it is zero for (c), it is $-1$ for (d), and $-2$ for (e) and (f).

Perhaps,  the most striking `coincidence'
is that the cases (e)--(f) are the only cases when the logarithmic
Kodaira dimension of $\bp^2\setminus C$ is 2 (for the first four cases  it is $-\infty$),
while from the point of view of
the present discussion (e)--(f)  are the  only cases when  $\partial F$ is an irreducible
Seifert 3--manifold, \ix{Seifert manifold}
representable  by a negative definite graph (and which is not a lens space).\ix{Kodaira dimension}
\ix{graph!negative definite}

This suggests that the topology of $\partial F$  probably carries
a great  amount of analytic information about $f$.
\end{bekezdes}

\section{\ The weight filtration of the mixed Hodge structure}\labelpar{s:MHS2}
\setcounter{equation}{0} \

\begin{proposition} If $f$ is homogeneous
then the  mixed Hodge structure of  $H^*(\partial F)$ satisfies Conjecture
\ref{be:mhs2}.
\end{proposition}
\ix{mixed Hodge structure!weight filtration}

\begin{proof}
Let $d$ denote the degree of  $f$, and  let $C=\{f=0\}\subset \bp^2$ be the projective curve
as above. Recall that $H_1(\bp^2\setminus C,\Z)=\Z^{|\Lambda|}/(d_1,\ldots, d_{|\Lambda|})$. Therefore,
the representation
$$\pi_1(\bp^2\setminus C,*)\stackrel{ab}{\longrightarrow}
H_1(\bp^2\setminus C,\Z)\stackrel{\sigma}\longrightarrow \Z_d,$$
where $ab$ is the Hurewicz epimorphism and $\sigma ([a_1,\ldots,a_{|\Lambda|}])=\sum a_\lambda$,
is well--defined, it is onto, and it provides a cyclic covering $\pi:Y\to \bp^2$, branched along $C$.
Moreover, $F:=\pi^{-1}(\bp^2\setminus C)$ is a smooth affine variety which can
be identified with the open Milnor fiber of $f$, cf. \ref{feature3}.
\ix{d@$d$--covering!of spaces}\ix{Hurewich homomorphism}

The singularities of $Y$ are isolated and are
situated  above the singular points of $C$: if $\{f'_j(u,v)=0\}$
is the local equation of a singular point of $C$, then its
cyclic covering ($d$--suspension) $\{w^d=f'_j(u,v)\}$ is the local equation
of the corresponding singular point in $Y$ above it. Let
$r:\widetilde{Y}\to Y$ be the minimal good resolution of the
singularities of $Y$ and set $\widetilde{\pi}:=\pi\circ r:\widetilde{Y}\to \bp^2$.
Then $\widetilde{Y}$ is smooth and $F=\widetilde{Y}\setminus \widetilde{\pi}^{-1}(C)$
is the complement of the normal crossing curve configuration $\widetilde{\pi}^{-1}(C)$.
 Moreover, the dual graph associated with this  curve configuration is exactly $-G$,
 where $G$ is the graph provided by the discussion
\ref{feature2} (which has the same $g$, $c$ and $\corank$ as the graph
provided by the Main Algorithm). Therefore, for any sufficiently small tubular
neighbourhood $\widetilde{T}$ of
$\widetilde{\pi}^{-1}(C)$, we have $\partial F=-\partial \widetilde{T}$, cf. \ref{feature3}.

The point is that
the natural mixed Hodge structure of $H^*(\partial\widetilde{ T})$
is transported by this isomorphism into the
 mixed Hodge structure of $H^*(\partial F)$. This follows, for example, by analyzing the terms of the exact sequence
 (\ref{eq:MHSmhs}).  Namely, the MHS on the local vanishing cohomology of the
  Milnor fiber $F$ can be identified with Deligne's MHS on the affine smooth hypersurface
 $\widetilde{Y}\setminus \widetilde{\pi}^{-1}(C)$, see for example Example (3.12) of \cite{StOslo}, or \cite{Stwh}.
 Similarly, the local (Steenbrink)
 mixed Hodge structures on $H^*_c(F)$ coincides with Deligne's MHS on  $H^*(\widetilde{Y},\widetilde{\pi}^{-1}(C))$.
 Since the maps $H^2_c(F)\to H^2(F)$ and $H^2(\widetilde{Y},\widetilde{\pi}^{-1}(C))\to
 H^2(\widetilde{Y}\setminus \widetilde{\pi}^{-1}(C))$ can also be identified,
 their cones also agree. This  proves  that the MHS of $H^*(\partial F)$ is the same as the MHS of $H^*(\partial
 \widetilde{T})$. \ix{Deligne}

 This combined with \ref{re:nss}(b)
 (or with  \cite[(6.9)]{ElNe}) shows that this mixed Hodge structure satisfies Conjecture \ref{be:mhs},
 hence  its dual structure in cohomology satisfies Conjecture \ref{be:mhs2}.
\end{proof}\ix{mixed Hodge structure!weight filtration}

Note that the above proof provides an alternative way to show that the  plumbing graph
constructed in \ref{feature2} is indeed a possible graph for $\partial F$.

 \begin{remark}\labelpar{re:MHSnot}
 As we already mentioned, the oriented handle absorption modifies the integers $c(\Gmod)$, $g(\Gmod)$
 and $\corank A_{\Gmod}$, and this can happen even in the homogeneous case. E.g., if $f=z(xy+z^2)$,
 (compare also with  \ref{ss:dsmall}(3c) and  \ref{re:cgdep}):

\begin{picture}(200,50)(-20,-5)
\put(50,15){\circle*{4}} \put(100,15){\circle*{4}}
\qbezier(50,15)(75,30)(100,15) \qbezier(50,15)(75,0)(100,15)

\put(38,15){\makebox(0,0){$0$}} \put(112,15){\makebox(0,0){$3$}}
\put(75,28){\makebox(0,0){$\circleddash$}}
%\put(75,12){\makebox(0,0){$+$}}
 \put(140,15){\makebox(0,0){$\sim$}}

\put(160,15){\circle*{4}} \put(160,25){\makebox(0,0){$3$}}
\put(160,5){\makebox(0,0){$[1]$}}
%\put(180,15){\makebox(0,0)[l]{(oriented handle absorption)}}
\end{picture}

\vspace{3mm}

The integers  $c(\Gmod)$, $2g(\Gmod)$
 and $\corank A_{\Gmod}$ read from the left hand side graph provide the ranks
 $Gr^W_\bullet H_1(\partial F)$, while the right hand side graph does not have this property.
 (In particular, the topological  methods from \ref{ss:XY} or \ref{ss:XAYB} are perfectly suitable
 to determine the oriented 3--manifold $\partial F$,  or even
 the characteristic polynomial of its algebraic Milnor monodromy, but they are not sufficiently fine
 to recover the weight filtration of the mixed Hodge structure of $H_1(\partial F)$.)
 \end{remark}\ix{mixed Hodge structure!weight filtration}

 \begin{example}\labelpar{ex:MHSnot} {\bf The weight filtration of the MHS  of $H_1(\partial F)$ is not a
 topological invariant.}

Consider the homogeneous function $f=xy(xy+z^2)$ of degree $d=4$. Its graph $G$ under
{\it strictly reduced calculus}, that is,  under reduced calculus excluding R5,
cf. \ref{strictly},   can be transformed into
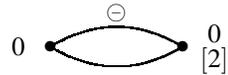
\ix{mixed Hodge structure!weight filtration!not topological}

\begin{picture}(200,60)(-20,-10)
\put(50,15){\circle*{4}} \put(100,15){\circle*{4}}
\qbezier(50,15)(75,30)(100,15) \qbezier(50,15)(75,0)(100,15)

\put(38,15){\makebox(0,0){$0$}} \put(112,20){\makebox(0,0){$0$}}\put(112,10){\makebox(0,0){$[2]$}}
\put(75,28){\makebox(0,0){$\circleddash$}}
%\put(75,12){\makebox(0,0){$+$}}
% \put(140,15){\makebox(0,0){$\sim$}}

%\put(160,15){\circle*{4}} \put(160,25){\makebox(0,0){$3$}}
%\put(160,5){\makebox(0,0){$[1]$}}
%\put(180,15){\makebox(0,0)[l]{(oriented handle absorption)}}
\end{picture}

In particular, $c(G)=1$, $2g(G)=4$ and $\corank A_G=2$, and these numbers equal the ranks of
$Gr^W_\bullet H_1(\partial F)$.

\vspace{2mm}

Next, consider the homogeneous function $f=z(x^4+y^4)$ of degree $d=5$. For its graph $G$ see
the left graph from \ref{ex:almostfan} with  $c(G)=3$, $2g(G)=0$ and $\corank A_G=4$.
In this case these are the ranks of $Gr^W_\bullet H_1(\partial F)$.

\vspace{2mm}

Note that using R3 and R5,  both graphs  can be transformed into
\begin{picture}(40,10)(0,0)
\put(20,4){\circle*{4}} \put(13,4){\makebox(0,0){$0$}}
\put(30,4){\makebox(0,0){$[3]$}}
\end{picture}. Hence, in both cases,
$\partial F$ is orientation preserving diffeomorphic
with the product of $S^1$ with a surface of genus 3.
In particular, the smooth  type of $\partial F$ does not determine the weight filtration
of $H_1(\partial F)$. (This phenomenon might happen with  links of isolated singularities
of higher dimension as well, see \cite{SS}.)

\vspace{2mm}

Additionally, the two characteristic polynomials of the algebraic monodromies acting on
$H_1(\partial F)$ associated with the above two examples are $(t^2+1)^2(t-1)^3$ and
$(t-1)^7$ respectively. Hence, from $\partial F$ one cannot, in general, read  the characteristic
polynomial. (Note that the multiplicity of $f$ cannot be read  either!)
 \end{example}

\begin{example}\label{ex:MILFIB}
In fact, the Milnor fibers of the above two functions treated in \ref{ex:MHSnot} are also
different. Indeed, let $S_g$ denote the oriented closed surface of genus $g$.
Then for  $f=xy(xy+z^2)$ the Milnor fiber is diffeomorphic to
$[0,1]\times S^1\times (S_1\setminus 2 \ \mbox{points})$
(cf. \ref{ss:XY}), while for $f=z(x^4+y^4)$, the Milnor fiber is
 $[0,1]\times S^1\times (S_0\setminus 4 \ \mbox{points})$ (cf. Chapter \ref{s:zf'}).
 Note  that, though these spaces are not homeomorphic,
 in fact, they are homotopic.

\end{example}

\bekezdes \labelpar{s:MHSEIG} The weight filtration on $H_1(\partial F)$ is compatible with the
eigenspace decomposition of the algebraic
monodromy action. Indeed, from previous computations   and from
 \ref{cor:EIG}(a) we obtain  for the generalized eigenspaces the following decomposition.
\ix{generalized eigenspace}
\begin{proposition}
\begin{equation*}\label{eq:mhs4}
\dim \, Gr^W_i(H_1(\partial F)_{\lambda=1})=\left\{
\begin{array}{ll}
\corank A_G=\corank A_{\gc}=|\Lambda|-1 \ \ & \mbox{if $i=-2$}, \\
2g(\gc) & \mbox{if $i=-1$}, \\
c(\gc) & \mbox{if $i=0$}, \end{array}\right.
\end{equation*}
and
\begin{equation*}\label{eq:mhs5}
\dim \, Gr^W_i(H_1(\partial F)_{\lambda\not =1})=\left\{
\begin{array}{ll}
0 & \mbox{if $i=-2$}, \\
2g(G)-2g(\gc)\ \  & \mbox{if $i=-1$}, \\
c(G)-c(\gc) & \mbox{if $i=0$}. \end{array}\right.
\end{equation*}
\end{proposition}

\begin{remark}\label{re:WHG} {\bf The case of weighted homogeneous germs.}

It is a natural task to generalize the above results valid for homogeneous singularities to
 weighted homogeneous non-isolated singularities. In fact, several examples of the present
 book are of this type,
nevertheless, in their study we did not exploit their weighted homogeneous action.
\ix{singularities!weighted homogeneous}\ix{Newton diagram}

The general strategy for the study of these germs,
which exploits their $\bfc^*$--action, is rather straightforward: one has to consider the toric resolution
associated with their Newton diagram. As an intermediate step, one has to consider a non-singular subdivision
of the fan determined by the Newton diagram. This, usually is not unique,
and depends on several choices. Then the identification of the curve  $\C$ and of the decorations of
$\G_\C$ can be done via computations specific to toric resolutions.

We completed this program for several examples presented in this  book,
and using the Main Algorithm and plumbing calculus we verified that
they provide the expected answer. Nevertheless, we were not satisfied with our computations:
the  general statement  formulated  {\it only} in
terms of the Newton diagram and  {\it independently of the choice of the subdivision} is still missing.
Therefore, we decided not to include these results;  they will be completed in the
near future.
\end{remark}

\section{\ Line arrangements}\labelpar{s:LINEARR}
%\section{\ Formulae. Examples}
\setcounter{equation}{0} Assume that
$C$ is a projective line arrangement. We will use the notations of
\ref{arrang}. Recall (cf. \ref{cor:cora}) that
$$\mbox{$|\cala(\gc)|=|\cala(G)|=d$, \ \  $\corank A_G=d-1$, \ and \
$g(\gc)=0$. }$$
Moreover, since
the covering data of $G$ over  $\gc$ is trivial, one has \ix{graph!covering!data}
 $$c(G)=c(\gc)=\sum_{j\in\Pi}(m_j-1) -(d-1), $$
 and by \ref{ex:gcG}
$$2g(G)=\sum_{j\in \Pi}(m_j-2)\big( \mbox{gcd}(m_j,d)-1\big).$$
\ix{singularities!line arrangements}

\begin{bekezdes}\labelpar{ss:arrchar}{\bf Characteristic polynomial.}
Since $\n_w=\n_e=1$ for $w\in \calw(\gc)$ and $e\in\cale_{\calw}(\gc)$,
Corollary  \ref{th:charpolhom} via a computation transforms into\end{bekezdes}
\begin{theorem}\labelpar{th:arrangpol}
Let $\{L_\lambda\}_{\lambda\in\Lambda}$ be an arrangement
with $d=|\Lambda|$ lines and with singular points $\Pi$. Then
the characteristic polynomial of the monodromy acting on $H_1(\partial F)$ is
$$
(t-1)^{|\Pi|}\cdot \prod_{j\in\Pi} (t^{(m_j,d)}-1)^{m_j-2}.$$
\end{theorem}
\begin{bekezdes}{\bf Jordan blocks of  $M_{\Phi,hor}$ and
  $M_{\Phi,ver}$.}\labelpar{ss:arrJordan}
Since $M'_{j,hor}$ is semisimple, from \ref{ss:applJo} we get
$$\#^2_\lambda M_{\Phi,hor}=\#^2_\lambda M_{\Phi,ver}=0 \ \ \mbox{for
  $\lambda\not=1$, and}$$
$$\#^2_1M_{\Phi,hor}=\#^2_1M_{\Phi,ver}=c(\gc).$$
\end{bekezdes}
\ix{monodromy!characteristic polynomial!of $\partial F$}
\ix{monodromy!Jordan block}

\begin{example}\labelpar{ex:genarrr}{\bf The generic arrangement.}\ix{singularities!line arrangements!generic}
Consider the generic arrangement:
all the intersection points of the $d$ lines are transversal. In this case
there are $|\Pi|=d(d-1)/2$ intersection points, all with $m_j=2$. Therefore,
 the characteristic polynomial is just $(t-1)^{|\Pi|}$.\ix{singularities!line arrangements!generic}

A possible  graph for  $\partial F$ can be constructed as follows.

Consider the complete graph $\mathfrak{G}$
 with $d$ vertices (i.e. any two different vertices
are connected by an edge). Decorate all these vertices by $-1$.
Put on each edge $e$ of    $\mathfrak{G}$ a new vertex with decoration $-d$.
In this way, the edge $e$ is `cut' into two edges;
decorate one of them by $\circleddash$ (and do this with all the edges of
$\mathfrak{G}$).

Notice that any $(-1)$--vertex is adjacent to $d-1$ vertices (all decorated
by $-d$). Hence, if $d\leq 3$, this graph is not minimal, but otherwise
in this way we get the `normal form'. The graph has $(d-1)(d-2)/2 $ cycles
and $\corank A=d-1$.
\end{example}

\begin{example}\labelpar{ex:A3arrr}{\bf The ${\mathbf A_3}$ arrangement.}
Consider the arrangement from \ref{ex:a3}. The characteristic polynomial is
$(t^3-1)^4\cdot (t-1)^7$. The graph can be deduced easily from $\gc$, which is
presented in \ref{ex:a3}. There are 4 vertices with $g=1$, 6 cycles and
$\corank A_G=d-1=5$.
%E.g., in the case of (\ref{ex:A3arrr}), $\#^2_1=6$.
 \end{example}\ix{singularities!line arrangements!$A_3$}

\begin{example}\labelpar{ex:fan}{\bf The pencil.}
Assume that all the lines contain a fixed point. (E.g., $f=x^d+y^d$.)
Then the characteristic polynomial is
$(t-1)(t^d-1)^{d-2}$, and  $\Gmod$ is

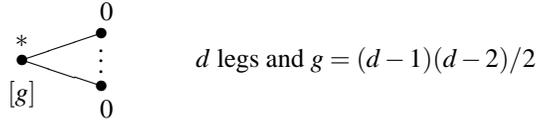
\begin{picture}(140,80)(-30,-10)
\put(30,30){\circle*{4}} \put(60,40){\circle*{4}}
\put(60,20){\circle*{4}}
 \put(30,30){\line(3,1){30}} \put(30,30){\line(3,-1){30}}
\put(30,17){\makebox(0,0){$[g]$}}\put(30,37){\makebox(0,0){$*$}}
\put(62,45){\makebox(0,0)[b]{$0$}}
\put(62,8){\makebox(0,0)[b]{$0$}}
\put(60,25){\makebox(0,0)[b]{$\vdots$}}
\put(160,30){\makebox(0,0){$d$ legs and $g=(d-1)(d-2)/2$}}
\end{picture}

In fact, $\partial F\approx \#_{(d-1)^2}S^2\times S^1$.
\end{example}

\begin{example}\labelpar{ex:almostfan}
Assume that $f=z(x^{d-1}+y^{d-1})$. Then the characteristic polynomial is
$(t-1)^{2d-3}$, and $\Gmod$ is

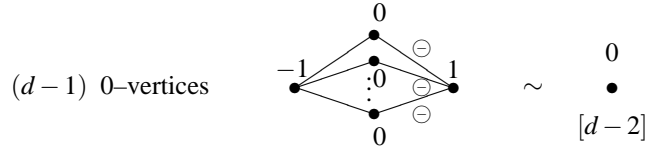
\begin{picture}(140,85)(-100,-10)
\put(30,30){\circle*{4}} \put(60,40){\circle*{4}}
\put(60,20){\circle*{4}}\put(60,50){\circle*{4}}
 \put(30,30){\line(3,1){30}} \put(30,30){\line(3,-1){30}}\put(30,30){\line(3,2){30}}
\put(29,37){\makebox(0,0){$-1$}}
\put(62,30){\makebox(0,0)[b]{$0$}}\put(62,55){\makebox(0,0)[b]{$0$}}
\put(62,8){\makebox(0,0)[b]{$0$}}
\put(58,25){\makebox(0,0)[b]{$\vdots$}}
\put(-40,30){\makebox(0,0){$(d-1)$ \ 0--vertices}}
\put(90,37){\makebox(0,0){$1$}}
\put(90,30){\circle*{4}}
 \put(90,30){\line(-3,1){30}} \put(90,30){\line(-3,-1){30}}\put(90,30){\line(-3,2){30}}

\put(78,45){\makebox(0,0){$\circleddash$}}
\put(78,30){\makebox(0,0){$\circleddash$}}\put(78,20){\makebox(0,0){$\circleddash$}}
\put(120,30){\makebox(0,0){$\sim$}}
\put(150,30){\circle*{4}}\put(150,40){\makebox(0,0)[b]{$0$}}
\put(150,20){\makebox(0,0)[t]{$[d-2]$}}
\end{picture}

\noi In particular, for any oriented surface $S$, the product $S\times S^1$
can be realized as $\partial F$.
\end{example}

\bekezdes\labelpar{ss:TOR} {\bf The torsion of the integral homology $H_1(\partial F,\Z)$.}
In general, even for homogeneous $f$,
$H_1(\partial F,\Z)$ might have nontrivial torsion.
It would be important to characterize completely  this torsion group
in the case of arrangements.
The interest in such a question
is motivated  by a conjecture which predicts
that for arrangements $H_1(F,\Z)$ has no torsion. Hence it is natural to
attack this conjecture via the epimorphism $H_1(\partial F,\Z)\to H_1(F,\Z)$.

\begin{example}\label{ex:TORTOR}
Assume that $f$ is the generic arrangement with $d=4$. Then the torsion part of
$H_1(\partial F,\Z)$ is $\Z_4$.\ix{singularities!line arrangements!generic}

Indeed, by \ref{re:TOR}, the torsion part of $H_1(\partial F,\Z)$ is exactly the torsion part of
$\coker A_G$. Since the intersection matrix $A$ is determined explicitly in \ref{ex:genarrr},\ix{matrix!intersection}
a computation provides the result.

In fact, for generic arrangements and larger $d$, the torsion part of $H_1(\partial F,\Z)$ is even bigger.
One can prove that  $A\otimes \Z_d$ has corank $\geq (d^2-3d+4)/2$, which is considerably larger than
$\corank A=d-1$.

\end{example}

%\setcounter{temp}{\value{section}}
%{\part{More examples.}
%\setcounter{section}{\value{temp}}

\chapter{Cylinders of plane curve
singularities: $f=f'(x,y)$}\labelpar{s:cyl}

\section{\ Using the Main Algorithm. The graph $G$}
\label{ss:cylcyl}\setcounter{equation}{0}\ix{singularities!cylinders}
 Assume that $f(x,y,z)=f'(x,y)$,
as in \ref{cyl}. Assume that $f'$ has $\#=\#(f')$ local irreducible
components, and let $\mu=\mu(f')$ be its Milnor number.\ix{Milnor!number}

For  $g(x,y,z)=z$ a  graph $\gc$ is determined in
\ref{cyl}. In this section we determine the graph $G^m$ using
 the main steps of the collapsing algorithm.
Nevertheless, we provide certain  numerical invariants for the original $G$ as well.

Let

\vspace{3mm}

\begin{picture}(140,50)(-100,-10)
\put(20,30){\circle*{4}} \put(100,30){\circle*{4}}
\put(20,30){\line(1,0){80}} \put(20,35){\makebox(0,0)[b]{$(m_w)$}}
\put(100,35){\makebox(0,0)[b]{$(m_w')$}}
\put(20,15){\makebox(0,0)[b]{$[g]$}}
\put(100,15){\makebox(0,0)[b]{$[g']$}}
\put(20,0){\makebox(0,0)[b]{$w$}}
\put(100,0){\makebox(0,0)[b]{$w'$}}
%\put(60,35){\makebox(0,0)[b]{$x$}}
\end{picture}

\noindent be an edge $e$ of the minimal embedded resolution graph
$\Gamma'=\G(\bfc^2,f')$ of $f'$. Then using the recipe of \ref{cyl}, the
transformation step from \ref{re:w3}, and the algorithm
\ref{algo}, this edge will be replaced in $G$ by
$\n_e=(m_{w_1},m_{w_2})$ strings  with the   decorations:

\vspace{2mm}

\begin{picture}(140,50)(-30,10)
\put(20,30){\circle*{4}} \put(80,30){\circle*{4}}
\put(140,30){\circle*{4}} \put(200,30){\circle*{4}}
\put(260,30){\circle*{4}}  \put(20,30){\line(1,0){240}}
%\put(20,35){\makebox(0,0)[b]{$(1)$}}
\put(80,35){\makebox(0,0)[b]{$0$}}
\put(140,35){\makebox(0,0)[b]{$\frac{m_w+m_w'}{\n_e}$}}
\put(200,35){\makebox(0,0)[b]{$0$}}
\put(110,32){\makebox(0,0)[b]{$\circleddash$}}
\put(20,15){\makebox(0,0)[b]{$[\tilde{g}_w]$}}
\put(260,15){\makebox(0,0)[b]{$[\tilde{g}_{w'}]$}}
\put(170,32){\makebox(0,0)[b]{$\circleddash$}}
%\put(340,35){\makebox(0,0)[b]{$(1)$}}
%\put(60,35){\makebox(0,0)[b]{$x$}}
\end{picture}

\vspace{2mm}

\noindent Compare also with \ref{bad:edge}.

\vspace{2mm}

Above $w$ there are  $\n_w$ vertices in $G$, where $\n_w$ is the
greatest common divisor of $m_w$ and all the multiplicities of the
adjacent vertices of $w$ in $\Gamma'$. Their genus  decoration
$\tilde{g}_w$ is given by
$$\n_w(2-2\tilde{g}_w)=(2-\delta_w)m_w+\sum _{e \ \footnotesize{\mbox{adjacent
to } \ }w}\n_e.$$ There is a similar picture for edges supporting arrowheads.
Note that $$|\cala(G)|=\#,$$ the number of local irreducible components of $f'$.

 In general, $G$ is not  a tree, the number $c(G)$ of its cycles is given by
\begin{equation}\label{NN}
1-c(G)=\sum_{w\in\calw(\Gamma')}\n_w-\sum_{e\in
\cale_w(\Gamma')}\n_e.\end{equation}
 After executing all the
$0$--chain absorptions we get

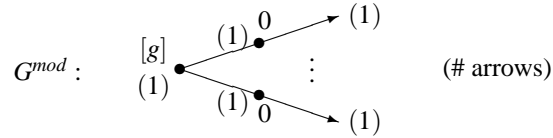
\begin{picture}(140,90)(-100,-10)
\put(30,30){\circle*{4}} \put(60,40){\circle*{4}}
\put(60,20){\circle*{4}}
 \put(30,30){\vector(3,1){60}} \put(30,30){\vector(3,-1){60}}
\put(20,37){\makebox(0,0){$[g]$}}
\put(20,23){\makebox(0,0){$(1)$}}
\put(62,45){\makebox(0,0)[b]{$0$}}
\put(50,42){\makebox(0,0){$(1)$}}
\put(62,10){\makebox(0,0)[b]{$0$}}
\put(50,12){\makebox(0,0)[b]{$(1)$}}\put(-20,30){\makebox(0,0){$G^{mod}:$}}
\put(80,25){\makebox(0,0)[b]{$\vdots$}}
\put(100,50){\makebox(0,0){$(1)$}}
\put(100,10){\makebox(0,0){$(1)$}} \put(150,30){\makebox(0,0){($\#$
arrows)}}
\end{picture}

Here,  the Euler number of the central vertex is
missing, since it is irrelevant. The genus $g$ is determined
as a sum provided by the 0--chain absorption formula:
 $2g=2c(G)+2\sum_{w\in\calw(\Gamma')}\n_w\tilde{g}_w$,
which via the above identities is equal to
$2-\#-\sum_w(2-\delta_w)m_w$. Hence, by A'Campo's formula
(\ref{eq:ACampo})  for  $\mu$, we get \ix{A'Campo's formula}
$$g=(\mu+1-\#)/2.$$

\begin{bekezdes} Clearly, the fiber $F_\Phi$ of $\Phi=(f',z)$
is the same as the Milnor fiber $F'$ of $f'$. Moreover,
the vertical monodromy of $F_\Phi$ is isotopic to the identity
(this follows either from \ref{ex:cylver}, or by the observation
that $\Delta_\Phi=\{c=0\}$, hence $\Phi$ is a  trivial fibration over
$D_{\delta}$). Therefore, from (\ref{prop:gc1}) we
get \begin{equation}\label{EX1} \mu=2g(G)+2c(G)+\#-1.\end{equation}
From the above diagram of $G^{mod}$,  and  by  \ref{le:h-1} we get
\begin{equation*}\begin{split}
\dim\, H_1(\partial
F)&=2g(G^{mod})+c(G^{mod})+\corank A_{G^{mod}}\\
& =(\mu+1-\#)+0+(\#-1)=\mu.\end{split}\end{equation*}
On the other hand, again by \ref{le:h-1} applied for $G$,  we get
\begin{equation}\label{EX2}
\mu=2g(G)+c(G)+\corank A_G.\end{equation} Then, from (\ref{EX1}),
(\ref{EX2}) and  Corollary \ref{cor:Jordan} we obtain
\begin{equation}\label{EX3}
\left\{\begin{array}{l} \corank A_G=c(G)+\#-1,\\
\corank(A,\inc)_G=c(G)+\#.\end{array}\right.\end{equation} This shows
that the upper bounds  from \ref{re:c} can be realized.

In fact, $c(G)$ has an intrinsic meaning (hence, by (\ref{EX3}),
$\corank(A,\inc)_G$ and $\corank A_G$ too): the expression (\ref{NN})
combined with (\ref{CG})
gives that
\begin{equation}\label{UJ:cg}
c(G)=\mbox{the number of 2--Jordan blocks of the
monodromy of $f'$}.\end{equation}
\end{bekezdes}

\begin{bekezdes}\labelpar{ss:partF}%{\bf The Collapsing Algorithm. The graph $\widehat{G}$.}
If the arrowheads and multiplicities of $G^{mod}$  above are deleted,
we get a graph which coincides with
 $\widehat{G}$ from Chapter \ref{ss:ELI}.

 It has  $c(\widehat{G})=0$ and $\corank A_{\widehat{G}}=\#-1$:

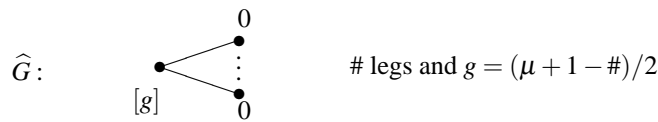
\begin{picture}(140,60)(-50,0)
\put(30,30){\circle*{4}} \put(60,40){\circle*{4}}
\put(60,20){\circle*{4}}
 \put(30,30){\line(3,1){30}} \put(30,30){\line(3,-1){30}}
\put(25,17){\makebox(0,0){$[g]$}}%\put(25,37){\makebox(0,0){$*$}}
%\put(20,23){\makebox(0,0){$(1)$}}
\put(62,45){\makebox(0,0)[b]{$0$}}
%\put(52,42){\makebox(0,0){$(1)$}}
\put(62,10){\makebox(0,0)[b]{$0$}}
%\put(52,12){\makebox(0,0)[b]{$(1)$}}
\put(-20,30){\makebox(0,0){$\widehat{G}:$}}
\put(60,25){\makebox(0,0)[b]{$\vdots$}}
%\put(100,50){\makebox(0,0){$(1)$}}
%\put(100,10){\makebox(0,0){$(1)$}}
\put(160,30){\makebox(0,0){$\#$ legs and $g=(\mu+1-\#)/2$}}
\end{picture}

\ix{connected sum}
By the `splitting operation' R6 of the plumbing claculus, we get another possible non-connected
plumbing graph for
$\partial F$: the  disjoint union of $\mu$ vertices (without any
edges), all of them decorated with Euler number and genus zero. In
other words:
$$\partial F\approx \#_\mu S^2\times S^1.$$ Although this presentation of $\partial F$  can be
deduced by other methods too (see e.g. section
\ref{ss:comgeo}), the graph of $G$ associated with the open book
decomposition of $g=z$, or any graph $G$ computed in this way
associated with any germ $g$,   is a novelty of the present
method.
\end{bekezdes}

\begin{bekezdes}\labelpar{bek:chf}{\bf The characteristic
polynomial.} Since $M_{\Phi,ver}=id$, the characteristic
polynomial $P_{H_1(\partial F)}$ of the Milnor monodromy acting on
$H_1(\partial F)$ is the formula given in  \ref{charpols}(a). It
turns out that this expression coincides with the expression of
the characteristic polynomial of the monodromy of the isolated
plane curve singularity $f'$ provided by A'Campo's formula (\ref{eq:ACampo}).\ix{A'Campo's formula}
\end{bekezdes}

\section{\ Comparing with a different geometric
construction}\labelpar{ss:comgeo}\setcounter{equation}{0}
 Let $F'$ be the Milnor fiber
of $f'$. In the above situation it is easy to see (since
$\Delta_\Phi=\{c=0\}$) that $F\approx F'\times D$, where $D$ is a real
2--disc. In particular, we also have the following geometric
description for $\partial F$:
\begin{equation}\label{eq:pargeo}
\partial F=F'\times S^1\ \cup_{\partial F'\times S^1}
\ \partial F'\times D.
\end{equation}
Using the Mayer--Vietoris exact sequence for this decomposition,
one gets that
\begin{equation}\label{eq:F'PF}
H_1(\partial F,\Z)=H_1(F',\Z).\end{equation}
 Since the
monodromy acts on this sequence, we also obtain that the
monodromy on $H_1(\partial F,\Z)$ is the same as the monodromy
of the plane curve singularity $f'$ acting on  $H_1(F',\Z)$.
In particular, in this way we get the Jordan block structure of
the monodromy acting on $H_1(\partial F,\Z)$ as well.

This shows that,
in general, the {\em algebraic monodromy acting on $H_1(\partial F)$ is
not finite}: take e.g. for $f'$ the germ from \ref{ex:cyln}.

\section{\ The mixed Hodge structure on $H_1(\partial F)$}\labelpar{cyl:MHS}\setcounter{equation}{0}

The isomorphism (\ref{eq:F'PF}) can also be proved as follows.
Here we prefer to  discuss  the cohomological case,
which is more traditional from the point of view of MHS.

Since $H^2_c(F)=H^0_c(F')=0$,
by (\ref{eq:MHSmhs}) the inclusion induces an isomorphism $H^1(F)\to H^1(\partial F)$. Moreover, the inclusion
of $F'$ into $F$ (cut out by $z=0$)
induces also an isomorphism $H^1(F)\to H^1(F')$. Being induced by inclusions,
these isomorphisms preserve the mixed Hodge structures. Therefore, the mixed Hodge structures on
$H^1(F')$ and $ H^1(\partial F)$ coincide.
\ix{mixed Hodge structure!weight filtration}

The mixed Hodge structure of $H^1(F') $  is  Steenbrink's MHS defined  on the
vanishing cohomology of the plane curve singularity $f'$ \cite{StOslo,StHodge}.
 Its weight filtration is compatible with the generalized eigenspace decompositions. According to the general theory, $H^1(F')_{\lambda=1}$ is pure of weight 2, and it has rank $\#-1$.
On the other hand, $H^1(F')_{\lambda\not =1}$, in general, has weights 0, 1 and 2, and the
weight filtration is the monodromy weight of the algebraic monodromy. In particular,
the ranks of $Gr^W_0 H^1(F')_{\lambda\not=1}$ and $Gr^W_2 H^1(F')_{\lambda\not=1}$
are equal, and they agree with the number of
2--Jordan blocks of the monodromy (recall that the monodromy has no 2--blocks with eigenvalue one).
Hence, this  rank is exactly $c(G)$ by (\ref{UJ:cg}).
In particular, the rank of  $Gr^W_0 H^1(F')$ is $c(G)$, while the rank of
$Gr^W_2 H^1(F')$ is $c(G)+\#-1=\corank A_G$, cf. (\ref{EX3}).
Hence, by dimension computation, we get that the remaining
$Gr^W_1 H^1(F')$ has rank $2g(G)$. This supports Conjecture \ref{be:mhs2}:
\ix{generalized eigenspace}

\begin{corollary}  If $f$ is a cylinder of an isolated  plane curve singularity
then the  mixed Hodge structure of  $H^*(\partial F)$ satisfies Conjecture
\ref{be:mhs2}.
\end{corollary}

\vspace{2mm}

Note that the information about the ranks of $Gr^W_\bullet H^1(F')$ cannot be read from $\widehat{G}$:
the graph  $\widehat{G}$ contains information about $\#$ and $\mu$ only.
\ix{mixed Hodge structure}

\chapter{Germs $f$ of type $zf'(x,y)$}\labelpar{s:zf'}

\section{\ A geometric representation of $F$ and $\partial F$}\labelpar{ss:GEOM}\setcounter{equation}{0}
In this section we assume that $f(x,y,z)=zf'(x,y)$, where $f'$ is an
isolated plane curve singularity.

Similarly as in the case of cylinders (or, in the case of `composed  singularities'
\cite{NPhg}), consider the ICIS $\bar{\Phi'}:(\bfc^3,0)\to (\bfc^2,0)$ given by
$\bar{\Phi'}=(f'(x,y),z)$. By similar notations as in Chapter \ref{s:ICIS}, the Milnor fiber $F$ of $f$ is
$$F=B_\epsilon^3\cap \bar{\Phi'}^{-1} (\{cd=t\}\cap D_\eta^2),$$
where $0<t\ll \eta\ll \epsilon$. For any $\eta>0$ consider the disc $D_\eta:=\{c:\, |c|\leq \eta\}$.
Then, by isotopy, the above representation of $F$ can be transformed into
$$F=B_\epsilon^3\cap \bar{\Phi'}^{-1} (D_\eta\times \{t\}\setminus D^\circ _{\eta'}\times \{t\}),$$
for some  $0<\eta'\ll\eta$. From this  the variable $z$ can be eliminated, that is, if $B^2_\epsilon$ is
the $\epsilon$--ball in the $(x,y)$--plane, then
\begin{equation}\label{eq:MFIB}
F=B_\epsilon^2\cap (f')^{-1} (D_\eta\setminus D^\circ _{\eta'}).
\end{equation}
It can also be verified that the monodromy on $F$ is isotopic to the identity, since  it is the rotation
by $2\pi$ of the annulus  $D_\eta\setminus D_{\eta'}$.

In particular, the homotopy type of $F$ is the same as the homotopy type of the complement of the link of
$f'$ in $S^3_\epsilon$.

\vspace{2mm}

(\ref{eq:MFIB}) provides a geometric picture for $\partial F$ as well,
and proves that its monodromy action is trivial.

\bekezdes
Projecting $D_\eta\setminus D_{\eta'}$ to $S^1=\partial D_\eta$,
and composing with $f'$ we get a map $\partial F\to S^1$ which is a locally trivial fibration.
Hence, $$\mbox{$\partial F$ fibers over $S^1$.}$$

\begin{bekezdes}\labelpar{bek:plumf}
From (\ref{eq:MFIB}) one can read the following plumbing representation for $\partial F$.
In order to understand the geometry behind the statement, let us analyze different constituent
 parts of  $\partial F$. Notice that $B_\epsilon^2\cap (f')^{-1}(\partial D_\eta)$ can be identified with
 the complement of the link of $f'$ in $S^3$. Similarly, $B_\epsilon^2\cap (f')^{-1}(\partial D_{\eta'})$
 is the same, but with opposite orientation. Moreover, they are glued together along their
 boundaries in a natural way.
 Therefore, the plumbing graph can be constructed as follows.

 Take the (minimal)  embedded resolution graph of $(\bfc^2,V_{f'})$. This has $\#$ arrows, where $\#$
 is the number of irreducible components of $f'$.  Keep all the Euler numbers and
  delete all the multiplicity decorations. Let the  schematic form of the result
  be the next `box':

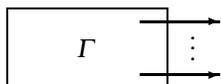
 \begin{picture}(100,60)(-100,-10)
\put(10,25){\vector(1,0){30}}\put(10,5){\vector(1,0){30}}
\put(-40,0){\framebox(60,30){$\Gamma$}}
\put(30,18){\makebox(0,0){$\vdots$}}
\end{picture}

\vspace{2mm}

\noindent Then a possible plumbing graph for $\partial F$ is:

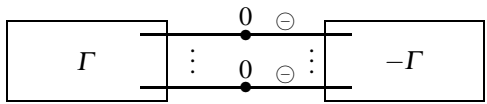
\vspace{2mm}

 \begin{picture}(100,60)(-100,-10)
\put(10,25){\line(1,0){80}}\put(10,5){\line(1,0){80}}
\put(50,25){\circle*{4}}
\put(50,5){\circle*{4}}
\put(-40,0){\framebox(60,30){$\Gamma$}}\put(80,0){\framebox(60,30){$-\Gamma$}}
\put(50,12){\makebox(0,0){$0$}}
\put(50,32){\makebox(0,0){$0$}}
\put(75,18){\makebox(0,0){$\vdots$}}
\put(30,18){\makebox(0,0){$\vdots$}}
\put(65,30){\makebox(0,0){$\circleddash$}}
\put(65,10){\makebox(0,0){$\circleddash$}}
\end{picture}

\vspace{2mm}

Here, in $-\Gamma$, we change the sign of all Euler numbers and edge--decorations, that is,
 we put $\circleddash$ on all edges.

Notice that any 3--manifold $M$ obtained in this way is
orientation preserving diffeomorphic with the manifold obtained by changing its orientation:
$M\approx  -M$.

The  statement about the above shape of the graph
 can be tested using the examples listed in \ref{ss:PROD} as well.
\end{bekezdes}

\begin{example}\labelpar{ex:ARA} The first (and simplest)  example,
when $f'=x^{d-1}+y^{d-1}$, hence $f$ defines an arrangement,  was already  considered in
\ref{ex:almostfan}. Its graph produced by the Main Algorithm is
the left diagram of \ref{ex:almostfan} supporting
(together with the characteristic polynomial computation) the above statements.
\end{example}

\begin{example}\label{ex:zf'2}
Assume that $f=z(x^2+y^3)$ and take $g$ to be a generic linear form.
The graph $\gc$ is given in \ref{ex:zf'}.
Running the algorithm and reduced calculus, a possible $\Gmod$ is

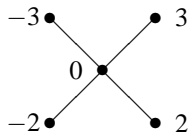
\begin{picture}(100,70)(-100,-10)
\put(0,5){\circle*{4}}\put(40,5){\circle*{4}}
\put(0,45){\circle*{4}}\put(40,45){\circle*{4}}
\put(20,25){\circle*{4}}
\put(0,5){\line(1,1){40}}\put(0,45){\line(1,-1){40}}
\put(-10,5){\makebox(0,0){$-2$}}\put(-10,45){\makebox(0,0){$-3$}}
\put(50,5){\makebox(0,0){$2$}}\put(50,45){\makebox(0,0){$3$}}
\put(10,25){\makebox(0,0){$0$}}
\end{picture}

\noi which is compatible with the predicted form from \ref{bek:plumf}.\end{example}

\begin{example}\labelpar{re:SEIFERT}\ix{Seifert manifold}
More generally,  if $f'=x^a+y^b$, then in the graph of  \ref{bek:plumf},
all the 0--vertices can be eliminated by 0--chain or handle absorption, hence
we get a star--shaped graph with four legs and central vertex with genus
$\mbox{gcd}(a,b)-1$. In particular, $\partial F$ has a Seifert structure.
\end{example}

\begin{remark}\label{re:zf'2}
Let us consider again the ICIS $\Phi=(f,g)$ as in the original construction.\vspace{2mm}

(a) This family might also serve as a testing family for some of the characteristic polynomial formulae
\ref{charpols},
regarding the vertical monodromies $M'_{j,ver}$ of the transversal types
(which, in fact, are much harder to test).

Indeed, if $f=zf'$ and $g$ is a generic linear function, then $\Sigma_f$ has two components.
 $\Sigma_1=\{x=y=0\}$ with  $d_1=1$, and
 $\Sigma_2=\{z=f'=0\}$. Let us concentrate on the first component $\Sigma_1$.
   The transversal type $T\Sigma_1$ is the same as  the type of
   $f'$ (in two variables). Hence,  the Milnor fiber $F'_1$ of the transversal type can be identified
with the Milnor fiber of $f'$.

Since $\{f'e^{it}=\delta\}=\{f'=\delta e^{-it}\}$, we get that
the vertical monodromy of $F'_1$ coincides with the inverse of the
Milnor monodromy acting on the Milnor fiber of $f'$. Hence, the
characteristic polynomial of $M'_{1,ver}$,  determined by \ref{charpols}(c),  should coincide with the
characteristic polynomial of $f'$ provided by the classical A'Campo formula (\ref{eq:ACampo}).
The interested reader is invited to verify this on all the available graphs $\gc$.\vspace{2mm}
\ix{A'Campo's formula}

(b) Let us also test  the invariants of the fiber of $\Phi$. Let us write $g$ as $z+g'(x,y)$, where
$g'$ is a generic linear form with respect to $f'(x,y)$. Then, solving the system $zf'(x,y)=c$ and
$z+g'(x,y)=d$, we get the fiber $F_\Phi$ of $\Phi$. Eliminating $z$ we get $(d-g'(x,y))f'(x,y)=c$, hence
the fiber $F_\Phi$ is the same as the Milnor fiber of the plane curve singularity $g'f'$.

For example, in the case of $f'=x^2+y^3$, whose  graph $\gc$ is given in \ref{ex:zf'},
the first formula of (\ref{eq:H1Phi}) provides for $H_1(F_\Phi)$ the rank 5, which is the Milnor number of
$y(x^2+y^3)$.

It is interesting to identify and analyze the vertical monodromy of $F_\Phi$ via the local
 equation $(\eta e^{it}-g'(x,y))f'(x,y)=\delta$ in two variables with $0<\delta \ll\eta\ll\epsilon$.

Note that the present strategy provides a method for the study of the monodromy of such a deformation
for an arbitrary plane curve singularity pair $(f',g')$:  only has to be computed the graph
$\G_\C$ for $(zf',z+g')$. \ix{deformation!of curves}
\end{remark}

\begin{bekezdes}\labelpar{comments}
While  computing  the plumbing graph
of $\partial F$ for this family $f=zf'$ via the Main Algorithm,  the following amazing fact emerged:
although the graph
$\gc$ shows no symmetry, after running the algorithm and calculus the output graph has
the symmetry (up to orientation) predicted in \ref{bek:plumf}.
For example, looking at the starting graph \ref{ex:zf'},  we realize absolutely no symmetry,
nor even the hidden potential presence of  symmetry.  Indeed,  the two parts of $\G_\C$, which  provide
$\G$ and $-\G$ of \ref{bek:plumf} respectively, are rather different;
they codify two different geometric situations.
Nevertheless,  after calculus we get the two symmetric parts represented by $\Gamma$ and $-\Gamma$.

To construct a resolution $r$ which produces $\Gamma$ (the embedded resolution of a plane curve singularity)
is rather natural (see e.g. the the case
of cylinders), but to get  a resolution $r$ which
(or part of it)  produces $-\Gamma$ in a natural way,  is very
tricky. But, in fact, this is what  $\gc$ does: a part of it produces $\Gamma$, another part
of it produces $-\Gamma$. This anti--duality is still a mystery for the authors.
%(i.e. if we have, roughly speaking, a resolution with dual
%graph $\Gamma$, then in some geometric situations resolving some other
%singularities, we produce another resolution with dual graph $-\Gamma$).
This shows  how hard it is to recognize and follow  global geometric
properties by   manipulating (local) equations/resolutions.
\end{bekezdes}

\chapter{The $T_{*,*,*}$--family}\labelpar{s:zgf'}

\section{\ The series $T_{a,\infty,\infty}$}\labelpar{Taii}\setcounter{equation}{0}
We start our list of  examples with the series associated with
$T_{\infty,\infty,\infty}$, where $T_{\infty,\infty,\infty}$ denotes the germ
$f=xyz$. This germ  was already clarified
either as an arrangement, see \ref{ss:dsmall}(3e),  or by the algorithm of Chapter \ref{s:zf'}, hence
$\partial F=S^1\times S^1\times S^1$.

 Next we  consider the
series $T_{a,\infty,\infty}$ with one--parameter $a\geq 2$ given by $f=z^a+xyz$.
For $g$ we take the generic linear function.

The case $a=2$ can be rewritten as $z^2=x^2y^2$ and its graph $\gc$ is given in \ref{ex:AB}, Case 1.
 The case $a=3$ is treated in \ref{ss:dsmall}(3c), the general  $a>3$ in \ref{ZFG}.
The Main Algorithm  provides for the boundary of the
Milnor fiber $\partial F$ the plumbing graph

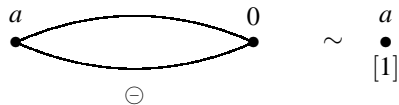
\begin{picture}(110,60)(-50,20)
% kozepso sor csucsai +sulyok
\put(50,40){\circle*{4}} \put(140,40){\circle*{4}}
\put(50,50){\makebox(0,0){$a$}}
\put(140,50){\makebox(0,0){$0$}}

% osszekoto ivek + sulyok:
\qbezier(50,40)(95,60)(140,40)
\qbezier(50,40)(95,20)(140,40)
%\put(95,60){\makebox(0,0){+}}
\put(95,20){\makebox(0,0){$\circleddash$}}

 \put(190,40){\circle*{4}}
\put(190,30){\makebox(0,0){$[1]$}}\put(190,50){\makebox(0,0){$a$}}

\put(170,40){\makebox(0,0){$\sim$}}

\end{picture}

\section{\ The series $T_{a,2,\infty}$}\labelpar{Tabi}\setcounter{equation}{0}

Running the Main Algorithm and calculus for the graphs of \ref{a2infty} and \ref{a2double},
we get for $\partial F$  the plumbing graph:

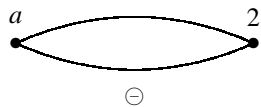
\vspace{2mm}

\begin{picture}(100,70)(-50,0)
% kozepso sor csucsai +sulyok
\put(50,40){\circle*{4}} \put(140,40){\circle*{4}}
\put(50,50){\makebox(0,0){$a$}}
\put(140,50){\makebox(0,0){$2$}}

% osszekoto ivek + sulyok:
\qbezier(50,40)(95,60)(140,40)
\qbezier(50,40)(95,20)(140,40)
%\put(95,60){\makebox(0,0){+}}
\put(95,20){\makebox(0,0){$\circleddash$}}

\end{picture}

\noindent
In fact,  $\partial F$ for $T_{a,b,\infty}$ ($f=x^a+y^b+xyz$) is given by the plumbing graph:

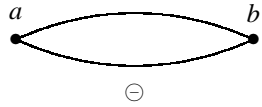
\vspace{2mm}

\begin{picture}(100,70)(-50,0)
% kozepso sor csucsai +sulyok
\put(50,40){\circle*{4}} \put(140,40){\circle*{4}}
\put(50,50){\makebox(0,0){$a$}}
\put(140,50){\makebox(0,0){$b$}}

% osszekoto ivek + sulyok:
\qbezier(50,40)(95,60)(140,40)
\qbezier(50,40)(95,20)(140,40)
%\put(95,60){\makebox(0,0){+}}
\put(95,20){\makebox(0,0){$\circleddash$}}

\end{picture}

 By a different geometric argument this was verified
by the first author's student in \cite{Ba}.
For $a=b=3$ it follows  from \ref{ex:dminusz2} or \ref{ss:dsmall}(3b).

\chapter{Germs $f$ of type $\tilde{f}(x^ay^b,z)$. Suspensions}\label{s:fAB}
In this chapter we again provide an alternative way to identify the boundary of the Milnor fiber
for a special class of germs.\ix{singularities!suspensions}

\section{\ $f$ of type $\tilde{f}(xy,z)$}\labelpar{ss:XY}\setcounter{equation}{0}

First, we sketch  a geometric construction which provides $\partial F$ and its monodromy, provided
that $f(x,y,z)=
\tilde{f}(xy,z)$, where $\tilde{f}$ is an isolated plane curve singularity in variables $(u,z)$.

We will use the following notations: $\tilde{F}=\{\tilde{f}=\eta\}$ is the Milnor fiber
and  $\tilde{\mu}$ the Milnor number of $\tilde{f}$.  Define
$I$ as zero if $u$ is a component of $\tilde{f}$, otherwise $I$ denotes the intersection multiplicity
$(\tilde{f},u)_0$ at the origin. Set $\Delta:=\tilde{F}\cap \{u=0\}$. Then, clearly, $|\Delta|=I$.

Next,
consider the Morse singularity $h:(\bfc^2,0)\to (\bfc,0)$, $h(x,y)=xy$. Then the Milnor fiber $F$ of $f$
can be reproduced via the projection $\bar{\Phi}:(\bfc^3,0)\to (\bfc^2,0)$ given by
$(x,y,z)\mapsto (h(x,y),z)=(u,z)$. Indeed, $\bar{\Phi}$ maps  $F$  onto
$\tilde{F}$.
The fibers of $\bar{\Phi}$ are as follows:  above generic points of $\tilde{F}$ we have the Milnor fiber
$F_h=S^1\times [0,1]$ of $h$, while over the special points $P\in \Delta$ we have the contractible central
fiber $D \vee D$ of $h$ (here $D$ is the real 2--disc).

Hence, $\partial F$ decomposes into two parts, one of them is
$(S^1\sqcup S^1)\times \tilde{F}$, while the other is $\bar{\Phi}^{-1}(\partial \tilde{F})$, which is a locally trivial fiber bundle over $\partial \tilde{F}$ with fiber $F_h=S^1\times [0,1]$.
The monodromy of this bundle is given by the variation map of $\bar{\Phi}$, hence it is the composition of $I$ Dehn twists
(corresponding to the variation maps around the Morse points $P_i$). Therefore, if we define the `double of $\tilde{F}$' as $d\tilde{F}:=
\tilde{F}\sqcup_{\partial \tilde{F}} (-\tilde{F})$, then $\partial F$ is an $S^1$--bundle over
$d\tilde{F}$: two trivial copies of $S^1\times \tilde{F}$ are glued above
$\partial \tilde{F}$ so that the Euler number of the resulting $S^1$--bundle is $I$.

Notice that the genus of $d\tilde{F}$ is exactly $\tilde{\mu}$, hence:

\begin{picture}(20,55)(-130,-25) \put(100,4){\circle*{4}}
\put(100,-6){\makebox(0,0){$[\tilde{\mu}]$}}\put(100,12){\makebox(0,0){$I$}}
\put(-0,4){\makebox(0,0){$\partial F$ {\em is given by the plumbing graph }  \ }}
\end{picture}

In particular, any $S^1$--bundle with arbitrary genus and non--negative Euler number
can be realized as $\partial F$ (and any such bundle might have many realizations by rather different singularities).

It is interesting to note that even if $\tilde{f}$ has many Puiseux pairs (hence its embedded resolution graph has many rupture vertices), $\partial F$ is still Seifert ---, in fact, it is  an $S^1$--bundle.\ix{Seifert manifold}

\begin{example}\label{ex:AXYZ}  Assume that  $f=z^a+xyz$ with $a\geq 2$ as in \ref{Taii}.
 Then $\tilde{F}=\{z^a+uz=\eta\}$ ($0<\eta\ll 1$)
 and $\Delta=\{u=0,\ z^a=\eta\}$. Then   $\tilde{F}$, via the projection $(u,z)\mapsto z$,
is diffeomorphic with the annulus  $A:=\{x: \eta_1\leq |z|\leq \eta_2\}$,
where $\eta_1<\sqrt[a]{\eta}<\eta_2$, and by this identification
its special points from $\Delta$ are
$\cup_{i=1}^aP_i=\{z^a=\eta\}$.  Notice that $d\tilde{F}$ is a torus. Moreover,
the monodromy is the rotation of $A$,
hence it is isotopic to the identity.  In particular, we get (a fact compatible with \ref{Taii}):

\begin{picture}(20,30)(-130,-10) \put(20,4){\circle*{4}}
\put(20,-6){\makebox(0,0){$[1]$}}\put(20,12){\makebox(0,0){$a$}}
\put(-50,4){\makebox(0,0){$\partial F$ {\em is the torus bundle} }}
\put(90,4){\makebox(0,0){{\em with trivial monodromy.} }}
\end{picture}

\end{example}

\begin{bekezdes} More generally, in the situation of \ref{ss:XY},
the monodromy on $\partial F$ is induced by the monodromy of $\tilde{f}$. Hence,
if $P_{\partial F}$ (resp. $P_{\tilde{f}}$) denotes the characteristic polynomial of the algebraic monodromy acting on $H_1(\partial F)$ (resp. $H_1(\tilde{F})$), then
\begin{equation}\label{eq:CHPOL}
P_{\partial F}(t)=\big( P_{\tilde{f}}(t)\big)^2\cdot (t-1)^{1-sign(I)}
\end{equation}
where $sign(I)$ is 1 for $I>0$ and zero for $I=0$.

The above result can be applied for $f=xyz$, $f=z(z^2+xy)$, $f=xy(z^2+xy)$ and
$f=z^a+x^ky^k$. These cases can be compared with \ref{ss:dsmall}(3e), (3c), (4b') or (4c').
Moreover, (\ref{eq:CHPOL})
can also be compared with the formula from Theorem
\ref{th:charpolG} valid for the characteristic polynomial.
%But, definitely, the interested reader might take an arbitrary germ $\tilde{f}$ to test the result.

\end{bekezdes}
\begin{bekezdes} The above construction provides the structure of the Milnor fiber $F$ as well:
one may  extract from it key information, such as  the  Euler characteristic,
zeta--function of the monodromy, etc.
\end{bekezdes}

\section{\ $f$ of type $\tilde{f}(x^ay^b,z)$}\labelpar{ss:XAYB}
\setcounter{equation}{0}
Assume now that $a$ and $b$ are two positive relative prime integers.
Then, the discussion of \ref{ss:XY} can be modified as follows.

Fix any isolated plane curve singularity $\tilde{f}(u,z)$, and replace $h$ by $u=h(x,y):=x^ay^b$ in order to get $f:=\tilde{f}\circ \bar{\Phi}$. Then all the arguments of \ref{ss:XY} remain valid with the following modification.

The difference is that in the case $h=xy$, $\bar{\Phi}$ restricted on $\partial F$ is a trivial fibration
above any small  neighbourhood of any point of $\Delta$. In the new situation of $h=x^ay^b$,
above a small neighbourhood of a  point $P$ of $\Delta$, $\bar{\Phi}|_{\partial F}$ is {\it not} a
fibration; in fact,  in $\partial F$  two special $S^1$--fibers appear above $P$.
One of them has the same  local neighbourhood (Seifert structure) as $\{x=0\}$ in  $S^3=\{|x|^2+|y|^2=1\}$
with $S^1$--fibers/orbits cut out by the family $x^ay^b=$constant; the other has local behaviour as
$\{y=0\}$ in the same space. These two local Seifert neighbourhoods in $S^3$ are well--understood, they are
guided by the continued fractions of $a/b$, respectively $b/a$. \ix{Hirzebruch--Jung continued fraction}
\ix{Seifert manifold}

Therefore, we get the following result:

\begin{theorem}\labelpar{prop:XAYB}Let $\tilde{f}(u,z)$ be an isolated plane curve
singularity, and $a$ and $b$ two positive relative prime integers. Then $\partial F$ associated with
$f=\tilde{f}(x^ay^b,z)$ is a Seifert 3--manifold whose minimal star--shaped plumbing graph
can be constructed as follows:\ix{Seifert manifold}

\vspace{2mm}

(i) The central vertex has genus $\tilde{\mu}$, the Milnor number of $\tilde{f}$, while its  Euler number is $I$,
where $I$ is zero if $u|\tilde{f}$, otherwise $I=(\tilde{f},u)_0$.

\vspace{2mm}

(ii) Let $a/b=[p_0,p_1,\ldots,p_s]$, respectively $b/a=[q_0,q_1,\ldots,q_t]$ be the
Hirzebruch--Jung continued fraction expansions \ix{Hirzebruch--Jung continued fraction}
 of $a/b$ and $b/a$ with $p_0,q_0\geq 1$, $p_i\geq 2$ for $i\geq 1$ and
$p_j\geq 2$ for $j\geq 1$. Then the graph has $2I$ legs (with all vertices
having genus--decoration zero). $I$ strings have length $s$, the vertices are
decorated by Euler numbers $p_1,\ldots, p_s$
 such that the vertex decorated by $p_s$
is the one glued to the central vertex; while the other $I$ legs are strings decorated by Euler numbers
$q_1,\ldots, q_t$, and the $q_t$--vertex is the one glued to the central vertex.

If $b=1$ then the first set of $I$ legs is empty, if $a=1$ then the second set of $I$ legs is  empty.
(If $a=b=1$ then we recover \ref{ss:XY}, in which case the graph has no legs.)

\vspace{2mm}

(iii) The orbifold Euler number of the Seifert 3--manifold is $\frac{I}{ab}\geq 0$.

\vspace{2mm}

(iv) The characteristic polynomial of monodromy action on $H_1(\partial F)$ is
$$P_{\partial F}(t)=\big( P_{\tilde{f}}(t)\big)^2\cdot (t-1)^{1-sign(I)}.$$
\end{theorem}
\begin{proof}
We need to prove only {\it (iii)}.
For this use  the fact that for any $a$ and $b$ one has:
$$\frac{[p_1,\ldots,p_{s-1}]}{[p_1,\ldots,p_s]}+\frac{[q_1,\ldots,q_{t-1}]}{[q_1,\ldots, q_t]}=1-\frac{1}{ab}.$$
This and  (\ref{eq:eorb1}) show that $e^{orb}(-\partial F)=-I/ab$. Then use (\ref{eq:eorb2}).
\end{proof}
\ix{Seifert manifold}
\ix{orbifold Euler number}

\begin{example}\label{ex:MPW}
If we take $\tilde{f}=z^n+u^d$, then $f=z^n+x^{da}y^{db}$, $\tilde{\mu}=(n-1)(d-1)$ and $I=n$.
In this way we recover the main result of \cite{MPW}.

See also \ref{ex:347b} for an explicit example determined via $\gc$.

The reader can also verify that the graphs $\gc$ from \ref{ss:ds} produce compatible answers
via the Main Algorithm. The next remark emphasizes certain  advantages of the Main Algorithm.
\end{example}

\begin{remark}\label{re:AB}
Let us assume that $f=x^{2n}y^{2m}+z^2$
as in \ref{ex:AB}, Case 1. Set $d:={\rm gcd}(m,n)$.
Since $\G_\C$ is unicolored, \ix{graph!unicolored}
the boundary $\partial F$ and its characteristic polynomial can be determined in  two ways, either by
\ref{th:charpolG}, or by the above topological theorem \ref{prop:XAYB}.
Both theorems provide for the characteristic polynomial
$$P_{\partial F}(t)=(P_{\tilde{f}}(t))^2=\frac{(t^{2d}-1)^2(t-1)^2}{(t^2-1)^2}.$$
Nevertheless, the Main Algorithm  also gives that $\#^2_1M_{\Phi,ver}=1$ provided that $g$
is the generic linear form,
and also the following data about the graph $G$:
$$\corank A_G=c(G)=1, \ \ \mbox{and} \ \ 2g(G)=4(d-1).$$
This data provides not only $\rank H_1(\partial F)=2g(G)+c(G)+\corank A_G=4d-2$,  but also the ranks of its
 `weight' decomposition.

Note that in this case, the mixed Hodge structure on $H^1(\tilde{F})$ (where $\tilde{F}$ is the Milnor fiber of
$\tilde{f}$ as above) has two weights, the 1--dimensional eigenspace $H^1(\tilde{F})_1$ has weight 2, while
 $H^1(\tilde{F})_{\lambda\not=1}$ has weight 1.\ix{mixed Hodge structure!weight filtration}

 On the other hand, the Conjecture (\ref{eq:mhs2}) regarding the weight decomposition
  of  $H^1(\partial F)$  predicts  a similar decomposition compatible with the eigenspace
  decomposition of the monodromy: in case of eigenvalue 1 weights 0 and 2 both survive in rank one,
  while $H^1(\partial F)_{\lambda\not=1}$ has weight 1 and rank $4d-4=2g(G)$.

\vspace{2mm}

Similar study can be done for all the examples of \ref{ss:ds}.

\end{remark}

\begin{remark}\label{re:212v}
In Theorem \ref{prop:XAYB}, if the total number of legs is
less than two and $\tilde{\mu}=0$, then we get a lens space.
In particular the graph (a posteriori) will have no central vertex.

For example, take $z^2=xy^2$. Then $\tilde{f}= z^2-u$, hence $I=2$ and $\tilde{\mu}=0$. Thus we have to glue to a vertex (with Euler number 2) two other vertices, both decorated by 2. Hence, $\partial F$ is the lens space $L(4,1)$, which can also be represented by a unique vertex decorated by $-4$,  cf.  \ref{221b}.

\end{remark}

{\part{What next?}

\chapter{Peculiar structures on $\partial F$. Topics for future research}\labelpar{ss:PECUL}

In this chapter we list some topics that are closely related with the oriented
3--manifold $\partial  F$, and are natural extensions  of the present work.
With this we plan to generate some research in this direction.

We  omit definitions and  do not strive for a comprehensive treatment of the subjects
involved. %These can easily be found  by the interested reader.
We simply  wish to arouse interest and generate further research by
pointing  out some new phenomena  generated by the `new'
manifold $\partial F$ exhibited.

\section{\ Contact structures}\labelpar{CONTACT}\setcounter{equation}{0}
 Recently, there is an intense activity in the
theory of {\it contact structures of 3--manifolds}, see e.g. \cite{OS}.
\ix{contact structure}\ix{Stein fillings}\ix{contact structure!tight}
From the point of view of complex geometry,
 a central place is occupied by links of normal surface singularities.
 This targets   the classification of their (tight)  contact structures and the
classification of the corresponding {\it Stein fillings}. In the case of normal/isolated
complex surface singularities,
the analytic structure of the singularity induces a {\it canonical } contact structure on the link.
Moreover, all the {\it Milnor open book decompositions}
(that is,  open book decompositions associated with analytic map--germs)
support (in the sense of Giroux, cf. \cite{Gi})\ix{Giroux}
 exactly this canonical contact structure,  for details, see \cite{CNP}.
 In fact, \cite{CNP} also proves that
this canonical contact structure can be recovered from the topology of the link, and can be topologically
identified among all the contact structures.
Any resolution of the singularity (with perturbed  analytic structure)
 appears as a natural Stein filling of the canonical  contact structure. Furthermore,
if the singularity is smoothable, then all the Milnor fibers (smoothings) appear as natural Stein fillings
of this contact structure.

One may ask the validity of similar properties  in the present situation, that is,
 for $\partial F$, where $F$ is the Milnor fiber
  associated with a non-isolated singularity $f:(\bfc^3,0)\to (\bfc,0)$.

Although, in this case,
the link is not smooth, hence from the point of view of the theory of 3--manifolds it is not interesting,
we can concentrate on the boundary of the Milnor fiber.  As it was emphasized in several places in the body
of this book, this class of 3--manifolds grows out from the class of
 negative definite plumbed 3--manifolds \ix{graph!negative definite}
(although we do not know a precise characterization of it).

For any such $\partial F$, we can ask about the classification of its (tight)
contact structures. Moreover, the Milnor fiber, as a Stein manifold with boundary, induces a contact
structure on $\partial F$. A crucial question is to  characterize and identify
this structure  among all the contact structures. Note that all the open book
decompositions of $\partial F$ cut out by  germs $g$ considered in this article
(namely when the pair $(f,g)$ is an ICIS) support the very same contact
structure induced by $F$, cf. \cite{CC}. Hence this structure too should have  some universal property.

The point we wish to emphasize is that
the present new geometric situation (considering non--isolated $f$ instead of isolated singularities)
introduces a {\it new set of contact structures} together with a new set of {\em  Stein fillings}
 realizable by singularity theory. Their classification is of major interest.

 For example, consider
  a 3--manifold which can appear  as $\partial F$ for some non--isolated $f$ as in the present  work,
 and also as a singularity link  (that is,  it can be represented by a connected
 negative definite plumbing graph). In such a case,
  the contact structure induced by the Milnor fiber $F$ of $f$, let us call it
 {\it Milnor fiber contact structure},
 and the canonical contact structure (as singularity link) are, in general, not contactomorphic.

 The simplest proof of this statement is by comparison of  the Chern classes of the corresponding
 structures in $H_1(\partial F,\Z)=H^2(\partial F,\Z)$, a fact already noticed in \cite{CC}.
 The Chern class of the Milnor fiber contact structure is zero,
 since  $F$ is   parallelizable. On the other hand, the Chern class of the canonical contact structure is the
 class of the canonical cycle  in $H_1(\partial F,\Z)$ (that is, the restriction of
 the canonical class of a resolution to its boundary). In particular, the Chern class of the
 restriction of the canonical class  to $\partial F$  is zero
 if and only if the singularity/link is numerically Gorenstein
 (see e.g. \cite{Five} for the terminology). Therefore, if $\partial F$
 can be realized by a negative definite graph which is not numerically
  Gorenstein, then the two contact structures are different. This is
  happening in the case of all lens spaces which are not $A_n$--singularity links,
 and also in \ref{bek:classifi2}, cases (e)-(f).
\ix{Chern class} \ix{numerically Gorenstein} \ix{canonical class}\ix{contact structure}

\section{\ Triple product. Resonance varieties}\labelpar{ss:TRP}\setcounter{equation}{0}

 \bekezdes  {\bf Triple product.} \ For any oriented 3--manifold $Y$, the {\it cohomology  ring}
$H^*(Y,\Z)$ of $Y$ carries rather subtle information. This can be reformulated in the {\em triple product},
induced by the cup product $\mu_Y\in \Lambda^3H'$ given by $\mu_Y(a,b,c)=\langle a\cup b\cup c,[Y]\rangle$
for $a,b,c\in H^1(Y,\Z)$ (and $H'$ is the dual of $H^1(Y,\Z)$).
\ix{triple product}

Sullivan in \cite{SU} proved that for any pair $(H,\mu)$ (where $H$ is a finitely generated free abelian group, and $\mu\in \Lambda^3H'$), there exists a 3--manifold $Y$ with $H^1(Y)=H$ and $\mu_Y=\mu$.
Moreover, for any singularity link, $\mu$ is trivial. In fact, by the proof of Sullivan,
if $Y$ can be represented by a plumbing graph with non--degenerate intersection form, then
the cup--product $H^1(Y)\otimes H^1(Y)\to H^2(Y)$ is trivial.
\ix{Sullivan} \ix{Mark}

In our situation, however, $Y=\partial F$ might have non-trivial $\mu_Y$; see e.g.
the example of $S^1\times S^1\times S^1$ realized by $f=xyz$. It would be very interesting to determine
the triple product for all the 3--manifolds appearing as $\partial F$, and connect it with other singularity invariants.
The same project can be formulated for the related numerical invariant introduced by T. Mark in \cite{Mark}.

For partial results see \ref{cor:EIG} and \ref{cor:EIGGG}. For results regarding the  ring structure
 for Seifert and graph 3--manifolds, see  e.g.  \cite{Aa}.

\bekezdes {\bf Resonance varieties.}
 The $d$--th {\em resonance variety} of a space $X$ is the set
${\mathcal R}_d(X)$ of cohomology classes $\lambda\in H^1(X,\bfc)$ for which there is a subspace
$W\subset H^1(X,\bfc)$, of dimension $d+1$, such that $\lambda \cup W=0$. The resonance varieties of arrangements were introduced by Falk in \cite{FALK}, and since then they play a central  role in several parts of mathematics.
%Their relationship with characteristic varieties is crucial, see e.g. \cite{Suciu} and the references therein.
\ix{resonance variety}\ix{Falk}

Since the cup--product on $H^1(\partial F)$ might be non--trivial, and rather subtle,
it would be of major interest to determine the resonance varieties of $\partial F$ and connect them
with other singularity invariants.

\section{\ Relations with the homology of the Milnor fiber}\labelpar{ss:MILNOR}\setcounter{equation}{0}
\bekezdes
From the cohomology long exact sequence of the pair $(F,\partial F)$, one gets a monomorphism
$$H^1(F,\Z)\hookrightarrow H^1(\partial F,\Z),$$
which is compatible with the monodromy action. In particular, we get an upper bound
 for the first Betti number  of $F$: $\rank H_1(F)\leq \rank\,H^1(\partial F)$.
In fact, the characteristic polynomial of the monodromy acting on
$H^1(F)$ divides the characteristic polynomial of $H^1(\partial F)$, which is expressed
 combinatorially  in  \ref{ss:arrchar}.

This is important, since, in general, the behaviour of  $\rank H_1(F)$ can be rather involved, mysterious.
Even in the case of arrangements (that is, in the `simplest homogeneous case'), it is not known whether
$\rank H_1(F)$ can  be deduced from the combinatorics of the arrangement  or not;
see for example \cite{BDS,S3,LIB} and the references therein.
Note that by our algorithm, $\partial F$ is deduced from the combinatorics of
the arrangement, hence we get a combinatorial upper bound for $H_1(F)$ as well.

In fact, one can determine an even  better bound. Notice that $a\cup b\cup c=0$ for any
$a,b,c\in H^1(F)$. Therefore, $H^1(F)$ should be injected  in such a subspace of $H^1(\partial F)$
on which the restriction of the triple product vanishes.

For homogeneous singularities, \ref{s:MHS2} combined with known facts regarding the
mixed Hodge structure of $H^1(F)$ might produce even stronger restrictions.
Recall that in the homogeneous case,  $H^1(F)_{\lambda\not=1}$ is pure of weight 1, while
$H^1(F)_{\lambda=1}$ is pure of weight 2.\ix{mixed Hodge structure!weight filtration}

\begin{bekezdes}  From a different point of view, more in the spirit of \ref{CONTACT},
the Milnor fiber $F$ is a Stein filling of $\partial F$. One of the most
intriguing  questions is if  this Milnor fiber can   be characterized universally from
$\partial F$ (eventually under some restrictions on $\partial F$).

%For example, the possible Stein filling of lens spaces are classified, and they can be uniquely
%identified by their Chern classes in

%\marginpar{Classification of\\  Stein fillings with \\ $c_1=0$}

The Milnor fiber also might help in  the classification of the possible boundaries as well.
 For example, we have the following statements:
\end{bekezdes}

\begin{proposition} Let $f:(\bfc^3,0)\to (\bfc,0)$ be a hypersurface singularity with Milnor fiber $F$.

\vspace{1mm}

(a) If $F$ is a rational ball (that is, $\widetilde{H}^*(F,{\mathbb Q})=0$)
then $f$ is smooth.

(b)  If the boundary $\partial F$ of the Milnor fiber is $S^3$ then $f$ is
smooth.
\end{proposition} \begin{proof}
(a) follows from A'Campo's theorem \cite{AC2}, which says that for $f$ non--smooth
 the Lefschetz number of the monodromy is zero. Part (b) follows from (a) and a theorem of
Eliashberg, which says that the {\it only} Stein filling of $S^3$ is the ball \cite{El}.
\ix{A'Campo}
\end{proof}
\ix{sphere!$\partial F$} \ix{Eliashberg}
This result can be compared with the
celebrated theorem of Mumford which states that if the link of a normal surface singularity is
$S^3$ then the germ is smooth, and also with the
famous conjecture of L\^e D\~ung Tr\'ang and M. Oka, which predicts  that
if the {\it link} of the hypersurface germ
$f$ with 1--dimensional singular locus
is homeomorphic to $S^3$, then $V_f$ is an equisingular family of
irreducible plane curves.\ix{equisingularity}
\ix{L\^e} \ix{Oka} \ix{Mond}

\section{\ Open problems}\labelpar{OPEN-PROB}\setcounter{equation}{0}

Some general questions/open problems, or natural tasks for further study:
 \\

\begin{bekezdes}\labelpar{O1}
 Determine/characterize all oriented plumbed  3--manifolds which might appear as
$\partial F$ for some non--isolated hypersurface  singularity $f$.
(Recall that  the classification of  those normal surface
singularity  links which might appear as
hypersurface singularity or complete intersection links is also open.)

\end{bekezdes}

\begin{bekezdes}\labelpar{O6}
 Classify all lens spaces realized as $\partial F$. Classify all Seifert manifolds
realized as $\partial F$.
\end{bekezdes}

\begin{bekezdes}\labelpar{O9}
 Find a $\partial F$ (with $f$ a non--isolated singularity)
which is an integral homology sphere (cf. also  Question 3.21 of \cite{Si3}).

Note that although we have the criteria from \ref{re:SiRa}, the structure of monodromy operators is
so rigid, that simultaneous realizations of those unimodularity properties  of the operators is
seriously obstructed. This open problem might lead to some compatibility conditions connecting these
operators too.
\end{bekezdes}

\bekezdes Consider germs of map  $h:(\bfc^2,0)\to (\bfc^3,0)$, and
let $V_f$ be its image. Connect the invariants
of the present book with the invariants of $h$ as they are treated, for example, in the series of
articles of D. Mond, see e.g. \cite{Mond}.

\bekezdes Similarly, compare the analytic invariants and the more algebraic study of non--isolated
surface singularities, as it is presented for example in \cite{deJong,Pe,vanSt}, with the invariants of the
present book.

\begin{bekezdes}\labelpar{JJ}
Determine the Jordan block--structure of the algebraic monodromy acting on $H_1(\partial F)$.
\end{bekezdes}

\begin{bekezdes}\labelpar{O8}
  Develop the mixed Hodge structure of $H^1(\partial F)$. Prove Conjecture \ref{be:mhs} from
 Chapter  \ref{s:MHS}.  (For the case of homogeneous germs and cylinders, see \ref{s:MHS2} and \ref{cyl:MHS}.)
\end{bekezdes}\ix{mixed Hodge structure!weight filtration}

\begin{bekezdes}\labelpar{O7}
Determine the graph $G^\circ $  for which $c(G^\circ)$ is minimal among all $c(\Gmod)$, where
$\Gmod\sim G$. Has $G^\circ$ got any intrinsic significance? Has $c(G^\circ)$ got any intrinsic significance?
(For example, is  $c(G^\circ)$ independent of $g$ ? In what situations
 is this minimum realized by $\widehat{G}$ ?)
\end{bekezdes}

\begin{bekezdes}\labelpar{O3}
 Determine completely the ring structure  of $H^*(\partial F)$ (that is, the triple product on $H^1(\partial F)$),
 and the resonance varieties, cf. Chapter  \ref{ss:PECUL}.
\end{bekezdes}

\begin{bekezdes}\labelpar{O4}
 Can $H^1(F)$ be determined from $\partial F$ (at least in particular cases, say,
for arrangements)? Or, from $G$ of $\gc$ used here?
Understand the monomorphism from \ref{ss:MILNOR} better.
(Recall, that it is a famous conjecture
for arrangements,  that $\rank H^1(F)$ is determined combinatorially.)
\end{bekezdes}

\begin{bekezdes}\labelpar{O5}
 In the case of arrangements, can the combinatorics of the arrangement be recovered from
$\partial F$? (We conjecture that yes.)
\end{bekezdes}

\begin{bekezdes}\labelpar{O2} \ix{contact structure}
 Classify/characterize  all the `Milnor fiber contact structures' induced on
3--manifolds realized as $\partial F$, cf. \ref{CONTACT}. Are there natural
families for which one can classify all the
Stein fillings of the Milnor fiber (or all the)
contact structures ? Find examples when the Milnor fiber contact
structure (besides the Milnor fiber) has  other Stein fillings as well.
\end{bekezdes}

\begin{bekezdes}\labelpar{O10}
 Compute the Seiberg--Witten invariants, or more generally,
the Heegaard Floer homologies (or generalized versions of the
lattice homologies) for those 3--manifolds $\partial F$
which are rational homology spheres. How are they related to the
signature of $F$?  Is there any analogue of the Seiberg--Witten Invariant Conjecture
of the first author and Nicolaescu, cf.  \cite{NeNi,Ninv}?
\end{bekezdes}
\ix{Seiberg--Witten invariants}\ix{Heegaard Floer homology}
\ix{Nicolaescu}

\begin{bekezdes}\labelpar{O11}
 Clarify the open question  \ref{OP1} (i.e. determine the resolution of
the real analytic variety
$\{(u,v,w)\in (\bfc^3,0)\, :\, u^mv^{m'}w^n=|w|^{k}\}^+$).
\end{bekezdes}

\begin{bekezdes}\labelpar{O12}
 Establish the relationship between the present work (which determines
the boundary $\partial F$) and \cite{eredeti}, which determines the links $K(f,g)_k$ of the
Iomdin series $f+g^k$ ($k\gg 0$). In what sense is $\partial F$ the `limit'
of $\{K(f,g)_k\}_k$ ?\ix{Iomdin series}
\end{bekezdes}

\begin{bekezdes}\labelpar{O13}
 We conjecture that there exists some kind of   rigidity property
 restricting the pairs $(\partial_1F,\partial_2F)$ (cf. \ref{felo})
which form  together  a possible $\partial F$. That is, we expect that
the normalization has some effect on
the possible types of transversal singularities, and vice versa.
\end{bekezdes}

\begin{bekezdes}\labelpar{O14}
Determine $\partial F$ for any quasi--ordinary singularity $f$ in terms of the
characteristic pairs of $f$. For some related homological results see \cite{GPN,KMCE}.
\end{bekezdes}

\begin{bekezdes}\labelpar{O14Sie}
Determine the variation operator $VAR^{III}$ of Siersma \cite{Si,Si3}
in terms of $\G_\C$ in the spirit of the present work.
\end{bekezdes}\ix{Siersma}

\begin{bekezdes}\labelpar{O14tor}
Find a nice formula for the torsion of $H_1(\partial F,\Z)$,  at least for  arrangements.
\end{bekezdes}

\begin{bekezdes}\labelpar{ONewton}
Determine $\partial F$ for weighted homogeneous or Newton non-nondegenerate singularities in terms
of their Newton diagram. For the isolated singularity case, see Oka's algorithm \cite{Oka}.
See also \ref{re:WHG}.
\end{bekezdes}\ix{Oka}

\begin{bekezdes}
Develop the  analytic aspects and study the
analytic invariants related to   $\partial F$; moreover,
analyze  also  its relations with deformation theory (for this last subject
  see for example the thesis of D. van Straten or T. de Jong and their series of articles
 \cite{deJong,deJongvanSt,vanSt}).
\end{bekezdes}

\begin{bekezdes}
Are  the techniques and results of the present book
applicable in the context of  the equisingularity problems of non--isolated germs?
 (For such problems, consult  for example the article of F. de
 Bobadilla \cite{deB2} containing several key conjectures as well).
\end{bekezdes}
\ix{Straten@van Straten}\ix{Jong@de Jong} \ix{Bob@de Bobadilla} \ix{deformation!theory}\ix{equisingularity}

\vspace{2mm}

Questions related mostly to  the technicalities used in the proofs and results:

\vspace{2mm}

\begin{bekezdes}\labelpar{Oa1}
 Develop the `calculus' of decorated graphs such as  $\gc$, cf.
\ref{re:CONJ}.
\end{bekezdes}

\begin{bekezdes}\labelpar{Oa2}
  Analyze the possible relations connecting the decorations of $\gc$.
Provide an independent proof of the fact that the expressions \ref{charpols}(c) are
independent of $g$.
\end{bekezdes}

\begin{bekezdes}\labelpar{Oa3}
 Find closed formulae for $\corank A$ and  $\corank(A,\inc)$.
In what situations can we expect the validity of the relation
 $\corank(A,\inc)_{\widehat{G}}=c(\widehat{G})+|\cala(G)|$ (or similar formula for
 $G^\circ$ instead of $\widehat{G}$) ?
\end{bekezdes}

\begin{bekezdes}\labelpar{Oa4}
  Is it true that $\#^2_1 M_{\Phi,ver}=c(\widehat{G})$ ?
Or for $G^\circ$ instead of $\widehat{G}$ ?
\end{bekezdes}

\begin{bekezdes}\labelpar{Oa5}
  Is the technical lemma \ref{lem:unitw} true in general ?
\end{bekezdes}

\begin{bekezdes}\labelpar{Oa6}
  Discuss the case of {\em all} eigenvalues of the vertical monodromy
$M_{\Phi,ver}$.
\end{bekezdes}

%\setcounter{temp}{\value{section}}
%{\part{List of examples and notations.}
%\setcounter{section}{\value{temp}}

\chapter{List of examples}\label{s:LISTEX}

\
\setlength{\extrarowheight}{1ex}

 \begin{tabular}{lll}
 plane curve singularities  & \hspace{5mm} & \ref{ex:planecurves}, \ \ref{1.6}, \ \\
 cyclic coverings   & & \ref{1.6}, \ \ref{ss:b}, \                                          \\
Hirzebruch--Jung singularities &&  \ref{1.6}, \                                  \\
Brieskorn singularities && \ref{8.7}, \                                         \\

Seifert 3--manifolds && \ref{bek:SEIFERT}, \ \ref{ex:347b}, \ \ref{ss:dsmall}, \ref{ss:unicusp} \\
 &&  \ref{re:SEIFERT}                                             \\
homogeneous singularities && \ref{homgen}, \ \ref{ex:ACirr},  \ \ref{ex:homch}, \ \ref{re:c}, \\
 && \ \ref  {bek:bounD}, \ \ref{re:REMARK}, \ Ch. \ref{s:HOMOGEN}, \
                               \\
homogeneous with small degree &&  \ref{ss:dsmall} \\
rational unicuspidal curve,  one Puiseux pair && \ref{ss:unicusp} \\
cylinders of  plane curve singularities && \ref{cyl}, \   \ref{re:c}, \  \ref{bek:bounD}, \  \ref{re:REMARK}\\
 && \ref{ex:cylver}, \  Ch. \ref{s:cyl}       \\
type $f=zf'(x,y)$ &&  \ref{ss:PROD}, \   Ch. \ref{s:zf'}                                \\
suspensions $f=f'(x,y)+z^k$  && \ref{ss:ds}, \                                     \\
Key Example && \ref{keyex} \\
$(\{x^2+y^7-z^{14}=0\},0)$, $f_1=z^2$, $f_2=z^2-y$ & & \ref{3.14}, \ \ref{5.9}, \\
$(\{x^2+y^7-z^{14}=0\},0)$, && \\ \hspace{2cm}
$f_1=z^2+y^2$, $f_2=z^2-y+y^2$ && \ref{3.17}, \  \ref{5.10}, \ \ref{rem:HOM}, \ \\
$(\{z^2+(x^2-y^3)(x^3-y^2)=0\},0)$,  && \\ \hspace{2cm} $f_1=x^2$, $f_2=x^2-y^3$ && \ref{3.16},  \   \\

\end{tabular}

 \begin{tabular}{lll}

 $(\{z^2+(x^2-y^3)(x^3-y^2)=0\},0)$,  && \\
\hspace{2cm} $f_1=x^2+y^4$, $f_2=x^2-y^3+y^4$ && \ref{3.18}, \  \ref{5.12}, \  \\

$A_3$ arrangement $f=xyz(x-y)(y-z)(z-x)$ && \ref{ex:a3}, \ \ref{ex:A3arrr} \\
generic arrangement && \ref{ex:genarr}, \  \ref{ex:genarrr}, \ \ref{ex:TORTOR} \\
$T_{a,2,\infty}$ family $f=x^a+y^2+xyz$ && \ref{a2double}, \ \ref{xyz3}, \ \ref{xyz5}, \ \ref{Tabi}, \\
&& \ref{ex:charpol3}    \\
$T_{a,\infty,\infty}$ family $f=x^a+xyz$ && \ref{bek:ZFG}, \ \ref{Taii}, \  \ref{ex:AXYZ} \\
$f=\tilde{f}(x^ay^b,z)$ && \ref{ss:XY}, \ \ref{ss:XAYB}   \\

 \end{tabular}

 \begin{tabular}{lll}

$f=x^2y^2+z^2(x+y)$ && \ref{nu2}, \  \ref{rm:dj}, \ \ref{ex:nu2b}, \
\ref{UUUU}, \\ &&   \ref{nu2b},  \ \ref{ex:charpol4}, \                              \\
$f=y^3+(x^2-z^4)^2$ && \ref{ex:ketA2}, \    \ref{ex:ketA2b}, \ \ref{ex:ketA2b'}, \ \ref{ex:charpol2}, \         \\
$f=x^3y^7-z^4$  && \ref{ex:347}, \ \ref{ex:347b'}, \  \ref{ex:347bal}, \
             \ref{ex:347bal2}, \\ &&  \ref{ex:347bal3}, \ \ref{ex:347b}, \  \ref{ex:charpol1}, \
              \ref{ex:MPW}, \                           \\

$f=z^d-xy^{d-1}$ && \ref{ex:dd-1}, \                                              \\
$f=x^2y^2+y^2z^2+z^2x^2-2xyz(x+y+z)$ && \ref{ex:3A2}, \ \ref{ss:dsmall}(4a)                    \\
$f=x^d+y^d+xyz^{d-2}$ && \ref{ex:loop}, \ \ref{ex:dminusz2}, \ \ref{ex:DDD} \\
$f(x,y,z)=f'(x,y)=(x^2-y^3)(y^2-x^3)$ && \ref{ex:cyln},                       \\
$f=z(x^2+y^3)$ && \ref{ex:zf'}, \ \ref{ex:zf'2}, \  \ref{ex:zf'2}, \ref{re:zf'2}(b) \  \\
$f=x^2y+z^2$   && \ref{221}, \ \ref{221b},  \ \ref{ex:charpol4}, \  \ref{re:212v}  \\
$f=x^ay^b+z^2$  && \ref{ex:AB}, \ \ref{re:AB}, \ \ref{prop:XAYB}, \ \ref{re:AB}                         \\
$f=x^{md}y^{nd}+z^d$  && \ref{mdndd}, \  \ref{ex:MPW}\\
$f=x^ay(x^2+y^3)+z^2$  && \ref{AnPi}, \ \ref{AnPi2} \\

$f=x^5+y^2+xyz$ && \ref{xyz5}, \ \ref{UUUU}, \    \ref{ex:charpol3}                                  \\
$f=x^3+y^2+xyz$ && \ref{xyz3}, \ \ref{ex:Uu}, \ \ref{a2double}, \ \ref{ex:d1d2}, \\ &&  \ref{UUUU}, \
          \ref{ex:charpol3}, \ \ref{re:UUU}                                     \\
$f=z(xy-z^2)$ && \ref{ss:dsmall}(3c), \ \ref{re:MHSnot},    \\
$f=xy(xy+z^2)$ && \ref{ex:MHSnot}, \ \ref{ex:MILFIB} \\
$f=z(x^4+y^4)$ && \ref{ex:MHSnot}, \ \ref{ex:MILFIB} \\
$f=x^d+y^d$   && \ref{ex:fan}   \\
$f=z(x^{d-1}+y^{d-1})$ && \ref{ex:almostfan}, \ \ref{ex:ARA} \\
$f=z(x^a+y^b)$ && \ref{re:SEIFERT}
 \end{tabular}

\chapter{List of notations}\label{s:LISTNOT}

\
\setlength{\extrarowheight}{1ex}

\noindent\begin{tabular}{p{2.5cm}>{\small}p{7cm}>{\hspace{1mm}}p{2.5cm}}
{\bf Symbol:}& {\bf Description:} & {\hspace{-1mm}\bf Appears in:}\\
 & & \\
%%%%%%%%%% symbols and A through C
$(a,b)$ & ${\rm gcd}(a,b)$ & \\
$\chi$ & Euler characteristic of a space & \\
$[k_1,k_2,\ldots,k_s]$ & continued fraction & \ref{re:HJCF}\\
$\circleddash$, +, - &  edge decorations of a plumbing graph & \ref{bek:pl}\\
$\#^k_\lambda$ &  number of Jordan blocks of size $k$ with eigenvalue
$\lambda$ & \ref{geneig} \\
$\#$, $\#(f)$ & number of local irreducible components of plane curve singularity &
\ref{ex:planecurves}, \ref{ss:cylcyl}\\
$\#T\Sigma_j$ &number of local irreducible components of $T\Sigma_j$ & \ref{ss:2.0b}\\
 & & \\
$A$ & intersection matrix & \ref{def:incidence}\\
$(A,\inc)$ & block matrix & \ref{def:incidence}\\
$(\alpha_\ell,\omega_\ell)_\ell$ & Seifert invariants &  \ref{bek:SEIFERT}\\
${\cala}(\G)$, ${\cala}$ &  the set of arrowhead vertices of a graph $\Gamma$ &  \ref{links}\\
$\arg$ & $\arg(g)=g/|g|$ & \ref{be:fibr}\\
$\arg_*$ & induced in homology by $\arg$  & \ref{3.9}\\
 & & \\
$B_\epsilon$ & closed ball of $\epsilon$ radius in $\bfc^n$  &  \ref{ss:2.0a}  \\
& & \\
$c(\G)$ &  the number of independent cycles of a graph $\G$ & \ref{bek:RESCAL} \\
$C_\Phi$ &  critical locus of an ICIS $\Phi$ & \ref{ex:ICIS} \\
\end{tabular}

\noindent\begin{tabular}{p{2.5cm}>{\small}p{7cm}>{\hspace{1mm}}p{2.5cm}}${\C}$  &  the `special curve arrangement' in $\de$ & \ref{gc}, \ref{zart2}\\
 ${\csj}$, ${\csi}$  &   collections of certain irreducible components of ~$\C$ & \ref{def:csj}\\
$C=\cup_{\lambda\in\Lambda}C_\lambda$ & projective curve & \ref{homgen}\\
$(C_{j,i},p_j)_{i\in I_j}$ & local analytic irreducible components of $(C,p_j)$ & \ref{homgen}\\
$c$ & identification map $c:\cup_jI_j\to \Lambda$ & \ref{homgen}\\
& & \\
$d_\lambda$ & degree of $C_\lambda$, $\lambda\in\Lambda$ & \ref{homgen}\\
$Div(H;M)$ & the divisor of a characteristic polynomial & \ref{ss:chars}\\
$D_r$  & complex disk of radius $r$  &  \ref{ss:2.0a}    \\
$D_r^2$, $D^2_\eta$  & bidisc  &  \ref{ss:2.0a}, \ref{ss:ICIS}  \\
${\de}$  & total transform of $V_{fg}$ in an embedded resolution & \ref{gc}\\
 ${\dec}$, ${\ded}$,  ${\deo}$   &  collections of certain components  of ${\de}$& \ref{gc} \\
$\delta(f)$ & Serre--invariant (or delta--invariant) of $f$& \ref{ex:planecurves}\\
$\partial $ & homological operator & \ref{COMPAR}\\
$\delta_j'$ & Serre--invariant of  $T\Sigma_j$ & \ref{ss:2.0b} \\
$d_j$ & covering degree of $g|_{\Sigma_j}$ & \ref{bek:dj}\\
$d(e)$ & invariant  of a cutting edge &  \ref{bek:ce1}\\
$\Delta_\Phi$, $\Delta_j$ & discriminant of an ICIS $\Phi$; its irreducible components & \ref{ex:ICIS} \\
   ${\bd}$  & covering data, \ {\it see } $(\bn,\bd)$ too& \ref{def:2.3.2}\\
   $\delta_w$ & the number of vertices adjacent to a vertex $w$ & \ref{rem:HOM}\\
$\partial F$, $\partial F_{\epsilon,\delta}$ & boundary of the Milnor fiber & \ref{intro}, \ref{ss:2.0a}\\
$\partial _1F$, $\partial _2F$ & decomposition  subsets of $\partial F$ & \ref{intro}, \ref{ss:2.1} \\
$\partial _{2,j}F$  & connected components of $\partial _2F$ & \ref{ss:2.1} \\
$\overline{\partial _{2,j}F}$ & canonical closure of $\partial _{2,j}F$ & \ref{rem:closure}\\
$\partial' \Phi^{-1}(D_\delta)$ & subset of $\partial F_{\epsilon,\delta}$ & \ref{rem:phif}, \ref{felbo}\\
& & \\
 $\epsilon_e$ & edge decoration & \ref{bek:MULT}\\
$e_w$ & Euler number decoration  of  $E_w$ & \ref{bek:MULT}, \ref{resgraph}\\
$e^{orb} $ & orbifold Euler number & \ref{bek:SEIFERT}\\
${\evg}$,  ${\cale}$ & the set of edges of a graph $\Gamma$ &  \ref{bek:pl}\\
$\cale_{cut}$ & the set of cutting edges of $\Gamma$ & \ref{ss:parcut}\\

$\cale_\calw(\G)$ &  the set of edges connecting non--arrowhead vertices & \ref{cale_calw}\\
\end{tabular}

%%%%%%%%%%%% starting "D"
\noindent\begin{tabular}{p{2.5cm}>{\small}p{7cm}>{\hspace{1mm}}p{2.5cm}}
$F$, $F_{\epsilon,\delta}$  & Milnor fiber of $f$ & \ref{ss:2.0a}\\
$F_\Phi$, ${\fcd}$  & Milnor fiber of an ICIS $\Phi$ &  \ref{milnorfibre}\\
$F_j'$ & Milnor fiber of transversal type singularity $T\Sigma_j$ & \ref{ss:2.0b} \\
${\wfp}$  & lifted Milnor fiber & Ch. \ref{ss:1.3}\\
${\wfv}$,  ${\wfw}$  & subsets of ${\wfp}$ near curves $\C_v$ or $\C_w$ & Ch. \ref{ss:1.3}\\
${\wfW}$  & collection of subsets ${\wfw}$ & \ref{diagram}\\
$\Phi$ & good representative of an ICIS & \ref{gc} \\
${\wP}$  & a lift of $\Phi$ & \ref{gc}\\
$\varphi_j$ & Fibonacci numbers & \ref{bek:classifi}\\
& & \\
$g_w$ & genus of  $E_w$ & \ref{resgraph}\\
$d_\lambda$ & degree of $C_\lambda$, $\lambda\in\Lambda$ & \ref{homgen}\\
$g(\G)$ & the sum of all genera in a graph  $\G$ & \ref{bek:RESCAL}\\
$\G_1+\G_2$ & disjoint union of graphs & \ref{bek:PlCon} \\
$\gc$ & the dual graph of $\C$ & \ref{gc}, \ref{summary}\\
${\gce}$, ${\gck}$ & `complementary' subgraphs of $\gc$ &  \ref{ss:parcut}\\
${\gj}$   & connected components of ${\gck}$ & \ref{def:csj}\\
$\gj/\sim$ & a `base graph' derived from ${\gj}$   & \ref{back}\\
${\gbk}$  & a subgraph of $\gc$ having vanishing 2-edges only & \ref{bad:graph}\\
$\G(X)$ & resolution graph of  a normal surface singularity
$X$ & \ref{1.4}, \ref{ss:a}\\
$\G(X,f)$ & embedded resolution graph of $X$ and a germ $f$ & \ref{resgraph} \ref{ss:NSS}\\
$\widehat{\gc}$& an undecorated  graph obtained from $\gc$ & \ref{algoim}\\
$\G(f',g')$ & shorthand for $\Gamma(\bfc^2,f'g')$ & \ref{ss:ds} \\
$G$, $G_1$, $G_2$ & graphs derived from $\gc$, $\gce$, $\gck$ using the  Main Algorithm & \ref{def:GGm} \\
$G_{2,j}$ & components of the graph $G_2$; related to $\G^2_{\C,j}$ & \ref{2} \\
$\overline{G_{2,j}}$ & canonical closure of $G_{2,j}$ & \ref{rem:closure} \\
$\Gmod$, $\Gemod$,  $\Gkjmod$ & graphs obtained from $G$, $G_1$, $G_2$ using plumbing calculus & \ref{def:GGm} \\
$G(K_\ell)$ & subgraphs of ${\gj}$ & \ref{bek:G(K)}\\
$\gcce$ & an embedded resolution graph derived from ${\gce}$ & \ref{bek:G1C}\\
$G(T\Sigma_j)$ & embedded resolution graph of $T\Sigma_j$ & \ref{back}\\
$G(X,f)$ & universal cyclic covering graph of $\G(X,f)$ &  \ref{ss:a}\\
$\widehat{G}$ & output graph of Collapsing Main Algorithm & \ref{algoim}\\
$\widehat{G_1}$, $\widehat{G_2}$ & graphs associated with $\partial_1 F$ and $\partial_2F$; outputs of Collapsing Main Algorithm & \ref{GHALG} \\
\end{tabular}

%%%%%%%%%%%%%%%%%% starting "G" Part I.
\noindent\begin{tabular}{p{2.5cm}>{\small}p{7cm}>{\hspace{1mm}}p{2.5cm}}
${\fedog}$  & equivalence classes of covering graphs of $\Gamma$ associated with covering data $(\bn,\bd)$ & \ref{eq:2.3.3}\\
$\Gamma_1\sim\Gamma_2$ & equivalent graphs & \ref{def:sim}\\
$-\G$ & $\G$ with reversed orientation (Euler and edge decorations) & \ref{bek:-G}\\
$|Gr|$ & topological realization of the graph $Gr$ & \ref{rem:HOM}(3) \\
$\gamma_w$, \ $\gamma_a$ & oriented meridian & \ref{Openbooks}, \ref{bek:zeroarrowhead}\\
$Gr^W_mH$ & $W_mH/W_{m-1}H$ associated with a weight filtration & \ref{ss:MHS}\\
& & \\
 $H_{M,\lambda}$ &     generalized $\lambda$--eigenspace    &  \ref{ss:chars}\\
  & & \\
   $\inc$ & incidence  matrix & \ref{def:incidence}\\
   $I$ & identity matrix & \\
    & & \\
   $K$, $K_X$ & the link of a surface singularity $X$  &  \ref{ss:2.0a}\\
$K_f\subset K_X$ &the link $V_f\cap K_X\subset  K_X$ of a germ $f$ defined on $X$ & \ref{3.6}\\
$K^{norm}$ & link of the normalization $V_f^{norm}$ & \ref{d1}   \\
& & \\
$L$, $L_j$ & singular part of the link $K$; its components & \ref{ss:2.0b}\\
 $\Lambda(m;n,\nu)$    &   a special divisor associated with $(m;n,\nu)$    &  \ref{ss:chars} \\
 $\cup_{\lambda\in\Lambda}L_\lambda$ & line arrangement  & \ref{arrang}\\
& & \\
$M(\G)$, $M$  & oriented plumbed 3-manifold associated with $\G$ & Ch. \ref{s:PLU}\\
$-M$ & the oriented manifold $M$ with reversed orientation & \ref{LE}\\
$M_1\#M_2 $ & oriented connected sum of $M_1$ and $M_2$ & \ref{bek:PlCon}\\
$m_{geom,\Phi}$ & geometric monodromy representation of an ICIS $\Phi$& \ref{milnorfibre} \\
$M_{\Phi}$  & algebraic monodromy  representation of an ICIS $\Phi$& \ref{milnorfibre} \\
$M_q$, \ $M$ & a monodromy operator & \ref{re:SiRa}, \ref{rem:HOM}, Ch. \ref{s:vh}\\
$m_{j,hor}'$, $m_{j,ver}'$  & horizontal, vertical geometric monodromies of $T\Sigma_j$& \ref{ss:2.0b}\\
$M_{j,hor}'$, $M_{j,ver}'$  & horizontal, vertical algebraic monodromies of $T\Sigma_j$& \ref{ss:MO}\\
 $m_{\Phi,hor}$, $m_{\Phi,ver}$  & horizontal, vertical geometric monodromies of $(\Phi,\Delta_1)$& \ref{hvmonodromy}\\
 $M_{\Phi, hor}$, $M_{\Phi, ver}$ &  algebraic monodromies induced by
 $m_{\Phi,hor}$ and  $m_{\Phi,ver}$ & \ref{ss:MO}\\
 $m_{j,hor}^\Phi$, $m_{j,ver}^\Phi$  & horizontal, vertical geometric monodromies of $\Phi$ near $\Sigma_j$& \ref{felo}\\
$M_{j,hor}^\Phi$, $M_{j,ver}^\Phi$  &  algebraic monodromies
induced by $M_{j,hor}^\Phi$ and  $M_{j,ver}^\Phi$  & \ref{ss:MO}\\
\end{tabular}

%%%%%%%%%%%%%%%%%% starting "G" Part I.
\noindent\begin{tabular}{p{2.5cm}>{\small}p{7cm}>{\hspace{1mm}}p{2.5cm}}
$\mu(f)$, $\mu$ & Milnor number of $f$& \ref{th:ISO}\\
$\mu'_j$ & Milnor number of $T\Sigma_j$ & \ref{ss:2.0b} \\
$(m_v)$, $(m)$  & multiplicities of a plumbing graph  & \ref{bek:MULT}\\
$\{m_\nu\}_{\nu\in{\cal V}}$ & multiplicity system & \ref{bek:MULT}, \ref{summary} \\
{\small $(m_w;n_w,\nu_w)$, $(m;n,\nu)$ }& multiplicity weights of the graph $\G_\C$ & Ch. \ref{ss:1.3}\\
$\{m_j\}_{j\in \Pi}$ & cardinality of $I_j=\{L_\lambda:L_\lambda\ni p_j\}$ for an arrangement & \ref{arrang}\\
& & \\
$n$ & normalization & \ref{d1}\\
${\bn}$, $(\bn, \bd)$ & covering data & \ref{def:2.3.2}\\
$\nu(e)$ & invariant  of a cutting edge &  \ref{bek:ce1}\\
$\bnu$ & the covering data for $\gj/\sim$ & \ref{back} \\
&&\\
$P(\G)$, $P$ & plumbed 4-manifold associated with $\G$ & Ch. \ref{s:PLU}\\
$P_M(t)$ & characteristic polynomial of  $M$ & \ref{rem:HOM}, \ref{ss:chars}\\
$P^\#$, $P^\#_j$ & certain characteristic polynomials & \ref{def:P}\\
${\calp}$, $\calp(t_{hor},t_{ver})$  & subsets of $\bfc^*\times \bfc^*$ & Ch. \ref{s:vh}\\
$\{p_j\}_{j\in \Pi}$ & singular points of $C=\cup_{\lambda\in\Lambda}C_\lambda$ & \ref{homgen}, \ref{arrang}\\
& & \\
$R0$, $R1$, \ldots, $R7$ & operations of plumbing graphs & \ref{CALC}\\
& & \\
$S_\epsilon$, $S_\epsilon^{2n-1}$ & sphere of radius $\epsilon$; $(2n-1)$--dimensional &  \ref{ss:2.0a}  \\
$S^k$  & $k$--dimensional sphere & \ref{ss:2.0a} \\
$Sl_q$ & transversal slice & \ref{ss:2.0b}, \ref{m2tetel} \\
${\cals}_k$  & a real analytic surface germs & \ref{prop:LINK}\\
  ${\wsk}$  & strict transform of ${\cals}_k$ & \ref{ss:sklifted} \\
  ${\nsk}$  & normalization of $\wsk$ & \ref{NORM1} \\
  $\overline{\cals_k}$ & `resolution' of ${\cals}_k$ & \ref{outline}, \ref{ss:skres}\\
  $St$, $St(\Sigma)$ & strict transform in a resolution & \ref{1.1}, \ref{d1} \\
  $St_e$ & strict transform associated with a cutting edge &  \ref{bek:ste}\\
 $Str$, \small{$Str(a,b;c\,|\,i,j;k)$} & a (graph) string & \ref{def:2.2.1}\\
 \small{$Str^\circleddash(a,b;c\,|\,i,j;k)$} &  the same  string with $\circleddash$ edge-decorations &  \ref{def:2.2.1}\\
\end{tabular}

%%%%%%%%%%%%%%% starting "P-S"
\noindent\begin{tabular}{p{2.5cm}>{\small}p{7cm}>{\hspace{1mm}}p{2.5cm}}
 $\Sigma$, $\Sigma_f$; $\Sigma_j$  & singular locus; its irreducible components & \ref{ss:2.0a}, \ref{th:ISO}\\
 & & \\ ${\tv}$,  $T({\C})$   & tubular neighborhoods of certain curves & \ref{ss:a} \\
$T^\circ$ & interior of $T$ & \ref{ss:2.0a}\\
$T\Sigma_j$ & transversal type singularity & \ref{ss:2.0b} \\
${\rm twist}(h;A)$ & twist & \ref{def:twist}\\
$T_{a,*,*}$ & family of certain singularities & \ref{ZFG}\\
& & \\
${\cvg}$,  ${\calv}$ & the set of all vertices of a graph $\Gamma$ &  \ref{bek:pl}\\
$\calv_v$ & set of vertices  adjacent to the vertex $v$ & \ref{ex:LINKRHS}\\
$V_f$ & the zero locus of a germ $f$&  \ref{ss:2.0a}\\
 $V_f^{norm}$ & the normalization of $V_f$ &  \ref{d1} \\
&&\\
$\calw(\Gamma)$, $\calw$ & the set of non-arrowhead vertices of a graph $\Gamma$ &  \ref{links}\\
${\Wedge}$  & wedge neighbourhood & Ch. \ref{s:ICIS}, \ref{why} \\
$W_\bullet$ & weight filtration of a mixed Hodge structure & \ref{ss:MHS}\\
&&\\
$X_{f,N}$ & cyclic covering & \ref{1.6}
%\\
%$\Xi^{(d)}$ &  a morphism & \ref{be:dcov}
\end{tabular}

%\backmatter%%%%%%%%%%%%%%%%%%%%%%%%%%%%%%%%%%%%%%%%%%%%%%%%%%%%%%%

%\include{glossary}
%\include{solutions}

%IDE KENE, HOGY KERULJON AZ INDEX!!!!!!!!

%\printindex

%\include{acronym}

%%%%%%%%%%%%%%%%%%%%%%%%%%%%%%%%%%%%%%%%%%%%%%%%%%%%%%%

\backmatter

\printindex

\end{document}